\documentclass[english,11pt, reqno, oneside]{amsart}

\usepackage{mathbbol}
\usepackage{amsmath}
\usepackage{amssymb}
\usepackage{amsthm}
\usepackage{amsfonts}
\usepackage{mathtools}
\usepackage{bbm,bm}
\usepackage{stmaryrd}
\usepackage[usenames,dvipsnames]{xcolor}
\usepackage[colorlinks=true,linkcolor=blue!95!black, citecolor = green!55!black,bookmarksdepth=3]{hyperref}

\DeclareSymbolFontAlphabet{\mathbb}{AMSb}%

\usepackage[margin=1in]{geometry}

\usepackage[style=trad-alpha,sorting=nty, doi=false, url=false, isbn=false, maxnames=99, maxalphanames=6, useprefix=true]{biblatex}
\usepackage{etoolbox}
\usepackage{enumitem}
\setlist{itemsep=4pt}
\usepackage{float}
\usepackage[mathscr]{euscript}
\usepackage{graphicx}
\usepackage{subcaption}
\usepackage{tikz}
\usepackage{xspace}

\DeclareSymbolFont{eulerletters}{U}{zeur}{m}{n}
\DeclareMathSymbol{\eulL}{\mathord}{eulerletters}{`L}

\usepackage{microtype}

\allowdisplaybreaks

\DeclareMathOperator*{\argmax}{argmax}

\newcommand{\hair}{\ifmmode\mskip1mu\else\kern0.08em\fi}
\renewcommand{\P}{\mathbb{P}}

\newcommand{\E}{\mathbb{E}}

\newcommand{\F}{\mathcal{F}}
\newcommand{\Fext}{\mathcal{F}_{\mathrm{ext}}}

\newcommand{\PF}{\mathbb P_{\mathcal{F}}}
\newcommand{\EF}{\mathbb E_{\mathcal{F}}}

\newcommand{\R}{\mathbb{R}}
\newcommand{\C}{\mathbb{C}}
\newcommand{\N}{\mathbb{N}}
\newcommand{\Z}{\mathbb{Z}}
\newcommand{\Q}{\mathbb{Q}}

\newcommand{\Cov}{\mathrm{Cov}}

\newcommand{\one}{\mathbbm{1}}
\newcommand{\intint}[1]{\llbracket #1 \rrbracket}

\newcommand{\mrm}{\mathrm}
\newcommand{\msf}{\mathsf}
\newcommand{\mc}{\mathcal}
\newcommand{\mf}{\mathfrak}

\newcommand{\cP}{\mathcal P}
\newcommand{\cL}{\mathfrak L}
\renewcommand{\S}{\mathcal S}
\newcommand{\cA}{\mathcal A}
\newcommand{\h}{\mathfrak h}

\newcommand{\diff}{\mathrm d}
\newcommand{\midd}{\ \Big|\  }
\newcommand{\middd}{\ \bigg|\  }

\newcommand{\floor}[1]{\lfloor #1 \rfloor}
\newcommand{\ceil}[1]{\lceil #1 \rceil}

\newcommand{\Bbr}{B^{\mrm{Br}}}

\newcommand{\Brw}{B^{\mrm{RW}}}

\newcommand{\hssv}{h^{\mrm{S6V}}}

\newcommand{\hasep}{h^{\mrm{ASEP},\theta}}

\newcommand{\oldN}{\scalebox{1.05}{\mbox{\ensuremath{\mathtt{N}}}}}
\newcommand{\bigL}{\scalebox{1.05}{\mbox{\ensuremath{\mathtt{L}}}}}
\newcommand{\bigR}{\scalebox{1.05}{\mbox{\ensuremath{\mathtt{R}}}}}

\newcommand{\rrparen}{\Rparen}
\newcommand{\llparen}{\scalebox{0.97}{$\Lparen$}}

\newcommand{\intray}[1]{\llbracket #1,\infty\rrparen}

\renewcommand{\th}{\textsuperscript{th}\xspace}
\newcommand{\st}{\textsuperscript{st}\xspace}

\colorlet{lightgray}{gray!20!white}

\DeclareMathAlphabet{\mathdutchcal}{U}{dutchcal}{m}{n}

\usetikzlibrary{calc}
\usetikzlibrary{math}
\usetikzlibrary{decorations.markings, decorations.pathreplacing,shapes.misc}

\newtheoremstyle{indented}{4pt}{4pt}{}{}{\bfseries}{.}{.5em}{}

\newtheorem{theorem}{Theorem}[section]
\newtheorem*{theorem*}{Theorem}
\newtheorem*{proposition*}{Proposition}
\newtheorem{proposition}[theorem]{Proposition}
\newtheorem*{corollary*}{Corollary}
\newtheorem{corollary}[theorem]{Corollary}
\newtheorem{lemma}[theorem]{Lemma}

\theoremstyle{definition}
\newtheorem{definition}[theorem]{Definition}

\newtheorem{remark}[theorem]{Remark}

\numberwithin{equation}{section}

\theoremstyle{indented}
\newtheorem{assumption}{Assumption}

\title[Scaling limit of the colored ASEP and S6V models]{\vspace*{-1.5cm}Scaling limit of the \\colored ASEP and stochastic six-vertex models}

 \usepackage{amsaddr}
\author{Amol Aggarwal$^{1}$, Ivan Corwin$^2$, Milind Hegde$^3$}
\address{$^1$ Department of Mathematics, Stanford University, Stanford, CA USA. \newline $^2$ Department of Mathematics, Columbia University, New York, NY USA. \newline $^3$ School of Physical and Mathematical Sciences, Nanyang Technological University, Singapore.}
\email{amolagga@gmail.com, ivan.corwin@gmail.com, milind.hegde@ntu.edu.sg}

\newcommand{\PT}{\mrm{PT}}
\newcommand{\slope}{\beta}

\begin{document}
\lineskiplimit=0pt

\begin{abstract}

We consider the colored asymmetric simple exclusion process (ASEP) and stochastic six vertex (S6V) model with fully packed initial conditions; the states of these models can be encoded by 2-parameter height functions. We show under Kardar-Parisi-Zhang (KPZ) scaling of time, space, and fluctuations  that these height functions converge to the Airy sheet.

Several corollaries follow. (1) For ASEP and the S6V model under the basic coupling, we consider the 4-parameter height function at position $y$
and time $t$ with a step initial condition at position $x$ and time $s < t$, and prove that under KPZ scaling it converges to the directed landscape. 
(2) As a consequence, in our paper \cite{aggarwal2024kpz}, we prove that ASEPs and S6V models under the basic coupling, with multiple general initial data, converge to KPZ fixed points coupled through the directed landscape. 
(3) We prove that the colored ASEP stationary measures converge to the stationary horizon. 
(4) We prove a form of decoupling for the colored ASEP height functions, as well as for the stationary two-point function, as broadly predicted by the theory of non-linear fluctuating hydrodynamics.

The starting point for our Airy sheet convergence result is an embedding of these colored models into a larger structure---a color-indexed family of coupled line ensembles with an explicit Gibbs property, i.e., a colored Hall-Littlewood line ensemble. The core of our work then becomes to develop a framework to analyze the edge scaling limit of these ensembles.
\end{abstract}

\maketitle

 \setcounter{tocdepth}{1}
 \tableofcontents

\section{Introduction}\label{s.intro}

\subsection{Preface}

The asymmetric simple exclusion process (ASEP) was introduced by MacDonald-Gibbs-Pipkin \cite{MacDonald1968} in 1968 in the biology literature, and by Spitzer \cite{SPITZER1970246} in 1970 in the probability literature. The stochastic six-vertex (S6V) model was introduced in 1992 by Gwa-Spohn \cite{PhysRevLett.68.725} in the physics literature (as a Markovian analog of the square ice model introduced by Pauling in 1935 in the chemistry literature \cite{doi:10.1021/ja01315a102}). Interest in these models has come from applications across science, as well as connections to areas of mathematics such as interacting particle systems, conservation law differential equations, random matrix theory and quantum integrable systems.

In this paper we focus on the colored (also called multi-species, -type, or -class) versions of the above models, in which each particle is assigned an integer ``color'' (or ``priority''). The colored ASEP was introduced soon after its uncolored version, arising naturally from the basic coupling developed by Liggett and Harris in the 1970s \cite{d1f5897c-df90-3f6d-83e8-ff67f002d4ad,harris1978additive}. Since then, colored ASEP has been intensely investigated, both as a means to analyze uncolored ASEP and as a multi-component lattice gas of intrinsic interest, including as an archetypal model for understanding non-linear fluctuating hydrodynamics; see the survey of Spohn \cite{Spohn2014}. The colored S6V model weights were introduced in the 1980s by Bazhanov \cite{BAZHANOV1985321} and Jimbo \cite{Jimbo1986}, which were used to define a Markov process more recently by Kuniba-Mangazeev-Maruyama-Okado \cite{KUNIBA2016248} and Borodin-Wheeler \cite{borodin2018coloured}.

A basic question about both systems is to understand their limiting fluctuations, after being run for a long time. Especially over the past decade, aspects of the answer to this question have been attained for uncolored ASEP and a fairly complete one for the uncolored totally asymmetric simple exclusion process (TASEP). This topic has an extensive history (dating at least back to the physics work of van Beijeren-Kutner-Spohn \cite{ENDDS}). The first one-point fluctuation results under step initial conditions were proven by Johansson \cite{johansson2000shape} for the uncolored TASEP and by Tracy-Widom \cite{tracy2009asymptotics} for uncolored ASEP. The most general results now known have identified the full scaling limit for the fluctuations of uncolored TASEP under arbitrary initial conditions, after dividing time by a large parameter $t$, space by $t^{2/3}$, and fluctuations by $t^{1/3}$; these are commonly called the Kardar-Parisi-Zhang (KPZ) \cite{kardar1986dynamic} exponents in one dimension (see the surveys of Corwin \cite{corwin2012kardar} and Quastel-Spohn \cite{quastel2015one}). This was shown for uncolored TASEP in \cite{matetski2016kpz} by Matetski-Quastel-Remenik, where the scaling limit was introduced and named the KPZ fixed point, using exact determinantal formulas for TASEP due to Borodin-Ferrari-Pr\"{a}hofer-Sasamoto \cite{FPPIC}. The analogous result for uncolored ASEP was not previously available.

Less is known for the uncolored  S6V model. One-point fluctuation results were proven by Borodin-Corwin-Gorin \cite{borodin2016stochastic} under step initial conditions, by adapting the coordinate Bethe ansatz analysis implemented in \cite{tracy2009asymptotics} for ASEP. This exact analysis was later extended by Dimitrov \cite{dimitrov2023two} to obtain the scaling limit of its two-point height fluctuations (assuming the S6V asymmetry parameter $q \in [0, 1)$ is sufficiently small). 

In spite of considerable activity in recent years, progress on the colored ASEP and S6V model has been more limited. Stochastic dualities for them were uncovered  by Belitsky-Sch\"{u}tz \cite{SSD}, Kuan \cite{ADF}, and Chen-de Gier-Wheeler \cite{ISD}. Explicit formulas for ($q$-deformed) moments of their height functions were found by Borodin-Wheeler \cite{BorodinWheelerObserve}, Bufetov-Korotkikh \cite{Bufetov2021}, and de Gier-Mead-Wheeler \cite{deGier2023}.  However, even for colored TASEP, its remains open how to extract asymptotics from these results (outside of work by Chen-de Gier-Hiki-Sasamoto-Usui \cite{Chen2022} proving one-point convergence for a different colored TASEP model from the one studied here). Instead, the most fruitful way of obtaining fluctuation results for colored models had previously been through matching results, identifying certain marginals of the colored model with those of the uncolored ones; see the works of Borodin, Bufetov, Galashin, Gorin, Wheeler, and Zhang \cite{borodin2018coloured,10.1214/20-AOP1463,borodin2022shift,galashin2021symmetries,SICFT}. This matching framework is potent but, since it can only capture specific marginals of the colored models, it cannot identify their full universal scaling limit.

Even from a conjectural standpoint, what this universal scaling limit should exactly be has only come into focus within the past few years. It is given by the Airy sheet and directed landscape, introduced in 2018 by Dauvergne-Ortmann-Vir\'ag \cite{dauvergne2018directed}. These objects were originally constructed in a different context, namely, as scaling limits for Brownian last passage percolation (LPP). This is a model on $\mathbb{R} \times \mathbb{N}$, in which one has a notion of directed paths and an environment of independent Brownian motions $\bm{B} = (B_1, B_2, \ldots )$; for each $k$, we view $B_k$ as associated with the $k$\th row of the system. Any directed path is assigned a random weight, given by an integral of the noise along it. The LPP value between two given points $(x, i)$ and $(y, j)$, denoted by $\bm{B} [(x, i) \to (y,j)]$, is the maximum weight over all directed paths connecting them; all of these LPP values between different pairs of points are coupled through the common random environment $\bm{B}$. The Airy sheet $\mathcal{S} : \mathbb{R}^2 \to \mathbb{R}$ is then the random function prescribing the $n \to \infty$ limit, under KPZ scaling, of these coupled LPP values across all pairs of points on the $1$\st and $n$\th rows. Specifically, setting $\varepsilon = n^{-1}$,
\begin{gather*} 
\varepsilon^{1/3} \Big( \bm{B} \big[ (2\varepsilon^{-2/3} x, \varepsilon^{-1}) \to (2\varepsilon^{-2/3}y + \varepsilon^{-1}, 1) \big] - 2\varepsilon^{-1} + 2(x-y) \varepsilon^{-2/3} \Big)
\quad \text{converges to} \quad \mathcal{S}(x, y),
\end{gather*} 

\noindent as $\varepsilon \to 0$. The directed landscape is more generally the joint KPZ scaling limit of this coupling of LPP values across all possible pairs of points in the system (on possibly different pairs of rows). 

The Airy sheet and directed landscape have since been shown to be scaling limits for various other solvable LPP and polymer models. These include Poissonian, exponential, and geometric LPP, and the Sepp\"{a}l\"{a}inen-Johansson and Poisson lines models, by Dauvergne-Nica-Vir\'{a}g \cite{dauvergne2023uniform,dauvergne2021scaling}, as well as the KPZ equation by Wu \cite{wu2023kpz}. These works all use the Robinson-Schensted-Knuth (RSK) correspondence (or its geometric analog due to Noumi-Yamada \cite{NY04}) to equate the original LPP (or polymer) model with a ``dual'' one, whose environment is given by an ensemble of correlated random curves (such as Dyson Brownian motion in the case of Brownian LPP). In the dual model, all LPP values happen to be asymptotically determined by only the environment given by the curves near the top of the ensemble. The ensemble's top edge converges to a universal scaling limit called the Airy line ensemble, which was originally introduced by Pr\"{a}hofer-Spohn \cite{prahofer2002PNG} in terms of its finite-dimensional distributions, and later realized as an ensemble of continuous random curves by Corwin-Hammond \cite{corwin2014brownian}. This suggests that the scaling limit of the original model is given by an LPP across the Airy line ensemble; indeed, this is how the Airy sheet was defined in \cite{dauvergne2018directed}.

While this RSK framework is well-suited to (solvable) LPP and polymer models, the Airy sheet and directed landscape are also believed to arise as scaling limits for other models in the KPZ universality class. This particularly includes colored interacting particle systems, such as ASEP and S6V (see the paragraph following Theorem~\ref{t.sample result}). However, an appropriate variant of the RSK correspondence has not been found in such contexts; in fact, the relevant hypothetical output of such a correspondence, an exact representation for the system as a ``dual'' LPP or polymer, is not known (or really expected) to exist for colored ASEP or S6V in general. Partly for this reason, until now there had been no proof of convergence to the Airy sheet or directed landscape for any colored interacting particle system. 

The purpose of this paper is to provide one for the colored ASEP and S6V. \\


\subsection{Sample result} \label{s.intro.sample result}

Here we state one of our results for the colored ASEP; our conventions for this system are as follows. Across every bond of $\mathbb{Z}$, independently and in continuous time, we swap colors at rates $1$ or $q \in [0, 1)$ depending on whether the colors along the bond are in reverse order (e.g., $2$---$1$) or order (e.g., $1$---$2$), respectively; see Figure~\ref{f.colored ASEP}. We start under the ``packed'' initial condition, with site $k$ occupied by a particle of color $-k$, for each $k \in \mathbb{Z}$; this is equivalent to infinitely many uncolored ASEPs, the $k$\th one with a step initial condition at $-k \in \mathbb{Z}$, evolving jointly under the basic coupling. The colored ASEP height function $h(x, 0;y,t)$ is defined as the number of particles of color at least $-x$, whose location at time $t$ exceeds $y$. See Figure~\ref{f.sim} for a simulation.

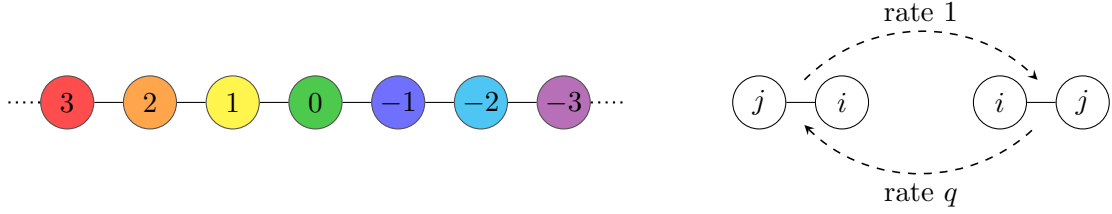
\begin{figure}[h]
\begin{tikzpicture}[scale=0.9]
\foreach \thecolor [count=\y] in {violet!80!white, cyan, blue!80!white, green!70!black, yellow, orange, red}
{ 
  \pgfmathtruncatemacro{\x}{\y-4}
  \node[circle, inner sep=7.1pt, draw=black, fill=\thecolor, opacity=0.7] (A\x) at (-1.1*\x, 0) {};
  \node at (-1.1*\x, 0) {$\x$};
}

\foreach \x in {-3,...,2}
{ 
  \pgfmathtruncatemacro{\nextx}{\x+1}
  \draw (A\x) -- (A\nextx);
}

\draw[dotted, thick] (A3) -- ++(-0.85,0);
\draw[dotted, thick] (A-3) -- ++(0.85,0);

\begin{scope}[shift={(7,0)}]
\node[circle, inner sep=7.1pt, draw=black] (aL) at (-1.1, 0) {};
\node at (-1.1, 0) {$j$};

\node[circle, inner sep=7.1pt, draw=black] (bL) at (0, 0) {};
\node at (0, 0) {$i$};

\draw (aL) -- (bL);

\node[circle, inner sep=7.1pt, draw=black] (bR) at (1.1+1, 0) {};
\node at (1.1+1, 0) {$i$};

\node[circle, inner sep=7.1pt, draw=black] (aR) at (2.2+1, 0) {};
\node at (2.2+1, 0) {$j$};

\draw (aR) -- (bR);

\draw[-stealth, dashed, semithick] (-0.5,0.3) to[out=45,in=135] node[midway, above] {rate $1$} (2.6,0.3) ;
\draw[stealth-, dashed, semithick] (-0.5,-0.3) to[in=-135,out=-45] node[midway, below] {rate $q$}  (2.6,-0.3);
\end{scope}
\end{tikzpicture}
\caption{The colored ASEP. On the left is the packed initial condition; the number inside the circle is the color of the particle. On the right are the transition rates, where $i < j$.}\label{f.colored ASEP}
\end{figure}

\begin{figure}[b]
\includegraphics[width=0.27\textwidth]{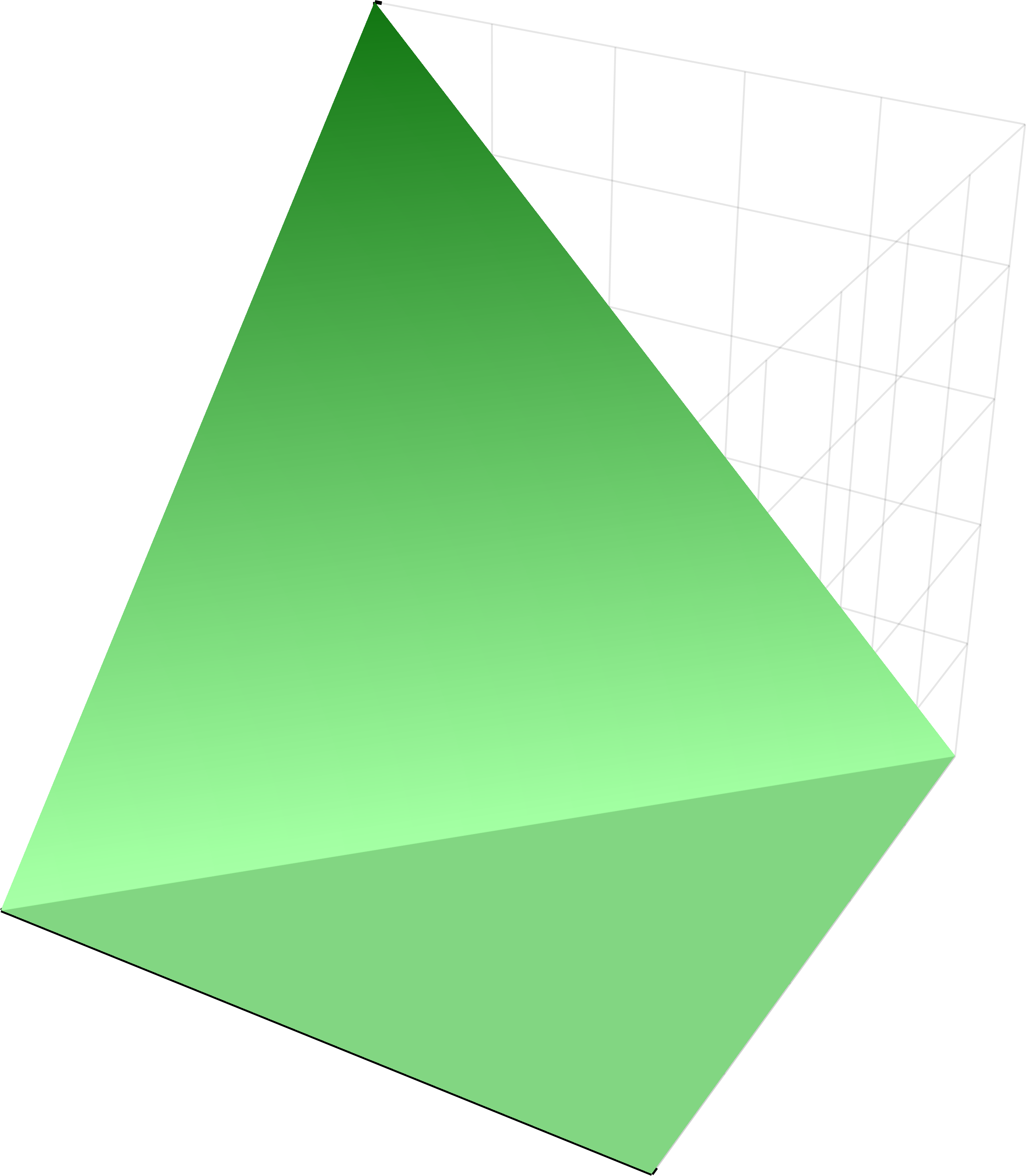}\hspace*{0.5cm}
\includegraphics[width=0.27\textwidth]{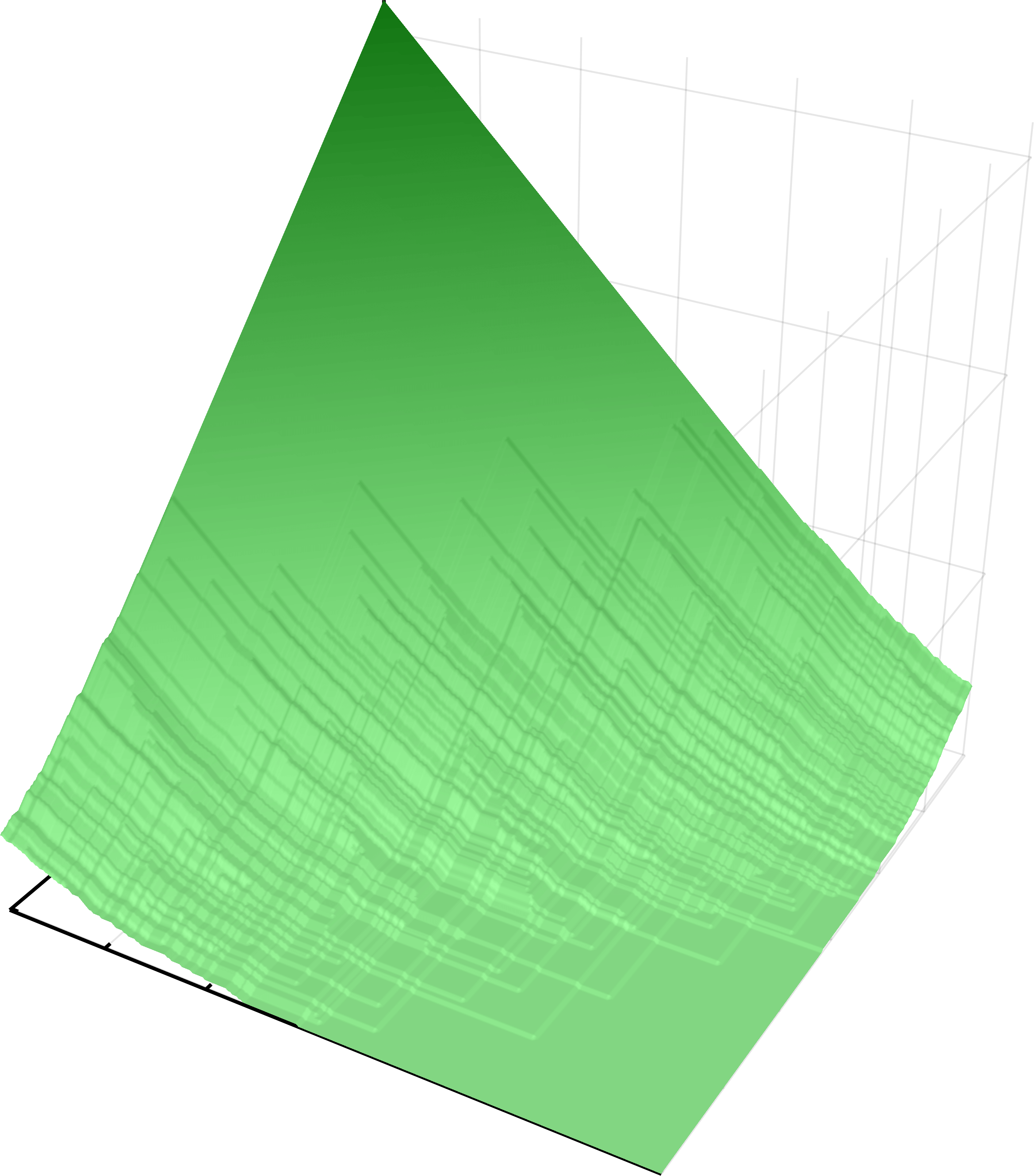}\hspace*{0.5cm}
\includegraphics[width=0.33\textwidth]{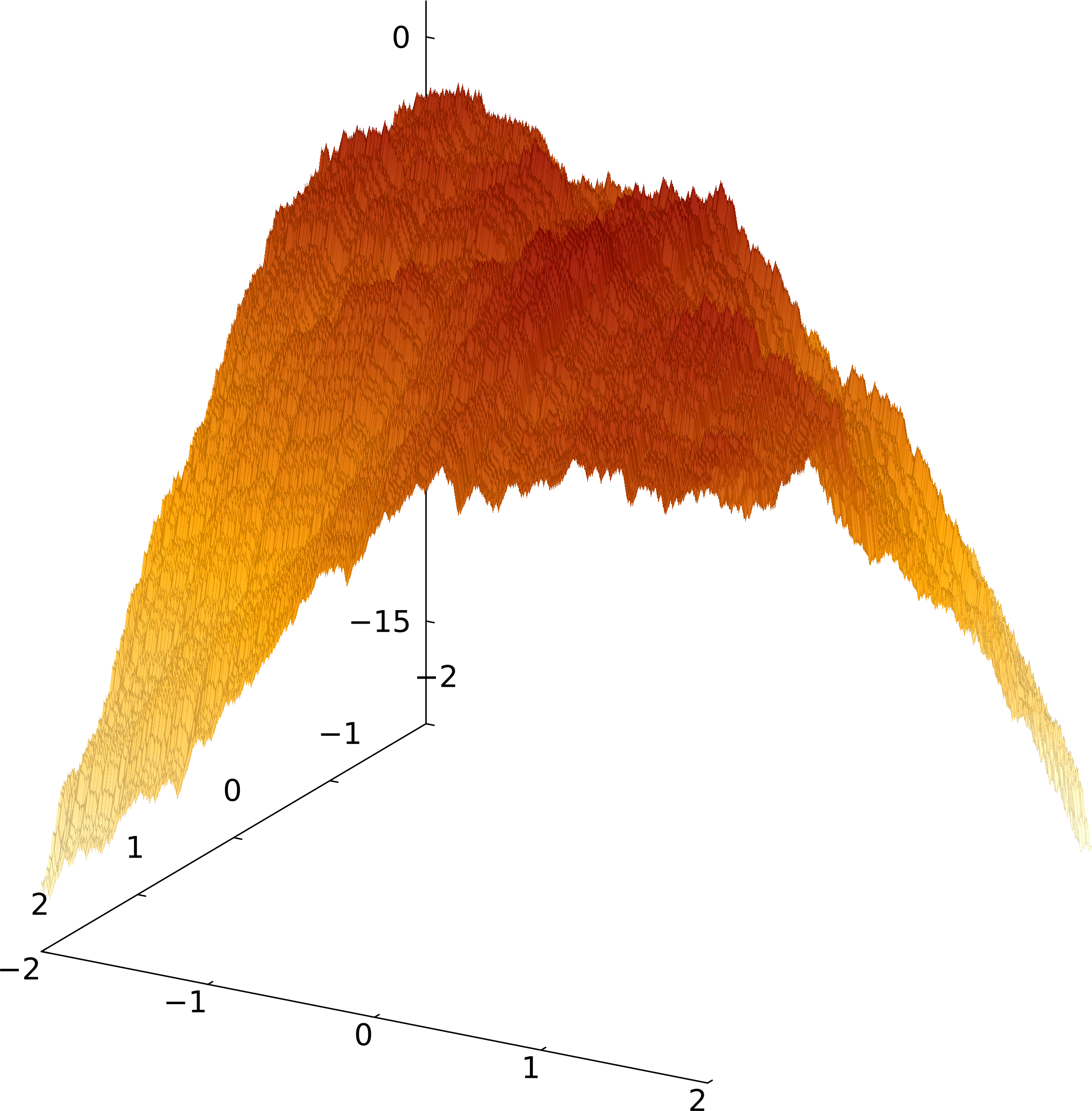}
\caption{Simulations of the colored ASEP height function $(x,y)\mapsto h(x,0;y,t)$ under packed initial conditions, with $q=0.3$. The vertical axis is the height, and the other two axes are the $x$ and $y$ coordinates. On the left is the height function at $t=0$, and in the middle is its crop at $t\approx300$. On the right is its KPZ rescaling, on $[-2,2]^2$ and at time $t=4000$.
}\label{f.sim}
\end{figure}

One of our main results states that, after KPZ rescaling, the colored ASEP height function $h (\bm\cdot, 0; \bm\cdot, t)$ converges to the Airy sheet $\mathcal{S}$ (see Definition~\ref{d.airy sheet} for its precise description). To lighten the notation in this introductory section, we only state a special case of that result (corresponding to the choice $\alpha=0$ for the velocity in the rarefaction fan); this will make the underlying normalization constants as simple as possible. For the most general version of the below result, see Theorem~\ref{t.asep airy sheet}.

\begin{theorem}[Theorem~\ref{t.asep airy sheet}, $\alpha=0$ case]\label{t.sample result}

Fix $q \in [0, 1)$. The rescaled colored ASEP height function 
\begin{gather*}
  \varepsilon^{1/3} \Big( \varepsilon^{-1} + 2(x-y) \varepsilon^{-2/3} - 2 h \big( 2 \varepsilon^{-2/3} x, 0; 2 \varepsilon^{-2/3} y , 2 \varepsilon^{-1} (1-q)^{-1} \big) \Big)
  \ \text{converges to} \  \mathcal{S}(x, y),
\end{gather*}

\noindent where $\S(x,y)$ is the Airy sheet,  weakly  as $\varepsilon \to 0$ as continuous functions on $\mathbb{R}^2$, under the topology of uniform convergence on compact sets.
\end{theorem} 

While on first glance it might not be obvious why one would expect the colored ASEP height function to converge to the Airy sheet as in Theorem~\ref{t.sample result}, from a certain perspective it is natural to anticipate the above result. Indeed, in the $q=0$ case (i.e., TASEP), it follows from the uncolored model's connection to exponential last passage percolation and a scaling limit result for the latter \cite{johansson2003discrete} that the above convergence holds marginally for any fixed $x$ and as a process in $y$. If one varies $x$ and couples the resulting TASEP evolutions by the basic coupling, at time $t$ one obtains exactly the process $(x,y)\mapsto h(x, 0; y, t)$, i.e., the colored height function is the same as the uncolored height functions coupled via the basic coupling (as mentioned above); note that this coupling differs from the coupling of the corresponding exponential LPPs with the starting point varying.

All other results about ASEP that we prove herein follow as quick corollaries of this theorem (sometimes, together with existing results in the literature). First, while the above result states a scaling limit in space, it directly implies that the space-time scaling limit for the colored ASEP under packed initial conditions is given by the directed landscape (Theorem~\ref{t.asep general initial condition} (1)). It follows by projection (also called color merging) that for any $k \ge 1$ the space-time scaling limit of ASEP with $k$ colors under step initial condition is (a marginal of) the directed landscape (Theorem~\ref{t.asep general initial condition} (2) and Remark~\ref{r.asep finitely many colors}). Second, in a follow-up paper \cite{aggarwal2024kpz} using these results as an important input, we generalize the initial conditions of these results to show that, for multiple uncolored ASEP under the basic coupling and started under arbitrary initial conditions, the height functions converge to KPZ fixed points coupled through the directed landscape; we record this result in this paper as Theorem~\ref{t.asep general initial condition} (3). Third, by combining this with results of Busani-Sepp{\"a}l{\"a}inen-Sorensen \cite{busani2022scaling, busani2023scaling}, we deduce (Corollary~\ref{c.SH convergence}) that the colored ASEP stationary measures converge to the stationary horizon, a limit introduced by Busani in \cite{busani2021diffusive}. Fourth, by combining our convergence results with geometric properties of the directed landscape \cite{dauvergne2018directed,NTM}, we deduce a strong form of decoupling for the colored ASEP height functions (Corollary~\ref{p.asymptotic independence ASEP}), as well as for the stationary two-point function (Corollary~\ref{p.covariance matrix}), as broadly predicted by the theory of non-linear fluctuating hydrodynamics \cite{Spohn2014}.

For the colored S6V model, we similarly prove its height function converges to the Airy sheet, after KPZ rescaling; see Theorem~\ref{t.s6v airy sheet}. The directed landscape space-time scaling limits under packed initial data (Theorem~\ref{t.s6v landscape}) then follows in a similar way as for ASEP. The convergence to KPZ fixed points coupled through the directed landscape when the S6V model is run from multiple general initial conditions under the basic coupling is recorded here as Theorem~\ref{t.s6v landscape}, and is proven in the same follow-up paper \cite{aggarwal2024kpz} as for ASEP. The stationary measures for the colored S6V coincide with those of the colored ASEP, so the stationary horizon convergence of the latter implies the same for the former. 
\\

To briefly describe the proofs of our results, we focus on the colored S6V model, whose discrete structure enables a more direct analysis. Our results on colored S6V can then be transferred to colored ASEP through an effective coupling of ASEP and S6V (see Lemma \ref{l.asep s6v comparison}).

The starting point for analyzing the colored S6V with $n \ge 1$ colors is an embedding, originally due to Aggarwal-Borodin \cite{aggarwalborodin}, of its height function as the top curves of a structure that we call a ``colored Hall-Littlewood line ensemble.'' In \cite{aggarwalborodin} this was proved as a consequence of the Yang-Baxter equation (we also provide a short, self-contained explanation here in Appendix~\ref{s.yang-baxter}, tailored to our special case), extending prior works of Borodin-Bufetov-Wheeler \cite{borodin2016between}, Bufetov-Mucciconi-Petrov \cite{FSSF,BUFETOV2021107865}, and Aggarwal-Borodin-Bufetov \cite{Aggarwal2019} proving similar results in the uncolored case. 

The colored Hall-Littlewood line ensemble is a family of $n$ line ensembles, one indexed by each color, where a ``line ensemble'' is a random infinite collection of Bernoulli paths; see the right side of Figure~\ref{0lmu}, where $n=2$ (and only the top three walks in each ensemble are depicted). The law of this colored line ensemble is described through certain local and explicit Boltzmann weights, which provide it with a Gibbs property that gives the law of its paths (of every color) in any given interval, conditional on the colored line ensemble outside of that interval. This couples its $n$ constituent line ensembles (of all colors) in a concrete but highly nontrivial way. We refer to any such structure as a ``colored Gibbsian line ensemble'' (see  \cite{aggarwalborodin} for other examples).

Due to this embedding, to identify the scaling limit of the colored S6V model, it suffices to pinpoint the scaling limit for the top curves of the colored Hall-Littlewood line ensemble. Most of this paper is devoted to this task. Over the past decade, a substantial theory has developed around the asymptotic analysis of Gibbsian line ensembles that are uncolored (that is, have $n=1$ colors), since the papers \cite{corwin2014brownian,corwin2016kpz} initiating a framework for proving their tightness. That uncolored context differs from our colored one, in that we must now understand how the $n$ lines ensembles in our colored family asymptotically couple together, as opposed to only understanding how any individual line ensemble in the family behaves marginally. 

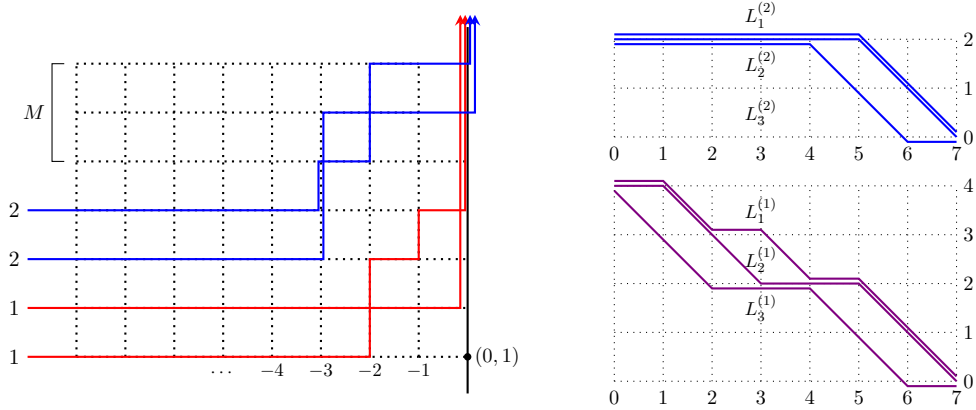
\begin{figure}[t]
  
  \begin{center}
    \begin{tikzpicture}[
      >=stealth,
      scale = .65]{ 
        
        \draw[thick] (23, -.25) -- (23, 7.25);
        
        \draw[black, thick, dotted] (15, .5) -- (15, 6.5);
        \draw[black, thick, dotted] (16, .5) -- (16, 6.5);
        \draw[black, thick, dotted] (17, .5)  -- (17, 6.5);
        \draw[black, thick, dotted] (18, .5) node[below = 2, scale = .65]{$\cdots$} -- (18, 6.5);
        \draw[black, thick, dotted] (19, .5) node[below, scale = .65]{$-4$}  -- (19, 6.5);
        \draw[black, thick, dotted] (20, .5) node[below, scale = .65]{$-3$} -- (20, 6.5);
        \draw[black, thick, dotted] (21, .5) node[below, scale = .65]{$-2$} -- (21, 6.5);
        \draw[black, thick, dotted] (22, .5) node[below, scale = .65]{$-1$} -- (22, 6.5);
        
        \draw[black, thick, dotted] (15, .5) -- (22, .5);
        \draw[black, thick, dotted] (15, 1.5) -- (22, 1.5);
        \draw[black, thick, dotted] (15, 2.5) -- (22, 2.5);
        \draw[black, thick, dotted] (15, 3.5) -- (22, 3.5);
        \draw[black, thick, dotted] (15, 4.5) -- (22, 4.5);
        \draw[black, thick, dotted] (15, 5.5) -- (22, 5.5);
        \draw[black, thick, dotted] (15, 6.5) -- (22, 6.5);
        
        \draw[thick, dotted] (22, 2.5) -- (23, 2.5);
        \draw[thick, dotted] (22, 4.5) -- (23, 4.5);
        \draw[thick, dotted] (22, .5) -- (23, .5);
        \draw[thick, dotted] (22, 1.5) -- (23, 1.5);
        
        \draw[->, red, thick] (14, .5) node[left, scale = .75, black]{$1$} -- ++(1, 0) -- ++(6, 0) -- ++(0, 1) -- ++(1.85, 0) -- ++(0, 6);
        \draw[->, red, thick] (14, 1.5) node[left, scale = .75, black]{$1$} -- ++(1, 0) -- ++(6, 0) -- ++(0, 1) -- ++(1, 0) -- ++(0, 1) -- ++(0.95, 0) -- ++(0, 4);
        \draw[->, blue, thick] (14, 2.5) node[left, scale = .75, black]{$2$} -- ++(1, 0) -- ++(5.05, 0) -- ++(0, 3) -- ++(0.95, 0) -- ++(0, 1) -- ++(1, 0) -- ++(1.05, 0) -- ++(0, 1);
        \draw[->, blue, thick] (14, 3.5) node[left, scale = .75, black]{$2$} -- (15, 3.5) -- (19.95, 3.5) -- (19.95, 4.5) -- (21, 4.5) -- (21, 5.5) -- (23.15, 5.5) -- (23.15, 7.5);
        
        \draw[-] (14.75, 4.5) -- (14.5, 4.5) -- (14.5, 6.5) -- (14.75, 6.5);
        \draw[] (14.5, 5.5) circle [radius = 0] node[left, scale = .7]{$M$};

        \node[circle, fill, inner sep = 1pt] at (23, 0.5) {};
        \node[anchor=west, scale=0.7] at (23,0.5) {$(0,1)$};

        \foreach \x in {0,...,4}
          \draw[dotted] (26, \x) -- ++(7, 0) node[right, scale = .7]{$\x$};

        \foreach \x in {0,...,7}
          \draw[dotted] (26+\x, 0) node[below = 1, scale = .7]{$\x$} -- ++(0, 4);

        \foreach \x in {0,...,2}
          \draw[dotted] (26, 5+\x) -- ++(7, 0) node[right, scale = .7]{$\x$};

        \foreach \x in {0,...,7}
          \draw[dotted] (26+\x, 5) node[below = 1, scale = .7]{$\x$} -- ++(0, 2);

        \draw[ thick, blue] (26, 7.1) -- (31, 7.1) -- (33, 5.1);
        \draw[ thick, blue] (26, 7) -- (31, 7) -- (33, 5);
        \draw[ thick, blue] (26, 6.9) -- (30, 6.9) -- (32, 4.9) -- (33, 4.9);

        \draw[ thick, violet] (26, 4.1) -- ++(1, 0) -- ++(1, -1) -- ++(1, 0) -- ++(1, -1) -- ++(1, 0) -- ++(1, -1) -- ++(1, -1);
        \draw[ thick, violet] (26, 4) -- (27, 4) -- (28, 3) -- (29, 2) -- (31, 2) -- (32, 1) -- (33, 0);
        \draw[ thick, violet] (26, 3.9) -- (27, 2.9) -- (28, 1.9) -- (29, 1.9) -- (30, 1.9) -- (31, .9) -- (32, -.1) -- (33, -.1);
        
        \draw[] (29, 7.1) circle [radius = 0] node[above, scale = .65]{$L_1^{(2)}$};
        \draw[] (29, 6.5) circle [radius = 0] node[scale = .65]{$L_2^{(2)}$};
        \draw[] (29, 5.9) circle [radius = 0] node[below, scale = .65]{$L_3^{(2)}$};
        
        \draw[] (29, 3.1) circle [radius = 0] node[above, scale = .65]{$L_1^{(1)}$};
        \draw[] (29, 2.5) circle [radius = 0] node[scale = .65]{$L_2^{(1)}$};
        \draw[] (29, 1.9) circle [radius = 0] node[below, scale = .65]{$L_3^{(1)}$};
      }
    \end{tikzpicture}
  \end{center}
  
  \caption{\label{0lmu} On the left is a colored $q$-Boson model (where color $1$ is red and color $2$ is blue). On the right is its associated colored line ensemble $\bm{L} = \big( \bm{L}^{(1)}, \bm{L}^{(2)} \big)$ (where $\bm{L}^{(1)}$ is purple,  and $\bm{L}^{(2)}$ is blue).}
  
\end{figure}

To do so, we will address these two issues separately, by proving the following two statements.

\begin{enumerate} 
  \item \emph{Identifying the coupling} (Theorem \ref{t.approxmate LPP problem representation general}): All of the line ensembles in the colored family are approximately determined via an LPP problem across just one of them (namely, the line ensemble indexed by the minimal color).
  \item \emph{Airy limit of the determining ensemble} (Theorem \ref{t.line ensemble convergence to parabolic Airy}): This ``determining'' line ensemble converges around its top edge to the parabolic Airy line ensemble.
\end{enumerate}

Hence, the other colored line ensembles converge to limits determined by LPP problems in the parabolic Airy line ensemble; as discussed above, this is how the Airy sheet was defined \cite{dauvergne2018directed}. 

Let us briefly comment on the two statements above. The first can be viewed as a substitute for the above-mentioned ``dual'' LPP (or polymer) representation that was used in \cite{dauvergne2018directed,dauvergne2023uniform,dauvergne2021scaling,wu2023kpz} to analyze LPP (and polymer) models, though now its underlying algebraic source is the Yang-Baxter equation instead of the RSK correspondence. Indeed, it enables the colored S6V (and ASEP) height functions to be interpreted as LPP values across a dual environment given by the determining line ensemble in the colored family, but now only approximately (instead of exactly). 

On the second statement, this determining line ensemble can be shown to marginally have the structure of an uncolored Gibbsian line ensemble that was first studied by Corwin-Dimitrov \cite{corwin2018transversal}; they observed that it does not satisfy the Fortuin-Kasteleyn-Ginibre (FKG) inequality, i.e., does not admit monotone couplings. To prove its tightness (which is the main new input needed to access its edge scaling limit, due to characterization results of Dimitrov-Matetski \cite{dimitrov2021characterization} or Aggarwal-Huang \cite{aggarwal2023strong}), we must thus circumvent many prior uses of monotonicity, which requires an appreciable reworking of the theory developed in \cite{corwin2014brownian,corwin2016kpz}. See Section \ref{s.intro.line ensemble tightness} for further discussion.

Before proceeding, let us make two remarks. First, as shown in \cite[Figure 1]{aggarwalborodin}, the structure of colored Gibbsian line ensembles is common to a wide class of stochastic models studied in integrable probability. These colored ensembles are of differing complexities but are likely to be of a similar qualitative nature to the colored Hall-Littlewood line ensemble studied here. Thus, this paper could serve as a blueprint for accessing the Airy sheet scaling limits for these other models (where the strong characterization result for the Airy line ensemble proven in \cite{aggarwal2023strong} is likely to play a more vital role). Second, the approximate LPP representation of the colored line ensembles in terms of a ``determining'' one applies for quite general initial data, not only the packed one; this might serve as a starting point for the direct analysis of these models from general initial conditions. We do not pursue either of these points further here and leave them as future directions.

The remainder of this section is a more detailed proof outline, to act as an informal guide (focusing on the main ideas and intuition) for readers examining the arguments of this paper in greater detail.

\subsection{Proof outline} 

\label{Modelq}

\subsubsection{Colored Hall-Littlewood line ensemble}

\label{ModelColoredEnsemble}

The colored S6V model (see Section~\ref{s.intro.cS6V} for a precise definition) is a system of up-right directed random paths, each assigned an integer color. At every vertex, a weight is assigned depending on the configuration of colored paths adjacent to it. Given parameters $q \in [0, 1)$ and $z \in (0, 1)$, these weights are depicted below, where we assume $i<j$:

\medskip

\begin{center}
  \begin{tikzpicture}
    \newcommand{\thedim}{4.2}
    \begin{scope}[scale=0.5]
      
      \draw[step=\thedim] (0,0) grid (5*\thedim, \thedim);
      
      \draw[->, line width=1.2pt, red] (0.5*\thedim, 0.25*\thedim) -- ++(0, 0.5*\thedim);
      \draw[->, line width=1.2pt, red] (0.25*\thedim, 0.5*\thedim) -- ++(0.5*\thedim, 0);
      \node[scale=0.8, anchor=west] at (0.75*\thedim+0.1, 0.5*\thedim) {$i$};
      \node[scale=0.8, anchor=south] at (0.5*\thedim, 0.75*\thedim+0.1) {$i$};
      \node[scale=0.8, anchor=east] at (0.25*\thedim-0.1, 0.5*\thedim) {$i$};
      \node[scale=0.8, anchor=north] at (0.5*\thedim, 0.25*\thedim-0.1) {$i$};
      
      \draw[->, line width=1.2pt, blue] (1.5*\thedim, 0.25*\thedim) -- ++(0, 0.5*\thedim);
      \draw[->, line width=1.2pt, red] (1.25*\thedim, 0.5*\thedim) -- ++(0.5*\thedim, 0);
      \node[scale=0.8, anchor=west] at (1.75*\thedim+0.1, 0.5*\thedim) {$i$};
      \node[scale=0.8, anchor=south] at (1.5*\thedim, 0.75*\thedim+0.1) {$j$};
      \node[scale=0.8, anchor=east] at (1.25*\thedim-0.1, 0.5*\thedim) {$i$};
      \node[scale=0.8, anchor=north] at (1.5*\thedim, 0.25*\thedim-0.1) {$j$};
      
      \draw[->, line width=1.2pt, red] (2.5*\thedim, 0.25*\thedim) -- ++(0, 0.5*\thedim);
      \draw[->, line width=1.2pt, blue] (2.25*\thedim, 0.5*\thedim) -- ++(0.5*\thedim, 0);
      \node[scale=0.8, anchor=west] at (2.75*\thedim+0.1, 0.5*\thedim) {$j$};
      \node[scale=0.8, anchor=south] at (2.5*\thedim, 0.75*\thedim+0.1) {$i$};
      \node[scale=0.8, anchor=east] at (2.25*\thedim-0.1, 0.5*\thedim) {$j$};
      \node[scale=0.8, anchor=north] at (2.5*\thedim, 0.25*\thedim-0.1) {$i$};
      
      \draw[->, line width=1.2pt, red] (3.25*\thedim, 0.5*\thedim) -- ++(0.25*\thedim, 0) -- ++(0,0.25*\thedim);
      \draw[->, line width=1.2pt, blue] (3.5*\thedim, 0.25*\thedim) -- ++(0, 0.25*\thedim) -- ++(0.25*\thedim, 0);
      \node[scale=0.8, anchor=west] at (3.75*\thedim+0.1, 0.5*\thedim) {$j$};
      \node[scale=0.8, anchor=south] at (3.5*\thedim, 0.75*\thedim+0.1) {$i$};
      \node[scale=0.8, anchor=east] at (3.25*\thedim-0.1, 0.5*\thedim) {$i$};
      \node[scale=0.8, anchor=north] at (3.5*\thedim, 0.25*\thedim-0.1) {$j$};
    
      \draw[->, line width=1.2pt, blue] (4.25*\thedim, 0.5*\thedim) -- ++(0.25*\thedim, 0) -- ++(0,0.25*\thedim);
      \draw[->, line width=1.2pt, red] (4.5*\thedim, 0.25*\thedim) -- ++(0, 0.25*\thedim) -- ++(0.25*\thedim, 0);
      \node[scale=0.8, anchor=west] at (4.75*\thedim+0.1, 0.5*\thedim) {$i$};
      \node[scale=0.8, anchor=south] at (4.5*\thedim, 0.75*\thedim+0.1) {$j$};
      \node[scale=0.8, anchor=east] at (4.25*\thedim-0.1, 0.5*\thedim) {$j$};
      \node[scale=0.8, anchor=north] at (4.5*\thedim, 0.25*\thedim-0.1) {$i$};
      
      \draw[xstep=\thedim, ystep=0.35*\thedim] (0,0) grid (5*\thedim, -0.35*\thedim);
      
      \newcommand{\they}{-0.175*\thedim}
      
      \node[scale=0.7] at (0.5*\thedim, \they) {$1$};
      \node[scale=0.7] at (1.5*\thedim, \they) {$\displaystyle\frac{q(1-z)}{1-qz}$};
      \node[scale=0.7] at (2.5*\thedim, \they) {$\displaystyle\frac{1-z}{1-qz}$};
      \node[scale=0.7] at (3.5*\thedim, \they) {$\displaystyle\frac{1-q}{1-qz}$};
      \node[scale=0.7] at (4.5*\thedim, \they) {$\displaystyle\frac{z(1-q)}{1-qz}$};
      
    \end{scope}
  \end{tikzpicture}
\end{center}

\medskip

These weights satisfy a Yang-Baxter equation that relates them to those of a different system called the (discrete) colored $q$-Boson model. This is also a system of colored paths, but where vertical edges may now accommodate arbitrarily many colored arrows (while horizontal edges still accommodate at most one); see the left side of Figure~\ref{0lmu}. The explicit forms for its weights, which are not stochastic, lead to the colored Hall-Littlewood Gibbs property; the probability of a configuration of paths in this model is equal to the product of all vertex weights, suitably normalized. These vertex weights are given below, where we again assume $i < j$, $u\in(0,1)$ is the spectral parameter, and we let $A_{[k, \ell]}$ denote the number of arrows of color in the interval $[k, \ell]$ that are vertically entering the vertex (writing $A_k$ if $k = \ell$):

\medskip

\begin{center}
  \begin{tikzpicture}[scale=0.9]
    \newcommand{\thedim}{4.2}
    \begin{scope}[scale=0.6]
      
      \draw[step=\thedim] (0,0) grid (6*\thedim, \thedim);
      
      \draw[->, very thick, blue] (0.5*\thedim, 0.2*\thedim) -- ++(0, 0.6*\thedim);
      \draw[->, very thick, green!80!black] (0.5*\thedim+0.15, 0.2*\thedim) -- ++(0, 0.6*\thedim);
      
      \draw[->, very thick, blue] (1.5*\thedim-0.075, 0.2*\thedim) -- ++(0, 0.6*\thedim);
      \draw[->, very thick, green!80!black] (1.5*\thedim+0.075, 0.2*\thedim) -- ++(0, 0.3*\thedim) -- ++(0.3*\thedim, 0);
      \node[scale=0.6, anchor=west] at (1.8*\thedim+0.1, 0.5*\thedim) {$i$};
      
      \draw[->, very thick, blue] (2.5*\thedim +0.075, 0.2*\thedim) -- ++(0, 0.6*\thedim);
      \draw[->, very thick, green!80!black] (2.2*\thedim, 0.5*\thedim) -- ++(0.3*\thedim +0.2, 0) -- ++(0, 0.3*\thedim);
      \node[scale=0.6, anchor=east] at (2.2*\thedim, 0.5*\thedim+0.1) {$i$};
      
      \draw[->, very thick, blue] (3.5*\thedim, 0.2*\thedim) -- ++(0, 0.6*\thedim);
      \draw[->, very thick, green!80!black] (3.2*\thedim, 0.5*\thedim) -- ++(0.3*\thedim -0.15, 0) -- ++(0, 0.3*\thedim);
      \draw[->, very thick, orange!90!black] (3.5*\thedim+0.15, 0.2*\thedim) -- ++(0, 0.3*\thedim) -- ++(0.3*\thedim, 0);
      \node[scale=0.6, anchor=east] at (3.2*\thedim, 0.5*\thedim+0.1) {$i$};
      \node[scale=0.6, anchor=west] at (3.8*\thedim+0.075, 0.5*\thedim) {$j$};
      
      \draw[->, very thick, blue] (4.5*\thedim, 0.2*\thedim) -- ++(0, 0.6*\thedim);
      \draw[->, very thick, orange!90!black] (4.2*\thedim, 0.5*\thedim+0.1) -- ++(0.3*\thedim +0.15, 0) -- ++(0, 0.3*\thedim-0.1);
      \draw[->, very thick, green!80!black] (4.5*\thedim-0.15, 0.2*\thedim) -- ++(0, 0.3*\thedim) -- ++(0.3*\thedim, 0);
      \node[scale=0.6, anchor=east] at (4.2*\thedim, 0.5*\thedim+0.1) {$j$};
      \node[scale=0.6, anchor=west] at (4.75*\thedim, 0.5*\thedim) {$i$};
      
      \draw[->, very thick, blue] (5.5*\thedim, 0.2*\thedim) -- ++(0, 0.6*\thedim);
      \draw[->, very thick, green!80!black] (5.5*\thedim+0.15, 0.2*\thedim) -- ++(0, 0.6*\thedim);
      \draw[->, very thick, orange!90!black] (5.2*\thedim, 0.5*\thedim) -- ++(0.6*\thedim, 0);
      \node[scale=0.6, anchor=west] at (5.8*\thedim, 0.5*\thedim) {$i$};

      \draw[xstep=\thedim, ystep=0.25*\thedim] (0,0) grid (6*\thedim, -0.25*\thedim);
      
      \newcommand{\they}{-0.125*\thedim}
      
      \node[scale=0.70] at (0.5*\thedim, \they) {$1$};
      \node[scale=0.70] at (1.5*\thedim, \they) {$u(1-q^{A_i})q^{A_{[i+1,N]}}$};
      \node[scale=0.70] at (2.5*\thedim, \they) {$1$};
      \node[scale=0.70] at (3.5*\thedim, \they) {$u(1-q^{A_j})q^{A_{[j+1,N]}}$};
      \node[scale=0.70] at (4.5*\thedim, \they) {$0$};
      \node[scale=0.70] at (5.5*\thedim, \they) {$uq^{A_{[i+1,N]}}$};
      
    \end{scope}
  \end{tikzpicture}
\end{center}

\medskip

These $q$-boson weights are fusions \cite{kulish1981yang} of the colored S6V weights (computed explicitly in \cite{bogoliubov1992q} in the uncolored case, and as a consequence of \cite{bosnjak2016construction} in the colored case). This implies that the $q$-Boson weights satisfy a Yang-Baxter equation.
In \cite{aggarwalborodin} the Yang-Baxter equation was shown to imply a distributional match between the colored S6V model and a colored $q$-Boson model on the negative quadrant $\mathbb{Z}_{< 0} \times \mathbb{Z}$ (see the left side of Figure~\ref{0lmu}), identifying the law of the former's height function with the marginal law of the latter's rightmost column; this colored $q$-Boson model is also sometimes called a ``colored Hall-Littlewood process'' \cite[Section 1.6]{borodin2018coloured}; the $n=1$ case of this object (with just one color) is called an uncolored Hall-Littlewood process, which was used in the work \cite{corwin2018transversal} to study tightness properties of the uncolored S6V/ASEP height functions. The paper \cite{aggarwalborodin} developed these types of matching results for considerably more general fused vertex models. Here, we only require the colored S6V model special case of that framework, whose Yang-Baxter equation origin is relatively quick to explain conceptually. Appendix~\ref{s.yang-baxter} provides a self-contained proof of the result we use from \cite{aggarwalborodin}.

The colored $q$-Boson model on the negative quadrant gives rise in a fairly direct way to a colored line ensemble, which we call the colored Hall-Littlewood line ensemble (see Section~\ref{s.colored line ensemble}). Specifically, we set the location of the $i$\th curve of the ($\geq j$)-colored line ensemble at site $y$ to be the number of paths in the colored $q$-Boson model of color at least $j$ that pass strictly above the point $(-i, y)$; see Figure~\ref{0lmu}. The fact that the colored $q$-Boson model is prescribed by local Boltzmann weights  implies a local Gibbs property for the full color-indexed collection of line ensembles. However, each constituent line ensemble in this colored family (except for the ``first'' one corresponding to all the colors in the system) will usually not marginally satisfy a Gibbs property on its own.

This correspondence identifies the rightmost column of the colored $q$-Boson model with the first (topmost) curves of the $n$ line ensembles in the colored family. Therefore, the distributional match described above implies that the height function of the colored S6V model embeds as the first curves in the colored line ensemble (see Proposition~\ref{p.colored line ensembles}). Extracting the large scale limit for the colored S6V model thus reduces to doing so for the above colored Gibbsian line ensemble around its (top) edge. This task simplifies considerably when $q = 0$, so we will first discuss that case (where our results are still new). We then explain what must be done to extend our results to all $q \in [0, 1)$.

\subsubsection{Simplifications and asymptotics when $q=0$}

\label{Asymptoticq0}

The weights of the colored (and uncolored) $q$-Boson model simplify when $q=0$ to be either zero or one. This has several useful consequences.

First, in the case of only one color, the uncolored Hall-Littlewood line ensemble (i.e., uncolored Hall-Littlewood process), whose topmost curve encodes the uncolored S6V height function, becomes at $q=0$ an ensemble of independent Bernoulli random walks conditioned on not intersecting. This is a Schur process, which is known to be a determinantal point process by work of Okounkov-Reshetikhin \cite{OR03}. The finite-dimensional distributions around the top edge of the ensemble can therefore be computed exactly and shown to converge as $t\to \infty$ to those of the Airy line ensemble. The convergence can then be upgraded to functional convergence of the curves by use of the non-intersecting Gibbs property and existing techniques in the theory of Gibbsian line ensembles, facilitated by the key property of monotone coupling that holds in this $q=0$ case: Gibbs measures with ordered boundary conditions can likewise be stochastically ordered.

Now we pass to multiple colors. The colored $q=0$ S6V height function is again related to the top curves of a colored line ensemble, consisting of a collection of several line ensembles indexed by color. The line ensemble corresponding to the lowest color in this family, which we assume to be $1$ (as we may, by a color merging projection), has the marginal law of an uncolored line ensemble, in this case the above Schur process (Proposition \ref{p.L has HL}). The other line ensembles in the colored family are nontrivially coupled to this ($\geq 1$)-colored one via the colored $q = 0$ Boson weights. Let us explain how to extract the joint asymptotics for the ($\geq 1$)-colored and ($\geq 2$)-colored line ensembles; the case for more colors then follows inductively.

We begin by examining the ($\geq 1$)-colored line ensemble and conclude (as described above) that it marginally converges around its top edge to the parabolic Airy line ensemble. Next, conditioned on the ($\geq 1$)-colored line ensemble, we study the ($\geq 2$)-colored one. As all $q=0$ Boson weights are either zero or one, this second line ensemble is seen to be determined completely by the first one. We then prove (Theorem~\ref{t.approxmate LPP problem representation general}), through a careful inspection of these $q=0$ Boson weights, that the ($\geq 2$)-colored line ensemble can be realized as an LPP problem (via a suitable ``Pitman transform'') across the ($\geq 1$)-colored one. Therefore, the first and second line ensembles jointly converge in the scaling limit to a parabolic Airy line ensemble and an LPP problem across it, giving the Airy sheet.

Turning this back into the $q=0$ colored S6V model, this implies the convergence of its height function under KPZ scaling to the Airy sheet. In the limit to ASEP, this shows that the height function for the colored TASEP under packed initial condition also converges to the Airy sheet.

When $q>0$, there are two obstructions in implementing the approach sketched above. The first is in extracting the edge limit of the uncolored Hall-Littlewood line ensemble; the second is in extracting the joint limit of the coupled line ensembles.

\subsubsection{Edge limit for uncolored line ensembles ($q>0$)}\label{s.intro.line ensemble tightness}

For $q > 0$ the basic input in Section~\ref{Asymptoticq0}, the convergence of the uncolored Hall-Littlewood line ensemble to the parabolic Airy one, is no longer known; the reason is that the former (and likewise, S6V and ASEP) is not determinantal. 

Hence, we adopt a more probabilistic route, by showing that any subsequential edge limit of the Hall-Littlewood line ensemble is Brownian Gibbsian; this means that it satisfies the Gibbs property for non-intersecting Brownian bridges, called the Brownian Gibbs property. This, with the recent strong characterization for Brownian Gibbsian line ensembles established in \cite{aggarwal2023strong}, would then uniquely identify the limit as the parabolic Airy line ensemble.

Assuming tightness under KPZ scaling of the Hall-Littlewood line ensemble around its top edge, the Brownian Gibbs property for subsequential limits follows from the Gibbs property of prelimits, along with Kolm\'{o}s-Major-Tusn\'{a}dy type couplings between random walks and Brownian bridges. The main effort therefore rests in proving this tightness; the inputs we will need to do this are its Gibbs property and known \cite{borodin2016stochastic} one-point asymptotics for its top curve.

Previous works establishing full tightness for Gibbsian line ensembles all make substantial use of a monotonicity property (FKG inequality), stating that if one increases the boundary data then all of its paths stochastically increase. This monotonicity property is false for the Hall-Littlewood line ensemble, a fact originally observed in \cite{corwin2018transversal}, where this Gibbs property was first studied. That work showed tightness of the top curve of the Hall-Littlewood line ensemble by way of a much weaker form of monotonicity wherein, upon changing boundary conditions, certain one-point probabilities change monotonically up to a large constant. In order to extend tightness to an arbitrary number of curves, one could pursue such a weak form of monotonicity involving more curves. We were unable to produce a suitable generalization.

To explain the method we instead use, we first informally describe the Hall-Littlewood Gibbs property (see Section~\ref{s.line ensemble convergence} for a more precise definition). It states that the law of the top $k$ curves on a spatial interval $(a,b)$ conditioned on the entire external state of the ensemble only depends on the boundary data, i.e., the height of the top $k$ curves at $\{ a, b \}$ and that of the $(k+1)$\st curve on $[a,b]$. The conditioned law is given explicitly by a reweighting of the law for $k$ Bernoulli random walk bridges with the same boundary data, conditioned to neither intersect each other nor the $(k+1)$\st curve. The weight is proportional to a product of factors (one for each position) equal to either $1$ or $1-q^{d}$, where $d$ is the distance between two consecutively labeled curves at that position. The latter factor only occurs when the height difference between two consecutive curves decreases by $1$.

Now, to establish tightness of our line ensemble, we show that it suffices to \emph{a priori} demonstrate a reasonable degree of separation between its curves. Indeed, since $q$ is fixed, the factors $1-q^d$ approach $1$ as the distances $d$ between the curves grows; although many such factors may exist in the product (up to the size of the spatial interval, of order $t^{2/3}$), if we can show a separation of $d\gg \log t$ with high probability, then the reweighting factor becomes very close to $1$. We therefore incur little error by replacing these factors by $1$, after which the tightness analysis follows from arguments similar to those used in previously studied models.

It therefore suffices to show this \emph{a priori} separation between curves. This is the primary technical component of the proof for tightness and is done by induction. Using the Gibbs property and the one-curve weak monotonicity (a minor refinement of that in \cite{corwin2018transversal}), we first prove that the $k$\th curve in the line ensemble is uniformly bounded below, assuming that the first $k-1$ curves are uniformly separated (which holds trivially at $k=1$); then we show that the first $k$ curves are likely separated at one spatial coordinate, using the control on how far down the $k$\th curve can be; and finally we conclude that the first $k$ curves are uniformly separated across the full spatial interval. See Section~\ref{s.lower tail} for further details.

\subsubsection{Approximate LPP representation for other colors  ($q>0$)}

As with the uncolored Gibbs property, the colored Gibbs property becomes more involved when $q > 0$; its weights are not all equal to zero or one. So, the ($\geq 2$)-colored line ensemble in the family is now not deterministic upon conditioning on the ($\geq 1$)-colored one (which has the marginal law of an uncolored Hall-Littlewood line ensemble; see Proposition \ref{p.L has HL}). In particular, the representation of the former as an LPP problem across the latter is no longer exactly true when $q>0$. 

We instead prove that this LPP representation remains approximately true with high probability for any $q \in [0, 1)$; see Theorem~\ref{t.approxmate LPP problem representation general}. Our proof of this fact is combinatorial. While the colored $q$-Boson weights are not all equal to zero or one conditional on the ($\geq 1$)-colored line ensemble, they involve certain powers of $q$ (namely, factors of $q^{A_{[k, \ell]}}$). We show that these powers become quite large, thereby exponentially penalizing the ($\geq 2$)-colored line ensemble, if the latter substantially differs from being an LPP problem across the ($\geq 1$)-colored one. Since $q < 1$, this can be used to prove the above approximate LPP representation; see Section~\ref{s.approximate LPP} for more details. 

Using this approximate LPP representation, and the convergence of the uncolored Hall-Littlewood line ensemble to the parabolic Airy one, our Airy sheet limit results follow just as in the $q=0$ case.

\subsection*{Notation}
For $a,b\in\Z$ with $a<b$, $\intint{a,b} := \{a, \ldots, b\}$, $\llbracket a,\infty \rrparen:= \{a,a+1, \ldots\}$, and  $\llparen {-}\infty, a \rrbracket := \{a,a-1, \ldots\}$. Symbols written in bold face such as $\bm A$ or $\bm v$ will represent vectors or tuples. For a finite set $A$, $\#A$ will denote its cardinality. For $a,b\in\R$, $a\wedge b$ is shorthand for $\min(a,b)$.

For random objects $X$ and $Y$ taking values in some measurable space, $X\stackrel{\smash{d}}{=} Y$ means that their distributions are the same; we will often omit explicitly specifying the $\sigma$-algebra on the target space, but will specify its topology (in which case we endow the space with the associated Borel $\sigma$-algebra). For random objects $X_n$ and $X$ taking values in some common topological space, $\smash{X_n\xrightarrow{d} X}$ means that $X_n$ converges weakly to $X$; the topological space will be specified or obvious in the context. We will sometimes abuse notation and say $\smash{X_n\xrightarrow{d} \mu}$ for some probability measure $\mu$ on the same space. The space of continuous functions from a topological space $\mc X$ to $\R$ will be denoted $\mc C(\mc X, \R)$. $\mc N(m,\sigma^2)$ will denote the normal distribution with mean $m$ and variance $\sigma^2$. For a $\sigma$-algebra $\F$, $\PF$ and $\EF$ are shorthand for the conditional probability and conditional expectation given $\F$; the conditional probability should be regarded as the regular conditional probability (which exists in the contexts we will consider by standard results, e.g., \cite[Theorem 8.5]{Kallenberg} or, in the case of the discrete setting of much of the paper, automatically). Events will be written in sans serif font, e.g., $\msf E$, and $\msf E^c$ will denote the complement of $\msf E$.

We will sometimes write $f(\bm\cdot)$ for a function of the variable $\bm\cdot$, and sometimes $x\mapsto f(x)$. For $\bm f = (f_1, f_2, \ldots )$ a collection of continuous functions, $\bm{f}[(\bm\cdot, \bm\cdot) \to (\bm\cdot, \bm\cdot)]$ will denote the last passage percolation values in the environment given by $\bm{f}$, and is defined ahead in Definition~\ref{d.lpp}.

\subsection*{Organization of paper} Section~\ref{s.models and results} precisely defines our models and states our main results. Section~\ref{s.key ideas} introduces the colored $q$-Boson model and colored Hall-Littlewood line ensemble; explains the latter's relation to the colored S6V model; sets up a framework for the convergence of the uncolored line ensemble to the parabolic Airy line ensemble (Theorem~\ref{t.line ensemble convergence to parabolic Airy}); and states the approximate last passage percolation representation of the colored Hall-Littlewood line ensemble (Theorem~\ref{t.approxmate LPP problem representation general}). Section~\ref{s.approximate LPP} proves the latter result. Section~\ref{s.properties of S6V and ASEP} explains how to fit ASEP into the vertex model framework of the other arguments. Section~\ref{s.tightness preliminaries} begins setting up the ingredients for the proof of Theorem~\ref{t.line ensemble convergence to parabolic Airy} and gives its proof assuming them. Section~\ref{s.airy sheet} gives the proofs of our main results. Sections~\ref{s.monotonicity}--\ref{s.bg in the limit} concern the proofs of the ingredients of Theorem~\ref{t.line ensemble convergence to parabolic Airy}: Section~\ref{s.monotonicity} explains the weak monotonicity property our arguments rely on; Sections~\ref{s.lower tail} and \ref{s.uniform separation} contain the inductive argument yielding lower tail control of lower curves in the line ensemble as well as their uniform separation; Section~\ref{s.partition function and non-intersection} uses these to obtain control on the partition function; and Section~\ref{s.bg in the limit} uses the same to establish all subsequential limits of the uncolored line ensembles have the Brownian Gibbs property.

There are also five appendices. Appendix~\ref{s.yang-baxter} gives a self-contained proof of the mapping from \cite{aggarwalborodin}, between the colored S6V model and the colored $q$-Boson model, using the Yang-Baxter equation. Appendices~\ref{s.random gibbs} and \ref{s.G_k asymptotics} give routine proofs of certain miscellaneous lemmas and of asymptotics of last passage percolation and maximizer values across the parabolic Airy line ensemble. Appendix~\ref{s.general initial condition} gives the proofs of convergence results of the colored ASEP to the directed landscape and coupled KPZ fixed point and colored S6V to the directed landscape, using an idea communicated to us by Shalin Parekh, combined with an approximate monotonicity result in the S6V case. Appendix~\ref{s.proofs of asymptotic independence} contains the proofs of our results on the decoupling of the colored ASEP height functions and the stationary two-point function.

\subsection*{Code for simulations} Code for simulations of the ASEP sheet (as in the third panel of Figure~\ref{f.sim}) and ASEP landscape is available in the arXiv source for this article.

\subsection*{Acknowledgements}
The authors wish to thank Shalin Parekh for explaining an approach to arrive at directed landscape convergence from our main result of Airy sheet convergence (see Remark~\ref{r.appendix contributions}), which, with his gracious permission, we have included in this text; Evan Sorensen for discussions related to the stationary horizon and his work \cite{busani2023scaling};  Alexei Borodin for discussions about colored stochastic vertex models and for helpful comments on an earlier draft of this manuscript; and Herbert Spohn and Patrik Ferrari for suggestions to pursue the strong decoupling for colored ASEP and discussions about its relationship to non-linear fluctuating hydrodynamics. The authors also wish to thank the anonymous referees for their thorough reading of the paper and helpful suggestions, and Patrik Ferrari for bringing to our attention the non-cancellation of boundary terms in Lemma~\ref{l.summation by parts}.

Amol Aggarwal was partially supported by a Packard Fellowship for Science and Engineering, a Clay Research Fellowship, by the NSF through grant DMS-1926686, and by the IAS School of Mathematics.
Ivan Corwin was partially supported by the NSF through grants DMS-1937254, DMS-1811143, DMS-1664650, DMS-2246576 and DMS-2348156, by the Simons Foundation through an Investigator Award, and through the W.M.~Keck Foundation through a Science and Engineering Grant.
Milind Hegde was partially supported by the NSF through grants DMS-1937254 and DMS-2348156, and by the NTU Support Grant for
Research Award Number \#025661-0000.


\section{Models and statements of results}\label{s.models and results}

Here we define our limiting objects (Section~\ref{s.limiting objects}) and the models of study in this article and state our results (Section~\ref{s.intro.colored asep} for ASEP and Section~\ref{s.intro.cS6V} for stochastic six-vertex).

\subsection{The Airy line ensemble, Airy sheet, and directed landscape}\label{s.limiting objects}

In this section we define our limiting objects. We start with the parabolic Airy line ensemble, which is needed to give the definitions of the Airy sheet and the directed landscape.

\subsubsection{Parabolic Airy line ensemble}

\begin{definition}[Parabolic Airy line ensemble]\label{d.parabolic Airy line ensemble}
The \emph{parabolic} and \emph{stationary Airy line ensembles} $\bm\cP, \bm{\mc A}: \N\times\R\to \R$ are $\N$-indexed families of random non-intersecting continuous curves $(\cP_i(\bm\cdot))_{i\in\N} := (\cP(i,\bm\cdot))_{i\in\N}$ and $(\mc A_i(\bm\cdot))_{i\in\N} := (\mc A(i,\bm\cdot))_{i\in\N}$, related by
$$\cP_i(x) := \mc A_i(x) - x^2.$$
The finite-dimensional distributions of $\bm{\mc A}$ are determinantal with kernel the extended Airy kernel $K_\mathrm{Ai}^{\mathrm{ext}}$ given by
$$K_\mathrm{Ai}^{\mathrm{ext}}\bigl((x,t); (y,s)\bigr) := \begin{cases}
\int_0^\infty e^{-\lambda(t-s)}\mathrm{Ai}(x+\lambda)\mathrm{Ai}(y+\lambda)\, \mathrm d\lambda & t\geq s\\
-\int^0_{-\infty} e^{-\lambda(t-s)}\mathrm{Ai}(x+\lambda)\mathrm{Ai}(y+\lambda)\, \mathrm d\lambda & t< s,
\end{cases}$$
where $\mathrm{Ai}(x) = \frac{1}{\pi}\int_0^\infty \cos(\frac{t^3}{3}+xt)\,\diff t$ is the classical Airy function. This means the following (adopting the notation $(t_j, y_j)\in \bm{\mc A}$ for the existence of some $k$ such that $\mc A_k(t_j) = y_j$ and $\mrm d\P$ for the density with respect to Lebesgue measure):
 for every $m\in\N$ and $(t_1, y_1), \ldots, (t_m, y_m)\in\R^2$,
\begin{equation}\label{e.airy multipoint}
\mrm d\P\Bigl((t_j, y_j)\in \bm{\mc A} \text{ for } j=1, \ldots, m\Bigr) = \det\Bigl[ K_{\mrm{Ai}}^{\mrm{ext}}((y_i,t_i); (y_j, t_j))\Bigr]_{1\leq i,j\leq m}\prod_{j=1}^m \mrm d y_j.
\end{equation}

These finite-dimensional distributions were first discovered in \cite{prahofer2002PNG} (see also \cite{Johansson2003}), and the existence of such an ensemble was shown in \cite[Theorem 3.1]{corwin2014brownian} (and the uniqueness follows from the prescription \eqref{e.airy multipoint} of the finite-dimensional distributions).

The distribution of $\mc A_1(0)$ is known as the \emph{GUE Tracy-Widom distribution} (discovered much earlier in random matrix theory \cite{tracy1994level}).
\end{definition}

An alternative, more probabilistic, description of the Airy line ensemble is that it is the unique \cite{aggarwal2023strong} family of continuous random curves satisfying what is known as the Brownian Gibbs property; see Definition~\ref{d.bg} and Proposition~\ref{p.strong characterization} for more details.

\begin{figure}[h]
\includegraphics[scale=0.9]{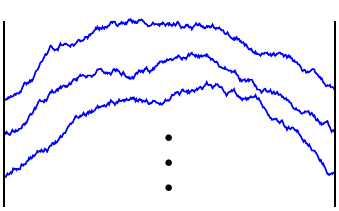}
\caption{A depiction of the parabolic Airy line ensemble.}\label{f.para airy}
\end{figure}

\subsubsection{Last passage percolation}

We next introduce the notion of a last passage percolation (LPP) problem; this will be needed in defining the Airy sheet, as well as in some of our proofs.

\begin{definition}[Last passage percolation]\label{d.lpp}
Given real numbers $u<v$, a sequence of continuous functions $\bm f=(f_1$, $f_2, \ldots, f_n)$ with $f_i:[u,v]\to\R$ for each $i\in\intint{1,n}$,  and natural numbers $j < k \leq n$, we define the LPP value from $(u,k)$ to $(v,j)$ in the environment given by $\bm f$ by (where $t_{j-1} = v$ and $t_k= u$)
\begin{align}\label{e.LPP definition}
\bm f[(u,k) \to (v,j)] = \sup_{t_k < \ldots < t_{j-1}} \sum_{i=j}^k \bigl(f_i(t_{i-1}) - f_i(t_i)\bigr).
\end{align}
\end{definition}

\subsubsection{Airy sheet, directed landscape, and KPZ fixed point}
The Airy sheet is defined via a coupling with the parabolic Airy line ensemble.

\begin{definition}[Airy sheet, {\cite[Definition 8.1]{dauvergne2018directed}}]\label{d.airy sheet}
The \emph{Airy sheet} $\S:\R^2\to\R$ is the unique (in law) random continuous function such that:
\begin{enumerate}
  \item[(i)] \emph{Stationarity}: $\S$ has the same law as $\S(\bm\cdot + x, \bm\cdot + x)$ for all $x\in\R$; and

  \item[(ii)] \emph{Coupling with $\bm\cP$}: $\S$ can be coupled with a parabolic Airy line ensemble $\bm\cP$ such that $\S(0;\bm\cdot) = \cP_1(\bm\cdot)$ and for all $(x,y,z)\in\Q_{>0}\times\Q^2$, almost surely,
  \begin{align}\label{e.airy sheet increment}
  \S(x;y)-\S(x;z) = \lim_{k\to\infty} \left(\bm\cP[(-\sqrt{k/2x}, k) \to (y,1)] - \bm\cP[(-\sqrt{k/2x}, k) \to (z,1)]\right).
  \end{align}

\end{enumerate}
The existence and uniqueness of the Airy sheet is established in  \cite[Proposition 8.2 and Theorem 8.3]{dauvergne2018directed}.
\end{definition}

The Airy sheet can be thought of as a marginal of a four-parameter space-time object (the parameters being a pair of space-time coordinates) when the two time coordinates are set to be $0$ and $1$. This four-parameter object is known as the directed landscape and is in fact constructed from the Airy sheet in a manner analogous to the construction of Brownian motion from the Gaussian distribution. We turn to defining it next.
 In what follows, we let $\R^4_{\shortuparrow} := \left\{(x,s;y,t) \in \R^4: s<t\right\}$.

\begin{definition}[Directed landscape, {\cite[Definition 10.1]{dauvergne2018directed}}]\label{d.directed landscape}
The \emph{directed landscape} is the uni\-que (in law) random continuous function $\mc L : \R^4_{\shortuparrow} \to \R$ that satisfies the following properties.
\begin{enumerate}
  \item[(i)] \emph{Airy sheet marginals}: For any $s \in \R$ and $t > 0$ the increment over time interval $[s, s + t)$ is a rescaled Airy sheet:
 \begin{align*}
  (x,y)\mapsto \mc L(x, s; y, s + t) \stackrel{d}{=} (x,y)\mapsto t^{1/3}\S(xt^{-2/3}; yt^{-2/3}).
 \end{align*}

\item[(ii)] \emph{Independent increments}: For any $k\in\N$ and disjoint time intervals $\{(s_i, t_i) : i \in \intint{1, k}\}$, the random functions
$$(x,y)\mapsto\mc L(x, s_i; y, t_i)$$
are independent as $i \in \intint{1,k}$ varies.

\item[(iii)] \emph{Metric composition law}: Almost surely, for any $r < s < t$ and $x, y \in \R$,
$$\mc L(x, r; y, t) = \max_{z\in\R}\Bigl(\mc L(x, r; z, s) + \mc L(z, s; y, t)\Bigr).$$
\end{enumerate}
The existence and uniqueness of the directed landscape is proved in \cite[Theorem 10.9]{dauvergne2018directed}.
\end{definition}

We also need to introduce height functions from general initial conditions in  the directed landscape.

\begin{definition}[KPZ fixed point] \label{d.kpz fixed point}

Let $\mrm{UC}$ be the space of upper semi-continuous functions
$f : \R \to \R \cup \{-\infty\}$ satisfying $f(x) \leq m|x| + c$ for some $m,c<\infty$ and all $x\in\R$. For a function $\mf h_0 \in\mrm{UC}$, and $t>0$, define the \emph{KPZ fixed point} $(y,t)\mapsto\mf h(\mf h_0; y,t): \R\times(0,\infty)\to \R$ started from $\mf h_0$ at time $0$ by
\begin{align}\label{e.kpz fixed point}
\mf h(\mf h_0; y, t) := \sup_{x\in\R} \Bigl(\h_0(x) + \mc L(x,0;y,t)\Bigr).
\end{align}
\end{definition}

The KPZ fixed point was first defined in \cite{matetski2016kpz} as a Markov process on UC (with the topology of local UC convergence, i.e., local Hausdorff convergence of hypographs; see \cite[Section 3.1]{matetski2016kpz}) by specifying its transition kernel through explicit formulas. This was later shown to be equivalent to \eqref{e.kpz fixed point}, which arises as a scaling limit of solvable last passage percolation models, in \cite[Corollary~4.2]{nica2020one}. Note that a natural coupling of $\smash{\mf h(\mf h^{(j)}_0; y, t)}$ for different initial conditions $\smash{\mf h^{(j)}_0}$ is obtained by using the same directed landscape $\mc L$ in the variational formula.

\subsection{Colored ASEP}\label{s.intro.colored asep}

The colored (also called multi-species, -type, or -class) asymmetric simple exclusion process is an interacting particle system on $\Z$, where each particle has an associated \emph{color}, which is an integer (multiple particles may share the same color and there may be an infinite number of distinct colors). There is also an asymmetry parameter $q\in[0,1)$ which is fixed.  At time zero, we have a particle of some color at each site of $\Z$, which is encoded by the configuration $\eta_0 = (\eta_0(i))_{i\in\Z}\in\Z^{\Z}$, where $\eta_0(i)$ is the color of the particle at site $i$. The state at any later time $t$ will be denoted by $\eta_t$.

\begin{figure}[h]
\begin{tikzpicture}[scale=1.6]

\draw[line width=1.2pt] (-0.5,0) -- ++ (5,0);
\draw[line width=1.2pt, dotted] (-1,0) -- ++(0.5,0);
\draw[line width=1.2pt, dotted] (4.5,0) -- ++(0.5,0);

\foreach \x/\thecolor in {0/red, 1/orange, 2/green!60!black, 3/blue!80!black, 4/cyan}
{
\node[circle, fill, inner sep = 2pt, \thecolor] at (\x,0) {};
}

\draw[->, line width=1.2pt] (1,1) --node[midway, green!60!black, above=-1pt, scale=0.8]{$\bm{\checkmark}$} ++(1,0);

\draw[->, line width=1.2pt] (2,1.7) --node[midway, green!60!black, above=-1pt, scale=0.8]{$\bm{\checkmark}$} ++(-1,0);
\draw[->, line width=1.2pt] (0,0.6) --node[midway, green!60!black, above=-1pt, scale=0.8]{$\bm{\checkmark}$} ++(1,0);
\draw[->, line width=1.2pt] (4,0.5) --node[midway, above=2pt, cross out, inner sep=2pt, outer sep=0pt, draw=red, minimum size=4pt, line width=1.2pt]{} ++(-1,0);
\draw[->, line width=1.2pt] (3,1.2) --node[midway, green!60!black, above=-1pt, scale=0.8]{$\bm{\checkmark}$} ++(1,0);
\draw[->, line width=1.2pt] (0,2.2) --node[midway, above=2pt, cross out, inner sep=2pt, outer sep=0pt, draw=red, minimum size=4pt, line width=1.2pt]{} ++(1,0);

\draw[line width=1.2pt, orange] (1,0) -- (1,0.6);
\draw[line width=1.2pt, orange] (0,0.6) -- (0,2.5);

\draw[line width=1.2pt, red] (0,0) -- (0,0.6);
\draw[line width=1.2pt, red] (1,0.6) -- (1,1);
\draw[line width=1.2pt, red] (2,1) -- (2,1.7);
\draw[line width=1.2pt, red] (1,1.7) -- (1,2.5);

\draw[line width=1.2pt, green!60!black] (2,0) -- (2,1);
\draw[line width=1.2pt, green!60!black] (1,1) -- (1,1.7);
\draw[line width=1.2pt, green!60!black] (2,1.7) -- (2,2.5);

\draw[line width=1.2pt, blue!80!black] (3,0) -- (3,1.2);
\draw[line width=1.2pt, blue!80!black] (4,1.2) -- (4,2.5);

\draw[line width=1.2pt, cyan] (4,0) -- (4,1.2);
\draw[line width=1.2pt, cyan] (3,1.2) -- (3,2.5);
\end{tikzpicture}
\caption{A depiction of the graphical construction of colored ASEP. In this illustration, initially particles are in decreasing color from left to right. The arrows indicate the times where a left clock $\xi^{\mrm{L}}_x$ or right clock $\xi^{\mrm{R}}_x$ rang. If an arrow goes from site $x$ to site $y$ at time $t$ and the color  at time $t^-$ of the particle at $x$ is greater than that of the particle at $y$, the particles swap positions at time $t$ (indicated by a green check mark); if not, nothing happens (red cross).}\label{f.graphical construction}
\end{figure}
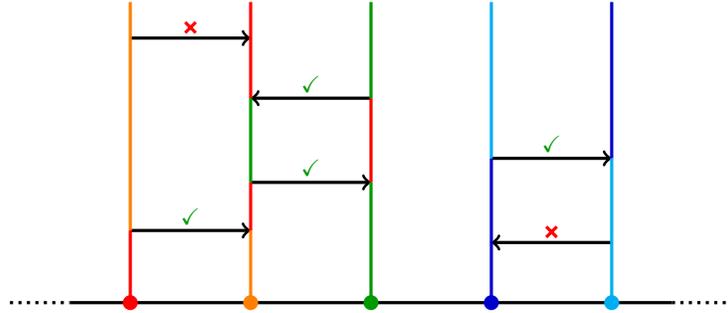

\subsubsection{The dynamics}\label{s.asep dynamics}
We may construct the dynamics of colored ASEP using a variant of Harris' graphical construction \cite{harris1978additive}  (see Figure~\ref{f.graphical construction}). Fix a family of independent Poisson clocks $\bm \xi = \{\xi_x^{\mrm{L}}, \xi_x^{\mrm{R}}\}_{x\in\Z}$, one pair for each site in $\Z$, with $\xi_x^{\mrm{L}}$ of rate $q$ and $\xi_x^{\mrm{R}}$ of rate $1$. Given the randomness of the clocks, the dynamics are deterministic. When the clock $\xi_x^{\mrm{L}}$ rings, the particle at site $x$ attempts a jump to the left by one site, and when $\xi_x^{\mrm{R}}$ rings, it attempts a jump to the right by one site. The attempt succeeds if the particle at the target site is of color lower than the particle attempting to jump, and fails otherwise. If the attempt succeeds, the two particles swap positions. The well-definedness of this description follows from \cite[Section 10]{harris1978additive}. Note that higher colors have higher priority, in contrast to the dynamics in terms of ``classes'', where typically particles of higher class have lower priority (e.g., second class particles in comparison to first class particles). Our convention is consistent with behavior of higher color arrows in the colored S6V model that we will shortly introduce.

The \emph{uncolored (or single-species) ASEP} corresponds to the case when there are only two colors in the system; in the typical terminology, the particles of higher color are referred to as particles, while the particles of lower color are referred to as holes.

\begin{remark}\label{r.asep color merging}
An important property of colored ASEP is known as \emph{color merging}. It says that one can project colored ASEP onto one of fewer colors by declaring all particles of color lying in a given interval $I\subseteq \Z$ to be the same color (which is itself an element of $I$, or, more generally, lies between the largest color below $I$ in the system and the smallest color above $I$ in the system); that this holds is readily apparent from the above description of the dynamics. Equivalently, any function $f:\Z\cup\{-\infty\}\to\Z\cup\{-\infty\}$ which is weakly monotone commutes with the Markov semigroup of the process; in particular, after merging colors we may relabel the groups in such a way that the relabeling maintains the order of the groups. The case of obtaining uncolored ASEP from colored ASEP is a particular example of color merging.
\end{remark}

\subsubsection{Packed and step initial conditions}
The initial condition we will work with from now on unless otherwise specified is the \emph{packed initial condition}, where we have a particle of color $i$ at site $-i$ for each $i\in\Z$. By color merging, we can easily obtain the case of \emph{step initial condition} for finitely many colors (and thus uncolored ASEP in the case of two colors): for integers $\ell_k< \ldots <\ell_1$, we merge the colors in $\llparen {-}\infty, \ell_k\rrbracket$ to have color $k$; those in $\intint{\ell_{i+1}+1,\ell_{i}}$ to have color $i$ for $i=1, \ldots, k-1$; and those in $\llbracket \ell_1+1, \infty\rrparen$ to have color $0$.

We note that under the packed initial condition, $\eta_t$ is a permutation on $\Z$ (i.e., $\eta_t:\Z\to\Z$ is a bijection), and thus colored ASEP can be thought of as a Markov chain on such permutations in which transpositions are applied at rate $q$ or $1$ depending on whether the entries are sorted or not respectively. In this sense, our results can be understood as a scaling limit for this infinite random transposition model. While this precise perspective of ASEP as a process on infinite permutations does not seem to have been explicitly mentioned in the literature previously, the equivalence of the simple symmetric exclusion process (i.e., $q=1$) and the interchange process is well-known \cite[Chapter VIII]{liggett1985interacting}. Similarly, models of biased random transpositions related to ASEP have been studied on finite spaces \cite{angel2009oriented,bufetov2020interacting,gheissari2025mixing}.

\subsubsection{The ASEP sheet}
As mentioned, we are now working with the packed initial condition.
The first observable of ASEP that we study, called the \emph{colored height function} is defined, for $t\geq 0$, by
\begin{align}\label{e.intro ASEP height function}
h^{\mrm{ASEP}}(x, 0; y, t) :=
\#\bigl\{z>y: \eta_t(z) \geq -x\bigr\},
\end{align}
or, in words, it is the number of particles of color greater than or equal to $-x$ which are strictly to the right of $y$ at time $t$. Equivalently, it is the number of particles that were initially at or to the left of $x$ which are strictly to the right of $y$ at time $t$. It can also be thought of as a collection indexed by $x\in\Z$ of coupled copies of the uncolored ASEP's height function or current under step initial condition from $x$ (i.e., a particle at every site at or to the left of $x$), by writing it as $(h^{\mrm{ASEP}}(x, 0;\bm\cdot, t))_{x\in\Z}$. We will refer to this coupling as the \emph{colored coupling}; it agrees with the well-known basic coupling, which we will introduce in more detail in Section~\ref{s.asep basic coupling}.

Our results concern the scaling limit of this collection of height functions, which we define next. Let $\gamma = 1-q$ and fix $\alpha\in (-1,1)$. Let $\mu^{\mrm{ASEP}}(\alpha)$ and $\sigma^{\mrm{ASEP}}(\alpha)$ be defined by
\begin{align}\label{e.mu sigma ASEP}
\mu(\alpha) = \mu^{\mrm{ASEP}}(\alpha) := \tfrac{1}{4}(1-\alpha)^2 \qquad\text{and}\qquad \sigma(\alpha) = \sigma^{\mrm{ASEP}}(\alpha) :=\tfrac{1}{2}(1-\alpha^2)^{2/3},
\end{align}
and let the spatial scaling factor $\beta^{\mrm{ASEP}}(\alpha)$ be given by
\begin{align}\label{e.nu ASEP}
\beta(\alpha)=\beta^{\mrm{ASEP}}(\alpha) := \frac{2\sigma(\alpha)^2}{|\mu'(\alpha)|(1-|\mu'(\alpha)|)} = 2(1-\alpha^2)^{1/3}.
\end{align}

\begin{definition}[ASEP sheet]\label{d.asep sheet}
For $\varepsilon>0$, we define the \emph{ASEP sheet} $\S^{\mrm{ASEP}, \varepsilon}: \R^2\to\R$, for $x$ and $y$ such that the arguments of $h^{\mrm{ASEP}}$ below are integers, by
\begin{align}
\MoveEqLeft[10]
\S^{\mrm{ASEP}, \varepsilon}(x; y) := \sigma(\alpha)^{-1}\varepsilon^{1/3}\Bigl(\mu(\alpha)2\varepsilon^{-1} + \mu'(\alpha)\beta(\alpha)(y-x)\varepsilon^{-2/3}\nonumber\\
&- h^{\mrm{ASEP}}\bigl(\beta(\alpha)x\varepsilon^{-2/3}, 0; 2\alpha \varepsilon^{-1} + \beta(\alpha)y\varepsilon^{-2/3}, 2\gamma^{-1} \varepsilon^{-1}\bigr)\Bigr); \label{e.rescaled asep definition}
\end{align}
the values at all other $x$, $y$ are determined by linear interpolation. Note that there is a dependence on $\alpha$ and $\gamma$ which is omitted from the notation $\S^{\mrm{ASEP}, \varepsilon}$, and, indeed, the limit of $\S^{\mrm{ASEP},\varepsilon}$ will not have any such dependence.
\end{definition}

Thus the ASEP sheet captures the spatial scaling limit of the colored ASEP at time $2\gamma^{-1}\varepsilon^{-1}$.

\subsubsection{An explanation of the scalings}\label{s.asep scalings} Let us say a few words on a heuristic to understand the scalings in Definition~\ref{d.asep sheet}. The $\varepsilon^{-2/3}$ spatial scaling and $\varepsilon^{1/3}$ factor outside come from the KPZ universality class scaling exponents (see the surveys \cite{corwin2012kardar,quastel2015one}). We evaluate the height function at time $2\gamma^{-1}\varepsilon^{-1}$ for the following reasons. The factor of $2$ is essentially just a convention; it allows us to include fewer constant factors in the definition of $\S^{\mrm{ASEP},\varepsilon}$ to ensure that it will converge to $\S$ and matches the conventions in earlier works such as \cite{matetski2016kpz,quastel2022convergence}. On the other hand, the factor of $\gamma^{-1}$ is important as it allows all other scaling constants to not depend on $q$; in other words, by speeding up the process by the single particle drift, all ASEPs (i.e., corresponding to different values of $q$) look the same on the fluctuation scale, a fact which was previously known for various marginals. Finally, $\alpha$ controls the macroscopic velocity of the particles around which we look, and the range $(-1,1)$ of $\alpha$ corresponds to the full ``rarefaction'' fan of the step initial condition's hydrodynamic limit shape (as given by $\mu(\alpha)$); thus, equivalently, $\alpha$ specifies the region of the limit shape where we consider the fluctuations.

Next we explain the definitions of $\mu(\alpha)$ and $\sigma(\alpha)$ from \eqref{e.mu sigma ASEP} and $\beta(\alpha)$ from \eqref{e.nu ASEP}, as well as their role in \eqref{e.rescaled asep definition}. Essentially, the latter two are chosen such that in the limit, the local diffusion rate and parabolic curvature of $y\mapsto \S^{\mrm{ASEP},\varepsilon}(0;y)$ are $2$ and $1$ respectively, matching those of the parabolic Airy line ensemble (Definition~\ref{d.parabolic Airy line ensemble}). More precisely, it is known \cite{NBM,HLA,HLAP} (see also \cite[Theorem 3]{tracy2009asymptotics} or \cite[Theorem 11.3]{borodin2017asep}) that the first order behavior of $\smash{h^{\mrm{ASEP}}(0, 0; 2\alpha \varepsilon^{-1}, 2\gamma^{-1}\varepsilon^{-1})}$ is $\smash{\mu(\alpha)\cdot2\varepsilon^{-1}=\frac{1}{4}}(1-\alpha)^2\cdot2\varepsilon^{-1}$ (i.e., $\mu(\alpha) = \lim_{\varepsilon\to 0}\varepsilon\smash{h^{\mrm{ASEP}}(0, 0; \alpha \varepsilon^{-1}, \gamma^{-1}\varepsilon^{-1})}$); thus, locally, one would expect the process to behave like a Bernoulli random walk (with increments of $0$ or $-1$) with drift $\mu'(\alpha)=\frac{\diff\ }{\diff\alpha}\frac{1}{4}\smash{(1-\alpha)^2} = -\frac{1}{2}(1-\alpha)$, which must be subtracted off to see the fluctuations. Finally, note that the local curvature of the limit profile is $\frac{1}{2}\mu''(\alpha)=\smash{\frac{1}{2}\frac{\diff^2\ }{\diff\alpha^2}\frac{1}{4}(1-\alpha)^2 = \frac{1}{4}}$ and the approximately Bernoulli random walk behavior has increment variance $|\mu'(\alpha)|(1-|\mu'(\alpha)|)=\frac{1}{4}(1-\alpha^2)$. The spatial scaling factor of $\beta(\alpha)$ and the external factor of $\sigma(\alpha)^{-1}$ are chosen such that
\begin{align}\label{e.asep scaling relations}
\beta(\alpha) |\mu'(\alpha)|(1-|\mu'(\alpha)|)\sigma(\alpha)^{-2} = 2 \quad\text{and}\quad \tfrac{1}{4}\beta(\alpha)^2\mu''(\alpha)\sigma(\alpha)^{-1} =1,
\end{align}
i.e., the local diffusion rate is $2$ and the local curvature is $1$; for the latter relation, note that there is an extra factor of $\frac{1}{2}$ compared to the curvature discussion above. This arises because we are looking at the process with an extra factor of $2$ in the time coordinate; more precisely, it arises because we expect $\S^{\mrm{ASEP};\varepsilon}(0;y)$ to be approximately
\begin{align*}
\MoveEqLeft[0]
\mu\left(\frac{1}{2\varepsilon^{-1}}(2\alpha \varepsilon^{-1}+ \beta(\alpha)y\varepsilon^{-2/3})\right)\cdot 2\varepsilon^{-1}\\
&\approx \left(\mu(\alpha) + \beta(\alpha)\mu'(\alpha)y \varepsilon^{-2/3}\cdot (2\varepsilon^{-1})^{-1} + \frac{1}{2}\mu''(\alpha)\beta(\alpha)^2y^2\varepsilon^{-4/3}\cdot\frac{1}{(2\varepsilon^{-1})^2} \right)\cdot 2\varepsilon^{-1},
\end{align*}
and we see the factor of $\frac{1}{2}$ coming from simplifying $(2\varepsilon^{-1})/(2\varepsilon^{-1})^2$ in the last term.

These scaling coefficients can also be derived from the general KPZ scaling theory explained in \cite{spohn2012kpz}, which makes use of information about the stationary measures of the system (and indeed the above explanation shares many features with this theory). 
Thus we expect relations of the form \eqref{e.asep scaling relations}, which allows one to derive $\sigma(\alpha)$ from an expression for the law of large numbers $\mu(\alpha)$, to hold in greater generality.

\subsubsection{Main result for ASEP}
Our main result for colored ASEP is proven in  Section~\ref{s.convergence to Airy sheet}.

\begin{theorem}[Airy sheet convergence for ASEP]\label{t.asep airy sheet}
Fix any asymmetry $q \in [0,1)$ and any velocity $\alpha\in(-1,1)$ in the rarefaction fan. Then, as $\varepsilon\to 0$, $\smash{\S^{\mrm{ASEP},\varepsilon} \stackrel{d}{\to} \S}$ weakly in $\mc C(\R^2,\R)$ under the topology of uniform convergence on compact sets.
\end{theorem}

In the rest of Section~\ref{s.intro.colored asep}, we give various consequences of Theorem~\ref{t.asep airy sheet} (some of which are obtained on combining with existing results in the literature). This requires us to first introduce the basic coupling.

\subsubsection{The basic coupling}\label{s.asep basic coupling}

The graphical construction of ASEP provided in Section~\ref{s.asep dynamics} provides a natural way to couple together the evolution of uncolored ASEP from all choices of initial conditions---this is also called the \emph{basic coupling}.

Consider countably many initial conditions $\eta_0^{(1)}, \eta_0^{(2)}, \ldots \in \{0,1\}^\Z$, where $\eta_0^{(j)}(i)$ represents the state at location $i$ in the $j$\th initial condition, with $0$ for a hole and $1$ for a particle. To define the basic coupling, as in Section~\ref{s.asep dynamics}, we fix a family of independent Poisson clocks $\bm \xi = \{\xi_x^{\mrm{L}}, \xi_x^{\mrm{R}}\}_{x\in\Z}$. Then we evolve $\smash{\eta_0^{(j)}}$ for each $j$ according to the dynamics described in Section~\ref{s.asep dynamics}, where we use the same family of clocks $\bm \xi$ for all $j\in\N$. For example, if $\xi_x^{\mrm{L}}$ rings at time $t$, then the particle at location $x$ in the configuration $\smash{\eta^{(j)}_{t^-}}$ attempts a jump to the left for each $j\in\N$.

We observe that this coupling recovers the colored coupling in the case that the set of particle locations in $\smash{\eta_0^{(j+1)}}$ is a subset of those in $\smash{\eta_0^{(j)}}$ for each $j$, i.e., $\smash{\eta_0^{(j+1)}(x) = 1}$ implies $\smash{\eta_0^{(j)}(x)=1}$; in colored ASEP, this corresponds to the initial condition with the particle at site $x$ having color $\smash{\min\{j\geq 0:\smash{\eta^{(j+1)}_0(x) = 0}\}} = \smash{\#\{j: \eta_0^{(j)}(x)=1\}}$. Of course, the basic coupling itself is more general since it allows the coupling of arbitrary initial conditions.

We also note that we may consider the same evolution, using the fixed Poisson clocks, of an initial condition started at time $s>0$, i.e., we set the configuration at time $s$ to be some $\eta$ and apply the dynamics from Section~\ref{s.asep dynamics} from there. In this way the basic coupling gives a coupling of initial conditions started at different times as well.

Using the basic coupling, we can define a four parameter process extending the ASEP sheet, which we call the ASEP landscape; here, the two additional parameters arise from allowing the time at which the colored height function is evaluated, as well as the time at which the evolution from the step initial condition is begun, to vary. We first define the evolution of the uncolored ASEP height function from an arbitrary initial condition or an initial height function; this definition is more general than \eqref{e.intro ASEP height function} in that it allows initial particle configurations which are not right-finite, i.e., it is possible to have infinitely many particles initially to the right of zero.

\begin{definition}[Bernoulli path]\label{d.bernoulli path}
Let $\Lambda\subseteq \Z$ be a (possibly infinite) interval. We say a function $\gamma:\Lambda\to\Z$ is a \emph{Bernoulli path} if $\gamma(x+1) - \gamma(x)\in\{0,-1\}$ for all $x$ with $x, x+1\in\Lambda$.
\end{definition}

 Let $h_0:\Z\to\Z$ be a Bernoulli path. We define an associated configuration $\eta_{h_0}$ by
\begin{align}\label{e.height to config}
\eta_{h_0}(y) := h_0(y-1) - h_0(y);
\end{align}
see Figure~\ref{f.general height function}. Note that, in the case that $h_0(y) = 0$ for all large enough $y$, $h_0$ is related to $\eta_{h_0}$ by $h_0(y) = \sum_{i=y+1}^\infty \eta_{h_0}(i)$; this is exactly the case of a right-finite initial condition.

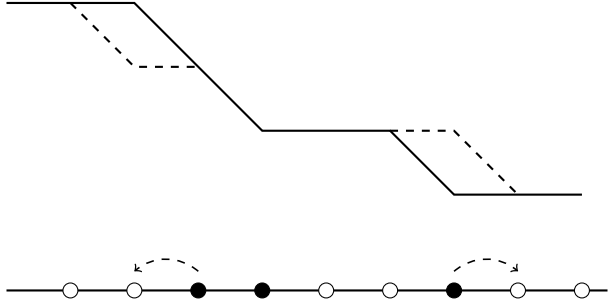
\begin{figure}
\begin{tikzpicture}[scale=0.85]
\draw[thick] (0,0) -- ++(2,0) -- ++(2,-2) -- ++(1,0) -- ++ (1,0) -- ++(1,-1) -- ++(2,0);

\draw[thick] (0,-4.5) -- ++(9.4,0);

\foreach \thecolor [count=\x] in {white, white, black, black, white, white, black, white, white}
\node[circle, fill=\thecolor, inner sep = 2pt, draw=black] at (\x,-4.5) {};

\draw[->, dashed, semithick] (7, -4.2) to[out=40, in=140] (8,-4.2);

\draw[->, dashed, semithick] (3, -4.2) to[out=140, in=40] (2,-4.2);

\draw[thick, dashed] (1,0) -- ++(1,-1) -- ++(1,0);
\draw[thick, dashed] (6,-2) -- ++(1,0) -- ++(1,-1);
\end{tikzpicture}
\caption{A depiction of a Bernoulli path $h_0$ and the associated particle configuration $\eta_{h_0}$. The dashed portions of the path illustrate the definition of the new height function, under the corresponding displayed particle movements.}\label{f.general height function}
\end{figure}

For $h_0:\Z\to\Z$ a Bernoulli path and $s\geq 0$, define the (uncolored) height function $(y,t)\mapsto h^{\mrm{ASEP}}(h_0, s; y,t)$ for $t>s$ and $y\in\Z$ as follows (see also Figure~\ref{f.general height function}). For a function $f$ with left limits, we use the notation $f(r^-) := \lim_{s\shortuparrow r} f(s)$; note that this does not require $f(r)$ to be defined. Start ASEP under basic coupling at time $s$ with initial particle configuration $\eta_{h_0}$. Whenever a particle jumps from $y$ to $y+1$ at time $r>s$, define $h^{\mrm{ASEP}}(h_0, s; y,r) = h^{\mrm{ASEP}}(h_0,s; y,r^-) + 1$ (and $h^{\mrm{ASEP}}(h_0, s; x,r) = h^{\mrm{ASEP}}(h_0, s; x,r^-)$ for all $x\neq y$), and whenever a particle jumps from $y$ to $y-1$ at time $r$, define $h^{\mrm{ASEP}}(h_0, s; y-1, r) = h^{\mrm{ASEP}}(h_0, s; y-1,r^-)-1$ (and $h^{\mrm{ASEP}}(h_0, s; x,r) = h^{\mrm{ASEP}}(h_0, s; x,r^-)$ for all $x\neq y-1$).
The utility of the definition of $h^{\mrm{ASEP}}(h_0, s; y,t)$ is that, while the ASEP dynamics only sees the discrete derivative of $h_0$ (through $\eta_{h_0}$), here we keep track of the constant shift in $h_0$ as well.

For multiple initial conditions $(\smash{h_0^{(j)}})_{j=1}^\infty$ and times $s_j>0$, we couple the height functions $(y,t)\mapsto h^{\mrm{ASEP}}(\smash{h_0^{(j)}}, s_j; y,t)$ using the basic coupling for the underlying ASEP dynamics.

\subsubsection{The ASEP landscape} Now we may define the unscaled version of the ASEP landscape.
For $x,y\in\Z$ with $x<y$ and $0<s<t$, we overload the notation and define
\begin{align*}
h^{\mrm{ASEP}}(x, s; y,t) &:= h^{\mrm{ASEP}}\bigl(h_{0,x},s; y,t\bigr),
\end{align*}
coupled via the basic coupling, where
$$h_{0,x}(z) := (x-z)\one_{z\leq x}.$$
In words, we start the ASEP under basic coupling at time $s$ with a step initial condition at $x$, and then evaluate the height function at location $y$ at time $t$.

It follows immediately that this definition generalizes the definition of $h^{\mrm{ASEP}}(x,0;y,t)$ from \eqref{e.intro ASEP height function} as a process in $x$, $y$, and $t$. We also note that $h^{\mrm{ASEP}}(\bm\cdot, s; \bm\cdot, \bm\cdot)$ corresponds to starting colored ASEP with packed initial condition at time $s$, coupled (for different values of $s$) to the same environment of Poisson clocks.

We now define the ASEP landscape and the scaled general initial condition height function.

\begin{definition}[ASEP landscape]\label{d.asep landscape}
Recall $\mu(\alpha)$, $\sigma(\alpha)$, and $\beta(\alpha)$ from \eqref{e.mu sigma ASEP} and \eqref{e.nu ASEP}. For $\varepsilon>0$, the \emph{ASEP landscape} is defined by
\begin{align*}
\MoveEqLeft[4]
\mc L^{\mrm{ASEP}, \varepsilon}(x,s;y,t) := \sigma(\alpha)^{-1}\varepsilon^{1/3}\Bigl(\mu(\alpha)(t-s)\cdot 2\varepsilon^{-1} + \mu'(\alpha)\beta(\alpha)(y-x)\varepsilon^{-2/3}\\
&- h^{\mrm{ASEP}}\bigl(\beta(\alpha)x\varepsilon^{-2/3}, 2\gamma^{-1}\varepsilon^{-1}s; 2\alpha(t-s) \varepsilon^{-1} + \beta(\alpha)y\varepsilon^{-2/3}, 2\gamma^{-1}\varepsilon^{-1}t\bigr)\Bigr).
\end{align*}
As in the case of $\S^{\mrm{ASEP},\varepsilon}$ in \eqref{e.rescaled asep definition}, this definition is for those $x$, $y$ such that the arguments of $\smash{h^{\mrm{ASEP}}}$ are integers, and the value at all other $x$, $y$ are determined by linear interpolation.
\end{definition}

\begin{definition}\label{d.asep general initial condition}
For general initial condition $h_0$, we define a scaled version $\mf h_0^\varepsilon$ by
\begin{align*}
\mf h^{\varepsilon}_0(y) := \sigma(\alpha)^{-1}\varepsilon^{1/3}\Bigl(\mu'(\alpha)\beta(\alpha)y\varepsilon^{-2/3}- h_0\bigl(\beta(\alpha)y\varepsilon^{-2/3}\bigr)\Bigr),
\end{align*}
for $y$ such that the argument of $h_0$ is an integer, and by linear interpolation elsewhere. Similarly, we define a scaled version of $h^{\mrm{ASEP}}(h_0, 0; y,t)$ by
\begin{align*}
\MoveEqLeft[12]
\mf h^{\mrm{ASEP}, \varepsilon}(h_0; y,t) := \sigma(\alpha)^{-1}\varepsilon^{1/3}\Bigl(\mu(\alpha)t\cdot2\varepsilon^{-1} + \mu'(\alpha)\beta(\alpha)y\varepsilon^{-2/3}\\
&- h^{\mrm{ASEP}}\bigl(h_0, 0; 2\alpha t\varepsilon^{-1} + \beta(\alpha)y\varepsilon^{-2/3}, 2\gamma^{-1}\varepsilon^{-1}t\bigr)\Bigr),
\end{align*}
for $y$ such that the argument of $h^{\mrm{ASEP}}$ is an integer, and by linear interpolation elsewhere.
\end{definition}

The next result gives the scaling limit of the ASEP landscape, as well as of the height functions under general initial conditions coupled via the basic coupling. It consists of three parts, and the third part (concerning convergence to the KPZ fixed point) is proven not in this article, but in a paper \cite{aggarwal2024kpz} of the authors using Theorem~\ref{t.asep airy sheet} as an important ingredient. An earlier version of this article gave the proofs of all three parts (the third for $\alpha=0$ only), but the proof of the last part relied on the main result of \cite{quastel2022convergence}, which was recently withdrawn. \cite[Theorem 2.10]{aggarwal2024kpz} gives an independent proof using as inputs from this paper only the first part of Theorem~\ref{t.asep general initial condition}, which is thus the only logical dependency of the third part on the rest of the present article. 

For a set $\mc T$, we define $\mc T^2_{<} := \{(s,t) \in\mc T^2: s<t\}$. Below, for a topological space $\mc X$ and a set $A$, $\mc C(\mc X, \R)^A$ is endowed the product topology (across $A$) of the topology on $\mc C(\mc X, \R)$ of uniform convergence on compact sets.

\begin{theorem}[Directed landscape and KPZ fixed point convergence for ASEP]\label{t.asep general initial condition}
Fix $q\in [0,1)$, velocity $\alpha\in(-1,1)$, and a countable set $\mc T\subseteq [0,\infty)$. Then, the following limits hold as $\varepsilon\to 0$.

\begin{enumerate}
  \item (Directed landscape). In $\smash{\mc C(\R^2,\R)^{\mc T^2_<}}$,
  $$\smash{\bigl(\mc L^{\mrm{ASEP},\varepsilon}(\bm\cdot, s; \bm\cdot, t)\bigr)_{(s,t)\in\mc T^2_<}\xrightarrow{d} \bigl(\mc L(\bm\cdot, s; \bm\cdot, t)\bigr)_{(s,t)\in\mc T^2_<}}.$$

  \item As a special case, for any $k\in\N$ and $x_k< \ldots <x_1$, in $\smash{\mc C(\R,\R)^{\intint{1,k}\times\mc T}}$,
  $$\smash{\bigl(\mc L^{\mrm{ASEP},\varepsilon}(x_i, 0; \bm\cdot, t)\bigr)_{i\in\intint{1,k}, t\in\mc T}}\xrightarrow{d} \bigl(\mc L(x_i, 0; \bm\cdot, t)\bigr)_{i\in\intint{1,k}, t\in\mc T}.$$

  \item (Coupled KPZ fixed point). For $k\in\N$ and each $i\in \{1, \ldots, k\}$, suppose $\smash{h_0^{(i),\varepsilon}}:\Z\to\Z$ is a sequence of (possibly random, and independent of the ASEP dynamics) Bernoulli paths and $\smash{\mf h_0^{(i)}:\R\to\R}$ is a (possibly random, and independent of $\mc L$) continuous function. Let $\mf h_0^{(i), \varepsilon}$ be as in Definition~\ref{d.asep general initial condition} and assume they satisfy $\inf_{[-C,C]}\smash{\h_0^{(i),\varepsilon}} \geq -C$ and $\smash{\mf h_0^{(i), \varepsilon}}(x)\leq C(1+|x|)$ for some $C<\infty$ almost surely for all $x\in\R$, $i\in\intint{1,k}$, and $\varepsilon>0$. Suppose $x\mapsto \smash{\mf h_0^{(i),\varepsilon}(x)}$ converges to $\smash{\mf h_0^{(i)}}$ as $\varepsilon\to0$ for $i=1, \ldots, k$  jointly in distribution under the topology of uniform convergence on compact sets. Then, in $\smash{\mc C(\R,\R)^{\intint{1,k}\times\mc T}}$,
  $$\bigl(\mf h^{\mrm{ASEP}, \varepsilon}(h_0^{(i),\varepsilon}; \bm\cdot,t)\bigr)_{{i\in\intint{1,k}, t\in\mc T}} \xrightarrow{d} \smash{\bigl(\mf h(\mf h^{(i)}_0}; \bm\cdot, t)\bigr)_{i\in\intint{1,k}, t\in\mc T},$$
   where the lefthand side is coupled via the basic coupling and the righthand side is coupled as in~\eqref{e.kpz fixed point} under the same directed landscape $\mc L$.
\end{enumerate}

\end{theorem}


Note that the convergence to the directed landscape and KPZ fixed point stated here is for times in a fixed countable set, and not uniform convergence as a continuous function for all times in an interval. Obtaining the latter form of convergence would require proving certain tightness estimates on ASEP as a temporal process that we do not pursue here.

\begin{remark}[Scaling limit of ASEP with finitely many colors]\label{r.asep finitely many colors}
By color merging, Theorem~\ref{t.asep general initial condition} (2) is equivalent to the following statement on the space-time scaling limit of colored ASEP with finitely many colors. Let $k\in\N$, $\alpha\in(-1,1)$, and fix real numbers $x_k <  \ldots < x_1$. Consider colored ASEP with initial condition of holes (or particles of color zero) at locations $\llbracket\floor{\beta(\alpha)x_1\varepsilon^{-2/3}}+1, \infty \rrparen$; particles of color $k$ at locations $\llparen {-}\infty,\floor{\beta(\alpha)x_k\varepsilon^{-2/3}}\rrbracket$; and particles of color $i$ at all locations in $\intint{\floor{\beta(\alpha)x_{i+1}\varepsilon^{-2/3}}+1, \floor{\beta(\alpha)x_i\varepsilon^{-2/3}}}$ for $i=1, \ldots, k-1$.
By color merging, $(y,t)\mapsto(\mc L^{\mrm{ASEP},\varepsilon}(x_1, 0; y, t),  \ldots, \mc L^{\mrm{ASEP},\varepsilon}(x_k,0 ; y,t))$ is exactly a centered and rescaled version of the height functions associated to ASEP with $k$ different colors, evaluated around the location $2\alpha t\varepsilon^{-1}$ and, by Theorem~\ref{t.asep general initial condition} (2), they converge to the directed landscape.
\end{remark}

\subsubsection{Convergence to stationary horizon}\label{s.asep stationary horizon}
Introduced in \cite{busani2021diffusive}, the \emph{stationary horizon} $G = (G_{\xi})_{\xi \in \mathbb{R}}$ is the unique process  (under certain slope conditions) \cite{busani2022stationary} taking values in $\mathcal{C} (\mathbb{R}, \mathbb{R})$ of jointly invariant initial conditions for the KPZ fixed point under the directed landscape coupling \eqref{e.kpz fixed point}. A precise definition of the process can be found, for example, in \cite[Appendix D]{busani2022scaling}. It is expected to be the universal scaling limit of multi-type invariant distributions in the KPZ universality class.

By combining Theorem~\ref{t.asep general initial condition} with recent results of \cite{busani2023scaling}, we can obtain weak convergence to the stationary horizon for the (rescaled) height function corresponding to the stationary colored ASEP. In the case of colored TASEP, this was proven (under a stronger topology) in \cite{busani2022scaling}, using detailed information of the structure of the multi-species invariant measure \cite{ferrari2007stationary}. For colored ASEP, these stationary measures have been the topic of several recent studies \cite{prolhac2009matrix,Cantini_2015,martin2020stationary,FMQ,aggarwal2023colored}, and their structure is considerably more involved. While a direct proof of their convergence to the stationary horizon might take some effort, here we deduce it as a quick consequence of a different convergence statement (Theorem~\ref{t.asep general initial condition}).

Now we give the setup for stating this stationary horizon convergence.
First, given $k\in\N$ and a sequence of densities $\bm \rho=(\rho_1, \ldots, \rho_k) \in (0,1)^k$ such that $\sum_{i=1}^k\rho_i \leq 1$, there exists a unique translation-invariant stationary distribution $\mu_{\bm \rho}$ for colored ASEP with colors $1, \ldots, k$ such that the marginal distribution of the particle configuration obtained by considering only particles of colors $j$ through $k$ (i.e., merging colors $j$ through $k$ and regarding particles of lower colors as holes)  is product Bernoulli with density $\sum_{i=j}^k\rho_i$ \cite[Theorem 3.2]{martin2020stationary}.

In order to obtain the stationary horizon scaling limit, we need the component densities of $\bm \rho$ to be perturbed on the diffusive scale around a fixed density $\rho \in (0,1)$. More precisely, let $\varepsilon$ be a scaling parameter and $\bm \xi = (\xi_1, \ldots, \xi_k) \in \R^k$ with $\xi_1 <  \ldots <\xi_k$. Consider $\bm\rho = \bm \rho^\varepsilon_{\bm \xi} = (\rho_1, \ldots, \rho_k)$ with 
$$\rho_k = \rho + \left(\tfrac{1}{2}\rho(1-\rho)\right)^{1/3}\xi_1\varepsilon^{1/3} \quad\text{ and }\quad \rho_{k-i} = \left(\tfrac{1}{2}\rho(1-\rho)\right)^{1/3}(\xi_{i+1}-\xi_{i})\varepsilon^{1/3}$$
for $i=1, \ldots, k-1$. Let $\smash{\eta\in\intint{0,k}^{\smash\Z}}$ (0 for holes) be a particle configuration sampled according to $\mu_{\bm\rho}$. Define the initial conditions $h_0^{(j), \varepsilon}$ for $j=1, \ldots, k$~by
\begin{align*}
h_0^{(j), \varepsilon}(0) = 0 \quad\text{and}\quad h_0^{(j), \varepsilon}(x-1) - h_0^{(j), \varepsilon}(x) = \one_{\eta(x) \geq k-j+1}.
\end{align*}
By the stationarity of $\eta$, $(h_0^{(j), \varepsilon})$ satisfies $h^{\mrm{ASEP}}(h_0^{(j), \varepsilon}, 0; \bm\cdot, t) - h^{\mrm{ASEP}}(h_0^{(j), \varepsilon}, 0; 0, t) \stackrel{d}{=} h_0^{(j), \varepsilon}(\bm\cdot)$ jointly over $j=1, \ldots, k$. Further, for each $j$, $\smash{h_0^{(j), \varepsilon}}$ is marginally distributed as a two-sided Bernoulli random walk with drift $-\rho-(\tfrac{1}{2}\rho(1-\rho))^{1/3}\xi_j\varepsilon^{1/3}$; here, the factor in front of $\xi_j$ ensures that the scaling limit of $\smash{h_0^{(j),\varepsilon}}$ is a Brownian motion of drift $2\xi_j$. Below we regard $\smash{h_0^{(j),\varepsilon}}$ as a continuous function on $\R$ by linear interpolation.

\begin{corollary}[Stationary horizon convergence for ASEP]\label{c.SH convergence}
Fix $\rho\in(0,1)$,  $k\in\N$, and $\xi_1< \ldots <\xi_k$. Let $\smash{h_0^{(j),\varepsilon}}$ be as above, and let $\mf h_0^{\varepsilon}$ be as in Definition~\ref{d.asep general initial condition} with $\alpha=1-2\rho$, i.e., such that $\rho = |\mu'(\alpha)|$. As $\varepsilon\to0$, in the topology of uniform convergence on compact sets,
\begin{equation}\label{e.SH convergence}
\left(\mf h_0^\varepsilon\right)_{j=1}^k \stackrel{d}{\to} (G_{\xi_j})_{j=1}^k.
\end{equation}
\end{corollary}

\begin{proof}
This is an immediate consequence of Theorem~\ref{t.asep general initial condition} (3) combined with \cite[Theorem~2.4]{busani2023scaling}. Assumption 1 of the latter is satisfied by the fact that the $i$\th component of the lefthand side of \eqref{e.SH convergence} converges to a Brownian motion of drift $2\xi_i$, for each $i$; Assumption 2 is merely that we consider the jointly invariant distribution; and Assumption 3 is satisfied due to Theorem~\ref{t.asep general initial condition}~(3).
\end{proof}

\subsubsection{Decoupling for colored ASEP and non-linear fluctuating hydrodynamics}

Over the past decade, a set of predictions for the asymptotic behavior of multi-component KPZ models has been developed under the name of non-linear fluctuating hydrodynamics---see for instance \cite{Spohn2014,Ferrari2013,mendl2013dynamic,Spohn2015,spohn2016fluctuating}. Among those predictions is the decoupling of height function fluctuations and stationary two-point functions between two components started in initial data with different characteristic lines. 
Here we consider the colored ASEP as a multi-component KPZ model and verify such decoupling, focusing on the two-component case for simplicity. Specifically, we run the colored ASEP for time of order $\varepsilon^{-1}$ and examine its two colored height functions when their associated characteristic slopes differ by order $\beta \varepsilon^{1/3}$. We show that they become asymptotically independent, after first sending $\varepsilon \rightarrow 0$ and then $\beta \rightarrow \infty$. The advantage in this double scaling limit is that, after the first KPZ limit $\varepsilon \rightarrow 0$, we can relate the height functions (and stationary two-point functions) to the directed landscape; we then prove the decoupling in the $\beta \rightarrow \infty$ limit by geometric considerations and decorrelation properties proven in \cite{NTM} of the latter. Our results show that the asymptotic decoupling of the colored height functions already manifests itself if the associated characteristic slopes differ by an amount growing barely larger than $\varepsilon^{1/3}$. This suggests, though does not directly imply as stated, that this decoupling should also hold when the characteristic slopes are uniformly bounded away from each other (formally corresponding to taking $\beta \rightarrow \infty$ simultaneously as $\varepsilon \rightarrow 0$, at rate $\varepsilon^{-1/3}$); we expect this could be proved by using the proof machinery developed here but do not pursue that.

In order to state our first decoupling result (about height functions), let $\slope\in\R$. Let $\h_0^{(1)}, \h_0^{(2)} : \R\to\R$ be continuous (possibly random) and satisfy the following slope property: there exists $R<\infty$ such that, almost surely, for all $y\in\R$,
\begin{equation}\label{e.uniform slope condition}
\begin{split}
|\h_0^{(1)}(y)| &\leq |y| + R|\slope|\\
|\h_0^{(2)}(y) - 2\slope y| &\leq |y| + R|\slope|.
\end{split}
\end{equation}
The above condition will be assumed on the limit of the ASEP initial condition height functions we work with. We note that the condition does not specify the slopes precisely and only up to a window of width $2$, but this will not be important for our purposes.

The following corollary gives the asymptotic independence of the ASEP height functions as the scaling parameter $\varepsilon\to 0$ followed by taking the slope parameter $\slope\to\infty$. The first part concerns when the height functions are evaluated at the same point, while in the second they are evaluated along their respective (distinct) characteristic lines emanating from the origin (note that $\slope$ controls the slope of the characteristic, so that, at time $2\gamma^{-1}\varepsilon^{-1}t$, the characteristic of the initial condition converging to $\smash{\h_0^{(1)}}$ is near $0$ and that of the initial condition converging to $\smash{\h_0^{(2)}}$ is near $\floor{2\slope\varepsilon^{-2/3}t}$). The corollary will be proven in Appendix~\ref{s.proofs of asymptotic independence}.

\begin{corollary}[Asymptotic independence of colored ASEP height functions]\label{p.asymptotic independence ASEP}
Fix $R>0$, $T\geq 1$, and $\mc T\subseteq [T^{-1}, T]$ finite. For any $\beta\in \R$, $\varepsilon>0$, and $i\in\{1,2\}$, suppose $h_0^{(i),\varepsilon} : \Z\to \Z$ is a Bernoulli path such that, as $\varepsilon\to 0$,
$$-2\varepsilon^{1/3}\left(h_0^{(i),\varepsilon}(\floor{2x\varepsilon^{-2/3}}) + \floor{x\varepsilon^{-2/3}}\right) \stackrel{d}{\to} \h^{(i)}_0$$
in the topology of uniform convergence on compact sets, where $\h^{(1)}_0, \h^{(2)}_0:\R\to\R$ are independent (in particular, they may be deterministic), continuous, and satisfy \eqref{e.uniform slope condition} with $\beta$ and $R$ almost surely. For $i\in\{1,2\}$ and $\varepsilon>0$, define $\h^{(i),\varepsilon}:\R\times[0,\infty)\to\R$ by (recall Definition~\ref{d.asep general initial condition})
$$\h^{(i), \varepsilon}(x,t) = -2\varepsilon^{1/3}\left(h^{\mrm{ASEP}}\bigl(h^{(i),\varepsilon}_0,0; \floor{2x\varepsilon^{-2/3}},2\gamma^{-1}\varepsilon^{-1}t\bigr) + \floor{x\varepsilon^{-2/3}} - \tfrac{1}{2}t\varepsilon^{-1}\right),$$
coupled by the basic coupling across $i\in\{1,2\}$. There exist constants $C,\delta>0$ depending on $R$ and $T$ such that the following holds. (1) For every $\varepsilon>0$, there exist independent processes $\tilde \h^{(1),\varepsilon}, \tilde \h^{(2), \varepsilon}:\R\times[T^{-1},T]\to\R$ and a coupling of $(\h^{(1),\varepsilon}, \h^{(2), \varepsilon})$ with $(\tilde \h^{(1),\varepsilon}, \tilde \h^{(2), \varepsilon})$ such that, for $i\in\{1,2\}$,
\begin{align}\label{e.asymptotic independence asep}
\lim_{|\slope|\to\infty}\,\lim_{\varepsilon\to0}\,\P\left(\max_{s\in\mc T}\sup_{|x|\leq \delta|\slope|^{1/2}} \left|\h^{(i), \varepsilon}(x,s) - \tilde \h^{(i),\varepsilon}(x,s)\right| \geq C|\slope|^{-1/12}\log|\slope|\right) = 0.
\end{align}
(2) For every $\varepsilon>0$, there exist independent processes $\tilde \h^{(1),\varepsilon}_{\mrm{char}}, \tilde \h^{(2),\varepsilon}_{\mrm{char}}:\R\times[T^{-1},T]\to\R$ and a coupling of $(\h^{(1),\varepsilon}, \h^{(2),\varepsilon})$ with $(\tilde \h^{(1),\varepsilon}_{\mrm{char}}, \tilde \h^{(2),\varepsilon}_{\mrm{char}})$ such that,
\begin{equation}\label{e.asymptotic independence along characteristic asep}
\begin{split}
\lim_{|\slope|\to\infty}\,\lim_{\varepsilon\to 0}\,\P\left(\max_{s\in\mc T}\sup_{|x|\leq \delta|\slope|^{1/12} t}\left|\h^{(1), \varepsilon}(x,s) - \tilde \h^{(1), \varepsilon}_{\mrm{char}}(x,s)\right| \geq C|\slope|^{-1/12}\log|\slope|\right) &= 0,\\
\lim_{|\slope|\to\infty}\,\lim_{\varepsilon\to 0}\,\P\left(\max_{s\in\mc T}\sup_{|x+\slope s|\leq \delta|\slope|^{1/2} t}\left|\h^{(2), \varepsilon}(x,s) - \tilde \h^{(2), \varepsilon}_{\mrm{char}}(x,s)\right| \geq C|\slope|^{-1/12}\log|\slope|\right) &= 0,
\end{split}
\end{equation}
i.e., asymptotic independence holds as $|\slope|\to\infty$ when evaluating at locations near the respective characteristic lines emanating from the common point $0$.
\end{corollary}

A version of Corollary~\ref{p.asymptotic independence ASEP} can also be formulated and proved for more than two colors in an entirely similar way, though we do not do so here. Also, though the statement assumes the initial conditions satisfy \eqref{e.uniform slope condition} almost surely for a fixed value of $R$, one can of course apply the result to independent random initial conditions on the event that they satisfy \eqref{e.uniform slope condition} with that $R$, and bound the probability that they violate \eqref{e.uniform slope condition} separately. 
One can also weaken the assumption of independence of the limiting initial conditions $\smash{\h_0^{(1)}}$ and $\smash{\h_0^{(2)}}$ to a suitably quantified form of asymptotic independence as $\slope\to\infty$, e.g., of the form of a coupling with independent processes with quantified error as in \eqref{e.asymptotic independence asep}; see Remark~\ref{r.weakening independence assumption}. 
For instance, such a condition would be satisfied by the multi-color limiting stationary initial condition, the stationary horizon (see Lemma~\ref{l.sh representation}).

Corollary~\ref{p.asymptotic independence ASEP} is really a statement about the limiting objects, KPZ fixed points coupled via the directed landscape. Indeed, Corollary~\ref{p.asymptotic independence ASEP} will be proved via an analogous statement for the limiting objects (Proposition~\ref{p.asymptotic independence of kpz fixed point}).

Next we turn to the statement on the decorrelation of the space-time covariance, or two-point function, under stationary initial data for the colored model.
For each $\slope>0$ and $\varepsilon>0$, let $\tilde\eta_0\in\{0,1,2\}^{\Z}$ be sampled from the two-color stationary distribution $\mu_{\bm\rho}$ of ASEP with density vector $\bm\rho = (\frac{1}{2}\slope \varepsilon^{1/3}, \frac{1}{2})$ as described in Section~\ref{s.asep stationary horizon}. Thus particles of color $2$ form an independent Bernoulli($\frac{1}{2}$) sequence, while particles of color $1$ and $2$ combined form an independent Bernoulli($\frac{1}{2}+\frac{1}{2}\slope \varepsilon^{1/3}$) sequence. The particle configuration at time $t$ after evolution under the colored ASEP dynamics is denoted $\tilde\eta_t$. For $t\geq 0$, we define $\smash{\eta_t^{(1)}},\smash{\eta_t^{(2)}}\in\{0,1\}^\Z$ to be the particles of color at least 1 and particles of color 2, respectively, i.e., by
$$\smash{\eta_t^{(1)}}(x) = \one_{\tilde\eta_t(x) \geq 1} \quad\text{and}\quad\smash{\eta_t^{(2)}(x)} = \one_{\tilde \eta_t(x) = 2}.$$

Then the two-point function (or the space-time covariance matrix) $\bm S^{\slope, \varepsilon} = (S^{\slope,\varepsilon}_{kl})_{k,\ell \in \{1,2\}}$ is defined by 
\begin{align*}
S^{\slope,\varepsilon}_{k\ell}(t,x) = \Cov(\eta^{(k)}_0(0), \eta^{(\ell)}_t(x)).
\end{align*}
In the below, we will say that, for fixed $k,\ell\in\{1,2\}$ and a function $A:\R\to\R$,
$$\lim_{\slope\to\infty} \lim_{\varepsilon\to 0}\varepsilon^{-2/3}S^{\slope, \varepsilon}_{k\ell}(2\gamma^{-1}\varepsilon^{-1}, \floor{2x\varepsilon^{-2/3}}) = A(x)$$
in the sense of integration against smooth, compactly supported test functions if for any $f:\R\to\R$ which is smooth and compactly supported it holds that
\begin{align*}
\lim_{\slope\to\infty}\lim_{\varepsilon\to 0} \varepsilon^{-2/3} \int_{\R} S^{\slope,\varepsilon}_{k\ell}(2\gamma^{-1}\varepsilon^{-1}, \floor{2x\varepsilon^{-2/3}})f(x)\,\diff x = \int_\R A(x)f(x)\,\diff x.
\end{align*}
Equivalently (by discretizing the integral over a mesh of width $\frac{1}{2}\varepsilon^{2/3}$), for all smooth and compactly supported $f:\R\to\R$, with $f_\varepsilon:\Z\to\R$ defined by $f_\varepsilon(x) = f(\varepsilon^{2/3}x)$, 
\begin{align*}
\lim_{\slope\to\infty}\lim_{\varepsilon\to 0} \frac{1}{2}\sum_{i\in\Z} S^{\slope,\varepsilon}_{k\ell}(2\gamma^{-1}\varepsilon^{-1}, i)f_\varepsilon(\tfrac{1}{2}i) = \int_\R A(x)f(x)\,\diff x.
\end{align*}

The following is the statement on the limit of the space-time covariance. It will be proven in Appendix~\ref{s.proofs of asymptotic independence}. In the below, $g_{\mrm{BR}}(x)$ the variance of the Baik-Rains distribution $F_{\mrm{BR}; x}$ as defined, for instance, in \cite[eq. (1.20)]{ferrari2006scaling} (and originally in \cite[Definition 3]{baik2000limiting} under a slightly different scaling and in a different form).

\begin{corollary}[Scaling limit of stationary colored ASEP two-point function]\label{p.covariance matrix}
Let $x\in\R$ be fixed. It holds that, in the sense of integration against smooth compactly supported test functions,
\begin{align}\label{e.on diagonal covvariance}
\lim_{\slope\to\infty} \lim_{\varepsilon\to 0} \varepsilon^{-2/3}\bm S^{\slope, \varepsilon}\bigl(2\gamma^{-1}\varepsilon^{-1}, \floor{2x\varepsilon^{-2/3}}\bigr)
= \begin{bmatrix}
0 & 0\\
0 & \frac{1}{32}g''_{\mrm{BR}}(x).
\end{bmatrix}
\end{align}
Further, for each $\slope\in\R$ and $\delta>0$ and when $\{k,\ell\} = \{1,2\}$, if $x = x(\slope)\in\R$ is such that $|x| \leq |\slope|^{2-\delta}$, then it holds that, in the sense of integration against smooth compactly supported test functions,
\begin{equation}\label{e.space-time covariance vanishes}
\begin{split}
\lim_{\slope\to\infty}\lim_{\varepsilon\to0} \varepsilon^{-2/3}S^{\slope, \varepsilon}_{kl}(2\gamma^{-1}\varepsilon^{-1}, \floor{2x\varepsilon^{-2/3}}) &= 0.
\end{split}
\end{equation}
In particular, \eqref{e.space-time covariance vanishes} holds when $x = -\slope$, i.e, $x$ is the location on the characteristic line at time $2\gamma^{-1}\varepsilon^{-1}$ associated to an initial density of $\frac{1}{2}+\frac{1}{2}\slope \varepsilon^{1/3}$.
\end{corollary}

In the above, the statement \eqref{e.on diagonal covvariance} that the single-species two-point function (appropriately scaled) has a limit governed by KPZ statistics when evaluated in the vicinity of the associated characteristic line (of the color 2 particles), and (under the same scaling) converges to zero away from the characteristic line, is known from prior work (following quickly from results in \cite{landon2023tail}).

Instead of evaluating the space-time covariance near the characteristic line associated to the color 2 particles, one could do the same near the characteristic line associated to the combined color 1 and color 2 particles (which together have density $\frac{1}{2}(1+\slope\varepsilon^{1/3})$). Indeed, if one evaluates the space-time covariance at $\floor{2(x-\slope) \varepsilon^{-2/3}}$ for fixed $x$, then one would obtain the limit (again $\varepsilon\to0$ first and then $\slope\to\infty$) of $\varepsilon^{-2/3}S^{\slope,\varepsilon}_{k,\ell}(2\gamma^{-1}\varepsilon^{-1}, \floor{2(x-\slope)\varepsilon^{-2/3}})$ to be $\frac{1}{32}g''_{\mrm{BR}}(x)$ for $k=\ell=1$ and $0$ for $k=\ell=2$. That this covariance converges to $0$ when $k \ne \ell$ is already included in \eqref{e.space-time covariance vanishes}.

\subsection{Colored stochastic six-vertex model}

\label{s.intro.cS6V}

The colored stochastic six-vertex model (S6V) we consider is defined on the domain $\Z_{\geq 1}\times \llbracket-N,\infty\rrparen$ for some choice of $N\in\N$. The model has two parameters: a quantization (or asymmetry) parameter $q\in[0,1)$ and a rapidity (or spectral) parameter $z\in(0,1)$, which are both fixed.

\subsubsection{Configurations and weights} \label{s.s6v configuration}
A configuration of the model is described as follows. First, an \emph{arrow configuration} is a tuple $(a, i; b, j)$ such that $i, j, a, b \in \Z\cup\{-\infty\}$ and $\{i, a\} = \{j, b\}$; see the left panel of Figure~\ref{f.colored S6V}. Under this notation, $i$, $a$, $j$, and $b$ can be interpreted as the \emph{colors}  of the arrows horizontally entering, vertically entering, horizontally exiting, and vertically exiting the vertex, respectively. Here, an arrow of color $-\infty$ may be viewed as the absence of an arrow along an edge. Thus, in each direction, at most one arrow can enter the vertex and at most one arrow can exit the vertex. The condition $\{i, a\} = \{j, b\}$ is called \emph{colored arrow conservation}, i.e., the number of incoming arrows of each color is equal to the number of outgoing arrows of the same color.

A configuration of the model consists of an assignment of arrow configurations, one to each vertex $v$ of the domain $\Z_{\geq 1}\times \llbracket-N,\infty\rrparen$; we will denote the arrow configuration at $v$ by $(a_v, i_v; b_v, j_v)$. We require that the arrow configurations are \emph{consistent}: $i_{(x,y)} = j_{(x-1,y)}$ and $a_{(x,y)} = b_{(x,y-1)}$, i.e., an arrow of color $k$ is horizontally entering at $(x,y)$ if and only if an arrow of color $k$ is horizontally exiting at $(x-1,y)$, and similarly an arrow of color $k$ is vertically entering at $(x,y)$ if and only if an arrow of color $k$ is vertically exiting at $(x,y-1)$. We will refer to the colors of the arrows entering and exiting a vertex $v$ (horizontally and vertically) as the incoming and outgoing arrow configurations at $v$ respectively, and the combined set of incoming and outgoing arrows as the arrow configuration at $v$. 
One consequence of this description is that the colored arrows form up-right paths; see the right panel of Figure~\ref{f.colored S6V}.

The \emph{boundary condition} of the model is specified as follows. Let $\sigma : \llbracket -N, N \rrbracket \to \Z \cup \{ -\infty \}$. Then the colored S6V model with boundary condition $\sigma$ is defined as the model where an arrow of color $\sigma (k)$ enters horizontally at $(1, k)$ (i.e., $i_{(1,k)} = \sigma(k)$) for $k \in \intint{-N, N}$, and an arrow of color $-\infty$ enters vertically at $(\ell, -N)$ (i.e., $a_{(\ell,-N)} = -\infty$) for $\ell \in \mathbb{N}$; see the right panel of Figure 6. Note that arrows enter only at heights in $\intint{-N,N}$, so there are only finitely many arrows in the system.

\begin{remark}\label{r.N restriction for S6V}
The restriction that we impose of no arrows entering above height $N$ is to make the upcoming definition \eqref{e.s6v height function} of the colored height function simpler (because of the just noted finiteness); we will soon restrict our attention to  the height function's behavior in $\Z_{\geq 1}\times \intint{-N,N}$, where the entrance of arrows above height $N$ would have no distributional effect anyway. Allowing arrows to enter on $\intint{-N,N}$ allows us to consider arrows of negative color without recentering, which will make our height function definition simpler, and will later allow us to couple directly with colored ASEP on all of $\Z$.
\end{remark}

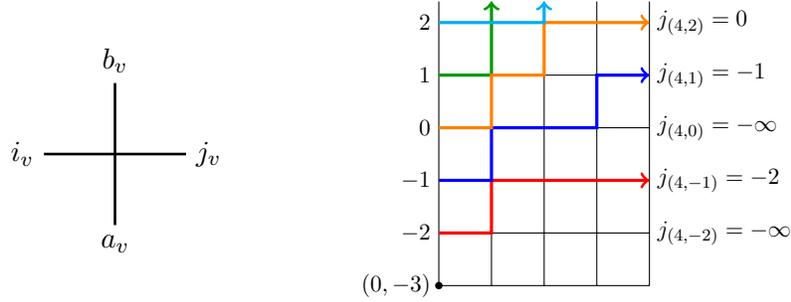
\begin{figure}[t!]
\begin{tikzpicture}[scale=0.7]

\begin{scope}[shift={(-7.5,1.5)}, scale=0.9]

\draw[line width=1pt] (0,0) -- ++(3,0);
\draw[line width=1pt] (1.5,-1.5) -- ++(0,3);

\node[anchor = east] at (0,0) {$i_v$};
\node[anchor = west] at (3,0) {$j_v$};
\node[anchor = north] at (1.5,-1.5) {$a_v$};
\node[anchor = south] at (1.5,1.5) {$b_v$};
\end{scope}

\draw (0,-1) grid (4,4.4);

\draw[red, very thick, ->] (0,0) -- ++(1,0) -- ++(0,1) -- ++(3,0);

\draw[blue, very thick, ->] (0,1) -- ++(1,0) -- ++(0,2) -- ++(1,0) -- ++(0,1) -- ++(2,0);

\draw[green!60!black, very thick, ->] (0,3) -- ++(1,0) -- ++(0,1.4);

\draw[orange, very thick, ->] (0,2) -- ++(2,0) -- ++(0,1) -- ++(2,0);

\draw[cyan, very thick, ->] (0,4) -- ++(2,0)-- ++(0,0.4);

\foreach \x in {-2,...,2}
{
  \node[anchor=east, scale=0.8] at (0,\x+2) {$\x$};
}

\node[circle, fill, inner sep=1pt] at (0,-1) {};
\node[anchor=east, scale=0.8] at (0,-1) {$(0,-3)$};

\node[anchor=west, scale=0.8] at (4,0) {$j_{(4,-2)}=-\infty$};
\node[anchor=west, scale=0.8] at (4,1) {$j_{(4,-1)}=-2$};
\node[anchor=west, scale=0.8] at (4,2) {$j_{(4,0)}=-\infty$};
\node[anchor=west, scale=0.8] at (4,3) {$j_{(4,1)}=0$};
\node[anchor=west, scale=0.8] at (4,4) {$j_{(4,2)}=-1$};

\end{tikzpicture}
\caption{Left: a depiction of the labels for the arrow configuration at a vertex $v$. Right: A sample configuration of the colored stochastic six-vertex model with $N=2$ and the packed boundary condition, i.e., $\sigma(k) = k$ for $k\in\intint{-N,N}$. The numbers on the left boundary next to an arrow is the color of the arrow. Any edges without arrows should be interpreted as an arrow with color $-\infty$, i.e., a color lower than the colors of all the other arrows in the system.}\label{f.colored S6V}
\end{figure}

\begin{figure}[b]
\begin{tikzpicture}
\newcommand{\thedim}{4.2}
\begin{scope}[scale=0.65]

      \draw[step=\thedim] (0,0) grid (5*\thedim, \thedim);

      \draw[->, line width=1.2pt, red] (0.5*\thedim, 0.25*\thedim) -- ++(0, 0.5*\thedim);
      \draw[->, line width=1.2pt, red] (0.25*\thedim, 0.5*\thedim) -- ++(0.5*\thedim, 0);
      \node[scale=0.8, anchor=west] at (0.75*\thedim+0.1, 0.5*\thedim) {$i$};
      \node[scale=0.8, anchor=south] at (0.5*\thedim, 0.75*\thedim+0.1) {$i$};
      \node[scale=0.8, anchor=east] at (0.25*\thedim-0.1, 0.5*\thedim) {$i$};
      \node[scale=0.8, anchor=north] at (0.5*\thedim, 0.25*\thedim-0.1) {$i$};

      \draw[->, line width=1.2pt, blue] (1.5*\thedim, 0.25*\thedim) -- ++(0, 0.5*\thedim);
      \draw[->, line width=1.2pt, red] (1.25*\thedim, 0.5*\thedim) -- ++(0.5*\thedim, 0);
      \node[scale=0.8, anchor=west] at (1.75*\thedim+0.1, 0.5*\thedim) {$i$};
      \node[scale=0.8, anchor=south] at (1.5*\thedim, 0.75*\thedim+0.1) {$j$};
      \node[scale=0.8, anchor=east] at (1.25*\thedim-0.1, 0.5*\thedim) {$i$};
      \node[scale=0.8, anchor=north] at (1.5*\thedim, 0.25*\thedim-0.1) {$j$};

      \draw[->, line width=1.2pt, red] (2.5*\thedim, 0.25*\thedim) -- ++(0, 0.5*\thedim);
      \draw[->, line width=1.2pt, blue] (2.25*\thedim, 0.5*\thedim) -- ++(0.5*\thedim, 0);
      \node[scale=0.8, anchor=west] at (2.75*\thedim+0.1, 0.5*\thedim) {$j$};
      \node[scale=0.8, anchor=south] at (2.5*\thedim, 0.75*\thedim+0.1) {$i$};
      \node[scale=0.8, anchor=east] at (2.25*\thedim-0.1, 0.5*\thedim) {$j$};
      \node[scale=0.8, anchor=north] at (2.5*\thedim, 0.25*\thedim-0.1) {$i$};

      \draw[->, line width=1.2pt, red] (3.25*\thedim, 0.5*\thedim) -- ++(0.25*\thedim, 0) -- ++(0,0.25*\thedim);
      \draw[->, line width=1.2pt, blue] (3.5*\thedim, 0.25*\thedim) -- ++(0, 0.25*\thedim) -- ++(0.25*\thedim, 0);
      \node[scale=0.8, anchor=west] at (3.75*\thedim+0.1, 0.5*\thedim) {$j$};
      \node[scale=0.8, anchor=south] at (3.5*\thedim, 0.75*\thedim+0.1) {$i$};
      \node[scale=0.8, anchor=east] at (3.25*\thedim-0.1, 0.5*\thedim) {$i$};
      \node[scale=0.8, anchor=north] at (3.5*\thedim, 0.25*\thedim-0.1) {$j$};

      \draw[->, line width=1.2pt, blue] (4.25*\thedim, 0.5*\thedim) -- ++(0.25*\thedim, 0) -- ++(0,0.25*\thedim);
      \draw[->, line width=1.2pt, red] (4.5*\thedim, 0.25*\thedim) -- ++(0, 0.25*\thedim) -- ++(0.25*\thedim, 0);
      \node[scale=0.8, anchor=west] at (4.75*\thedim+0.1, 0.5*\thedim) {$i$};
      \node[scale=0.8, anchor=south] at (4.5*\thedim, 0.75*\thedim+0.1) {$j$};
      \node[scale=0.8, anchor=east] at (4.25*\thedim-0.1, 0.5*\thedim) {$j$};
      \node[scale=0.8, anchor=north] at (4.5*\thedim, 0.25*\thedim-0.1) {$i$};

      \draw[xstep=\thedim, ystep=0.35*\thedim] (0,0) grid (5*\thedim, -0.35*\thedim);

      \newcommand{\they}{-0.175*\thedim}

      \node[scale=0.8] at (0.5*\thedim, \they) {$1$};
      \node[scale=0.8] at (1.5*\thedim, \they) {$\displaystyle\frac{q(1-z)}{1-qz}$};
      \node[scale=0.8] at (2.5*\thedim, \they) {$\displaystyle\frac{1-z}{1-qz}$};
      \node[scale=0.8] at (3.5*\thedim, \they) {$\displaystyle\frac{1-q}{1-qz}$};
      \node[scale=0.8] at (4.5*\thedim, \they) {$\displaystyle\frac{z(1-q)}{1-qz}$};

     \end{scope}
\end{tikzpicture}
\caption{The weights for colored S6V written for colors $i$ and $j$ satisfying $i < j$.}
\label{f.R weights}
\end{figure}

The probability measure on configurations is defined by the following sampling procedure. We start at the vertex $(1,-N)$. Here there is a horizontally incoming arrow determined by the boundary condition $\sigma$ and no vertically entering arrow. We determine the outgoing arrow configuration by choosing each of the two possible configurations (whether the horizontally incoming arrow exits vertically or horizontally) with probability equal to the vertex weight read from Figure~\ref{f.R weights}; note that this uses that the vertex weights are \emph{stochastic}, i.e., non-negative and for a given incoming arrow configuration, the sum of the weights of the possible outgoing arrow configurations satisfying arrow conservation is 1. We then proceed iteratively in a Markovian fashion: if the outgoing arrow configuration at all vertices in the domain $\mc D_n=\{(x,y)\in\Z_{\geq 1}\times\llbracket-N,\infty\rrparen: x+y < n\}$ has been determined, the configuration for vertices on the line $x+y=n$ is determined by picking the outgoing arrow configuration at each vertex independently with probability equal to that vertex's weight (again regarding an absence of an arrow along an edge as an arrow of color $-\infty$); note that the incoming arrows configuration at each such vertex is already determined by the requirement of consistency of the overall configuration, and that the order in which the outgoing arrow configurations are sampled does not matter.

The \emph{uncolored stochastic six vertex model} corresponds to the case of colored S6V with two colors: the arrows of lower color are typically not drawn, leaving only arrows of a single color which may then be regarded as uncolored.

\begin{remark}\label{r.s6v color merging}
Like the colored ASEP, the colored S6V model also satisfies color merging: we can project onto a colored S6V model with fewer colors by declaring all arrows of color lying in an interval $I\subseteq \Z$ to be the same color, which is an element of $I$. Equivalently, the Markov semigroup of the process (viewed as a Markov chain on configurations on vertical lines $x=t$ indexed by the time parameter $t$) commutes with any weakly monotone function $\Z\cup\{-\infty\}\to\Z\cup\{-\infty\}$. That this is true follows from the observation that the vertex weights in Figure~\ref{f.R weights} are determined by only the order relation of the colors of the incident arrows and not their precise value. The uncolored S6V model can again be obtained from the colored one via color merging.
\end{remark}

\subsubsection{The S6\kern-0.04em V sheet}
Recall the notation $j_v$ from Section~\ref{s.s6v configuration} for the color of the arrow horizontally exiting the vertex $v$. Fix a boundary condition $\sigma:\intint{-N,N}\to\intint{-N,N}$. We define the \emph{colored height function}
\begin{align}\label{e.s6v height function}
\hssv(x, 0; y, t) := \#\bigl\{k> y: j_{(t,k)} \geq x\bigr\}
\end{align}
for $x\in\intint{-N, N}$, $y\in\llbracket -N, \infty\rrparen$, and $t\in\N$. In words, $\hssv(x, 0; y,t)$ is the number of arrows of color at least $x$ that exit horizontally from some vertex in $\{t\}\times\llbracket y+1,\infty\rrparen$.

Note that we did not include the boundary condition $\sigma$ in the notation \eqref{e.s6v height function}. The boundary configuration we primarily work with is the \emph{packed boundary condition} (see right panel of Figure~\ref{f.colored S6V}), though some of our results will hold for general boundary conditions as well (see Sections~\ref{s.colored line ensemble} and \ref{s.LPP in discrete line ensemble}). The packed boundary condition corresponds to $\sigma(k) = k$, i.e, the arrows entering each vertex on the left boundary of the domain have color decreasing by $1$ from top to bottom. We will restrict to the packed boundary condition in the rest of this article unless explicitly mentioned otherwise.

In terms of starting position under packed boundary condition, $\hssv(x, 0;y, t)$ above is the number of arrows of any color entering horizontally at a vertex $(1,k)$ for some $k\in\intint{x,N}$ that exit horizontally from $\{t\}\times\llbracket y+1,\infty\rrparen$. As suggested by the notation, the horizontal coordinate should be thought of as time.
Note that $\hssv(x,0;y,t)$ is finite since we started with finitely many arrows on the left boundary. Note also that though the distribution of $\hssv$ depends on $N$, we omit it from the notation (though as alluded to earlier in Remark~\ref{r.N restriction for S6V}, the distribution of $h^{\mrm{S6V}}(x,0;y,t)$ jointly over $x, y\in\intint{-N,N}$ and $t\in\N$ will not depend on $N$).

As in the colored ASEP, the collection $(\hssv(x,0;\bm\cdot, t))_{x\in\intint{-N,N}}$ is a coupled collection of copies of the height functions under \emph{step initial condition} from $x$ for uncolored S6V, which means arrows enter at all heights in $\intint{x,N}$ (again recalling from Remark~\ref{r.N restriction for S6V} that the $N$ cutoff is irrelevant for the domain where we will consider the colored height function). We again refer to this coupling as the \emph{colored coupling}.

Now we define a rescaled version of the height function, which we will call the S6V sheet.
Fix $q\in[0,1)$, $z\in (0,1)$ and a velocity $\alpha \in (z,z^{-1})$ (which corresponds to the full rarefaction fan). Define $\mu(\alpha)$ and $\sigma(\alpha)$ by
\begin{equation}\label{e.mu and sigma}
\begin{split}
\mu(\alpha) &= \mu^{\mrm{S6V}}(\alpha) := - \frac{(\sqrt{\alpha} - \sqrt{z})^2}{1-z} \quad\text{and}\\
\sigma(\alpha) &= \sigma^{\mrm{S6V}}(\alpha) :=\frac{\alpha^{-1/6}z^{1/6}(1-\sqrt{z\alpha})^{2/3}(\sqrt{\alpha}-\sqrt{z})^{2/3}}{1-z};
\end{split}
\end{equation}
these are exactly the limit shape and fluctuation scale expressions as first proven in \cite{borodin2016stochastic}.
Next define the spatial scaling factor $\beta(\alpha) = \beta^{\mrm{S6V}}(\alpha)$ by
\begin{equation}\label{e.S6V spatial scaling}
\beta(\alpha) = \beta^{\mrm{S6V}}(\alpha) := \frac{2\sigma(\alpha)^2}{|\mu'(\alpha)|(1-|\mu'(\alpha)|)}.
\end{equation}

\begin{definition}[S6V sheet]\label{d.s6v sheet}
For $\varepsilon>0$, we define the \emph{S6\kern-0.04em V sheet} $\S^{\mrm{S6V}, \varepsilon} : \R^2 \to \R$, in the case that $|x|,|y|\leq \varepsilon^{-1/6}$ and the arguments of $\hssv$ below are integers, by
\begin{align}
\MoveEqLeft[9]
\S^{\mrm{S6V},\varepsilon}(x;y) := \sigma(\alpha)^{-1}\varepsilon^{1/3}\Bigl(\hssv\bigl(\beta(\alpha) x \varepsilon^{-2/3}, 0; \alpha \varepsilon^{-1}+\beta(\alpha) y\varepsilon^{-2/3}, \floor{\varepsilon^{-1}}\bigr)\nonumber\\
& + \beta(\alpha) x \varepsilon^{-2/3} -N -\mu(\alpha)\varepsilon^{-1} - \mu'(\alpha)\beta(\alpha)(y-x)\varepsilon^{-2/3}\Bigr); \label{e.rescaled s6v definition}
\end{align}
for $|x|,|y|\in [-\varepsilon^{-1/6},\varepsilon^{-1/6}]$ where the arguments of $\hssv$ are not integers we define $\S^{\mrm{S6V},\varepsilon}(x;y)$ by linear interpolation so that $\S^{\mrm{S6V},\varepsilon}$ is continuous on $\smash{[-\varepsilon^{-1/6}, \varepsilon^{-1/6}]^2}$, and we define it on $\smash{\R^2}\setminus(\smash{[-\varepsilon^{-1/6}, \varepsilon^{-1/6}]^2})$ in an arbitrary way such that the resulting function is again continuous.
\end{definition}

The choice of $[-\varepsilon^{-1/6}, \varepsilon^{-1/6}]$ is merely to ensure that $\hssv$ is well-defined, i.e., its arguments fall within its domain (for all $\varepsilon>0$); this is true under the assumption we will later adopt that $N$ is at least of order $\varepsilon^{-1}$. The value of $\S^{\mrm{S6V},\varepsilon}$ outside $[-\varepsilon^{-1/6},\varepsilon^{-1/6}]^2$ will not be relevant for our results which will only concern uniform convergence on compact sets as $\varepsilon\to 0$. Note that again we omit $N$ from the notation.
The choice of scalings (i.e., $\beta (\alpha)$ and $\sigma(\alpha)$) can be understood by the same heuristics as in Section~\ref{s.asep scalings}. We also observe that this definition differs qualitatively slightly from that of the ASEP sheet in Definition~\ref{d.asep sheet} in that $\S^{\mrm{S6V},\varepsilon}$ involves $h^{\mrm{S6V}}$ without a negative sign, unlike in ASEP, and there is an extra additive term of $\beta(\alpha)x\varepsilon^{-2/3}$ here. The first is because the proof of Theorem~\ref{t.asep airy sheet} will be via a limit of $h^{\mrm{S6V}}$ (as mentioned before), based on a matching between $h^{\mrm{S6V}}(x,0;y,t)$ and $-h^{\mrm{ASEP}}(-x,0; -y,t)$ (omitting some terms; see \eqref{e.approx asep height function definition} for more precision), and relying on a symmetry of the Airy sheet: $\S(x;y) \stackrel{\smash{d}}{=} \S(-x;-y)$ as processes. The introduction of the $\beta(\alpha)x\varepsilon^{-2/3}$ term is to ensure $\S^{\mrm{S6V}, \varepsilon}(\bm\cdot; \bm\cdot) \stackrel{\smash{d}}{=} \S^{\mrm{S6V}, \varepsilon}(\bm\cdot+r;\bm\cdot+r)$, since $h^{\mrm{S6V}}(\bm\cdot, 0; \bm\cdot, t)$ does not itself have this property; the term is absent in $\S^{\mrm{ASEP}, \varepsilon}$ since $h^{\mrm{ASEP}}(\bm\cdot, 0; \bm\cdot, t)$ does have this property.

Our main result for the colored stochastic six-vertex model is the following, where recall we work with the packed boundary condition. It is proved in Section~\ref{s.convergence to Airy sheet}.

\begin{theorem}[Airy sheet convergence for S6V]\label{t.s6v airy sheet}
Fix any asymmetry $q \in [0,1)$, spectral parameter $z\in (0,1)$, and velocity $\alpha\in(z,z^{-1})$ in the rarefaction fan. Assume $N =N(\varepsilon)\geq 2\alpha \varepsilon^{-1}$. Then, as $\varepsilon\to 0$, $\smash{\S^{\mrm{S6V},\varepsilon} \stackrel{d}{\to} \S}$ weakly in $\mc C(\R^2,\R)$ with the topology of uniform convergence on compact sets.
\end{theorem}

The lower bound on $N$ in Theorem~\ref{t.s6v airy sheet} is merely to ensure that $\S^{\mrm{S6V}, \varepsilon}$ is well-defined.

\subsubsection{The basic coupling for S6V}\label{s.s6v basic coupling}

We next give an analog of ASEP's basic coupling for S6V, which again corresponds to using the same randomness for the evolution of different initial conditions. This will allow us to define the S6V landscape which we show converges to the directed landscape.
We work in the model of uncolored S6V, or, equivalently, colored S6V with two colors, which we take to be $1$ and $-\infty$, where as usual $-\infty$ represents the absence of an arrow (as discussed in Section~\ref{s.s6v configuration}).

Consider finitely many initial conditions $\eta_0^{(1)}, \ldots, \eta_0^{(K)}\in \{-\infty,1\}^{\llbracket -N, \infty\rrparen}$ where $\eta_0^{(j)}(x)$ represents the presence or absence of a horizontal arrow at location $(1,x)$, i.e., $\smash{i_{(1,x)} = \eta_0^{(j)}(x)}$. As a notational convention, we will use $\eta$ for uncolored initial conditions (as opposed to $\sigma$ for colored initial conditions as in Section~\ref{s.s6v configuration}).

Let $b^{\shortuparrow} = q(1-z)/(1-qz)$ and $b^{\shortrightarrow} = (1-z)/(1-qz)$ which, from Figure~\ref{f.R weights}, are the probabilities of the arrow of higher color going straight through a vertex vertically or horizontally, respectively. We start with a collection of independent $0$-$1$ Bernoulli random variables $\{\smash{X^{\shortuparrow}_{v}, X^{\shortrightarrow}_{v}\}_{v\in\Z_{\geq 1}\times\llbracket -N, \infty\rrparen}}$, where $\P(\smash{X^{\shortuparrow}_{v}} = 1) = b^{\shortuparrow}$ and $\P(\smash{X^{\shortrightarrow}_{v}} = 1) = b^{\shortrightarrow}$ for each $\smash{v\in\Z_{\geq 1}\times\intray{-N}}$.

Given this randomness, the configuration is specified deterministically from the initial condition as follows. For any vertex $(x,y)\in\Z_{\geq 1}\times\llbracket -N, \infty\rrparen$ such that the incoming arrow configuration has been assigned to the vertex at $(x,y)$ (either by the boundary conditions or by the assignment of outgoing arrow configurations to the vertices at $(x-1, y)$ and $(x,y-1)$, along with consistency), let $I_{\shortrightarrow}, I_{\shortuparrow}\in\{-\infty,1\}$ respectively be $i_{(x,y)}$ and $a_{(x,y)}$, the colors of the horizontally and vertically incoming arrows at $(x,y)$ respectively. The colors $O_{\shortrightarrow}, O_{\shortuparrow}\in\{-\infty,1\}$ of the horizontally and vertically outgoing arrows from $(x,y)$ (i.e., $j_{(x,y)}$ and $b_{(x,y)}$), respectively  are defined by:
\begin{enumerate}
  \item If $I_{\shortrightarrow} = I_{\shortuparrow}$, then $O_{\shortrightarrow} = O_{\shortuparrow} = I_{\shortrightarrow}= I_{\shortuparrow}$.

  \item If $(I_{\shortrightarrow}, I_{\shortuparrow})= (-\infty,1)$, then $(O_{\shortrightarrow}, O_{\shortuparrow}) = (-\infty,1)$ if $X^{\shortuparrow}_{(x,y)} = 1$ and $(O_{\shortrightarrow}, O_{\shortuparrow}) = (1,-\infty)$ if $X^{\shortuparrow}_{(x,y)} = 0$.

  \item If $(I_{\shortrightarrow}, I_{\shortuparrow})= (1,-\infty)$, then $(O_{\shortrightarrow}, O_{\shortuparrow}) = (1,-\infty)$ if $X^{\shortrightarrow}_{(x,y)} = 1$ and $(O_{\shortrightarrow}, O_{\shortuparrow}) = (-\infty,1)$ if $X^{\shortrightarrow}_{(x,y)} = 0$.
\end{enumerate}

Note that for any times $s_j\in\N$, we may also specify (horizontally entering) boundary data $\eta_0^{(j)}$ at vertices on a vertical line $(s_j, i)$, $i\in\llbracket -N, \infty\rrparen$ and obtain the evolution of the arrows to the right using this definition, i.e., we may couple together the evolutions starting from different initial conditions at different times (by using the same randomness $(X_v^{\shortuparrow}, X_v^{\shortrightarrow})_v$).

We observe as in the case of the ASEP that the basic coupling specializes to the colored coupling in the case that the set of arrows in $\smash{\eta_0^{(j+1)}}$ is a subset of those in $\smash{\eta_0^{(j)}}$ for each $j$, i.e., $\smash{\eta_0^{(j+1)}(x) = 1}$ implies $\smash{\eta_0^{(j)}(x) = 1}$ for any $x\in\llbracket -N, \infty\rrparen$.

Next we define the S6V height function when started from a general initial height function.  Let  $h_0:\llbracket -N-1, \infty\rrparen\to \Z$ be a Bernoulli path such that $h_0(x) = 0$ for all large enough $x$. As in ASEP (recall \eqref{e.height to config}), we define an associated arrow configuration $\eta_{h_0}\in\{-\infty,1\}^{\llbracket -N, \infty\rrparen}$ by
\begin{align}\label{e.s6v particle from h}
\eta_{h_0}(x) := \begin{cases}
1 & h_0(x-1) - h_0(x) = 1\\
-\infty & h_0(x-1) - h_0(x) = 0.
\end{cases}
\end{align}
The assumption of $h_0$ being asymptotically zero implies that $\eta_{h_0}$ consists of only finitely many arrows (note that we did not assume the analogous condition of right-finiteness for ASEP from a general height function, but it will be convenient to impose this condition for S6V).

Now, for such a Bernoulli path $h_0$ and $s\in\N$, consider the S6V model with initial condition $\eta_{h_0}$ started at time $s\in\N\cup\{0\}$ (i.e., $\eta_{h_0}$ is the configuration on $\{s\}\times\llbracket -N, \infty\rrparen$), and let $j_{(x,y)}$ be the arrow exiting horizontally from $(x,y)$. Define the height function $(y,t)\mapsto \hssv(h_0, s; y,t)$ for $t>s$ (with $s,t\in\N\cup\{0\}$) and $y\in\llbracket -N, \infty\rrparen$ by
\begin{equation}\label{e.s6v height function gen initial condition}
\hssv(h_0,s; y,t) := \#\bigl\{k> y: j_{(t,k)} = 1\bigr\}.
\end{equation}

For multiple initial conditions $(h_0^{(k)})_{k=1}^K$ and times $s_k>0$, the height functions $(y,t)\mapsto\smash{\hssv(h_0^{(k)}, s_j; y,t)}$ are coupled together via the basic coupling for S6V.

\subsubsection{The S6V landscape} Using this coupling we may define the S6V landscape. As in the ASEP case, we start with a precursor: for $x,y\in\intint{-N,N}$ with $x<y$ and $0\leq s<t$ in $\N$, we overload the notation and define
\begin{align*}
\hssv(x,s; y, t) := \hssv(h_{0,x}, s; y,t)
\end{align*}
coupled via the basic coupling, where, for $z\in\llbracket -N-1,\infty\rrparen$,
\begin{equation}\label{e.S6V step initial condition}
h_{0,x}(z) = (N-x+1)\one_{z\leq x-1} + (N-z)\one_{z\in\intint{x, N}};
\end{equation}
%
note that this function is zero for all large $z$. This corresponds to having arrows entering the left boundary of the domain at locations whose heights are from $x$ till $N$. Thus it follows that $\hssv(x,s; y, t)$ generalizes $\hssv(x,0; y, t)$ as a process in $x$, $y$, and $t$.

\begin{definition}[S6V landscape]\label{d.s6v landscape}
Recall $\mu(\alpha)$, $\sigma(\alpha)$, and $\beta(\alpha)$ from \eqref{e.mu and sigma} and \eqref{e.S6V spatial scaling}. We define the \emph{S6V landscape} for $\varepsilon>0$ by
\begin{align*}
\MoveEqLeft[0]
\mc L^{\mrm{S6V}, \varepsilon}(x,s;y,t)\\
&:= \sigma(\alpha)^{-1}\varepsilon^{1/3}\Bigl(\hssv\bigl(\beta(\alpha)x\varepsilon^{-2/3}, \floor{s\varepsilon^{-1}}; \alpha (t-s)\varepsilon^{-1} + \beta(\alpha)y\varepsilon^{-2/3}, \floor{t\varepsilon^{-1}}\bigr)\\
&\qquad\qquad\qquad\quad + \beta(\alpha)x\varepsilon^{-2/3} -N - \mu(\alpha)(t-s)\varepsilon^{-1} - \mu'(\alpha)\beta(\alpha)(y-x)\varepsilon^{-2/3}\Bigr);
\end{align*}
again this definition is for $|x|, |y| \leq \varepsilon^{-1/6}$ such that the arguments of $h^{\mrm{S6V}}$ are integers, under the assumption that $N$ is at least of order $\varepsilon^{-1}$. For all other $x,y$, the function is defined by linear interpolation so that it is continuous in $x$ and $y$.
\end{definition}

\begin{definition}\label{d.s6v general initial condition}
For general initial condition $h_0:\llbracket -N, \infty \rrparen \to \Z$ and $\alpha\in(z,z^{-1})$ fixed, we define a scaled version $\mf h_0^\varepsilon$ by
\begin{align*}
\mf h^{\varepsilon}_0(y) := \sigma(\alpha)^{-1}\varepsilon^{1/3}\Bigl(h_0\bigl(\beta(\alpha)y\varepsilon^{-2/3}\bigr) +\mu'(\alpha)(N - \beta(\alpha)y\varepsilon^{-2/3})\Bigr),
\end{align*}
for $y$ such that the argument of $h_0$ is an integer, and by linear interpolation elsewhere. Similarly, we define a scaled version of $h^{\mrm{S6V}}(h_0, 0; y,t)$ by
\begin{align*}
\MoveEqLeft[16]
\mf h^{\mrm{S6V}, \varepsilon}(h_0; y,t) := \sigma(\alpha)^{-1}\varepsilon^{1/3}\Bigl(h^{\mrm{S6V}}\bigl(h_0, 0; \alpha t\varepsilon^{-1} + \beta(\alpha)y\varepsilon^{-2/3}, \floor{t\varepsilon^{-1}}\bigr)\\
& -\mu(\alpha)t\cdot\varepsilon^{-1} + \mu'(\alpha)(N-\beta(\alpha)y\varepsilon^{-2/3}) \Bigr),
\end{align*}
for $y$ such that the argument of $h^{\mrm{S6V}}$ is an integer, and by linear interpolation elsewhere.
\end{definition}

We now state the convergence results for $\mc L^{\mrm{S6V},\varepsilon}$ and $\mf h^{\mrm{S6V}, \varepsilon}$. The first is proved in Appendix~\ref{s.general initial condition} by using an approximate monotonicity estimate for the S6V model under basic coupling (Proposition~\ref{p.s6v approximate monotonicity}) together with an idea communicated to us by Shalin Parekh. As with Theorem~\ref{t.asep general initial condition} (3), the convergence of $\mf h^{\mrm{S6V}, \varepsilon}$ is proved not in this article but in the paper \cite{aggarwal2024kpz}, and the only inputs used from this paper are the aforementioned convergence of $\mc L^{\mrm{S6V},\varepsilon}$, the approximate height monotonicity of the S6V model, and a certain finite speed of discrepancy estimate (Lemma~\ref{l.finite speed of prop for s6v basic coupling}) proved in Appendix~\ref{s.general initial condition}. Recall that for a set $\mc T$, $\mc T^2_< = \{(s,t)\in\mc T^2: s<t\}$.

\begin{theorem}[Directed landscape and KPZ fixed point convergence for S6V]\label{t.s6v landscape}
Fix $q\in[0,1)$, $z\in(0,1)$, $\alpha\in (z,z^{-1})$, and a countable set $\mc T\subseteq [0,\infty)$. Assume $N\geq 2\alpha \varepsilon^{-1}$. Then, the following limits hold as $\varepsilon\to 0$.

\begin{enumerate}
  \item (Directed landscape). $\smash{(\mc L^{\mrm{S6V},\varepsilon}(\bm\cdot, s; \bm\cdot, t))_{(s,t)\in\mc T^2_<}\xrightarrow{d} (\mc L(\bm\cdot, s; \bm\cdot, t))_{(s,t)\in\mc T^2_<}}$ in \\$\smash{\mc C(\R^2,\R)^{\mc T^2_<}}$ under the topology of uniform convergence on compact sets.

  \item As a special case, for $k\in\N$ and $x_k< \ldots <x_1$,
  $$\smash{(\mc L^{\mrm{S6V},\varepsilon}(x_i, 0; \bm\cdot, t))_{i\in\intint{1,k}, t\in\mc T}\xrightarrow{d} (\mc L(x_i, 0; \bm\cdot, t))_{i\in\intint{1,k}, t\in\mc T}}$$
  in $\smash{\mc C(\R,\R)^{\intint{1,k}\times\mc T}}$ under the topology of uniform convergence on compact sets.

  \item (Coupled KPZ fixed point). For $k\in\N$ and each $i\in \{1, \ldots, k\}$, suppose $\smash{h_0^{(i),\varepsilon}}: \llbracket -N, \infty\rrparen\to\Z$ is a sequence of (possibly random, and independent of the S6V dynamics) Bernoulli paths with $h_0^{(i),\varepsilon}(x) = 0$ for all large enough $x$, and $\smash{\mf h_0^{(i)}:\R\to\R}$ is a (possibly random, and independent of $\mc L$) continuous function. Let $\mf h_0^{(i), \varepsilon}$ be as in Definition~\ref{d.s6v general initial condition} and assume they satisfy $\inf_{[-C,C]}\smash{\h_0^{(i),\varepsilon}} \geq -C$ and $\smash{\mf h_0^{(i), \varepsilon}}(x)\leq C(1+|x|)$ for some $C<\infty$ almost surely for all $x\in\R$, $i\in\intint{1,k}$, and $\varepsilon>0$. Suppose $x\mapsto \smash{\mf h_0^{(i),\varepsilon}(x)}$ converges to $\smash{\mf h_0^{(i)}}$ as $\varepsilon\to0$ for $i=1, \ldots, k$  jointly in distribution under the topology of uniform convergence on compact sets. Then, in $\smash{\mc C(\R,\R)^{\intint{1,k}\times\mc T}}$,
  $$\bigl(\mf h^{\mrm{S6V}, \varepsilon}(h_0^{(i),\varepsilon}; \bm\cdot,t)\bigr)_{{i\in\intint{1,k}, t\in\mc T}} \xrightarrow{d} \smash{\bigl(\mf h(\mf h^{(i)}_0}; \bm\cdot, t)\bigr)_{i\in\intint{1,k}, t\in\mc T},$$
   where the lefthand side is coupled via the basic coupling and the righthand side is coupled as in~\eqref{e.kpz fixed point} under the same directed landscape $\mc L$.

\end{enumerate}

\end{theorem}

\begin{remark}[Scaling limit of S6V with finitely many colors]\label{r.s6v finitely many colors}
Analogous to the ASEP case (Remark \ref{r.asep finitely many colors}), Theorem~\ref{t.s6v landscape} (2) can be interpreted as a scaling limit result for the S6V with finitely many $k \ge 1$ colors via color merging. More precisely, fix real numbers $x_k <  \ldots < x_1$. Consider colored S6V with initial condition of arrows of color $k$ at locations $\intint{\floor{\beta(\alpha)x_1\varepsilon^{-2/3}}+1, N}$; arrows of color $0$ at locations $\intint{-N,\floor{\beta(\alpha)x_k\varepsilon^{-2/3}}}$; and arrows of color $k-i$ at all locations in $\intint{\floor{\beta(\alpha)x_{i+1}\varepsilon^{-2/3}}+1, \floor{\beta(\alpha)x_i\varepsilon^{-2/3}}}$ for $i=1, \ldots, k-1$.
Then, by color merging, $(\mc L^{\mrm{S6V},\varepsilon}(x_1; \bm\cdot),  \ldots, \mc L^{\mrm{S6V},\varepsilon}(x_k; \bm\cdot))$ is exactly a centered and rescaled version of the height functions associated to the colored S6V with $k$ different colors and, by Theorem~\ref{t.s6v landscape} (2), they converge to the directed landscape.
\end{remark}

\begin{remark}\label{r.s6v landscape}

The stationary measures for the colored S6V coincide with those of the colored ASEP (as is a consequence of the convergence of the former to the latter; see \cite[Proposition B.2]{aggarwalborodin}); so, the stationary horizon convergence of the colored ASEP (Corollary~\ref{c.SH convergence}) implies the same for the colored S6V. 
\end{remark}


\section{Colored $q$-Boson model, line ensembles, and Gibbs properties}\label{s.key ideas}

As we saw in Definition~\ref{d.airy sheet}, the Airy sheet is given in terms of a last passage percolation problem in the parabolic Airy line ensemble $\bm\cP$. As outlined in Section~\ref{Asymptoticq0}, the starting point of our proof is a prelimiting analog of this relation, in terms of a certain colored line ensemble, called the \emph{colored Hall-Littlewood line ensemble}, that is associated to the colored S6V model and was recently introduced in \cite{aggarwalborodin} (we reproduce the argument for the relation between the colored line ensemble and the colored S6V model in Appendix~\ref{s.yang-baxter}). This line ensemble is defined in terms of an auxiliary vertex model which we refer to as the \emph{colored $q$-Boson model}, and we introduce it in Section~\ref{s.intro.vertex model}. We will then introduce the colored Hall-Littlewood line ensemble and its Gibbs property in Section~\ref{s.colored line ensemble}. Section~\ref{s.line ensemble convergence} states the convergence of uncolored line ensembles to $\bm\cP$. Section~\ref{s.LPP in discrete line ensemble} states the precise prelimiting LPP representation of higher color height functions.

\subsection{Colored $q$-Boson model}\label{s.intro.vertex model}

The \emph{colored and uncolored $q$-Boson models} we are introducing are vertex models on the domain $\Z_{\leq 0}\times \intint{1, N+M}$ for $N,M\in\N$; note that this is different from the domain for the colored stochastic six-vertex model, and we will make the connection between the two in Proposition~\ref{p.colored line ensembles}. As in the stochastic six vertex model introduced in Section~\ref{s.intro.cS6V}, there are two parameters: a \emph{quantization parameter} $q\in[0,1)$ and a \emph{spectral parameter} $z\in(0,1)$, both of which are held fixed. We will focus on the colored model for the time being. Again as in the colored S6V model, informally (see the next paragraph for a more precise description), a configuration of the model consists of colored arrows, with the color of an arrow being an integer in $\intint{1,N}$. These arrows enter the domain at the left boundary and move through the domain forming up-right paths, satisfying colored arrow conservation at each vertex; we will specify the colors of these arrows via a boundary condition $\sigma$ shortly (Section~\ref{s.higher spin boundary condition}). There is also an exiting boundary condition: all arrows in the system exit vertically from $(0,N+M)$.

An important difference compared to the colored S6V model is that here an arbitrary number of arrows may occupy the same vertical edge (though horizontal edges are still restricted to at most 1 arrow), as well as the fact that the weights (as will be introduced ahead in Figure~\ref{f.L weights}) are not stochastic, i.e., do not sum to $1$ when summing over all possible outgoing vertex configurations for fixed incoming vertex configuration; thus the model's probability distribution is defined via a product of vertex weights with an additional non-trivial normalization factor. Finally, note that in the colored $q$-Boson model we allow arrows to have only positive colors.

\begin{figure}[b]
\begin{tikzpicture}[scale=0.7]
\newcommand{\Nval}{3}
\newcommand{\Mval}{4}
\newcommand{\thelength}{6.5}

  \foreach \y in {1,..., \Mval}
    \draw (\Mval-1, -\Mval + \y) -- ++(\thelength+1,0);

  \foreach \y in {1,..., \Nval}
   \draw (\Mval-1, \Nval - \y + 1) -- ++(\thelength+1,0);

  \foreach \x in {1,..., 3, 4.5, 5.5, ...,\thelength}
    \draw[thick] (\Mval-1 + \x, -\Mval + 0.5) -- ++ ($(0,\Mval+\Nval)$);

  \node[scale=1.8] at (\Mval-1+3.85, 0.5) {$\cdots$};

  \foreach \x in {1,..., \Mval}
    \draw[->]  (\Mval-1.5, \x-1) -- node[scale=0.7, anchor = east, left=2pt]{$z$} ++(0.35, 0);

  \foreach \x in {1,..., \Nval}
    \draw[->]  (\Mval-1.5, \x-1 -\Nval) -- node[scale=0.7, anchor = east, left=2pt]{$1$} ++(0.35, 0);

  \foreach \x/\thecolor in {1/blue, 2/orange, 3/red}
  {
    \draw[\thecolor, very thick] (\Mval-1, \x-\Nval-1) -- ++(3,0);
    \draw[\thecolor, very thick, dashed] (\Mval+2.25, -\Mval+\x) -- ++(1,0);
  }

  \draw[red, very thick, ->] (\Mval + 3.5, -\Mval + 3) -- ++(1,0) -- ++(0,2) -- ++(1-0.085,0) -- ++(0,2.5);
  \draw[orange, very thick, ->] (\Mval + 3.5, -\Mval + 2) -- ++(1,0) -- ++(0,1) -- ++(1,0) -- ++(0,4.5);
  \draw[blue, very thick, ->] (\Mval + 3.5, -\Mval + 1) -- ++(2+0.085,0) -- ++(0,6.5);

  \foreach \x in {1,..., \Nval}
    \node[scale=0.8, anchor=south] at  (\Mval - 0.6, \x-\Nval-1) {$\sigma(\x)$};

  \draw [decorate,decoration={brace, mirror}, semithick]
  (\Mval-2.2, \Nval) --node [pos=0.5, anchor = east, scale=0.8, left=2pt] {$M$}  ++(0, -\Mval+1) ;

  \draw [decorate,decoration={brace, mirror}, semithick]
  (\Mval-2.2, -1) --node [pos=0.5, anchor = east, scale=0.8, left=2pt] {$N$}  ++(0, -\Nval+1) ;

  \node[circle, fill, inner sep = 1pt] (dot) at (\Mval-1+\thelength-1, \Mval-3) {};
  \node[scale=0.6] (A) at (\Mval + \thelength+0.8, \Mval-2.5) {$(-1,N+2)$};
  \draw[->, dashed, semithick] (A) to[out=165, in=45] (dot);
\end{tikzpicture}
\caption{The colored $q$-Boson model. The $z\to$ and $1\to$ at the left indicate the spectral parameters of the $\bigL$ weights associated to vertices in the corresponding rows, the expressions above the entry point of the arrows indicate the corresponding arrow's color, and the $\bm\cdots$ indicate that there are infinitely many columns, though, by definition of the boundary conditions, the arrows travel purely horizontally except until finitely many columns from the right.} 
\label{f.vertex model}
\end{figure}
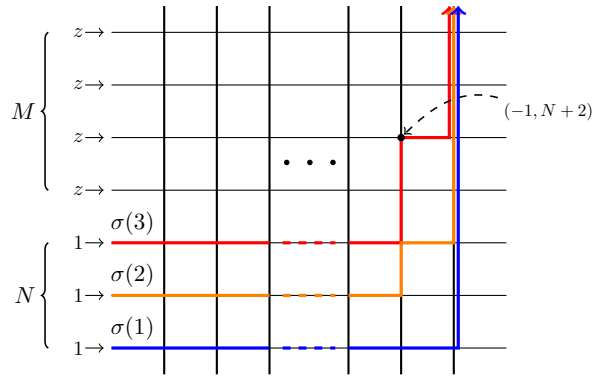

\subsubsection{Arrow configurations} \label{s.higher spin configuration} More formally, for a vertex $v\in\Z_{\leq0}\times\intint{1,N+M}$, we denote the arrow configuration at $v$ by a tuple $(\bm A_v, i_v; \bm B_v, j_v)$, where $i_v,j_v \in \intint{0,N}$ and $\bm A_v,\bm B_v\in \Z_{\geq 0}^N$ are vectors $(A_{v,1}, \ldots, A_{v,N})$ and $(B_{v,1}, \ldots, B_{v,N})$, such that $A_{v,k} + \one_{i_v=k} = B_{v,k} + \one_{j_v=k}$ for each $k\in\intint{1,N}$.  Here $i_v$ and $j_v$ represent the colors of the horizontally incoming and outgoing arrows at $v$, respectively, while $A_{v,k}, B_{v,k}$ represent the number of arrows of color $k$ vertically entering and exiting at $v$, respectively, for $k\in\intint{1,N}$. For this model, color 0 will represent no arrow (unlike the colored S6V model in Section~\ref{s.s6v configuration}, where $-\infty$ did). The condition $A_{v,k} + \one_{i_v=k} = B_{v,k} + \one_{j_v=k}$ is that of \emph{colored arrow conservation}. A configuration of the model is an assignment of arrow configurations to all vertices of $\Z_{\leq 0}\times \intint{1,N+M}$ that is \emph{consistent}, i.e., the arrows entering a vertex horizontally or vertically are respectively the same as the arrows horizontally exiting the vertex immediately to the left or vertically exiting the vertex immediately below.

\subsubsection{Boundary conditions} \label{s.higher spin boundary condition}

Let $\sigma: \intint{1,N}\to\intint{1,N}$ be a function. The boundary condition $\sigma$ of arrows is specified as follows: informally, at $(-\infty, i)$ for $i\in\intint{1,  N}$ an arrow of color $\sigma(i)$ enters and moves purely horizontally for all but finitely many edges, and no arrow enters at row $i$ for  $i\in\intint{N+1, N+M}$; further, all arrows which originally entered the system vertically exit at $(0,N+M)$. More formally, the arrow configuration at $v = (-K,i)$ is given by $(0^N, \sigma(i); 0^N, \sigma(i))$ for $i\in\intint{1,N}$ and by $(0^N, 0; 0^N, 0)$ for $i\in\intint{N+1,N+M}$, for all large enough $K$, and no arrows vertically enter the system through the bottom boundary $\Z_{\leq 0}\times\{1\}$. Further, the configuration of arrows vertically exiting $(0,N+M)$ is $(\#\sigma^{-1}(\{j\}))_{j=1}^N$, and no arrows horizontally or vertically exit the system elsewhere.
We emphasize that we only consider boundary conditions $\sigma$ taking values at least $1$; in other words, we require that an arrow enters the $i$\th row for each~$i \in \intint{1, N}$.%

\begin{figure}[h]
\begin{tikzpicture}
\newcommand{\thedim}{4.2}
\begin{scope}[scale=0.6]

      \draw[step=\thedim] (0,0) grid (6*\thedim, \thedim);

      \draw[->, very thick, blue] (0.5*\thedim, 0.2*\thedim) -- ++(0, 0.6*\thedim);
      \draw[->, very thick, green!80!black] (0.5*\thedim+0.15, 0.2*\thedim) -- ++(0, 0.6*\thedim);

      \draw[->, very thick, blue] (1.5*\thedim-0.075, 0.2*\thedim) -- ++(0, 0.6*\thedim);
      \draw[->, very thick, green!80!black] (1.5*\thedim+0.075, 0.2*\thedim) -- ++(0, 0.3*\thedim) -- ++(0.3*\thedim, 0);
      \node[scale=0.6, anchor=west] at (1.8*\thedim+0.1, 0.5*\thedim) {$i$};

      \draw[->, very thick, blue] (2.5*\thedim +0.075, 0.2*\thedim) -- ++(0, 0.6*\thedim);
      \draw[->, very thick, green!80!black] (2.2*\thedim, 0.5*\thedim) -- ++(0.3*\thedim +0.2, 0) -- ++(0, 0.3*\thedim);
      \node[scale=0.6, anchor=east] at (2.2*\thedim, 0.5*\thedim+0.1) {$i$};

      \draw[->, very thick, blue] (3.5*\thedim, 0.2*\thedim) -- ++(0, 0.6*\thedim);
      \draw[->, very thick, green!80!black] (3.2*\thedim, 0.5*\thedim) -- ++(0.3*\thedim -0.15, 0) -- ++(0, 0.3*\thedim);
      \draw[->, very thick, orange!90!black] (3.5*\thedim+0.15, 0.2*\thedim) -- ++(0, 0.3*\thedim) -- ++(0.3*\thedim, 0);
      \node[scale=0.6, anchor=east] at (3.2*\thedim, 0.5*\thedim+0.1) {$i$};
      \node[scale=0.6, anchor=west] at (3.8*\thedim+0.075, 0.5*\thedim) {$j$};

      \draw[->, very thick, blue] (4.5*\thedim, 0.2*\thedim) -- ++(0, 0.6*\thedim);
      \draw[->, very thick, orange!90!black] (4.2*\thedim, 0.5*\thedim+0.1) -- ++(0.3*\thedim +0.15, 0) -- ++(0, 0.3*\thedim-0.1);
      \draw[->, very thick, green!80!black] (4.5*\thedim-0.15, 0.2*\thedim) -- ++(0, 0.3*\thedim) -- ++(0.3*\thedim, 0);
      \node[scale=0.6, anchor=east] at (4.2*\thedim, 0.5*\thedim+0.1) {$j$};
      \node[scale=0.6, anchor=west] at (4.75*\thedim, 0.5*\thedim) {$i$};

      \draw[->, very thick, blue] (5.5*\thedim, 0.2*\thedim) -- ++(0, 0.6*\thedim);
      \draw[->, very thick, green!80!black] (5.5*\thedim+0.15, 0.2*\thedim) -- ++(0, 0.6*\thedim);
      \draw[->, very thick, orange!90!black] (5.2*\thedim, 0.5*\thedim) -- ++(0.6*\thedim, 0);
      \node[scale=0.6, anchor=west] at (5.8*\thedim, 0.5*\thedim) {$i$};

      \draw[xstep=\thedim, ystep=0.25*\thedim] (0,0) grid (6*\thedim, -0.5*\thedim);

      \newcommand{\they}{-0.125*\thedim}

      \node[scale=0.775] at (0.5*\thedim, \they) {$(\bm A, 0; \bm A, 0)$};
      \node[scale=0.775] at (1.5*\thedim, \they) {$(\bm A, 0; \bm A_i^-, i)$};
      \node[scale=0.775] at (2.5*\thedim, \they) {$(\bm A, i; \bm A_i^+,0)$};
      \node[scale=0.775] at (3.5*\thedim, \they) {$(\bm A, i; \bm A_{ij}^{+-}, j)$};
      \node[scale=0.775] at (4.5*\thedim, \they) {$(\bm A, j; \bm A_{ji}^{+-}, i)$};
      \node[scale=0.775] at (5.5*\thedim, \they) {$(\bm A, i; \bm A, i)$};

      \node[scale=0.775] at (0.5*\thedim, 3*\they) {$1$};
      \node[scale=0.775] at (1.5*\thedim, 3*\they) {$u(1-q^{A_i})q^{A_{[i+1,N]}}$};
      \node[scale=0.775] at (2.5*\thedim, 3*\they) {$1$};
      \node[scale=0.775] at (3.5*\thedim, 3*\they) {$u(1-q^{A_j})q^{A_{[j+1,N]}}$};
      \node[scale=0.775] at (4.5*\thedim, 3*\they) {$0$};
      \node[scale=0.775] at (5.5*\thedim, 3*\they) {$uq^{A_{[i+1,N]}}$};

     \end{scope}
\end{tikzpicture}
\caption{The $\bigL_u$ weights for $1\leq i < j \leq N$; if the label $i$ is present in a panel without $j$, it can also take the value $N$. Here for $\bm A\in\Z_{\geq 0}^N$ and $i\in\intint{1,N}$, $A_{[i,N]} = \smash{\sum_{j=i}^N A_j}$. The notation $\smash{\bm A_\ell^+}$ and $\smash{\bm A_\ell^-}$ respectively means the vector obtained by increasing or decreasing the $\ell$\smash{\th} component of $\bm A$ by $1$; similarly $\bm A_{k\ell}^{+-}$ is the vector obtained by increasing the $k$\th component of $\bm A$ by $1$ and decreasing the $\ell$\th component by 1.}
\label{f.L weights}
\end{figure}

\subsubsection{The colored $q$-Boson weights and probability measure}
The weight of a configuration is given by the product of the vertex weights at each vertex of the domain. The vertex weights are specified as follows. First we define the $\bigL_u$ weights in Figure~\ref{f.L weights}, which will also be called the \emph{colored $q$-Boson weights} (and are the same as the ones in \cite[eq. (1.2.2)]{borodin2018coloured} or \cite[Figure 5]{aggarwalborodin}, with $s=0$); here, the weights are defined for $u\in(0,\infty)$. The vertex weights in our model are of two types, coming from the weights in Figure~\ref{f.L weights} with two different values of $u$: those on the bottom $N$ rows, i.e, on $\Z_{\leq 0}\times\intint{1,N}$, are given by the $\bigL_u$ weights with $u=1$, while those on the top $M$ rows, i.e., on $\Z_{\leq 0}\times\intint{N+1,N+M}$, are given by the same with $u=z$ as fixed above; see Figure~\ref{f.vertex model}. Thus for the entirety of the main text $u$ will always be in $\{1,z\}$.

\begin{definition}\label{d.colored q-Boson}
Fix a boundary condition $\sigma:\intint{1,N}\to\intint{1,N}$. This specifies a collection of consistent configurations that satisfy this boundary condition and which have positive weight. We may define a probability measure supported on this collection, which we call the \emph{colored $q$-Boson measure} with boundary condition $\sigma$. The measure is defined by specifying that
\begin{align}\label{e.q-boson measure}
\parbox{5in}{\centering the probability of a given configuration is proportional to its weight;\\the normalization constant is finite and equals $\bigl(\frac{1-qz}{1-z}\bigr)^{NM}$.}
\end{align}
The \emph{uncolored $q$-Boson model} is the case where only arrows of color $1$ are present, i.e., $\sigma(k) = 1$ for all $k\in\intint{1,N}$ (and thus there is no boundary  condition to be specified, i.e., there is only one choice for $\sigma$).
\end{definition}
The normalization constant in \eqref{e.q-boson measure} is the sum of weights of all configurations and its expression is derived in \cite[Lemma~2.9]{aggarwalborodin} via the Yang-Baxter equation relating the colored $q$-Boson and colored S6V weights; equivalently, this can be seen as a Cauchy identity for certain functions associated to the vertex model. We reproduce the former argument in Appendix~\ref{s.yang-baxter}.

\subsubsection{Color merging}\label{s.q-boson color merging}

The colored $q$-Boson model satisfies a form of color merging, like the colored S6V model (Remark~\ref{r.s6v color merging}). Take the colored $q$-Boson measure with boundary condition $\sigma$ and identify any interval $I\subseteq \intint{1,N}$ of colors with an element of $I$, and if desired relabel all the colors in a weakly increasing way (preserving the original orderings). The resulting projection is the colored $q$-Boson measure with boundary condition $\sigma'$ obtained by performing the same color identification and relabeling in $\sigma$. A more precise statement is as follows.

\begin{lemma}\label{l.q-Boson color merging}
Let $N, M\in\N$ and $\sigma:\intint{1,N}\to\intint{1,N}$. Let $\tau:\intint{0,N}\to\intint{0,N}$ be non-decreasing and such that $\tau^{-1}(\{0\}) = \{0\}$. Then $(j_v)_{v\in\Z_{\leq 0}\times\intint{1,N+M}}$ under the colored $q$-Boson measure with boundary condition $\tau\circ\sigma$ has the same distribution as $(\tau(j_v))_{v\in\Z_{\leq 0}\times\intint{1,N+M}}$ under the colored $q$-Boson measure with boundary condition $\sigma$.

In particular, defining $\tau$ by $\tau(0) = 0$ and $\tau(j) = 1$ for $j\in\intint{1,N}$, we see that the colored $q$-Boson model with all colors ignored (i.e., only total arrow counts are observed) yields the uncolored $q$-Boson model.
\end{lemma}

That this color merging property holds is not as immediate as in the case of colored S6V. Indeed, in the latter it was a consequence of the vertex weights being determined by the relative order of the colors of arrows and not their precise values; this implicitly used that the normalization constant or partition function for the measure is $1$ (or, rather, independent of the boundary condition), since the weights are stochastic, i.e, for a fixed incoming arrow configuration, summing the weights over all valid outgoing arrow configurations yields $1$. In contrast, the $\bigL_u$ weights are not determined purely by the relative order and are also not stochastic.

\begin{proof}[Proof of Lemma~\ref{l.q-Boson color merging}]
The verification of color merging requires the following fact. First, for $\tau:\intint{0,N}\to\intint{0,N}$ non-decreasing with $\tau^{-1}(\{0\}) = \{0\}$, define $\rho_{\tau}:\Z_{\geq 0}^N \to \Z_{\geq 0}^N$ by the following: for $\bm A\in\Z_{\geq 0}^N$ and $k\in\intint{1,N}$, $(\rho_{\tau}(\bm A))_k = \sum_{i\in\tau^{-1}(\{k\})} A_i$. Then, it is a straightforward calculation that for $\bm A_v\in\Z_{\geq 0}^N$, $i_v\in\intint{0,N}$, $\widetilde{\bm B}_v \in \rho_\tau(\Z_{\geq 0}^N)$,  and $\lambda\in \tau(\intint{0,N})$,
\begin{align}\label{e.color merging weights}
\sum_{\substack{\bm B_v, j_v : \tau(j_v)=\lambda,\\ \rho_\tau(\bm B_v) = \widetilde{\bm B}_v}} \bigL_u(\bm A_v, i_v; \bm B_v, j_v) = \bigL_u\bigl(\rho_\tau(\bm A_v), \tau(i_v); \widetilde{\bm B}_v, \lambda\bigr).
\end{align}
In words, if we fix the incoming arrow configurations and sum the vertex weight over all outgoing arrow configuration under the constraint that the outgoing arrow colors, on merging according to $\tau$, are fixed, we obtain the vertex weight with the colors merged according to $\tau$. That the $\bigL_u$ weights have this property has been observed a number of times in the literature; see, for example, \cite[eq. (4.9)]{aggarwalborodin}, or \cite[Proposition 2.4.2]{borodin2018coloured} or \cite{foda2013colour} in the case of merging all colors with color 1.

Fix a configuration $\bm\omega$ of the colored $q$-Boson model with boundary condition $\tau\circ\sigma$. Repeatedly applying \eqref{e.color merging weights} implies that the weight of $\bm\omega$ is the same as the sum of the weights of all configurations of the colored $q$-Boson model with boundary condition $\sigma$ such that one obtains $\bm\omega$ on color merging according to $\tau$. Next, repeatedly applying the special case of \eqref{e.color merging weights},  with $\tau(k) = 1$ for all $k\in\intint{1,N}$ and $\tau(0) = 0$, yields that the partition functions for the models with boundary conditions $\sigma$ and $\tau\circ\sigma$, respectively, are the same. These two inferences complete the proof by recalling the definition \eqref{e.q-boson measure} of the colored $q$-Boson measure.
\end{proof}

\begin{remark}\label{r.no color merging with zero}
Unlike the colored S6V model, in the colored $q$-Boson model one cannot merge colors with color zero (no arrows). This is a consequence of the $\bigL_u$ weights not being stochastic. It can be seen by the fact that the analog of \eqref{e.color merging weights} for merging with $0$ (i.e., $\tau^{-1}({0})$ contains a non-zero integer) does not hold.
\end{remark}

\subsubsection{Gibbs property for $q$-Boson models}\label{s.intro.vertex model gibbs property}
Next we discuss two Gibbs properties, one each for the colored and uncolored $q$-Boson models, that are  consequences of the above description, and which will be used in later arguments. By Gibbs property, we mean a spatial Markov property: informally, for a domain $\Lambda$, the conditional distribution of the configuration strictly inside $\Lambda$, given the configuration on the boundary of $\Lambda$ as well as outside $\Lambda$, is explicit and depends only on the configuration on the boundary of $\Lambda$.

We start with the uncolored Gibbs property. We may write the arrow configuration at a vertex $v$ as $(A_v, i_v; B_v, j_v)$ where $i_v, j_v\in\{0,1\}$ (presence of arrow) and $A_v, B_v\in\Z_{\geq 0}$ (number of arrows).  Because of the form of the weight of a configuration as a product of vertex weights, that this model enjoys a Gibbs property is immediate; the dependency of the conditional distribution on the configuration on the boundary of the domain is imposed via the requirement of consistency and arrow conservation. To state the Gibbs property more precisely, fix a rectangle $\Lambda:=\intint{-k,-\ell}\times\intint{a,b}\subseteq \Z_{\leq 0}\times \intint{1,N+M}$, let $\Lambda^c = \Z_{\leq 0}\times\intint{1,N+M}\setminus \Lambda$, and let $\smash{\F^{\mrm{uB}}_{\Lambda^c}}$ (``uB'' short for uncolored $q$-Boson) be the $\sigma$-algebra (where $\sigma(\bm\cdot)$ for $\bm\cdot$ being some collection of random variables means the $\sigma$-algebra generated by that collection)
\begin{align}\label{e.F^ub}
\F^{\mrm{uB}}_{\Lambda^c} := \sigma\left(\left\{(A_v; B_v, j_v): v\in \Lambda^c\right\}\right);
\end{align}
note that, by consistency, for every $v$ in the left boundary of $\Lambda$, it holds that $i_v$ is $\F^{\mrm{uB}}_{\Lambda^c}$-measurable. 

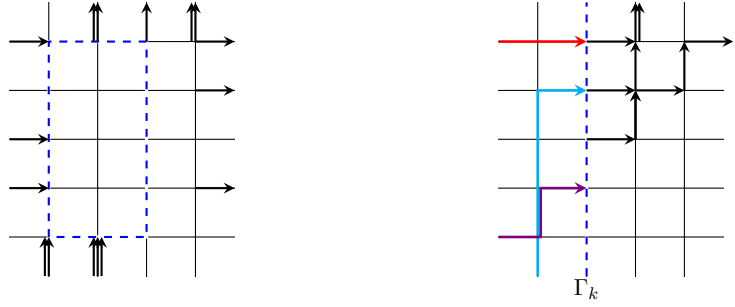
\begin{figure}
\begin{tikzpicture}[scale=1.3]

\begin{scope}[shift={(5,0)}]
\begin{scope}
\clip (-0.4, -0.4) rectangle (1.9, 2.4);
\draw (-1,-1) grid[step=0.5] (3,3);
\end{scope}

\draw[white, very thick] (0.5, -0.4) -- ++(0,2.8);
\draw[dashed, blue, thick] (0.5,-0.4) -- ++(0,2.8);
\node[anchor=north, scale=0.8] at (0.5,-0.35) {$\Gamma_k$};

\draw[line width=1pt, -stealth, cyan] (0, -0.4) -- ++(0, 1.9) -- ++(0.5,0);
\draw[line width=1pt, -stealth, violet] (-0.4, 0) -- ++(0.43, 0) -- ++(0,0.5) -- ++(0.47,0);
\draw[line width=1pt, -stealth, red] (-0.4, 2) -- ++(0.9,0);

\foreach \x/\y in {0.5/2, 0.5/1.5, 0.5/1, 1/1.5, 1.5/2}
  \draw[thick, -stealth] (\x, \y) -- ++(0.5,0);

\foreach \x/\y in {1/1, 1/1.5, 1/1, 1.5/1.5}
  \draw[thick, -stealth] (\x, \y) -- ++(0, 0.5);

\foreach \x/\y in {1/2, 1.04/2}
  \draw[thick, -stealth] (\x, \y) -- ++(0, 0.4);

\end{scope}

\begin{scope}
\clip (-0.4, -0.4) rectangle (1.9, 2.4);
\draw (-1,-1) grid[step=0.5] (3,3);
\end{scope}

\draw[white, very thick] (0,0) rectangle (1,2);
\draw[dashed, blue, thick] (0,0) rectangle (1,2);

\foreach \x [count=\c from 6] in {1.46, 1.5, 1, 0.5, 0.46}
  \draw[thick, -stealth] (\x, 2) -- ++(0,0.4);

\foreach \x [count=\c] in {-0.04, 0, 0.46, 0.5, 0.54}
  \draw[thick, -stealth] (\x, -0.4) -- ++(0,0.4);

\useasboundingbox (current bounding box);

\foreach \y [count =\c from 3] in {0.5, 1.5, 2}
  \draw[thick, -stealth] (1.5, \y) --  ++(0.4,0);

\foreach \y [count =\c from 9] in {2, 1, 0.5}
  \draw[thick, -stealth] (-0.4, \y) -- ++(0.4,0);

\end{tikzpicture}
\caption{Left: The setup for the Gibbs property in the uncolored $q$-Boson model. Outgoing arrow counts outside the region $\Lambda$ outlined in blue are conditioned upon, and arrows have been displayed in this depiction if the arrow count is non-zero and has been conditioned on. Arrow counts entering the blue region are conditioned on due to consistency, but those exiting it are not. Right: The setup for the Gibbs property in the colored $q$-Boson model. Horizontally outgoing colored arrows counts are conditioned on strictly to the left of the blue dashed line (which extends to the ends of the domain in both directions), while only total outgoing arrow counts (depicted in black) are conditioned on to the right (including on~$\Gamma_k$).}\label{f.vertex gibbs}
\end{figure}

For the statement of the Gibbs property, we adopt the notation $\bigL_{z,v}$ for the vertex weights in the model: $\bigL_{z,v} = \bigL_{1}$ for $v \in \Z_{\leq 0}\times \intint{1,N}$ and $\bigL_{z,v} = \bigL_{z}$ for $v \in \Z_{\leq 0}\times \intint{N+1, N+M}$.

\begin{lemma}[Uncolored $q$-Boson Gibbs property]\label{l.uncolored vertex gibbs}
Let $\Lambda=\intint{-k,-\ell}\times\intint{a,b}$ and $\smash{\F^{\mrm{uB}}_{\Lambda^c}}$ be as in \eqref{e.F^ub}. Under the uncolored $q$-Boson measure, the conditional distribution of $((A_v, i_v; B_v, j_v):v\in\Lambda)$ given $\smash{\F^{\mrm{uB}}_{\Lambda^c}}$ is supported on tuples $((\smash{A'_v, i'_v; B'_v, j'_v}): v\in\Lambda)$ such that arrow conservation and consistency are satisfied inside $\Lambda$ and with respect to the boundary conditions; call the (random) set of such tuples $\mrm{Supp}^{\mrm{uB}}$. Further, the conditional probability that $((A_v, i_v; B_v, j_v): v\in\Lambda) = ((A'_v, i'_v; B'_v, j'_v): v\in\Lambda)$, for any fixed (i.e., deterministic) selection of the righthand side, is proportional to $\prod_{v\in \Lambda} {\bigL}_{z,v}(A'_v, i'_v; B'_v, j'_v)\one_{(A'_v, i'_v; B'_v, j'_v)_{v\in\Lambda}\in \mrm{Supp}^{\mrm{uB}}}$.
\end{lemma}

\begin{proof}
This follows from the form of the uncolored $q$-Boson measure as a product of vertex weights.
\end{proof}

Next we turn to a Gibbs property for the colored $q$-Boson model. Informally, it explains how to ``reveal'' colored arrow counts, conditional on total arrow counts.
Fix $k\in\N$. Let $\Lambda_k = \llparen {-}\infty, -k\rrbracket\times\intint{1,N+M}$ and $\Gamma_k = \{-k\}\times\intint{1,N+M}$, and let $\smash{\F^{\mrm{cB}}_k}$ (``cB'' short for colored $q$-Boson) be the $\sigma$-algebra
\begin{equation}\label{e.F^cB}
\smash{\F^{\mrm{cB}}_k} := \sigma\left(\bigl\{\one_{i_{v}\geq 1}: v\in \Z_{\leq 0}\times\intint{1,N+M}\bigr\}\cup\bigl\{i_v: v\in\Lambda_k\bigr\}\right).
\end{equation}
For $\bm A\in\Z_{\geq 0}^N$, let $|\bm A| := \sum_{i=1}^N A_i$. Note that by colored arrow conservation and consistency, the collection of random variables $\{(|\bm A_v|; |\bm B_v|, \one_{j_v\geq 1}) : v\in\Z_{\leq 0}\times\intint{1,N+M}\}$ and $\{(\bm A_v; \bm B_v, j_v) : v\in\Lambda_{k+1}\}$ is $\F^{\mrm{cB}}_k$-measurable. Thus, in words, $\smash{\F^{\mrm{cB}}_k}$ contains the information of the incoming colored arrow counts at $\Gamma_k$, incoming and outgoing total arrow counts of every vertex in the domain, and incoming and outgoing colored arrow counts of every vertex in $\Lambda_{k+1}$. 
See the right panel of Figure~\ref{f.vertex gibbs}.

The following is the Gibbs property when conditioning on $\F^{\mrm{cB}}_k$; because of the heterogeneous nature of the random variables being conditioned (a mixture of colored and total arrow counts), its proof is slightly more complicated than that of Lemma~\ref{l.uncolored vertex gibbs}. Recall the definition of $\bigL_{z,v}$ from before Lemma~\ref{l.uncolored vertex gibbs}.

\begin{lemma}[Colored $q$-Boson Gibbs property]\label{l.colored vertex gibbs}
Fix $k\in\N$ and $\sigma:\intint{1,N}\to\intint{1,N}$, and let $\Gamma_k$ and $\smash{\F^{\mrm{cB}}_k}$ be as in \eqref{e.F^cB}. Under the colored $q$-Boson measure with boundary condition $\sigma$, the conditional distribution of $((\bm A_v; \bm B_v, j_v):v\in\Gamma_k)$ given $\smash{\F^{\mrm{cB}}_k}$ is supported on tuples $((\smash{\bm A'_v; \bm B'_v, j'_v}): v\in\Gamma_k)$ such that (i) $\one_{j'_v \geq 1}=\one_{j_v\geq 1}$ for all $v\in\Gamma_k$ and (ii) colored arrow conservation and consistency are satisfied. Call the (random) set of all such tuples $\mrm{Supp}^{\mrm{cB}}$. Further, the conditional probability that $((\bm A_v; \bm B_v, j_v): v\in\Gamma_k) = ((\bm A'_v; \bm B'_v, j'_v): v\in\Gamma_k)$, for any fixed (i.e., deterministic) selection of the righthand side, is proportional to $\prod_{v\in \Gamma_k}\bigL_{z,v}(\bm A'_v, i_v; \bm B'_v, j'_v)\one_{(\bm A'_v; \bm B'_v, j'_v)_{v\in\Gamma_k} \in \mrm{Supp}^{\mrm{cB}}}$.
\end{lemma}

\begin{proof}
Let $\Lambda_{k}^c := \Z_{\leq 0}\times\intint{1,N+M}\setminus\Lambda_{k} = \intint{-k+1,0}\times\intint{1,N+M}$. Let $\bm\omega = ((\bm A'_v; \bm B'_v, j'_v): v\in\Gamma_k)$ and let $\mrm{Val}_k(\bm\omega)$ (for ``valid'') be the collection of configurations $(\bm A'_v, i'_v; \bm B'_v, j'_v)_{v\in \Lambda_k^c}$ which satisfy colored arrow conservation and consistency as well as the following:
\begin{itemize}
  \item $i'_{(-k+1, x)} = j'_{(-k,x)}$ for $x\in\intint{1,N+M}$ (consistency of incoming arrows to $\Gamma_{k-1}$ with outgoing arrows from $\Gamma_k$ as given by $\bm \omega$),

  \item $\bm B'_{(0,N+M)} = (\#\sigma^{-1}(\{j\}))_{j=1}^N$ (all arrows exit vertically from the rightmost column), and

  \item $\one_{i'_v\geq 1} = \one_{i_v\geq 1}$, $\one_{j'_v\geq 1} = \one_{j_v\geq 1}$, $|\bm A'_v| = |\bm A_v|$, and $|\bm B'_v| = |\bm B_v|$ for all $v\in \Lambda_k^c$ (respecting the total arrow counts conditioned on).
\end{itemize}
It is immediate from the definition of the colored $q$-Boson measure that the $\F^{\mrm{cB}}_k$-conditional probability of the event $\{((\bm A_v; \bm B_v, j_v): v\in\Gamma_k) = \bm \omega\}$ is proportional to
\begin{align}\label{e.a priori proportional}
\prod_{v\in \Gamma_k}\bigL_{z,v}(\bm A'_v, i_v; \bm B'_v, j'_v)\cdot \left(\sum_{(\bm A'_v, i'_v; \bm B'_v, j'_v)_{v\in\Lambda_k^c}\in\mrm{Val}_{k}(\bm\omega)}\prod_{w\in \Lambda_{k}^c}\bigL_{z,w}(\bm A'_w, i'_w; \bm B'_w, j'_w)\right).
\end{align}
To absorb the factor in parentheses into the normalization constant, we must show that it is $\smash{\F^{\mrm{cB}}_k}$-measurable. The lemma will therefore be implied by the following claim: that factor's value depends only on $((|\bm A_v|, \one_{i_v \geq 1}; |\bm B_v|, \one_{j_v\geq 1}): v\in\Lambda_{k}^c)$, which has been conditioned upon (it is determined by $(\one_{i_v\geq 1}: v\in \Lambda_k^c)$ by consistency and arrow conservation). 
This claim in turn follows by applying \eqref{e.color merging weights} (with $\tau(k) = 1$ for all $k\in\intint{1,N}$ and $\tau(0) = 0$) iteratively for each factor in the second product in \eqref{e.a priori proportional} (i.e., for each vertex in $\Lambda_k^c$); notice that the sum in \eqref{e.a priori proportional} is over $\mrm{Val}_k(\bm \omega)$, which fixes the total arrow counts at each vertex according to $((|\bm A_v|, \one_{i_v \geq 1}; |\bm B_v|, \one_{j_v\geq 1}): v\in\Lambda_{k}^c)$. This yields that the parenthetical quantity in \eqref{e.a priori proportional} equals the partition function for the uncolored $q$-Boson measure on configurations $(\tilde A_v, \tilde i_v; \tilde B_v, \tilde j_v)_{v\in \Lambda_k^c}$ (where $A_v, B_v\in\Z_{\geq 0}$ and $i_v, j_v\in\{0,1\}$ represent total arrow counts, as in the beginning of Section~\ref{s.intro.vertex model gibbs property}) constrained to satisfy the uncolored analogs of the three bullet points above, namely, $\tilde i_{(-k+1,x)} = \one_{j_{(-k,x)}\geq 1}$ for $x\in\intint{1,N+M}$; $\tilde B_{(0,N+M)} = N$; and $\tilde i_v = \one_{i_v\geq 1}, \tilde j_v = \one_{j_v\geq 1}$, $\tilde A_v = |\bm A_v|$, and $\tilde B_v = |\bm B_v|$ for all $v\in\Lambda_k^c$ (this actually includes the first point since $\one_{j_{(-k,x)}\geq 1} = \one_{i_{(-k+1,x)}\geq 1}$). This verifies the claim and completes the proof.
\end{proof}

\subsection{Colored Hall-Littlewood line ensemble} \label{s.colored line ensemble}
Next we define a rewriting of a configuration sampled from the colored $q$-Boson measure in a more probabilistically appealing form, namely, the colored line ensemble. First we specify what we mean by a discrete line ensemble.

\begin{definition}[Discrete line ensemble]\label{d.discrete line ensemble}
Fix a (possibly infinite) interval $\Lambda\subseteq \Z$.  A \emph{$\Lambda$-indexed discrete line ensemble} $\bm L = (L_1, L_2, \ldots)$ is a random variable defined on a probability space $(\Omega, \F, \P)$ taking values in the space of functions $\N\times\Lambda\to\Z$ (here the first argument is written as a subscript, i.e., $L_i(\bm\cdot) := \bm L(i,\bm\cdot)$), endowed with the discrete topology, such that: (i) $L_i(y+1) - L_i(y) \in\{0,-1\}$ for any $i \in \N$ and $y$ such that $y,y+1\in\Lambda$, and (ii) $L_{i}(y) \geq L_{i+1}(y)$ for any $y\in\Lambda$ and~$i\in\N$.
\end{definition}


We will refer to the set of vertices $\{-k\}\times\intint{1,N+M}$ in the domain $\Z_{\leq 0}\times\intint{1,N+M}$ of the colored $q$-Boson model as the $k$\th column. The colored line ensemble will be given in terms of colored arrow counts in columns at a finite distance from the rightmost (zeroth) column. Recall that, in the notation from Section~\ref{s.higher spin configuration}, for $y\in \intint{1, N+M}$, and $k\in\N$, $j_{(-k,y)}$ is the color of the arrow exiting horizontally from vertex $(-k, y)$; and $j_{(-k,y)} = 0$ if there is no such arrow.

\begin{definition}[Colored and uncolored Hall-Littlewood line ensembles]\label{d.colored line ensemble}
Let $\sigma:\intint{1,N}\to\intint{1,N}$ and let $\{(\bm A_{v}, i_v; \bm B_v, j_v):v\in\Z_{\leq 0}\times\intint{1,N+M}\}$ be a configuration sampled from the colored $q$-Boson measure with boundary condition $\sigma$ (from Section~\ref{s.higher spin boundary condition}). For each $k\in\intint{1,N}$, we define the discrete line ensemble $\bm L^{\mrm{cHL}, \smash{(k)}} = (L^{\mrm{cHL}, \smash{(k)}}_1, L^{\mrm{cHL}, \smash{(k)}}_2, \ldots) : \N\times\intint{0,N+M}\to\Z$ by letting $\smash{L^{\mrm{cHL}, (k)}_i(y)}$ be the number of arrows of color at least $k$ which exit horizontally in the $i$\textsuperscript{th} column strictly above row $y$, i.e,
\begin{equation}\label{e.colored line ensemble definition}
L^{\mrm{cHL}, (k)}_i(y) := \#\bigl\{y' > y : j_{(-i,y')} \geq k\bigr\}.
\end{equation}
The object $\bm L^{\mrm{cHL}} = (\bm L^{\mrm{cHL}, (1)}, \bm L^{\mrm{cHL}, (2)}, \ldots ,\bm L^{\mrm{cHL}, (N)})$ is called the \emph{colored Hall-Littlewood line ensemble} with boundary condition $\sigma$ (see Figure~\ref{0lmu}). The line ensemble $\bm L^{\mrm{cHL}, (1)}$ is called the \emph{uncolored Hall-Littlewood line ensemble}.
\end{definition}

\begin{remark}
The terminology comes from the fact that, if one interprets the colored Hall-Littlewood line ensemble as random sequences of compositions (see \cite[Definition 4.1]{aggarwalborodin}), the resulting process is the same  as the colored Hall-Littlewood process introduced in \cite[Section 1.6]{borodin2018coloured} (though we do not need or prove this). We also note that the uncolored Hall-Littlewood line ensemble $\bm L^{\mrm{cHL},(1)}$ was introduced in \cite{corwin2018transversal}.

We call $\bm L^{\mrm{cHL}, (1)}$ the uncolored Hall-Littlewood line ensemble as it is determined purely by total arrow counts in the colored $q$-Boson model. As in the uncolored $q$-Boson model, there is no boundary condition to be specified for the uncolored Hall-Littlewood line ensemble.
Analogous to the colored case, the uncolored Hall-Littlewood line ensemble is associated to the (uncolored) Hall-Littlewood process (see e.g., \cite{borodin2014macdonald,LRPP}), a collection of measures on sequences of partitions that is specified by a choice of specializations and a down-right path. The specializations (in some cases) can be thought of as our choices of spectral parameters, and the down-right path as the choice of boundary conditions of the uncolored $q$-Boson model (allowing arrows to horizontally enter only at rows with spectral parameter $1$, but which may be different than the bottom $N$ rows as we have fixed here; see also \cite[Remark~3.11]{aggarwalborodin}). The choice of boundary condition and spectral parameters we have fixed (and the only one considered in this article) is equivalent to the ``ascending Hall-Littlewood process'' with a homogeneous choice of specializations.
\end{remark}

Note that we do not include $\sigma$ in the notation for the colored Hall-Littlewood line ensembles, but we will specify $\sigma$ in the relevant statements. Next we record some simple consequences of Definition~\ref{d.colored line ensemble}.

\begin{lemma}\label{l.colored line ensemble basic properties}
Fix $\sigma:\intint{1,N}\to\intint{1,N}$, and for each $k\in\intint{1,N}$, let $\bm L^{\mrm{cHL}, (k)} = (L^{\mrm{cHL}, (k)}_1, L^{\mrm{cHL}, (k)}_2$, $\ldots)$ be defined as in \eqref{e.colored line ensemble definition} with boundary condition $\sigma$. It holds deterministically that, for any $i\in\N$ and $y\in\intint{0,N+M-1}$,
\begin{equation}\label{e.L^(j) properties}
\parbox{14.7cm}{
\begin{enumerate}[leftmargin=1cm, label=(\roman*)]
  \item $L^{\mrm{cHL}, (k)}_i(y) \geq L^{\mrm{cHL}, (k+1)}_i(y)$,

  \item $L^{\mrm{cHL}, (k)}_i(y) \geq L^{\mrm{cHL}, (k)}_{i+1}(y)$, and

  \item $L^{\mrm{cHL}, (k)}_i(y+1) - L^{\mrm{cHL}, (k)}_i(y) = -\one_{j_{(-i,y+1)}\geq k} \in \{0,-1\}$.
\end{enumerate}}
\end{equation}
\end{lemma}

\begin{proof}
The first and third points are immediate from the definition. The second relies on the fact that arrows can only move up and to the right, so that any arrow exiting the $(i+1)$\textsuperscript{st} column (i.e., column $-(i+1)$) at a row higher than $y$ must exit the $i$\textsuperscript{th} column at a row higher than $y$.
\end{proof}

We next record a statement of color merging for the colored Hall-Littlewood line ensemble.

\begin{lemma}[Color merging for $\bm L^{\mrm{cHL}}$]\label{l.color merging cHL}
Fix $\sigma:\intint{1,N}\to\intint{1,N}$ and $k\in\intint{1,N}$, and let $\tau:\intint{1,N}\to\intint{1,k}$ be surjective and non-decreasing. Let $\bm L^{\mrm{cHL}, \sigma}$ and $\bm L^{\mrm{cHL}, \tau\circ\sigma}$ be the colored Hall-Littlewood line ensembles with boundary conditions $\sigma$ and $\tau\circ\sigma$ as in Definition~\ref{d.colored line ensemble}, respectively. For $j\in\intint{1,k}$, let $t_j = \min \tau^{-1}(\{j\})$. Then
\begin{align*}
(\bm L^{\mrm{cHL}, \sigma, (t_j)})_{j\in\intint{1,k}} \stackrel{d}{=} (\bm L^{\mrm{cHL}, \tau\circ\sigma, (j)})_{j\in\intint{1,k}}.
\end{align*}
In particular, defining $\tau$ by $\tau(j) = 1$ for all $j\in\intint{1,N}$, it follows that $\bm L^{\mrm{cHL}, \sigma, (1)}$  under any boundary condition $\sigma$ has the same distribution as $\bm L^{\mrm{cHL}, \sigma', (1)}$ with $\sigma'(j) = 1$ for all $j\in\intint{1,N}$, i.e., the colored line ensemble associated with the uncolored $q$-Boson model.
\end{lemma}

\begin{proof}
This is an immediate consequence of combining Definition~\ref{d.colored line ensemble} with the color merging statement Lemma~\ref{l.q-Boson color merging}.
\end{proof}

\subsubsection{Restricting the domain}\label{s.line ensemble to colored S6V} We need to connect the top curve of $\smash{\bm L^{\mrm{cHL}, (j)}}$ for different $j$ with the height function of the colored S6V model (we focus on S6V since we will ultimately deal with ASEP via a coupling with it from the former; see Section~\ref{s.properties of S6V and ASEP}). To do so, we must first modify the setup of the colored S6V model slightly. Recall from \eqref{e.s6v height function} that the prelimiting analog $\S^{\mrm{S6V}, \varepsilon}$ of $\S$ concerns counts of the number of arrows above a varying vertex on the vertical line $x=\floor{\varepsilon^{-1}}$, and that arrows horizontally enter the system from the left side at coordinates $(1,i)$ for $i\in\intint{-N,N}$. In the packed boundary condition, the arrow entering horizontally at $(1,i)$ has color $i$, and thus all arrow colors entering at or below the $x$-axis are not positive. In contrast, note that the colored Hall-Littlewood line ensemble we defined are associated with only positive colors. (We needed to allow arrows of negative color in the colored S6V model, as  $\S^{\mrm{S6V},\varepsilon}$ needs to allow negative values in the first argument.)

We first record a consequence of color merging (of the colored S6V model) that says that, when considering the height function for colors $1$ or higher, we may restrict the domain to $\Z_{\geq 1}\times\intint{1,N}$ without changing the distribution of the colored height function. This will address the discrepancy pointed out in the previous paragraph. Looking at arrows of only positive color will suffice for our purposes of establishing convergence to the Airy sheet since the second part of (ii) of Definition~\ref{d.airy sheet} of the latter only involves $x\in\Q^+$.

\begin{lemma}\label{l.can look at positive colors only}
Fix $N\in\N$ and let $\sigma:\intint{-N,N}\to\llparen {-}\infty, N\rrbracket\cup\{-\infty\}$, $\sigma_{\geq 1}:\intint{1,N}\to\intint{1,N}$ be such that $\sigma(k) = \sigma_{\geq 1}(k) \geq 1$ for $k \in \intint{1,N}$ and $\sigma(k)\leq 0$ for $k\in\intint{-N,0}$. Consider the colored S6V model on the domains $\Z_{\geq 1}\times\llbracket -N, \infty\rrparen$ and $\Z_{\geq 1}\times \llbracket1,\infty\rrparen$ with boundary conditions $\sigma$ and $\sigma_{\geq 1}$, respectively (i.e., an arrow of color $\sigma(k)$ or $\sigma_{\geq 1}(k)$ enters horizontally at $(1,k)$ for $k\in\intint{-N,N}$ and $\intint{1,N}$, respectively), and let $h^{\mrm{S6V}}$ and $h^{\mrm{S6V}}_{\geq 1}$ be their respective colored height functions as defined in \eqref{e.s6v height function}. Then $(h^{\mrm{S6V}}(x,0; y,t))_{x,y, t\in\N}\stackrel{\smash d}{=} (h^{\mrm{S6V}}_{\geq 1}(x,0; y,t))_{x,y, t\in\N}$.
\end{lemma}

\begin{proof}
This is an immediate consequence of color merging for colored S6V (Remark~\ref{r.s6v color merging}): in the colored S6V model on $\Z_{\geq 1}\times\llbracket -N, \infty\rrparen$ with boundary condition $\sigma$ in the statement, we merge the colors $\llparen {-}\infty,0\rrbracket\cup\{-\infty\}$ with $-\infty$ and thus obtain the colored S6V model on $\Z_{\geq 1}\times\llbracket 1,\infty\rrparen$ with boundary condition $\sigma_{\geq 1}$.
\end{proof}

As noted, the convergence of the S6V and ASEP sheets to $\S$ is a statement about the colored height functions at a fixed time, so to prove Theorems~\ref{t.asep airy sheet} and \ref{t.s6v airy sheet} we can restrict our attention to the colored height function at time $M$ for fixed $M$, i.e., $h^{\mrm{S6V}}(\bm\cdot, 0; \bm\cdot, M)$.
Note that we switched the notation for the time argument from $t$ to $M$; this is so as to match the notation of the colored $q$-Boson model as well as \cite{aggarwalborodin}.  

\subsubsection{Relation to colored S6V} With these preliminaries, we may state the crucial link between the colored Hall-Littlewood line ensemble and the height functions in colored S6V. Its proof, originally due to \cite{aggarwalborodin}, relies on the Yang-Baxter equation of colored S6V and the colored $q$-Boson model. Because our arguments fundamentally rely on this result, we give a self-contained proof in Appendix~\ref{s.yang-baxter}.

\begin{proposition}[{\cite[Theorem 4.7]{aggarwalborodin}}]\label{p.colored line ensembles}
Fix $N,M\in\N$, $\sigma:\intint{1,N}\to\intint{1,N}$, $q\in[0,1)$, and $z\in(0,1)$. Let $h^{\mrm{S6V}}$ be the height function associated with the colored S6V on $\Z_{\geq  1} \times \llbracket 1, \infty\rrparen$ as in \eqref{e.s6v height function} and $\bm L^{\mrm{cHL}, (j)}$ be as in Definition~\ref{d.colored line ensemble} (with the associated colored $q$-Boson model on domain $\Z_{\leq 0}\times\intint{1,N+M}$), both with boundary condition $\sigma$. We have the equality in distribution of $(\hssv(k, 0; y,M))_{y\in\intint{1,N-1},k\in\intint{1,N}}$ and $(L^{\mrm{cHL}, \smash{(k)}}_1(y))_{y\in\intint{1,N-1}, k\in\intint{1,N}}$.
\end{proposition}

Next we move to describing a Gibbs property of the colored Hall-Littlewood line ensemble.

\subsubsection{Colored Hall-Littlewood Gibbs property} From the definition of the colored Hall-Littlewood line ensemble in terms of the colored $q$-Boson model, and the Gibbs property of the latter recorded in Lemma~\ref{l.colored vertex gibbs}, we obtain a Gibbs property for the former. We restrict ourselves to stating the case where the system contains only two colors, i.e., $\sigma:\intint{1,N}\to\intint{1,2}$, as this will turn out to suffice for our purposes (by performing a color merging later, see Section~\ref{s.approximate LPP}), though more general Gibbs properties involving more colors also hold; indeed, in a system with $n$ colors, the joint distribution of $(\bm L^{\mrm{cHL}, (1)}, \bm L^{\mrm{cHL}, (j)})$ for some $j\geq 2$ is the same as the joint distribution of $(\bm L^{\mrm{cHL}, (1)}, \bm L^{\mrm{cHL}, (2)})$ in the new system obtained after merging colors  in $\intint{1,j-1}$ to be 1 and in $\intint{j,N}$ to be 2. In this case of only two colors, $\bm L^{\mrm{cHL}, (1)}$ and $\bm L^{\mrm{cHL}, (2)}$ determine the colored arrow configurations at all vertices in the system via Definition~\ref{d.colored line ensemble}, and their coupling is captured in Theorem~\ref{t.approxmate LPP problem representation general} and analysed in detail in Section~\ref{s.approximate LPP} using the following colored Gibbs property. Let $\smash{\F^{\mrm{cHL}}_k}$ (``cHL'' short for colored Hall-Littlewood) be the $\sigma$-algebra
\begin{align}\label{e.F^cHL}
\smash{\F^{\mrm{cHL}}_k} = \sigma\left(\bigl\{\smash{L^{\mrm{cHL},(1)}_i}(x), \smash{L^{\mrm{cHL},(2)}_j}(x) : i\in\N, j\in\llbracket k+1, \infty\rrparen, x\in\intint{0,N+M}\bigr\}\right);
\end{align}
in words, we condition on all curves of the line ensemble $\bm L^{\mrm{cHL},(1)}$, and the curves indexed $k+1$ or larger of $\bm L^{\mrm{cHL},(2)}$.
We now state the Gibbs property of the colored Hall-Littlewood line ensemble that we will need. Recall  Definition \ref{d.bernoulli path} of Bernoulli paths, the notation for arrow configurations in the colored $q$-Boson model from Section~\ref{s.higher spin configuration}, and the notation $\bigL_{z,v}$ from before Lemma~\ref{l.uncolored vertex gibbs}.

\begin{lemma}[Colored Hall-Littlewood Gibbs property]\label{l.colored line ensemble gibbs}
\linespread{1.1}\selectfont{Under the colored Hall-Littlewood line ensemble measure (with boundary condition $\sigma:\intint{1,N}\to\intint{1,2}$), the conditional distribution of $\smash{L^{\mrm{cHL},(2)}_k}(\bm\cdot)$ given $\smash{\F^{\mrm{cHL}}_k}$ is supported on the collection of Bernoulli paths $L':\intint{0,N+M}\to\Z$ such that (i) $\smash{L^{\mrm{cHL},(1)}_k}(\bm\cdot) - L'(\bm\cdot)$ is a Bernoulli path and (ii) $L'(\bm\cdot) \geq L^{\mrm{cHL},(2)}_{k+1}(\bm\cdot)$ with $L'(N+M) = 0$ and $L'(0) = L^{\smash{\mrm{cHL},(2)}}_{k+1}(0)$; call the (random) set of such Bernoulli paths $\mrm{Supp}^{\mrm{cHL}}$. 

For such a Bernoulli path $L'$, let $(\bm A'_v, i_v; \bm B'_v, j_v')_{v\in\Gamma_k}$ be the associated tuple determined by $L'(\bm\cdot)$, $L^{\mrm{cHL},\smash{(1)}}_{k}(\bm\cdot)$, $L^{\mrm{cHL},\smash{(1)}}_{k+1}(\bm\cdot)$, and $L^{\mrm{cHL},\smash{(2)}}_{k+1}(\bm \cdot)$ on the event that $L^{\mrm{cHL}, \smash{(2)}}_k(\bm\cdot) = L'(\bm\cdot)$ (we write $i_v$ and not $i_v'$ as it has no dependence on $L'(\bm\cdot)$). The condition $L'(\bm\cdot)\in\mrm{Supp}^{\mrm{cHL}}$ is equivalent to $(\bm A'_v, i_v; \bm B'_v, j'_v)_{v\in\Gamma_k}\in \mrm{Supp}^{\mrm{cB}}$ (the latter as in Lemma~\ref{l.colored vertex gibbs}). Further, the conditional probability that $L^{\mrm{cHL},(2)}_k(\bm\cdot) = L'(\bm\cdot)$ , for any fixed (i.e., deterministic) selection of the righthand side, is proportional to $\prod_{v\in \Gamma_k} \bigL_{z,v}(\bm A'_v, i_v; \bm B'_v, j'_v)\one_{(\bm A'_v; \bm B'_v, j'_v)_{v\in\Gamma_k}\in \mrm{Supp}^{\mrm{cB}}}$.}
\end{lemma}

\begin{proof}
The lemma follows from Lemma~\ref{l.colored vertex gibbs} and Definition~\ref{d.colored line ensemble} relating the color Hall-Littlewood line ensemble and the arrow configurations of the colored $q$-Boson model; note that (i) and (ii) in the hypotheses of Lemma~\ref{l.colored line ensemble gibbs} correspond to (i) and (ii) in Lemma~\ref{l.colored vertex gibbs} (for (i), using that $L^{\smash{\mrm{cHL},(1)}}_k(\bm\cdot) - \smash{L^{\mrm{cHL},(2)}_k}(\bm\cdot)$ decreases by $1$ from $x$ to $x+1$ exactly at those $x$ such that an arrow of color $1$ exits horizontally from $(-k,x)$, by Lemma~\ref{l.colored line ensemble basic properties} (iii)).
\end{proof}

There is also a Gibbs property for the uncolored Hall-Littlewood line ensemble, i.e., $\bm L^{\mrm{cHL},(1)}$, obtained by translating Lemma~\ref{l.uncolored vertex gibbs}. We will discuss this in more detail in Section~\ref{s.line ensemble convergence} next.

\subsection{Convergence of line ensembles} \label{s.line ensemble convergence}
In this section we discuss convergence results for a class of discrete line ensembles possessing a certain Gibbs property (that of the uncolored Hall-Littlewood line ensemble).
The statement we give for the limit of line ensembles is under a  general framework which we introduce now. Because of the general setup, these line ensembles are not necessarily related to the $q$-Boson vertex model; having said that, as mentioned, the example the framework will be applied to in this paper is the uncolored Hall-Littlewood line ensemble $\bm L^{\mrm{cHL}, (1)}$.

\subsubsection{Hall-Littlewood Gibbs property} The line ensembles we consider will possess the \emph{Hall-Littlewood Gibbs property}, which describes the interactions that curves in the ensemble have with each other. 
We start by precisely specifying the data we condition on in the Gibbs property.

\begin{definition}\label{d.F_ext}
Let $\bm L:\N\times\Lambda\to\Z$ be a discrete line ensemble (recall Definition~\ref{d.discrete line ensemble}). Let $\llbracket j,k\rrbracket \subset\N$ and $\llbracket a,b \rrbracket\subset \Lambda$ be intervals. We define  $\Fext(\llbracket j,k\rrbracket, \llbracket a,b\rrbracket, \bm L)$ to be the $\sigma$-algebra generated by the collection of random variables $\{L_i(x): (i,x)\in \N\times\Lambda\setminus (\intint{j,k}\times\intint{a+1,b-1})\}$.

When the line ensemble under discussion is clear, we will drop $\bm L$ from the notation. Similarly when $j=1$, we will replace $\intint{1,k}$ by simply $k$. Thus $\Fext(k, \intint{a,b})$ would denote $\Fext(\intint{1,k}, \intint{a,b}, \bm L)$.

\end{definition}

Next, recall the definition of a Bernoulli path from Definition~\ref{d.bernoulli path}.  For $a,b,x,y\in\Z$ with $a< b$, $y \leq x \leq y + b-a $, we define a \emph{Bernoulli random walk bridge} $B:\intint{a,b}\to\Z$ from $(a,x)$ to $(b,y)$ to be a uniformly selected random path from the set of all Bernoulli paths $\gamma:\intint{a,b}\to\Z$ with $\gamma(a) = x$ and $\gamma(b)=y$ (assuming this set is non-empty).

\begin{definition}[Weight factor]\label{d.weight factor}
Let $\llbracket a,b\rrbracket\subseteq \Z$ and fix $q\in[0,1)$. Let $\bm\gamma = (\gamma_1, \ldots, \gamma_k)$, $f$, and $g$ be a collection of Bernoulli paths ($f$ and $g$ should be thought of as upper and lower boundary curves respectively), all defined on $\llbracket a,b\rrbracket$, with $\gamma_i(x) \geq \gamma_{i+1}(x)$ for $i\in\intint{1,k-1}$, $\gamma_1(x)\leq f(x)$, and $\gamma_{k}(x)\geq g(x)$ for all $x\in\llbracket a, b\rrbracket$. We define the weight factor $W(\bm\gamma, f, g, \llbracket a,b\rrbracket)$ by
\begin{align}\label{e.weight function}
W(\bm\gamma, f, g, \llbracket a,b\rrbracket) := \prod_{i=0}^k \prod_{x=a+1}^b \left(1- q^{\Delta_i(x-1)}\one_{\Delta_i(x) = \Delta_i(x-1) -1}\right),
\end{align}
where $\Delta_i(x) = \gamma_i(x) - \gamma_{i+1}(x)$ and $\gamma_0 = f$ and $\gamma_{k+1} = g$ by convention. For curves which do not satisfy the ordering conditions mentioned above, we define $W(\bm\gamma, f, g, \intint{a,b}) = 0$; observe that \eqref{e.weight function} equals 0 if at any location the separation between consecutive curves drops from $0$ to $-1$. See also Figure~\ref{f.HL Gibbs}.
When the interval under consideration is obvious from the context we will drop $\intint{a,b}$ from the notation. Moreover, if $f = \infty$, we drop it from the notation; if in addition $g = -\infty$, then we drop it from the notation as well.

Sometimes, to emphasize whether we are focusing on interaction with the upper curve $f$, between the curves of $\bm \gamma$, or with the lower curve $g$, we will write 
\begin{equation}\label{e.W up int low split up}
W(\bm\gamma, f, g) = W_{\mrm{up}}(\gamma_1, f)\cdot W_{\mrm{int}}(\bm \gamma)\cdot W_{\mrm{low}}(\gamma_{k}, g),
\end{equation}
where,
\begin{equation}\label{e.W up int low definition}
\begin{split}
W_{\mrm{up}}(\gamma_1, f) &= \prod_{x=a+1}^b \left(1- q^{\Delta_0(x-1)}\one_{\Delta_0(x) = \Delta_0(x-1) -1}\right),\\
W_{\mrm{int}}(\bm \gamma) &= \prod_{i=1}^{k-1}\prod_{x=a+1}^b \left(1- q^{\Delta_i(x-1)}\one_{\Delta_i(x) = \Delta_i(x-1) -1}\right), \text{ and}\\
W_{\mrm{low}}(\gamma_k, g) &= \prod_{x=a+1}^b \left(1- q^{\Delta_k(x-1)}\one_{\Delta_k(x) = \Delta_k(x-1) -1}\right),
\end{split}
\end{equation}
i.e., in \eqref{e.weight function}, $W_{\mrm{up}}(\gamma_1, f)$ is the factor corresponding to $i=0$, $W_{\mrm{int}}(\bm \gamma)$ is the product corresponding to $i=1, \ldots, k-1$, and $W_{\mrm{low}}(\gamma_k, g)$ to the factor with $i=k$.
\end{definition}

The following is the Gibbs property that the line ensembles in the general framework we are setting up will possess; it originally appeared and was studied in \cite{corwin2018transversal}.

\begin{definition}(Hall-Littlewood Gibbs property)\label{d.HL Gibbs}
We say a discrete line ensemble $\bm L : \N\times\Lambda\to \Z$ satisfies the (uncolored) \emph{Hall-Littlewood Gibbs property} (with parameter $q\in[0,1)$) if the following holds for any $\intint{j, k} \subset \N$ and $\intint{a, b}\subset \Lambda$. Conditionally on $\F=\Fext(\intint{j,k}, \intint{a,b}, \bm L)$, the distribution of $\bm L|_{\intint{j,k}\times \intint{a,b}}$ is given by a collection of $k-j+1$ Bernoulli random walk bridges $\bm B=(B_j, \ldots, B_k)$, where $B_i(a) = L_i(a)$ and $B_i(b) = L_i(b)$ for each $i\in\intint{j, k}$, reweighted by the Radon-Nikodym derivative $W(\bm B, L_{j-1}, L_{k+1})/\EF[W(\bm B, L_{j-1}, L_{k+1})]$, where $L_0\equiv \infty$ by convention.

Equivalently, for any bounded function $F:\Z^{\intint{j,k}\times\intint{a,b}} \to \R$,
\begin{align*}
\EF\left[F(L_j|_{\intint{a,b}}, \ldots, L_k|_{\intint{a,b}})\right] = \frac{\EF\left[F(B_j, \ldots, B_k)W(\bm B, L_{j-1}, L_{k+1})\right]}{\EF\left[W(\bm B, L_{j-1}, L_{k+1})\right]}
\end{align*}
(here $\EF[\bm\cdot]$ is shorthand for $\E[\bm\cdot\mid \F]$, and we will also use analogous notation $\PF(\bm\cdot)$ for $\P(\bm\cdot\mid \F)$).
We will sometimes write $Z = \EF[W(\bm B, L_{j-1}, L_{k+1})]$, but this will always be made explicit.
\end{definition}

As above, we will usually refer to this property as the Hall-Littlewood Gibbs property (dropping the ``uncolored'' adjective); this will not cause confusion as we will always explicitly say ``colored'' when referring to its colored analog.

This Gibbs property is satisfied by $\smash{\bm L^{\mrm{cHL}, (1)}}$, but not in general by $\bm L^{\mrm{cHL}, (j)}$ for $j\geq 2$.

\begin{proposition}\label{p.L has HL}
$\bm L^{\mrm{cHL}, (1)}|_{\N\times\intint{0,N-1}} = (L^{\mrm{cHL}, (1)}_1|_{\intint{0,N-1}}, L^{\mrm{cHL}, (1)}_2|_{\intint{0,N-1}}, \ldots)$ has the Hall-Littlewood Gibbs property for any boundary condition $\sigma:\intint{1,N}\to\intint{1,N}$.
\end{proposition}

\begin{proof}
This is a direct consequence of Lemma~\ref{l.uncolored vertex gibbs} after translating to the line ensemble notation from Definition~\ref{d.colored line ensemble}, using also from Lemma~\ref{l.color merging cHL} that $\bm L^{\mrm{cHL}, (1)}$ is in distribution the line ensemble associated to the uncolored $q$-Boson model.
\end{proof}

\begin{remark}
Observe that the Hall-Littlewood Gibbs property does not involve the underlying spectral parameter $u$ of the uncolored $q$-Boson weights $\mathtt L_u$. This is because of the fact that in the uncolored case (so $i,j\in\{0,1\}$ and we consider $A,B\in\Z_{\geq 0}$ rather than vectors $\bm A, \bm B$), $\mathtt{L}_u(A, i; B, j) = u^{j}\cdot\mathtt{L}_1(A, i; B, j)$ (see Figure~\ref{f.L weights}). The Hall-Littlewood Gibbs property corresponds to conditioning the uncolored $q$-Boson model outside of a rectangular domain, and the overall exponent of $u$ in the weight of the configuration inside the domain (i.e., the value of $W$ up to normalization) is determined by the boundary conditions as a result of the observation. Thus, the potential $u$ dependence in the Gibbs weight is cancelled by the normalization constant $Z$.
\end{remark}

A derivation of Proposition~\ref{p.L has HL} is also contained in \cite[Proposition~5.1]{aggarwalborodin}, and an alternate derivation is present in \cite[Proposition~3.9]{corwin2018transversal} after matching $\bm L^{\mrm{cHL},(1)}$ with the (uncolored) Hall-Littlewood process.

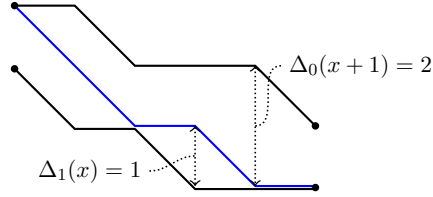
\begin{figure}
\begin{tikzpicture}[scale=0.8]
\draw[thick] (0,0) -- ++(1,0) -- ++(1,-1) -- ++(1,0) -- ++(1,0) -- ++(1,-1);
\draw[thick, blue] (0,0) -- ++(1,-1) -- ++(1,-1) -- ++(1,0) -- ++(1,-1) -- ++(1,0);
\draw[thick] (0,-1.05) -- ++(1,-1) -- ++(1,0) -- ++(1,-1) -- ++(2,0);

\node[circle, fill, inner sep = 1pt] at (0,0) {};
\node[circle, fill, inner sep = 1pt] at (0,-1.05) {};
\node[circle, fill, inner sep = 1pt] at (5,-2) {};
\node[circle, fill, inner sep = 1pt] at (5,-3.025) {};

\draw[semithick, <->, densely dotted] (3,-2) -- ++(0,-1);
\draw[semithick, <->, densely dotted] (4,-1) -- ++(0,-2);

\useasboundingbox (current bounding box);

\node[scale=0.8] (A0) at (5.75, -1) {$\Delta_0(x+1) = 2$};
\node[scale=0.8] (A1) at (1.25, -2.75) {$\Delta_1(x) = 1$};

\draw[densely dotted, semithick] (A1) to[out=0, in=180] (3,-2.5);
\draw[densely dotted, semithick] (A0) to[out=180, in=0] (4,-2);
\end{tikzpicture}
\caption{An example of the Hall-Littlewood weight factor with a single curve $\gamma$. The top and bottom curves (in black) are $f$ and $g$, respectively, and are fixed, and the blue curve is $\gamma$.
  In calculating the weight factor $W$, the separation of $\gamma$ from $f$ and $g$ as captured by $\Delta_i(\bm\cdot)$ is used. In this example, $W(\gamma, f, g) = (1-q^2)\cdot(1-q)^2$: the separation $\Delta_0(\bm\cdot)$ between $f$ and $\gamma$ reduces only at the penultimate step, from $2$ to $1$, and the separation $\Delta_1(\bm\cdot)$ between $\gamma$ and $g$ reduces twice, each time from $1$ to $0$.}\label{f.HL Gibbs}
\end{figure}

\subsubsection{Heuristic behavior of the Gibbs property} \label{ssub:heuristic_behavior_of_the_gibbs_property}For intuition, the Gibbs property should be thought of as a resampling property, as in Figure~\ref{f.para airy and bg} ahead: it says that the following resampling procedure leaves the distribution of the line ensemble unchanged. First, with $\intint{j,k}\subseteq \N$ and $\intint{a,b}\subseteq \Lambda$ fixed, we erase $\bm L$ on $\intint{j,k}\times\intint{a+1,b-1}$ while retaining the information of its values on $\N\times\Lambda\setminus(\intint{j,k}\times\intint{a+1,b-1})$. Then we attempt to replace $L_j, \ldots, L_k$ on $\intint{a,b}$ with a tuple $\bm B$ of $k-j+1$ independent Bernoulli random walk bridges with endpoints of $B_i$ equaling $L_i(a)$ and $L_i(b)$ at $a$ and $b$, respectively: the attempt succeeds with probability $W(\bm B, L_{j-1}, L_{k+1})$ as given in Definition~\ref{d.weight factor}, independent of everything else; note that if any individual curve of $\bm B$ crosses $L_{j-1}$, $L_{k+1}$, or any other curve of $\bm B$, then $W(\bm B, L_{j-1}, L_{k+1}) = 0$ and the attempt certainly fails. On the event of $\bm B$ successfully replacing $L_j, \ldots, L_k$ on $\intint{a,b}$, the former's law is exactly that of $k$ independent Bernoulli random walk bridges under the reweighting by the Radon-Nikodym derivative $W(\bm B, L_{j-1}, L_{k+1})/\EF[W(\bm B, L_{j-1}, L_{k+1})]$. The basic idea underlying many of our arguments is that behavior which is unlikely for such reweighted Bernoulli random walk bridges must also be unlikely for the original line ensemble.

The above is really a perspective on Gibbs properties of line ensembles in general, and now we discuss the particular features of the Hall-Littlewood Gibbs property that constitute a major challenge and necessitate the development of new techniques, compared to previous studies of line ensembles with different Gibbs properties such as \cite{dimitrov2021tightnessH,SERIO2023225,dimitrov2021tightness}. In our setup, we will consider sequences of line ensembles indexed by a parameter $N$ (which should not be confused with the $N$ parameter of the colored S6V model or the colored $q$-Boson model in Section~\ref{s.intro.vertex model}). Our goal in Section~\ref{s.line ensemble convergence} is to give a result of convergence of these line ensembles, rescaled by the KPZ exponents, to the parabolic Airy line ensemble. This means that we scale the spatial argument of the line ensemble by $N^{2/3}$ and the value of the lines by $N^{1/3}$ (both up to constants) in order to obtain a tight, or unit-order, sequence.

At short range, the resampling property penalizes curves coming closer together, as encoded by the indicator function in \eqref{e.weight function}. This is a somewhat complicated interaction as it is not always active, and, more importantly, does not satisfy a certain convenient stochastic monotonicity property. Before discussing these aspects, let us first consider how the interaction behaves on the larger scale on which our limits will be taken. This means evaluating the weight function \eqref{e.weight function} when $\intint{a,b}$ is an interval of size of order $N^{2/3}$ and $\Delta_i(x)$ is of order $N^{1/3}$: then $W$ is $\smash{(1-q^{\Theta(N^{1/3})})^{\Theta(kN^{2/3})}}$, i.e., essentially 1 when $q\in[0,1)$ is fixed. This assumes the separation between consecutive curves is of order $N^{1/3}$ throughout the interval; if instead at some location consecutive curves cross, the weight factor becomes zero as before. Thus at large $N$ the weight factor is expected to behave like the indicator function of the curves not intersecting, mimicking the Brownian Gibbs property (see Definition~\ref{d.bg}) for the parabolic Airy line ensemble (the putative limit of our discrete one).

\subsubsection{Lack of stochastic monotonicity}\label{s.lack of monotonicity discussion}
 Actually establishing this separation (that the above heuristic relies on) is non-trivial and will constitute the bulk of the proof of our convergence results. Some challenges are exactly the earlier mentioned complicated nature of the interaction at short range, from which we must argue we can escape to a larger scale, and that the Hall-Littlewood Gibbs property does not possess a stochastic monotonicity property present in other Gibbs properties previously studied. Stochastic monotonicity means that if one considers the law of the curves of the line ensemble inside $\intint{j,k}\times\intint{a,b}$ conditional on the boundary data, and then considers the same law with the boundary points increased and/or the upper and lower curves increased, the second law is stochastically larger than the first (i.e., there is a coupling of the two laws under which each curve in the second collection, with higher boundary data, is pointwise larger than the corresponding curve in the first collection). The Hall-Littlewood Gibbs property fails to satisfy this because $W$ from \eqref{e.weight function} is a monotone function of the discrete derivative of the distance between consecutive curves and not of the distance itself. See Figure~\ref{f.stoch mono failure} for an example demonstrating the lack of stochastic monotonicity. The violations of stochastic monotonicity can be fairly severe and it is also possible to find counterexamples to versions of stochastic monotonicity with more conditions imposed, e.g., comparing only boundary conditions which share the same endpoints.

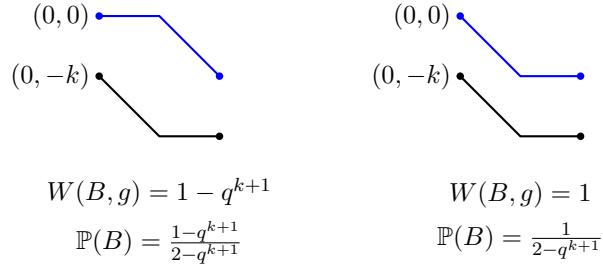
\begin{figure}[h!]
\begin{tikzpicture}[scale=0.8]
\draw[thick, blue] (0,1) -- ++(1,0) -- ++(1,-1);
\draw[thick] (0,0) -- ++(1,-1) -- ++(1,0);

\node[anchor=east, scale=0.9] at (0,1) {$(0,0)$};
\node[circle, fill, inner sep=1pt, blue] at (0,1) {};
\node[circle, fill, inner sep=1pt, blue] at (2,0) {};

\node[anchor=east, scale=0.9] at (0,0) {$(0,-k)$};
\node[circle, fill, inner sep=1pt] at (0,0) {};
\node[circle, fill, inner sep=1pt] at (2,-1) {};

\node[scale=0.9, align=center] at (1,-2.4) {$W(B,g) = 1-q^{k+1}$\\[5pt]$\P(B) = \frac{1-q^{k+1}}{2-q^{k+1}}$};

\begin{scope}[shift={(6,0)}]
\draw[thick, blue] (0,1) -- ++(1,-1) -- ++(1,0);
\draw[thick] (0,0) -- ++(1,-1) -- ++(1,0);

\node[anchor=east, scale=0.9] at (0,1) {$(0,0)$};
\node[circle, fill, inner sep=1pt, blue] at (0,1) {};
\node[circle, fill, inner sep=1pt, blue] at (2,0) {};

\node[anchor=east, scale=0.9] at (0,0) {$(0,-k)$};
\node[circle, fill, inner sep=1pt] at (0,0) {};
\node[circle, fill, inner sep=1pt] at (2,-1) {};

\node[scale=0.9, align=center] at (1,-2.4) {$W(B,g) = 1$\\[5pt]$\P(B) = \frac{1}{2-q^{k+1}}$};
\end{scope}
\end{tikzpicture}
\caption{An example of the failure of stochastic monotonicity for the Hall-Littlewood Gibbs property. Here $g$ is the lower boundary curve drawn in black (indexed by its starting location $(0,-k)$) and $B$ is the random curve drawn in blue (whose endpoints are fixed). We see that the probability of the higher possibility for $B$ equals $\smash{\frac{1-q^{k+1}}{2-q^{k+1}}}$, which is increasing in $k$, while stochastic monotonicity requires it to be non-increasing.}\label{f.stoch mono failure}
\end{figure}

The approach to line ensemble tightness we develop relies instead on a form of weak monotonicity for a single curve (a minor refinement of a notion introduced in \cite{corwin2018transversal}); see Section~\ref{s.monotonicity}. By this we mean a comparison of partition functions (up to a, possibly large, constant) of certain pairs of laws corresponding to a single curve under different boundary conditions, which translates to a comparison (up to the same constant) of one-point probabilities for a single curve. In particular, it does not hold for collections of more than one curve, and does not imply the existence of a coupling of the two laws such that all points on the pair of curves are ordered. Using this weak monotonicity, we iteratively establish uniform separation of order $N^{1/3}$  between curves further and further down in the line ensemble, simultaneously establishing control on how far down such curves can be. We will discuss and develop this more later (see Section~\ref{s.lower tail}); here we next introduce the general framework under which our arguments~operate.

\subsubsection{Assumptions on line ensembles and their tightness} \label{s.line ensemble hypotheses}

In this section we formulate some assumptions on general line ensembles under which we will prove tightness.
Before stating them, we specify the space we will work on and some conventions about the objects.

Our sequence of discrete line ensembles will be denoted by $\bm{L}^N = (L_1^N, L_2^N, \ldots )$. Here $N$ is an indexing parameter which has no relation to the parameter in the colored S6V or colored $q$-Boson models.
Though discrete valued, we will regard $\bm L^N$ as an element of $\mc C(\N\times\R, \R)$, by linearly interpolating $L^N_i(x)$ for $x\in [j,j+1]$ for each $i\in\N$ and $j\in\Z$; here $\N\times\R$ is given the product topology of the discrete topology on $\N$ and the usual Euclidean topology on $\R$. We endow the space $\mc C(\N\times\R, \R)$ with the topology of uniform convergence on compact sets, i.e., $\smash{\bm f^{(n)}}, \bm f \in \mc C(\N\times\R,\R)$ satisfy $\smash{\bm f^{(n)}}\to \bm f$ if for any compact set $K\subseteq \R$ and $i\in\N$ fixed, $\smash{f^{(n)}_i}(x)\to f_i(x)$ uniformly over~$x\in K$.

Further, we allow a sequence $g_N$ going to $\infty$ with $N\to\infty$ such that the values of $L^N_i(x)$ are not specified for $|x|\geq g_NN^{2/3}$, as these values do not affect the existence or value of the limits we consider (which are determined by uniform convergence on compact sets on the $N^{2/3}$ spatial scale); in particular, $L^N_i(x)$ may be defined for such values in any way such that $\bm L^N$ is an element of $\mc C(\N\times\R, \R)$, without regard for satisfying the assumptions below.

Now we may turn to the precise assumptions we assume on our line ensembles $\bm L^N=(L^N_1, L^N_2, \ldots)$.

\smallskip

\begin{assumption}[Hall-Littlewood Gibbs property]\label{as.HL Gibbs}
$\bm L^N$ has the Hall-Littlewood Gibbs property with parameter $q\in[0,1)$ fixed (Definition~\ref{d.HL Gibbs}).
\end{assumption}

\smallskip

\begin{assumption}[One-point tightness of top curve around parabola]\label{as.one-point tightness}
\label{as.one-point tightness} There exist $p\in(-1,0)$ and $\lambda>0$ such that, for any $\delta>0$ there exists $M = M(\delta)$ such that
  \begin{align*}
  \sup_{x\in\R}\, \limsup_{N\to\infty}\, \P\left(|L^N_1(xN^{2/3}) - pxN^{2/3} + \lambda x^2 N^{1/3}| > MN^{1/3}\right) \leq \delta.
  \end{align*}
  Equivalently, for any $\delta>0$ there exists $M$ such that for all $x\in\R$ there exists $N_0$ such that, for $N> N_0$,
  \begin{align}\label{e.one point tightness assumption}
  \P\left(|L^N_1(xN^{2/3}) - pxN^{2/3} + \lambda x^2 N^{1/3}| > MN^{1/3}\right) \leq \delta.
  \end{align}
\end{assumption}

By our convention on the definition of $L^N_i(x)$ for $|x|\geq g_NN^{2/3}$, Assumption \ref{as.HL Gibbs} is only needed for $\bm L^N|_{\N\times[-g_NN^{2/3},g_NN^{2/3}]}$. We will not explicitly mention this point or the general point about the definition of $L^N_i(x)$ for $|x|\geq g_NN^{2/3}$ further.

We note that Assumption~\ref{as.one-point tightness} is satisfied if it is known that the union over $x\in\R$ of the set of subsequential weak limits of $\{N^{-1/3}(L^N_1(xN^{2/3}) - pxN^{2/3} + \lambda x^2 N^{1/3})\}_{N=1}^\infty$ is a tight collection (for instance, as will be the case for our models of interest, if it is a single distribution, e.g., a constant multiple of the GUE Tracy-Widom distribution).

We will need to rescale our line ensembles to state our results. For $p\in(-1,0)$, $\lambda>0$, and $N\in\N$, we introduce a scaling operator $T_{p, \lambda, N}:\mc C(\N\times\R, \R) \to \mc C(\N\times\R, \R)$ defined by the following for $\bm f\in\mc C(\N\times\R, \R)$, $i\in\N$, and $x\in\R$:
\begin{align}\label{e.scaling oeprator}
\bigl(T_{p,\lambda, N}(\bm f)\bigr)_i(x) := \sigma^{-1}N^{-1/3}\left(f_i(\beta xN^{2/3}) - p \beta xN^{2/3}\right)
\end{align}
where $\beta$ and $\sigma$ are such that
\begin{align}\label{e.scaling coefficients relation}
\sigma^{-1}\beta^2\lambda = 1 \quad\text{and}\quad \sigma^{-2} \beta p(1-p) = 2.
\end{align}
For $\bm L^{N}$ satisfying Assumption~\ref{as.one-point tightness}, and with $p$ and $\lambda$ as there, define $\bm\cL^N : \N\times\R\to\R$~by
\begin{align}\label{e.cL definition}
\bm\cL^{N} := T_{p,\lambda, N}(\bm L^{N}).
\end{align}
The equations \eqref{e.scaling coefficients relation} are the analogs of the conditions \eqref{e.asep scaling relations}. The first condition will ensure that $\smash{\cL^{N}_1(x)}$ decays like $-x^2$; the second will ensure that the large $N$ limit of $\bm \cL^{N}$ has the Brownian Gibbs property (Definition~\ref{d.bg}) with the correct diffusion coefficient of 2. 

Now we may state the tightness result we prove, which is a key result needed to establish our main results, Theorems~\ref{t.asep airy sheet} and \ref{t.s6v airy sheet}; the proof will be given in Section~\ref{s.tightness ingredients}.

\begin{theorem}[Tightness of line ensembles]\label{t.tightness}
Let $\bm L^N : \N\times \R \to\R$ be a line ensemble satisfying Assumptions \ref{as.HL Gibbs} and \ref{as.one-point tightness}. Then $\{\bm\cL^N\}_{N=1}^\infty$ is tight in the space $\mc C(\N\times\R, \R)$, under the topology of uniform convergence on compact sets.
\end{theorem}

As mentioned in Section~\ref{ssub:heuristic_behavior_of_the_gibbs_property}, black-box tightness statements have also been proved in discrete settings previously for Gibbsian line ensembles, but with different Gibbs properties enjoying stochastic monotonicity properties \cite{dimitrov2021tightnessH,SERIO2023225,dimitrov2021tightness}.

The following is a third assumption which will be needed to ensure that there is a unique limiting line ensemble, and that it is the parabolic Airy line ensemble $\bm\cP$ from Definition~\ref{d.parabolic Airy line ensemble}.

\smallskip

\begin{assumption}[GUE Tracy-Widom convergence]
There exists $x\in\R$ such that $\cL^{N}_1(x) + x^2\xrightarrow{\smash{d}} \mrm{GUE}\text{-}\mathrm{TW}$ (the GUE Tracy-Widom distribution, as defined at the end of Definition~\ref{d.parabolic Airy line ensemble}) as $N\to\infty$.
  \label{as.GUE TW}
\end{assumption}

The following gives the convergence to $\bm\cP$ of the rescaled line ensemble $\bm{\cL}^N$. It will be proved in Section~\ref{s.tightness preliminaries} as a quick consequence of the tightness from Theorem~\ref{t.tightness}, together with the recent characterization result for Brownian Gibbsian line ensembles from \cite{aggarwal2023strong}.

\begin{theorem}\label{t.line ensemble convergence to parabolic Airy}
Suppose $\bm L^N$ satisfies Assumptions~\ref{as.HL Gibbs} and \ref{as.one-point tightness}. Then any subsequential limit of $\{\bm\cL^N\}_{N=1}^\infty$ is $\bm\cP +\mf c$, where $\mf c\in\R$ is a random variable independent of $\bm\cP$ (but which may depend on the subsequence taken),
 under the topology of uniform convergence on compact sets. If additionally Assumption~\ref{as.GUE TW} is satisfied, then $\mf c = 0$, i.e, $\bm\cL^N\xrightarrow{\smash{d}} \bm\cP$.
\end{theorem}

\subsection{Last passage percolation in the colored Hall-Littlewood line ensemble}\label{s.LPP in discrete line ensemble}

We now turn to stating our results on the relation between $\bm L^{\mrm{cHL}, (1)}$ and $(\bm L^{\mrm{cHL}, (2)}, \bm L^{\mrm{cHL}, (3)}, \ldots)$ (under general boundary condition $\sigma$) as defined in \eqref{e.colored line ensemble definition}.
In particular, we give a precise statement of the approximate representation of $\smash{L^{\mrm{cHL}, (j)}_1}$ as a last passage percolation (LPP) problem in $\bm L^{\mrm{cHL}, (1)}$, which in fact holds exactly when $q=0$.

Recall the definition of $\smash{\bm L^{\mrm{cHL}, (j)}}$ from \eqref{e.colored line ensemble definition} and the notation for LPP problems from \eqref{e.LPP definition}; this notation is for an LPP problem in an environment given by a family of continuous functions, and for this purpose as well as in the next result, we regard $\bm L^{\mrm{cHL}, (1)}$ as a collection of continuous functions on the real interval $[0,N+M]$ by linear interpolation, as in the setup introduced in Section~\ref{s.line ensemble hypotheses}. We prove the next result in Section~\ref{s.approximate LPP}.

\begin{theorem}\label{t.approxmate LPP problem representation general}
Fix  $N,M\in\N$, any boundary condition $\sigma:\intint{1,N}\to\intint{1,N}$ and $q\in[0,1)$. Let $(\bm L^{\mrm{cHL}, (1)}, \ldots, \bm L^{\mrm{cHL}, (N)})$ be the associated colored Hall-Littlewood line ensemble as given by Definition~\ref{d.colored line ensemble}. There exist positive constants $c$ (absolute), and $K=K(q)$ such that for any $j,k\geq 1$ and $m\geq K\log N$,
\begin{align}
\MoveEqLeft[30]
\P\left(\sup_{y\in [0,N+M]} \left|L^{\mrm{cHL}, (j)}_1(y) - \sup_{z\leq y} \left(L^{\mrm{cHL},(j)}_{k+1}(z) + \bm L^{\mrm{cHL}, (1)}[(z,k)\to(y,1)]\right)\right| \geq k m \right)\nonumber\\
&\leq kq^{cm^2}. \label{e.approximate LPP problem}
\end{align}
When $q=0$, $L^{\mrm{cHL}, (j)}_1(y) = \sup_{z\leq y} (L^{\mrm{cHL},(j)}_{k+1}(z) + \bm L^{\mrm{cHL}, (1)}[(z,k)\to(y,1)])$ for all $y\in[0,N+M]$, $j\in\intint{1,N}$, and $k\in\N$.
\end{theorem}

Recall the definition of the Airy sheet from Definition~\ref{d.airy sheet}, the definition of $\S^{\mrm{S6V},\varepsilon}$ from \eqref{e.rescaled s6v definition}, and the relation of $\smash{L^{\mrm{cHL}, (j)}_1}$ to the S6V height function $h^{\mrm{S6V}}$ from Proposition~\ref{p.colored line ensembles}. At this point in our discussion, it should not be hard to believe that, given the representation of $\smash{L^{\mrm{cHL}, (j)}_1}$ in Theorem~\ref{t.approxmate LPP problem representation general} and the convergence of $\bm L^{\mrm{cHL}, (1)}$ to $\bm\cP$ from Theorem~\ref{t.line ensemble convergence to parabolic Airy}, the convergence of $\S^{\mrm{S6V}, \varepsilon}$ to $\S$ will follow (and similarly the convergence of $\S^{\mrm{ASEP}, \varepsilon}$ which, as mentioned in the beginning of Section~\ref{s.key ideas}, will be proved by a limit to ASEP from S6V). Indeed, given the results we have stated in this section, the argument to establish this is at this point well-known, following presentations in \cite{dauvergne2018directed} and \cite{dauvergne2021scaling}; as a result we give the proof of Theorems~\ref{t.s6v airy sheet} and \ref{t.asep airy sheet} in Section~\ref{s.convergence to Airy sheet} after relegating some of the standard details to Appendix~\ref{s.G_k asymptotics}.

We also note that while we obtain a probability decay of order $q^{cm^2}$, this strength is not required for the rest of the arguments; indeed, it would suffice to have a tail decay which ensures that deviations from the LPP problem are of smaller order than the fluctuation scale and makes such deviation probabilities summable as a function of the scale parameter (here $N$).

\section{Approximate LPP representation}\label{s.approximate LPP}

\newcommand{\numrows}{U}

In this section we prove Theorem~\ref{t.approxmate LPP problem representation general}. 
Recall that we allow a general boundary condition encoded by a function $\sigma:\intint{1,N}\to \intint{1,N}$, where $\sigma(k)$ is the color of the arrow horizontally entering the domain $\Z_{\leq 0}\times\intint{1,N+M}$ of the colored $q$-Boson model at $(1,k)$.

\subsection{Pitman transform and the $q=0$ case} For two functions $f,g:\Z_{\geq 0}\to \R$, we define the \emph{Pitman transform} $\PT(f,g):\Z_{\geq 0}\to\R$ by
\begin{align}\label{e.pitman transform definition}
\PT(f,g)(y) := f(y) + \max_{0\leq y'\leq y}\left( g(y') - f(y')\right).
\end{align}
The Pitman transform should be thought of as an LPP problem across two curves, with the bottom one equal to $g$ and the top equal to $f$ (this interpretation assumes $g(0) = 0$, as will be the case in our context). Indeed, we will ultimately access the LPP value across $k$ lines by iterated applications of $\PT$.

We start by stating and proving a statement in the case of $q=0$, in order to make the source of the Pitman transform in the colored Hall-Littlewood line ensemble more transparent; we will subsequently turn to obtaining a similar statement in the case of general $q\in[0,1)$ in Section~\ref{s.general q}.

In this section we will mainly restrict to the case where $\sigma:\intint{1,N}\to\intint{1,2}$, i.e., arrows in the colored $q$-Boson model are only of color $1$ or $2$, as this will suffice for the proof of Theorem~\ref{t.approxmate LPP problem representation general} (see the end of this section) by color merging. We drop the superscript of ``cHL'' when writing the colored Hall-Littlewood line ensemble; for instance we will write $\smash{L^{(j)}_i}$ for $\smash{L^{\mrm{cHL}, (j)}_i}$. We will also refer to arrows of color 2 as $2$-arrows, and analogously $1$-arrows for arrows of color $1$.
In particular, recalling the relation of the colored Hall-Littlewood line ensemble and the colored $q$-Boson model from Definition~\ref{d.colored line ensemble}, $\smash{L^{(2)}_k}$ corresponds to counts of 2-arrows in the $k$\th column (i.e., $\{-k\}\times\intint{1,N+M}$) of the colored $q$-Boson model, while $\smash{L^{(1)}_k}$ corresponds to combined arrow counts of $1$- and $2$-arrows in the same column.

\begin{lemma}\label{l.q=0 pitman}
Let $q=0$. For any boundary condition $\sigma:\intint{1,N}\to\intint{1,2}$ and any $k\in\N$,
$$L^{(2)}_k = \PT\left(L^{(1)}_k, L^{(2)}_{k+1}\right).$$
\end{lemma}

We point out that this is a deterministic equality and not only one of distribution.

It turns out that a lower bound actually holds deterministically for all $q\in[0,1)$. We state and prove it first before giving the proof of Lemma~\ref{l.q=0 pitman} that equality holds when $q=0$.

\begin{lemma}\label{l.general q pitman lower bound}
For any $q\in[0,1)$, boundary condition $\sigma:\intint{1,N}\to\intint{1,2}$, and $k\in \N$,
$$L^{(2)}_k \geq \PT\left(L^{(1)}_k, L^{(2)}_{k+1}\right).$$
\end{lemma}

\begin{proof}
Fix $y$. Recalling the definition \eqref{e.pitman transform definition} of the Pitman transform, we must show that
\begin{align*}
L^{(2)}_k(y) \geq L^{(1)}_k(y) + \max_{y'\leq y}\left(L^{(2)}_{k+1}(y') - L^{(1)}_k(y')\right).
\end{align*}
Let us write $O_{k}(y) = L^{(1)}_k(y) - L^{(2)}_k(y)$, so that $O_{k}(y)$ is the number of $1$-arrows exiting the $k$\textsuperscript{th} column in the colored $q$-Boson model at row strictly higher than $y$. Then, from the previous display, what we have to show is equivalent to
\begin{align}\label{e.PT equivalent to prove}
\min_{y'\leq y}\left(O_{k}(y') + L^{(2)}_k(y') - L^{(2)}_{k+1}(y')\right) \geq O_{k}(y).
\end{align}
Now, recall from \eqref{e.L^(j) properties} that $L^{(2)}_k(y') \geq L^{(2)}_{k+1}(y')$ for all $y'$ and observe that \smash{$O_{k}(y)$} is non-increasing in $y$ by its definition as the count of $1$-arrows (analogous to the fact that $\smash{L^{(j)}_k(y+1) \leq L^{(j)}_k(y)}$ noted in the third part of \eqref{e.L^(j) properties}). This yields that the lefthand side of the previous display is at least the righthand side.
\end{proof}

\begin{figure}[H]
  \begin{tikzpicture}[scale=0.9]
    \newcommand{\thedim}{4.2}
    \begin{scope}[scale=0.6]
      
      \draw[step=\thedim] (0,0) grid (4*\thedim, \thedim);

      \draw[very thick] (2*\thedim, \thedim) -- ++(0,-1.25*\thedim);
      
      \draw[->, very thick, blue] (0.5*\thedim, 0.2*\thedim) -- ++(0, 0.6*\thedim);
      \draw[->, very thick, green!80!black] (0.5*\thedim+0.15, 0.5*\thedim) -- ++(0, 0.3*\thedim);
      \draw[very thick, green!80!black, dashed] (0.5*\thedim+0.15, 0.2*\thedim) -- ++(0, 0.3*\thedim);
      \draw[very thick, green!80!black, dashed] (0.2*\thedim, 0.5*\thedim) -- ++(0.3*\thedim, 0);
      
      \draw[->, very thick, blue] (1.5*\thedim-0.075, 0.2*\thedim) -- ++(0, 0.6*\thedim);
      \draw[very thick, green!80!black, dashed] (1.5*\thedim+0.075, 0.2*\thedim) -- ++(0, 0.3*\thedim);
      \draw[very thick, green!80!black, dashed] (1.2*\thedim, 0.5*\thedim) -- ++(0.3*\thedim, 0);
       \draw[->, very thick, green!80!black] (1.5*\thedim+0.075, 0.5*\thedim) -- ++(0.3*\thedim, 0);
      \node[scale=0.6, anchor=west] at (1.8*\thedim+0.1, 0.5*\thedim) {$i$};
      
      
      
      \begin{scope}[shift={(-2*\thedim,0)}]
            \draw[->, very thick, blue] (4.5*\thedim, 0.2*\thedim) -- ++(0, 0.6*\thedim);
            \draw[->, very thick, orange!90!black] (4.2*\thedim, 0.5*\thedim+0.1) -- ++(0.3*\thedim +0.15, 0) -- ++(0, 0.3*\thedim-0.1);
            \draw[->, very thick, green!80!black] (4.5*\thedim-0.15, 0.2*\thedim) -- ++(0, 0.3*\thedim) -- ++(0.3*\thedim, 0);
            \node[scale=0.6, anchor=east] at (4.2*\thedim, 0.5*\thedim+0.1) {$2$};
            \node[scale=0.6, anchor=west] at (4.75*\thedim, 0.5*\thedim) {$1$};
            
            \draw[->, very thick, blue] (5.5*\thedim, 0.2*\thedim) -- ++(0, 0.6*\thedim);
            \draw[->, very thick, green!80!black] (5.2*\thedim, 0.5*\thedim) -- ++(0.6*\thedim, 0);
            \node[scale=0.6, anchor=west] at (5.8*\thedim, 0.5*\thedim) {$1$};
            \node[scale=0.6, anchor=west] at (5.5*\thedim+0.15, 0.8*\thedim) {$2$};
            \draw[->, very thick, orange!90!black] (5.5*\thedim+0.15, 0.2*\thedim) -- ++(0, 0.6*\thedim);
      \end{scope}

      \draw[xstep=\thedim, ystep=0.25*\thedim] (0,0) grid (4*\thedim, -0.25*\thedim);
      
      \newcommand{\they}{-0.125*\thedim}
      
      \node[scale=0.7] at (0.5*\thedim, \they) {$1$};
      \node[scale=0.7] at (1.5*\thedim, \they) {$u\one_{i=1, A_2=0 \text{ or } i=2}$};
      \node[scale=0.7] at (2.5*\thedim, \they) {$0$};
      \node[scale=0.7] at (3.5*\thedim, \they) {$0$};
      
    \end{scope}
  \end{tikzpicture}
  \caption{The $\mathtt{L}_u$ weights in the case that $q=0$ and only two colors are present. The first two panels actually capture all cases, where the pair of horizontal and vertical dashed lines indicate that an arrow of color $i\in\{1,2\}$ enters vertically, horizontally, or both. The second two panels emphasize the two cases where an arrow of color 1 exits horizontally though an arrow of color 2 was available, each of which has weight $0$.}\label{f.L weights when q=0}
\end{figure}

\begin{proof}[Proof of Lemma~\ref{l.q=0 pitman}]
We continue with the notation $O_k(y) = L^{(1)}_k(y) - L^{(2)}_k(y)$ from the previous proof for the count of 1-arrows exiting the $k$\th column at row strictly higher than $y$. By Lemma~\ref{l.general q pitman lower bound}, we must show \eqref{e.PT equivalent to prove} holds with equality, i.e., for some $y'\leq y$,
\begin{equation}\label{e.to show}
O_{k}(y') + L^{(2)}_k(y') - L^{(2)}_{k+1}(y') = O_{k}(y).
\end{equation}
We start by observing that $\smash{L^{(2)}_k(0) = L^{(2)}_{k+1}(0)}$, since both quantities are simply the total number of $2$-arrows in the system.

We break now into two cases. First suppose that $O_k(y') = O_k(y)$ for all $0\leq y'\leq y$, which is the case that no $1$-arrows exit horizontally from any such vertex $(-k,y')$. Since $L^{(2)}_k(0) = L^{(2)}_{k+1}(0)$, we observe that in this first case $y'=0$ satisfies \eqref{e.to show}. 

In the second case we assume that there exists at least one vertex $(-k,y')$ with $y'\leq y$ from which a 1-arrow exits horizontally. Let $y_0$ be the largest such $y'$; equivalently, $y_0$ is the largest $y'\leq y$ such that $O_k(y') = O_k(y'-1) - 1$. Observe that then $O_k(y_0) = O_{k}(y)$ since $O_k$ is non-increasing for each $k$.

We claim that then $L^{(2)}_k(y_0) = L^{(2)}_{k+1}(y_0)$, which would establish \eqref{e.to show}. Indeed, we know from the weights in the $q=0$ case from Figure~\ref{f.L weights when q=0} that since a $1$-arrow exits horizontally from $(-k,y_0)$, there is no $2$-arrow entering that vertex from the left or bottom. The arrows vertically entering $(-k,y_0)$ are precisely the arrows that have exited $(-(k+1), y')$ horizontally for some $y'< y_0$ and which have not exited the $k$\textsuperscript{th} column strictly below $y_0$. Thus all 2-arrows in the $(k+1)$\textsuperscript{st} column that have horizontally exited that column strictly below row $y_0$ have also exited the $k$\textsuperscript{th} column strictly below row $y_0$, and, as we already noted, a $2$-arrow does not exit horizontally from $(-k,y_0)$ or $(-(k+1),y_0)$. In other words, \smash{$L^{(2)}_k(0)-L^{(2)}_k(y_0) = L^{(2)}_{k+1}(0)-L^{(2)}_{k+1}(y_0)$}, which implies our claim, since $L^{(2)}_k(0) = L^{(2)}_{k+1}(0)$.
\end{proof}

As we saw in Lemma~\ref{l.general q pitman lower bound}, $L^{(2)}_k(y) \geq \PT(L^{(1)}_k, L^{(2)}_{k+1})(y)$ deterministically. The next statement gives a bound on how much the deviation can be in the case of general $q\in[0,1)$ and is the basic ingredient of the proof of Theorem~\ref{t.approxmate LPP problem representation general}.

\begin{proposition}\label{p.PT deviation probability}
There exist positive constants $c$ and $K = K(q)$ such that, for any $N\in\N$, boundary condition $\sigma:\intint{1,N}\to\intint{1,2}$, $k\in\N$, $q\in[0,1)$, and $m\geq K\log N$,
\begin{align*}
\P\left(\max_{y\in\intint{0,N+M}} \left(L^{(2)}_k(y) - \PT(L^{(1)}_k, L^{(2)}_{k+1})(y)\right) \geq m\right) \leq q^{cm^2}.
\end{align*}
\end{proposition}

\noindent Before turning to the proof of this statement in Section~\ref{s.general q}, we use it to establish Theorem~\ref{t.approxmate LPP problem representation general}.

\begin{proof}[Proof of Theorem~\ref{t.approxmate LPP problem representation general}]
Recall that we have dropped the superscript $\mrm{cHL}$ from the notation of the colored line ensembles. Next, it suffices to specialize to the case of $j=2$. Indeed, the original statement will follow from this special case by Lemma~\ref{l.color merging cHL} (color merging), i.e., the fact that the joint distribution of $(\bm L^{(1)}, \bm L^{(j)})$ associated to boundary condition $\sigma$ is the same as that of $(\bm L^{(1)}, \bm L^{(2)})$ associated to boundary condition $\tau\circ\sigma$, where $\tau(i) = 1$ if $i \in\intint{1,j-1}$ and $\tau(i) = 2$ if $i\in\intint{j,N}$. 
In other words, we merge the colors in $\intint{1,j-1}$ to be $1$ and those in $\intint{j,N}$ to be $2$.

Inductively define $\PT^{(1)} = \PT$ and for $k\geq 2$
\begin{align}\label{e.PT^(j) definition}
\PT^{(k)}\left((L^{(1)}_i)_{i=1}^k, L^{(2)}_{k+1}\right) := \PT^{(k-1)}\left((L^{(1)}_i)_{i=1}^{k-1}, \PT\left(L^{(1)}_k, L^{(2)}_{k+1}\right)\right).
\end{align}
It is easy to show that
\begin{align}\label{e.PT LPP representation}
\PT^{(k)}\left((L^{(1)}_i)_{i=1}^k, L^{(2)}_{k+1}\right) = \max_{z\leq y} \left(L^{(2)}_{k+1}(z) + \bm L^{(1)}[(z,k)\to(y,1)]\right).
\end{align}
With this we see that
\begin{align}
\MoveEqLeft[0]
\P\left(\max_{y\in\intint{0,N+M}}\left| L^{(2)}_1(y) - \max_{z\leq y} \left(L^{(2)}_{k+1}(z) + \bm L^{(1)}[(z,k)\to(y,1)]\right) \right| \geq km\right)\nonumber\\
&\leq \P\left(\max_{y\in\intint{0,N+M}}\left|L^{(2)}_1(y) - \PT(L^{(1)}_1, L^{(2)}_2)(y)\right| \geq m \right)\label{e.iterated pitman decomposition}\\
&+ \sum_{j=1}^{k-1} \P\left(\max_{y\in\intint{0,N+M}}\left|\PT^{(j)}\left((L^{(1)}_i)_{i=1}^j, L^{(2)}_{j+1}\right)(y) - \PT^{(j+1)}\left((L^{(1)}_i)_{i=1}^{j+1}, L^{(2)}_{j+2}\right)(y)\right| \geq m \right). \nonumber
\end{align}
In the $q=0$ case, Lemma~\ref{l.q=0 pitman} says that the previous display is $0$ for any $m>0$. For the case of $q\in(0,1)$, we observe that if $h$ is a function with $\max |h| \leq \varepsilon$ for some $\varepsilon>0$, then it follows from the representation \eqref{e.PT LPP representation} of $\PT^{(k)}$ and the triangle inequality that, for any $f_1, \ldots, f_k$ and $g$,
\begin{align}\label{e.pt perturbation inequality}
|\PT^{(k)}((f_i)_{i=1}^k, g) - \PT^{(k)}((f)_{i=1}^k, g+h)| \leq 2\varepsilon.
\end{align}
From Lemma~\ref{l.general q pitman lower bound} and Proposition~\ref{p.PT deviation probability} it holds that, for any $j\geq 0$ and $m\geq K\log N$,
\begin{align*}
\P\left(\max_{y\in\intint{0,N+M}}\left| L^{(2)}_{j+1}(y) - \PT(L^{(1)}_{j+1}, L^{(2)}_{j+2})(y)\right| \geq \tfrac{1}{2}m\right) \leq q^{cm^2}.
\end{align*}
Combining this with \eqref{e.pt perturbation inequality} and the definition \eqref{e.PT^(j) definition} of $\PT^{(j)}$ yields that the $j$\textsuperscript{th} summand in \eqref{e.iterated pitman decomposition} is upper bounded by $q^{cm^2}$ which, on tracing back, completes the proof.
\end{proof}

The rest of the section is devoted to proving Proposition~\ref{p.PT deviation probability}.

\subsection{The case of $\bm{q\in(0,1)}$}\label{s.general q}
As we saw in the proof of Lemma~\ref{l.q=0 pitman}, the Pitman transform representation was essentially a consequence of the fact that, when $q=0$, the colored $q$-Boson vertex weights force $2$-arrows to exit horizontally before $1$-arrows. In $q>0$, $1$-arrows can exit before $2$-arrows, but will be penalized in weight by a factor of $q^m$, where $m$ is the number of $2$-arrows vertically entering at the vertex where the $1$-arrow is exiting (see Figure~\ref{f.L weights}). At a high level, the heuristic behind the proof of Proposition~\ref{p.PT deviation probability} is as follows: on the event that the deviation of the height function in the $q>0$ case from the Pitman transform expression (as it would be in the $q=0$ case) is $m$, we can find order $m$ many different locations where a 1-arrow horizontally exits before a $2$-arrow, and where there are order $m$ many $2$-arrows ``in queue'', i.e., vertically entering; this will lead to the $\smash{q^{cm^2}}$ probability bound for this event.

The precise proof of Proposition~\ref{p.PT deviation probability} requires introducing some notation.
Let $\numrows=N+M$ (the number of rows in the colored $q$-Boson model). Fix $\bm v = (v_1, \ldots, v_\numrows)\in\{0,1,2\}^\numrows$ and $\bm x=(x_1, \ldots, x_\numrows)\in\{0,1\}^\numrows$. Here $v_y$ represents the color of the arrow horizontally entering at the $y$\smash{\th} row and $x_y$ represents the number (of non-zero color) arrows horizontally exiting at the $y$\smash{\th} row, both counted from the bottom (see Figure~\ref{f.v z w}).

For a $\bm v \in\{0,1,2\}^\numrows$, $j\in\{1,2\}$ and $y\in\intint{1,\numrows}$, let the height function (which we call $\mrm{Ht}$ for short) of the $j$\th color be defined by
$$\mrm{Ht}_j(\bm v,y) := \#\Bigl\{y'> y: v_{y'}= j\Bigr\} = \sum_{y'=y+1}^\numrows \one_{v_{y'} = j},$$
i.e., the number of arrows of color $j$ strictly above row $y$ in $\bm v$.

We say that  $\bm v$ and $\bm x$ are \emph{compatible} if they can arise as the incoming color counts and outgoing total arrow counts at a column, respectively, while satisfying arrow conservation; more precisely, if and only if $\sum_{j=1}^2(\mrm{Ht}_j(\bm v,0) - \mrm{Ht}_j(\bm v,y)) \geq \sum_{y'=1}^y x_{y'}$ for each $y$ and equality holds at $y=\numrows$.

For given $\bm v\in\{0,1,2\}^\numrows$ and $\bm x\in\{0,1\}^\numrows$ which are compatible, define $\mrm{Val}(\bm v, \bm x)$ (short for ``valid'')
$$\mrm{Val}(\bm v,\bm x) := \left\{\bm w\in\{0,1,2\}^\numrows : \parbox{8.4cm}{\centering$w_y = 0 \iff x_y=0$,  $\mrm{Ht}_j(\bm w,0) = \mrm{Ht}_j(\bm v,0),$\\[3pt]
$\mrm{Ht}_j(\bm w,y) \geq \mrm{Ht}_j(\bm v,y) \text{ for } j=1,2,\ y \in \intint{0, \numrows-1}$}\right\},$$
to be the set of valid assignments of outgoing colors at a column with incoming color counts given by $\bm v$ and outgoing arrow counts given by $\bm x$  (see Figure~\ref{f.v z w}); here, valid informally means that colored arrow conservation holds.

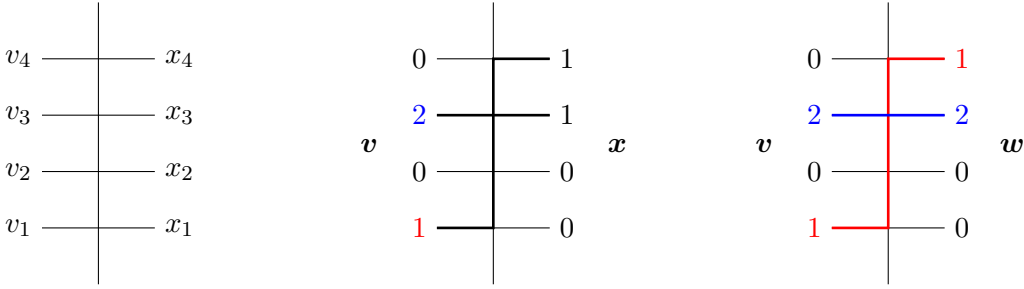
\begin{figure}
\begin{tikzpicture}
\begin{scope}
\draw (0,0) -- ++(0,3.75);

\foreach \y/\i in {0.75/1, 1.5/2, 2.25/3, 3/4}
{
 \draw (-0.75, \y) -- ++(1.5,0);
 \node[anchor=east] at (-0.75, \y) {$v_\i$};
 \node[anchor=west] at (0.75, \y) {$x_\i$};
}
\end{scope}

\begin{scope}[shift={(5.25,0)}]
\draw (0,0) -- ++(0,3.75);

\foreach \y/\i/\z/\thecolor in {0.75/1/0/red, 1.5/0/0/black, 2.25/2/1/blue, 3/0/1/black}
{
 \draw (-0.75, \y) -- ++(1.5,0);
  \node[anchor=east, \thecolor] at (-0.75, \y) {$\i$};
  \node[anchor=west] at (0.75, \y) {$\z$};
}

  \node at (-1.65, 1.85) {$\bm v$};
  \node at (1.65, 1.85) {$\bm x$};

  \draw[line width=1pt] (-0.75, 0.75) -- ++(0.75,0) -- ++(0,2.25) -- ++(0.75,0);
  \draw[line width=1pt] (-0.75, 2.25) -- ++(1.5,0);

\end{scope}

\begin{scope}[shift={(10.5,0)}]
\draw (0,0) -- ++(0,3.75);

\foreach \y/\i/\z/\theLcolor/\theRcolor in {0.75/1/0/red/black, 1.5/0/0/black/black, 2.25/2/2/blue/blue, 3/0/1/black/red}
{
 \draw (-0.75, \y) -- ++(1.5,0);
  \node[anchor=east, \theLcolor] at (-0.75, \y) {$\i$};
  \node[anchor=west, \theRcolor] at (0.75, \y) {$\z$};
}

\draw[red, line width=1pt] (-0.75, 0.75) -- ++(0.75,0) -- ++(0,2.25) -- ++(0.75,0);
\draw[blue, line width=1pt] (-0.75, 2.25) -- ++(1.5,0);

  \node at (-1.65, 1.85) {$\bm v$};
  \node at (1.65, 1.85) {$\bm w$};
\end{scope}
\end{tikzpicture}
\caption{A depiction of $\bm v\in\{0,1,2\}^\numrows$, $\bm x\in\{0,1\}^\numrows$ (left panel), an example where they are compatible (middle panel), and an example of a $\bm w\in\mrm{Val}(\bm v,\bm x)$ (right panel). Here the $1$ and $0$ in black on the right side of the middle panel respectively indicate the presence or absence of an arrow of any color, while the colored numbers on the left side of the middle panel and both sides of the right panel indicate the color of the arrow.}\label{f.v z w}
\end{figure}

For $\bm v\in\{0,1,2\}^\numrows$ and $\bm w\in\mrm{Val}(\bm v,\bm x)$, define $\mrm{VertCt}_j$ (short for ``vertical count'') by
\begin{equation}\label{e.vertct definition}
\begin{split}
\mrm{VertCt}_j(\bm v,\bm w,y) &:= \bigl(\mrm{Ht}_j(\bm v,0) - \mrm{Ht}_j(\bm v,y)\bigr) - \bigl(\mrm{Ht}_j(\bm w,0) - \mrm{Ht}_j(\bm w,y)\bigr)\\
&= \mrm{Ht}_j(\bm w,y) - \mrm{Ht}_j(\bm v,y)
= \sum_{y'=1}^y \one_{v_{y'}=j} - \one_{w_{y'} = j}
\end{split}
\end{equation}
be the number of arrows of color $j$ vertically exiting the vertex at height $y$; the second equality is since $\mrm{Ht}_j(\bm v,0) = \mrm{Ht}_j(\bm w,0)$ is the total number of $j$-arrows in the system (equal since $\bm w\in\mathrm{Val}(\bm v,\bm x)$). Note that $\mrm{VertCt}_j(\bm v,\bm w,y)$ is a function only of $v_1, \ldots, v_y$ and $w_1, \ldots, w_y$.

We next define a specific assignment $\bm w^*\in \mrm{Val}(\bm v,\bm x)$ for any given compatible pair $(\bm v,\bm x)$, which is the deterministic (and unique) assignment of colors to the outgoing arrows that would occur in the $q=0$ colored $q$-Boson model. It is defined as follows: we specify $w^*_y$ iteratively starting with $y=1$ and continuing in increasing order by:
\begin{itemize}

  \item $w^*_y = 2$ if $\one_{v_y = 2} + \mrm{VertCt}_2(\bm v,\bm w^*,y-1) \geq 1$ (i.e., if there is at least one $2$-arrow that can exit the vertex at height $y$) and $x_y=1$; note that the condition is well-defined since $\mrm{VertCt}_2(\bm v,\bm w^*,y-1)$ does not depend on $w^*_{y'}$ for any $y'\geq y$;

  \item $w^*_y = 1$ if $\one_{v_y = 2} + \mrm{VertCt}_2(\bm v, \bm w^*,y-1) = 0$ and $\one_{v_y = 1} + \mrm{VertCt}_1(\bm v, \bm w^*,y-1) \geq 1$ (i.e., if there are no 2-arrows and at least one $1$-arrow that can exit the vertex at height $y$) and $x_y=1$;

  \item $w^*_y = 0$ if $x_y=0$.
\end{itemize}
The fact that $(\bm v,\bm x)$ are compatible ensures that exactly one of the above three conditions holds for each $y$.

Fix $\bm v\in\{0,1,2\}^\numrows$ and $\bm x\in\{0,1\}^\numrows$ which are compatible, let $\bm w\in\mrm{Val}(\bm v,\bm x)$, and let $\bm w^*$ be as defined above. We define the Pitman error $\mrm{PE}(\bm w)$ of $\bm w$ by
\begin{align*}
\mrm{PE}(\bm w) := \max_{y\leq \numrows} \bigl(\mrm{Ht}_2(\bm w, y) - \mrm{Ht}_2(\bm w^*, y)\bigr),
\end{align*}
i.e., it is the maximal height difference of $2$-arrows between $\bm w$ and $\bm w^*$ over all rows. Observe that $\mrm{Ht}_2(\bm w, y) \geq \mrm{Ht}_2(\bm w^*,y)$ for each $y\leq \numrows$. Indeed, as we said, it is easy to show that $\bm w^*$ is the configuration that would arise in the $q=0$ setting, so by Lemma~\ref{l.q=0 pitman}, $\mrm{Ht}_2(\bm w^*,\bm\cdot)$ may be regarded as the Pitman transform of the incoming color counts $\bm v$ and the outgoing arrow counts $\bm x$; then Lemma~\ref{l.general q pitman lower bound} implies that $\mrm{Ht}_2(\bm w, y) \geq \mrm{Ht}_2(\bm w^*, y)$. Alternatively, since $\bm w^*$ always releases $2$-arrows as soon as it is possible, $\mrm{Ht}_2(\bm w^*, y)$ is minimal over all $\bm w\in\mrm{Val}(\bm v,\bm x)$ for any $y$, which again verifies that $\mrm{Ht}_2(\bm w, y) \geq \mrm{Ht}_2(\bm w^*, y)$.

We need some notation for the weight of configurations. For a compatible pair $(\bm v,\bm x)$ and $\bm w\in\mrm{Val}(\bm v,\bm x)$, let 
$$\bm{\mrm{VertCt}}(\bm v,\bm w,y) = (\mrm{VertCt}_1(\bm v,\bm w,y),\mrm{VertCt}_2(\bm v,\bm w,y)).$$
Further let $\bigL_{z,y} = \bigL_1$ if $y\in\intint{1,N}$ and $\bigL_{z,y} = \bigL_z$ (recall that $z\in(0,1)$ is the fixed spectral parameter) if $y\in\intint{N+1,N+M}$, where $\bigL_u$ is read from Figure~\ref{f.L weights}. Now define
\begin{equation}\label{e.wgt_y definition}
\mrm{Wgt}_y(\bm v,\bm w) := \bigL_{z,y}\bigl(\bm{\mrm{VertCt}}(\bm v,\bm w,y-1), v_y; \bm{\mrm{VertCt}}(\bm v,\bm w,y), w_y\bigr)
\end{equation}
to be the vertex weight at $y$, with incoming and outgoing horizontal arrows of $v_y$ and $w_y$, respectively, and incoming and outgoing vertical arrows of $\bm{\mrm{VertCt}}(\bm v,\bm w,y-1)$ and $\bm{\mrm{VertCt}}(\bm v,\bm w,y)$, respectively. Next, define the weight of the configuration determined by $\bm v$ and $\bm w$ by
 \begin{align}\label{e.wgt definition}
 \mrm{Wgt}(\bm v,\bm w) := \prod_{y=1}^\numrows \mrm{Wgt}_y(\bm v,\bm w).
 \end{align}

The next lemma is a deterministic statement about the weights of configurations $w$ with high values of $\mrm{PE}(\bm w)$.

\begin{lemma}\label{l.alternate configuration weight ratio}
Fix $N, M \in\N$. There exist positive absolute constants $c$ and $C$ and $K=K(q)$, such that the following holds for any spectral parameter $z\in(0,1)$ (which may depend on $N,M$). Fix a compatible pair $(\bm v,\bm x)\in \{0,1,2\}^\numrows\times\{0,1\}^\numrows$. There exists a mapping $f:\mrm{Val}(\bm v,\bm x)\to \mrm{Val}(\bm v,\bm x)$ such that, if $\bm w\in\mrm{Val}(\bm v,\bm x)$ with $m:=\mrm{PE}(\bm w)$ and $\mrm{Wgt}(\bm v, \bm w) > 0$, then $\mrm{Wgt}(\bm v,\bm w)/\mrm{Wgt}(\bm v,f(\bm w)) \leq Cq^{cm^2}$ whenever $m> K\log N$. Further, for any $\bm w\in\mrm{Val}(\bm v, \bm x)$ with $m:=\mrm{PE}(\bm w)$ and $m\geq K\log N$, $\#f^{-1}(\{\bm w\}) \leq N^{2m}$.
\end{lemma}

This statement should be interpreted as giving a bound on the conditional probability of a given configuration exhibiting $\max_{y} (\smash{L^{(2)}_k(y) - \PT(L^{(1)}_k, L^{(2)}_{k+1})(y)}) = m$ given $\smash{L^{(1)}_k}$ and $\smash{L^{(2)}_{k+1}}$ (which will determine $\bm v$ and $\bm x$). Then Proposition~\ref{p.PT deviation probability} will follow by averaging over the conditioned data and summing over $m$; we give this proof now before proving Lemma~\ref{l.alternate configuration weight ratio} in Section~\ref{s.alternate config}.

\begin{proof}[Proof of Proposition~\ref{p.PT deviation probability}]
Recall from \eqref{e.F^cHL} that $\F^{\mrm{cHL}}_k$ is the $\sigma$-algebra generated by $\bm L^{(1)}$, $L^{(2)}_{k+1}, L^{(2)}_{k+2}$, $\ldots$. We condition on $\F^{\mrm{cHL}}_k$ and use the tower property of conditional expectations to write
\begin{align*}
\MoveEqLeft[6]
\P\left(\max_{y\in\intint{0,N+M}} \left(L^{(2)}_k(y) - \PT(L^{(1)}_k, L^{(2)}_{k+1})(y)\right) \geq m\right)\\
&= \E\left[\P\left(\max_{y\in\intint{0,N+M}} \left(L^{(2)}_k(y) - \PT(L^{(1)}_k, L^{(2)}_{k+1})(y)\right) \geq m \midd \F^{\mrm{cHL}}_k\right)\right].
\end{align*}
Let $\bm v \in \{0,1,2\}^{\numrows}$ and $\bm x\in\{0,1\}^\numrows$ be given by
\begin{align*}
v_y &= \left(L^{(2)}_{k+1}(y-1) - L^{(2)}_{k+1}(y)\right) + \left(L^{(1)}_{k+1}(y-1) - L^{(1)}_{k+1}(y)\right) \quad \text{and}\\
x_y &= L^{(1)}_k(y-1) - L^{(1)}_k(y);
\end{align*}
in words, they are the horizontally incoming colors and the horizontally outgoing arrow counts at column $k$. They are compatible. Let $A(\bm v,\bm x, m')$ be the set of $\bm w\in\mrm{Val}(\bm v,\bm x)$ such that $\mrm{PE}(\bm w) = m'$ and $\mrm{Wgt}(\bm v, \bm w) > 0$. Now, it is immediate from the Gibbs property of the colored Hall-Littlewood line ensemble, Lemma~\ref{l.colored line ensemble gibbs} (and can also be seen directly from the Gibbs property of the colored $q$-Boson model, Lemma~\ref{l.colored vertex gibbs}), that
\begin{align}
\MoveEqLeft[18]
\P\left(\max_{y\in\intint{0,N+M}} \left(L^{(2)}_k(y) - \PT(L^{(1)}_k, L^{(2)}_{k+1})(y)\right) \geq m \midd \F^{\mrm{cHL}}_k\right) \nonumber\\
&= \frac{\sum_{m'=m}^\infty\sum_{\bm w\in A(\bm v,\bm x,m')} \mrm{Wgt}(\bm v,\bm w)}{\sum_{\bm w\in\mrm{Val}(\bm v,\bm x)} \mrm{Wgt}(\bm v,\bm w)}. \label{e.pitman transform conditional prob}
%
\end{align}
For each $\bm w\in A(\bm v,\bm x,m')$, let $f(\bm{w})\in\mrm{Val}(\bm v,\bm x)$ be the configuration obtained from Lemma~\ref{l.alternate configuration weight ratio}, and recall $\#f^{-1}(\bm w)\leq N^{2m'}$. Since $m'\geq m\geq K\log N$, we know that $\mrm{Wgt}(\bm v,\bm w)/\mrm{Wgt}(\bm v,f(\bm{w}))$ is upper bounded by $Cq^{c(m')^2}$. 
So we see that, since $\#f^{-1}(\{\bm w\})\leq N^{2m'}$ for $\bm w\in A(\bm v, \bm x, m')$,
\begin{align*}
\sum_{\bm w\in A(\bm v,\bm x, m')}\! \mrm{Wgt}(\bm v,\bm w) &\leq Cq^{c(m')^2}\!\sum_{\bm w\in A(\bm v,\bm x, m')}\! \mrm{Wgt}(\bm v, f(\bm w))\\
&\leq CN^{2m'}q^{c(m')^2}\!\sum_{\bm w\in \mrm{Val}(\bm v, \bm x)}\! \mrm{Wgt}(\bm v,\bm w).
\end{align*}
%
Substituting this into \eqref{e.pitman transform conditional prob}, and using that $q^{\frac{1}{2}c(m')^2}\leq C^{-1}N^{-2m'}$ by raising $K$ if necessary (and recalling $m'\geq m\geq K\log N$), yields that the righthand side of that equation is upper bounded by $q^{c'm^2}$ for some absolute constant $c>0$.  This completes the proof.
\end{proof}

\subsection{Construction of the alternate configuration}\label{s.alternate config}

\begin{proof}[Proof of Lemma~\ref{l.alternate configuration weight ratio}]
First, to ensure that the $f$ we need to produce is defined for all inputs, for any $\bm w\in \mrm{Val}(\bm v, \bm x)$ such that $\mrm{PE}(\bm w)< K\log N$ (for a $K=K(q)$ to be set later) or $\mrm{Wgt}(\bm v, \bm w) = 0$, we define $f(\bm w) = \bm w$. So in the remainder of the proof, we may assume that $\bm w$ is such that $\mrm{PE}(\bm w) > K\log N$ and $\mrm{Wgt}(\bm v, \bm w) > 0$.
It will be useful to keep track of a running Pitman error by row $y$, for which we define
\begin{align*}
\mrm{PE}(\bm w, y) := \max_{y'\leq y} \left(\mrm{Ht}_2(\bm w, y') - \mrm{Ht}_2(\bm w^*, y')\right).
\end{align*}
Let $m' \in \intint{1, m}$. We fix $y = y(m')$ to be the smallest such that $\mrm{PE}(\bm w,y) = m'$; since $\mrm{PE}(\bm w,0) = 0$, $\bm w$ is such that $\mrm{PE}(\bm w,\numrows) = \mrm{PE}(\bm w) = m$, and $\mrm{PE}(\bm w,y')$ increases by at most $1$ when $y'$ increases by $1$, such a $y$ exists.

We claim that $w_y = 1$ and $w_y^*=2$. We first observe that since $\mrm{PE}(\bm w,y) > \mrm{PE}(\bm w,y-1)$ and $\mrm{Ht}_2(\bm w, \bm\cdot)$, $\mrm{Ht}_2(\bm w^*, \bm\cdot)$ are non-increasing and have increments in $\{-1,0\}$, it must hold that $\mrm{Ht}_2(\bm w,y) = \mrm{Ht}_2(\bm w,y-1)$ and $\mrm{Ht}_2(\bm w^*,y) = \mrm{Ht}_2(\bm w^*,y-1)-1$. This means that $w^*_y=2$, which implies $w_y=1$ since $w_y\in\{1,2\}$ (both $\bm w$ and $\bm w^*$ have the same outgoing arrow counts) and $\mrm{Ht}_2(\bm w,y) = \mrm{Ht}_2(\bm w,y-1)$.

Next we claim that $\mrm{VertCt}_2(\bm v, \bm w, y-1)\geq m'$. To see this we first write, using the definition \eqref{e.vertct definition} of $\mrm{VertCt}_2(\bm v, \bm w, y-1)$ and noting that, by definition of $y$, $\mrm{Ht}_2(\bm w, y-1)- \mrm{Ht}_2(\bm w^*, y-1) = m'-1$,
\begin{align}
\mrm{VertCt}_2(\bm v, \bm w, y-1)
&= \mrm{Ht}_2(\bm w, y-1) - \mrm{Ht}_2(\bm v, y-1)\nonumber\\
&= \mrm{VertCt}_2(\bm v, \bm w^*, y-1) + \mrm{Ht}_2(\bm w, y-1)- \mrm{Ht}_2(\bm w^*, y-1)\label{e.vert count pe error bound}\\
&= \mrm{VertCt}_2(\bm v, \bm w^*, y-1) + m'-1.\nonumber
\end{align}
Now it suffices to show that $\mrm{VertCt}_2(\bm v, \bm w^*, y-1) \geq 1$. We know that $w_y^* = 2$, so $v_y=2$ or $\mrm{VertCt}_2(\bm v, \bm w^*, y-1) \geq 1$ (or both). But $w_y=1$ and the weight of the vertex with a horizontally incoming $2$-arrow and an outgoing $1$-arrow is zero, by Figure~\ref{f.L weights}. So we have that $v_y\neq 2$ (since we have assumed $\mrm{Wgt}(\bm v, \bm w) > 0$), which implies that $\mrm{VertCt}_2(\bm v, \bm w^*, y-1) \geq 1$. Returning to \eqref{e.vert count pe error bound}, we have shown that $\mrm{VertCt}_2(\bm v, \bm w, y-1) \geq m'$.
Thus we have now established that with $y$ defined as above, it holds that
\begin{align}\label{e.bdy conditions established for alt config}
w_y = 1,\, w^*_y=2,\, \mrm{VertCt}_2(\bm v, \bm w, y-1) \geq m', \text{ and } \mrm{VertCt}_2(\bm v, \bm w^*, y-1) \geq 1.
\end{align}

With these preliminaries in place, we turn to exhibiting a particular configuration $f(\bm w)$ such that $\mrm{Wgt}(\bm w)/\mrm{Wgt}(f(\bm w)) \leq Cq^{cm^2}$.
We let $K_0 := 2(\log q^{-1})^{-1}$ and $K:=2K_0$. For $m\geq K\log N$, we will define a sequence of configurations $\bm w^{(0)}, \bm w^{(1)}, \ldots, \bm w^{(R)} \in \mrm{Val}(\bm v,\bm x)$ where $\bm w^{(0)} := \bm w$ and $R := m - K_0\log N$ such that $\bm w^{(i)}$ has $\mrm{PE}(\bm w^{(i)}) \geq m-i$ and
\begin{align}\label{e.configuration comparison}
\frac{\mrm{Wgt}(\bm v,\bm w^{(i-1)})}{\mrm{Wgt}(\bm v,\bm w^{(i)})} \leq q^{m-i-K_0\log N-1}(1+q^{K_0\log N-1})^N.
\end{align}
We define $f(\bm w)$ to be $\bm w^{(R)}$. Clearly this suffices to establish $\mrm{Wgt}(\bm w)/\mrm{Wgt}(f(\bm w)) \leq Cq^{cm^2}$ as then by multiplying together the inequalities we would obtain that, for absolute constants $c$ and $C$,
\begin{align}
\frac{\mrm{Wgt}(\bm v,\bm w)}{\mrm{Wgt}(\bm v,\bm w^{(R)})} \leq q^{mR-\sum_{i=1}^R i - RK_0\log N - R}(1+q^{K_0\log N-1})^{RN} \leq Cq^{cm^2},
\end{align}
the final inequality for $m \geq K\log N$. So now we must show \eqref{e.configuration comparison}; we will establish that $\# f^{-1}(\{\tilde{\bm w}\})\leq N^{2m}$ for any $\tilde{\bm w}$ with $\mrm{PE}(\tilde{\bm w}) = m \geq K\log N$ and $\mrm{Wgt}(\bm v, \bm w)>0$ after specifying the sequence of configurations.

Reindexing, to establish \eqref{e.configuration comparison}, we must show that if we have a configuration $\bm w$ with $\mrm{PE}(\bm w) \geq m' \geq K\log N$, then we can produce a configuration $\bm w'$ with $\mrm{PE}(\bm w') \geq m'-1$ and
\begin{equation}\label{e.comparison to show}
\frac{\mrm{Wgt}(\bm v,\bm w)}{\mrm{Wgt}(\bm v, \bm w')} \leq q^{m'-K_0\log N-1}(1+q^{K_0\log N-1})^N.
\end{equation}

Now, by \eqref{e.bdy conditions established for alt config}, since $\mrm{PE}(\bm w) = m$, for any $m'\in\intint{K\log N,m}$ there exists a first vertex (whose row height we labeled $y$), where a $1$-arrow exits horizontally and there are at least $m'$ 2-arrows entering vertically, i.e., $w_y = 1$ and $\mrm{VertCt}_2(\bm v,\bm w,y-1)\geq m'$. Since we have assumed $m'\geq K\log N$ and all $2$-arrows eventually exit, there exists a value $y' > y$ which is the smallest one greater than $y$ such that $v_{y'}\neq 2$, $w_{y'} = 2$, and $\mrm{VertCt}_2(\bm w, y'-1) = \ceil{\frac{1}{2}K\log N} = \ceil{K_0\log N}$. In particular, for all $u\in\intint{y, y'-1}$,
\begin{equation}\label{e.vertct u bound}
\mrm{VertCt}_2(\bm v,\bm w, u) \geq \tfrac{1}{2}K\log N = K_0\log N.
\end{equation}

We define $\bm w'$ by swapping the exiting arrows at $y$ and $y'$, i.e., by $w'_u = w_u$ for all $u\not\in\{y, y'\}$, $w'_{y'} = w_y=1$, and $w'_y = w_{y'}=2$. Observe that since there are at most $N$ possible locations for such a swap, and $f(\bm w)$ is defined by a sequence of $R = m- K_0\log N < m$ such swaps, it follows that $\#f^{-1}(\{\tilde{\bm w}\}) \leq \binom{N}{2}^m < N^{2m}$ for any $\tilde{\bm w}\in\mrm{Val}(\bm v, \bm x)$ with $\mrm{PE(}\tilde{\bm w}) = m$. Next observe that, for $u\in \intint{y,y'-1}$,
\begin{equation}\label{e.arrow quiver relations}
\begin{split}
\mrm{VertCt}_1(\bm v, \bm w', u) &= \mrm{VertCt}_1(\bm v,\bm w, u) + 1 \ \text{  and  }\\
\mrm{VertCt}_2(\bm v, \bm w', u) &= \mrm{VertCt}_2(\bm v,\bm w, u) - 1.
\end{split}
\end{equation}

It is immediate that $\mrm{PE}(\bm w') \geq m'-1$. To upper bound $\mrm{Wgt}(\bm v,\bm w)/\mrm{Wgt}(\bm v,\bm w')$, we note that the vertex configurations at $u$ are identical in $\bm w$ and $\bm w'$ for $u\not\in \intint{y,y'}$, so the ratio is the ratio of the product of the vertex weights for vertices at $u$ with $u\in \intint{y, y'}$, i.e., (recall the definition of $\mrm{Wgt}_u$ and $\mrm{Wgt}$ from \eqref{e.wgt_y definition} and \eqref{e.wgt definition} respectively)
\begin{align*}
\frac{\mrm{Wgt}(\bm v,\bm w)}{\mrm{Wgt}(\bm v,\bm w')} = \prod_{u=y}^{y'} \frac{\mrm{Wgt}_u(\bm v,\bm w)}{\mrm{Wgt}_u(\bm v,\bm w')}.
\end{align*}
To estimate this, we first note that $w_u = w'_u$ for $u\in \intint{y+1, y'-1}$ and recall that \eqref{e.arrow quiver relations} holds for $u\in\intint{y,y'-1}$.

Recall the spectral parameter $z\in(0,1)$. For a vertex height $u\in\intint{1,\numrows}$, let
\begin{align}\label{e.tilde z definition}
\tilde z(u) := z \text{ if } u\in \intint{N+1, N+M} \text{ and } \tilde z(u) := 1 \text{ if } u\in\intint{1,N},
\end{align}
and, for $j\in\{1,2\}$,
\begin{align}\label{e.ell_j definition}
A_j(u) := \mrm{VertCt}_j(\bm v, \bm w, u-1).
\end{align}

In the following, we will be reading vertex weights from Figure~\ref{f.L weights}. Now, since $w_y = 1$, $\mrm{Wgt}_{y}(\bm v, \bm w) = \tilde z(y)q^{A_2(y)}$ or $\tilde z(y)(1-q^{A_1(y)})q^{A_2(y)}$, depending on whether $v_y = 1$ or $0$ respectively. Since $A_2(y) \geq m$ as shown in \eqref{e.bdy conditions established for alt config}, both possibilities for the weight are upper bounded by $\tilde z(y)q^m$. Similarly since $w_y' = 2$ and $\mrm{VertCt}_2(\bm v, \bm w', y-1) = \mrm{VertCt}_2(\bm v, \bm w, y-1) = A_2(y)$, $\mrm{Wgt}_{y}(\bm v, \bm w') = \smash{\tilde z(y)(1-q^{A_2(y)})} \geq \tilde z(y)(1-q^m)$. Thus
\begin{align*}
\frac{\mrm{Wgt}_y(\bm v, \bm w)}{\mrm{Wgt}_y(\bm v, \bm w')} \leq \frac{\tilde z(y) q^m}{\tilde z(y)(1-q^m)} = \frac{q^m}{1-q^m}.
\end{align*}
Next we turn to the $y'$ vertex. Here, $w_{y'} = 2$ and $A_2(y') = \ceil{K_0\log N}$ (see before \eqref{e.vertct u bound}), so $\mrm{Wgt}_{y'}(\bm v, \bm w) = \tilde z(y')(1-q^{\ceil{K_0\log N}})$ independently of $v_{y'}\in\{0,1\}$, while $\mrm{Wgt}_{y'}(\bm v, \bm w') = \tilde z(y')q^{\ceil{K_0\log N}-1}$ or $\tilde z(y')(1-q^{\mrm{VertCt}_1(\bm v,\bm w', y'-1)})q^{\ceil{K_0\log N} -1}$ (depending on $v_{y'} = 1$ or $0$ respectively, since $v_{y'}\neq 2$), which is lower bounded by $\tilde z(y')(1-q)q^{\ceil{K_0\log N} -1}$, since, as noted in \eqref{e.arrow quiver relations}, $\mrm{VertCt}_2(\bm v,\bm w',y'-1) = \mrm{VertCt}_2(\bm v,\bm w,y'-1)-1 = \ceil{K_0\log N} -1$. So
\begin{align*}
\frac{\mrm{Wgt}_{y'}(\bm v,\bm w)}{\mrm{Wgt}_{y'}(\bm v,\bm w')} \leq \frac{\tilde z(y')(1-q^{\ceil{K_0\log N}})}{\tilde z(y') (1-q)q^{\ceil{K_0\log N}-1}} = \frac{1-q^{\ceil{K_0\log N}}}{(1-q)q^{\ceil{K_0\log N}-1}}.
\end{align*}
Finally we turn to looking at the weight of vertex $u$ for $u\in \intint{y+1, y'-1}$. Recall that $w_u = w'_u$ for this range of $u$. If $w_u = 0$, the vertex weights are $1$ for both $\bm w$ and $\bm w'$. If $w_u = 1$, then, using \eqref{e.arrow quiver relations}, and with $\tilde z(u)$ as in \eqref{e.tilde z definition} and $A_j(u)$ as in \eqref{e.ell_j definition},
$$\frac{\mrm{Wgt}_{u}(\bm v, \bm w)}{\mrm{Wgt}_u(\bm v, \bm w')} =
\begin{cases}
\displaystyle\frac{\tilde z(u)(1-q^{A_1(u)}) q^{A_2(u)}}{\tilde z(u)(1-q^{A_1(u)+1}) q^{A_2(u)-1}} & v_u = 0\\[12pt]
\displaystyle\frac{\tilde z(u)q^{A_2(u)}}{\tilde z(u) q^{A_2(u)-1}} & v_u = 1,
\end{cases}$$
which is at most $1$. Finally, we consider the case $w_u = 2$. Recall that $A_2(u) \geq \ceil{K_0\log N}$ by \eqref{e.vertct u bound}. If $v_u = 2$, $\mrm{Wgt}_{u}(\bm v, \bm w)/\mrm{Wgt}_u(\bm v, \bm w') = 1$, while if $v_u = 0$ or $1$,
\begin{align*}
\frac{\mrm{Wgt}_{u}(\bm v, \bm w)}{\mrm{Wgt}_u(\bm v, \bm w')} = \frac{\tilde z(u)(1-q^{A_2(u)})}{\tilde z(u)(1-q^{A_2(u)-1})} = 1+q^{A_2(u)-1}\cdot\frac{1-q}{1-q^{A_2(u)-1}} \leq 1+q^{\ceil{K_0\log N}-1}.
\end{align*}
Combining the above inequalities, and using that the number of locations $u$ where $w_u\neq 0$ is $N$ (and that if $w_u=0$ the ratio is 1), we obtain that
\begin{align*}
\frac{\mrm{Wgt}(\bm v,\bm w)}{\mrm{Wgt}(\bm v,\bm w')} = \prod_{u=y}^{y'} \frac{\mrm{Wgt}_u(\bm v,\bm w)}{\mrm{Wgt}_u(\bm v,\bm w')}
&\leq \frac{q^m}{1-q^m}\cdot \frac{1-q^{\ceil{K_0\log N}}}{q^{\ceil{K_0\log N} -1}} \left(1+q^{K_0\log N-1}\right)^N\\
&\leq q^{m-\ceil{K_0\log N}} \left(1+q^{\ceil{K_0\log N}-1}\right)^N.
\end{align*}
This implies \eqref{e.comparison to show}, completing the proof.
\end{proof}

\section{Properties of ASEP and S6V} \label{s.properties of S6V and ASEP}

In this section we collect some properties of the colored S6V model and ASEP that will be useful in future sections, in particular, a form of stationarity and one-point tightness.

\subsection{Colored stochastic six-vertex model}
We work with packed boundary condition for the colored S6V model; recall the definition from Section~\ref{s.intro.cS6V}. In particular, we have arrows entering horizontally at every site along $\{1\}\times\intint{-N,N}$, with the color of the arrow entering $(1,k)$ being $k$. Also recall that the color of the arrow exiting horizontally from $(t,k)$ is $j_{(t,k)}$ and from \eqref{e.s6v height function} that
\begin{align*}
\hssv(x, 0; y, t) = \#\left\{k> y: j_{(t,k)}\geq x\right\}
\end{align*}
for $x\in\intint{-N,N}$ and $y\in\llbracket-N,\infty\rrparen$. 
It is easy to check from this description that, for any $k\in\intint{-N,N}$ fixed,
\begin{align}\label{e.two parameter stationarity}
\hssv(x, 0;y, t)  \stackrel{d}{=} \hssv(x+k,0;y+k,t) +k
\end{align}
as processes in $(x,y,t)$ on the domain $\intint{\max(-N-k,-N), \min(N-k, N)}\times \llbracket\max(-N-k,-N), \infty\rrparen\times\N$; this uses that, when $y\geq x$, $\hssv(x,0;y,t) - (N-y)$ is the number of arrows starting at a site $(1,i)$ for some $i\in\intint{x,y}$ which horizontally exit a vertex $(t,i)$ for some $i\in\llbracket y+1,\infty\rrparen$, which is a quantity whose distribution is invariant on shifting $x$ and $y$ by the same constant; note that for a shift by $k$, $x$ and $y$ are restricted to a $k$-dependent domain to ensure that there exist particles at all the sites of $\intint{x+k,y+k}$, as required for this invariance to hold.

The following records a one-point convergence result on the limiting one-point fluctuations of $\hssv$ in its rescaled form $\S^{\mrm{S6V},\varepsilon}$ as defined in \eqref{e.rescaled s6v definition} (which recall involves a parameter $\alpha$ controlling the location in the rarefaction fan where the height function is evaluated). It was originally essentially shown in \cite{borodin2016stochastic}, but proved in the slightly improved form we require in  \cite{corwin2018transversal} .

\begin{proposition}[One-point distributional limit of S6V]\label{p.s6v one-point}
Let $q\in$ $[0,1)$, $z\in(0,1)$, and $\alpha\in(z, z^{-1})$. Fix $x,y\in\R$. Then, as $\varepsilon\to0$, $\smash{\S^{\mrm{S6V},\varepsilon}(x;y) + (x-y)^2 \stackrel{d}{\to} \mrm{GUE}\text{-}\mrm{TW}}$, the GUE Tracy-Widom distribution (see end of Definition~\ref{d.parabolic Airy line ensemble}).

\end{proposition}

\begin{proof}
By \eqref{e.two parameter stationarity} and the expression defining $\S^{\mrm{S6V},\varepsilon}$ in Definition~\ref{d.s6v sheet}, it follows that $\S^{\mrm{S6V}, \varepsilon}(x;y)\stackrel{\smash{d}}{=} \S^{\mrm{S6V}, \varepsilon}(0; y-x)$, so it suffices to prove the case $x=0$. When $q\in(0,1)$, this case is an immediate consequence of combining \cite[Theorems 2.3 and 2.5]{corwin2018transversal} and translating the notation. When $q=0$, we first note that by Proposition~\ref{p.colored line ensembles}, $\S^{\mrm{S6V},\varepsilon}(0;\bm\cdot)$ is a rescaled version of the top line of the uncolored Hall-Littlewood line ensemble $\smash{\bm L^{\mrm{cHL},(1)}(\bm\cdot)}$ from Section~\ref{s.colored line ensemble}, which consists of non-intersecting Bernoulli paths. The scaling limit of the latter is proven in \cite[Corollary 1.3]{dimitrov2021tightness} as well as in \cite[Theorem 1.5]{dauvergne2023uniform} to be $\cP$, from which the one-point statement we claim follows immediately.
\end{proof}

\subsection{Colored ASEP and relation to colored S6V}\label{sec.colorconnection}
Next we recall the model of colored ASEP from Section~\ref{s.intro.colored asep}.
Because the approximate LPP representation of Theorem~\ref{t.approxmate LPP problem representation general} is in terms of the colored Hall-Littlewood line ensemble associated to the colored $q$-Boson model which has a direct distributional equivalence only with colored S6V, we need to relate colored ASEP to colored S6V.

Consider the colored S6V model under packed boundary condition with
\begin{equation}\label{e.ASEP N value}
N = \floor{\delta^{-1}t}
\end{equation}
for $\delta, t>0$  with $q\in[0,1)$ and $z=\frac{1-\delta}{1-\delta q}$ (from Figure~\ref{f.R weights}, this means that the probability of the higher color arrow moving straight through a vertex vertically is $q\delta$ and horizontally is $\delta$). We will evaluate the colored height function on the vertical line $\{M\}\times \llbracket -N,\infty\rrparen$ with $M=\floor{\delta^{-1}t}$. Consider colored ASEP with the usual jump rates of $1$ to the right and $q$ to the left. It is known from \cite{borodin2016stochastic,aggarwal2017convergence} that by evaluating the height function of the S6V model near $(M,N) = (\floor{\delta^{-1}t}, \floor{\delta^{-1}t})$ one obtains the ASEP height function in the distributional limit as $\delta\to 0$.

Unfortunately, this description poses issues to use directly, namely, in the $\delta\to 0$ limit the domain becomes infinite even for fixed $t$. This means that estimates we wish to apply in the ASEP setting like Theorem~\ref{t.approxmate LPP problem representation general} (probability bound on the deviation $m$ from the approximate LPP representation), which only apply for $m \geq K(q)\log N$, would not hold.

\subsubsection{Effective coupling} We instead make use of an effective coupling proven in \cite{aggarwalborodin} of colored S6V with $z=\frac{1-\delta}{1-\delta q}$ and $N = \floor{\delta^{-1}t}$ as in \eqref{e.ASEP N value} and colored ASEP at time $t$ which allows to take $\delta\to0$ at a particular rate with $t\to\infty$. We will take $\delta= t^{-\theta}$ for some $\theta>0$ large enough. We define
\begin{align}\label{e.approx asep height function definition}
\hasep(x, 0; y, t) := \floor{t^{1+\theta}} + x + 1 - \hssv(-x, 0; \floor{t^{1+\theta}}-y-1, \floor{t^{1+\theta}});
\end{align}
the first argument of $\hssv$ is $-x$ because $h^{\mrm{ASEP}}(x,0;\bm\cdot, t)$ involves counts of particles of color greater than or equal to $-x$ and not $x$ (recall \eqref{e.intro ASEP height function}). As we will see, the presence of negative signs in going between $h^{\mrm{ASEP},\theta}$ and $h^{\mrm{S6V}}$ arises from the fact that the ASEP particle configuration can be read off (under a coupling that we will introduce shortly) of a horizontal line in the S6V model, but the S6V height function considers arrow counts on a vertical line.

We note that \eqref{e.two parameter stationarity} implies that
\begin{equation}\label{e.two parameter stationarity for ASEP}
\hasep(\bm\cdot, 0; \bm\cdot, t) \stackrel{\smash{d}}{=} \hasep(\bm\cdot + k, 0; \bm\cdot +k, t)
\end{equation}
for any $k\in\Z$ fixed, matching the relation satisfied by $h^{\mrm{ASEP}}$ (it is easily seen from the latter's definition \eqref{e.intro ASEP height function} that \eqref{e.two parameter stationarity for ASEP} holds with $h^{\mrm{ASEP}}$ in place of $\hasep$).

The properties of the coupling of $\hasep$ and $h^{\mrm{ASEP}}$ that we need are as follows.

\begin{lemma}[Effective coupling of ASEP with S6V]\label{l.asep s6v comparison}
There exists $C>0$ such that, for any $\theta>0$ and $t\geq 1$, there is a coupling such that $\hasep(x,0;y,t) = h^{\mrm{ASEP}}(x, 0;y, t)$ for all $|x|,|y|\leq 2t$ with probability at least $1-Ct^{3-\theta/8}$.
\end{lemma}

In the remainder of the paper we will assume $\theta$ is a large fixed constant such as $40$ so that, by the Borel-Cantelli lemma, Lemma~\ref{l.asep s6v comparison} implies that with probability one $\hasep(x,0;y,t) = h^{\mrm{ASEP}}(x,0;y,t)$ for all $|x|,|y|\leq 2t$ and for all $t$ large enough. To prove Lemma~\ref{l.asep s6v comparison}, we will need the following statements for colored ASEP and colored S6V, which are essentially implications of the finite speed of propagation in the models. As their proofs are straightforward, we defer them to Appendix~\ref{s.random gibbs}. Define
\begin{align}\label{e.h^asep cutoff}
\tilde h^{\mrm{ASEP}}(x,0;y,t) := \#\{z\in\intint{y+1, \floor{4t}}: \eta_t(z) \geq -x\},
\end{align}
where recall $(\eta_t(k))_{k\in\Z}$ records the state of the colored ASEP at time $t$. In words, $\tilde h^{\mrm{ASEP}}$ is a variant of the colored ASEP height function \eqref{e.intro ASEP height function} where particles beyond location $\floor{4t}$ are not counted.

\begin{lemma}\label{l.asep finite speed of propagation}
There exist $C, c>0$ such that, for all $t>0$, with probability at least $1-C\exp(-ct)$,
\begin{align}
\left(h^{\mrm{ASEP}}(x,0;y,t)\right)_{|x|,|y|\leq 2t} = \bigl(\tilde h^{\mrm{ASEP}}(x,0;y,t)\bigr)_{|x|,|y|\leq 2t},
\end{align}
\end{lemma}

\begin{lemma}\label{l.s6v finite pseed of propagation}
Let $\delta$, $t>0$, $N=\floor{\delta^{-1}t}$, and $z=\frac{1-\delta}{1-\delta q}$. There exist absolute constants $C,c>0$ such that, with probability at least $1-C\exp(-ct)$, it holds that, for all $|x|\leq 2t$ and $z\geq \floor{4t}$,
\begin{align*}
\hssv(x, 0;\floor{\delta^{-1}t} - z, \floor{\delta^{-1}t}) = \floor{\delta^{-1}t}-x+1,
\end{align*}
where note the righthand side is the total number of arrows of color at least $x$ in the system.
\end{lemma}

\begin{proof}[Proof of Lemma~\ref{l.asep s6v comparison}]
Let $\delta>0$. Recall $\hssv(-x, 0; \floor{\delta^{-1}t}-y-1, \floor{\delta^{-1}t})$ is the number of arrows of color greater than or equal to $-x$ passing strictly above $(\floor{\delta^{-1}t}, \floor{\delta^{-1}t}-y-1)$, and let $\tilde h^{\mrm{ASEP}}$ be as in \eqref{e.h^asep cutoff}. Also recall the packed boundary and initial conditions for colored S6V and colored ASEP, respectively. We note that in both colored S6V and colored ASEP, color merging (Remarks~\ref{r.asep color merging} and \ref{r.s6v color merging}) implies that the distribution of $(\tilde h^{\mrm{ASEP}}(x,0;y,t))_{|x|,|y|\leq 2t}$ and $\smash{(h^{\mrm{S6V}}(x,0;y,t))_{|x|,|y|\leq 2t}}$ are unchanged by merging all particles/arrows of color strictly greater than $\floor{2t}$ to have color $\floor{2t}+1$, and merging all particles/arrows of color strictly smaller than $-\floor{2t}$ to have color $-\floor{2t}-1$. We work with these color merged versions of the packed initial conditions (this is required as the result we invoke from \cite{aggarwalborodin} states that only finitely many colors exist in the systems).

We work with the coupling whose existence is asserted in \cite[Proposition~B.2]{aggarwalborodin}, where we set the parameters there as follows: $L=q$, $R=1$, $\varepsilon=\delta$ (this corresponds to the spectral parameter $z=\frac{1-\delta}{1-\delta q}$), $n=4t$, $T=t$, $\mc T=\{t\}$, $U=-4t$, and $V=4t$.  Then we know that, on an event $\msf E_1$ with probability at least $1-C\delta^{1/8}t^3$ (and we work on this event in the remainder of the proof), for all $|x| \leq 2t$, $|y|\leq 4t$, there is an arrow of color $x$ exiting vertically from $(\floor{\delta^{-1}t}+y, \floor{\delta^{-1}t})$ (i.e., along the horizontal line at height $\floor{\delta^{-1}t}$) if and only if there is a particle of color $x$ at location $y$ in colored ASEP at time $t$. However, the S6V height function we defined in \eqref{e.s6v height function} is in terms of arrow counts on a vertical and not horizontal line.

To handle this, we claim that, with probability at least $1-32\delta t^2$, the color of the arrow exiting horizontally from $(\floor{\delta^{-1}t}, \floor{\delta^{-1}t}- y)$ is the same as that of the arrow exiting vertically from $(\floor{\delta^{-1}t}+y, \floor{\delta^{-1}t})$ for all $y\in\intint{-\floor{4t},\floor{4t}}$; call the event that this occurs for a given $y$ by $\msf E_2(y)$ and define $\msf E_2 = \cap_{y\in\intint{-\floor{4t}, \floor{4t}}} \msf E_2(y)$. Let $\bm v_t = (\floor{\delta^{-1}t}, \floor{\delta^{-1}t})$ and $\ell_t = \floor{4t}$, and define the squares $S_1$ and $S_2$ with corners given by $\{\bm v_t, \bm v_t + (-\ell_t,0), \bm v_t +(0,\ell_t), \bm v_t +(-\ell_t, \ell_t)\}$ and $\{\bm v_t, \bm v_t + (\ell_t,0), \bm v_t +(0,-\ell_t), \bm v_t +(\ell_t, -\ell_t)\}$, respectively. To see that our claim holds, note first that for $\msf E_2(y)$ to not hold for some $y \in \intint{-\floor{4t}, \floor{4t}}$, it must be the case that at least one vertex in the squares $S_1$ and $S_2$ (for the cases of $y\leq 0$ and $y>0$ respectively) allows at least one arrow to pass straight through vertically or horizontally; the probability of this happening at any given vertex is at most $\delta$. A union bound over the $32t^2$ vertices in the two squares proves the claim that $\P(\msf E_2^c)\leq 32\delta t^2$.

Let $\msf E_3$ and $\msf E_4$ be the events that $\tilde h^{\mrm{ASEP}}(x,0;y,t) = h^{\mrm{ASEP}}(x,0;y,t)$ and $\hssv(x, 0;\floor{\delta^{-1}t} -\floor{4t}, \floor{\delta^{-1}t}) = \floor{\delta^{-1}t}-x+1$ hold for all $|x|, |y|\leq 2t$. By Lemmas~\ref{l.asep finite speed of propagation} and \ref{l.s6v finite pseed of propagation}, $\P((\msf E_3\cup\msf E_4)^c) \leq C\exp(-ct)$ for some $C,c>0$.

Set $\delta = t^{-\theta}$. Under the coupling above and on $\msf E:=\msf E_1\cap\msf E_2\cap\msf E_3\cap\msf E_4$, we will show that $\hasep(x,0;y,t) = h^{\mrm{ASEP}}(x,0;y,t)$ for all $|x|,|y|\leq 2t$. First, by the definition of $\tilde h^{\mrm{ASEP}}$, on $\msf E_1\cap\msf E_2$ and for $|x|, |y|\leq 2t$,
\begin{align*}
\MoveEqLeft[1]
\tilde h^{\mrm{ASEP}}(x,0;y,t)\\
&= \#\left\{k\in\intint{y+1, \floor{4t}} : j_{(\floor{\delta^{-1}t}, \floor{\delta^{-1}t}-k)}\geq -x\right\}\\
&= h^{\mrm{S6V}}(-x,0; \floor{\delta^{-1}t}-\floor{4t}-1, \floor{\delta^{-1}t}) - h^{\mrm{S6V}}(-x, 0; \floor{\delta^{-1}t}-y-1, \floor{\delta^{-1}t})\\
&= \floor{\delta^{-1}t} + x + 1- h^{\mrm{S6V}}(-x, 0; \floor{\delta^{-1}t}-y-1, \floor{\delta^{-1}t}) = \hasep(x,0; y,t),
\end{align*}
the penultimate equality on $\msf E_3$. The proof is completed by noting $\tilde h^{\mrm{ASEP}}(x,0;y,t) = h^{\mrm{ASEP}}(x,0;y,t)$ on $\msf E_4$ and that it holds for all large enough $C$ and all $t\geq 1$ that $\P(\msf E^c) \leq C(\exp(-ct) + \delta t^2 + \delta^{1/8}t^3) \leq C\delta^{1/8}t^3$ (since $\delta=t^{-\theta}$).
\end{proof}

\subsubsection{One-point tightness} The one-point tightness for $\hasep$ that we require follows from Lemma~\ref{l.asep s6v comparison} and the following one-point convergence result for $h^{\mrm{ASEP}}$. Recall the ASEP sheet $\S^{\mrm{ASEP},\varepsilon}$ from \eqref{e.rescaled asep definition} and the parameter $\alpha$ controlling the macroscopic location around which the height function $h^{\mrm{ASEP}}$ is evaluated.

\begin{proposition}\label{p.asep one-point tightness}
Let $q\in[0,1)$ and $\alpha\in(-1,1)$. Fix $x,y\in\R$. Then, as $\varepsilon\to0$, $\S^{\mrm{ASEP},\varepsilon}(x;y) + (x-y)^2\stackrel{d}{\to} \mrm{GUE}\text{-}\mrm{TW}$, the GUE Tracy-Widom distribution.

\end{proposition} 

\begin{proof}
Again by stationarity of $h^{\mrm{ASEP}}$ (as noted after \eqref{e.two parameter stationarity for ASEP}) it suffices to prove the case $x=0$. This follows from \cite[Theorem 3]{tracy2009asymptotics}, and is also stated in a form closer to ours in \cite[Theorem 11.3]{borodin2017asep}.
\end{proof}

We can combine this result with the coupling in Lemma~\ref{l.asep s6v comparison} to obtain a one-point tightness statement for $\hasep$. The rescaled version $\S^{\mrm{ASEP},\theta,\varepsilon}$ we work with is defined by \eqref{e.rescaled asep definition} with $\hasep$ in place of $h^{\mrm{ASEP}}$ with $\theta=40$, i.e., for $\alpha\in(-1,1)$ fixed and recalling $\mu(\alpha)$, $\sigma(\alpha)$, and $\beta(\alpha)$ from \eqref{e.mu sigma ASEP} and \eqref{e.nu ASEP},
\begin{equation}\label{e.rescaled approximate asep definition}
\begin{split}
\MoveEqLeft[9]
\S^{\mrm{ASEP},\theta,\varepsilon}(x; y) := \sigma(\alpha)^{-1}\varepsilon^{1/3}\Bigl(\mu(\alpha)2\varepsilon^{-1}  + \mu'(\alpha)\beta(\alpha)(y-x)\varepsilon^{-2/3}\\
&- h^{\mrm{ASEP},\theta}\bigl(\beta(\alpha)x\varepsilon^{-2/3}, 0; 2\alpha \varepsilon^{-1} + \beta(\alpha)y\varepsilon^{-2/3}, 2\gamma^{-1}\varepsilon^{-1}\bigr)\Bigr).
\end{split}
\end{equation}

\begin{corollary}\label{c.approx asep tightness}
Fix $\theta=40$. Let $q\in[0,1)$, $\alpha\in(-1,1)$ and $x,y\in\R$. Then, as $\varepsilon\to 0$,  $\S^{\mrm{ASEP},\theta,\varepsilon}(x;y) +(x-y)^2 \stackrel{d}{\to} \mrm{GUE}\text{-}\mrm{TW}$, the GUE Tracy-Widom distribution.
\end{corollary}

\begin{proof}
This follows immediately from Lemma~\ref{l.asep s6v comparison} and Proposition~\ref{p.asep one-point tightness}.
\end{proof}


\section{Preliminaries for line ensemble tightness}\label{s.tightness preliminaries}

In this section we start by introducing some line ensemble notions that will be necessary in the proof of Theorem~\ref{t.tightness} on tightness of line ensembles. Then in Section~\ref{s.bg} we introduce the Brownian Gibbs property and give the proof of Theorem~\ref{t.line ensemble convergence to parabolic Airy} (on the convergence of line ensembles to the parabolic Airy line ensemble), assuming Theorem~\ref{t.tightness} and that all subsequential limits possess the Brownian Gibbs property, which will be proven in later sections. In Section~\ref{s.assumptions for S6V and ASEP} we will relate the colored Hall-Littlewood line ensemble, with appropriate centerings, to colored ASEP (in the approximate form \eqref{e.scaled approximate LPP representation}) and colored S6V such that the former's top curves are the height functions of the respective models, in such a way that the corresponding uncolored line ensembles satisfy the assumptions introduced in Section~\ref{s.line ensemble hypotheses}. Finally in Section~\ref{s.tightness ingredients} we collect some ingredient statements needed to prove Theorem~\ref{t.tightness} and give the latter's proof assuming them; the ingredient statements will be proved in later sections.

\subsection{The Strong Gibbs property}

We will need at one or two places a stronger version of the Gibbs property which allows us to resample on certain random domains known as stopping domains, analogous to the strong Markov property and stopping times.

\begin{definition}[Stopping domain]\label{d.stopping domain}
Fix $k\in\N$ and let $\bm L$ be a discrete line ensemble (Definition~\ref{d.discrete line ensemble}). A pair of random variables $\mf l, \mf r\in\Z$ forms a stopping domain $\intint{\mf l, \mf r}$ with respect to $(L_1, \ldots, L_k)$ if, for all $\ell, r\in\Z$ with $\ell < r$,
\begin{align*}
\{\mf l \leq \ell, \mf r \geq r\} \in \Fext(k, \intint{\ell, r}, \bm L).
\end{align*}
We define the $\sigma$-algebra $\Fext(k, \intint{\mf l, \mf r}, \bm L)$ to be the collection of all events $A$ such that $A\cap\{\mf l \leq \ell, \mf r \geq r\} \in \Fext(k,\intint{\ell, r},  \bm L)$ (as in Definition~\ref{d.F_ext}) for all $\ell < r$.
\end{definition}

\begin{proposition}[Strong Gibbs property]\label{p.strong gibbs}
Let $k\in\N$ and let $\intint{\mf l,\mf r}$ be a stopping domain. Let $F: \{(\ell, r, f_1, \ldots,  f_k) : \ell < r \text{ and } (f_1, \ldots,f_k) \in  \R^{\intint{1,k}\times\intint{\ell, r}}\} \to \R$ be bounded. If $\bm L$ is a discrete line ensemble that has the Hall-Littlewood Gibbs property, it holds that, almost surely,
\begin{align*}
\MoveEqLeft[8]
\E\left[F(\mf l, \mf r, L_1, \ldots, L_k) \mid \Fext(k, \intint{\mf l,\mf r}, \bm L)\right]\\
&= \frac{\E[F(\mf l, \mf r, B_1, \ldots, B_k) W(\bm B,L_{k+1}) \mid \Fext(k, \intint{\mf l,\mf r}, \bm L)]}{\E[W(\bm B,L_{k+1}) \mid \Fext(k, \intint{\mf l,\mf r}, \bm L)]},
\end{align*}
where $\bm  B = (B_1, \ldots, B_k)$ is a collection of $k$ independent Bernoulli random walk bridges with $B_i(x) = L_i(x)$ for $x\in\{\mf l, \mf r\}$ and $i=1, \ldots, k$.
\end{proposition}

Similar statements have been proven for continuous line ensembles, e.g., \cite[Lemma 2.5]{corwin2014brownian}. The proof amounts to decomposing based on the exact value of $\mf l$ and $\mf r$ (which is straightforward here as everything is discrete) and using the Hall-Littlewood Gibbs property on each of the resulting events; we do not record the details here.

\subsection{The Brownian Gibbs property}\label{s.bg}

To introduce the Brownian Gibbs property, we will need notation for the $\sigma$-algebra we condition on for continuous line ensembles. We first precisely define the meaning of a continuous line ensemble.

\begin{definition}[Continuous line ensemble]\label{d.continuous line ensemble}
Fix a (possibly infinite) interval $\Lambda\subseteq \R$.  A \emph{$\Lambda$-indexed continuous line ensemble} $\bm \cL = (\cL_1, \cL_2, \ldots)$ is a random variable defined on a probability space $(\Omega, \F, \P)$, taking values in  $\mc C(\N\times\Lambda, \R)$, such that (again writing the first argument of $\bm \cL$ as a subscript, i.e., $\cL_i(\bm\cdot) := \bm \cL(i,\bm\cdot)$) $\cL_{i}(y) \geq \cL_{i+1}(y)$ for any $y\in\Lambda$ and~$i\in\N$. Here $\N\times\Lambda$ is given the product topology of the discrete topology and the Euclidean topology, and $\mc C(\N\times\Lambda, \R)$ has the topology of uniform convergence on compact sets.
\end{definition}

 Recall again the $\sigma$-algebra $\Fext$ in the case of discrete line ensembles from Definition~\ref{d.F_ext}.
In the case of a continuous line ensemble $\bm \cL:\N\times\R\to\R$, we will write, for $a,b\in\R$ and $j,k\in\N$ with $a<b$ and $j<k$, $\smash{\Fext(\llbracket j,k\rrbracket, [a,b], \bm \cL)}$ to be the $\sigma$-algebra generated by the collection of random variables
$$\bigl\{\cL_i(x): (i,x)\in \N\times \R\setminus(\intint{j,k} \times [a,b])\bigr\}.$$
(Compared to Definition~\ref{d.F_ext}, we have replaced $\intint{a+1,b-1}$ with $[a,b]$ in the second bullet point, so that $\cL_i(a)$ and $\cL_i(b)$ are still conditioned on, for each $i\in\intint{j,k}$). We will again often adopt the shorthand $\F_{\mrm{ext}}(k, [a,b], \bm L) = \F_{\mrm{ext}}(\intint{1,k}, [a,b], \bm L)$.

\begin{definition}[Brownian Gibbs property]\label{d.bg}
A line ensemble $\bm \cL = (\cL_1,\cL_2, \ldots):\N\times\R\to\R$ is said to satisfy the Brownian Gibbs property if the following holds for any $\intint{j,k}\subseteq \N$ and $[a,b]\subseteq \R$. Conditional on $\Fext(\intint{j,k}, [a,b], \bm \cL)$, the distribution of $\bm \cL|_{\intint{j,k}\times[a,b]}$ is that of $k-j+1$ independent rate two Brownian bridges $B^{\mrm{Br}}_j, \ldots, B^{\mrm{Br}}_k$, with $B^{\mrm{Br}}_i(x) = \cL_i(x)$ for each $x\in\{a,b\}$ and $i\in \intint{j,k}$, conditioned on $\cL_{j-1}(x) > B^{\mrm{Br}}_j(x) >  \ldots >B^{\mrm{Br}}_k(x) > \cL_{k+1}(x)$ for all $x\in[a,b]$ (with $\cL_{0}\equiv \infty$ by convention).
\end{definition}

\begin{figure}[h]
\includegraphics[scale=0.7]{figures/para-airy-le-figure0.pdf}\hspace{1cm}
\includegraphics[scale=0.7]{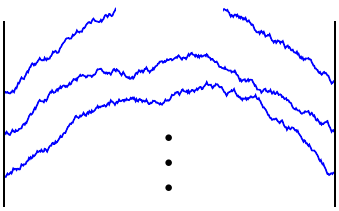}\hspace{1cm}
\includegraphics[scale=0.7]{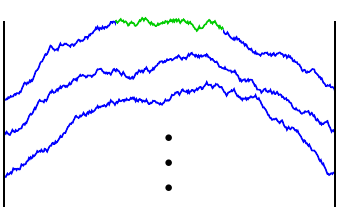}
\caption{A depiction of a special case, with $k=1$, of the Brownian Gibbs property: in the left panel is the original line ensemble $\bm\cL$, in the middle panel we condition on the data which is depicted, essentially forgetting what has been erased, and in the right panel we resample the forgotten data, which is conditionally distributed as a rate two Brownian bridge between the endpoints, conditioned to avoid the second curve $\cL_2$.}\label{f.para airy and bg}
\end{figure}

Recall now the framework for line ensemble convergence in Section~\ref{s.line ensemble hypotheses}. In particular, Theorem~\ref{t.tightness} asserts that for any sequence of discrete line ensembles $\bm L^N$ satisfying Assumptions~\ref{as.HL Gibbs} and \ref{as.one-point tightness}, the rescaled line ensembles $\bm \cL^N$ defined in \eqref{e.cL definition} forms a tight sequence, and Theorem~\ref{t.line ensemble convergence to parabolic Airy} characterizes the subsequential limits and says that, under the additional Assumption~\ref{as.GUE TW}, the original sequence converges to $\bm\cP$.

Theorem~\ref{t.line ensemble convergence to parabolic Airy} is an immediate consequence of knowing that all subsequential limit points of $\bm \cL^N$ satisfy the Brownian Gibbs property, on combining with a recent characterization result for $\bm\cP$ \cite{aggarwal2023strong}. These two statements are given precisely next. Proposition~\ref{p.limits have BG} is proved in Section~\ref{s.bg in the limit}.

\begin{proposition}\label{p.limits have BG}
Let $\bm L^N$ satisfy Assumptions~\ref{as.HL Gibbs} and \ref{as.one-point tightness}, let $\bm \cL^N$ be defined as in \eqref{e.cL definition}, and let $\bm \cL$ be any subsequential weak limit of $\bm \cL^N$ as $N\to\infty$. Then $\bm \cL$ has the Brownian Gibbs property.
\end{proposition}

\begin{proposition}[{\cite[Corollary 2.11 (2)]{aggarwal2023strong}}]\label{p.strong characterization}
Suppose $\bm \cL$ is a line ensemble with the Brownian Gibbs property such that, for any $\varepsilon>0$, there exists a constant $C = C(\varepsilon)>0$ such that
\begin{align*}
\P\left(|\cL_1(x) + x^2| > \varepsilon |x| + C\right) \leq \varepsilon
\end{align*}
for all $x\in\R$. Then there exists a parabolic Airy line ensemble $\bm\cP$ and an independent random variable $\mf c$ such that $\cL_i(x) = \cP_i(x) + \mf c$ for all $i\in\N$ and $x\in\R$.
\end{proposition}

We give the proof of Theorem~\ref{t.line ensemble convergence to parabolic Airy} given these two statements and Theorem~\ref{t.tightness} now.

\begin{proof}[Proof of Theorem~\ref{t.line ensemble convergence to parabolic Airy}]
By Theorem~\ref{t.tightness} and Proposition~\ref{p.limits have BG}, we know that $\{\bm\cL^N\}_{N=1}^\infty$ is tight and that all subsequential limits have the Brownian Gibbs property. Let $\bm\cL$ be such a subsequential limit. By the characterization theorem (Proposition~\ref{p.strong characterization}), it follows that $\bm\cL$ equals the parabolic Airy line ensemble up to an independent random constant shift. Then Assumption~\ref{as.GUE TW} guarantees that the constant shift is zero, completing the proof.
\end{proof}


In the next section we will define centered versions of the colored Hall-Littlewood line ensembles which are associated to S6V and ASEP, and verify that they satisfy the assumptions from Section~\ref{s.line ensemble hypotheses}.

\subsection{Verifying the assumptions for colored S6V and colored ASEP} \label{s.assumptions for S6V and ASEP}

\subsubsection{S6V} Recall that we work in the domain $\Z_{\geq 1}\times\llbracket -N, \infty\rrparen$ and that we evaluate the colored height function on the vertical line $x = \floor{\varepsilon^{-1}}$.
Fix $q\in[0,1)$, $z\in(0,1)$, and $\alpha\in(z,z^{-1})$ and recall $\mu(\alpha)$ from \eqref{e.mu and sigma} and $\smash{L^{\mrm{cHL}, (j)}_i}$ from \eqref{e.colored line ensemble definition} with parameters $q$, $z$, $N$, and $M=\floor{\varepsilon^{-1}}$. Define $\bm L^{\mrm{S6V},(j), \varepsilon}$ for $j\in\intint{1, N}$ by, for $y\in\intint{-\floor{\alpha\varepsilon^{-1}},N+M - \floor{\alpha\varepsilon^{-1}}}$,
\begin{align}\label{e.S6V line ensemble}
L^{\mrm{S6V},(j),\varepsilon}_i(y) := L^{\mrm{cHL}, (j)}_i(\floor{\alpha \varepsilon^{-1}} + y) + j - N - \floor{\mu(\alpha)\varepsilon^{-1}}.
\end{align}
Observe that, since  $h^{\mrm{S6V}}(j, 0; \bm\cdot, \floor{\varepsilon^{-1}}) \stackrel{d}{=} L^{\mrm{cHL}, (j)}_1(\bm\cdot)$ as functions on $\intint{1,N-1}$ from Proposition~\ref{p.colored line ensembles}, it follows that, as functions on $\intint{-\floor{\alpha\varepsilon^{-1}}+1, N-1-\floor{\alpha\varepsilon^{-1}}}$ and jointly over $j\in\intint{1,N}$,
\begin{align}\label{e.S6V LE and height function}
L^{\mrm{S6V},(j),\varepsilon}_1(\bm \cdot) \stackrel{d}{=} h^{\mrm{S6V}}(j, 0; \floor{\alpha \varepsilon^{-1}} + \bm \cdot, \floor{\varepsilon^{-1}}) + j - N - \floor{\mu(\alpha)\varepsilon^{-1}}.
\end{align}

\begin{lemma}\label{l.s6v line ensemble assumptions}
$\bm L^{\mrm{S6V},(1),\varepsilon}$ satisfies Assumptions~\ref{as.HL Gibbs}, \ref{as.one-point tightness}, and \ref{as.GUE TW} with $p = \mu'(\alpha)$ and $\lambda = \frac{1}{2}\mu''(\alpha)$ (as defined in \eqref{e.mu and sigma}) and with $\varepsilon^{-1}$ in place of $N$.
\end{lemma}

\begin{proof}
Since $\bm L^{\mrm{cHL}, (1)}$ marginally has the Hall-Littlewood Gibbs property (Proposition~\ref{p.L has HL}), and since the Gibbs property is not changed by deterministic constant shifts of the curves, it follows that $\bm L^{\mrm{S6V},(1),\varepsilon}$ has the Hall-Littlewood Gibbs property.
That Assumptions \ref{as.one-point tightness} and \ref{as.GUE TW} hold follows immediately from combining Propositions~\ref{p.colored line ensembles} and \ref{p.s6v one-point} with \eqref{e.S6V LE and height function}.
\end{proof}

\subsubsection{ASEP}
Fix $q\in[0,1)$, $\alpha\in (-1,1)$ and define $\gamma = 1-q$. Let $\theta=40$ and
$$K_{\varepsilon,\theta} = \floor{(2\gamma^{-1} \varepsilon^{-1})^{1+\theta}} - \floor{2\alpha \varepsilon^{-1}}-1.$$
Next we define line ensembles $\bm L^{\mrm{ASEP}, (j),\theta,\varepsilon}$ for $j\in\intint{1,\floor{(2\gamma^{-1}\varepsilon^{-1})^{1+\theta}}}$. Recall $\mu(\alpha) = \smash{\frac{1}{4}(1-\alpha)^2}$ from \eqref{e.mu sigma ASEP}. They are specified by setting, for $y\in\intint{1-K_{\varepsilon,\theta},N+M-K_{\varepsilon,\theta}}$, and where the parameters for the righthand side are $z=\frac{1-\varepsilon}{1-\varepsilon q}$ and $M = N = \smash{\floor{(2\gamma^{-1}\varepsilon^{-1})^{1+\theta}}}$, 
\begin{equation}\label{e.ASEP line ensemble}
\begin{split}
L^{\mrm{ASEP},(j), \theta,\varepsilon}_i(y)%
& := L^{\mrm{cHL},(j)}_i\left(K_{\varepsilon,\theta}+y\right)
- \floor{(2\gamma^{-1} \varepsilon^{-1})^{1+\theta}} + j-1 + \floor{\tfrac{1}{2}(1-\alpha)^2\varepsilon^{-1}}.
\end{split}
\end{equation}
The relation between $L^{\mrm{ASEP},(j),\theta,\varepsilon}_1$ and $\hasep$ (recall \eqref{e.approx asep height function definition} and the fact that $h^{\mrm{S6V}}(j,0;\bm\cdot,\floor{(2\gamma^{-1}\varepsilon^{-1})^{1+\theta}}) \stackrel{d}{=} L^{\smash{\mrm{cHL},(j)}}_1(\bm\cdot)$) is
\begin{align}\label{e.line ensemble asep height fn}
L^{\mrm{ASEP},(j),\theta,\varepsilon}_1(\bm\cdot) \stackrel{d}{=} -\hasep(-j, 0; \floor{2\alpha \varepsilon^{-1}}-\bm\cdot, \floor{2\gamma^{-1}\varepsilon^{-1}})  + \floor{\tfrac{1}{2}(1-\alpha)^2\varepsilon^{-1}}.
\end{align}

\begin{lemma}\label{l.asep line ensemble assumptions}
$\bm L^{\mrm{ASEP},(1),\theta,\varepsilon}$ satisfies Assumptions~\ref{as.HL Gibbs}, \ref{as.one-point tightness}, and \ref{as.GUE TW} with $p = -\frac{1}{2}(1-\alpha)$ and $\lambda = \frac{1}{8}$ and with $\varepsilon^{-1}$ in place of $N$.
\end{lemma}

\begin{proof}
As in the case of $\bm L^{\mrm{S6V},(1),\varepsilon}$, it is immediate from Proposition~\ref{p.L has HL} that $\bm L^{\mrm{ASEP},(1),\theta,\varepsilon}$ has the Hall-Littlewood Gibbs property. It follows from Corollary~\ref{c.approx asep tightness} that Assumptions~\ref{as.one-point tightness} and \ref{as.GUE TW}~hold (uniformly over the underlying spectral parameter $z$ varying with $\varepsilon$ as well).
\end{proof}

Next, we set up some of the ingredients for the proof of Theorem~\ref{t.tightness}, which will also be useful in Section~\ref{s.airy sheet} on the tightness and limit of $\S^{\mrm{S6V}, \varepsilon}$ and $\S^{\mrm{ASEP},\varepsilon}$. In particular, we are working in the setup of Theorem~\ref{t.tightness}, i.e., we have a sequence of discrete line ensembles $\bm L^N$ indexed by $N$ and defined on $\N\times\Z$ that satisfy Assumptions~\ref{as.HL Gibbs} and \ref{as.one-point tightness}.

\subsection{Ingredients for proof of tightness}\label{s.tightness ingredients}

Proving tightness of the path measures requires us to establish one-point tightness, as well as control on the modulus of continuity,  for all curves. We state the estimates we need and then use them to give the proof of Theorem~\ref{t.tightness}.

\subsubsection{One-point tightness of all curves} Note that from Assumption~\ref{as.one-point tightness} we have one-point tightness for $L^N_1$ and, since line ensembles which satisfy the Hall-Littlewood Gibbs property must by definition be ordered, we therefore have control on the upper tail of $L^N_k$ for every $k$. The lower tail for the $k$\textsuperscript{th} curve for $k\geq 2$ is more delicate, however, and we record a statement for it now.

\begin{theorem}\label{t.infimum lower tail}
Fix $k\in \N$, $T > 0$, and $\varepsilon>0$. Then there exist $N_0 = N_0(k,T,\varepsilon)$ and $R = R(k,T,\varepsilon)>0$ such that, for $N\geq N_0$,
\begin{align*}
\P\left(\inf_{sN^{-2/3}\in[-T,T]} \left(L^N_k(s) - ps + \lambda s^2N^{-1}\right) \leq -RN^{1/3}\right) < \varepsilon.
\end{align*}

\end{theorem}

We will prove Theorem \ref{t.infimum lower tail} in Sections~\ref{s.lower tail} and \ref{s.uniform separation}. As mentioned in Sections~\ref{s.intro.line ensemble tightness} and \ref{s.lack of monotonicity discussion}, the proof inductively interweaves uniform lower tail control of the $k$\th curve over an interval with control over uniform separation between the $(k-1)$\st and $k$\th curves, both on the fluctuation scale $N^{1/3}$. Obtaining control on these two quantities forms the bulk of the technical work of the proof of Theorem~\ref{t.tightness}.

\subsubsection{Modulus of continuity}

To prove tightness of path measures, we also need control on the modulus of continuity of $L^N_k$. As in earlier works establishing tightness for line ensembles with Gibbs properties, this comes down to controlling the partition function of the Gibbs property, in order to transfer modulus of continuity estimates from the base path measure (in this case, Bernoulli random walk bridges) to the line ensemble. For an interval $[-TN^{2/3}, TN^{2/3}]$, recall from Definition~\ref{d.HL Gibbs} that the partition function of the top $k$ curves~is
\begin{align}\label{e.Z original definition}
Z^{k, [-T,T], \bm L^N(-TN^{2/3}), \bm L^N(TN^{2/3}), L^N_{k+1}}_N := \EF[W(B_1, \ldots, B_k, L^N_{k+1})],
\end{align}
where $\F=\Fext(k, \intint{-TN^{2/3}, TN^{2/3}}, \bm L^N)$ and $(B_1, \ldots, B_k)$ is a collection of $k$ independent Bernoulli random walk bridges with $B_i$ from $(-TN^{2/3}, L^N_i(-TN^{2/3}))$ to $(TN^{2/3}, L^N_i(TN^{2/3}))$.

\begin{remark}
Note here that we have written certain quantities, such as $\pm TN^{2/3}$, without floors, though they are meant to be integers to be completely precise. Here as in the rest of the paper we will often drop the requisite floor symbols to ease the notation, since this will not affect our arguments.
\end{remark}

The control we need on the partition function is the following. It is proved in Section~\ref{s.partition function and non-intersection}.

\begin{proposition}\label{p.partition function strong lower bound}
Let $T\geq 2$ and $k\in\N$. Then for any $\varepsilon>0$ there exist $\delta = \delta(k,T,\varepsilon)>0$ and $N_0 = N_0(k, T,\varepsilon)$ such that, for $N\geq N_0$,
\begin{align*}
\P\left(Z^{k, [-T,T], \bm L^N(-TN^{2/3}), \bm L^N(TN^{2/3}), L^N_{k+1}}_N > \delta\right) \geq 1-\varepsilon.
\end{align*}

\end{proposition}

The control we need on the base path measure of Bernoulli random walk bridges is as follows. As its proof is quite standard, we defer it to Appendix~\ref{s.random gibbs}.

\begin{lemma}\label{l.random walk bridge fluctuation}
There exist $C$ and $c>0$ such that, for $n\in\N$, $p\in(-1,0)$, $B$ a Bernoulli random walk bridge from $(0,0)$ to $(n, pn)$, all $0\leq y<z \leq n$, and $M>0$,
\begin{align*}
\P\left(\sup_{x\in[y,z]}|B(x) - B(y) - p(x-y)| > M|z-y|^{1/2}\right) \leq C\exp(-cM^2).
\end{align*}

\end{lemma}

At several locations, we will also need to use normal or Brownian bridge approximations for Bernoulli random walk bridges. We will require the Brownian bridges to have an explicit variance coming from the increment of the Bernoulli random walk bridge being approximated. For this we make use of the following Koml\'os-Major-Tusn\'ady (KMT) coupling result for Bernoulli random walk bridges, which is obtained by combining \cite[Theorem 4.5]{corwin2018transversal} with the exponential Markov inequality.

\begin{lemma}[KMT coupling, {\cite[Theorem 4.5]{corwin2018transversal}}]\label{l.KMT}
Let $p \in (-1, 0)$ and $\{z_N\}_{N\in\N}\subseteq \R$. For each $N$, let $B^N$ be a Bernoulli random walk bridge from $(0,0)$ to $(N,z_N)$. There exist positive constants  $C, c < \infty$ (depending on $p$) such that
there is a coupling of $B^N$ and a Brownian bridge $\Bbr$ with variance $p(1 - p)$ from $(0,0)$ to $(N, z_N)$ with
\begin{align*}
\P\left(\sup_{0\leq t\leq N} |\Bbr(t) - B^N(t)| > R\right) \leq C\exp\Bigl(-cR + C(\log N)^2 + N^{-1}|z_N-pN|^2\Bigr).
\end{align*}
\end{lemma}

Note that the variance of the Brownian bridge being $p(1-p)$ can be understood by the fact that the Bernoulli random walk bridge, if it has the endpoint $(N,z_N)\approx (N, pN)$, will behave like a Bernoulli random walk with parameter $p$, whose increments indeed have variance $p(1-p)$.

Assuming Proposition~\ref{p.partition function strong lower bound}, we can give the following corollary on the modulus of continuity of $L^N_k$. Recall from Section~\ref{s.line ensemble hypotheses} that we view $L^N_k$ as a $\R$-valued function via linear interpolation.

\begin{corollary}\label{c.mod con for L}
Let $T\geq 2$ and $k\in\N$. For any $\varepsilon>0$ and $\rho>0$, there exist $\delta = \delta(k,T,\varepsilon,\rho)>0$ and $N_0 = N_0(k,T,\varepsilon,\rho)$ such that, for $N\geq N_0$,
\begin{align*}
\P\left(\sup_{|x|, |y| \leq TN^{2/3}, |x-y|\leq \delta N^{2/3}} | L^N_k(x) - L^N_k(y) - p(x-y)| \geq \rho N^{1/3}\right) \leq \varepsilon.
\end{align*}

\end{corollary}

\begin{remark}\label{r.event name convention}
Before giving the proof of Corollary~\ref{c.mod con for L}, we point out a notational practice that will be used many times throughout the rest of the paper: we will often introduce events (e.g., $\msf A(X)$) whose definition involves an abstract $\R$-valued process $X$ and then in the course of the proof write different processes (e.g., $L^N_k$) in place of $X$ to mean the event whose definition is that of $\msf A(X)$ with the process in question in place of $X$. In spite of a random process being part of the notation for the event, the event itself is fixed, i.e., does not depend on the instantiation of the process. This convention will be convenient because when applying the Hall-Littlewood Gibbs property we will often be working with events whose only difference is that a collection of Bernoulli random walk bridges is in place of a collection of curves of the line ensemble $\bm L^N$.
\end{remark}

\begin{proof}[Proof of Corollary~\ref{c.mod con for L}]
For a random process $X:[-TN^{2/3},TN^{2/3}]\to\R$, define
\begin{equation*}
\msf{ModCon}_N(\delta,\rho, X) := \left\{\sup_{|x|, |y| \leq TN^{2/3}, |x-y|\leq \delta N^{2/3}} | X(x) - X(y) - p(x-y)| \geq \rho N^{1/3}\right\}.
\end{equation*}
So $\msf{ModCon}_N(\delta,\rho, L^N_k)$ is the event whose probability we are trying to upper bound. Let $\F = \Fext(k, [-TN^{2/3},TN^{2/3}], \bm L^N)$ and
\begin{align*}
Z_N := Z^{k, [-T,T], \bm L^N(-T N^{2/3}), \bm L^N(T N^{2/3}), L^N_{k+1}}_N.
 \end{align*}
For $j\in\N$, let $\overline L^N_j(x) = L^N_j(x) - px$ and define the event $\msf{Fav}(\delta', M)$ by
\begin{align*}
\msf{Fav}(\delta', M) := \left\{Z_N > \delta'\right\} \cap \left\{|\overline L^N_1(\pm T N^{2/3})|, |\overline L^N_k(\pm T N^{2/3})| \leq MN^{1/3}\right\};
\end{align*}
we pick $\delta'$ and $M$ such that $\P(\msf{Fav}(\delta', M)) \geq 1-\varepsilon$, which is possible by Proposition~\ref{p.partition function strong lower bound}, one-point tightness of $L^N_1$ from Assumption~\ref{as.one-point tightness}, and Theorem~\ref{t.infimum lower tail}. Note that this event is $\F$-measurable.

Now by using this bound on $\P(\msf{Fav}(\delta', M))$, the Hall-Littlewood Gibbs property, and the fact that $W\leq 1$ ($W$ being from Definition~\ref{d.weight factor}), the probability we are bounding equals
\begin{align*}
\MoveEqLeft[6]
\E\left[\PF\left(\msf{ModCon}_N(\delta,\rho, L^N_k)\right)\right]\\
&\leq \E\left[\PF\left(\msf{ModCon}_N(\delta,\rho, L^N_k)\right)\one_{\msf{Fav}(\delta', M)}\right] +  \varepsilon\\
&= \E\left[Z_N^{-1}\EF\left[\one_{\msf{ModCon}_N(\delta,\rho, B_k)}W\left(\bm B, L^N_{k+1}\right)\right]\one_{\msf{Fav}(\delta', M)}\right] +  \varepsilon\\
&\leq \E\left[Z_N^{-1}\PF\left(\msf{ModCon}_N(\delta,\rho, B_k)\right)\one_{\msf{Fav}(\delta', M)}\right] + \varepsilon,
\end{align*}
where  $\bm B=(B_1, \ldots, B_k)$ is a collection of $k$ independent Bernoulli random walk bridges with $B_i$'s endpoints given by $L^N_i(-TN^{2/3})$ and $L^N_i(T N^{2/3})$. Since we have a lower bound of $\delta'$ on $Z_N$ on the event $\msf{Fav}(\delta', M)$, we obtain that the righthand side of the previous display is upper bounded by
\begin{align*}
(\delta')^{-1}\cdot \E\left[\PF\left(\msf{ModCon}_N(\delta,\rho, B_k)\right)\one_{\msf{Fav}(\delta', M)}\right] + \varepsilon.
\end{align*}
By the modulus of continuity estimates for Bernoulli random walks from Lemma~\ref{l.random walk bridge fluctuation}, the probability inside the expectation in the previous line can be made smaller than $\delta'\varepsilon$ for an appropriately small choice of $\delta$. Relabeling $\varepsilon$ completes the proof.
\end{proof}

Assuming Theorem~\ref{t.infimum lower tail} on the lower tail and Proposition~\ref{p.partition function strong lower bound} (therefore Corollary~\ref{c.mod con for L}), we may quickly give the proof of tightness.

\begin{proof}[Proof of Theorem~\ref{t.tightness}]
Recall that $\bm \cL^N$ is defined by \eqref{e.cL definition}.
By an easy extension of \cite[Theorem 7.3]{billingsley2013convergence} (see also \cite[Lemma 2.4]{dimitrov2021tightness}), to show tightness of $\bm\cL^N$ it suffices to show, for every $k\in\N$:
\begin{enumerate}
  \item $\lim_{R\to\infty}\limsup_{N\to\infty} \P\left(|\cL^N_k(0)| \geq R\right) = 0$ and

  \item For all $\varepsilon>0$ and $M\in\N$,
  \begin{align*}
  \lim_{\delta\to 0}\limsup_{N\to\infty} \P\left(\sup_{ |x|, |y|\leq M, |x-y|\leq \delta} |\cL^N_k(x) - \cL^N_k(y)| \geq \varepsilon\right) = 0.
  \end{align*}
\end{enumerate}

The first item is immediate from Theorem~\ref{t.infimum lower tail} and Assumption~\ref{as.one-point tightness}, using the ordering of the curves, while (2) follows from Corollary~\ref{c.mod con for L}.
\end{proof}

\section{Airy sheet and tightness of S6V sheet}
\label{s.airy sheet}

In this section we address tightness and convergence to the Airy sheet of the colored height functions of our models. As in Section~\ref{s.properties of S6V and ASEP}, we restrict to packed boundary or initial conditions.

\subsection{Convergence to the Airy sheet}\label{s.convergence to Airy sheet}
We will give a single proof of Theorems~\ref{t.asep airy sheet} and \ref{t.s6v airy sheet} using a unified notation for the objects arising from both S6V and ASEP, as follows. We will use $\bm L^{(j), N}$ for line ensembles corresponding to different colors in the two models. The S6V case will correspond to identifying $\bm L^{(j), N}$ with $\bm L^{\mrm{S6V}, (j), \varepsilon}$, and the ASEP case to identifying the same with $\bm L^{\mrm{ASEP},(j), \theta, \varepsilon}$, as defined in \eqref{e.S6V line ensemble} and \eqref{e.ASEP line ensemble}, and in both cases under the identification of $N$ with $\varepsilon^{-1}$. Next recall the scaling operator $T_{p,\lambda,N}$ from \eqref{e.scaling oeprator}. For a sequence (indexed by $N$) of collections of line ensembles $\{\bm L^{(j),N}\}_j$ such that $\smash{\bm L^{(1),N}}$ satisfies Assumptions~\ref{as.HL Gibbs}, \ref{as.one-point tightness} (with parameters $p\in(-1,0)$ and $\lambda>0$), and \ref{as.GUE TW}, we define the rescaled line ensemble $\bm \cL^{(j), N}$ by
\begin{align*}
\bm \cL^{(j),N} = T_{p,\lambda,N}(\bm L^{(j),N});
\end{align*}
we will use this definition and assume that $\bm L^{(1)}$ satisfies Assumptions~\ref{as.HL Gibbs}, \ref{as.one-point tightness}, and \ref{as.GUE TW} throughout this section without further comment. Note that Theorem~\ref{t.line ensemble convergence to parabolic Airy} implies $\smash{\bm \cL^{(1), N}\stackrel{d}{\to} \bm\cP}$, the parabolic Airy line ensemble (Definition~\ref{d.parabolic Airy line ensemble}).

Recall that the top lines of the line ensembles $\bm L^{\mrm{S6V}, (j), \varepsilon}$ and $\bm L^{\mrm{ASEP},(j), \theta, \varepsilon}$ correspond to the height functions of arrows or particles of color above a threshold determined by $j$ (the threshold being $j$ for S6V and $-j$ for ASEP), and that varying $j$ under the colored coupling of the underlying models corresponds to varying the location of the start of the step initial condition.
To encode the line ensembles corresponding to the step initial condition started, in rescaled coordinates, at $x$ in a notationally unified way, we define the following. Recalling $\beta$ from \eqref{e.scaling coefficients relation}, for $x \geq \beta^{-1} N^{-2/3}$ and $y\in\R$, let
\begin{align}\label{e.rescaled line ensemble with initial point}
\cL^N_{k}(x;y) := \cL^{(\floor{\beta xN^{2/3}}), N}_{k}(y).
\end{align}

Next we assume we have a continuous process $\S^N:\R^2\to\R$ such that
\begin{align}\label{e.rescaled h general definition}
\S^N(x;y) \stackrel{d}{=} \cL^N_1(x;y)
\end{align}
for $x\geq \beta^{-1}N^{-2/3}$ and $y\in\R$. 
The S6V and ASEP cases correspond to taking $(x,y)\mapsto \S^N(x;y)$ to equal $\S^{\mrm{S6V},\varepsilon}$ and $(x,y)\mapsto\S^{\mrm{ASEP},\theta,\varepsilon}(-x; -y)$ respectively (as defined in \eqref{e.rescaled s6v definition} and \eqref{e.rescaled asep definition}), with $N=\floor{\varepsilon^{-1}}$. That \eqref{e.rescaled h general definition} is consistent with the definitions of $\S^{\mrm{S6V}, \varepsilon}$ and $\bm L^{\mrm{S6V},(j),\varepsilon}$ (and the analogs for ASEP) follows from the definitions \eqref{e.S6V line ensemble} for S6V and \eqref{e.ASEP line ensemble} for ASEP, along with Lemma~\ref{l.can look at positive colors only} and Proposition~\ref{p.colored line ensembles} (allowing us to restrict to positive colors only as in $\cL^{N}_1(x;y)$ and relating the colored Hall-Littlewood line ensemble and the S6V height functions for such non-negative colors, respectively).

We also assume the following stationarity property for all $z$ fixed, where $g_N\leq N$ is a fixed sequence with $g_N\to\infty$ as $N\to\infty$: as processes in $x,y$ over the domain defined by $x,y,x+z,y+z\in[-g_N, g_N]$,
\begin{align}\label{e.rescaled two parameter stationarity}
\S^N(x;y) \stackrel{d}{=} \S^{N}(x+z;y+z);
\end{align}
that this property holds for the S6V and ASEP cases follows for S6V from \eqref{e.two parameter stationarity} (and the presence of the $+j$ term in the definition \eqref{e.S6V line ensemble}) along with \eqref{e.S6V LE and height function}, and for ASEP from \eqref{e.two parameter stationarity for ASEP} and \eqref{e.line ensemble asep height fn}.

We will soon need to invoke Theorem~\ref{t.approxmate LPP problem representation general} on the approximate representation of $\bm L^{\mrm{cHL}(j)}$ in terms of a last passage percolation problem in $\bm L^{\mrm{cHL}, (1)}$. To avoid a notational clash with the symbol $N$ indexing the sequence of line ensembles $\bm L^N$, in this section we will use $\oldN$ for the parameter implicit in $\bm L^{\mrm{cHL},(j)}$ from Definition~\ref{d.colored line ensemble} (and thus also the parameter defining the domain of the colored S6V model from Section~\ref{s.intro.cS6V}). We will also make the assumption that there exists $D>0$ such that
\begin{align}\label{e.old N new N inequality}
\oldN \leq N^D,
\end{align}
which we will later verify holds in our cases of interest.

The following is a convergence statement for $\S^N$ (i.e., S6V and ASEP simultaneously); it will be proved in Section~\ref{s.airy sheet convergence setup}.

\begin{theorem}\label{t.general Airy sheet convergence}
Assume that $\bm L^{(1),N}$ satisfies Assumptions~\ref{as.HL Gibbs} and \ref{as.one-point tightness}, and is such that $\bm \cL^{(1),N}\stackrel{d}\to \bm\cP$ as $N\to\infty$, in the topology of uniform convergence on compact sets of $\mc C(\N\times\R,\R)$, and that \eqref{e.rescaled two parameter stationarity} and \eqref{e.old N new N inequality} hold.
 Further assume \eqref{e.approximate LPP problem} holds with $\bm L^{(1),N},\bm L^{(j),N}$ in place of $\bm L^{\mrm{cHL},(1)}, \bm L^{\mrm{cHL},(j)}$, respectively. Then $\S^N \stackrel{\smash{d}}{\to} \S$ in $\mc C(\R^2,\R)$ in the topology of uniform convergence on compact sets as $N\to\infty$.
\end{theorem}

\noindent We now prove Theorems~\ref{t.asep airy sheet} and \ref{t.s6v airy sheet}.

\begin{proof}[Proof of Theorem~\ref{t.asep airy sheet}]
We take $\S^N(x;y) = \S^{\mrm{ASEP},\theta, \varepsilon}(-x;-y)$ as processes in $(x,y)$ with $N=\varepsilon^{-1}$, and note from \eqref{e.ASEP N value} that $\oldN$ in this case equals $\floor{(2\gamma^{-1})^{1+\theta}\varepsilon^{-(1+\theta)}} = \floor{(2\gamma^{-1}\varepsilon^{-1})^{41}}$; so \eqref{e.old N new N inequality} holds. Analogously we take $\bm L^{(j), N} = \bm L^{\mrm{ASEP}, (j),\theta, \varepsilon}$. Now, $\bm L^{(1),N}$ satisfies Assumptions~\ref{as.HL Gibbs}, \ref{as.one-point tightness}, and \ref{as.GUE TW} by Lemma~\ref{l.asep line ensemble assumptions}, so, by Theorem~\ref{t.line ensemble convergence to parabolic Airy}, $\bm\cL^{(1),N}\stackrel{\smash{d}}{\to} \bm \cP$. Next, \eqref{e.rescaled two parameter stationarity} holds due to \eqref{e.two parameter stationarity for ASEP}, and \eqref{e.approximate LPP problem} holds due to Theorem~\ref{t.approxmate LPP problem representation general} along with \eqref{e.ASEP line ensemble}.
Thus Theorem~\ref{t.general Airy sheet convergence} yields that $\S^{\mrm{ASEP},\theta,\varepsilon}(x;y) \to \S(-x;-y)$ in distribution as continuous processes on $\R^2$ as $\varepsilon\to 0$. By Lemma~\ref{l.asep s6v comparison}, we then obtain that $\S^{\mrm{ASEP},\varepsilon}(x;y)\to\S(-x;-y)$ in the same topology as $\varepsilon\to0$ as well. That $\smash{\S(-x;-y) \stackrel{d}{=} \S(x;y)}$ by \cite[Proposition~1.23]{dauvergne2021scaling} as processes completes the proof of Theorem~\ref{t.asep airy sheet}.
\end{proof}

\begin{proof}[Proof of Theorem~\ref{t.s6v airy sheet}]
Recall that $\S^{\mrm{S6V}, \varepsilon}$ is defined in terms of a colored S6V model defined on $\Z_{\geq 1}\times\llbracket-\oldN,\infty\rrparen$. Let $\tilde\S^{\mrm{S6V}, \varepsilon}$ be the same with $\oldN = \varepsilon^{-2}$ if the original $\oldN > \varepsilon^{-2}$, and let $\tilde\S^{\mrm{S6V}, \varepsilon} = \S^{\mrm{S6V}, \varepsilon}$ otherwise; also recall that Theorem~\ref{t.s6v airy sheet} assumes that $\oldN \geq 2\alpha\varepsilon^{-1}$. let $\tilde{\bm L}^{\mrm{S6V}, (j), \varepsilon}$ be the associated colored line ensemble from \eqref{e.S6V line ensemble}. We claim that there is a coupling of $\S^{\mrm{S6V},\varepsilon}$ and $\tilde\S^{\mrm{S6V},\varepsilon}$ such that the two agree on $[-\varepsilon^{-1/6},\varepsilon^{-1/6}]^2$ with probability $1-o(1)$ as $\varepsilon\to 0$. Assuming this claim, it is sufficient to prove that $\tilde\S^{\mrm{S6V},\varepsilon}\stackrel{\smash{d}}{\to} \S$ as $\varepsilon\to 0$, and this will follow from Theorem~\ref{t.general Airy sheet convergence} with $N=\varepsilon^{-1}$, $\S^N = \tilde\S^{\mrm{S6V}, \varepsilon}$, and $\bm L^{(j), N} = \tilde{\bm L}^{\mrm{S6V}, (j), \varepsilon}$ once its assumptions are verified. Now, $\bm L^{(1), N}$ satisfies Assumption~\ref{as.HL Gibbs}, \ref{as.one-point tightness}, and \ref{as.GUE TW} by Lemma~\ref{l.s6v line ensemble assumptions}, so, by Theorem~\ref{t.line ensemble convergence to parabolic Airy}, $\bm \cL^{(1),N}\stackrel{\smash{d}}{\to} \bm \cP$. The remaining assumptions of Theorem~\ref{t.general Airy sheet convergence} are verified since the conditions \eqref{e.rescaled two parameter stationarity}, \eqref{e.old N new N inequality}, and \eqref{e.approximate LPP problem} are satisfied by \eqref{e.two parameter stationarity}, by \eqref{e.S6V LE and height function}, and by combining Theorem~\ref{t.approxmate LPP problem representation general} with \eqref{e.S6V line ensemble}, respectively.

The claim we made on the coupling of $\S^{\mrm{S6V},\varepsilon}$ and $\tilde\S^{\mrm{S6V},\varepsilon}$ follows immediately from \cite[Lemma B.3]{aggarwalborodin}, which is a finite speed of propagation estimate: roughly speaking, it says that two colored S6V models which have boundary conditions agreeing on a large interval have a coupling such that the arrow configurations agree in a large region. In our case, the boundary conditions of the two systems agree on $\intint{-\floor{\varepsilon^{-2}}, \floor{\varepsilon^{-2}}}$ and we require the arrow configurations to agree on $\{\floor{\varepsilon^{-1}}\}\times\intint{\floor{\alpha\varepsilon^{-1}-2\varepsilon^{-5/6}}, \floor{\alpha\varepsilon^{-1}+2\varepsilon^{-5/6}}}$. \cite[Lemma B.3]{aggarwalborodin} yields that there is a coupling such that this holds with probability at least $1-\exp(-c\varepsilon^{-2})$ for some $c>0$ (take $U,V,T,K$ there to respectively be $-\floor{\frac{1}{2}\varepsilon^{-2}}$, $\floor{\frac{1}{2}\varepsilon^{-2}}$, $\floor{R\varepsilon^{-1}}$, and $\floor{\delta R^{-1}\varepsilon^{-1}}$ for some large constant $R$ and small constant $\delta>0$). This completes the proof of Theorem~\ref{t.s6v airy sheet}.
\end{proof}

\noindent In the next section we begin setting up some statements needed for the proof of Theorem~\ref{t.general Airy sheet convergence} before giving the proof at the end of the section.

\subsection{Setup for Airy sheet convergence} \label{s.airy sheet convergence setup}
For a sequence of random variables $\{X_{N,k}\}_{N,k=1}^{\infty}$ defined on a common probability space and a sequence of deterministic numbers $\{r_{N,k}\}_{N,k=1}^\infty$, we adopt the notation
\begin{equation}\label{e.Borel-Cantelli notation}
X_{N,k} = \mf o(r_{N,k}) \quad \text{if} \quad \sum_{k=1}^{\infty} \limsup_{N\to\infty} \P(|X_{N,k}| \geq \varepsilon r_{N,k}) < \infty
\end{equation}
for every $\varepsilon>0$, i.e., $\lim_{k\to\infty}\lim_{N\to\infty}X_{N,k}/r_{N,k} = 0$, where the inner limit should be understood as any weak limit of $\{X_{N,k}/r_{N,k}\}_{N}$ and the outer limit as an almost sure one, by the Borel-Cantelli lemma. The notation also makes sense if $X_{N,k} = X_N$ and $r_{N,k} = r_N$ for all $k$ for some sequences $X_N$, $r_N$ such that $X_{N}/r_N\to 0$ almost surely as $N\to\infty$. Also, $X_{N,k}\geq \mf o(r_{N,k})$ means that $X_{N,k}\one_{X_{N,k}\leq 0} = \mf o(r_{N,k})$, and $X_{N,k}\leq \mf o(r_{N,k})$ means $X_{N,k}\one_{X_{N,k}\geq 0} = \mf o(r_{N,k})$.

Next recall the notation \eqref{e.LPP definition} for last passage percolation (LPP) values.
In the notation \eqref{e.rescaled line ensemble with initial point} and for $\delta>0$, consider \eqref{e.approximate LPP problem} of Theorem~\ref{t.approxmate LPP problem representation general}, as assumed in Theorem~\ref{t.general Airy sheet convergence}, with $k$ set to $N^{1/6}-1$ and $m=\sigma \delta N^{1/3}$ with $\sigma$ as in \eqref{e.scaling coefficients relation}. A union bound over $x$ implies that, for $\delta > K(q)N^{-1/3}\log \oldN$ (which, under \eqref{e.old N new N inequality}, holds if $\delta > DK(q)N^{-1/3}\log N$),
\begin{align}\label{e.rescaled LPP}
\MoveEqLeft[25]
\!\P\!\left(\sup_{\substack{x\in (0,g_N]\\|y|\in (0,g_N]}}\! \left|\S^{N}(x; y) - \max_{w\leq y}\! \left(\cL^N_{N^{1/6}}(x;w) + \bm \cL^{(1),N}[(w,N^{1/6}-1)\to(y,1)]\right)\right| \geq \delta N^{1/6}\! \right)\nonumber\\
& \leq C\beta N^{\frac{11}{6}} q^{c\sigma^2\delta^2 N^{2/3}}.
\end{align}
Observe that the deviation written above is $\delta N^{1/6}$ rather than $km = \sigma \delta N^{1/3}(N^{1/6}-1)$ as in \eqref{e.approximate LPP problem}; this is because the scaling to go from the prelimiting quantities in \eqref{e.approximate LPP problem} to their scaled counterparts in \eqref{e.rescaled LPP} is exactly multiplying the former by $\sigma^{-1}N^{-1/3}$ (see \eqref{e.scaling oeprator}), apart from the spatial scaling. Similarly, the spatial scaling in \eqref{e.scaling oeprator} implies that the union bound over $x\in(0,g_N]$ corresponds to $g_N\cdot \beta N^{2/3} \leq \beta N^{5/3}$ many terms, since $g_N\leq N$. The exponent of $\frac{11}{6}$ for $N$ in the prefactor of the exponential in \eqref{e.rescaled LPP} equals $\frac{5}{3}+\frac{1}{6}$, the $\frac{1}{6}$ coming from the choice of $k$ (recall that the probability bound for a single $x$ provided by \eqref{e.approximate LPP problem} is $kq^{cm^2}$).

Now set $\delta=N^{-1/4}$ and note that $N^{-1/11} \gg \delta N^{1/6} = N^{-1/12}$.
For fixed $q\in[0,1)$ and $k\in\intint{1,N^{1/6}-1}$, \eqref{e.rescaled LPP} implies that for all $x, |y| \in (0,g_N]$  and  $N\geq N_0$, with $\mf o(N^{-1/11})$ a random quantity satisfying \eqref{e.Borel-Cantelli notation} whose distribution does not depend on $x$ or $y$,
\begin{align}
\MoveEqLeft[0]
\S^{N}(x; y)\nonumber\\
&=\max_{w\leq y} \left(\cL^N_{N^{1/6}}(x;w) + \bm \cL^{(1),N}[(w,N^{1/6}-1)\to(y,1)]\right) + \mf o( N^{-1/11}),\nonumber\\
&= \max_{w\leq z\leq y} \left(\cL^N_{N^{1/6}}(x;w) + \bm \cL^{(1),N}[(w,N^{1/6}-1)\to(z,k)] + \bm \cL^{(1),N}[(z,k)\to(y,1)]\right)\nonumber\\
&\qquad\quad + \mf o( N^{-1/11}),\label{e.scaled approximate LPP representation}
\end{align}
where in the second line we used the fact that the LPP problem is a variational problem to split it up at line $k$ (see for example \cite[Lemma~3.2]{dauvergne2018directed}).

Let $Z^N_k(x,y)$ be the maximizer (of the maximum over $z$) in \eqref{e.scaled approximate LPP representation} (the largest one in the event that there are multiple). The following records control on the location of $Z^N_k$ which will be necessary for the proofs of Theorem~\ref{t.asep airy sheet} and \ref{t.s6v airy sheet}. Because its proof is similar to \cite[Lemma 7.1]{dauvergne2018directed},  we defer it to Appendix~\ref{s.G_k asymptotics}.

\begin{proposition}\label{p.maximizer location}
Under the hypotheses of Theorem~\ref{t.general Airy sheet convergence}, it holds that, for any $x>0$ and $y\in\R$,
\begin{align*}
Z^N_k(x,y) = -\sqrt{\frac{k}{2x}} + \mf o(k^{1/2}).
\end{align*}
Moreover, $\{Z^N_k(x,y)\}_{N\in\N}$ is tight for each fixed $k\in\N$, $x>0$, and $y\in\R$.
\end{proposition}

\noindent To prove Theorem~\ref{t.general Airy sheet convergence}, we also need to know that $\S^N$ forms a tight sequence in $\mc C(\R^2, \R)$, recorded in the following proposition. It will be proved in Section~\ref{s.modulus of continuity}, where we will establish the necessary modulus of continuity estimates.

\begin{proposition}\label{p.S_N tightness}
Under the hypotheses of Theorem~\ref{t.general Airy sheet convergence}, $\{\S^N\}_{N=1}^\infty$ is tight in $\mc C(\R^2, \R)$ under the topology of uniform convergence on compact sets.
\end{proposition}

The following monotonicity statement on differences of LPP values will be useful in the proof of Theorem~\ref{t.general Airy sheet convergence} and has been proved many times in the literature; see for example \cite[Proposition 3.8]{dauvergne2018directed}, \cite[Lemma 2.10]{ganguly2023local}, or \cite[Lemma B.2]{balazs2020non}.

\begin{lemma}[Crossing lemma, \cite{dauvergne2018directed,ganguly2023local,balazs2020non}]\label{l.crossing lemma}
Let $\bm f=(f_1, \ldots, f_k)$ be a sequence of continuous functions with $f_i:\R\to\R$. Then, for any $y_1 < y_2$
\begin{align*}
z\mapsto f[(z,k)\to (y_1,1)] - f[(z,k)\to (y_2,1)]
\end{align*}
is non-increasing.
\end{lemma}

Now we may give the proof of Theorem~\ref{t.general Airy sheet convergence}; given all the machinery developed by now, its proof is quite similar to that of \cite[Theorem 4.4]{dauvergne2021scaling}.

\begin{proof}[Proof of Theorem~\ref{t.general Airy sheet convergence}]
We know from Proposition~\ref{p.S_N tightness} that $\S^N$ is tight as a law on $\mc C(\R^2, \R)$ with the topology of uniform convergence on compact sets. So we only need to show that all subsequential limits equal~$\S$, equivalently, that they satisfy (i) and (ii) in Definition~\ref{d.airy sheet}.

That (i) holds follows from \eqref{e.rescaled two parameter stationarity}. We will next show that (ii) holds for any subsequential limit of $\S^N$. We start by setting up some of the objects and relations we will need to take weak limits of, as these will be necessary to specify the relevant further subsequences along which we will consider convergence.
Recall from \eqref{e.scaled approximate LPP representation} that, for all $x\in [0, g_N]$, $y_i \in [-g_N,g_N]$ simultaneously, with $i\in\{1,2\}$,
\begin{align}
\MoveEqLeft[0]
\S^N(x; y_i)\nonumber\\
&= \sup_{w\leq z\leq y_i} \left(\cL^N_{N^{1/6}}(x;w) + \bm \cL^{(1),N}[(w,N^{1/6}-1) \to (z,k)] + \cL^{(1),N}[(z,k) \to (y_i,1)]\right)\nonumber\\
&\qquad + \mf o(N^{-1/11}) \label{e.approximate variational problem in convergence proof}
\end{align}
holds. We assume in what follows that $y_1 < y_2$, without loss of generality.

From Proposition~\ref{p.maximizer location}, for every $k\in\N$, $x>0$ and $y_1,y_2\in\R$, there exists $Z^N_k(x,y_1)$ and $Z^N_k(x,y_2)$ which are tight sequences in $N$ (for each fixed $k$, $x$, and $y$) such that the above supremum in $z$ is attained at $Z^N_k(x,y_i)$.

Consider the tight collection of random objects
\begin{align}\label{e.tight collections}
\left\{\S^N\right\}_{N=1}^\infty, \{\bm \cL^{(1),N}\}_{N=1}^\infty, \left\{Z^N_k(x,y) : k\in\N, x\in\Q_{>0}, y\in\Q\right\}_{N=1}^{\infty};
\end{align}
for the first collection, each member is an element of $\mc C(\R^2,\R)$ endowed with the topology of uniform convergence on compact sets; for the second collection, each member is an element of $\mc C(\R\times\N, \R)$ again endowed with the topology of uniform convergence on compact sets; and for the third collection, each member is a function $\N\times\Q_{>0}\times\Q\to\R$, i.e., an element of $\R^{\N\times\Q_{>0}\times\Q}$. For this last space we have the usual topology generated by cylinders, which corresponds to convergence of all finite collections.

As we noted above, tightness for $\S^N$ is a consequence of Proposition~\ref{p.S_N tightness}. We have assumed tightness for $\bm \cL^{(1),N}$ and that it weakly converges to $\bm\cP$, the parabolic Airy line ensemble. Finally, tightness of the final collection in the previous display follows from the tightness of $Z^N_k(x,y)$ for fixed $k$, $x$, and $y$ from Proposition~\ref{p.maximizer location} (since the topology is characterized by convergence of finite collections).

As a result, we have tightness of all the collections in \eqref{e.tight collections} jointly, which means we can extract a subsequence along which all four collections converge. We assume for notational convenience that the subsequence is the whole sequence, i.e., we will continue labeling the sequence by $N$. By Skorokhod's representation theorem, we may also assume that we are working on a probability space in which all these convergences occur almost surely. We denote the limiting objects by
\begin{align*}
\tilde \S, \bm\cP, \left\{Z_k(x,y) : k\in\N, x\in\Q_{>0}, y\in\Q\right\},
\end{align*}
where $\bm\cP$ is the parabolic Airy line ensemble, as noted above.
That $\tilde\S(0;\bm\cdot) = \cP_1(\bm\cdot)$ follows immediately from the definition \eqref{e.rescaled h general definition} of $\smash{\S^N(0; \bm\cdot)}$, so the first part of (ii) of Definition~\ref{d.airy sheet} is verified.
Next recall from Proposition~\ref{p.maximizer location} that, almost surely,
\begin{align}\label{e.Z_k asymptotic}
\lim_{k\to\infty}\frac{Z_k(x,y_i)}{\sqrt{k}} = -\frac{1}{\sqrt{2x}}.
\end{align}
Returning to \eqref{e.approximate variational problem in convergence proof} and by the definition of $Z^N_k(x,y_2)$, we see that, for every $k\in\N$,
\begin{equation}\label{e.S difference setup}
\begin{split}
\S^N(x;y_2) &= \sup_{w\leq Z^N_k(x,y_2)} \left(\cL^N_{N^{1/6}}(x;w) + \bm \cL^{(1),N}[(w,N^{1/6}-1) \to (Z^N_k(x,y_2),k)]\right) \\
&\qquad + \bm \cL^{(1),N}[(Z^N_k(x,y_2),k) \to (y_2,1)] + \mf o(1)\\
\S^N(x;y_1) &\geq \sup_{w\leq Z^N_k(x,y_2)} \left(\cL^N_{N^{1/6}}(x;w) + \bm \cL^{(1),N}[(w,N^{1/6}-1) \to (Z^N_k(x,y_2),k)]\right)\\
&\qquad + \bm \cL^{(1),N}[(Z^N_k(x,y_2),k) \to (y_1,1)]  + \mf o(1).
\end{split}
\end{equation}
Subtracting the first from the second then yields that
\begin{align}\label{e.S increment lower bound}
\S^N(x;y_1) - \S^N(x;y_2) &\geq \bm \cL^{(1),N}\left[(Z^N_k(x,y_2), k) \to (y_1,1)\right]\nonumber\\
&\qquad  - \cL^{(1),N}\left[(Z^N_k(x,y_2), k) \to (y_2,1)\right] + \mf o(1).
\end{align}
By the convergence of $Z^N_k(x,y_2)$ to $Z_k(x,y_2)$ and the uniform convergence on compact sets of $\cL^{(1),N}$ to $\bm\cP$, taking $N\to\infty$ yields
\begin{align*}
\tilde\S(x;y_1) - \tilde\S(x;y_2) \geq \bm\cP\left[(Z_k(x,y_2), k) \to (y_1,1)\right] - \bm\cP\left[(Z_k(x,y_2), k) \to (y_2,1)\right].
\end{align*}
Now, we know from the crossing lemma (Lemma~\ref{l.crossing lemma}) and since $y_1 < y_2$ that the righthand side is non-increasing in the argument which equals $Z_k(x,y_2)$; we will use this to replace $Z_k(x,y_2)$ by approximately $-\sqrt{k/(2x)}$. In more detail, fix $\varepsilon>0$. We know from \eqref{e.Z_k asymptotic} that for all large enough $k$ (depending on $\varepsilon$),
\begin{align}\label{e.upper bound on Z_k}
Z_k(x,y_2) \leq -\sqrt\frac{k}{2(x+\varepsilon)}.
\end{align}
Thus it follows that for all such $k$,
\begin{align}\label{e.S tilde and P inequality}
\tilde\S(x;y_1) - \tilde\S(x;y_2) &\geq \bm\cP\left[(\sqrt{k/(2(x+\varepsilon))}, k) \to (y_1,1)\right]\\
&\qquad\qquad - \bm\cP\left[(\sqrt{k/(2(x+\varepsilon))}, k) \to (y_2,1)\right]. \nonumber
\end{align}
Now taking $k\to\infty$ in the righthand side and recalling the definition \eqref{e.airy sheet increment} of $\S$ increments yields that, for every $\varepsilon>0$,
\begin{align*}
\tilde\S(x;y_1) - \tilde\S(x;y_2) \geq \S(x+\varepsilon, y_1) - \S(x+\varepsilon, y_2).
\end{align*}
Taking $\varepsilon\to 0$ and invoking the continuity of $\S$ yields $\tilde\S(x;y_1) - \tilde\S(x;y_2) \geq \S(x; y_1) - \S(x; y_2)$. (Note that it is only in this step that we crucially rely on the fact that $\S$ exists, has the coupling with $\bm\cP$, and is continuous; without these facts, it would not be possible to take $\varepsilon\to 0$ in \eqref{e.S tilde and P inequality} directly to obtain the same with $\varepsilon=0$,  as one first has to take $k\to\infty$, at which point continuity is not \emph{a priori} available.)

An analogous argument using $Z^N_k(x,y_1)$ in place of $Z^N_k(x,y_2)$ will give \eqref{e.S difference setup} with an equality for $\S^N(x;y_1)$ and a lower bound for $\S^N(x;y_2)$ (with $Z^N_k(x,y_1)$ on the righthand side for both); then we obtain a lower bound for $\S^N(x;y_2) - \S^N(x;y_1)$ analogous to \eqref{e.S increment lower bound}. Since $y_1 < y_2$, a lower bound on $Z_k(x,y_1)$ in place of \eqref{e.upper bound on Z_k}, again invoking Lemma~\ref{l.crossing lemma}, and taking the limit $N\to\infty$ and then $\varepsilon\to 0$ as above yields $\tilde\S(x;y_1) - \tilde\S(x;y_2) \leq \S(x; y_1) - \S(x; y_2)$. This completes the proof after noting $\S$ itself enjoys the desired coupling with $\bm\cP$, which yields it for $\tilde \S$ as~well.
\end{proof}

\subsection{Tightness of $\S^N$}\label{s.modulus of continuity}

In this section we show tightness of $\S^N$ in $\mc C(\R^2, \R)$, i.e., prove Proposition~\ref{p.S_N tightness}; recall that $\S^N(x;y)$ is just shorthand for the objects $\S^{\mrm{S6V}, \varepsilon}(x;y)$ and $\S^{\mrm{ASEP}, \theta,\varepsilon}(-x;-y)$, as processes.
To prove tightness of the process, it is sufficient to prove one-point tightness and uniform-in-$N$ modulus of continuity estimates.
The basic idea is to reduce the modulus of continuity estimates for $\S^N$ into the same where one argument of $\S^N$ is fixed; then, we can make use of modulus of continuity estimates for line ensembles.

The first ingredient for this is the following symmetrization lemma, which follows as a  consequence of general shift-invariance results for colored vertex models \cite{borodin2022shift,galashin2021symmetries}.

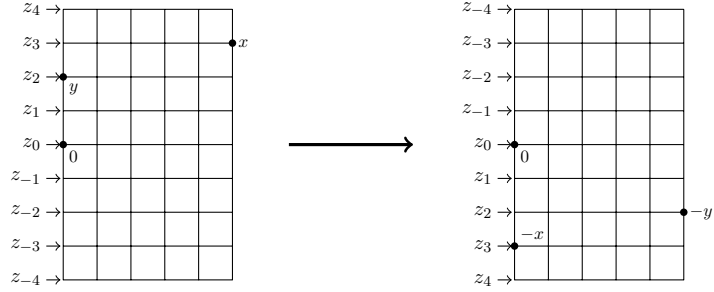
\begin{figure}
  \begin{tikzpicture}[scale=0.75]

  \draw (0,-2.405) grid[step=0.6] (3,2.4);

  \node[circle, fill, inner sep=1pt] at (0,0) {};
  \node[anchor=north west, scale=0.6] at (0,0) {$0$};

  \node[circle, fill, inner sep=1pt] at (0,1.2) {};
  \node[anchor=north west, scale=0.6] at (0,1.2) {$y$};

  \node[circle, fill, inner sep=1pt] at (3, 1.8) {};
  \node[anchor=west, scale=0.6] at (3,1.8) {$x$};

  \foreach \x in {-4, ..., 4}
  {
    \node[anchor=east, scale=0.7] at (-0.27, \x*0.6) {$z_{\x}$};
    \draw[->] (-0.3, \x*0.6) -- ++(0.25,0);
  }

  \draw[->, very thick] (4,0) -- ++(2.2,0);

  \begin{scope}[shift={(8,0)}]
  \draw (0,-2.405) grid[step=0.6] (3,2.4);

  \node[circle, fill, inner sep=1pt] at (0,0) {};
  \node[anchor=north west, scale=0.6] at (0,0) {$0$};

  \node[circle, fill, inner sep=1pt] at (3,-1.2) {};
  \node[anchor=west, scale=0.6] at (3,-1.2) {$-y$};

  \node[circle, fill, inner sep=1pt] at (0, -1.8) {};
  \node[anchor=south west, scale=0.6] at (0,-1.8) {$-x$};

  \foreach \x in {-4, ..., 4}
  {
    \node[anchor=east, scale=0.7] at (-0.27, -\x*0.6) {$z_{\x}$};
    \draw[->] (-0.3, -\x*0.6) -- ++(0.25,0);
  }
  \end{scope}
  \end{tikzpicture}
\caption{A depiction of the transformation of the domain of the colored S6V model used in the proof of Lemma~\ref{l.S_N symmetry}. The $z_i$ labels at the left side indicate the spectral parameter for the colored S6V weights on the corresponding row.}\label{f.color position symmetry}
\end{figure}

\begin{lemma}\label{l.S_N symmetry}
It holds that $\S^N(y;x) \stackrel{d}{=} \mc \S^N(-x; -y)$ as processes in $(x,y)$.
\end{lemma}

\begin{proof}
Recall that $\oldN$ is the parameter controlling the dimension of the domain of the colored S6V model. To establish the lemma, it suffices to show, for fixed $K\in\intint{-\oldN,\oldN}$ and $t\in\N$, that 
\begin{align}\label{e.symmetry to show}
\hssv(y, 0;K + x, t) + y \stackrel{d}{=} \hssv(-x, 0;K-y, t) - x
\end{align}
jointly across all $x, y$ such that $K+x,K-y\in\intint{-\oldN,\oldN}$, as $\S^N(y;x)$ and $\S^N(-x;-y)$ are scaled and shifted versions of the lefthand and righthand sides, respectively, in both the cases of ASEP and S6V (recall \eqref{e.rescaled h general definition}, \eqref{e.rescaled s6v definition}, \eqref{e.rescaled approximate asep definition}). It further suffices to establish the case of $K=0$; for combining the $(K,x,y)\mapsto (0, K+x,y)$ case of \eqref{e.symmetry to show} with \eqref{e.two parameter stationarity} yields
\begin{align*}
\hssv(y, 0; K + x, t) + y 
&\stackrel{\smash{d}}{=} \hssv(-K-x, 0; -y, t)  -K-x\\
&\stackrel{\smash{d}}{=} \hssv(-x, 0; K -y, t) -x.
\end{align*}
We show the $K=0$ case of \eqref{e.symmetry to show} now. First, in the case of $x<y$, the height functions are trivially equal: $\hssv(y, 0;x, t) = \oldN-y+1$ and $\hssv(-x, 0;-y, t) = \oldN+x+1$, since in both cases the lefthand sides are just the total number of arrows of the corresponding colors in the system. 

So we may assume $x\geq y$. We now invoke \cite[Theorem 1.6]{galashin2021symmetries}; let us explain the special case of its content that we need briefly now. Consider two colored S6V models on $\intint{1,t}\times\intint{-\oldN,\oldN}$, with spectral parameters $z_i$ at row $i$ in the first and $z_{\tau(i)}$ at row $i$ in the second, where $\tau:\intint{-\oldN, \oldN}\to\intint{-\oldN, \oldN}$ is a permutation; see Figure~\ref{f.color position symmetry}. In more detail, the measure is defined by setting the vertex weights (as in Figure~\ref{f.R weights}) of vertices in row $i$  with spectral parameter $z_i$ (or $z_{\tau(i)}$ in place of $z$, and the probability of a particular arrow configuration is the product of the vertex weights.

Then \cite[Theorem 1.6]{galashin2021symmetries} yields the following distributional equality. Let $k\in\N$ and $x_i, x_i', y_i, y_i' \in \intint{-\oldN, \oldN}$ for $i=1, \ldots, k$ with  $x_i\geq y_i$ and $x_i'\geq y_i'$, and under the additional condition that $\{z_j : y_i\leq j\leq x_i\} = \{z_{\tau(j)} : y_i' \leq j\leq x_i'\}$ (as multisets) for each $i$. Then,
\begin{equation}\label{e.galashin symmetry}
\begin{split}
\MoveEqLeft[6]
\left(\hssv_1(y_1, 0; x_1, t)+y_1, \ldots, \hssv_1(y_k, 0; x_k, t)+y_k\right)\\
&\stackrel{d}{=} \left(\hssv_2(y_1', 0; x_1', t) + y_1', \ldots, \hssv_2(y_k', 0; x_k', t) + y_k'\right),
\end{split}
\end{equation}
where the subscript corresponds to whether the statistic is in the first or second colored S6V model (that is, left or right in Figure~\ref{f.color position symmetry}) just described. (In the notation of \cite{galashin2021symmetries}, Theorem 1.6 there does not include the $+y_i$ or $+y_i'$ terms above; this is because the height function considered in \cite{galashin2021symmetries} counts the number of colored arrows below a given height and not above as in our setting. Translating from one to the other introduces the mentioned terms.)

For our situation, let $\tau(i) = -i$, i.e., we reverse the order of the formal spectral parameters, and for any choice of $x_i\geq y_i$, define $x_i' = -y_i$ and $y_i' = -x_i$. Then $x_i'\geq y_i'$ and clearly it holds that $\{z_{\tau(j)} : y_i' \leq j\leq x_i'\} = \{z_j : y_i \leq j\leq x_i\}$. Applying the distributional equality of \eqref{e.galashin symmetry} and specializing to the case that all spectral parameters are equal yields that
\begin{align*}
\left(\hssv(y, 0;x,t) + y\right)_{x,y\in\intint{-\mathtt N,\mathtt N}, x\geq y} \stackrel{d}{=} \left(\hssv(-x, 0; -y, t) -x\right)_{x,y\in\intint{-\mathtt N,\mathtt N}, x\geq y},
\end{align*}
which yields the $K=0$ case of \eqref{e.symmetry to show} and completes the proof.
\end{proof}

A second ingredient to prove the modulus of continuity estimates is the following monotonicity statement, sometimes called a quadrangle inequality, which will allow us to reduce the supremum over two arguments present in the modulus of continuity to a supremum over a single one; it is the analog of the crossing lemma (Lemma~\ref{l.crossing lemma}) for S6V and ASEP. In the context of last passage percolation models, such quadrangle inequalities have previously been used in similar ways (e.g., \cite{dauvergne2021scaling}) as well as for other purposes (e.g., \cite{basu2019fractal,balazs2020non,busani2022stationary,ganguly2023local,ganguly2023brownian}).

\begin{lemma}[Quadrangle inequality]\label{l.quadrangle}
For any $x_1 \leq x_2$ and $y_1\leq y_2$,
$$\S^N(x_1; y_1) + \S^N(x_2; y_2) \geq \S^N(x_1; y_2) + \S^N(x_2; y_1).$$
\end{lemma}

\begin{proof}
Recall from \eqref{e.s6v height function} the definition $\hssv(x,0; y,t) = \#\left\{k> y: j_{(t,k)}\geq x\right\}$. It suffices from the definition of $\S^{N}$ in the cases of ASEP \eqref{e.rescaled approximate asep definition} and S6V \eqref{e.rescaled s6v definition} in terms of $\hssv$ to prove that $\hssv(x_1, 0; y_1, t) + \hssv(x_2, 0; y_2,t) \geq \hssv(x_1, 0; y_2,t) + \hssv(x_2,0;y_1,t)$ for any $x_1\leq x_2$, $y_1\leq y_2$, $t\in\N$. 
To that end, observe that for any $x\in\Z$ and $y_1\leq y_2$,
\begin{align*}
\hssv(x,0; y_1,t) - \hssv(x, 0; y_2,t) &= \#\bigl\{k\in \intint{y_1+1, y_2} : j_{(t,k)}\geq x\bigr\},
\end{align*}
which is clearly non-increasing in $x$, as was to be shown.
\end{proof}

 With these ingredients we give the proof of Proposition~\ref{p.S_N tightness}.

\begin{proof}[Proof of Proposition~\ref{p.S_N tightness}]
We need to prove one-point tightness of $\S^N$ as well as establish uniform-in-$N$ modulus of continuity estimates. The tightness of $\{\S^N(x;y)\}_{N=1}^\infty$ for fixed $x,y$ follows immediately from the stationarity of $\S^N$ (recall \eqref{e.rescaled two parameter stationarity}) and Assumption~\ref{as.one-point tightness}. For uniform-in-$N$ modulus of continuity, we first observe that, from Lemma~\ref{l.quadrangle}, for any $x, y\in[-K,K]$ and $x_1 \leq x_2$ and $y_1 \leq y_2$,
\begin{align*}
\S^N(K; y_1) - \S^N(K; y_2) \leq \S^N(x; y_1) - \S^N(x; y_2) \leq \S^N(-K; y_1) - \S^N(-K; y_2)\\
\S^N(x_1; K) - \S^N(x_2; K) \leq \S^N(x_1; y) - \S^N(x_2; y) \leq \S^N(x_1; -K) - \S^N(x_2; -K).
\end{align*}
In particular, we see that, for any $x_1, x_2, y_1, y_2\in[-K,K]$,
\begin{align*}
\MoveEqLeft[2]
|\S^N(x_1; y_1) - \S^N(x_2; y_2)|\\
&\leq |\S^N(x_1; y_1) - \S^N(x_1; y_2)| + |\S^N(x_1; y_2) - \S^N(x_2; y_2)|\\
&\leq \max_{z\in\{+K, -K\}}|\S^N(z; y_1) - \S^N(z; y_2)|
+ \max_{z\in\{+K, -K\}}|\S^N(x_1; z) - \S^N(x_2; z)|.
\end{align*}
By Lemma~\ref{l.S_N symmetry} the marginal distribution of the second term is the same as that of the first with $x_i$ in place of $y_i$. Finally, the modulus of continuity over a single argument, i.e., of $y\mapsto \S^N(x;y)$ for $x$ fixed, is the same as that of the top line of a rescaled line ensemble (see \eqref{e.rescaled h general definition} and \eqref{e.rescaled two parameter stationarity}) and therefore follows from Corollary~\ref{c.mod con for L}. So by Corollary~\ref{c.mod con for L} and a union bound, for any $\rho>0$ and $\varepsilon>0$ there exists $\delta>0$ such that, with probability at least $1-\varepsilon$, whenever $|y_1-y_2|\leq \delta$ and $|x_1-x_2|\leq \delta$, the previous display is upper bounded by $\rho$. This completes the proof.
\end{proof}

\section{Weak monotonicity statements for one path}\label{s.monotonicity}

As indicated in Section~\ref{s.lack of monotonicity discussion}, a technical challenge in proving tightness for Hall-Littlewood Gibbs line ensembles is that the Gibbs property does not enjoy a form of stochastic monotonicity that previous proofs of tightness of prelimiting line ensembles have relied heavily upon.

This was a difficulty faced in \cite{corwin2018transversal} as well, and is one of the reasons the analysis there is restricted to the top curve. To get around the lack of monotonicity, they found a weaker form of monotonicity for a single curve with a lower boundary condition that does hold, in terms of a comparison of partition functions up to a constant multiplicative factor. Our analysis will rely on a similar weak monotonicity statement, but with the observation that it holds under slightly more general conditions (Lemma~\ref{l.partition function comparison with non-int boundaries}), namely in the presence of fixed upper and lower curves that the random curve must avoid, in addition to a possibly different lower boundary curve it interacts with via the weight factor $W$ from Definition~\ref{d.weight factor}.

In the next section, we introduce some monotonicity statements for basic objects such as non-intersecting Bernoulli random walk bridges. In Section~\ref{s.weak monotonicity} we give the weak monotonicity statements.

\subsection{Monotonicity statements for random walks and Gaussian objects}

Here we collect some well-known statements of forms of monotonicity enjoyed by objects underlying the Gibbs properties, such as non-intersecting random walks and Gaussian random variables.

The following is a straightforward computation using the form of the density of the normal distribution, and we point the reader to its proof in \cite{calvert2019brownian}. Recall the notation $\mc N(m,\sigma^2)$ for the normal distribution with mean $m$ and variance $\sigma^2$.

\begin{lemma}[{\cite[Lemma 5.15]{calvert2019brownian}}]\label{l.normal conditional prob montonicity}
Fix $r > 0$, $m \in \R$, and $\sigma^2 > 0$. Let $X$ be distributed as $\mc N(m, \sigma^2)$. Then the quantity
$\P(X \geq s + r \mid X \geq s)$ is a strictly decreasing function of $s \in \R$.
\end{lemma}

The next lemma quotes a form of stochastic monotonicity in the boundary data of non-intersecting Bernoulli random walk bridges. The result has appeared in the literature a number of times, e.g., \cite[Lemma 18]{cohn2000local} or, in notation closer to ours, \cite[Lemma 7.3]{dimitrov2021tightness}. Recall the definition of a Bernoulli path from Definition~\ref{d.bernoulli path} and that of a Bernoulli random walk bridge from before Definition~\ref{d.weight factor}.

\begin{lemma}[{\cite{cohn2000local,dimitrov2021tightness}}]\label{l.random walk bridge monotonicity}
Fix an interval $\llbracket a,b\rrbracket \subseteq \Z$. Let $\vec x^{*}, \vec y^*\in \Z^k$ be such that $x^*_1 \geq  \ldots \geq x^*_k$ and $y^*_1 \geq  \ldots \geq y^*_k$ for $*\in\{\shortuparrow, \downarrow\}$. Let $f^*, g^* : \llbracket a,b\rrbracket\to \Z$ be Bernoulli paths such that $f^*(a) \geq x^*_1$, $f^*(b) \geq y^*_1$, $g^*(a) \leq x^*_k$, $g^*(b) \leq y^*_k$ for $*\in\{\shortuparrow, \downarrow\}$. Suppose $f^{\shortuparrow}(x) \geq f^{\downarrow}(x)$, $g^{\shortuparrow}(x) \geq g^{\downarrow}(x)$ for all $x\in \llbracket a,b\rrbracket$ and $x^{\shortuparrow}_i \geq x^{\downarrow}_i$, $y^{\shortuparrow}_i \geq y^{\downarrow}_i$ for all $i=1, \ldots, k$.

Let $\bm B^{*} = (B^{*}_1, \ldots,  B^{*}_k)$ be a collection of independent Bernoulli random walk bridges, with $B^*_i$ from $(a, x^*_i)$ to $(b,y^*_i)$, conditioned on $g^*(x) \leq B^*_k(x) \leq  \ldots \leq B^*_1(x) \leq f^*(x)$ for all $x\in\llbracket a,b\rrbracket$ and for $*\in\{\shortuparrow, \downarrow\}$. Then there is a coupling of $\bm B^\shortuparrow$ and $\bm B^{\downarrow}$ such that
\begin{align*}
B^{\shortuparrow}_i(x) \geq B^{\downarrow}_i(x)
\end{align*}
for all $x\in \llbracket a,b\rrbracket$ and $i=1, \ldots, k$.
\end{lemma}

\subsection{Weak monotonicity}\label{s.weak monotonicity}

We start with a deterministic lemma about the weight of certain pairs of ordered paths with the same endpoints constrained to lie above a curve $\ell_{\mrm{bot}}$. Define
$$c(q) := \prod_{i=1}^\infty(1-q^i) > 0.$$
In the following, we call a down-step an increment of $-1$ in a Bernoulli path and a flat step an increment of $0$. Flipping a down-step to a flat-step (or vice versa) means keeping the left endpoint the same and changing the increment from $-1$ to $0$ (or from $0$ to $-1$) and keeping all other increments the same; note that is different from the more common flipping often used in the context of tilings/Bernoulli walks in which one flips a corner, i.e., a pair of consecutive increments $(-1,1)$ to $(1,-1)$ or vice versa. We also recall the definition of the weight factor $W$ and the notational device $W_{\mrm{low}}$ from \eqref{e.weight function}, \eqref{e.W up int low split up}, and \eqref{e.W up int low definition}.

\begin{lemma}\label{l.determinstic weak monotonicity}
Let $\intint{a,b}\subseteq \Z$ be a fixed interval and $\ell_{\mrm{bot}}:\intint{a,b}\to\Z$ be a Bernoulli path. Let $\ell_{\downarrow}, \ell_\shortuparrow:\intint{a,b}\to\Z$ be Bernoulli paths with the same endpoints satisfying the property that there exist $t\in\intint{a,b}$ and $k\leq \min(t-a, b-t)$ such that flipping $k$ down-steps in $\intint{a,t}$ in $\ell_{\downarrow}$ to flat-steps and $k$ flat-steps in $\intint{t,b}$ in $\ell_{\downarrow}$ to down-steps yields $\ell_{\shortuparrow}$. Then
\begin{align*}
W_{\mrm{low}}(\ell_{\downarrow},\ell_{\mrm{bot}}) \leq c(q)^{-1}\cdot W_{\mrm{low}}(\ell_{\shortuparrow},\ell_{\mrm{bot}}).
\end{align*}

\end{lemma}

While the hypotheses on $\ell_{\downarrow}$ and $\ell_{\shortuparrow}$ imply that $\ell_{\downarrow} \leq \ell_{\shortuparrow}$ pointwise, the two are not equivalent. This is because we require all the flat-to-up switches and all the up-to-flat switches to occur on either side of a given location $t$, which is a stronger condition; an example of a pair of paths which are pointwise ordered but do not satisfy this stronger condition is given in Figure~\ref{f.counterexample for ordering}.

We also point out the important detail that Lemma~\ref{l.determinstic weak monotonicity} requires the interaction through $W$ to be only of a single path $\ell_{\downarrow}$ or $\ell_{\shortuparrow}$, and with only a lower boundary curve; in particular, the comparison of partition functions is not claimed in the presence of an upper boundary.

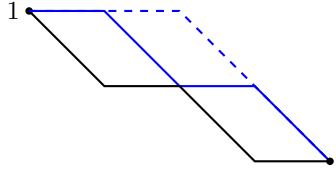
\begin{figure}
\begin{tikzpicture}
\draw[thick, blue] (0,0) -- ++(1,0) --++ (1,-1)-- ++(1,0) --++ (1,-1);
\draw[thick] (0,0) -- ++(1,-1) --++ (1,0)-- ++(1,-1) --++ (1,0);
\node[fill, circle, inner sep=1pt] at (0,0) {};
\node[fill, circle, inner sep=1pt] at (4,-2) {};
\node[anchor=east, scale=0.9] at (0,0) {$1$};

\draw[thick, blue, dashed] (0,0) -- ++(1,0) --++ (1,0)-- ++(1,-1) --++ (1,-1);

\end{tikzpicture}
\caption{The solid blue path is pointwise (weakly) larger than the black path, but the two paths do not satisfy the hypotheses of Lemma~\ref{l.determinstic weak monotonicity}; to go from black to blue, one needs to switch the black path's down-steps at positions 1 and 3 to flat-steps, and its flat-steps at positions 2 and 4 to down-steps. The dashed blue path is the changed portion of the solid blue path obtained after flipping the first down-step (at position 2) to a flat-step and the second flat-step (at position 3) to a down-step.}\label{f.counterexample for ordering}
\end{figure}

Lemma~\ref{l.determinstic weak monotonicity} was originally proved in the course of the proof of \cite[Lemma~4.1]{corwin2018transversal};  we reproduce the proof here.

\begin{proof}[Proof of Lemma~\ref{l.determinstic weak monotonicity}]
We assume that $\ell_{\downarrow}(i)\geq \ell_{\mrm{bot}}(i)$ for all $i\in \intint{a,b}$ as otherwise the lefthand side of the inequality to be shown is zero and we are done.
Let $x_1 < \ldots < x_k\leq t$ and $t\leq y_k <  \ldots < y_1$ be respectively the coordinates of the left endpoints of the down-steps that are flipped to flat-steps and of the flat-steps that are flipped to down-steps in going from $\ell_{\downarrow}$ to $\ell_{\shortuparrow}$. Define $\ell_{\downarrow} = \ell_0, \ell_1, \ldots, \ell_k = \ell_{\shortuparrow}$ by letting $\ell_j$ be obtained from $\ell_{\downarrow}$ after performing the flips at $x_1, \ldots, x_j$ and $y_1, \ldots, y_j$ only.

Now, $\ell_j$ differs from $\ell_{j-1}$ only on $\intint{x_j+1, y_j}$, where the former is raised by 1 compared to the latter. Recall from \eqref{e.W up int low definition} that $W_{\mrm{low}}(\ell_i, \ell_{\mrm{bot}}) = \prod_{x=a+1}^b (1-\one_{\Delta_i(x-1)-\Delta_i(x) = 1}q^{\Delta_i(x-1)})$, where $\Delta_i(x) = \ell_i(x) - \ell_{\mrm{bot}}(x)$. By examining how each factor of the product changes on going from $W_{\mrm{low}}(\ell_{j-1}, \ell_{\mrm{bot}})$ to $W_{\mrm{low}}(\ell_j, \ell_{\mrm{bot}})$, it follows that:
\begin{itemize}
  \item the factor corresponding to $x=x_j+1$ changes from $$1-\one_{\Delta_{j-1}(x_j)-\Delta_{j-1}(x_j+1) = 1}q^{\Delta_{j-1}(x_j)}$$ to 1, i.e., it (weakly) increases;

  \item the factors corresponding to $x\in\intint{x_j+2, y_j}$ change from $$1-\one_{\Delta_{j-1}(x-1)-\Delta_{j-1}(x) = 1}q^{\Delta_{j-1}(x-1)}$$ to $1-\one_{\Delta_j(x-1)-\Delta_j(x) = 1}q^{\Delta_j(x-1)} = 1-\one_{\Delta_{j-1}(x-1)-\Delta_{j-1}(x) = 1}q^{\Delta_{j-1}(x-1)+1}$, i.e., they (weakly) increase;

  \item the factor corresponding to $x=y_j+1$ changes from $$1-\one_{\Delta_{j-1}(y_j)-\Delta_{j-1}(y_j+1) = 1}q^{\Delta_{j-1}(y_j)}$$ to $1-\one_{\Delta_{j}(y_j)-\Delta_j(y_j+1) = 1}q^{\Delta_j(y_j)} = 1-\one_{\Delta_{j-1}(y_j)-\Delta_{j-1}(y_j+1) = 0}q^{\Delta_{j-1}(y_j)+1}$.
\end{itemize}
The last change can be a decrease depending on the indicator function, but the decrease is by at most a factor of $1-q^{m_j}$, where $m_j = 1+ \min_{x\in\intint{x_j+1,y_j}}(\ell_{j-1}(x) - \ell_{\mrm{bot}}(x))$. Thus we have shown that
\begin{align*}
W_{\mrm{low}}(\ell_{j-1}, \ell_{\mrm{bot}}) \leq (1-q^{m_j})^{-1}W_{\mrm{low}}(\ell_{j}, \ell_{\mrm{bot}}).
\end{align*}
Since $\ell_j$ differs from $\ell_{j-1}$ by increasing the path by 1 on $\intint{x_j+1, y_{j}}$, we obtain $m_j \geq m_{j-1}+1$. This implies the claim.
\end{proof}

\begin{corollary}\label{c.coupling yields one-point comparison}
Let $\intint{a,b}\subseteq \Z$ be a fixed interval and $\ell_{\mrm{bot}}:\intint{a,b}\to\Z$ be a Bernoulli path. Suppose $B^{\downarrow}$ and $B^{\shortuparrow}$ are random Bernoulli paths with respective laws $\P_1$ and $\P_2$ which can be coupled such that they almost surely satisfy the hypothesis of Lemma~\ref{l.determinstic weak monotonicity} (in particular $t$ and $k$ there may be random). Then
\begin{align*}
\E_1\left[W_{\mrm{low}}(B^{\downarrow}, \ell_{\mrm{bot}})\right] \leq c(q)^{-1}\E_2\left[W_{\mrm{low}}(B^{\shortuparrow}, \ell_{\mrm{bot}})\right].
\end{align*}

\end{corollary}

\begin{proof}
This follows by taking expectations in  the inequality of Lemma~\ref{l.determinstic weak monotonicity}.
\end{proof}

The following shows that we have the comparison of partition functions in the case of paths which interact with a lower boundary via the Hall-Littlewood Gibbs weight and additionally avoid some upper and lower curves (the latter may differ from the lower boundary curve). We will use this fact heavily in our arguments as a replacement for the full monotonicity property enjoyed by line ensembles and Gibbs properties previously studied.

\begin{lemma}\label{l.partition function comparison with non-int boundaries}
Let $\intint{a,b} \subseteq \Z$ be a fixed interval and $\ell_{\mrm{bot}}:\intint{a,b}\to\Z$ be a Bernoulli path. Let $f, g: \intint{a,b}\to\Z$ be Bernoulli paths with $g(i)\leq f(i)$ for all $i\in \intint{a,b}$. Let $t\in \intint{a,b}$, $z\in\Z$, $x\in \intint{g(a), f(a)}$, and $y\in \intint{g(b), f(b)}$ be such that, with $B$ a Bernoulli random walk bridge from $(a,x)$ to $(b,y)$ (we assume there exists at least one Bernoulli path connecting those two points),
$$\P\Bigl(g(i) \leq B(i)\leq f(i) \text{ for all } i\in \intint{a,b}, B(t)\leq z\Bigr) > 0.$$
Let $B^{\downarrow}$ and $B^{\shortuparrow}$ be distributed as Bernoulli random walk bridges from $(a,x)$ to $(b,y)$ conditioned to satisfy $g(x) \leq B^{\downarrow}(x) \leq f(x)$ and $g(x) \leq B^{\shortuparrow}(x) \leq f(x)$ for all $x\in\intint{a,b}$, with $B^{\downarrow}$ additionally conditioned on $B^{\downarrow}(t) \leq z$. Let $\E_1$ and $\E_2$ be the expectation operators associated with their respective laws. Then it holds that
\begin{align*}
\E_1\left[W_{\mrm{low}}(B^{\downarrow},\ell_{\mrm{bot}})\right] \leq c(q)^{-1}\E_2\left[W_{\mrm{low}}(B^{\shortuparrow},\ell_{\mrm{bot}})\right].
\end{align*}

\end{lemma}

Before proving Lemma~\ref{l.partition function comparison with non-int boundaries}, let us give an example showing how it will be used to obtain a comparison of one-point probabilities. Let $\ell_{\mrm{bot}}$, $f$, and $g$ be as in the corollary and fix $t$. Let $B$ be a Bernoulli random walk bridge from $(a,x)$ to $(b,y)$, conditioned to avoid $f$ and $g$, and let $L$ be the same but reweighted by the Radon-Nikodym derivative $W_{\mrm{low}}(L, \ell_{\mrm{bot}})/\E[W_{\mrm{low}}(L, \ell_{\mrm{bot}})]$. Then we see that
\begin{align*}
\P\left(L(t) \leq z\right) &= \frac{\E[\one_{B(t)\leq z}W_{\mrm{low}}(B, \ell_{\mrm{bot}})]}{\E[W_{\mrm{low}}(B,\ell_{\mrm{bot}})]}\\
&= \frac{\E[W_{\mrm{low}}(B, \ell_{\mrm{bot}})\mid B(t)\leq z]}{\E[W_{\mrm{low}}(B,\ell_{\mrm{bot}})]}\cdot \P\left(B(t)\leq z\right)\\
&\leq c(q)^{-1}\P\left(B(t)\leq z\right),
\end{align*}
the last inequality by Lemma~\ref{l.partition function comparison with non-int boundaries}. The final probability is simply a non-intersecting Bernoulli random walk bridge, for which tools like monotonicity (Lemma~\ref{l.random walk bridge monotonicity}) and two-point estimates for the underlying unconditioned Bernoulli random walk bridge (Lemma~\ref{l.random walk bridge fluctuation}) are available. We emphasize again, however, that while Lemma~\ref{l.partition function comparison with non-int boundaries} allows the interaction via $W_{\mrm{low}}$ to be with a lower curve $\ell_{\mrm{bot}}$, it required that there be no upper curve for that interaction (unlike the case with the full weight factor $W$); in applications, when we will have to work with the full factor $W$, we will need to make other context-specific arguments to remove the upper boundary interaction before performing the above argument to get a probability comparison. We also note that the comparison is only of one-point probabilities and is thus not simply the same up to a multiplicative constant as the form of stochastic monotonicity available for non-intersecting Bernoulli random walk bridges (Lemma~\ref{l.random walk bridge monotonicity}).

Now we turn to the proof of Lemma~\ref{l.partition function comparison with non-int boundaries}. We will need the following quantitative form of monotonicity for a single Bernoulli random walk bridge conditioned on avoiding an upper and lower boundary curve. A similar statement appeared for (multiple) non-intersecting Brownian bridges in \cite[Lemma~2.4]{LSCNB}. Its proof consists of a construction of a Markov chain on pairs of paths (respectively with the two boundary conditions) whose dynamics preserve the quantitative ordering and whose stationary distribution has the desired marginals. Similar ideas have appeared many times, e.g., \cite{corwin2014brownian,dimitrov2021tightness,LSCNB}, and thus we defer the proof to Appendix~\ref{s.random gibbs}.

\begin{lemma}\label{l.quantitative monotonicity for n=1}
Let $\intint{a,b} \subseteq \Z$ be a fixed interval and $f,g: \intint{a,b}\to\Z$ be Bernoulli paths with $g(i)\leq f(i)$ for all $i\in \intint{a,b}$. Let $x^{\shortuparrow}, x^{\downarrow}, y^{\shortuparrow}, y^{\downarrow}\in\Z$ be such that $g(a) \leq x^{\downarrow}\leq x^{\shortuparrow} \leq f(a)$ and $g(b) \leq y^{\downarrow}\leq y^{\shortuparrow} \leq f(b)$. Let $M=\max(x^{\shortuparrow}-x^{\downarrow}, y^{\shortuparrow}-y^{\downarrow})$. Let $B^{*}$ be a Bernoulli random walk bridge from $(a,x^*)$ to $(b, y^*)$ conditioned on $g(x) \leq B^*(x)\leq f(x)$ for $x\in\intint{a,b}$, for $*\in\{\shortuparrow,\downarrow\}$. There exists a coupling of $B^{\shortuparrow}$ and $B^{\downarrow}$ such that, almost surely for all $x\in\intint{a,b}$,
\begin{align*}
B^{\downarrow}(x) \leq B^{\shortuparrow}(x) \leq B^{\downarrow}(x) + M.
\end{align*}

\end{lemma}

\begin{proof}[Proof of Lemma~\ref{l.partition function comparison with non-int boundaries}]
We may assume that $z < f(t)$ as otherwise $B^{\shortuparrow}$ and $B^{\downarrow}$ have the same distribution and the claim is trivial. Next we observe that $B^{\shortuparrow}(t)$ stochastically dominates $B^{\downarrow}(t)$, using Lemma~\ref{l.random walk bridge monotonicity}.

As a consequence of Corollary~\ref{c.coupling yields one-point comparison}, it is sufficient to show the following statement: a coupling meeting the hypotheses of Lemma~\ref{l.determinstic weak monotonicity} exists between two Bernoulli random walk bridges $B^{\shortuparrow}$ and $B^{\downarrow}$ from $(a,x)$ to $(b,y)$ conditioned on avoiding both $f$ and $g$ with $B^{*}$ additionally conditioned on $B^{*}(t) = z^{*}$ for $*\in\{\downarrow,\shortuparrow\}$, where $z^{\shortuparrow} \geq z^{\downarrow}$. Indeed, with the latter statement in hand, one can obtain the desired coupling of the original $B^{\downarrow}$ and $B^{\shortuparrow}$ by first coupling $B^{\shortuparrow}(t)$ and $B^{\downarrow}(t)$ via the stochastic domination, and then making use of the previous statement which is conditioned on the value of $B^{*}(t)$.

To show the statement from the previous paragraph regarding the existence of a coupling which we have reduced to, it is sufficient to prove it in the case where $z^{\shortuparrow} = z^{\downarrow} + 1$, as then the full statement follows by iteration. By the Gibbs property for non-intersecting Bernoulli random walk bridges, it is further sufficient to show that there exists a coupling of $B^{\shortuparrow}$ and $B^{\downarrow}$ where the former is now a Bernoulli random walk bridge from $(a,x)$ to $(t,z)$ and the latter is one from $(a,x)$ to $(t,z-1)$, both conditioned on avoiding $f$ and $g$ on $\intint{a,t}$, with the property that $B^{\shortuparrow}$ is obtained from $B^{\downarrow}$ by flipping one down-step of the latter to a flat-step; a similar proof will then show the analogous statement on $\intint{t,b}$, and one obtains the desired coupling by concatenating the couplings on either side of $t$ (which yields the correct marginal distributions by the Gibbs property for non-intersecting Bernoulli random walk bridges).

Finally, we prove the last statement we have reduced to by induction on the interval length $k=t-a$. More explicitly, we establish by induction on $k$ that there exists a coupling of two Bernoulli random walk bridges $B^{\shortuparrow}$ and $B^{\downarrow}$, the first from $(a,x)$ to $(t,z)$ and the second from $(a,x)$ to $(t,z-1)$, each conditioned on avoiding both $f$ and $g$, such that $B^{\shortuparrow}$ is obtained from $B^{\downarrow}$ by flipping a single down-step to a flat-step. Here $z$ is assumed to be such that such paths exist.

In the case that $k=1$ the claim is trivial, as $B^{\shortuparrow}$ must consist of a single flat-step and $B^{\downarrow}$ a single down-step. Assuming the statement holds for $k$, we now show it for $k+1$. Using Lemma~\ref{l.quantitative monotonicity for n=1} (as the right endpoint of $B^{\shortuparrow}$ is higher by 1 than $B^{\downarrow}$), we first couple $B^{\shortuparrow}(t-1)$ and $B^{\downarrow}(t-1)$ so that $B^{\downarrow}(t-1)\leq B^{\shortuparrow}(t-1)\leq B^{\downarrow}(t-1)+1$. Note that there are three possibilities for the individual values of $B^{\shortuparrow}(t-1)$, and $B^{\downarrow}(t-1)$ and we give the coupling of the whole paths depending on the case:
\begin{enumerate}
  \item $B^{\shortuparrow}(t-1) = z+1$, $B^{\downarrow}(t-1) = z$: in this case both paths must follow down steps in the last slot, and we couple $B^{\shortuparrow}|_{\intint{a,t-1}}$ and $B^{\downarrow}|_{\intint{a,t-1}}$ using the induction hypothesis.

  \item $B^{\shortuparrow}(t-1) = z$, $B^{\downarrow}(t-1) = z$: in this case $B^{\shortuparrow}$ must have an flat-step in the last slot and $B^{\downarrow}$ must have a down-step, and we couple $B^{\shortuparrow}|_{\intint{a,t-1}}$ and $B^{\downarrow}|_{\intint{a,t-1}}$ to be equal.

  \item $B^{\shortuparrow}(t-1) = z$, $B^{\downarrow}(t-1) = z-1$: in this case $B^{\shortuparrow}$ and $B^{\downarrow}$ must both have a flat-step in the last slot, and we couple $B^{\shortuparrow}|_{\intint{a,t-1}}$ and $B^{\downarrow}|_{\intint{a,t-1}}$ using the induction hypothesis.
\end{enumerate}
(Note that the case of $B^{\shortuparrow}(t-1) = z+1$, $B^{\downarrow}(t-1) = z-1$ is excluded by the coupling at $t-1$.)
It is immediate that the above coupling satisfies the hypothesis of Lemma~\ref{l.determinstic weak monotonicity} and has the correct marginals. This completes the induction step and hence the proof.
\end{proof}

\section{Lower tail estimates}\label{s.lower tail}

In this section we start developing the proof of Theorem~\ref{t.infimum lower tail} on the lower tail of the $k$\th curve of $\bm L^N$ (recall the definitions and hypotheses from Section~\ref{s.tightness preliminaries}). In particular, in this section, as well as Sections~\ref{s.uniform separation}, \ref{s.partition function and non-intersection}, and \ref{s.bg in the limit}, $\bm L^N$ will be assumed to satisfy Assumptions~\ref{as.HL Gibbs} and \ref{as.one-point tightness} without explicit mention.

\subsection{The pieces for the lower tail}
Theorem~\ref{t.infimum lower tail} is proved by an induction argument. The induction hypothesis will combine the lower tail estimate for a given curve with a statement that curves remain uniformly separated; the ultimate result of uniform separation will also be useful at several other locations (in particular, in the proof of Proposition~\ref{p.limits have BG}, that weak limits have the Brownian Gibbs property) and we record it as a separate theorem.

In what follows, for a (measurable) set $A\subseteq \R$, an integer $k\in\N$, a (possibly infinite) interval $\Lambda\subseteq \N$  with $\intint{1,k+1}\subseteq \Lambda$, a stochastic process $\bm X:\Lambda\times\R\to\R$, and $\delta>0$, we define the event
\begin{equation}\label{e.Sep definition}
\msf{Sep}^N_k(\delta, A, \bm X) = \left\{\min_{i=1, \ldots, k}\inf_{xN^{-2/3}\in A}\bigl(X_i(x) - X_{i+1}(x)\bigr) \geq \delta N^{1/3}\right\}.
\end{equation}
If $k=-1$ or $0$, we take the event to be the whole space (i.e., it trivially holds). We will sometimes drop some of the parameters from the notation, depending on the context,  when it is obvious what they are. We also recall the convention for the notation for event names mentioned in Remark~\ref{r.event name convention}. The use of the event $\msf{Sep}^N_k$ (and similar ones) is that, on it, the weight factor $W$ \eqref{e.weight function} is very close to 1 for all large enough $N$, as stated next.

\begin{lemma}\label{l.W bound}
Fix an interval $[a,b]$, $k\in\N$, $q\in[0,1)$ $\varepsilon>0$, and $\delta>0$.  Let $\bm X:\intint{1,k}\times\intint{aN^{2/3}, bN^{2/3}}$ be a collection of Bernoulli paths. There exists $N_0 = N_0(q, \varepsilon, \delta, b-a,k)$ such that, if $X_i(x) \geq X_{i+1}(x)+\delta N^{1/3}$ for all $x\in\intint{aN^{2/3}, bN^{2/3}}$ and $i\in\intint{1,k-1}$, then, for $N>N_0$,
$W(\bm X, \intint{aN^{2/3}, bN^{2/3}}) > 1-\varepsilon.$
\end{lemma}

\begin{proof}
Under the conditions assumed, $W(\bm X, \intint{aN^{2/3}, bN^{2/3}}) > (1-q^{\delta N^{1/3}})^{k(b-a)N^{2/3}}$, from which the claim follows.
\end{proof}

Next we state the result on uniform separation. It will be proved later in this section.

\begin{theorem}[Uniform separation]\label{t.uniform separation}
Fix $k\in\N$, $T>2$, and $\varepsilon>0$. There exist $\delta = \delta(k,T,\varepsilon)>0$ and $N_0= N_0(k,T,\varepsilon)$ such that, for $N\geq N_0$,
$$\P\left(\msf{Sep}^N_k(\delta, [-T,T], \bm L^N)\right) \geq 1-\varepsilon.$$
\end{theorem}

We also adopt the following convenient notational convention in the following sections: $\overline L^N_k$ will denote $L^N_k$ with the linear growth removed, i.e., with $p$ as in Assumption~\ref{as.one-point tightness}
\begin{align*}
\overline L^N_k(x) = L^N_k(x) - px.
\end{align*}

An important part of the induction step proving Theorems~\ref{t.infimum lower tail} and \ref{t.uniform separation} is captured in the next proposition.

\begin{proposition}[Lower tail inductive step]\label{p.lower tail induction step}
Fix $k\in\N$. Suppose for any $T\geq 2$ and $\varepsilon>0$, there exist $M = M(k,T,\varepsilon)$ and $\delta = \delta(k,T,\varepsilon)>0$ such that
\begin{align}\label{e.k-1 curve lower tail bound induction step}
\P\left(\inf_{sN^{-2/3}\in[-T,T]} \overline L^N_{k-1}(s) < -MN^{1/3}\right) &\leq \varepsilon
\quad\text{and} \quad
\P\left(\msf{Sep}^N_{k-2}(\delta, [-T,T], \bm L^N)^c\right) \leq \varepsilon
\end{align}
(the first trivially holding for $k=1$ and the second for $k=1,2$). Then for any $T\geq 2$ and $\varepsilon>0$, there exist $R = R(k, T,\varepsilon)$ and $N_0=N_0(k,T,\varepsilon)$ such that, for $N\geq N_0$,
\begin{align*}
\P\left(\inf_{sN^{-2/3}\in[-T,T]} \overline L^N_k(s) < -RN^{1/3}\right) \leq \varepsilon.
\end{align*}

\end{proposition}

We will prove this statement in Section~\ref{s.proof of lower tail induction}.
To complete the induction step we also need a bound on the separation between the $k$ and $(k+1)$\textsuperscript{th} curves. This is the next proposition.

\begin{proposition}[Uniform separation inductive step]\label{p.uniform separation}
Let $k\in\N$ and, for any $j\in\N$, let $\msf{Sep}^N_{j}$ be as in \eqref{e.Sep definition}. Suppose for any $\varepsilon>0$ and $T\geq 2$ there exist $\eta = \eta(k,T,\varepsilon)>0$, $M = M(k,T,\varepsilon)$, and $N_0=N_0(k,T,\varepsilon)$ such that, for $N\geq N_0$,
\begin{align}\label{e.separation induction hypothesis}
\P\Bigl(\msf{Sep}^N_{k-2}(\eta, [-T,T], \bm L^N)^c\Bigr) < \varepsilon \quad\text{and}{\quad}
\P\left(\inf_{xN^{-2/3}\in[-T,T]} \overline L^N_{k}(x) < -MN^{1/3}\right) \leq \varepsilon
\end{align}
(the first trivially holding for $k=1,2$). Then for any $\varepsilon>0$ and $T\geq 2$, there exist $\delta = \delta(k, T, \varepsilon) > 0$ and $N_0=N_0(k,T,\varepsilon)$ such that, for $N\geq N_0$
\begin{align*}
\P\left(\msf{Sep}^N_{k-1}(\delta, [-T,T], \bm L^N)\right) \geq 1-\varepsilon.
\end{align*}

\end{proposition}

We will prove Proposition~\ref{p.uniform separation} in Section~\ref{s.uniform separation}. Next we will prove Theorems~\ref{t.infimum lower tail} and \ref{t.uniform separation} together using Propositions~\ref{p.lower tail induction step} and \ref{p.uniform separation}.

\begin{proof}[Proof of Theorems~\ref{t.infimum lower tail} and \ref{t.uniform separation}]
We prove the following statement by induction, which implies both theorems: for any $\varepsilon>0$ there exists $M = M(\varepsilon, k, T)$ and $\delta = \delta(\varepsilon, k, T)$ such that
\begin{align}\label{e.lower tail and unif sep step}
&\P\left(\inf_{xN^{-2/3}\in[-T,T]} \overline L^N_k(x) \geq -MN^{1/3}\right) \geq 1-\varepsilon\\
\text{and} \quad &\P\left(\msf{Sep}^N_{k-1}(\delta, [-T,T], \bm L^N)\right)\geq 1-\varepsilon.
\end{align}
By convention, $L^N_0 \equiv \infty$ so the $k=0$ case is trivially true. We now show the induction statement that if the displayed statement holds for $k-1$, it also holds for $k$.

Indeed, the induction hypothesis says that the hypotheses \eqref{e.k-1 curve lower tail bound induction step} of Proposition~\ref{p.lower tail induction step} hold, which yields that for any $\varepsilon>0$ and $T\geq 2$, there exists $M$ such that the first inequality in \eqref{e.lower tail and unif sep step} holds. Now the first hypothesis in \eqref{e.separation induction hypothesis} of Proposition~\ref{p.uniform separation} holds by the induction hypothesis, and the second hypothesis in \eqref{e.separation induction hypothesis} holds by the deduction just made. Thus we obtain the second inequality in \eqref{e.lower tail and unif sep step} from Proposition~\ref{p.uniform separation}, completing the induction step and thus the proof.
\end{proof}

\subsection{A note on conditioning}

At several points in this and upcoming sections, we will consider various types of conditional probabilities, conditioning on $\sigma$-algebras, events, and sometimes both. Here we briefly lay out the meanings of these expressions.

For a $\sigma$-algebra $\F$, a random variable $X$, and an event $\msf A$, we define
\begin{align}\label{e.conditional expectation with conditioning definition}
\EF[X\mid \msf A] := \frac{\EF[X\one_{\msf A}]}{\PF(\msf A)},
\end{align}
mirroring the usual definition of $\E[X\mid \msf A]$. When $X = \one_{\msf E}$ for an event $\msf E$, we write the lefthand side as $\PF(\msf E\mid \msf A)$.

Next consider a $\sigma$-algebra $\F$, an event $\msf A$ which is $\F$-measurable, and an arbitrary event $\msf E$. Then it holds that
\begin{align}\label{e.conditional probability definition}
\P\left(\msf E \mid \msf A\right) = \E\left[\PF(\msf E) \mid \msf A\right],
\end{align}
since the lefthand side is $\P(\msf E\cap \msf A)/\P(\msf A) = \E[\PF(\msf E)\one_{\msf A}]/\P(\msf A)$, which is the definition of the righthand side by \eqref{e.conditional expectation with conditioning definition}; here we used the tower property of conditional probabilities and that $\msf A \in \F$.

\subsection{Setting up the proof of Proposition~\ref{p.lower tail induction step}}\label{s.proof of lower tail induction}

We start by stating that the $k$\textsuperscript{th} curve cannot be uniformly low on a large interval. This is the first step in the proof of Proposition~\ref{p.lower tail induction step}, and will be proved in Section~\ref{s.kth curve not uniformly low}.
For an interval $I\subseteq \R$ and $R>0$, we define the event
\begin{equation}\label{e.Low event definition}
\msf{Low}^{I}_k(R) := \left\{\sup_{sN^{-2/3}\in I} \overline L^N_k(s) \leq -RN^{1/3}\right\}.
\end{equation}

\begin{lemma}\label{l.low event bound}
Let $k\in\N$ and let $\msf{Sep}^N_{k-2}(\delta, \{T, U'\}, \bm L^N)$ be as in \eqref{e.Sep definition}.
Suppose for any $\varepsilon>0$, $T$, and $U'$ there exist $\delta = \delta(\varepsilon, k, T,U') >0$, $M = M(\varepsilon, k, U')$, and $N_0=N_0(k,T,\varepsilon)$ such that, for $N\geq N_0$,
$$\P\left(\inf_{sN^{-2/3}\in[-U,U]} \overline L^N_{k-1}(s) < -MN^{1/3}\right) \leq \varepsilon \quad\text{and}\quad \P\left(\msf{Sep}^N_{k-2}(\delta, \{T, U'\}, \bm L^N)^c\right) \leq \varepsilon.$$
Then for any $\varepsilon>0$ and $t>0$, there exist $U = U(\varepsilon, T)$, $R_{\mrm{low}} = R_{\mrm{low}}(T,U,\varepsilon, k)$, and $N_0 = N_0(k, U, \varepsilon)$ such that, for $N\geq N_0$,
\begin{align*}
\P\left(\msf{Low}^{[T,U]}_k(R_{\mrm{low}})\right) \leq \varepsilon.
\end{align*}
\end{lemma}

Next we state a lower tail bound for the $k$\textsuperscript{th} curve at a single point on the event that there exist points on either side which are not too low, and conditioned on the $(k-1)$\textsuperscript{st} curve not being too low. This is the second step in the proof of Proposition~\ref{p.lower tail induction step}; it will be proved in Section~\ref{s.kth curve not too low at a point}.

\begin{lemma}\label{l.midpoint lower tail}
Let $[a,b]$ be an interval, $k\in\N$, and $\delta \geq 0$. Then there exist $N_0 = N_0(b-a)$, $C = C(b-a)>0$, and $c = c(k)>0$ such that, for any $t\in[a, b]$, $R>0$, $K\geq 1$, and $N>N_0$,
\begin{align*}
\MoveEqLeft[22]
\P\left(\left.\parbox{0.42\textwidth}{\centering$\overline L^N_k(tN^{2/3}) < -(R+K)N^{1/3}$,\\[2pt] $\displaystyle\sup_{xN^{-2/3}\in[a, t]} \overline L^N_k(x) \geq -RN^{1/3}$, \\[-1pt] $\displaystyle\sup_{xN^{-2/3}\in[t, b]} \overline L^N_k(x) \geq -RN^{1/3}$ }\!\! \right| \inf_{xN^{-2/3}\in[a,b]} \overline L^N_{k-1}(x)  > -(R-\delta)N^{1/3}\right)\\
&\leq C\exp\left(-cK^2/(b-a)\right).
\end{align*}
Further, $C$ and and $N_0$ can be taken to be increasing in $b-a$.
\end{lemma}

Assuming these two lemmas, getting the uniform bound Proposition~\ref{p.lower tail induction step} is a fairly standard chaining argument; we give it now before returning to the proofs of Lemmas~\ref{l.low event bound} and \ref{l.midpoint lower tail} in Sections~\ref{s.kth curve not uniformly low} and \ref{s.kth curve not too low at a point}, respectively.

\begin{proof}[Proof of Proposition~\ref{p.lower tail induction step}]

We start by proving a lower tail bound for $L^N_k$ at the boundaries and midpoint of $[-TN^{2/3},TN^{2/3}]$, and then we will turn to setting up a chaining argument to extend to the infimum over the whole interval.

For $U>T$ and $R>0$, let $\msf{Low}_k^{[-U, -T]}(\frac{1}{4}R)$ and $\msf{Low}_k^{[T, U]}(\frac{1}{4}R)$ be as in \eqref{e.Low event definition}.
Using Lemma~\ref{l.low event bound} (whose hypothesis is assumed in \eqref{e.k-1 curve lower tail bound induction step}), and the same after reflecting in the line $\{x=0\}$, we obtain $U>T$ and $R$ such that
\begin{align}\label{e.bound on P(low)}
\P\left(\msf{Low}_k^{[-U, -T]}(\tfrac{1}{4}R)\cup \msf{Low}_k^{[T, U]}(\tfrac{1}{4}R)\right) \leq \tfrac{1}{3}\varepsilon;
\end{align}
we may raise $R$ if necessary so that $\exp(-cR^2/(32U)) < \tfrac{1}{3}\varepsilon$, where $c$ is as in Lemma~\ref{l.midpoint lower tail} (which is possible by the monotonicity of the above probability in $R$).
Next, using \eqref{e.k-1 curve lower tail bound induction step}, we may raise $R$ further if necessary so that
\begin{align}\label{e.k-1 curve lower tail bound induction step in proof}
\P\left(\inf_{sN^{-2/3}\in[-U,U]} \overline L^N_{k-1}(s) \leq -\tfrac{1}{8}RN^{1/3}\right) \leq \tfrac{1}{3}\varepsilon.
\end{align}

Let $\msf{NotLow}_k = (\msf{Low}_k^{[-U, -T]}(\tfrac{1}{4}R)\cup \msf{Low}_k^{[T,U]}(\tfrac{1}{4}R))^c$. Let $t\in\{-T,0,T\}$. Then we see that
\begin{align*}
\MoveEqLeft[2]
\P\left(\overline L^N_k(tN^{2/3}) < -\tfrac{1}{2}RN^{1/3} \midd \inf_{xN^{-2/3}\in[-U,U]}\overline L^N_{k-1}(x) > -\tfrac{1}{8}RN^{1/3}\right)\\
&\leq \P\left(\overline L^N_k(tN^{2/3}) < -\tfrac{1}{2}RN^{1/3}, \msf{NotLow}_k \midd \inf_{x N^{-2/3}\in [-U,U]}\overline L^N_{k-1}(x) > -\tfrac{1}{8}RN^{1/3}\right)\\
&\qquad + \P\left(\msf{Low}_k^{[-U, -T]}(\tfrac{1}{4}R)\cup \msf{Low}_k^{[T,U]}(\tfrac{1}{4}R)\midd \inf_{xN^{-2/3}\in[-U,U]}\overline L^N_{k-1}(x) > -\tfrac{1}{8}RN^{1/3}\right)\\
&\leq \exp(-cR^2/(32U))  + \tfrac{1}{3}\varepsilon \leq \varepsilon,
\end{align*}
the penultimate inequality using Lemma~\ref{l.midpoint lower tail} (with the parameters there set as $a=-U$, $b=U$, $t=t$, $R$ and $K$ there both taken to be $R/4$ as here, and $\delta=R/8$ with $R$ as here) and \eqref{e.bound on P(low)}. Using this with $t=-T,0,T$, along with a union bound and a relabeling of $\varepsilon$, yields
\begin{align}\label{e.E_0 prob bound}
\P\left(\min_{xN^{-2/3}\in\{-T,0,T\}} \overline L^N_k(x) < -\tfrac{1}{2}RN^{1/3}\midd \inf_{xN^{-2/3}\in[-U,U]}\overline L^N_{k-1}(x) > -\tfrac{1}{8}RN^{1/3}\right) < \varepsilon.
\end{align}

Now we set up the chaining argument to obtain a bound on the lower tail of the infimum of $\overline L^N_k$. For any integer $\ell\geq 0$, let $S_\ell := \{\floor{T(-1+ i2^{-\ell})N^{2/3}}: i=0,1, \ldots, 2^{\ell+1}\}$ (we will omit the floors in the notation for convenience in the rest of the proof) and define
\begin{align*}
\msf E_\ell := \left\{\inf_{x\in S_\ell} \overline L^N_k(x) > -\alpha_\ell RN^{1/3}\right\},
\end{align*}
where
\begin{align*}
\alpha_\ell := \frac{1}{2}+ \delta\sum_{j=1}^\ell 2^{-j/4},
\end{align*}
with $\delta = \frac{1}{4}(\sum_{j=1}^\infty 2^{-j/4})^{-1}>0$, so that $\frac{1}{2}<\alpha_\ell < \frac{3}{4}$ for all $\ell\geq 0$. It is easy to see that
\begin{align*}
\left\{\inf_{xN^{-2/3}\in[-T,T]} \overline L^N_k(x) < -RN^{1/3}\right\} \subseteq \bigcup_{\ell=0}^{\ceil{\frac{1}{2}\log_2 N}} (\msf E_\ell)^c.
\end{align*}
Indeed, at $\ell = \frac{1}{2}\log_2 N$, the spacing between consecutive points in $S_\ell$ becomes equal to $TN^{2/3 - 1/2} = TN^{1/6}$. Since $L^N_k$ is a Bernoulli path and since $\alpha_\ell \leq \frac{3}{4}$, this implies that, on the complement of the righthand side of the previous display, for $R$ and $N$ sufficiently large we have
\begin{align*}
\inf_{xN^{-2/3}\in [-T, T]} \overline L^N_k(x) \geq \inf_{x\in\cup_{\ell=0}^{2^{-1}\log_2 N}S_\ell} \overline L^N_k(x) - TN^{1/6} &> -\tfrac{3}{4}RN^{1/3} - TN^{1/6}\\
&> -RN^{1/3}.
\end{align*}
Further, it holds that (for convenience defining $\msf E_{-1} = \Omega$, i.e., the entire space)
\begin{align*}
\bigcup_{\ell=0}^{\ceil{\frac{1}{2}\log_2 N}} \msf E_\ell^c = \bigcup_{\ell=0}^{\ceil{\frac{1}{2}\log_2 N}} (\msf E_{\ell-1}\cap \msf E_\ell^c).
\end{align*}
Thus, since we have
\begin{align*}
\MoveEqLeft[2]
\P\left(\inf_{xN^{-2/3}\in[-T,T]} \overline L^N_k(x) < -RN^{1/3}\right)\\
&\leq \P\left(\inf_{xN^{-2/3}\in[-T,T]} \overline L^N_k(x) < -RN^{1/3} \midd \inf_{sN^{-2/3}\in[-U,U]} \overline L^N_{k-1}(s) > -\tfrac{1}{8}RN^{1/3}\right)\\
&\qquad+ \P\left(\inf_{sN^{-2/3}\in[-U,U]} \overline L^N_{k-1}(s)  \leq -\tfrac{1}{8}RN^{1/3}\right)
 \end{align*}
and the second term is bounded by $\frac{1}{3}\varepsilon$ by \eqref{e.k-1 curve lower tail bound induction step in proof},
 it is enough to bound
\begin{align*}
\sum_{\ell=0}^{\frac{1}{2}\log_2 N}\P\left(\msf E_{\ell-1}\cap \msf E_\ell^c \midd \inf_{xN^{-2/3}\in[-U, U]}\overline L^N_{k-1}(x) > -\tfrac{1}{8}RN^{1/3}\right).
\end{align*}

We have bounded $\P(\msf E_0^c \mid \inf_{xN^{-2/3}\in[-U, U]} \overline L^N_{k-1}(x) > -\tfrac{1}{8}RN^{1/3})$ by $\varepsilon$ above in \eqref{e.E_0 prob bound}. For the $\ell \geq 1$ terms, we observe that Lemma~\ref{l.midpoint lower tail} (with $R$ there equal to $\alpha_{\ell-1}R$, $K=(\alpha_\ell-\alpha_{\ell-1})R$, and $\delta = \alpha_{\ell-1}-\frac{1}{8}>0$ uniformly over $\ell$) implies that there exists $N_0$ independent of $\ell$ such that, for $N>N_0$, the $\ell$\smash{\th} term is upper bounded (using a union bound over the $2^{\ell+1}$ points in $S_\ell$ and noting that $\alpha_\ell - \alpha_{\ell-1} = \delta 2^{-\ell/4}$ and the separation between points in $S_{\ell}$ is $2^{-\ell}$) by
\begin{align*}
2^{\ell+1}\cdot C\exp\left(-c(\delta2^{-\ell/4}R)^2/2^{-\ell}\right) = 2^{\ell+1}\cdot C\exp(-c\delta^2 2^{\ell/2}R^2).
\end{align*}
Summing this from $\ell=1$ to $\infty$ yields a bound of $C\exp(-c\delta^2R^2)$, which can be made smaller than $\varepsilon$ by raising $R$ further if required. Relabeling $\varepsilon$ completes the proof.
\end{proof}

\subsection{The $k$\textsuperscript{th} curve cannot be uniformly too low on an interval} \label{s.kth curve not uniformly low}

\begin{proof}[Proof of Lemma~\ref{l.low event bound}]
The idea of the proof is the following. Under the event $\msf{Low}^{[T,U]}_k(R_{\mrm{low}})$ for a large enough interval $[T,U]$,  it will be likely that the top $k-1$ curves fall, as the $k$\th curve that typically ``supports'' them is too far down. In particular, under this conditioning, the 1\smash\st curve will with reasonable probability be at a low position around say the midpoint of $[T,U]$. This event's (unconditional) probability is known to be small using lower tail control from Assumption~\ref{as.one-point tightness}. Since the event we conditioned on turned an unlikely event into a likely one, it must have a low probability. The basic proof idea originates in \cite{corwin2014brownian}, but is made more complicated here by the lack of stochastic monotonicity of the Hall-Littlewood Gibbs property (though we actually will not need to invoke the weak monotonicity of Lemma~\ref{l.partition function comparison with non-int boundaries}). Now we give the details.

We let $U$ be a variable whose value will be set at the end of the proof. For $R_1$ a constant to be chosen next, let
\begin{align*}
\msf{NotHigh}_1(x, R_1) := \left\{\overline L^N_{1}(xN^{2/3}) + \lambda x^2N^{1/3} \leq R_1 N^{1/3}\right\},
\end{align*}
and define, for $R_{\mrm{low}}$ to be set,
\begin{equation}\label{e.A definition}
\msf A^{[T,U]}_k(R_1) := \msf{NotHigh}_{1}(T, R_1) \cap \msf{NotHigh}_{1}(U, R_1) \cap \msf{Low}^{[T,U]}_k(R_{\mrm{low}}).
\end{equation}
Now by Assumption~\ref{as.one-point tightness}, there exist $R_1 = R_1(\varepsilon)$ and $N_0= N_0(T,U,\varepsilon)$ large enough such that, for all $U$ and for $N\geq N_0$,
\begin{align}\label{e.not high prob bound}
\P\Bigl(\msf{NotHigh}_{1}(T, R_1)^c\cup \msf{NotHigh}_{1}(U, R_1)^c\Bigr) \leq \tfrac{1}{2}\varepsilon;
\end{align}
note that $R_1$ can be set independently of $U$. We set $R_1$ to be such a value for the remainder of the proof.

Next let $\bm B = (B_1, \ldots, B_{k-1})$ be a collection of $k-1$ independent Bernoulli random walk bridges, the $i$\textsuperscript{th} one having endpoints
\begin{align}\label{e.not low B endpoints}
u_i &= (TN^{2/3}, L^N_i(TN^{2/3})) \qquad \text{and} \qquad
v_i = (UN^{2/3}, L^N_i(UN^{2/3})).
\end{align}
Note that these points are all upper bounded by $pTN^{2/3} - \lambda T^2N^{1/3} + R_1N^{1/3}$ and $pUN^{2/3} - \lambda U^2N^{1/3} + R_1N^{1/3}$, respectively, on $\msf{NotHigh}_{1}(T, R_1)\cap \msf{NotHigh}_{1}(U, R_1)$.

Let $\msf{Sep}^N_{k-2}(\delta,\{T,U\}, \bm L^N)$ be as in \eqref{e.Sep definition}.
By hypothesis we know there exists $\delta = \delta(T, U, \varepsilon, k)$ such that
\begin{align}\label{e.not too low sep}
\P\left(\msf{Sep}^N_{k-2}(\delta,\{T,U\}, \bm L^N)^c\right) \leq \tfrac{1}{2}\varepsilon.
\end{align}
Also let $R_{k-1} = R_{k-1}(T, U, \varepsilon, k)$ be such that, with
$$\msf{BdyCtrl}_{k-1}(R_{k-1}) := \left\{\min_{xN^{-2/3}\in\{T, U\}} \overline L^N_{k-1}(x) \geq -R_{k-1}N^{1/3}\right\},$$
it holds that
\begin{align}\label{e.not too low bdyctrl}
\P\left(\msf{BdyCtrl}_{k-1}(R_{k-1})^c\right) \leq \tfrac{1}{2}\varepsilon,
\end{align}
which is possible by the hypothesis of Lemma~\ref{l.low event bound} that we are working under.

We now set $R_{\mrm{low}} = R_{k-1} + U$ in the definition \eqref{e.A definition} of $\msf A_k^{[T,U]}(R_1)$. Also let $C_N(T,U)$ be the midpoint of the line joining the upper bounds on $u_1$ and $v_1$ mentioned above; more precisely, it is given by
\begin{align}\label{e.C(r,R) definition}
C_N(T,U) &:= p\cdot\tfrac{1}{2}(T+U)N^{2/3}-\tfrac{1}{2}\lambda (T^2 + U^2)N^{1/3} + R_1N^{1/3}\\
&= p\cdot\tfrac{1}{2}(T+U)N^{2/3} - \tfrac{1}{4}\lambda (T+U)^2N^{1/3} - \tfrac{1}{4} \lambda (U-T)^2N^{1/3} + R_1N^{1/3}.\nonumber
\end{align}

Let $\F=\Fext(k-1, \intint{TN^{2/3}, UN^{2/3}}, \bm L^N)$ be as in Definition~\ref{d.F_ext}, i.e., the $\sigma$-algebra generated by $\{L^N_i(x):(i,x)\not\in\intint{1,k-1}\times\intint{TN^{2/3}+1, UN^{2/3}-1}\}$. Note that $\msf A_k^{[T,U]}(R_1)$, $\msf{BdyCtrl}_{k-1}(R_{k-1})$, and $\msf{Sep}^N_{k-2}(\delta) := \msf{Sep}^N_{k-2}(\delta,\{T,U\}, \bm L^N)$ are $\F$-measurable. For a process $X:\R\to\R$, define the event $\msf{UpTail}(U, X)$ by
\begin{align}\label{e.uptail definition}
 \msf{UpTail}(U, X) := \left\{X(\tfrac{1}{2}(T+U)N^{2/3}) \geq C_N(T,U) +\tfrac{1}{8}\lambda  (U-T)^2N^{1/3}\right\},
\end{align}
and recall the notational convention from Remark~\ref{r.event name convention}.
Then, by the previous bounds \eqref{e.not too low sep} and \eqref{e.not too low bdyctrl} on $\P(\msf{Sep}^N_{k-2}(\delta)^c)$ and $\P(\msf{BdyCtrl}_{k-1}(R_{k-1})^c)$, respectively, 
\begin{align}
\MoveEqLeft[4]
\P\left(\msf{UpTail}(U, L^N_1)\mid \msf A_k^{[T,U]}(R_1)\right)\nonumber\\
&\leq \P\left(\msf{UpTail}(U, L^N_1), \msf{Sep}^N_{k-2}(\delta), \msf{BdyCtrl}_{k-1}(R_{k-1})\mid \msf A_k^{[T,U]}(R_1)\right)\nonumber\\
&\qquad+ \frac{\varepsilon}{\P(\msf A_k^{[T,U]}(R_1))};\nonumber
%
\intertext{next, by the conditional probability definition \eqref{e.conditional probability definition} (and recalling that the events $\msf A_k^{[T,U]}(R_1)$, $\msf{Sep}^N_{k-2}(\delta)$, $\msf{BdyCtrl}_{k-1}(R_{k-1}) \in\F$), the previous display equals}
%
&= \E\left[\PF\left(\msf{UpTail}(U, L^N_1)\right)\one_{\msf{Sep}^N_{k-2}(\delta), \msf{BdyCtrl}_{k-1}(R_{k-1})}\midd \msf A_k^{[T,U]}(R_1)\right] \nonumber\\
&\qquad + \frac{\varepsilon}{\P(\msf A_k^{[T,U]}(R_1))}\nonumber,\\
&= \E\left[\frac{\EF\left[\one_{\msf{UpTail}(U, B_1)} W(\bm B, L^N_k)\right]}{\EF\left[W(\bm B, L^N_k)\right]}\one_{\msf{Sep}^N_{k-2}(\delta), \msf{BdyCtrl}_{k-1}(R_{k-1})}\midd \msf A_k^{[T,U]}(R_1)\right] \label{e.midpoint line bound}\\
&\qquad + \frac{\varepsilon}{\P(\msf A_k^{[T,U]}(R_1))}\nonumber,
\end{align}
where, in the final inequality, we applied the Gibbs property, recalling from \eqref{e.not low B endpoints} that $\bm B= (B_1, \ldots ,B_{k-1})$ is a family of Bernoulli random walk bridges on $[TN^{2/3}, UN^{2/3}]$, with $B_i$ starting and ending at $L^N_i(TN^{2/3})$ and $L^N_i(UN^{2/3})$, respectively, for $i\in\intint{1, k-1}$, and $W$ is as in Definition~\ref{d.HL Gibbs}.

For an interval $I$, define the event
\begin{equation}\label{e.nonint definition}
\msf{NonInt}_{k-2}(I, \bm B) := \Bigl\{B_i(x)\geq B_{i+1}(x)\text{ for all } i\in\intint{1, k-2}, x\in I\Bigr\}.
\end{equation}
In the rest of the proof we adopt the shorthand 
$$\msf{NonInt}_{k-2}(\bm B) = \msf{NonInt}_{k-2}([TN^{2/3},UN^{2/3}], \bm B).$$
Since $W(\bm B) = 0$ on $\msf{NonInt}_{k-2}(\bm B)^c$, it follows that
\begin{align}\label{e.to bound in not low argument}
\frac{\EF\left[\one_{\msf{UpTail}(U, B_1)} W(\bm B)\right]}{\EF\left[W(\bm B)\right]} = \frac{\EF\left[\one_{\msf{UpTail}(U, B_1)} W(\bm B)\mid \msf{NonInt}_{k-2}(\bm B)\right]}{\EF\left[W(\bm B)\mid \msf{NonInt}_{k-2}(\bm B)\right]}.
\end{align}

\medskip

\noindent\textbf{Controlling the denominator.} Next, we lower bound the denominator of \eqref{e.to bound in not low argument}, which is a $\F$-measurable random variable. We do so on the $\F$-measurable event $\msf{Sep}^N_{k-2}(\delta) \cap \msf{BdyCtrl}_{k-1}(R_{k-1})\cap \msf A^{[T,U]}_k(R_1)$, the final event being the one we have conditioned on in \eqref{e.midpoint line bound}. 

First, we claim that there exists $\rho = \rho(\delta, T, U, k)>0$ such that, on $\msf{Sep}^N_{k-2}(\delta)$,
\begin{align}\label{e.independent RW bridge separation lower bound}
\PF\left(\min_{i=1, \ldots, k-2}\inf_{xN^{-2/3} \in [T,U]} \left(B_i(x) - B_{i+1}(x)\right)\geq \rho N^{1/3}\midd \msf{NonInt}_{k-2}(\bm B)\right)\geq \frac{1}{2},
\end{align}
recalling the meaning of the lefthand side from \eqref{e.conditional expectation with conditioning definition}.

To see this, we first let $\alpha = \alpha(T,U,\delta, k)>0$ be deterministic and such that $\PF(\msf{NonInt}_{k-2}(B)) \geq \alpha$ on $\msf{Sep}^N_{k-2}(\delta)$. The existence of $\alpha$ uses that the separation of $B_i$ from $B_{i+1}$ at the endpoints of $T N^{2/3}$ and $U N^{2/3}$ is at least $\delta N^{1/3}$ on $\msf{Sep}^N_{k-2}(\delta)$, and that $\msf{NonInt}_{k-2}(\bm B)$ holds if the fluctuation of each bridge is at most $\frac{1}{4}\delta N^{1/3}$ on the interval of size $(U-T)N^{2/3}$, which has uniformly (in $N$) positive probability, e.g., by convergence of (rescaled) $\bm B$ to independent Brownian bridges with endpoint separation at least $\delta$ (see e.g., \cite[Lemma 5.2]{dimitrov2021characterization}), since this weak convergence holds almost surely on $\msf{Sep}^N_{k-2}(\delta)$. Then, by a union bound and the independence of $B_i$ we obtain
\begin{align*}
\MoveEqLeft[6]
\PF\left(\min_{i=1, \ldots, k-2}\inf_{xN^{-2/3} \in [T,U]} \left(B_i(x) - B_{i+1}(x)\right)< \rho N^{1/3}\midd \msf{NonInt}_{k-2}(\bm B)\right)\\
&\leq \sum_{i=1}^{k-2} \PF\left(\inf_{xN^{-2/3} \in [T,U]} \left(B_i(x) - B_{i+1}(x)\right)< \rho N^{1/3}\midd \msf{NonInt}_{k-2}(\bm B)\right)\\
&\leq \alpha^{-1}\sum_{i=1}^{k-2} \PF\left(0\leq \inf_{xN^{-2/3} \in [T,U]} \left(B_i(x) - B_{i+1}(x)\right)< \rho N^{1/3}\right).
\end{align*}
This can be made smaller than $\frac{1}{2}$ for all large enough $N$ by setting $\rho>0$ small enough, depending on $\delta$, $T$, $U$, and $k$ (note that $\alpha$ depends on the same parameters); this can again be seen by the almost sure weak convergence of a rescaled version of $\bm B$ to a collection of independent Brownian bridges with endpoint separation at least $\delta$, and noting that the corresponding Brownian bridge probability can be made arbitrarily small by setting $\rho$ correspondingly. Thus \eqref{e.independent RW bridge separation lower bound} has been proved.

Recall from Definition~\ref{d.HL Gibbs} (in particular \eqref{e.W up int low split up} and \eqref{e.W up int low definition}) the decomposition $W = W_{\mrm{int}}\cdot W_{\mrm{low}}$ with $W_{\mrm{low}} = W_{\mrm{low}}(B_{k-1}, L^N_{k})$ being the weight associated to the lower boundary condition $L^N_{k}$ and $W_{\mrm{int}} = W_{\mrm{int}}(\bm B)$ being the weight associated to the interior interactions between the $B_i$. By Lemma~\ref{l.W bound}, on the event $\msf E_1:= \{\min_{i=1, \ldots, k-2}\inf_{xN^{-2/3} \in [T,U]} (B_i(x) - B_{i+1}(x))\geq \rho N^{1/3}\}$ in \eqref{e.independent RW bridge separation lower bound}, 
\begin{align}\label{e.not low W_int bound}
W_{\mrm{int}}(\bm B) \geq (1-q^{\rho N^{1/3}})^{(U-T)N^{2/3}\cdot(k-1)} \geq \tfrac{1}{2}
\end{align}
for large enough $N$ (depending on $\rho$, $U-T$, and $k$, therefore on $T$, $U$, $\varepsilon$, and $k$). %

Now we turn to lower bounding $W_{\mrm{low}}(B_{k-1}, L^N_k)$. Observe that, on $\msf A_k^{[T,U]}(R_1)\cap\msf{BdyCtrl}_{k-1}(R_{k-1})$, and with $\overline B_{k-1}(x) = B_{k-1}(x) - px$,
\begin{align*}
\min\left(\overline B_{k-1}(TN^{2/3}), \overline B_{k-1}(UN^{2/3})\right) - \sup_{xN^{-2/3}\in [T,U]} \overline L^N_k
&=(-R_{k-1} + R_{\mrm{low}})N^{1/3}\\
&\geq UN^{1/3},
\end{align*}
where  recall $R_{k-1}$ depends on $U$ and $R_{\mrm{low}} = R_{k-1}+U$ was set earlier.
Since  $\bm B$ almost surely converges weakly to non-intersecting Brownian bridges with endpoint separation at least $\delta$ on $\msf{Sep}^N_{k-2}(\delta)$ and $\frac{1}{2}U \gg U^{1/2}$, it follows that, on the $\F$-measurable event $\smash{\msf A_k^{[T,U]}(R_1)\cap \msf{Sep}^N_{k-2}(\delta)\cap\msf{BdyCtrl}_{k-1}(R_{k-1})}$, for large enough $N$ (depending on $U$, $T$, $\delta$),
\begin{align}\label{e.not low E_2 bound}
\PF\left(\inf_{[TN^{2/3} UN^{2/3}]} \left(B_{k-1} - L^N_k\right) \geq \tfrac{1}{2}UN^{1/3} \midd \msf{NonInt}_{k-2}(\bm B)\right) \geq \tfrac{3}{4}.
\end{align}
It also follows from Lemma~\ref{l.W bound} that, on the event $\msf E_2 := \{\inf_{[TN^{2/3} UN^{2/3}]} (B_{k-1} - L^N_k) \geq \tfrac{1}{2}UN^{1/3}\}$ in the previous display,
\begin{equation}\label{e.not low W_low bound}
W_{\mrm{low}}(B_{k-1}, L^N_k) \geq (1-q^{\frac{1}{2}UN^{1/3}})^{(U-T)N^{2/3}} \geq \tfrac{1}{2}
\end{equation}
for large enough $N$ (depending on $U$ and $T$).

Now we put the above together. Recall the bounds \eqref{e.not low W_int bound} and \eqref{e.not low W_low bound} on $W_{\mrm{int}}(\bm B)$ and $W_{\mrm{low}}(B_{k-1}, L^N_k)$ which hold on the events $\msf E_1$ and $\msf E_2$, whose $\F$-conditional probabilities were in turn lower bounded in \eqref{e.independent RW bridge separation lower bound} and \eqref{e.not low E_2 bound} on the $\F$-measurable events $\msf A_k^{[T,U]}(R_1)\cap\msf{Sep}^N_{k-2}(\delta) \cap \msf{BdyCtrl}_{k-1}(R_{k-1})$. Using these bounds along with $W(\bm B) = W_{\mrm{int}}(\bm B)\cdot W_{\mrm{low}}(B_{k-1}, L^N_k)$ from \eqref{e.W up int low split up}, we see that, on $\msf A_k^{[T,U]}(R_1)\cap\msf{Sep}^N_{k-2}(\delta) \cap \msf{BdyCtrl}_{k-1}(R_{k-1})$,
\begin{align*}
\EF[W(\bm B, L^N_k)\mid \msf{NonInt}_{k-2}(\bm B)]
&\geq \EF[W(\bm B, L^N_k)\one_{\msf E_1\cap\msf E_2}\mid \msf{NonInt}_{k-2}(\bm B)]\\
&\geq \tfrac{1}{2}\cdot\tfrac{1}{2}\cdot \PF\left(\msf E_1\cap \msf E_2 \mid \msf{NonInt}_{k-2}(\bm B) \right)\\
&\geq \tfrac{1}{16}.
\end{align*}

%

\smallskip

\noindent\textbf{Bounding the conditional random walk bridge probability.} Thus, using $W(\bm B,L^N_k)\leq 1$ along with \eqref{e.to bound in not low argument}, we see that the first term in \eqref{e.midpoint line bound} is upper bounded by
\begin{align}\label{e.original RWB upper tail prob}
16\cdot \E\left[\PF\bigl(\msf{UpTail}(U, B_1)\mid \msf{NonInt}_{k-2}(\bm B)\bigr) \midd \msf A_k^{[T,U]}(R_1)\right].
\end{align}
Recall from \eqref{e.not low B endpoints} that $B_i$ is a Bernoulli random walk bridge from
$$(TN^{2/3}, L^N_i(TN^{2/3})) \text{ to } (UN^{2/3}, L^N_i(UN^{2/3}))$$
and that, on $\msf A_k^{[T,U]}(R_1)$, $\overline L^N_1(TN^{2/3}) \leq (-\lambda T^2 + R_1)N^{1/3}$ and $\overline L^N_1(UN^{2/3}) \leq (-\lambda U^2 + R_1)N^{1/3}$. Let $\tilde{\bm B} = (\tilde B_1, \ldots, \tilde B_{k-1})$ be a collection of $k-1$ Bernoulli random walk bridges with $\tilde B_i$ from
\begin{align*}
&\left(TN^{2/3}, pTN^{2/3} + (-\lambda T^2 + R_1 + (k-i) U^{1/2})N^{1/3}\right) \quad\text{to}\\
&\left(UN^{2/3}, pUN^{2/3} + (-\lambda U^2 + R_1 + (k-i) U^{1/2})N^{1/3}\right).
\end{align*}
Since these endpoints are higher than the corresponding endpoints of $\bm B$ (which equal those of $\bm L^N$ as recorded just above, so are below those of $\overline L^N_1$, which in turn are below $(-\lambda T^2 + R_1)N^{1/3}$ on $\msf A^{[T,U]}_k(R_1)$), it holds that the law of $\tilde{\bm B}$ conditionally on $\msf{NonInt}_{k-2}(\tilde{\bm B})$ stochastically dominates that of $\bm B$ conditional on $\msf{NonInt}_{k-2}(\bm B)$ by Lemma~\ref{l.random walk bridge monotonicity}. Thus it follows that \eqref{e.original RWB upper tail prob} is upper bounded by (dropping the $\F$ and the outer expectation since $\tilde{\bm B}$ does not have any dependence on data from $\F$)
\begin{align*}
16\cdot \P\left(\msf{UpTail}(U, \tilde B_1)\midd \msf{NonInt}_{k-2}(\tilde{\bm B})\right).
\end{align*}

Observe that there is an absolute constant $c>0$ such that $\P(\msf{NonInt}_{k-2}(\tilde{\bm B})) > c$ for all large enough $U$, since $\tilde B_i$ are separated by order $U^{1/2}$ at their endpoints and the non-intersection is imposed on an interval of length $U-T$. Thus the previous display is upper bounded by
\begin{align}\label{e.not low intermediate step}
16c^{-1}\cdot \P\left(\msf{UpTail}(U, \tilde B_1)\right).
\end{align}
Now we recall the definitions \eqref{e.C(r,R) definition} of $C_N(T,U)$ as the midpoint of the line joining the points $(TN^{2/3}, pTN^{2/3}-\lambda T^2N^{1/3}+R_1N^{1/3})$ and $(UN^{2/3}, pUN^{2/3}-\lambda U^2N^{1/3}+R_1N^{1/3})$. Since $\tilde B_1(t) = ptN^{2/3}  + (- \lambda t^2 + R_1 + (k-1)U^{1/2})N^{1/3}$ for $t\in\{T,U\}$, it follows that
\begin{align*}
\E\left[\tilde B_1(\tfrac{1}{2}(T+U)N^{2/3})\right] \leq C_N(T,U)+ (k-1)U^{1/2}N^{1/3}.
\end{align*}
Thus, recalling the definition \eqref{e.uptail definition} of $\msf{UpTail}$, the quantity \eqref{e.not low intermediate step} is upper bounded by
\begin{align*}
\MoveEqLeft[0]
16c^{-1}\cdot \P\left(\tilde B_1\!\left(\tfrac{T+U}{2}N^{2/3}\right) \geq \E\left[\tilde B_1\!\left(\tfrac{T+U}{2}N^{2/3}\right)\right] - kU^{1/2}N^{1/3} +\tfrac{1}{8} (U-T)^2N^{1/3}\right)\\
&\qquad\leq \varepsilon
\end{align*}
for large enough $U$, the inequality using upper bounds on Bernoulli random walk bridge one-point tails (Lemma~\ref{l.random walk bridge fluctuation}).

\medskip

\noindent\textbf{Returning to $\bm{L^N_1}$.} 
Overall, returning to \eqref{e.midpoint line bound} and \eqref{e.to bound in not low argument}, we have shown that
\begin{align*}
\P\left(\msf{UpTail}(U,L^N_1) \midd \msf A_k^{[T,U]}(R_1)\right) \leq \varepsilon + \frac{\varepsilon}{\P\left(\msf A_k^{[T,U]}(R_1)\right)}.
\end{align*}
We claim that $\P(\msf A_k^{[T,U]}(R_1)) < 3\varepsilon$. Otherwise, we would have that $\smash{\P(\msf A_k^{[T,U]}(R_1))}\geq 3\varepsilon$, which would imply that $\smash{\P(\msf{UpTail}(U,L^N_1)\mid \msf A_k^{[T,U]}(R_1)) \leq \varepsilon + \frac{1}{3} \leq \frac{1}{2}}$ for large enough $U$. This would further imply (recalling the definition \eqref{e.uptail definition} of $\msf{UpTail}$ and that $\overline L^N_1(x) = L^N_1(x) - px$) that
\begin{align*}
\MoveEqLeft[1]
\P\left(\msf A_k^{[T,U]}(R_1)\right)\\
&\leq \left(\P\left(\msf{UpTail}(U,L^N_1)^{c}\mid \msf A_k^{[T,U]}(R_1)\right)\right)^{-1}\cdot \P\left(\msf{UpTail}(U,L^N_1)^c\right)\\
&\leq \left(1-\tfrac{1}{2}\right)^{-1} \P\left(\msf{UpTail}(U,L^N_1\right)^c)\\
&= 2\cdot\P\left(\overline L^N_{1}(\tfrac{1}{2}(T+U)N^{2/3}) < \left(-\tfrac{1}{4}\lambda(T+U)^2 - \tfrac{1}{8}\lambda(U-T)^2 + R_1\right)N^{1/3}\right).
\end{align*}
By the one-point tightness of $L^N_1$ (Assumption~\ref{as.one-point tightness}), for large enough $U$ (depending on $R_1$ and $T$) and all $N>N_0$ (for a $N_0$ depending on $U$), this is upper bounded by $\varepsilon$, which is a contradiction; recall here the important observation made after \eqref{e.not high prob bound} that $R_1$ was chosen independently of $U$. This establishes that $\P(\msf A_k^{[T,U]}(R_1)) < 3\varepsilon$.

With these choices of $R_1$, $R_{\mrm{low}}$, $U$, and $N_0$, and recalling the definition \eqref{e.A definition} of $\msf A^{[T,U]}_k(R_1)$ and the bound \eqref{e.not high prob bound} on the probability of $\msf{NotHigh}_1^c$, we see that, for $N\geq N_0$,
\begin{align*}
\P\left(\msf{Low}^{[T,U]}_k(R_{\mrm{low}})\right) &\leq \P\left(\msf A_k^{[T,U]}(R_1)\right) + \P\Bigl(\msf{NotHigh}_1(T,R_1)^c\cup\msf{NotHight}(U,R_1)^c\Bigr)\\
&\leq 4\varepsilon.
\end{align*}
Relabeling $\varepsilon$ completes the proof.
\end{proof}

\subsection{The $k$\textsuperscript{th} curve cannot be too low at a given point} \label{s.kth curve not too low at a point}

Here we prove Lemma~\ref{l.midpoint lower tail}. We will need a version of the Hall-Littlewood Gibbs property on a random interval. As we have seen, random intervals defined by information outside the interval satisfy a strong Gibbs property (Definition~\ref{d.stopping domain} and Proposition~\ref{p.strong gibbs}), analogous to a strong Markov property, in which case the behavior inside the interval under the resampling is the same. Here, however, we will need to resample on a random interval defined in a way that is dependent on the behavior of the process inside, and the description of the resampling is recorded here.

Let $a,b\in \Z$ be integers, let $k\in\N$ be a natural number, let $\Lambda\subseteq \Z$ be an interval containing $\intint{a,b}$, and let $\bm L:\N\times\Lambda\to\Z$ be a discrete line ensemble (Definition~\ref{d.discrete line ensemble}) with the Hall-Littlewood Gibbs property (Assumption~\ref{as.HL Gibbs}). Let $p\in(-1,0)$ and recall the notation $\overline L_k(x) = L_k(x) - px$ (note that we do not assume $\bm L$ satisfies Assumption~\ref{as.one-point tightness}, so $p$ need not be as in that assumption). Let $m\in\R$ and $t\in\intint{a,b}$, and define
\begin{equation}\label{e.a',b' definition}
\begin{split}
a' &:= \max\bigl\{x\in\intint{a, t} : \overline L_k(x) \geq m\bigr\} \quad\text{and}\quad
b' := \min\bigl\{x\in\intint{t, b} : \overline L_k(x) \geq m\bigr\}.
\end{split}
\end{equation}
Next define the $\sigma$-algebra
\begin{align*}
\MoveEqLeft[2]
\Fext(\{k\}, \intint{a',b'}, \bm L)\\
&:= \Bigl\{A\cap\{a'=x, b'=y\} : A\in\Fext(\{k\}, \intint{x,y}, \bm L), x\in\intint{a, t}, y \in \intint{t, b}\Bigr\}.
\end{align*}
In words, this $\sigma$-algebra contains the information of the process outside the random interval $\intint{a', b'}$.

For any $c,d$, let $\E^{x, y, \intint{c,d}}$ be the expectation operator associated to a probability measure $\P^{x, y, \intint{c,d}}$ under which $B$ is distributed as a Bernoulli random walk bridge from $(c, x)$ to $(d,y)$, and let $\overline B(x) := B(x) - px$.
Finally, let $\Omega = \{(\ell, c,d): c\in \intint{a,t}, b\in\intint{t,d}, \ell\in \Omega_{\intint{c,d}}\}$, where $\Omega_{\intint{c,d}} := \{f:\intint{c,d}\to\Z, f \text{ a Bernoulli path}\}$ is the space of Bernoulli paths (Definition~\ref{d.bernoulli path}) on~$\intint{c,d}$.
We will give the proof of the following lemma in Appendix~\ref{s.random gibbs}.

\begin{lemma}\label{l.Gibbs on random interval}
Let $k\in\N$, let $\Lambda\subseteq \Z$ be an interval, and let $\bm L:\N\times\Lambda\to \Z$ be a discrete line ensemble satisfying Assumption~\ref{as.HL Gibbs}, and let $p\in(-1,0)$. Let $m\in\R, \intint{a, b}\subseteq \Z,$ and $t\in\intint{a,b}$, and let $a'$ and $b'$ be as in \eqref{e.a',b' definition}. Let $\F = \Fext(\{k\}, \intint{a',b'}, \bm L)$. For any measurable function $F: \Omega\to \R$, it holds that
\begin{align*}
\MoveEqLeft[0]
\E_{\F}\left[F(L_k|_{\intint{a', b'}}, a', b')\right]\\
&= Z^{-1}\cdot\E^{z_1, z_2, \intint{a',b'}}\left[F(B, a', b')W(B,L_{k-1}, L_{k+1}, \intint{a',b'}) \midd\! \max_{s\in\intint{a'+1,b'-1}}\overline B(s) < m\right],
\end{align*}
where $z_1 = L_k(a')$, $z_2=L_k(b')$, and $Z=\E^{z_1, z_2, \intint{a',b'}}[W(B,L_{k-1}, L_{k+1}, \intint{a',b'}) \mid \max_{s\in\intint{a'+1,b'-1}}\overline B(s) < m]$.
\end{lemma}

In words, Lemma~\ref{l.Gibbs on random interval} says that the distribution of $L_k$ on the interval $\intint{a',b'}$ (i.e., between the last location before $t$ and first location after $t$ where $\overline L_k \geq m$) is given by a Bernoulli random walk bridge $B$ tilted by the weight factor $W$ from \eqref{e.weight function} as in the usual Hall-Littlewood Gibbs property but with the additional conditioning that $\overline B$ remains strictly below $m$ on this interval. This can be thought of as analogous to the fact that Brownian motion considered on the random interval formed by its last and first hitting times before and after time 1 is a (rescaled) Brownian excursion.

\begin{proof}[Proof of Lemma~\ref{l.midpoint lower tail}]
For notational convenience, we prove the statement with $2K$ in place of $K$. On the event
$$\msf E(t,a,b) := \left\{\parbox{4.5in}{\centering$\overline L^N_k(tN^{2/3}) \leq -(R+2K)N^{1/3},$\\ $\Bigl(\sup_{xN^{-2/3}\in[a, t]} \overline L^N_k(x)\Bigr)\wedge\Bigl(\sup_{xN^{-2/3}\in [t, b]} \overline L^N_k(x)\Bigr) \geq -RN^{1/3}$}\right\},$$
define the random times $a'$ and $b'$ by
\begin{align*}
a' &:= N^{-2/3}\cdot\max\left\{x\in \intint{aN^{2/3}, tN^{2/3}}: \overline L^N_k(x) \geq -(R+K)N^{1/3}\right\} \quad\text{and}\\
b' &:= N^{-2/3}\cdot\min\left\{x\in \intint{tN^{2/3}, bN^{2/3}}: \overline L^N_k(x) \geq -(R+K)N^{1/3}\right\}.
\end{align*}
For notational convenience we also define, for a process $X:[a'N^{2/3}, b'N^{2/3}]\to \R$,
\begin{align}\label{e.low definition for lower tail at midpoint}
\msf{Low}^t(X) := \left\{X(tN^{2/3}) \leq -(R+2K)N^{1/3}\right\}.
\end{align}

Since $L^N_k$ is a Bernoulli path, it follows that $\overline L^N_k(a'N^{2/3}), \overline L^N_k(b'N^{2/3}) = \ceil{-(R+K)N^{1/3}}$ on $\msf E(t, a,b)$; in the remainder of the proof, we ignore the ceiling and assume that $\smash{\overline L^N_k(a'N^{2/3})} = \overline L^N_k(b'N^{2/3}) = -(R+K)N^{1/3}$.

Let $B$ be a Bernoulli random walk bridge from
\begin{align*}
\bigl(a'N^{2/3}, pa'N^{2/3}-(R+K)N^{1/3}\bigr) \quad\text{to}\quad
\bigl(b'N^{2/3}, pb'N^{2/3}-(R+K)N^{1/3}\bigr)
\end{align*}
conditioned to remain strictly below the line $px-(R+K)N^{1/3}$ on $[a'N^{2/3}+1, b'N^{2/3}-1]$. Let $\overline B(x) := B(x) - px$.
By Lemma~\ref{l.Gibbs on random interval}, conditional on the $\sigma$-algebra $\F := \Fext(\{k\}, \intint{a'N^{2/3},b'N^{2/3}}, \bm L^N)$, the law of $L^N_k|_{[a'N^{2/3}, b'N^{2/3}]}$ is that of $B$ reweighted by the Radon-Nikodym derivative proportional to $W(B, L^N_{k-1}, L^N_{k+1}) = W_{\mrm{up}}(B, L^N_{k-1})W_{\mrm{low}}(B, L^N_{k+1})$ (recall from Definition~\ref{d.weight factor} and \eqref{e.W up int low split up}), where we drop the interval $\intint{a'N^{2/3}, b'N^{2/3}}$ on which $W$ is evaluated from the notation.
Then, by the Hall-Littlewood Gibbs property, the definitions of $a'$ and $b'$, and since $\smash{W_{\mrm{up}}(B,L^N_{k-1})\leq 1}$, with $Z = \EF[W(B, L^N_{k-1}, L^N_{k+1})]$,
\begin{align}
\MoveEqLeft[0]
\P\left(\msf E(t,a,b) \midd \inf_{xN^{-2/3}\in[a,b]} \overline L^N_{k-1}(x)>-(R-\delta)N^{1/3}\right) \nonumber\\
&= \E\left[\EF\!\left[\one_{\msf{Low}^t(\overline B)}W(B, L^N_{k-1}, L^N_{k+1})Z^{-1}\right]\!\!\midd\! \inf_{xN^{-2/3}\in[a,b]} \overline L^N_{k-1}(x)>-(R-\delta)N^{1/3}\right] \nonumber\\
&\leq \E\left[\EF\left[\one_{\msf{Low}^t(\overline B)}W_{\mrm{low}}(B, L^N_{k+1})Z^{-1}\right]\midd \inf_{xN^{-2/3}\in[a,b]} \overline L^N_{k-1}(x)> -(R-\delta)N^{1/3}\right]. \label{e.not low in middle breakup}
\end{align}
Recall $\delta\geq 0$ and $K\geq 1$. Since $\overline L^N_{k-1}\geq -(R-\delta)N^{1/3}$ throughout the interval $[aN^{2/3},bN^{2/3}]$ and $\overline  B \leq -(R+K)N^{1/3}$ on the interval $[a'N^{2/3}, b'N^{2/3}]$, it follows that 
$$W_{\mrm{up}}(B, L^N_{k-1}) \geq (1-q^{(\delta+K)N^{1/3}})^{(b-a)N^{2/3}} \geq (1-q^{N^{1/3}})^{(b-a)N^{2/3}} \geq \frac{1}{2}$$
for all $N > N_0$, for an $N_0$ which can be taken to be increasing in $b-a$ and independent of $\delta$, $k$, and $K$. Thus $Z\geq \frac{1}{2}\EF[W_{\mrm{low}}(B, L^N_{k+1})]$, which implies that \eqref{e.not low in middle breakup} is upper bounded~by (recalling the definition of $\msf{Low}^t(\overline B)$ from \eqref{e.low definition for lower tail at midpoint})
\begin{align*}
\MoveEqLeft[0]
2\cdot\E\left[\frac{\EF\left[\one_{\msf{Low}^t(\overline B)}W_{\mrm{low}}(B, L^N_{k+1})\right]}{\EF\left[W_{\mrm{low}}(B, L^N_{k+1})\right]}\ \bigg|\  \inf_{xN^{-2/3}\in[a,b]} \overline L^N_{k-1}(x)>-(R-\delta)N^{1/3}\right]\\
&= 2\cdot\E\Biggl[\frac{\EF\left[W_{\mrm{low}}(B, L^N_{k+1})\mid \overline B(tN^{2/3}) \leq -(R+2K)N^{1/3}\right]}{\EF\left[W_{\mrm{low}}(B, L^N_{k+1})\right]}\\
&\quad\times \PF\left(\overline B(tN^{2/3}) \leq -(R+2K)N^{1/3}\right)\ \bigg|\   \inf_{xN^{-2/3}\in[a,b]} \overline L^N_{k-1}(x)>-(R-\delta)N^{1/3}\Biggr].
\end{align*}
By weak monotonicity (Lemma~\ref{l.partition function comparison with non-int boundaries}) 
we obtain that the previous display is upper bounded by
\begin{align}\label{e.not low in middle second breakup}
2c(q)^{-1}\E\!\left[\PF\!\left(\overline B(tN^{2/3}) \leq-(R+2K)N^{1/3}\right)\!\!\ \bigg|\!\ \inf_{xN^{-2/3}\in[a,b]}\! \overline L^N_{k-1}(x)>-(R-\delta)N^{1/3}\right].
\end{align}
While the endpoints and upper boundary of $\overline B$ are deterministic, the conditional probability is still random due to the randomness of the interval $[a',b']$ on which $\overline B$ is defined. We claim that, for a deterministic constant $C$, almost surely,
\begin{equation}\label{e.not too low bound to prove}
\PF\!\left(\overline B(tN^{2/3}) \leq-(R+2K)N^{1/3}\right) \leq C\exp(-cK^2/(b'-a')),
\end{equation}
which suffices since $b'-a' \leq b-a$ implies the deterministic bound of $C\exp(-cK^2/(b-a))$. Substituting this into \eqref{e.not low in middle second breakup} would complete the proof.

So it remains to establish the upper bound from the previous display. This follows by stochastic monotonicity of Bernoulli random walk bridges (Lemma~\ref{l.random walk bridge monotonicity}): $B(tN^{2/3})$ stochastically dominates $B'(tN^{2/3})$, where $B'$ is a Bernoulli random walk bridge from 
$$\left(a'N^{2/3}, pa'N^{2/3}-(R+\tfrac{3}{2}K)N^{1/3}\right) \text{ to } \left(b'N^{2/3}, pb'N^{2/3}-(R+\tfrac{3}{2}K)N^{1/3}\right)$$
(i.e., lowered endpoints) conditioned on staying below $px-(R+K)N^{1/3}$. Let $\tilde B$ be a Bernoulli random walk bridge with the same endpoints as $B'$ (and not conditioned to stay below $px-(R+K)N^{1/3}$). So by the domination of $B'$ by $B$ and the definition of $\tilde B$, it suffices to upper bound
$$\PF\left(\tilde B(tN^{2/3}) \leq -(R+2K)N^{1/3}\middd \sup_{xN^{-2/3}\in[a',b']} (\tilde B(x)-px) \leq - (R+K)N^{1/3}\right).$$
Since there is a separation of $\frac{1}{2}KN^{1/3}$ between the endpoints of $x\mapsto \tilde B(x)-px$ at $a'N^{2/3}, b'N^{2/3}$ and $-(R+K)N^{1/3}$, and we have an upper bound of $(b-a)N^{2/3}$ on the length of the random interval $[a'N^{2/3}, b'N^{2/3}]$, it follows that there exists $\eta = \eta(b-a)>0$ such that the probability of the event being conditioned on in the previous display is at least $\eta$. Further, $\eta$ can be taken to be decreasing in $b-a$ ($\eta$ can be taken independent of $K$ since $K\geq 1$ yields a lower bound on the mentioned separation); in particular, the bound is uniform in $N$. The probability of $\tilde B(tN^{2/3}) \leq -(R+2K)N^{1/3}$ is at most $\tilde C\exp(-cK^2/(b'-a'))$ by Lemma~\ref{l.random walk bridge fluctuation}. This completes the proof of \eqref{e.not too low bound to prove} with $C=\tilde C\eta^{-1}$.
\end{proof}

\section{Uniform separation of curves}\label{s.uniform separation}

In this section we prove Proposition~\ref{p.uniform separation} which, recall, assumes control on the lower tail of the $k$\th curve and on separation between the first $k-1$ curves and obtains control on the separation between the first $k$ curves. In this section we again assume $\bm L^N$ satisfies Assumptions~\ref{as.HL Gibbs} and \ref{as.one-point tightness} without further explicit mention.

\subsection{Control on the lower curve}
We will need some control over the $k$\th curve, in particular that it does not increase too quickly, at least at some random location. For this we prove the following deterministic statement about continuous functions whose range is controlled.

\begin{lemma}\label{l.slope control}
Let $T>0$, $\delta\in(0,\frac{1}{30}T)$, and $f:[0,T]\to \R$ be continuous. Suppose $M>0$ is such that $\sup_{x,y\in [0,T]} |f(x)-f(y)|\leq M$. Then, with $P = 3MT^{-1}$, there exists $U\in [0,\frac{1}{2}T]$ such that
\begin{align*}
f(x) \leq f(U) + P(x-U) \text{ for } x\in[U, U+\delta].
\end{align*}

\end{lemma}

\begin{proof}
Let $f^-(x) = f(x) - Px$.
Next for $i=0,1, \ldots$, let $I_i = [i\cdot2\delta , (i+1)2\delta]$, $J_i = I_i + \delta = [(i+\frac{1}{2})2\delta, (i+\frac{3}{2})2\delta]$, $x^I_i = \argmax_{I_i} f^-$ and $x^J_i = \argmax_{J_i} f^-$ (taking the smallest in the case of multiple maximizers).

We claim that there is an $i$ such that $I_i\subseteq [0,\frac{1}{2}T]$, $J_i\subseteq [0,\frac{1}{2}T]$, and at least one of $x_i^I \in [\inf I_i, \inf I_i + \delta]$ or $x_i^J \in [\inf J_i, \inf J_i + \delta]$ holds, i.e., the maximizer lies in the  left half of the interval. Observe that this would complete the proof, since the length of the second half of $I_i$ or $J_i$ is $\delta$, and for such an $i$, by taking $U=x^I_i$  (if it is the $I_i$ case), it holds for $x\in[U, \sup I_i]$  that $f(x) - Px\leq f(x^I_i) - P(x^I_i)$, and similarly for the $J_i$ case.

Suppose our claim does not hold. Let $i_{\max}$ be the largest $i$ such that $J_i \subseteq [0,\frac{1}{2}T]$. Under our contrapositive hypothesis, $x^J_i \in [(i+1)2\delta, (i+\frac{3}{2})2\delta]$ and $x^I_i \in [(i+\frac{1}{2})2\delta, (i+1)2\delta]$ for each $i\in\intint{0,i_{\max}}$.
This implies that $x^I_{i} \in J_i$ and $x^J_i\in I_{i+1}$ for each $i$, so for each $i$,
\begin{align*}
f^-(x^J_{i}) \geq f^-(x^I_i) \geq f^-(x^J_{i-1}).
\end{align*}
Iterating, this gives that
\begin{align*}
f^-(x^J_{i_{\max}}) \geq f^-(x^J_{0}) \implies f(x^J_{i_{\max}}) \geq f(x^J_0) + P(x^J_{i_{\max}} - x^J_0).
\end{align*}
This contradicts the fact that $\sup_{x,y\in[0,T]} |f(x) - f(y)| \leq M$ since $x^J_{i_{\max}} \geq \frac{1}{2}T-2\delta$, $x^J_1 \leq 3\delta$, and $P = 3M/T$ (we implicitly use that $5\delta\leq \frac{1}{6}T$). This completes the proof.
\end{proof}

To have control on the range of $L^N_{k}$ as needed to apply Lemma~\ref{l.slope control}, we need control on the upper tail of $\smash{L^N_{k}}$ over an interval in addition to the lower tail control we have assumed. This uniform upper tail control is provided by the following bound due to \cite{corwin2018transversal} on the upper tail of $L^N_1$, and the fact that the curves in $\bm L^N$ are ordered. Recall the notation $\smash{\overline L^N_{j}(x) = L^N_j(x) - px}$.

\begin{proposition}[{\cite[Lemma 5.2]{corwin2018transversal}}]\label{p.upper tail}
Let $\varepsilon>0$ and $T>0$. There exist $N_0 = N_0(\varepsilon,T)$ and $M = M(\varepsilon,T)$ such that, for $N\geq N_0$,
\begin{align*}
\P\left(\sup_{s N^{-2/3}\in[-T,T]} \overline L^N_1(s) \geq MN^{1/3}\right) \leq \varepsilon.
\end{align*}

\end{proposition}

The following statement uses Lemma~\ref{l.slope control} to obtain separation between $L^N_{k-1}$ and $L^N_{k}$ at a single random point; see Figure~\ref{f.one point separation} for a depiction of the statement. It will be proved in Section~\ref{s.one point separation}.

\begin{figure}
\hspace*{-1.2cm}\includegraphics[width=0.9\textwidth]{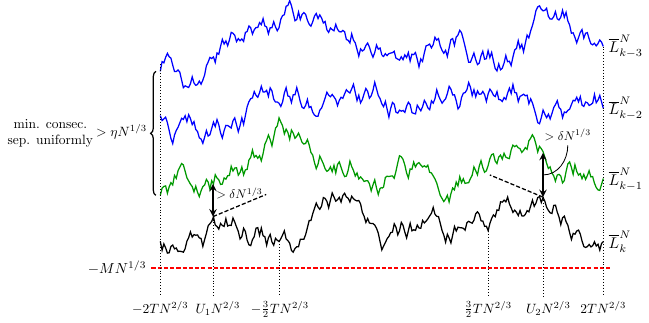}
\caption{A depiction of the setup for Lemma~\ref{l.single point separation}.  Uniform separation between the blue curves, as well as between the bottom blue curve and the green curve, holds with high probability by assumption \eqref{e.single point separation hypothesis}. By assumption~\eqref{e.single point sep lower tail hyp}, all the curves remain above $-MN^{1/3}$ (red dashed line) with high probability. Then, the lemma yields random points $U_1N^{2/3}$ and $U_2N^{2/3}$ such that separation between $\smash{\overline L^N_{k-1}}$ and $\smash{\overline L^N_k}$ holds at these points. The lemma also guarantees control on the growth of $\overline L^N_k$ in the vicinity of these random points: more precisely, the black slanted dashed lines represent the lines of slope $\pm PN^{-1/3}$ upper bounding $\smash{\overline L^N_k}$ on two intervals of size $N^{2/3}$ starting and ending at $U_1N^{2/3}$ and $U_2N^{2/3}$, respectively.}\label{f.one point separation}
\end{figure}

\begin{lemma}[One-point separation]\label{l.single point separation}
Fix $k\in\N$ and $T\geq 2$, and let 
$$\msf{Sep}^N_{k-2}(\eta, [-3T, 3T]) :=\msf{Sep}^N_{k-2}(\eta, [-3T, 3T], \bm L^N)$$
be as in \eqref{e.Sep definition}. 
Suppose for any $\varepsilon>0$ there exist $M= M(T, k, \varepsilon)$ and $N_0=N_0(k,T,\varepsilon)$ such that, for $N\geq N_0$, 
\begin{align}\label{e.single point sep lower tail hyp}
\P\left(\inf_{xN^{-2/3}\in[-2T,2T]} \overline L^N_{k}(x) \geq -MN^{1/3}\right) \geq 1-\varepsilon
\end{align}
and there exists $\eta = \eta(T,k, \varepsilon)>0$ such that, for $N\geq N_0$,
\begin{align}\label{e.single point separation hypothesis}
\P\left(\msf{Sep}^N_{k-2}(\eta, [-3T,3T])\right)\geq 1-\varepsilon.
\end{align}
Now fix any $\varepsilon>0$. There exists $P = P(k,T,\varepsilon)$, $\delta = \delta(k,T,\varepsilon)>0$, and $N_0= N_0(k,T,\varepsilon)$ such that the following holds for $N>N_0$. With probability at least $1-\varepsilon$, there exist random $L^N_{k}$-measurable points $U_1\in\left[-2T, -\frac{3}{2}T\right]$ and $U_2\in[\frac{3}{2}T, 2T]$ such that $L^N_{k-1}(U_iN^{2/3}) > L^N_{k}(U_iN^{2/3}) + \delta N^{1/3}$ for $i=1,2$, and
\begin{align*}
\overline L^N_{k}(x) &\leq \overline L^N_{k}(U_1N^{2/3}) + (x-U_1N^{2/3})PN^{-1/3} \quad \text{for}\quad xN^{-2/3}\in[U_1, U_1+1],\\
\overline L^N_{k}(x) &\leq \overline L^N_{k}(U_2N^{2/3}) + (U_2N^{2/3}-x)PN^{-1/3} \quad \text{for}\quad xN^{-2/3}\in[U_2-1, U_2].
\end{align*}

\end{lemma}

Note that the range of $L^N_k$ is of order $N^{1/3}$ and the interval size is of order $N^{2/3}$, so that the quantity $3MT^{-1}$ in Lemma~\ref{l.slope control} is of order $N^{-1/3}$; this is the source of the $N^{-1/3}$ factors in the previous display.

Lemma~\ref{l.single point separation} is the main step in the proof of Proposition~\ref{p.uniform separation} on uniform separation between these curves. A second technical ingredient that will be needed is an estimate about uniform separation for a Bernoulli random walk bridge, which we state next. As its proof is fairly straightforward, it will be proved in Appendix~\ref{s.random gibbs}.

\begin{lemma}[Random walk uniform separation]\label{l.non-int RW closeness}
Fix $M$, $p$, and $T\geq 2$. Let $U_1\in[-2T, -\smash{\frac{3}{2}}T]$ and $U_2\in[\smash{\frac{3}{2}}T, 2T]$. Let $\Brw$ be a Bernoulli random walk bridge from $(U_1N^{2/3}, z_1)$ to $(U_2N^{2/3}, z_2)$, for some $z_i\in\Z$ such that $z_i - pU_iN^{2/3}\in [-MN^{1/3}, MN^{1/3}]$ for $i\in\{1,2\}$;
 $\Bbr$ be a Brownian bridge of variance $p(1-p)$ with the same endpoints; and $f:[U_1N^{2/3},U_2N^{2/3}]\to \R$ be a function with $f(U_iN^{2/3})\leq z_i$ for $i\in\{1,2\}$. Let $\{\Bbr\geq f\}$ denote $\{\Bbr(x) \geq f(x) \ \forall x\in [U_1N^{2/3},U_2N^{2/3}]\}$ (and analogously for $\{\Brw\geq f\}$) and suppose, for some $\omega, \rho>0$, that
$$\P\left(\Bbr\geq f+\omega N^{1/3}\right) \geq \rho.$$
Then for any $\varepsilon>0$ there exists $\delta = \delta(\varepsilon, \rho, M, p, \omega, T)>0$ such that
\begin{align*}
\P\left(\inf_{x\in[-TN^{2/3},TN^{2/3}]}\left(\Brw(x) - f(x)\right) \leq \delta N^{1/3} \midd \Brw\geq f\right) \leq \varepsilon.
\end{align*}
\end{lemma}

We give the proof of Proposition~\ref{p.uniform separation} assuming the previous two statements now. The proof will involve an event $\msf{Fav}$ consisting of favorable conditions on $\bm L^N$. We point out that a different event with the same name was used in the proof of Corollary~\ref{c.mod con for L}, and in fact we will use the same name for a number of different events in upcoming proofs. The name will refer to a single event locally within each proof, however, and so should not cause confusion.

\begin{proof}[Proof of Proposition~\ref{p.uniform separation}]
By \eqref{e.separation induction hypothesis} applied with $T$ there set to $3T$ here, we know there exists $\eta = \eta(\varepsilon,T,k)>0$ such that  (dropping $\bm L^N$ from the notation of $\msf{Sep}^N_{k-2}$) the hypothesis \eqref{e.single point separation hypothesis} of Lemma~\ref{l.single point separation} holds, i.e.,
\begin{align}\label{e.k-1 separation bound}
\P\left(\msf{Sep}^N_{k-2}(\eta, [-3T,3T])^c\right) \leq \varepsilon,
\end{align}
 and there exists $M = M(k,T,\varepsilon)$ such that hypothesis \eqref{e.single point sep lower tail hyp}  is verified. 
Thus, by Lemma~\ref{l.single point separation} there exists $P=P(\varepsilon, T, k)$ such that (after reducing $\eta>0$ if necessary) with probability at least $1-\varepsilon$ there exist random $L^N_{k}$-measurable $U_1\in[-2T,-\frac{3T}{2}]$ and $U_2\in[\frac{3}{2}T, 2T]$ such that $L^N_{k-1}(U_iN^{2/3}) > L^N_{k}(U_iN^{2/3}) + \eta N^{1/3}$ for $i=1,2$ and
\begin{align}\label{e.separation induction proof lower curve control}
\overline L^N_{k}(x) &\leq \overline L^N_{k}(U_1N^{2/3}) + (x-U_1N^{2/3})PN^{-1/3} \quad \text{for}\quad xN^{-2/3}\in[U_1, U_1+1], \quad\text{and}\\
\overline L^N_{k}(x) &\leq \overline L^N_{k}(U_2N^{2/3}) + (U_2N^{2/3}-x)PN^{-1/3} \quad \text{for}\quad xN^{-2/3}\in[U_2-1, U_2].\nonumber
\end{align}
We let $\delta\in(0,\eta)$ be a real number, to be set precisely later, but for now we will use that $\delta<\eta$.  It clearly holds using \eqref{e.k-1 separation bound} that, since $\delta < \eta$,
\begin{align}\label{e.unif sep induction intermediate}
\P\left(\msf{Sep}^N_{k-1}(\delta,[-T,T])^c\right) \leq \P\left(\inf_{xN^{-2/3}\in[-T,T]}\left(L^N_{k-1}(x)-L^N_{k}(x)\right) \leq \delta N^{1/3}\right) + \varepsilon.
\end{align}
Let $\F = \Fext(k-1, \intint{U_1N^{2/3}, U_2N^{2/3}})$ be as in Definition~\ref{d.F_ext}, i.e., the $\sigma$-algebra generated by $\{L^N_i(x): (i,x)\not\in \intint{1,k-1}\times\intint{U_1N^{2/3}+1, U_2N^{2/3}-1}\}$. Recall that $U_i$ are $L^N_{k}$-measurable. Define the $\F$-measurable favourable event $\msf{Fav}$ by
\begin{align*}
\msf{Fav} &:= \left\{\min_{i=1, \ldots, k-1}\min_{xN^{-2/3}\in\{U_1,U_2\}} \left(L^N_i(x) - L^N_{i+1}(x)\right) > \eta N^{1/3}\right\}\\
&\qquad\cap \left\{\sup_{xN^{-2/3}\in[-2T,2T]} |\overline L^N_{k}(x)| \leq MN^{1/3}\right\}.
\end{align*}
By the definition of $U_1$ and $U_2$ and \eqref{e.k-1 separation bound}, the upper bound on the supremum (Proposition~\ref{p.upper tail}), and the upper bound on the lower tail \eqref{e.separation induction hypothesis}, there exists $M = M(\varepsilon, k, T)$ such that $\P(\msf{Fav}^c)\leq 3\varepsilon$. Let $\bm B=(B_1, \ldots, B_{k-1})$ be a $(k-1)$-tuple of independent Bernoulli random walk bridges, with $B_i$ from $(U_1N^{2/3}, L^N_i(U_1N^{2/3}))$ to $(U_2N^{2/3}, L^N_i(U_2N^{2/3}))$. Now $[U_1N^{2/3}, U_2N^{2/3}]$ forms a stopping domain (Definition~\ref{d.stopping domain}) with respect to $(L^N_1, \ldots, L^N_{k-1})$; by the strong Gibbs property (Proposition~\ref{p.strong gibbs}) and the just recorded bound on $\P(\msf{Fav}^c)$, the righthand side of \eqref{e.unif sep induction intermediate} is upper bounded by (with $W$ as in \eqref{e.weight function})
\begin{align}
\MoveEqLeft[0]
\E\left[\frac{\EF\left[\one_{\inf_{xN^{-2/3}\in[-T,T]}(B_{k-1}(x)-L^N_{k}(x)) \leq \delta N^{1/3}}W(\bm B,L^N_{k})\right]}{\EF[W(\bm B, L^N_{k})]}\one_{\msf{Fav}}\right] + 4\varepsilon\nonumber\\
&\leq \E\left[\frac{\EF\left[\one_{\inf_{xN^{-2/3}\in[-T,T]}(B_{k-1}(x)-L^N_{k}(x)) \leq \delta N^{1/3}}W(\bm B,L^N_{k})\ \bigl|\  B_{k-1}\geq L^N_{k}\right]}{\EF[W(\bm B, L^N_{k})]}\one_{\msf{Fav}}\right] + 4\varepsilon\label{e.first gibbs bound in separation induction argument},
\end{align}
where $B_{k-1}\geq L^N_{k}$ is shorthand for $B_{k-1}(x)\geq L^N_{k}(x)$ for all $xN^{-2/3}\in[U_1,U_2]$ and the inequality uses that $W(\bm B,L^N_{k}) = 0$ if $B_{k-1}(x)< L^N_{k}(x)$ for any $xN^{-2/3}\in[U_1,U_2]$.

\begin{figure}
\includegraphics[width=0.7\textwidth]{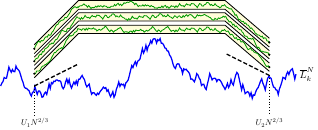}
\caption{The event we consider (after subtracting off the affine shift of slope $p$) to obtain a lower bound on the denominator in \eqref{e.first gibbs bound in separation induction argument}. The dashed lines are the lines of slope $\pm PN^{-1/3}$ that upper bound $L^N_k$ on $[U_iN^{1/3}, (U_i+1)N^{2/3}]$ for $i=1,2$ respectively, and the separation between the points in black at $U_iN^{2/3}$ is at least $\eta N^{1/3}$, implying that the separation between the resampled curves is at least $\frac{1}{3}\eta N^{1/3}$ throughout $[U_1N^{2/3}, U_2N^{2/3}]$.}
\label{f.denominator lower bound}
\end{figure}

We first lower bound the denominator of the first term of \eqref{e.first gibbs bound in separation induction argument}. We construct deterministic disjoint channels of width $\smash{\frac{2}{3}}\eta N^{1/3}$ which are separated by $\smash{\frac{1}{3}}\eta N^{1/3}$ for the paths to stay inside; see Figure~\ref{f.denominator lower bound}. Let 
$$Y = \max(MN^{1/3}, \smash{\overline L^N_k(U_1N^{2/3})} + PN^{1/3}, \smash{\overline L^N_k(U_2N^{2/3})}+PN^{1/3}).$$
Let $\smash{\ell^{\mrm{L}}_i}:[U_1N^{2/3}, (U_1+1)N^{2/3}]$ be the line joining
$$\left(U_1N^{2/3}, \overline L^N_i(U_1N^{2/3})-\tfrac{1}{3}\eta N^{1/3}\right)\quad \text{and}\quad \left((U_1+1)N^{2/3}, Y+(k-i)\eta N^{1/3}\right).$$
Similarly let $\ell^{\mrm{R}}_i:[(U_2-1)N^{2/3}, U_2N^{2/3}]$ be the line joining
$$\left(U_2N^{2/3}, \overline L^N_i(U_2N^{2/3})-\tfrac{1}{3}\eta N^{1/3}\right)\quad \text{and}\quad \left((U_2-1)N^{2/3}, Y+(k-i)\eta N^{1/3}\right).$$
Note that, on $\msf{Fav}$, it holds that 
\begin{align}\label{e.ell separation}
\ell^{*}_i(x) - \ell^{*}_{i+1}(x) \geq \eta N^{1/3}
\end{align}
for each $*\in\{\mrm{L},\mrm{R}\}$ and $xN^{-2/3}\in [U_1, U_1+1]$ and $[U_2-1, U_2]$ respectively, by the fact that $L^N_i(U_jN^{2/3}) - L^N_{i+1}(U_jN^{2/3})\geq \eta N^{1/3}$ for each $j\in\{1,2\}$ on $\msf{Fav}$. For $x\smash{N^{-2/3}}\in[U_1, U_2]$ and $i\in\intint{1,k-1}$, define the set $\mrm{Chan}_i(x)\subseteq \R$ (short for channels) by
\begin{align*}
\begin{cases}
  \left[(Y+(k-i)\eta)N^{1/3}, (Y+(k-i+\tfrac{2}{3})\eta)N^{1/3}\right], & \text{if }\  xN^{-2/3}\in[U_1+1, U_2-1]\\
  \left[\ell^{\mathrm{L}}_i(x), \ell^{\mathrm{L}}_i(x) + \tfrac{2}{3}\eta N^{1/3}\right], & \text{if }\  xN^{-2/3}\in[U_1, U_1+1]\\
  \left[\ell^{\mathrm{R}}_i(x), \ell^{\mathrm{R}}_i(x) + \tfrac{2}{3}\eta N^{1/3}\right], & \text{if }\  xN^{-2/3}\in[U_2-1, U_2].
\end{cases}
\end{align*}
Note, by \eqref{e.ell separation} and the first line of the above definition of $\mrm{Chan}_i(x)$, that for each $x\smash{N^{-2/3}}\in[U_1, U_2]$ and $i\in\intint{1,k-1}$, $\mrm{Chan}_i(x)$ is an interval of size $\frac{2}{3}\eta N^{1/3}$ and 
$$\inf \mrm{Chan}_i(x) \geq \sup \mrm{Chan}_{i+1}(x) + \smash{\frac{1}{3}}\eta N^{1/3}$$
(the infimum and supremum being of the set for fixed $x$). Also observe that $\inf \mrm{Chan}_{k-1}(x) \geq \overline L^N_{k}(U_1N^{2/3})+PN^{-1/3}(x-U_1N^{2/3}) + \frac{1}{3}\eta N^{1/3}$ for all $x\in [U_1N^{2/3}, (U_1+1)N^{2/3}]$ and $\inf \mrm{Chan}_{k-1}(x) \geq \overline L^N_{k}(U_2N^{2/3})+PN^{-1/3}(U_2N^{2/3}-x) + \frac{1}{3}\eta N^{1/3}$ for all $x\in [(U_2-1)N^{2/3}, U_2N^{2/3}]$.

 Let $\overline B_i(x) = B_i(x) - px$ and define the event $\msf{Chan}$ of staying in the channels by
 $$\msf{Chan} :=\left\{\overline B_i(x)\in\mrm{Chan}_i(x) \text{ for all } x\in[U_1N^{2/3}, U_2N^{2/3}], i\in\intint{1,k-1}\right\}.$$
 Then observe $\P(\msf{Chan})$ is lower bounded by some $\alpha = \alpha(\eta, k, T) = \alpha(\varepsilon, k, T)>0$ (as can be seen using the weak convergence of $\overline{\bm B}$ to independent Brownian bridges whose endpoints are separated by at least $\eta$, after rescaling); this also uses that $\min_{i\in\intint{1,k-1}, x\in\{U_1,U_2\}} \overline L^N_i(xN^{2/3})\geq -MN^{1/3}$ holds on $\msf{Fav}$ (since $L^N_i$ are ordered). Since $\smash{\overline L^N_{k}}(x)\leq MN^{1/3}$ for $xN^{-2/3}\in[-2T,2T]$ on $\msf{Fav}$ and the channels are separated from each other and from $\smash{\overline L^N_k}$ by at least $\smash{\frac{1}{3}\eta N^{1/3}}$, it follows that, on $\msf{Chan}$, by Lemma~\ref{l.W bound},
\begin{align*}
W(\bm B, L^N_{k}) \geq (1-q^{\frac{1}{3}\eta N^{1/3}})^{4kTN^{2/3}}\geq \tfrac{1}{2}
\end{align*}
for all large enough $N$. Thus we obtain that, on $\msf{Fav}$,
\begin{align*}
\EF[W(\bm B, L^N_{k})] \geq \EF[W(\bm B, L^N_{k})\mid \msf{Chan}]\cdot \PF(\msf{Chan}) \geq \tfrac{1}{2}\alpha.
\end{align*}
Putting this back into \eqref{e.first gibbs bound in separation induction argument} and using $W(\bm B,L^{N}_{k})\leq 1$ yields that
\begin{align*}
\MoveEqLeft[0]
\P\left(\inf_{xN^{-2/3}\in[-T,T]}\left(L^N_{k-1}(x)-L^N_{k}(x)\right) \leq \delta N^{1/3}\right)\\
&\leq 2\alpha^{-1}\E\left[\PF\left(\inf_{xN^{-2/3}\in[-T,T]}\left(B_{k-1}(x)-L^N_{k}(x)\right) \leq \delta N^{1/3}\midd B_{k-1} \geq L^N_{k}\right)\one_{\msf{Fav}}\right]\\
&\quad + 4\varepsilon.
\end{align*}
By Lemma~\ref{l.non-int RW closeness}, there exists $\delta = \delta(M,T,\varepsilon, k,\alpha) = \delta(T,\varepsilon, k)>0$ such that the conditional probability in the previous display is upper bounded by $\alpha\varepsilon$; the hypothesis on the lower bound on the probability of $\{\Bbr(x) \geq L^N_k(x) + \omega N^{1/3} \ \forall x\in[U_1,U_2]\}$ for some $\omega>0$ and with $\Bbr$ a Brownian bridge is guaranteed, using the separation of $\eta N^{1/3}$ at $U_1$, $U_2$ guaranteed on $\msf{Fav}$ and the control from \eqref{e.separation induction proof lower curve control} on $L^N_k$ in the immediate neighborhoods of $U_1$, $U_2$ included in their definition, by a construction similar to the one used to lower bound $\EF[W(\bm B,L^N_k)]$ just above, but with $k=1$; such an $\omega$ and the lower bound on the resulting probability depend on only $T, P$, and $\eta$, thus on $\varepsilon$,  $T$, and $k$. Combining with \eqref{e.unif sep induction intermediate}, this completes the proof after relabeling~$\varepsilon$.
\end{proof}

\subsection{Proof of one-point separation}\label{s.one point separation}

In this section we start the proof of Lemma~\ref{l.single point separation}.
In its proof, we will need to know that the $(k-1)$\st curve cannot drop too quickly. This is recorded in the next statement and will be proved in Section~\ref{s.no quick drops}.

\begin{lemma}[No quick drops]\label{l.no quick drop uniform}
Fix $k\in\N$.
Suppose for any $T\geq 2$ and $\varepsilon >0$ there exist $\eta = \eta(\varepsilon,k,T)>0$, $M = M(\varepsilon,k,T)$, and $N_0=N_0(\varepsilon,k,T)$ such that, for any $j\in\intint{1,k}$ (recalling the definition of $\msf{Sep}^N_{j-2}$ from \eqref{e.Sep definition}) and $N\geq N_0$,
\begin{align*}
&\P\left(\inf_{xN^{-2/3} \in [-T,T]}\overline L^N_{j-1}(x) \geq -MN^{1/3}\right) \geq 1-\varepsilon \quad\text{and} \\
&\P\left(\msf{Sep}^N_{j-2}(\eta, [-T,T], \bm L^N)\right) > 1-\varepsilon.
\end{align*}
Now fix $T\geq 2$, $\eta >0$ and $\varepsilon>0$. There exist $\delta = \delta(k,T,\eta, \varepsilon)>0$  and $N_0=N_0(k,T,\varepsilon,\eta)$ such that, for any random real number $a\in[-T,T]$ that is measurable with respect to $\Fext(k-1, \Z, \bm L^N)$ (recall Definition~\ref{d.F_ext}), i.e., with respect to $(L^N_{k}, L^N_{k+1}, \ldots)$, it holds for $N\geq N_0$ that %
\begin{align*}
\P\left(\inf_{xN^{-2/3}\in[-\delta,\delta]}\overline L^N_{k-1}(aN^{2/3}+x) < \overline L^N_{k-1}(aN^{2/3}) -\eta N^{1/3}\right) \leq \varepsilon.
\end{align*}

\end{lemma}

\begin{proof}[Proof of Lemma~\ref{l.single point separation}]
We prove the existence of such a $U_1$, as the proof for $U_2$ follows from reflection into the line $\{x=0\}$ and the whole statement then follows from a union bound.

\medskip

\noindent\textbf{Preliminaries.} Let $\msf{RangeCtrl}_{k}(M) = \{\sup_{xN^{-2/3}\in[-2T,2T]} |\overline L^N_{k}(x)|\leq MN^{1/3}\}$. By hypothesis \eqref{e.single point sep lower tail hyp} and Proposition~\ref{p.upper tail} there exists $M = M(T,k,\varepsilon)$ such that $\P(\msf{RangeCtrl}_{k}(M))\geq 1-\varepsilon$.
For a $P$ to be set, let $\msf{SlopeCtrl}_{k}$ be the event that there exists random $U_1\in[-2T,-\frac{3}{2}T]$ with
\begin{align*}
\overline L^N_{k}(x) &\leq \overline L^N_{k}(U_1N^{2/3}) + (x-U_1N^{2/3})PN^{-1/3} \quad\text{for } xN^{-2/3}\in[U_1, U_1 + 1].
\end{align*}
We define $U_1$ by taking it to be the minimum point satisfying the above if there are multiple. By Lemma~\ref{l.slope control}, there exists a deterministic $P=P(M) = P(k,T,\varepsilon)$ such that $\msf{RangeCtrl}_{k}(M)\subseteq \msf{SlopeCtrl}_{k}$ and we work on the former event for the remainder of the proof.

Next define (where the below name is short for ``single curve separation''), for $k\in\N$, an interval $I$, and a process $X: I \to \R$,
\begin{align}\label{e.singsep definition}
\msf{SingSep}^N_{k-2}(\eta,I,X) := \left\{\inf_{xN^{-2/3}\in I} \left(L^N_{k-2}(x) - X(x)\right) > \eta N^{1/3}\right\}.
\end{align}
We adopt the shorthand $\msf{SingSep}^N_{k-2}(\eta,L^N_{k-1}) = \msf{SingSep}^N_{k-2}(\eta,[-3T,-T], L^N_{k-1})$ in the rest of this proof.
By hypothesis \eqref{e.single point separation hypothesis}, we can find $\eta = \eta(T,k,\varepsilon)>0$ such that $\smash{\P(\msf{SingSep}^N_{k-2}(2\eta, L^N_{k-1})^c)} < \smash{\frac{1}{6}}\varepsilon$ (we set the space parameter to be $2\eta$ even though ahead we will mainly consider the event $\smash{\msf{SingSep}^N_{k-2}(\eta, L^N_{k-1})}$, as the extra space will be of use). Next let
\begin{align}\label{e.oneptsep definition}
\msf{OnePtSep}_k^N(\delta, L^N_{k-1}) = \left\{L^N_{k-1}(U_1N^{2/3}) > L^N_{k}(U_1N^{2/3}) + \delta N^{1/3}\right\};
\end{align}
We wish to bound the probability $\P\left(\msf{OnePtSep}_k^N(\delta, L^N_{k-1})^c\right)$. Let $\F=\Fext(\{k-1\}, \llbracket\floor{-3TN^{2/3}}$, $\floor{-TN^{2/3}}\rrbracket)$ be the $\sigma$-algebra generated by $\{L^N_i(x) : (i,x)\not\in\{k-1\}\times\intint{\floor{-3TN^{2/3}}+1, \floor{-TN^{2/3}}-1}\}$. Notice that $U_1$ and $L^N_{k+1}(U_1N^{2/3})$ are $\F$-measurable. For $\rho>0$ to be specified, define the $\F$-measurable event $\msf{Fav}$ by
\begin{align}
\msf{Fav}
&:= \left\{\sup_{xN^{-2/3}\in[-3T,-T]} \overline L^N_{k}(x) \leq MN^{1/3}\right\}\\
&\quad\cap \left\{\min_{xN^{-2/3}\in\{-3T,-T\}} \left(\overline L^N_{k-2}(x) - \overline L^N_{k-1}(x)\right) > 2\eta N^{1/3}\right\}\nonumber\\
&\quad\cap \left\{\min_{xN^{-2/3}\in\{-3T, -T\}} \overline L^N_{k-1}(x) \geq -MN^{1/3} \right\}\nonumber\\
&\quad \cap\left\{\min_{zN^{-2/3}\in\{-3T, U_1\}}\inf_{xN^{-2/3}\in[0,\rho]} \left(\overline L^N_{k-2}(z+x)-\overline L^N_{k-2}(z)\right) > -\tfrac{1}{2}\eta N^{1/3}\right\}\label{e.single point sep fav}\\
&\quad\cap \left\{\min_{zN^{-2/3}\in\{U_1, -T\}}\inf_{xN^{-2/3}\in[0,\rho]} \left(\overline L^N_{k-2}(z-x) - \overline L^N_{k-2}(z)\right) > -\tfrac{1}{2}\eta N^{1/3}\right\}\nonumber\\
&\quad \cap \Bigl\{\overline L^N_{k-2}(U_1N^{2/3}) > \overline L^N_{k}(U_1N^{2/3})+2\eta N^{1/3}\Bigr\}.\nonumber
\end{align}
Note that the parameter in the second line of $\msf{Fav}$ is $2\eta$ and not $\eta$ as in $\smash{\msf{SingSep}^N_{k-2}}(\eta, L^N_{k-1})$ that we work on; this will be needed later to guarantee a little extra space.

\begin{figure}
\includegraphics[width=0.8\textwidth]{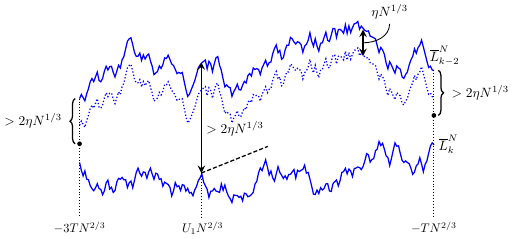}
\caption{A depiction of the properties assumed under $\msf{Fav}$. A gap of at least $2\eta N^{1/3}$ separate $L^N_{k-1}$ and $L^{N}_{k-2}$ at $-3T$ and $-T$, and $L^N_{k-1}$ can fall by at most $\frac{1}{2}\eta N^{1/3}$ on $[-3TN^{{2/3}}, (-3T+\rho)N^{2/3}]$ compared to its value at $-3TN^{2/3}$ (and similarly on $[(-T+\rho)N^{2/3}, -TN^{2/3}]$ compared to its value at $-TN^{2/3}$). The dotted blue curve is $\smash{\overline L^N_{k-2}}$ brought down by $\eta N^{1/3}$, which is what the Bernoulli random walk bridge between the endpoints of $L^N_{k-1}$ (black dots) that have been conditioned on must avoid under $\msf{SingSep}_{k-2}$. This occurs with positive probability because we can find a corridor of width $\eta N^{1/3}$ around the starting black points and avoiding $\smash{\overline L^N_k}$ and $\smash{\overline L^N_{k-2}}$ because of the other conditions included in $\msf{Fav}$.}
\end{figure}

By Proposition~\ref{p.upper tail} (upper bound on $\sup L^N_1$), the assumed lower tail bound \eqref{e.single point sep lower tail hyp}, and the above bound on $\P(\msf{SingSep}^N_{k-2}(2\eta, L^N_{k-1})^c)$, we can find $M$ such that the probability of the complement of the events in the first three lines is at most $\varepsilon/3$. By Lemma~\ref{l.no quick drop uniform} we can find $\rho>0$ such that the probability of the complement of the events in the fourth and fifth line is also at most $\varepsilon/3$; finally, the above bound on $\P(\msf{SingSep}^N_{k-2}(2\eta, L^N_{k-1})^c)$ along with the fact that $\smash{L^N_{k-1} \geq L^N_{k}}$ guarantees that the sixth line's event's complement's probability is at most $\frac{1}{3}\varepsilon$ for the same choice of $\eta$ as above. Thus overall we have $\P(\msf{Fav}^c)\leq \varepsilon$.

Now we bound $\P(\msf{OnePtSep}_k^N(\delta, L^N_{k-1})^c)$ by
\begin{align}\label{e.one point sep intermediate}
\P\left(\msf{OnePtSep}_k^N(\delta, L^N_{k-1})^c, \msf{SingSep}^N_{k-2}(\eta, L^N_{k-1}), \msf{Fav}\right) + 2\varepsilon.
\end{align}

\medskip

\noindent\textbf{Using the Gibbs property and weak monotonicity.} Applying the Gibbs property and letting $B$ be a Bernoulli random walk bridge with endpoints $(-3TN^{2/3}$, $L^N_{k-1}(-3TN^{2/3}))$ and $(-TN^{2/3}, L^N_{k-1}(-TN^{2/3}))$, the first term in \eqref{e.one point sep intermediate} is equal to (and recalling the definition of $W$ from \eqref{e.weight function})
\begin{align}
\MoveEqLeft[6]
\E\left[\frac{\EF\bigl[\one_{\msf{OnePtSep}_k^N(\delta, B)^c, \msf{SingSep}^N_{k-2}(\eta, B)}W(B,  L^N_{k-2},L^N_{k})\bigr]}{\EF\left[W(B, L^N_{k-2}, L^N_{k})\right]}\one_{\msf{Fav}}\right]\nonumber\\
&\leq \E\left[\frac{\EF\bigl[\one_{\msf{OnePtSep}_k^N(\delta, B)^c}W(B,  L^N_{k-2}, L^N_{k})\mid \msf{SingSep}^N_{k-2}(\eta, B)\bigr]}{\EF\left[W(B, L^N_{k-2}, L^N_{k})\mid \msf{SingSep}^N_{k-2}(\eta, B)\right]}\one_{\msf{Fav}}\right].\label{e.conditioned on sep}
\end{align}
We wish to use weak monotonicity (Lemma~\ref{l.partition function comparison with non-int boundaries}) to control this expression. This requires us to remove the interaction via $W$ with the upper boundary condition $L^N_{k-2}$, but we may include a non-intersection interaction with an upper curve. The following argument first aims to use the upper non-intersection freedom to remove the upper $W$-interaction.

Write $W(B, L^N_{k-2}, L^N_{k}) = W_{\mrm{up}}(B,  L^N_{k-2})\cdot W_{\mrm{low}}(B, L^N_{k})$ (as in \eqref{e.W up int low split up}). Then we see by Lemma~\ref{l.W bound} that, when $B$ satisfies $\msf{SingSep}^N_{k-2}(\eta, B)$,
\begin{align*}
W_{\mrm{up}}(B, L^N_{k-2}) \geq (1-q^{\eta N^{1/3}})^{2TN^{2/3}}\geq \tfrac{1}{2}
\end{align*}
almost surely for all large enough $N$.
Thus
\begin{align*}
\EF\left[W(B, L^N_{k-2}, L^N_{k})\mid \msf{SingSep}^N_{k-2}(\eta, B)\right] \geq \tfrac{1}{2}\E\left[W_{\mrm{low}}(B, L^N_{k})\mid \msf{SingSep}^N_{k-2}(\eta, B)\right].
\end{align*}
Let $B\geq L^N_{k}$ be shorthand for the event $B(x)\geq L^N_{k}(x)$ for all $xN^{-2/3}\in[-3T,-T]$. By using the previous display and that $W_{\mrm{up}}(B, L^N_{k-2}) \leq 1$, we see that \eqref{e.conditioned on sep} is upper bounded by
\begin{align}
\MoveEqLeft[0]
2\cdot \E\left[\frac{\EF\bigl[\one_{\msf{OnePtSep}_k^N(\delta, B)^c}W_{\mrm{low}}(B, L^N_{k})\mid\msf{SingSep}^N_{k-2}(\eta, B)\bigr]}{\EF\left[W_{\mrm{low}}(B, L^N_{k})\mid \msf{SingSep}^N_{k-2}(\eta, B)\right]}\one_{\msf{Fav}}\right]\nonumber\\
&= 2\cdot \E\left[\frac{\EF\bigl[\one_{\msf{OnePtSep}_k^N(\delta, B)^c}W_{\mrm{low}}(B, L^N_{k})\mid B\geq L^N_{k}, \msf{SingSep}^N_{k-2}(\eta, B)\bigr]}{\EF\left[W_{\mrm{low}}(B, L^N_{k})\mid B\geq L^N_{k}, \msf{SingSep}^N_{k-2}(\eta, B)\right]}\one_{\msf{Fav}}\right],\label{e.single point sep with both non-int conditions}
\end{align}
the equality due to the fact that $W_{\mrm{low}}(B, L^N_{k}) = 0$ if $B(x) < L^N_{k}(x)$ for any $xN^{-2/3}\in[-3T,-T]$ and since $B\geq L^N_{k}$ has positive (though \emph{a priori} possibly tending to $0$ with $N\to\infty$) probability.
We put $B\geq L^N_k$ in the conditioning because we need it to ensure that $B(U_1N^{2/3})$ does not come too close to $L^N_k(U_1N^{2/3})$ (the other interaction of $B$ with $L^N_k$ via $W$ will next be removed by weak monotonicity).

The display \eqref{e.single point sep with both non-int conditions} now equals
\begin{align*}
\MoveEqLeft[8]
2\cdot \E\biggl[\frac{\EF\bigl[W_{\mrm{low}}(B, L^N_{k})\mid B\geq L^N_{k}, \msf{SingSep}^N_{k-2}(\eta, B), \msf{OnePtSep}_k^N(\delta, B)^c\bigr]}{\EF\left[W_{\mrm{low}}(B, L^N_{k})\mid B\geq L^N_{k}, \msf{SingSep}^N_{k-2}(\eta, B)\right]}\\
&\times\PF\Bigl(\msf{OnePtSep}_k^N(\delta, B)^c\midd B\geq L^N_{k}, \msf{SingSep}^N_{k-2}(\eta, B)\Bigr)\one_{\msf{Fav}}\biggr].
\end{align*}
Note that $\msf{SingSep}^N_{k-2}(\eta, B)$ is simply an upper non-intersection boundary condition for $B$ and $\msf{OnePtSep}_k^N(\delta, B)^c$ is a one-point lower tail event at a deterministic location when conditioned on $\F$. Thus by the weak monotonicity statement Lemma~\ref{l.partition function comparison with non-int boundaries}, the first ratio in the previous display is upper bounded by $c(q)^{-1}$. Also invoking monotonicity of non-intersecting Bernoulli random walk bridges (Lemma~\ref{l.random walk bridge monotonicity}), the previous display is upper bounded by
\begin{align}\label{e.one-point separation bound intermediate}
\MoveEqLeft[0]
2c(q)^{-1}\\
&\cdot\E\biggl[\PF\Bigl(\msf{OnePtSep}_k^N(\delta, B)^c\midd B(U_1N^{2/3})\geq L^N_{k}(U_1N^{2/3}), \msf{SingSep}^N_{k-2}(\eta, B)\Bigr)\one_{\msf{Fav}}\biggr].\nonumber
\end{align}

 \begin{figure}
 \includegraphics[scale=1.6]{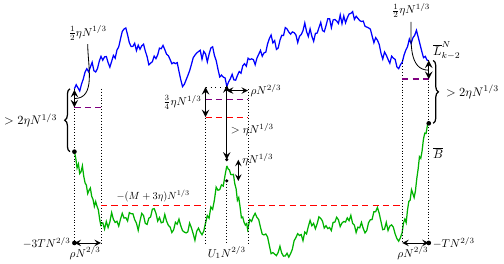}
 \caption{A depiction of the argument lower bounding \eqref{e.single sep to lower bound}. The green curve is the Bernoulli random walk bridge $\overline B$. If $\overline B$ stays below the dashed red lines then it is guarantee that $\msf{SingSep}^N_{k-2}(\eta, B)$ occurs (assuming $\msf{Fav}$ occurs). If $\overline B$ passes in between the two black dots at $\smash{U_1N^{2/3}}$, it is guaranteed that $\smash{\overline B(U_1N^{2/3}) \geq \overline L_k^N(U_1 N^{2/3})}$ (the latter's value is the height of the bottom black dot). The depicted instance of $\overline B$ is one such path. In each of the intervals of size $\rho N^{2/3}$ (one to the right of $-3TN^{2/3}$, one on either side of $U_1N^{2/3}$, and one to the left of $-TN^{2/3}$), $\smash{\overline L^N_{k-2}}$ falls from the relevant boundary point by at most $\frac{1}{2}\eta N^{1/3}$, which is a lower bound depicted by the purple dashed lines.}\label{f.single point sep lower bound}
 \end{figure}

\medskip

\noindent\textbf{Simplifying the conditioning.} We claim, on $\msf{Fav}$, that $\PF(\msf{SingSep}^N_{k-2}(\eta, B) \mid B(U_1N^{2/3}) \geq L^N_{k}(U_1N^{2/3}))\geq \alpha$ for some $\alpha = \alpha(\varepsilon, k, T) > 0$. First, it is sufficient to lower bound
\begin{equation}\label{e.single sep to lower bound}
\PF\left(\msf{SingSep}^N_{k-2}(\eta, B), B(U_1N^{2/3}) \geq L^N_{k}(U_1N^{2/3})\right).
\end{equation}
See Figure~\ref{f.single point sep lower bound}. Recall the definition \eqref{e.single point sep fav} of $\msf{Fav}$, and that $B$ is a Bernoulli random walk bridge between 
$$(-3TN^{2/3}, L^N_{k-1}(-3TN^{2/3})) \text{ and } (-TN^{2/3}, L^N_{k-1}(-TN^{2/3})).$$
On $\msf{Fav}$, the separation of $B$ from $L^N_{k-2}$ is at least $2\eta N^{1/3}$ at $-3TN^{2/3}$ and $-TN^{2/3}$; $\smash{\overline L^N_{k-2}}$ falls from its values at $-3TN^{2/3}$, $-TN^{2/3}$, and $U_1N^{2/3}$ by at most $\frac{1}{2}\eta N^{1/3}$ in an interval of size $\rho N^{2/3}$ on either side of those values (if that interval lies inside $[-3T,-T]$); the separation of $L^N_{k-2}$ from $L^N_{k}$ at $U_1N^{2/3}$ is at least $2\eta N^{1/3}$; and $\smash{\overline L^N_{k-2}} \geq -MN^{1/3}$.

Let $\overline B(x) = B(x) - px$.
 By the above, the event 
 $$\left\{\msf{SingSep}^N_{k-2}(\eta, B), B(U_1N^{2/3}) \geq L^N_{k}(U_1N^{2/3})\right\}$$
 occurs if, (i) $\overline B$ falls from $\overline B(-3TN^{2/3})$ at $-3TN^{2/3}$ to at or below $-(M+4\eta)N^{1/3}$ at $(-3T+\rho)N^{2/3}$ while not going above $\overline L^N_{k-2}(-3TN^{2/3})-\frac{3}{4}\eta N^{1/3}$ in $[-3TN^{2/3}, (-3T+\rho)N^{2/3}]$; (ii) stays below $-(M+3\eta)N^{1/3}$ on $[(-3T+\rho)N^{2/3}, (U_1-\rho)N^{2/3}]$; (iii) jumps up to lie in $[\overline L^N_{k}(U_1N^{2/3}), \overline L^N_{k}(U_1N^{2/3})+\eta N^{1/3}]$ at $U_1N^{2/3}$ while not exceeding $\overline L^N_{k-2}(U_1N^{2/3}) - \frac{3}{4}\eta N^{1/3}$ on $[(U_1-\rho)N^{2/3}, U_1N^{2/3}]$; and repeating similar properties in reverse on $[U_1N^{2/3}, -TN^{2/3}]$. Since on $\msf{Fav}$ it holds that $\smash{\overline L^N_{k}}(U_1N^{2/3}) \leq MN^{1/3}$ and $\inf_{xN^{-2/3}\in\{-3T,-T\}}\overline L^N_{k-1}(x) \geq -MN^{1/3}$, it follows from the weak convergence of $\overline B$ after rescaling to Brownian bridge that the probability of this event is indeed lower bounded by a positive constant $\alpha = \alpha(M,\eta, T, \rho) = \alpha(\varepsilon, k, T)$.

 Thus \eqref{e.one-point separation bound intermediate} is upper bounded by
 \begin{align*}
 \MoveEqLeft[1]
 2c(q)^{-1}\cdot\E\left[\frac{\PF\left(\msf{OnePtSep}_k^N(\delta, B)^c\mid B(U_1N^{2/3}) \geq L^N_{k}(U_1N^{2/3})\right)}{\PF(\msf{SingSep}^N_{k-2}(\eta, B)\mid B(U_1N^{2/3}) \geq L^N_{k}(U_1N^{2/3}))}\one_{\msf{Fav}}\right]\\
 &\leq 2c(q)^{-1}\alpha^{-1}\cdot\E\Bigl[\PF\left(\msf{OnePtSep}_k^N(\delta, B)^c \mid B(U_1N^{2/3}) \geq L^N_{k}(U_1N^{2/3})\right)\one_{\msf{Fav}}\Bigr].
 \end{align*}
\smallskip

\noindent\textbf{Invoking a Gaussian comparison.} It is not hard to see that for small enough $\delta$, the previous display is smaller than $\varepsilon$. Indeed, recalling the definition \eqref{e.oneptsep definition} of $\msf{OnePtSep}^N_k$  and letting $X\sim \mc N(0,1)$, by a normal approximation Bernoulli random walk bridges (Lemma~\ref{l.KMT}), for large enough $N$ (depending on $\delta$), the previous display is upper bounded by
\begin{align}\label{e.normal approximation to bound}
2c(q)^{-1}\alpha^{-1}\E\left[\PF\left(\sigma X+\mu \leq \delta N^{1/3} \midd \sigma X+\mu >0\right)\one_{\msf{Fav}}\right] + \tfrac{1}{2}\varepsilon,
\end{align}
where
\begin{align*}
\mu &:= \frac{U_1 + 3T}{2T}\overline L^N_{k-1}(-TN^{2/3}) - \frac{T+U_1}{2T}\overline L^N_{k-1}(-3TN^{2/3}) - \overline L^N_{k}(U_1N^{2/3})\quad \text{and} \\
\sigma^2 &:= p(1-p)\cdot\frac{(U_1+3T)(-T-U_1)N^{2/3}}{2T};
\end{align*}
the factor of $p(1-p)$ in $\sigma^2$ arises, as explained after Lemma~\ref{l.KMT}, since $B$ is a Bernoulli random walk bridge with slope $p$, making its increments have variance approximately $p(1-p)$. Since $\overline L^N_{k-1}(-3TN^{2/3})$,$\overline L^N_{k-1}(-TN^{2/3}) > -MN^{1/3}$ and $\overline L^N_{k}(U_1N^{2/3})\leq MN^{1/3}$ on $\msf{Fav}$, it follows that, again on $\msf{Fav}$,
$$\mu\geq -2MN^{1/3}.$$
Since $U_1 \in [-2T, -\tfrac{3}{2}T]$, it follows that %
$$\sigma^2 \geq \tfrac{3}{8}p(1-p)TN^{2/3}.$$
Let $K = \tfrac{3}{8}p(1-p)T$. Recall from Lemma~\ref{l.normal conditional prob montonicity} that $\PF\left(X< r+\delta N^{1/3} \mid X > r\right)$ is increasing as a function of $r$ for any fixed $\delta>0$ and $N$. Since $-\sigma^{-1}\mu \leq (\frac{8}{3}(p(1-p)T)^{-1})^{1/2}2M = 2K^{-1/2}M$, \eqref{e.normal approximation to bound} is upper bounded by
\begin{align*}
\MoveEqLeft[4]
2c(q)^{-1}\alpha^{-1}\E\left[\PF\left(X  \leq -\sigma^{-1}\mu + \sigma^{-1}\delta N^{1/3} \mid X> -\sigma^{-1}\mu\right)\one_{\msf{Fav}}\right] + \tfrac{1}{2}\varepsilon,\\
&\leq 2c(q)^{-1}\alpha^{-1}\E\left[\PF\left(X  \leq -\sigma^{-1}\mu + K^{-1/2}\delta \midd X> -\sigma^{-1}\mu\right)\one_{\msf{Fav}}\right] + \tfrac{1}{2}\varepsilon\\
&\leq 2c(q)^{-1}\alpha^{-1}\P\left(X  \leq 2K^{-1/2}M + K^{-1/2}\delta \midd X> 2K^{-1/2}M\right) + \tfrac{1}{2}\varepsilon.
\end{align*}
It is now a simple computation using standard normal bounds that the first term is upper bounded by $\tfrac{1}{2}\varepsilon$ if $\delta$ is made small enough (independent of $N$). Returning to \eqref{e.one point sep intermediate}, we see that we have bounded $\P(\msf{OnePtSep}^N_k(\delta,L^N_{k-1})^c)$ by $3\varepsilon$. This completes the proof after relabeling $\varepsilon$.
\end{proof}

\subsection{No quick drops}\label{s.no quick drops}

In this section we prove Lemma~\ref{l.no quick drop uniform}.
We first establish a one-point version of this statement, which will then be upgraded via a chaining argument. To state this one-point version we introduce notation for two events.
First, we use the shorthand $\msf{SingSep}^N_{k-2}(\eta, L^N_{k-1}) := \msf{SingSep}^N_{k-2}(\eta, [a,2T], L^N_{k-1})$ with the latter as defined in \eqref{e.singsep definition}; note that the interval is not $[-2T,2T]$.
Second, for any real numbers $M, \eta_0, \delta_0>0$, let $\msf{Fav} = \msf{Fav}(M, \eta_0, \delta_0)$ be defined by
\begin{equation}\label{e.fav definition no quick drop}
\begin{split}
\msf{Fav} &:= \Bigl\{\min_{xN^{-2/3}\in\{a,2T\}}\overline L^N_{k-1}(x) > -MN^{1/3}, \inf_{xN^{-2/3}\in[-2T,2T]} \overline L^N_{k-2}(x) \geq -MN^{1/3}\Bigr\}\\
&\qquad\cap \left\{\max_{xN^{-2/3}\in\{a,2T\}}\overline L^N_{k-1}(x) < MN^{1/3}\right\}\\
&\qquad \cap\left\{\min_{xN^{-2/3}\in\{a, 2T\}} \left(\overline L^N_{k-2}(x) - \overline L^N_{k-1}(x)\right) \geq 4\eta_0N^{1/3}\right\}\\
&\qquad\cap \left\{\inf_{xN^{-2/3}\in[0,\delta_0]} \overline L^N_{k-2}(aN^{2/3}+x) > \overline L^N_{k-2}(aN^{2/3})-\tfrac{1}{2}\eta_0 N^{1/3}\right\}\\
&\qquad\cap \left\{\inf_{xN^{-2/3}\in[0,\delta_0]} \overline L^N_{k-2}(2TN^{2/3}-x) > \overline L^N_{k-2}(2TN^{2/3})-\tfrac{1}{2}\eta_0 N^{1/3}\right\}.
\end{split}
\end{equation}

\begin{lemma}\label{l.no quick drop single point}
Fix $M$, $T\geq 2$, $\delta_0>0$ and $0 < \eta < \frac{1}{2}\eta_0$. There exist $c = c(M,T)>0$ and $C = C(M,T, \eta_0, \delta_0)$ such that, for any random real number $a\in[-T,T]$ that is measurable with respect to $(L^N_{k}, L^N_{k+1}, \ldots)$ and any $\delta>0$ and $N\geq 1$,
\begin{align*}
\MoveEqLeft[26]
\P\Bigl(\overline L^N_{k-1}((a+\delta)N^{2/3}) < \overline L^N_{k-1}(aN^{2/3}) - \eta N^{1/3}, \msf{Fav}(M,\eta_0, \delta_0), \msf{SingSep}^N_{k-2}(2\eta_0, L^N_{k-1})\Bigr)\\
&\leq C\exp\left(-c\frac{\eta^2}{|\delta|}\right).
\end{align*}

\end{lemma}

\begin{proof}[Proof of Lemma~\ref{l.no quick drop uniform}]
Recall that we must bound $\P(\inf_{xN^{-2/3}\in[-\delta,\delta]} L^N_{k-1}(aN^{2/3}+x) < L^N_{k-1}(x))$. We will replace the interval $[-\delta,\delta]$ by $[0,\delta]$ in the following argument; the case of $[-\delta,0]$ follows by reflection in the line $\{x=0\}$, and the case of $[-\delta,\delta]$ then follows by a union bound.

 Fix $\eta>0$. For any $j\in\N$, $\delta>0$, and random real number $a$ which is measurable with respect to $(L^N_{j+1}, L^N_{j+2}, \ldots)$, let the event $\msf{Drop}^N_j(a, \delta, \eta)$ be defined by
\begin{align*}
\msf{Drop}^N_j(a,\delta, \eta) = \left\{\inf_{xN^{-2/3}\in[0,\delta]}\overline L^N_{j}(aN^{2/3}+x) < \overline L^N_{j}(aN^{2/3}) -\eta N^{1/3}\right\}.
\end{align*}
We prove the lemma by induction on $j\in\intint{1,k-1}$. So we may assume that for any $T\geq 2$, $\varepsilon>0$, and $\eta>0$ there exists $\delta_0$ such that, for any $a\in[-2T,2T]$ that is measurable with respect to $(L^N_{j}, L^N_{j+1}, \ldots)$,
\begin{align}\label{e.no quick drop induction hypothesis}
\P\left(\msf{Drop}^N_{j-1}(a,\delta_0, \eta)\right) \leq \varepsilon.
\end{align}
Then we must show that, for any fixed $T\geq 2$, $\eta >0$ and $\varepsilon>0$, there exists $\delta>0$  such that, for any $a\in[-T,T]$ that is measurable with respect to $(L^N_{j+1}, L^N_{j+2}, \ldots)$,
\begin{align*}
\P\left(\msf{Drop}^N_{j}(a, \delta, \eta)\right) \leq \varepsilon.
\end{align*}
Note that it is sufficient to prove this for a smaller $\eta$ if needed since $\P(\msf{Drop}^N_{j}(a, \delta, \eta))$ is decreasing in $\eta$.

We prove the base case $j=1$ and the induction step simultaneously (in the case of $j=1$ the induction hypothesis \eqref{e.no quick drop induction hypothesis} holds trivially).
Our argument is by a chaining argument to upgrade the one-point bound from Lemma~\ref{l.no quick drop single point} to a uniform bound. Let $x_{i,\ell} = a+ i2^{-\ell}\delta$ for $i=0,1, \ldots, 2^{\ell}$ and define
\begin{align*}
\msf E_\ell := \bigcap_{i=1}^{2^{\ell}}\Bigl\{\overline L^N_{j}(x_{i,\ell}N^{2/3}) > \overline L^N_{j}(x_{\lfloor i/2\rfloor,\ell-1}N^{2/3})- \alpha_\ell\eta N^{1/3}\Bigr\},
\end{align*}
where $\alpha_\ell := \rho2^{-\ell/4}$
with $\rho := \frac{1}{4}(\sum_{j=1}^\infty 2^{-j/4})^{-1}$ (which is clearly positive), so that $\sum_{\ell=0}^\infty\alpha_\ell < \frac{1}{4}$.
We claim that, with $\ell_{\max} = \ceil{\frac{1}{2}\log_2 N}$,
\begin{align}\label{e.no quick drop claim}
\msf{Drop}^N_{j}(a, \delta, \eta) \subseteq \bigcup_{\ell=0}^{\ell_{\max}} (\msf E_\ell)^c.
\end{align}
Indeed, it holds that $(x_{i+1,\ell_{\max}}-x_{i,\ell_{\max}})N^{2/3}=\delta N^{2/3-1/2} = \delta N^{1/6}$; since $L^N_{j}$ is a Bernoulli path, this implies that, on the complement of the righthand side of the previous display, i.e., on $\cap_{\ell=0}^{\ell_{\max}} \msf E_\ell$, and for all large enough $N$ (depending on $\eta$, $\delta$),
\begin{align*}
\inf_{xN^{-2/3}\in[0,\delta]} \overline L^N_{j}(aN^{2/3}+x)
 &\geq \inf_{\substack{i=1, \ldots, 2^\ell\\ \ell= 0, \ldots, \ell_{\max}}} \overline L^N_{j}(x_{i,\ell}N^{2/3}) - \delta N^{1/6}\\
 &> \inf_{\substack{i=1, \ldots, 2^\ell\\ \ell= 0, \ldots, \ell_{\max}}} \overline L^N_{j}(x_{i,\ell}N^{2/3}) - \tfrac{1}{2}\eta N^{1/3}.
\end{align*}
For any $i$ and $\ell$, setting $i_{m+1} := \lfloor i_{m}/2\rfloor$ for $m\geq 1$ and $i_0 = i$, again on $\cap_{\ell=0}^{\ell_{\max}} \msf E_\ell$,
\begin{align*}
\overline L^N_{j}(x_{i,\ell}N^{2/3})
&= \sum_{m=0}^{\ell} \overline L^N_{j}(x_{i_m,\ell-m}N^{2/3}) - \overline L^N_{j}(x_{i_{m+1},\ell-m-1}N^{2/3}) + \overline L^N_{j}(aN^{2/3})\\
&\geq  \overline L^N_{j}(aN^{2/3}) -\sum_{m=0}^\ell \alpha_{\ell -m}\eta N^{1/3} \geq  \overline L^N_{j}(aN^{2/3}) -\tfrac{1}{4}\eta N^{1/3}
\end{align*}
by the definition of $\alpha_\ell$. So, on the event $\cap_{\ell=0}^{\ell_{\max}} E_\ell$, it indeed holds that 
$$\inf_{xN^{-2/3}\in[0,\delta]} \smash{\overline L^N_{j}(aN^{2/3}+x)} > \smash{\overline L^N_{j}(aN^{2/3})} - \eta N^{1/3},$$
i.e., \eqref{e.no quick drop claim} holds.

Let $\msf{SingSep}^N_{j-1}(2\eta_0) = \msf{SingSep}^N_{j-1}(2\eta_0, [a,2T], L^N_j)$ be as defined in \eqref{e.singsep definition}. Let $\eta_0>0$ be such that $\P(\msf{SingSep}^N_{j-1}(2\eta_0))$, as well as the probability of the event on the third line in \eqref{e.fav definition no quick drop}, is at least $\smash{1-\frac{1}{8}\varepsilon}$, which is possible by the assumptions in Lemma~\ref{l.no quick drop uniform} and since $a\in[-T,T]$. By \eqref{e.no quick drop induction hypothesis}, let $\delta_0>0$ be such that 
$$\P\left(\msf{Drop}^N_{j-1}(a,\delta_0, \tfrac{1}{2}\eta_0) \cup \msf{Drop}^N_{j-1}(2T,\delta_0, \tfrac{1}{2}\eta_0)\right) \leq \tfrac{1}{8}\varepsilon.$$
Let $\msf{Fav}(M,\eta_0, \delta_0)$ be as in Lemma~\ref{l.no quick drop single point}, i.e., \eqref{e.fav definition no quick drop}. Now, by the assumptions in Lemma~\ref{l.no quick drop uniform} and these choices (and using Assumption~\ref{as.one-point tightness} along with the fact that $L^N_{k-1}\leq L^N_1$), performing a union bound shows that we may choose $M$ such that $\P(\msf{Fav}(M,\eta_0,\delta_0)^c\cup \msf{SingSep}^N_{j-1}(2\eta_0)^c)\leq \frac{1}{2}\varepsilon$. Reduce $\eta$ if necessary so that it is less than $\frac{1}{2}\eta_0$. Then, by \eqref{e.no quick drop claim},
\begin{align*}
\P\left(\msf{Drop}^N_{j}(a, \delta, \eta)\right)
&\leq \P\left(\msf{Drop}^N_{j}(a, \delta, \eta), \msf{Fav}(M,\eta_0,\delta_0), \msf{SingSep}^N_{j-1}(2\eta_0)\right) + \tfrac{1}{2}\varepsilon\\
&\leq \sum_{\ell=0}^{\ell_{\max}} \P\left(\msf E_\ell^c, \msf{Fav}(M,\eta_0,\delta_0), \msf{SingSep}^N_{j-1}(2\eta_0)\right) + \tfrac{1}{2}\varepsilon.
 \end{align*}
By Lemma~\ref{l.no quick drop single point} and a union bound over the $2^\ell$ events in $\msf E_{\ell}^c$, we know that the sum in the previous display is upper bounded, for $\delta >0$, by
\begin{align*}
\sum_{\ell=0}^\infty 2^{\ell}C\exp\left(-c\alpha_\ell^2\eta^2/(2^{-\ell}\delta)\right) = \sum_{\ell=0}^\infty 2^{\ell}C\exp\left(-c\delta^{-1}\rho^22^{-\ell/2+\ell}\eta^2\right).
\end{align*}
This is upper bounded by $C\exp(-c\delta^{-1}\eta^2)$. 
Setting $\delta>0$ small enough such that this is upper bounded by $\frac{1}{2}\varepsilon$ completes the proof.
\end{proof}

Next we turn to establishing the one-point statement of there being no quick drops (Lemma~\ref{l.no quick drop single point}).

\begin{proof}[Proof of Lemma~\ref{l.no quick drop single point}]

Note that $a$ is measurable with respect to $(\smash{L^N_{k}, L^N_{k+1}},\ldots)$ by hypothesis, so $[aN^{2/3}, 2TN^{2/3}]$ forms a stopping domain (Definition~\ref{d.stopping domain}) with respect to $(\smash{L^N_{1}, \ldots,  L^N_{k-1}})$. Let $\F = \Fext(\{k-1\}, \intint{aN^{2/3},2TN^{2/3}}, \bm L^N)$ be as in that definition, i.e., it is the $\sigma$-algebra generated by $\{L^N_i(x) : (i,x)\not\in\{k-1\}\times \intint{aN^{2/3}+1, TN^{2/3}-1}\}$.
Note also that $\msf{Fav}$ \eqref{e.fav definition no quick drop} is $\F$-measurable, and recall $\msf{SingSep}^N_{k-2} = \msf{SingSep}^N_{k-2}(2\eta_0, [a,2T], L^N_{k-1})$ from \eqref{e.singsep definition}. Let $B$ be a Bernoulli random walk bridge from $(aN^{2/3}, L^N_{k-1}(aN^{2/3}))$ to $(2TN^{2/3}, L^N_{k-1}(2TN^{2/3}))$ and $\overline B(x) = B(x) - px$. By the strong Hall-Littlewood Gibbs property (Proposition~\ref{p.strong gibbs}), the probability in the statement of the lemma equals
\begin{align*}
\MoveEqLeft[0]
\E\left[\frac{\EF\Bigl[\one_{\overline B((a+\delta)N^{2/3}) < \overline L^N_{k-1}(aN^{2/3}) - \eta N^{1/3},\msf{SingSep}^N_{k-2}(2\eta_0, B)}W(B,L^N_{k-2}, L^N_{k})\Bigr]}{\EF[W(B,L^N_{k-2}, L^N_{k})]}\one_{\msf{Fav}}\right]\\
&\leq \E\left[\frac{\EF[\one_{\overline B((a+\delta)N^{2/3}) < \overline L^N_{k-1}(aN^{2/3}) - \eta N^{1/3},}W(B,L^N_{k-2}, L^N_{k})\mid \msf{SingSep}^N_{k-2}(2\eta_0, B)]}{\EF[W(B, L^N_{k-2}, L^N_{k})\mid \msf{SingSep}^N_{k-2}(2\eta_0, B)]}\one_{\msf{Fav}}\right]\\
&= \E\biggl[\frac{\EF[W(B, L^N_{k-2}, L^N_{k})\mid \msf{SingSep}^N_{k-2}(2\eta_0, B), \overline B((a+\delta)N^{2/3}) < \overline L^N_{k-1}(aN^{2/3}) - \eta N^{1/3}]}{\EF[W(B,L^N_{k-2}, L^N_{k})\mid \msf{SingSep}^N_{k-2}(2\eta_0, B)]}\\
&\qquad\times \PF\left(\overline B((a+\delta)N^{2/3}) < \overline L^N_{k-1}(aN^{2/3}) - \eta N^{1/3}\mid \msf{SingSep}^N_{k-2}(2\eta_0, B)\right)\one_{\msf{Fav}}\biggr].
\end{align*}
Now, on $\msf{SingSep}^N_{k-2}(2\eta_0, B)$, it holds by Lemma~\ref{l.W bound} that $W_{\mrm{up}}(B, L^N_{k-2}) \geq \frac{1}{2}$ (recall from \eqref{e.W up int low definition}) for all large enough $N$, and $W_{\mrm{up}}(B, L^N_{k-2}) \leq 1$ trivially. Thus the previous display is upper bounded by
\begin{align*}
\MoveEqLeft[2]
2\cdot\E\biggl[\frac{\EF[W_{\mrm{low}}(B, L^N_{k})\mid \msf{SingSep}^N_{k-2}(2\eta_0, B), \overline B((a+\delta)N^{2/3}) < \overline L^N_{k-1}(aN^{2/3}) - \eta N^{1/3}]}{\EF[W_{\mrm{low}}(B, L^N_{k})\mid \msf{SingSep}^N_{k-2}(2\eta_0, B)]}\\
&\times \PF\left(\overline B((a+\delta)N^{2/3}) < \overline L^N_{k-1}(aN^{2/3}) - \eta N^{1/3}\mid \msf{SingSep}^N_{k-2}(2\eta_0, B)\right)\one_{\msf{Fav}}\biggr].
\end{align*}
By Lemma~\ref{l.partition function comparison with non-int boundaries}, the first factor in the previous line is upper bounded by $c(q)^{-1}$, yielding that the previous display is upper bounded by
\begin{align*}
2c(q)^{-1}\E\biggl[\PF\Bigl(\overline B\bigl((a+\delta)N^{2/3}\bigr) < \overline L^N_{k-1}(aN^{2/3}) - \eta N^{1/3} \midd \msf{SingSep}^N_{k-2}(2\eta_0, B)\Bigr)\one_{\msf{Fav}}\biggr].
\end{align*}

\begin{figure}
\includegraphics[scale=1.6]{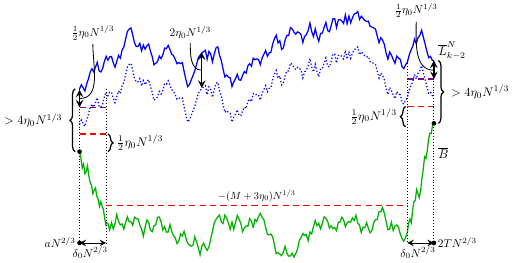}
 \caption{A depiction of the argument lower bounding $\PF(\msf{SingSep}^N_{k-2}(2\eta_0, B))$ on $\msf{Fav}$. The green curve is the affinely shifted Bernoulli random walk bridge $\overline B$, and the dashed red lines are curves which it must avoid. On $\msf{Fav}$, in each of the intervals of size $\delta_0 N^{2/3}$ (one to the right of $aN^{2/3}$ and one to the left of $2TN^{2/3}$), $\smash{\overline L^N_{k-2}}$ (solid blue curve) falls from the relevant boundary point by at most $\frac{1}{2}\eta_0 N^{1/3}$, which is indicated by the dashed purple lines in each interval. Also on $\msf{Fav}$, $\overline L^N_{k-2}$ stays above $-MN^{1/3}$ throughout the interval, which is not explicitly depicted. The dotted blue curve is $\overline L^N_{k-2}$ shifted down by $2\eta_0N^{1/3}$, which is what $\overline B$ must stay below on all of $[aN^{2/3}, 2TN^{2/3}]$ to achieve $\msf{SingSep}^N_{k-2}(2\eta_0, B)$.}
 \label{f.no quick drop}
\end{figure}

We wish to remove the conditioning event, for which we first lower bound its probability, on $\msf{Fav}$. See Figure~\ref{f.no quick drop} for a depiction of the argument; it is similar to Figure~\ref{f.single point sep lower bound} in the proof of Lemma~\ref{l.single point separation}.  On $\msf{Fav}$, the separation between $\smash{\overline B}$ and $\smash{\overline L^N_{k-2}}$ at $aN^{2/3}$ and $2TN^{2/3}$ is at least $4\eta_0 N^{1/3}$. Further, $\smash{\overline L^N_{k-2}}$ drops by at most $\smash{\frac{1}{2}\eta_0 N^{1/3}}$ in $[aN^{2/3}, (a+\delta_0)N^{2/3}]$ and $[(2T-\delta_0)N^{2/3}, 2TN^{2/3}]$ and remains above $-MN^{1/3}$ on all of $[-2TN^{2/3}, 2TN^{2/3}]$. Thus for $\overline B$ to maintain a separation of at least $2\eta_0N^{1/3}$ with $\smash{\overline L^N_{k-2}}$ on $[aN^{2/3}, 2TN^{2/3}]$, it is sufficient if it does the following: drops to or below $-(M+4\eta_0)N^{1/3}$ in $[aN^{2/3}, (a+\delta_0)N^{2/3}]$ and never goes above $\smash{\overline L^N_{k-1}}(aN^{2/3})+\frac{1}{2}\eta_0 N^{1/3}$ in the same interval, stays below $-(M+3\eta_0)N^{1/3}$ on $[(a+\delta_0)N^{2/3}, (2T-\delta_0)N^{2/3}]$, and jumps up to $\smash{\overline L^N_{k-1}}(2TN^{2/3})$ at $2TN^{2/3}$ while staying below $\overline L^N_{k-1}(2TN^{2/3}) + \frac{1}{2}\eta_0 N^{1/3}$ on $[(2T-\delta_0)N^{2/3}, 2TN^{2/3}]$. The probability that $\overline B$ satisfies these conditions is lower bounded by some deterministic $\alpha = \alpha(T, \delta_0, M, \eta_0) > 0$ (since the same is true for the Brownian bridge weak limit of the rescaled $\overline B$), using also that $\smash{\overline L^N_{k-1}}(aN^{2/3}), \smash{\overline L^N_{k-1}}(2TN^{2/3}) \in [-MN^{1/3}, MN^{1/3}]$ on $\msf{Fav}$. This yields that the previous display is upper bounded by
\begin{align*}
2\alpha^{-1}c(q)^{-1} \E\left[\PF\Bigl(\overline B((a+\delta)N^{2/3}) < \overline L^N_{k-1}(aN^{2/3}) - \eta N^{1/3}\Bigr)\one_{\msf{Fav}}\right].
\end{align*}
Since $\overline B(2TN^{2/3}), \overline B(aN^{2/3}) \in [-MN^{1/3}, MN^{1/3}]$ (so that the endpoints of $B$ define a line of slope $p+O(1)MN^{-1/3}$), it follows by one-point bounds on Bernoulli random walk bridges (Lemma~\ref{l.random walk bridge fluctuation}) that, for all $\delta>0$, the previous display is upper bounded by
\begin{align*}
C\alpha^{-1}c(q)^{-1}\exp\left(-c\delta^{-1}\eta^2\right).
\end{align*}
This completes the proof.
\end{proof}

\section{Partition function and non-intersection probability estimates}\label{s.partition function and non-intersection}

In this section we prove Proposition~\ref{p.partition function strong lower bound}. The argument is fairly standard given what we have already done; it first appeared in \cite{corwin2014brownian}, and has been used in a number of subsequent works \cite{corwin2016kpz,corwin2018transversal,dimitrov2021tightness,dimitrov2021tightness,barraquand2023spatial}.
It relies on sampling from a measure on paths on a larger interval where the interaction between paths is turned off in the interval of interest and considering the Radon-Nikodym derivative of this measure with the original one; if one has a probability lower bound on the partition function of the new measure, it can be made into a high probability lower bound on the partition function of the original measure, by an argument that can be understood as a form of ``size-biased'' sampling.

However, our arguments are slightly more complicated than in earlier works as the larger interval we work on is $[U_1, U_2]$ (defined as in Lemma~\ref{l.single point separation} such that there is control on the growth of the lower curve near $U_i$) and so is random. This is because we need some such control on the lower curve at the boundaries of the interval to obtain a lower bound on the partition function for the new measure mentioned above. In the earlier works, this lower bound was achieved  on a deterministic interval, using the stochastic monotonicity and other  features of the Gibbs property that were available in those models.

As in the last two sections, we assume here that $L^N_k$ satisfies Assumptions~\ref{as.HL Gibbs} and \ref{as.one-point tightness} without explicitly mentioning it in the statements.

\subsection{Lower bounding the partition function}

Fix $T\geq 2$. Let $U_1 \in [-2T,-\frac{3}{2}T]$ be the largest point and $U_2\in[\frac{3}{2}T,2T]$ be the smallest point such that
\begin{equation}\label{e.lower bdy control for partition function}
\begin{split}
\overline L^N_{k+1}(x) &\leq \overline L^N_{k+1}(U_1N^{2/3}) + (x-U_1N^{2/3})PN^{-1/3} \quad\text{for}\quad xN^{-2/3}\in[U_1, U_1+1]\\
\overline L^N_{k+1}(x) &\leq \overline L^N_{k+1}(U_2N^{2/3}) + (U_2N^{2/3}-x)PN^{-1/3} \quad\text{for}\quad xN^{-2/3}\in[U_2-1, U_2],
\end{split}
\end{equation}
for a $P$ to be specified later. In the event that such a $U_1$ does not exist, we set it to $-2T$, and in the event that such a $U_2$ does not exist, we set it to $2T$.

Fix an interval $[-\frac{3}{2}T, \frac{3}{2}T]\subseteq [a,b]\subseteq [-2T,2T]$ and $\bm x, \bm y\in\{\bm z\in\R^k:z_1> \ldots >z_k\}$. We define a discrete line ensemble (Definition~\ref{d.discrete line ensemble}) $\smash{\tilde{\bm L}^N = (\tilde L^N_1, \ldots, \tilde L^N_k)}$, where each $\tilde L^N_i: \intint{aN^{2/3}, bN^{2/3}}\to \R$ with $\tilde L^N_i(aN^{2/3}) = x_i$ and $\tilde L^N_i(bN^{2/3}) = y_i$, as follows. The law  $\tilde\P^{[a,b], \bm x, \bm y}_N$ of $\tilde{\bm L}^N$  is specified via its Radon-Nikodym derivative with respect to the law $\P^{k,[a,b],\bm x, \bm y}_{\mrm{free},N}$ of $k$ independent Bernoulli random walk bridges on $\intint{aN^{2/3}, bN^{2/3}}$, the $j$\th from $(aN^{2/3},x_j)$ to $(bN^{2/3},y_j)$ for each $j$: for $\bm \gamma:\intint{1,k}\times\intint{aN^{2/3}, bN^{2/3}}\to \Z$ a collection of $k$ Bernoulli paths with endpoints $(aN^{2/3},x_j)$ and $(bN^{2/3},y_j)$ for $j=1, \ldots, k$,
\begin{align}\label{e.P^N law}
\frac{\diff \tilde\P^{[a,b], \bm x, \bm y}_N}{\diff \P^{k,[a,b],\bm x, \bm y}_{\mrm{free}, N}}(\bm \gamma) \propto \prod_{i=1}^k \prod_{x\in I_{N,T}} \left(1-q^{\Delta_i(x)}\one_{\Delta_i(x) = \Delta_i(x-1)-1}\right)
\end{align}
where $I_{N,T} = \intint{U_1N^{2/3}, U_2N^{2/3}}\setminus \intint{-TN^{2/3},TN^{2/3}}$ and $\Delta_i(x) = \gamma_i(x) - \gamma_{i+1}(x)$ with $\gamma_{k+1} = L^N_{k+1}$ by convention. In words, $\tilde\P^{[a,b], \bm x, \bm y}_N$ is the usual Hall-Littlewood Gibbs measure but where the interaction is absent on $\intint{-TN^{2/3}, TN^{2/3}}$.

Writing in the just introduced notation, the definition \eqref{e.Z original definition} reads
\begin{align}\label{e.Z definition}
Z^{k, [-T,T], \bm x, \bm y, g}_N = \E^{k,[-T,T], \bm x, \bm y}_{\mrm{free}, N}[W(\bm B, g)]
\end{align}
(recall $W$ from \eqref{e.weight function}), where, under $\E^{k,[-T,T], \bm x, \bm y}_{\mrm{free}}$, $\bm B=(B_1, \ldots,  B_k)$  is a collection of $k$ independent Bernoulli random walk bridges, with $B_i$ from $(-TN^{2/3},x_i)$ to $(TN^{2/3},y_i)$.

\begin{lemma}\label{l.proxy partition function lower bound}
Suppose $g:[aN^{2/3},bN^{2/3}]\to\R$ satisfies (with $\overline g(x) := g(x) - px$)
\begin{equation}\label{e.g control}
\begin{split}
\overline g(x) &\leq \overline g(aN^{2/3}) + (x-aN^{2/3})PN^{-1/3} \quad\text{for}\quad xN^{-2/3}\in[a, a+1]\\
\overline g(x) &\leq \overline g(bN^{2/3}) + (bN^{2/3}-x)PN^{-1/3} \quad\text{for}\quad xN^{-2/3}\in[b-1, b],
\end{split}
\end{equation}
Suppose also that $M$ is such that $\smash{\sup_{xN^{-2/3}\in[a,b]}|g(x)| \leq MN^{1/3}}$, and that there exists $\eta>0$ such that $x_i > x_{i+1}+\eta \smash{N^{1/3}}$ and $y_i > y_{i+1}+\eta \smash{N^{1/3}}$ for each $i=1, \ldots, k$, with $x_{k+1} = g(a\smash{N^{2/3}})$ and $y_{k+1} = g(bN^{2/3})$. Then there exist $\delta>0$, $\varepsilon>0$, and $N_0$, all depending on $k, M, T, \eta$, such that, for all $N\geq N_0$,
\begin{align*}
\tilde\P^{[a,b], \bm x, \bm y}_N\left(Z^{k, [-T,T], \tilde{\bm L}^N(-TN^{2/3}), \tilde{\bm L}^N(TN^{2/3}), g}_N > \delta\right) > \varepsilon,
\end{align*}
where $\tilde{\bm L}^N(TN^{2/3})$ and $\tilde{\bm L}^N(-TN^{2/3})$ in the superscript indicate $(\tilde L^N_i(TN^{2/3}))_{i=1}^k$ and $(\tilde L^N_i(-TN^{2/3}))_{i=1}^k$ respectively.
\end{lemma}

\begin{proof}
We adopt the shorthand $Z^{k,T,g}_N = Z^{k, [-T,T], \tilde{\bm L}^N(-TN^{2/3}), \tilde{\bm L}^N(TN^{2/3}), g}_N$.
Let $\bm B=(B_1, \ldots, B_k)$  be $k$ independent Bernoulli random walk bridges starting from $(aN^{2/3},\bm x)$ and going to $(bN^{2/3}, \bm y)$ and define the event $\msf{Corr}$ (for corridors), with $B_{k+1}\equiv g$ for notational convenience, by
\begin{align*}
\msf{Corr} &:= \left\{\parbox{4in}{\centering $B_i(xN^{2/3}) \in [2(M+k-i)N^{1/3}, 2(M+k-i+\tfrac{1}{2})N^{1/3}]$ \\ \text{ for } $x\in\{a+1,b-1,\pm T\}, i\in\intint{1,k}$}\right\}\\
&\qquad\cap \left\{\min_{i=1, \ldots, k}\inf_{[aN^{2/3}, -TN^{2/3}] \cup [TN^{2/3}, bN^{2/3}]} \left(B_i - B_{i+1}\right) \geq \tfrac{1}{2}\eta N^{1/3}\right\}.
\end{align*}
In words, the paths $B_i$ reach points at $(a+1)N^{2/3}$ and $(b-1)N^{2/3}$ which are all above $2MN^{1/3}$ and are separated by at least $N^{1/3}$; reach points at $\pm TN^{2/3}$ which have the same properties; and do so while staying separated (from each other as well as from $g$) by at least $\frac{1}{2}\eta N^{1/3}$ on $[aN^{2/3},-TN^{2/3}]$ and $[TN^{2/3}, bN^{2/3}]$. 

This event holds with some positive probability $2\varepsilon>0$ (independent of $N$) for all large enough $N$. This is because the Brownian bridge weak limit of $\bm B$ has positive probability for the corresponding diffusive limit of the event, which uses the condition \eqref{e.g control} controlling $g$, that $\max_{[aN^{2/3}, bN^{2/3}]} g \leq MN^{1/3}$, and the initial separation of $\eta N^{1/3}$ between the endpoints of $\bm B$ as well as between the endpoints of $B_{k+1}$ and $g$.

On $\msf{Corr}$, it holds that the righthand side of \eqref{e.P^N law} is lower bounded by $\frac{1}{2}$ for all large enough $N$ by lower bounding it as in the proof of Lemma~\ref{l.W bound}; and the implicit normalization factor in \eqref{e.P^N law} is trivially upper bounded by $1$. Thus with $\smash{\tilde \P^{[a,b], \bm x, \bm y}_N}$-probability at least $\varepsilon$, we have that the entries of $\smash{\tilde{\bm L}^N(-TN^{2/3})}$, $\smash{\tilde{\bm L}^N(TN^{2/3})}$ are separated by at least $N^{1/3}$ and are all above $2MN^{1/3}$.

We claim that there exists $\delta>0$ such that, on the event that the entries of $\tilde{\bm L}^N(\pm TN^{2/3})$ have the just mentioned properties,  $\smash{Z^{k, T,g}_N} > \delta$. Indeed, with some positive probability $2\delta$, $k$ independent Bernoulli random walk bridges from $(-TN^{2/3}, \tilde{\bm L}^N(-TN^{2/3}))$ to $(TN^{2/3}, \tilde{\bm L}^N(TN^{2/3}))$ will maintain separation at least $\frac{1}{2}N^{1/3}$ throughout $[-TN^{2/3},TN^{2/3}]$ while staying above $MN^{1/3}$. On this event, the weight function $W$ is again lower bounded by $\smash{\frac{1}{2}}$ for all large enough $N$ by Lemma~\ref{l.W bound}. This completes the proof.
\end{proof}

\begin{proof}[Proof of Proposition~\ref{p.partition function strong lower bound}]
We adopt the shorthand $$Z^{k,T}_N = Z^{k, [-T,T], \bm L^N(-TN^{2/3}), \bm L^N(TN^{2/3}), L^N_{k+1}}_N.$$
 By Theorem~\ref{t.infimum lower tail} (lower tail control) and Proposition~\ref{p.upper tail} (upper tail control), there exists $M = M(\varepsilon,k,T)$ such that $\P(\msf{RangeCtrl}) \geq 1-\frac{1}{4}\varepsilon$, where
 $$\msf{RangeCtrl} = \left\{\sup_{xN^{-2/3}\in [-2T,2T]} |L^N_{k+1}(x)| \leq MN^{1/3}\right\}.$$
 Let $U_1$ and $U_2$ be as in \eqref{e.lower bdy control for partition function}. On $\msf{RangeCtrl}$, by Lemma~\ref{l.slope control}, there exists $P= P(\varepsilon, k, T, M) = P(\varepsilon, k, T)$ such that $U_1\neq -2T$ and $U_2\neq 2T$, i.e., \eqref{e.lower bdy control for partition function} holds. Next, by Theorem~\ref{t.uniform separation}, let $\eta = \eta(\varepsilon,k,T)>0$ be such that 
 $$\P(\msf{Sep}_k^N(\eta, \{U_1,U_2\}, \bm L^N)) > 1-\tfrac{1}{4}\varepsilon.$$
 Let $\msf{Fav} = \msf{RangeCtrl}\cap \msf{Sep}_k^N(\eta, \{U_1,U_2\}, \bm L^N)$, so that $\P(\msf{Fav}^c) \leq \frac{1}{2}\varepsilon$.

Let $\F=\Fext(k,\intint{U_1N^{2/3},U_2N^{2/3}}, \bm L^N)$, and note that $\intint{U_1N^{2/3},U_2N^{2/3}}$ forms a stopping domain with respect to $(L^N_1, \ldots, L^N_k)$ (Definition~\ref{d.stopping domain}), since $U_1, U_2$ are measurable with respect to $L^N_{k+1}$. Conditionally on $\F$, let $\P_N$ be the law of $\smash{\bm L^N|_{\intint{1,k}\times\intint{U_1N^{2/3}, U_2N^{2/3}}}}$ and adopt the shorthand $\smash{\tilde \P_N}$ for the law defined by \eqref{e.P^N law}. Let $\smash{\tilde{\bm L}^N \sim \tilde \P_N}$. Again conditionally on $\F$, define $\tilde{\bm L}^{N,\mrm{res}}$ and $\bm L^{N,\mrm{res}}$ (``res'' short for restricted)~by
$$\tilde L^{N,\mrm{res}}_i := \tilde L^N_i|_{\intint{U_1N^{2/3},U_2N^{2/3}}\setminus\intint{-TN^{2/3},TN^{2/3}}}$$
and $L^{N,\mrm{res}}_i :=  L^N_i|_{\intint{U_1N^{2/3}, U_2N^{2/3}}\setminus\intint{-TN^{2/3},TN^{2/3}}}$ for $i\in\intint{1, k}$, and let $\tilde \P^{\mrm{res}}_N$ and $\P^{\mrm{res}}_N$, respectively, be their laws. Note that, for any collection $\bm\gamma:\intint{1,k}\times\intint{U_1N^{2/3},U_2N^{2/3}}\setminus\intint{-TN^{2/3},TN^{2/3}}\to\Z$ of Bernoulli paths,
\begin{align}\label{e.rn derivative for turned off interaction}
\frac{\diff \P^{\mrm{res}}_N}{\diff \tilde \P^{\mrm{res}}_N}(\bm \gamma) = (\tilde Z^{k,T}_N)^{-1} Z^{k,T}_N(\bm\gamma),
\end{align}
where $Z^{k,T}_N(\bm\gamma) := Z^{k, [-T,T], \bm \gamma(-TN^{2/3}), \bm \gamma(TN^{2/3}), L^N_{k+1}}_N$ as in \eqref{e.Z definition} and $\tilde Z^{k,T}_N := \tilde \E^{\mrm{res}}_N[Z^{k,T}_N(\bm\gamma)]$. Since $Z^{k,T}_N(\bm\gamma)$ is a function of $\bm\gamma(\pm TN^{2/3})$ ($L^N_{k+1}$ is conditioned on), which have the same joint distribution when $\bm\gamma$ is distributed as $\smash{\tilde \P^{\mrm{res}}_N}$ or $\smash{\tilde \P_N}$, it holds that $\tilde Z^{k,T}_N = \tilde \E_N[Z^{k,T}_N(\bm\gamma)]$.
Similarly, $\smash{Z^{k,T}_N = Z^{k,T}_N(\bm L^N)}$ has the same distribution under $\P_N$ and $\P^{\mrm{res}}_N$. Thus, since $\msf{Fav}$ is $\F$-measurable,
\begin{align*}
\P\left(Z^{k,T}_N < \delta\right)
&\leq \E\left[\PF\left(Z^{k,T}_N < \delta\right)\one_{\msf{Fav}}\right] + \P\left(\msf{Fav}^c\right)\\
&= \E\left[\P^{\mrm{res}}_N\left(Z^{k,T}_N < \delta\right)\one_{\msf{Fav}}\right] + \tfrac{1}{2}\varepsilon.
\end{align*}
Now by \eqref{e.rn derivative for turned off interaction},  the previous display equals
\begin{align}\label{e.size biased sampling intermediate step}
\E\left[(\tilde Z^{k,T}_N)^{-1}\tilde\E^{\mrm{res}}_N\left[\one_{Z^{k,T}_N < \delta}Z^{k,T}_N\right]\one_{\msf{Fav}}\right] + \tfrac{1}{2}\varepsilon.
\end{align}
On $\msf{Fav}$, $(L^N_i(\pm TN^{2/3}))_{i=1}^k$ and $L^N_{k+1}$ satisfy the conditions in Lemma~\ref{l.proxy partition function lower bound}. So we have that, for some $\delta',\varepsilon'>0$ (the $\delta$, $\varepsilon$ provided by Lemma~\ref{l.proxy partition function lower bound}),
\begin{align*}
\tilde Z^{k,T}_N = \tilde \E[Z^{k,T}_N] \geq \rho :=\delta' \varepsilon'.
\end{align*}
This yields that \eqref{e.size biased sampling intermediate step} is upper bounded~by
\begin{align*}
\rho^{-1}\cdot\E\left[\tilde\E^{\mrm{res}}_N\left[\one_{Z^{k,T}_N < \delta}Z^{k,T}_N\right]\right] + \tfrac{1}{2}\varepsilon \leq \delta \rho^{-1} + \tfrac{1}{2}\varepsilon.
\end{align*}
Picking $\delta = \frac{1}{2}\varepsilon \rho$ completes the proof.
\end{proof}

\section{Brownian Gibbs property in the limit}\label{s.bg in the limit}

In this section we prove that all weak limit points of $\bm \cL^N$ as defined in \eqref{e.cL definition} possess the Brownian Gibbs property, thus establishing Proposition~\ref{p.limits have BG}. The first step is to compare the partition function associated to the Hall-Littlewood Gibbs property with the probability of separation (its analog for the Brownian Gibbs property); we do this in the next section. Then we will give the proof of Proposition~\ref{p.limits have BG} in Section~\ref{s.proof of BG}. As in the last few sections, we assume, without explicitly stating it again, that $\bm L^N$ satisfies Assumptions~\ref{as.HL Gibbs} and \ref{as.one-point tightness}.

\subsection{Comparing the partition function and the separation probability}

To emphasize the separation required with the lower boundary condition, in this section we extend the notation \eqref{e.Sep definition} as follows: for a set $A\subseteq \R$, an integer $k\in\N$, $\Lambda\subseteq \N$ a (possibly infinite) interval with $\intint{1,k}\subseteq \Lambda$, a  process $\bm X:\Lambda\times\R\to\R$, a process $Y:\R\to\R$, and $\delta>0$, we define the event
\begin{equation}\label{e.Sep extended definition}
\begin{split}
\msf{Sep}^N_k(\delta, A, \bm X, Y) &= \left\{\min_{i=1, \ldots, k-1}\inf_{xN^{-2/3}\in A}\bigl(X_i(x) - X_{i+1}(x)\bigr) \geq \delta N^{1/3}\right\}\\
&\qquad\cap \left\{\inf_{xN^{-2/3}\in A}\bigl(X_{k}(x) - Y(x)\bigr) \geq \delta N^{1/3}\right\},
\end{split}
\end{equation}
i.e., in the notation \eqref{e.Sep definition}, if we define $\tilde{\bm X}:\intint{1,k+1}\times\R\to\R$ by $\tilde X_i(\bm\cdot) = X_i(\bm\cdot)$ for $i=1, \ldots, k$ and $\tilde X_{k+1}(\bm\cdot) = Y(\bm\cdot)$, then $\msf{Sep}^N_k(\delta, A, \bm X, Y) = \msf{Sep}^N_k(\delta, A, \tilde{\bm X})$. At times we will drop the set $A$ from the notation, but this will be explicitly stated.

For the next statement, recall $W$ from Definition~\ref{d.weight factor}.

\begin{lemma}\label{l.partition function to non-int comparison}
Fix $a>b$ and $k\in \N$. Let $\F=\Fext(k, \intint{aN^{2/3}, bN^{2/3}}, \bm L^N)$ and $\bm B= (B_1, \ldots, B_k)$ be a collection of $k$ independent Bernoulli random walk bridges, with $B_i$ from $(aN^{2/3}, L^N_i(aN^{2/3}))$ to $(bN^{2/3}, L^N_i(bN^{2/3}))$. For any $\varepsilon>0$, there exist $N_0 = N_0(\varepsilon,k,b-a)$ and $\delta = \delta(\varepsilon,k, b-a)>0$ such that, for all $N\geq N_0$,
\begin{align*}
\P\left(\frac{\EF[W(\bm B, L^N_{k+1}, \intint{aN^{2/3}, bN^{2/3}})]}{\PF(\msf{Sep}^N_k(\delta,[a,b],\bm  B, L^N_{k+1}))} \in [1-\varepsilon, 1+\varepsilon]\right) \geq 1-\varepsilon.
 \end{align*}

\end{lemma}

\begin{proof}
Adopt the shorthand $\msf{Sep}^N_k(\delta,\bm  B, L^N_{k+1}) = \msf{Sep}^N_k(\delta,[a,b],\bm  B, L^N_{k+1})$ and drop the interval $\intint{aN^{2/3}, bN^{2/3}}$ from the $W$ notation. We first pick $\delta>0$ such that, by Theorem~\ref{t.uniform separation},
\begin{align*}
\P(\msf{Sep}^N_k(\delta, \bm L^N,  L^N_{k+1})) \geq 1-\varepsilon^2.
\end{align*}
Since the lefthand side of this inequality equals $\E[\PF(\msf{Sep}^N_k(\delta, \bm L^N,  L^N_{k+1}))]$, and the conditional probability is a random variable taking values in $[0,1]$, it follows that, with probability at least $1-\varepsilon$,
\begin{align*}
\PF(\msf{Sep}^N_k(\delta, \bm L^N,  L^N_{k+1})) \geq 1-\varepsilon.
\end{align*}
We work on the event that the previous inequality holds. We have the following string of inequalities, using $W(\bm B, L^N_{k+1})\leq 1$ in the first inequality, the Hall-Littlewood Gibbs property to go from the second to the third line, and the previous display for the inequality in the last line:
\begin{align*}
\PF(\msf{Sep}^N_k(\delta, \bm B, L^N_{k+1}))
&\geq \EF[\one_{\msf{Sep}^N_k(\delta, \bm B, L^N_{k+1})}W(\bm B,L^N_{k+1})]\\
&= \frac{\EF[\one_{\msf{Sep}^N_k(\delta,  \bm B, L^N_{k+1})}W(\bm B,L^N_{k+1})]}{\EF[W(\bm B,L^N_{k+1})]} \cdot \EF[W(\bm B,L^N_{k+1})]\\
&= \PF(\msf{Sep}^N_k(\delta, \bm L^N,  L^N_{k+1}))\cdot \EF[W(\bm B,L^N_{k+1})]\\
&\geq (1-\varepsilon) \EF[W(\bm B, L^N_{k+1})].
\end{align*}
It also trivially holds that $\EF[W(\bm B, L^N_{k+1})] \geq \EF[W(\bm B,\smash{L^N_{k+1}})\one_{\msf{Sep}^N_k(\delta, \bm B, \smash{L^N_{k+1}})}]$ and that, on the event $\smash{\msf{Sep}^N_k}(\delta,  \bm B, \smash{L^N_{k+1}})$ and by Lemma~\ref{l.W bound}, $W(\bm B, \smash{L^N_{k+1}})\geq 1-\varepsilon$ for all large enough $N$. This yields that, for such $N$,
\begin{align*}
\EF[W(\bm B, L^N_{k+1})] \geq (1-\varepsilon)\PF(\msf{Sep}^N_k(\delta, \bm B, L^N_{k+1})).
\end{align*}
This completes the proof after relabeling $\varepsilon$.
\end{proof}

\begin{corollary}\label{c.coupling with non-int bernoulli random walk bridge}
Fix $\varepsilon\in(0,1)$, $k\in\N$, $M>0$, and $a<b$. Let $\F= \Fext(k, [aN^{2/3},bN^{2/3}]$, $\bm L^N)$  and $\bm B=(B_1, \ldots, B_k)$ be a collection of $k$ independent Bernoulli random walk bridges with $B_i$ from $(aN^{2/3},L^N_i(aN^{2/3}))$ to $(bN^{2/3}, L^N_i(bN^{2/3}))$. There exist $N_0 = N_0(\varepsilon, k, b-a, M)$ and $\delta= \delta(\varepsilon,k,b-a, M)>0$ such that, for any $N\geq N_0$, the following holds with probability at least $1-\varepsilon$. For any bounded measurable function $F: \mc C(\intint{1,k}\times[aN^{2/3},bN^{2/3}], \R)\to \R$ such that $\sup|F| \leq M$,
\begin{align*}
\left|\frac{\EF\left[F(\bm B)W(\bm B,L^N_{k+1}, \intint{aN^{2/3}, bN^{2/3}})\right]}{\EF[W(\bm B,L^N_{k+1},\intint{aN^{2/3}, bN^{2/3}})]} - \EF\left[F(\bm B) \mid \msf{Sep}^N_k(\delta, [a,b], \bm B, L^N_{k+1})\right]\right| \leq \varepsilon.
\end{align*}
%
\end{corollary}

\begin{proof}
We again adopt the shorthand $\msf{Sep}^N_k(\delta,\bm  B, L^N_{k+1}) = \msf{Sep}^N_k(\delta,[a,b],\bm  B, L^N_{k+1})$ and drop the interval $\intint{aN^{2/3}, bN^{2/3}}$ from the $W$ notation. We may assume that $F\geq 0$ as we have assumed $F$ is bounded and adding a constant to $F$ does not change the lefthand side of the inequality we are proving.
Let $\varepsilon' = \varepsilon M^{-1}$. Pick $\delta>0$ from Lemma~\ref{l.partition function to non-int comparison} with $\varepsilon$ there equal to $\varepsilon'$, and let $N_0 = N_0(\varepsilon,k, b-a, M)$ be such that $W(\bm B,L^N_{k+1}) \geq 1-\varepsilon'$ on $\msf{Sep}^N_k(\delta, \bm B, L^N_{k+1})$ for $N>N_0$ (by Lemma~\ref{l.W bound}). Then
\begin{align*}
\frac{\EF\left[F(\bm B)W(\bm B,L^N_{k+1})\right]}{\EF[W(\bm B,L^N_{k+1})]}
&\geq \frac{\EF\left[F(\bm B)\one_{\msf{Sep}^N_k(\delta, \bm B, L^N_{k+1})}W(\bm B,L^N_{k+1})\right]}{\EF[W(\bm B,L^N_{k+1})]}\\
&\geq (1-\varepsilon')\frac{\EF\left[F(\bm B)\one_{\msf{Sep}^N_k(\delta, \bm B, L^N_{k+1})}\right]}{\EF[W(\bm B,L^N_{k+1})]}.
\end{align*}
Now invoking Lemma~\ref{l.partition function to non-int comparison} and our choice of $\delta$, with probability at least $1-\varepsilon$, the previous display is lower bounded by
\begin{align*}
\frac{1-\varepsilon'}{1+\varepsilon'}\cdot\frac{\EF\left[F(\bm B)\one_{\msf{Sep}^N_k(\delta, \bm B, L^N_{k+1})}\right]}{\PF(\msf{Sep}^N_k(\delta, \bm B, L^N_{k+1}))}
 = \frac{1-\varepsilon'}{1+\varepsilon'}\cdot \EF[F(\bm B)\mid \msf{Sep}^N_k(\delta, \bm B, L^N_{k+1})].
\end{align*}
Rearranging this gives one side of the claimed inequality after relabeling $\varepsilon$, since $F$ is bounded by $M$, and the other is obtained by replacing $F$ with $M-F$, and rearranging appropriately again.
\end{proof}

\subsection{Brownian Gibbs property in the limit}\label{s.proof of BG}

Throughout this section, we recall the definition of the Brownian Gibbs property from Definition~\ref{d.bg}.

\begin{lemma}\label{l.RWB to BB}
For each $N\in\N$, let $B^N$ be a Bernoulli random walk bridge from $(0,0)$ to $(N, z_N)$, and suppose $z_N/N\to p \in(-1, 0)$ as $N\to\infty$. Then
\begin{align*}
N^{-1/2}(B^N(Nt) - z_Nt) \stackrel{d}{\to} \Bbr(t)
\end{align*}
as $N\to\infty$ as processes on $\mc C([0,1], \R)$ with the topology of uniform convergence, where $\Bbr$ is a Brownian bridge on $[0,1]$ with zero endpoints and variance $p(1-p)$.
\end{lemma}

\begin{proof}%
This follows from the KMT coupling Lemma~\ref{l.KMT} (with the $p$ there equal to $z_N/N$ here).
\end{proof}

\begin{proof}[Proof of Proposition~\ref{p.limits have BG}]
It is sufficient to prove the resampling property when resampling the top $k$ curves for any $k\in\N$, i.e, prove Definition~\ref{d.bg} holds with $j=1$, as the conditional distribution in the case of $j>1$ is obtained from the $j=1$ case along with the fact that non-intersecting Brownian bridges have the Brownian Gibbs property themselves \cite[Lemma~2.13]{dimitrov2021characterization} (this is stated in the case of no lower boundary condition, but the proof also applies if there is one). More precisely, given an interval $[a,b]$ and knowledge that the Brownian Gibbs property holds for a line ensemble $\bm\cL$ when $j=1$ (i.e., when conditioned on $\Fext(k, [a,b], \bm\cL)$), to obtain the same for $j>1$, one further conditions on  $\Fext(\intint{j,k}, [a,b], \bm\cL)$ and makes use of \cite[Lemma~2.13]{dimitrov2021characterization} along with the fact that iterative conditioning is commutative \cite[Theorem 8.15]{Kallenberg}.

Now we turn to prove the $j=1$ case. For $k\in\N$ and $a<b$, let $\bm{\mf B}:\intint{1,k}\times[a,b]\to\R$ with $\bm{\mf B} = (\mf B_1, \ldots, \mf B_k)$ be a collection of $k$ rate two independent Brownian bridges, with $\mf B_i(x) = \cL_i(x)$ for $i\in\intint{1,k}$ and $x\in \{a,b\}$. Similar to \eqref{e.Sep extended definition}, we extend the definition \eqref{e.nonint definition} to emphasize the lower boundary condition: for an interval $I$, define
\begin{align*}
\msf{NonInt}_k(I, \bm{\mf B},\cL_{k+1}) := \Bigl\{\mf B_1(x) > \mf B_2(x) > \ldots  > \mf B_k(x) > \cL_{k+1}(x) \text{ for all } x\in I\Bigr\}.
\end{align*}
In the rest of the proof we take $I=[a,b]$ and omit it from the notation. Let $\F = \Fext(k,[a,b], \bm\cL)$. For any $k\in\N$, $a<b$, and $F:\mc C(\intint{1,k}\times[a,b], \R) \to \R$ bounded and continuous, we have to show that
\begin{align*}
\E\left[F(\bm\cL) \mid \F\right] = \EF[F(\bm{\mf B})\mid \msf{NonInt}(\bm{\mf B}, \cL_{k+1})],
\end{align*}
where $F(\bm X)$ is shorthand for $F(\bm X|_{\intint{1,k}\times[a,b]})$ and the righthand side is interpreted as in \eqref{e.conditional expectation with conditioning definition}.

By the definition of conditional expectations, and a monotone class argument, the previous display will be proven if we show that, for any $j\in\N$, $c<d$, and $G:\mc C(\intint{1,j}\times[c,d] \setminus (\intint{1,k}\times[a,b]), \R) \to \R$ bounded and continuous (with $G(\bm X)$ shorthand for $G(\bm X|_{\intint{1,j}\times[c,d] \setminus (\intint{1,k}\times[a,b])})$),
\begin{align}\label{e.BG to show}
\E\left[F(\bm \cL)\cdot G(\bm \cL) \right] = \E\bigl[\EF[F(\bm{\mf B})\mid \msf{NonInt}(\bm{\mf B}, \cL_{k+1})]\cdot G(\bm\cL)\bigr].
\end{align}
For notational convenience, we let $\bm\cL^N$ represent the subsequence which converges weakly to $\bm\cL$. By Skorokhod's representation theorem and by the weak convergence of $\bm \cL^N$ to $\bm \cL$, we may assume that we work on a probability space such that $\bm \cL^N\to \bm \cL$ uniformly on compact sets almost surely. Then we obtain, with 
$$\F^N := \Fext^N(k,\intint{\beta^{-1}aN^{2/3},\beta^{-1}bN^{2/3}}, \bm L^N),$$
by the dominated convergence theorem and the Hall-Littlewood Gibbs property, that
\begin{align}
\E\left[F(\bm \cL)\cdot G(\bm \cL) \right]
&= \lim_{N\to\infty} \E\left[F(\bm \cL^N)\cdot G(\bm \cL^N)\right] \nonumber\\
&= \lim_{N\to\infty}\E\left[\E[F(\bm \cL^N)\mid \F^N]\cdot G(\bm \cL^N)\right]\nonumber\\
&= \lim_{N\to\infty}\E\left[\frac{\E_{\mc F^N}[F(\bm{\mf B}^N) W(\bm B^N, L^N_{k+1})]}{\E_{\mc F^N}[W(\bm B^N, L^N_{k+1})]}\cdot G(\bm \cL^N)\right],\label{e.BG convergence step}
\end{align}
the last line using the Hall-Littlewood Gibbs property and with $\bm{\mf B}^N=(B^N_1, \ldots,  B^N_k):\intint{1,k}\times[a,b]\to\R$ given by (recall $p$ from Assumption~\ref{as.one-point tightness}, $\beta$ and $\sigma$ from \eqref{e.scaling coefficients relation}, and the definition of $\bm \cL^N$ in terms of $\bm L^N$ from \eqref{e.cL definition}) $\bm{\mf B}^N = T_{p,\lambda,N}(\bm B^N)$, i.e.,
$$\mf B^N_i(x) = \sigma^{-1}N^{-1/3}\left(B^N_i(\beta xN^{2/3}) - p\beta xN^{2/3}\right)$$
with ${\bm B}^N=(B^N_1, \ldots, B^N_k)$ a family of $k$ independent Bernoulli random walk bridges, with $B^N_i$ from $(\beta^{-1} aN^{2/3}, L^N_i(\beta^{-1} aN^{2/3}))$ to $(\beta^{-1} bN^{2/3}, L^N_i(\beta^{-1} bN^{2/3}))$. Now, by Lemma~\ref{l.RWB to BB}, it almost surely holds that $\bm{\mf B}^N$ converges weakly as $N\to\infty$ to $\bm{\mf B} = (\mf B_1, \ldots, \mf B_k)$, as above, a collection of $k$ independent rate two Brownian bridges with $\mf B_i$ going from $(a, \cL_i(a))$ to $(b, \cL_i(b))$; the fact that the rate of the Brownian bridges is $2$ is a consequence of the relation \eqref{e.scaling coefficients relation} between $\beta$ and $\sigma$ in the definitions of $\bm\cL^N$ and thus $\bm{\mf B}^N$.

Recall the extended definition of $\msf{Sep}$ from \eqref{e.Sep extended definition}. Because it requires separation of order $N^{1/3}$, which should not be present for the limiting objects, we define an event $\msf{ScSep}$ (scaled separation) for an interval $[a,b]\subset \R$, an integer $k\in\N$, a (possibly infinite) interval $\Lambda\subseteq \N$ such that $\intint{1,k}\subseteq \Lambda$, and  processes $\bm X:\Lambda\times\R\to\R$ and $Y:\R\to\R$, by
\begin{align*}
\msf{ScSep}_k(\delta, [a,b], \bm X, Y) &:= \left\{\min_{i=1, \ldots, k-1}\inf_{x\in[a,b]} \left(X_i(x) - X_{i+1}(x)\right) \geq \delta\right\}\\
&\qquad \cap \left\{\inf_{x\in[a,b]} \left(X_{k+1}(x) - Y(x)\right) \geq \delta\right\}.
\end{align*}
Then observe that, by the definitions of $\bm{\mf B}^N$ and $\bm\cL^N$, it holds that
\begin{align}\label{e.scsep definition}
\msf{Sep}^N_k(\delta\sigma, [\beta^{-1}a,\beta^{-1}b], \bm B^N, L^N_{k+1}) = \msf{ScSep}_k(\delta, [a,b], \bm{\mf B}^N, \cL^N_{k+1}).
\end{align}

By Corollary~\ref{c.coupling with non-int bernoulli random walk bridge} and the previous display, we obtain that there is a sequence $\varepsilon_{N,\delta}\downarrow \varepsilon_\delta$ as $N\to\infty$, where $\varepsilon_\delta$ is a sequence which goes to $0$ as $\delta\downarrow 0$, such that \eqref{e.BG convergence step} equals
\begin{align}
\MoveEqLeft[2]
\lim_{N\to\infty}\E\left[\left(\E_{\mc F^N}[F(\bm{\mf B}^N)\mid \msf{ScSep}_k(\delta, [a,b], \bm{\mf B}^N, \cL^N_{k+1})] + \varepsilon_{N,\delta}\right)\cdot G(\bm \cL^N)\right]\nonumber\\
&= \E\left[\lim_{N\to\infty}\left(\E_{\mc F^N}[F(\bm{\mf B}^N)\mid \msf{ScSep}_k(\delta, [a,b], \bm{\mf B}^N, \cL^N_{k+1})] + \varepsilon_{N,\delta}\right)\cdot \lim_{N\to\infty}G(\bm \cL^N)\right],\nonumber\\
&= \E\left[\left(\lim_{N\to\infty}\E_{\mc F^N}[F(\bm{\mf B}^N)\mid \msf{ScSep}_k(\delta, [a,b], \bm{\mf B}^N, \cL^N_{k+1})] + \varepsilon_{\delta}\right)\cdot G(\bm \cL)\right], \label{e.bg in limit verification}
\end{align}
using the dominated convergence theorem since $F$ and $G$ are bounded and then that $G$ is continuous and $\bm\cL^N \to \bm \cL$ uniformly on compact sets almost surely. 

Now, by Lemma~\ref{l.conditional expectation limit} ahead,
\begin{align*}
\MoveEqLeft[16]
\lim_{N\to\infty}\E_{\mc F^N}[F(\bm{\mf B}^N)\mid \msf{ScSep}_k(\delta, [a,b], \bm{\mf B}^N, \cL^N_{k+1})]\\
&= \E_{\mc F}[F(\bm{\mf B})\mid \msf{ScSep}_k(\delta, [a,b], \bm{\mf B}, \cL_{k+1})].
\end{align*}
Substituting this in \eqref{e.bg in limit verification} yields that 
\begin{align*}
\MoveEqLeft[9]
\lim_{N\to\infty}\E\left[\left(\E_{\mc F^N}[F(\bm{\mf B}^N)\mid \msf{ScSep}_k(\delta, [a,b], \bm{\mf B}^N, \cL^N_{k+1})] + \varepsilon_{N,\delta}\right)\cdot G(\bm \cL^N)\right]\\
&= \E\left[\left(\E_{\mc F}[F(\bm{\mf B})\mid \msf{ScSep}_k(\delta, [a,b], \bm{\mf B}, \cL_{k+1})] + \varepsilon_{\delta}\right)\cdot G(\bm \cL)\right].
\end{align*}

Now it remains to take $\delta\to 0$. Observe that $\one_{\msf{NonInt}(\bm{\mf B}, \cL_{k+1})}$ is almost surely continuous as a function of $\bm{\mf B}$ (for fixed $\cL_{k+1}$) under the law of $\bm{\mf B}$, for almost every instance of $(\cL_i(a), \cL_i(b))_{i=1}^k$ and $\cL_{k+1}$; this is because independent Brownian bridges have probability zero of touching but not crossing any given continuous function (see e.g. \cite[Corollary 2.9]{corwin2014brownian}). Also observe that
$$\lim_{\delta\to 0}\one_{\msf{ScSep}_k(\delta, [a,b], \bm{\mf B}, \cL_{k+1})} = \one_{\msf{NonInt}(\bm{\mf B}, \cL_{k+1})}$$
almost surely, and $\P_{\mc F}(\msf{NonInt}(\bm{\mf B}, \cL_{k+1})) > 0$ almost surely. Thus applying the dominated convergence theorem and the Portmanteau theorem yields that, almost surely,
\begin{align*}
\MoveEqLeft[8]
\lim_{\delta\to0} \E_{\mc F}[F(\bm{\mf B})\mid \msf{ScSep}_k(\delta, [a,b], \bm{\mf B}, \cL_{k+1})]\\
&= \lim_{\delta\to 0}\E_{\mc F}\left[F(\bm{\mf B})\frac{\one_{\msf{ScSep}_k(\delta, [a,b], \bm{\mf B}, \cL_{k+1})}}{\P_{\mc F}(\msf{ScSep}_k(\delta, [a,b], \bm{\mf B}, \cL_{k+1}))}\right]\\
%
%
&= \E_{\mc F}\left[F(\bm{\mf B})\frac{\lim_{\delta\to0}\one_{\msf{ScSep}_k(\delta, [a,b], \bm{\mf B}, \cL_{k+1})}}{\lim_{\delta\to0}\P_{\mc F}(\msf{ScSep}_k(\delta, [a,b], \bm{\mf B}, \cL_{k+1}))}\right]\\
&= \E_{\mc F}[F(\bm{\mf B})\mid \msf{NonInt}(\bm{\mf B}, \cL_{k+1})],
\end{align*}
which completes the proof.
\end{proof}

\begin{lemma}\label{l.conditional expectation limit}
In the setting of the proof of Proposition~\ref{p.limits have BG} above, for any bounded continuous function $F: \mc C(\intint{1,k}\times[a,b], \R) \to \R$,
\begin{align*}
\MoveEqLeft[16]
\lim_{N\to\infty}\E_{\mc F^N}[F(\bm{\mf B}^N)\mid \msf{ScSep}_k(\delta, [a,b], \bm{\mf B}^N, \cL^N_{k+1})]\\
&= \E_{\mc F}[F(\bm{\mf B})\mid \msf{ScSep}_k(\delta, [a,b], \bm{\mf B}, \cL_{k+1})].
\end{align*}
\end{lemma}

\begin{proof}
Let $\tilde{\F}^N = \sigma(\F^N \cup \sigma(\cL_{k+1}|_{[a,b]}))$; note that we have $\bm\cL^N$ and $\bm\cL$ on the same probability space. Consider the following difference of expectations:
\begin{align}
\MoveEqLeft[2]
\left|\E_{\F^N}[F(\bm{\mf B}^N)\one_{\msf{ScSep}_k(\delta, [a,b], \bm{\mf B}^N, \cL^N_{k+1})}] - \E_{\F}[F(\bm{\mf B})\one_{\msf{ScSep}_k(\delta, [a,b], \bm{\mf B}, \cL_{k+1})}]\right| \label{e.break up of conditional probabilities}\\
&\leq \left|\E_{\F^N}[F(\bm{\mf B}^N)\one_{\msf{ScSep}_k(\delta, [a,b], \bm{\mf B}^N, \cL^N_{k+1})}] - \E_{\tilde{\F}^N}[F(\bm{\mf B}^N)\one_{\msf{ScSep}_k(\delta, [a,b], \bm{\mf B}^N, \cL_{k+1})}]\right| \nonumber\\
&\quad + \left|\E_{\tilde{\F}^N}[F(\bm{\mf B}^N)\one_{\msf{ScSep}_k(\delta, [a,b], \bm{\mf B}^N, \cL_{k+1})}] - \E_{\F}[F(\bm{\mf B})\one_{\msf{ScSep}_k(\delta, [a,b], \bm{\mf B}, \cL_{k+1})}]\right|. \nonumber
\end{align}
Since $F$ is bounded and the law of $\bm{\mf B}^N$ is the same under $\E_{\F^N}$ and $\E_{\tilde F^N}$, the first term of the righthand side of \eqref{e.break up of conditional probabilities} is bounded by 
$$C\cdot \left|\P_{\F^N}(\msf{ScSep}_k(\delta, [a,b], \bm{\mf B}^N, \cL^N_{k+1})) - \P_{\tilde{\F}^N}(\msf{ScSep}_k(\delta, [a,b], \bm{\mf B}^N, \cL_{k+1}))\right|.$$
Since $\cL^N_{k+1} \to \cL_{k+1}$ uniformly on $[a,b]$ almost surely, for each $N$ there exists a  $\tilde{\F}^N$-measurable $\eta_N>0$ such that $\eta_N\to 0$ and $\sup_{[a,b]}|\cL^N_{k+1} - \cL_{k+1}| \leq \eta_N$ for all $N$ almost surely. Since the endpoints of $\bm{\mf B}^N$ in the two probabilities in the previous display are the same, we can couple the processes to be equal. Then the previous display is upper bounded by
\begin{align*}
C\cdot \P_{\tilde{\F}^N}\left(\inf_{[a,b]}|\mf B^N_k - \cL_{k+1}-\delta| \leq \eta_N\right)
\end{align*}
Since the endpoints of $\mf B^N_k$ almost surely converge to points which have strictly positive separation with $\cL_{k+1}$ and $\eta_N$ are $\tilde{\F}^N$-measurable, it follows that the limit of these probabilities is almost surely $0$.

Now, we turn to the second term of the righthand side of \eqref{e.break up of conditional probabilities}. Recall that $\bm{\mf B}^N$ converges weakly to $\bm{\mf B}$ almost surely (recalling also that $\bm{\mf B^N}$ and $\bm{\mf B}$ are random laws). Further, from \eqref{e.scsep definition} it follows that $\one_{\msf{ScSep}_k(\delta, [a,b], \bm{\mf B}, \cL_{k+1})}$ is almost surely continuous as a function of $\bm{\mf B}$ with respect to the (random) law of $\bm{\mf B}$, for almost every instance of this law (i.e., almost surely with respect to the law of the values of the endpoints $(\cL_{i}(a), \cL_{i}(b))_{i=1}^k$). These facts, together with the Portmanteau theorem, imply that, almost surely,
\begin{align*}
\lim_{N\to\infty}\E_{\tilde{\F}^N}[F(\bm{\mf B}^N)\one_{\msf{ScSep}_k(\delta, [a,b], \bm{\mf B}^N, \cL_{k+1})}] &= \E_{\mc F}[F(\bm{\mf B})\one_{\msf{ScSep}_k(\delta, [a,b], \bm{\mf B}, \cL_{k+1})}].
\end{align*}
Thus overall we see that, almost surely as $N\to\infty$,
\begin{align*}
\E_{\F^N}[F(\bm{\mf B}^N)\one_{\msf{ScSep}_k(\delta, [a,b], \bm{\mf B}^N, \cL^N_{k+1})}] \to \E_{\F}[F(\bm{\mf B})\one_{\msf{ScSep}_k(\delta, [a,b], \bm{\mf B}, \cL_{k+1})}],
\end{align*}
and, by taking $F\equiv 1$, almost surely
$$\P_{\F^N}(\msf{ScSep}_k(\delta, [a,b], \bm{\mf B}^N, \cL^N_{k+1})) \to \P_{\F}(\msf{ScSep}_k(\delta, [a,b], \bm{\mf B}, \cL_{k+1})).$$

Next note that $\P_{\F}(\msf{ScSep}_k(\delta, [a,b], \bm{\mf B}, \cL_{k+1})) > 0$ almost surely, since $\cL_{k+1}$ is a continuous function and, by Theorem~\ref{t.uniform separation}, there is almost surely a strictly positive separation between the endpoints of $\mf B_i$ and $\mf B_{i+1}$ for $i\in\intint{1,k-1}$ as well as between those of $\mf B_{k}$ and $\mf \cL_{k+1}$. Since
\begin{align*}
\E_{\mc F^N}[F(\bm{\mf B}^N)\mid \msf{ScSep}_k(\delta, [a,b], \bm{\mf B}^N, \cL^N_{k+1})] = \frac{\E_{\F^N}[F(\bm{\mf B}^N)\one_{\msf{ScSep}_k(\delta, [a,b], \bm{\mf B}^N, \cL^N_{k+1})}]}{\P_{\F^N}(\msf{ScSep}_k(\delta, [a,b], \bm{\mf B}^N, \cL^N_{k+1}))}
\end{align*}
and analogously for $\E_{\mc F}[F(\bm{\mf B})\mid \msf{ScSep}_k(\delta, [a,b], \bm{\mf B}, \cL_{k+1})]$ (by the positivity observation), the second term on the righthand side of \eqref{e.break up of conditional probabilities} converges to $0$ and $N\to\infty$ and the proof is complete.
\end{proof}

\appendix

\section{The Yang-Baxter equation, S6V, and the $q$-Boson model}\label{s.yang-baxter}

In this section we give a self-contained proof of the partition function expression in \eqref{e.q-boson measure} and of Proposition~\ref{p.colored line ensembles} on the distributional relation of the top line of the colored Hall-Littlewood line ensemble and the colored height function of the colored S6V model. First, recall the description of the colored $q$-Boson model from Section~\ref{s.intro.vertex model} and the colored S6V model from Section~\ref{s.intro.cS6V}, with the domain restriction performed as in Section~\ref{s.line ensemble to colored S6V}. In particular, in both models we have a spectral parameter $z\in(0,1)$ fixed and a boundary condition $\sigma:\intint{1,N}\to\intint{1,N}$ where the arrow horizontally entering $(1,k)$ (for colored S6V) and the arrow horizontally entering $(-K,k)$ for all large enough $K$ (for colored $q$-Boson) are both of color $\sigma(k)$, for each $k\in\intint{1,N}$.
To explain the relation between the colored $q$-Boson model and the colored S6V model, we need the Yang-Baxter relation enjoyed by these models, which we introduce next. Then in Section~\ref{s.partition function expression} we establish the expression for the partition function of the colored $q$-Boson model and in Section~\ref{s.proof of matching} we give the proof of Proposition~\ref{p.colored line ensembles}.

\subsection{Yang-Baxter equation} The Yang-Baxter equation is an equality of partition functions of configurations consisting of three vertices. In our case, the first vertex has vertex weight given by the weights of the colored S6V model (Figure~\ref{f.R weights}) and the other two vertices have weights given by the $\bigL_u$ weights (Figure~\ref{f.L weights}). Let us write the weights from Figure~\ref{f.R weights} with spectral parameter $u\in \C$ (though in Section~\ref{s.intro.cS6V} we assumed $u\in[0,1)$, the weights are well-defined as algebraic expressions for any $u\in\C$) as $\bigR_u(a, i; b, j)$ with $a,i,b,j\in\intint{0,N}$, where recall $a,i,b,j$ respectively represent the colors of the arrows vertically entering, horizontally entering, vertically exiting, and horizontally exiting from the vertex. Unlike the convention in Section~\ref{s.intro.cS6V}, here we will take $0$ to represent the absence of an arrow, so as to match the setup of the colored $q$-Boson model. In this notation, the Yang-Baxter equation can be written as follows; it was originally due to \cite{Kulish1981,BAZHANOV1985321,Jimbo1986}, but the notation we have adopted is from \cite{borodin2018coloured} (it can also be verified directly from the formulas for the weights).

\begin{lemma}[Yang-Baxter equation, {\cite[Proposition 2.3.1]{borodin2018coloured}}]\label{l.yang-baxter}
Fix $a_1, i_1, b_2, j_2\in\intint{0,N}$ and $\bm I, \bm J\in \Z_{\geq 0}^N$. Let $x,y\in \C$. Then it holds that
\begin{align*}
\MoveEqLeft[8]
\sum_{\substack{b_1, j_1\in\intint{0,N},\\\bm K\in\Z_{\geq 0}^N}} \bigR_{y/x}(a_1, i_1; b_1, j_1)\, \bigL_x(\bm I, j_1; \bm K, j_2)\, \bigL_y(\bm K, b_1; \bm J, b_2)\\
&= \sum_{\substack{b_1, j_1\in\intint{0,N},\\\bm K\in\Z_{\geq 0}^N}}\,  \bigL_y(\bm I, a_1; \bm K, b_1)\,\bigL_x(\bm K, i_1; \bm J, j_1)\, \bigR_{y/x}(b_1, j_1; b_2, j_2);
\end{align*}
here $\bigL_u$ and $\bigR_u$ are zero if their arguments do not satisfy colored arrow conservation.
\end{lemma}

Diagrammatically, the Yang-Baxter equation can be written as follows (we adopt a style of depiction similar to \cite{borodin2018coloured}).
\begin{align}\label{e.yang-baxter}
\sum_{\substack{b_1, j_1\in\intint{0,N},\\\bm K\in\Z_{\geq 0}^N}}
\vcenter{\hbox{
\tikz[scale=0.7]{
\draw[densely dotted] (-1.25,0.25) arc (-45:45:{1/(2*sqrt(2))});
\draw[line width=1pt,->, >=stealth]
(-2.5,1) node[above,scale=0.7] {$i_1$} -- (-2,1)  -- (-1,0) node[below,scale=0.7] {$j_1$} -- (1,0) node[right,scale=0.7] {$j_2$};
\draw[line width=1pt,->, >=stealth]
(-2.5,0) node[above,scale=0.7] {$a_1$} -- (-2,0) -- (-1,1) node[above,scale=0.7] {$b_1$} -- (1,1) node[right,scale=0.7] {$b_2$};
\draw[line width=2pt,->, >=stealth]
(0,-1) node[below,scale=0.7] {$\bm I$} -- (0,2) node[above,scale=0.7] {$\bm J$};
\path[line width = 2pt, tips, -stealth] (0,-1) -- (0,0.7);
\node[scale=0.7] at (0.5, 0.5) {$\bm K$};
\node[left, scale=0.9] at (-2.6,1) {$x \rightarrow$};
\node[left, scale=0.9] at (-2.6,0) {$y \rightarrow$};
}
}
}
=
\sum_{\substack{b_1, j_1\in\intint{0,N},\\\bm K\in\Z_{\geq 0}^N}}
\vcenter{\hbox{
\tikz[scale=0.7]{
\draw[densely dotted] (0,0.25) arc (-45:45:{1/(2*sqrt(2))});
\draw[line width=1pt,->, >=stealth]
(-2.5,1) node[above,scale=0.7] {$i_1$} -- (-0.75,1) node[above,scale=0.7] {$j_1$}  -- (0.25,0)  -- (1,0) node[right,scale=0.7] {$j_2$};
\draw[line width=1pt,->, >=stealth]
(-2.5,0) node[above,scale=0.7] {$a_1$} -- (-0.75,0) node[below,scale=0.7] {$b_1$} -- (0.25,1) -- (1,1) node[right,scale=0.7] {$b_2$};
\draw[line width=2pt,->, >=stealth]
(-1.75,-1) node[below,scale=0.7] {$\bm I$} -- (-1.75,2) node[above,scale=0.7] {$\bm J$};
\node[scale=0.7] at (-1.25, 0.5) {$\bm K$};
\path[line width = 2pt, tips, -stealth] (-1.75,-1) -- (-1.75,0.7);
\node[left, scale=0.9] at (-2.6,1) {$x \rightarrow$};
\node[left, scale=0.9] at (-2.6,0) {$y \rightarrow$};
}
}
}
\end{align}
Here, the left and righthand sides mean the sum over the weights of the depicted configurations of the three vertices, and the arc between two edges indicates which of those two edges corresponds to exiting vertically (upper end of arc) and horizontally (lower end of arc).
The weights of the two vertices at the intersection of the two horizontal lines and the thick vertical line are the $\bigL_u$ weights (so they can hold an arbitrary number of arrows) with spectral parameter $u$ as indicated by $x$ or $y$ at the beginning of the line (which should be followed straight through the cross), while the weight of the vertex at the center of the cross is the $\bigR_z$ weights with spectral parameter $z = y/x$.

\newcommand{\Nval}{3}
\newcommand{\Mval}{4}
\newcommand{\thelength}{6.5}

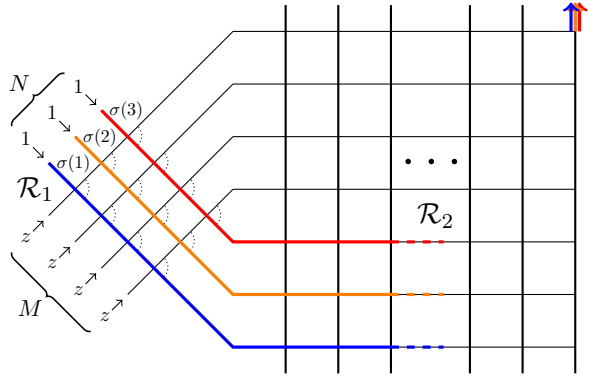
\begin{figure}[h]

\begin{tikzpicture}[scale=0.7]

  \foreach \x in {1,..., \Nval}
    \draw (-1+0.5*\x, 0.5*\x) -- ++(\Mval - 0.5*\x,-\Mval + 0.5*\x);

  \foreach \y in {1,..., \Mval}
    \draw (-1+0.5*\y, -0.5*\y) -- ++(\Mval - 0.5*\y, \Mval - 0.5*\y);

  \foreach \i in {1,...,\Mval}
  {
    \foreach \j in {1,...,\Nval}
      \draw[densely dotted] (0.5*\i+0.5*\j-1, -0.5*\i+0.5*\j) + (-40:{0.25)})  arc (-45:45:{0.25});
  }

  \foreach \y in {1,..., \Mval}
    \draw (\Mval-1, -\Mval + \y) -- ++(\thelength,0);

  \foreach \y in {1,..., \Nval}
   \draw (\Mval-1, \Nval - \y + 1) -- ++(\thelength,0);

  \foreach \x in {1,..., 3, 4.5, 5.5, ...,\thelength}
    \draw[thick] (\Mval-1 + \x, -\Mval + 0.5) -- ++ ($(0,\Mval+\Nval)$);

  \node[scale=1.8] at (\Mval-1+3.85, 0.5) {$\cdots$};

  \foreach \x in {1,..., \Nval}
    \draw[->]  (-1+0.5*\x - 0.3, 0.5*\x+ 0.3) -- node[scale=0.7, anchor = south east]{$1$} ++(0.2, -0.2);

  \foreach \x in {1,..., \Mval}
    \draw[->]  (-1+0.5*\x - 0.3, -0.5*\x- 0.3) -- node[scale=0.7, anchor = north east]{$z$} ++(0.2, 0.2);

  \foreach \x in {1, 2, 3}
    \node[scale=0.62, anchor=south] at  (-1+0.5*\x+0.475, 0.5*\x-0.3) {$\sigma(\x)$};

  \foreach \x/\thecolor in {1/blue,2/orange, 3/red}
  {
    \draw[\thecolor, very thick] (-1+0.5*\x, 0.5*\x) -- ++(\Mval - 0.5*\x,-\Mval + 0.5*\x) -- ++(3,0);
    \draw[\thecolor, very thick, dashed] (\Mval+2, -\Mval+\x) -- ++(1,0);
  }

  \draw[very thick, red, ->] (\Mval-1 + \thelength+0.08, \Nval) -- ++ (0,0.55);
  \draw[very thick, orange, ->] (\Mval-1 + \thelength, \Nval) -- ++ (0,0.55);
  \draw[very thick, blue, ->] (\Mval-1 + \thelength-0.08, \Nval) -- ++ (0,0.55);

  \draw [decorate,decoration={brace, mirror}, semithick]
  (-1+0.5 - 0.7, -0.5- 0.7) --node [pos=0.5, anchor = north east, scale=0.8] {$M$}  (-1+0.5*\Mval - 0.7, -0.5*\Mval- 0.7) ;

  \draw [decorate,decoration={brace}, semithick]
  (-1+0.5 -0.7, 0.5 +0.7) --node [pos=0.5, anchor = south east, scale=0.8] {$N$}  (-1+0.5*\Nval - 0.7, 0.5*\Nval + 0.7) ;

  \node[scale=1] at (-0.75, 0) {$\mc R_1$};
  \node[scale=1] at (6.85, -0.5) {$\mc R_2$};
\end{tikzpicture}
\caption{The model with the $M\times N$ rectangle of $\bigR_z$-weight vertices attached to the left. If any arrow tries to deviate from the configuration depicted except for finitely many columns at the right end (i.e., the portion to the right of the $\bm\cdots$), some arrow must travel infinitely far horizontally in the top $M$ rows, where the weight of a single arrow moving horizontally (with no other arrows incident) across a vertex is $z\in(0,1)$; thus any such configuration has weight zero and the depicted configuration must occur in all but finitely many columns (that is, to the left of the $\bm\cdots$). The arrows paths have not been shown in the region to the right of the $\bm\cdots$, where their trajectories are random.}
\label{f.cross on left}
\end{figure}

\subsection{Partition function of the colored $q$-Boson model} \label{s.partition function expression}
With the Yang-Baxter equation in hand, we may proceed to derive the expression for the partition function of the colored $q$-Boson model given in \eqref{e.q-boson measure}. To do this we consider an augmented vertex model, as in Figure~\ref{f.cross on left}, by the following informal description (see the next paragraph for a more formal one). We attach a copy of an $M\times N$ rectangular grid (which we denote $\mc R_1$), each of whose vertices has vertex weights given by the $\bigR_z$-weights (i.e., with spectral parameter $z$ as that of the colored $q$-Boson model), to the left of the original domain $\Z_{\leq 0}\times\intint{1,N+M}$ (which we denote $\mc R_2$). The augmented model has boundary conditions given by an arrow of color $\sigma(i)$ horizontally entering at the $i$\th row of the just added rectangle, no arrows vertically entering, no arrows horizontally exiting, and all arrows vertically exiting at $(0, N+M)$ in the original domain.

We are being somewhat informal in saying to attach the $M\times N$ rectangle $\mc R_1$ to the ``left'' of the semi-infinite domain $\mc R_2$; more formally, a configuration of the augmented model (with boundary condition $\sigma$) consists of an assignment of arrow configurations to the vertices of $\mc R_1$  as well as $\mc R_2$ such that (i) consistency and colored arrow conservation are satisfied within each, (ii) the entering boundary condition for $\mc R_1$ is given by $\sigma$, (iii) the exiting boundary condition for $\mc R_2$ is that $(\#\sigma^{-1}(\{j\}))_{j=1}^N$ vertically exits from $(0,N+M)$, and (iv) the exiting arrow configurations of $\mc R_1$ equals the boundary condition of $\mc R_2$ as shown in Figure~\ref{f.cross on left}, i.e., for all large enough $K$ and all $y\in\intint{1,N+M}$, the arrow configuration at $(-K, y)\in\mc R_2$ equals $(0^N, \tau_y; 0^n, \tau_y)$, where $\tau_y\in\intint{0,N}$ is the color of the arrow exiting $\mc R_1$ from the $y$\th vertex on its right and top boundaries (counted counterclockwise).

We first observe that the horizontally entering arrows to $\mc R_1$ are forced to go horizontally straight through: if any arrow does not pass straight through, at least one arrow will enter $\mc R_2$ through one of the top $M$ lines and then travel an infinite horizontal distance; any such configuration of arrows has weight zero since, under the $\bigL_z$ weights, the individual vertex weight of passing horizontally through a given vertex in the top $M$ rows is $z < 1$. Thus the partition function $\mc Z^{\mrm{aug}}_{\sigma}$ of the augmented model is the partition function $\mc Z^{\mrm{cB}}_\sigma$ of the original colored $q$-Boson model (with boundary condition $\sigma$)  multiplied by the weight of the configuration of arrows passing horizontally straight through $\mc R_1$; by Figure~\ref{f.R weights}, the latter is $\left((1-z)/(1-qz)\right)^{NM}$, so that
\begin{align}\label{e.Zaug relation}
\mc Z^{\mrm{aug}}_{\sigma} = \left(\frac{1-z}{1-qz}\right)^{NM}\cdot \mc Z^{\mrm{cB}}_\sigma.
\end{align}

Next, observe that there are $MN$ vertices present in the rectangle $\mathcal{R}_1$ in the augmented model. We move each of these vertices to the right, one at a time (starting at the rightmost corner vertex of the rectangle and ending at the leftmost one), all the way across $\mathcal{R}_2$. By the Yang-Baxter equation Lemma~\ref{l.yang-baxter} (see also \eqref{e.yang-baxter}), this does not change the partition function $\mathcal{Z}_{\sigma}^{\mathrm{aug}}$. Upon doing so, we obtain the situation depicted in Figure~\ref{f.cross on right}.
Observe that again the only configuration of arrows with non-zero weight is the one depicted, in which the arrows pass horizontally straight through (this uses that an arrow is entering at every one of the top $N$ rows). By Figure~\ref{f.L weights}, the configuration in $\mc R_2$ has weight $1$, and from Figure~\ref{f.R weights}, the empty configuration in $\mc R_1$ also has weight $1$. Thus, for any $\sigma$,
\begin{align}\label{e.Zaug = 1}
\mc Z^{\mrm{aug}}_\sigma = 1.
\end{align}

Note that the partition function was not changed between Figures~\ref{f.cross on left} and \ref{f.cross on right}. In other words, we can conclude from \eqref{e.Zaug relation} and \eqref{e.Zaug = 1} that $\mc Z^{\mrm{cB}}_\sigma = \left((1-qz)/(1-z)\right)^{NM}$ for any $\sigma$, completing the proof of the claim in \eqref{e.q-boson measure}.

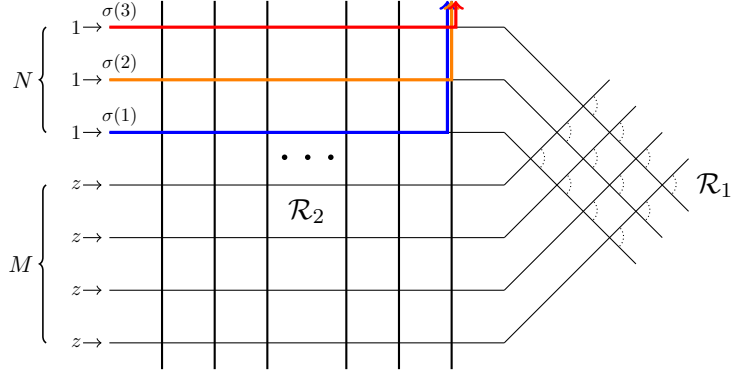
\begin{figure}[h]
\begin{tikzpicture}[scale=0.7]

  \foreach \x in {1,..., \Nval}
    \draw (\Mval+\thelength, \x) -- ++(0.5*\Mval + 0.5*\x, -0.5*\Mval - 0.5*\x);

  \foreach \y in {1,..., \Mval}
    \draw (\Mval+\thelength, -\y+1) -- ++(0.5*\Mval + 0.5*\y-0.5, 0.5*\Mval + 0.5*\y-0.5);

  \foreach \i in {1,...,\Mval}
  {
    \foreach \j in {1,...,\Nval}
    \draw[densely dotted] (\Mval+\thelength+0.5*\i+0.5*\j-0.5, -0.5*\i+0.5*\j+0.5) + (-40:{0.25)})  arc (-45:45:{0.25});
  }

  \foreach \y in {1,..., \Mval}
    \draw (\Mval-1, -\Mval + \y) -- ++(\thelength+1,0);

  \foreach \y in {1,..., \Nval}
   \draw (\Mval-1, \Nval - \y + 1) -- ++(\thelength+1,0);

  \foreach \x in {1,..., 3, 4.5, 5.5, ...,\thelength}
    \draw[thick] (\Mval-1 + \x, -\Mval + 0.5) -- ++ ($(0,\Mval+\Nval)$);

  \node[scale=1.8] at (\Mval-1+3.85, 0.5) {$\cdots$};

  \foreach \x in {1,..., \Nval}
    \draw[->]  (\Mval-1.5, \x) -- node[scale=0.7, anchor = east, left=2pt]{$1$} ++(0.35, 0);

  \foreach \x in {1,..., \Mval}
    \draw[->]  (\Mval-1.5, \x-1 -\Nval) -- node[scale=0.7, anchor = east, left=2pt]{$z$} ++(0.35, 0);

  %
  \foreach \x/\thecolor in {1/blue, 2/orange, 3/red}
    \draw[\thecolor, very thick, ->] (\Mval-1, \x) -- ++(\thelength + 0.08*\x-0.08*2,0) -- ++(0, \Nval-\x) -- ++(0,0.5);

  %

  %
  \foreach \x in {1,..., \Nval}
    \node[scale=0.65, anchor=south] at  (\Mval - 0.8, \x) {$\sigma(\x)$};

  %
  %
  %

  \draw [decorate,decoration={brace, mirror}, semithick]
  (\Mval-2.2, \Nval) --node [pos=0.5, anchor = east, scale=0.8, left=2pt] {$N$}  ++(0, -\Nval+1) ;

  \draw [decorate,decoration={brace, mirror}, semithick]
  (\Mval-2.2, 0) --node [pos=0.5, anchor = east, scale=0.8, left=2pt] {$M$}  ++(0, -\Mval+1) ;

  \node[scale=1] at (6.75, -0.5) {$\mc R_2$};
  \node[scale=1] at (14.5, 0) {$\mc R_1$};
\end{tikzpicture}
\caption{The configuration after all the crosses of $\mc R_1$ have been pulled all the way through $\mc R_2$ to the right by applying \eqref{e.yang-baxter} repeatedly. We note that now, due to the location where all the arrows must exit and by the condition that at most one arrow can be present on each horizontal edge, the arrows are forced into the single configuration of moving horizontally straight through till the very last column, where they move vertically.}
\label{f.cross on right}
\end{figure}

\subsection{Proof of Proposition~\ref{p.colored line ensembles}} \label{s.proof of matching}
Fix $M,N\in\N$ and a boundary condition $\sigma:\intint{1,N}\to\intint{1,N}$, and recall $h^{\mrm{S6V}}$ from \eqref{e.s6v height function}. We define a variant height function $\tilde h^{\mrm{S6V}}$ as follows. Let $(j^{\mrm{S6V}}_v, b^{\mrm{S6V}}_v)_{v\in\mc R_1}$ be sampled from the colored S6V measure on $\mc R_1$ with boundary condition $\sigma$, where (as in Section~\ref{s.s6v configuration}) $j^{\mrm{S6V}}_v$ and $b^{\mrm{S6V}}_v$ are the colors of the arrows exiting horizontally and vertically, respectively, from the vertex $v$. For $k\in\intint{1,N}$ and $y\in\intint{0,N-1}$, define
\begin{equation}\label{e.tilde h first definition}
\tilde h^{\mrm{S6V}}(k; y) := \#\left\{i > y : j^{\mrm{S6V}}_{(M,i)} \geq k\right\} = h^{\mrm{S6V}}(k,0; y, M).
\end{equation}
For $k\in\intint{1,N}$ and $y\in\intint{N,N+M}$, define
\begin{align*}
\tilde h^{\mrm{S6V}}(k; y) := \#\left\{i < N+M-y : b^{\mrm{S6V}}_{(i,N)} \geq k\right\}.
\end{align*}
In words, $\tilde h^{\mrm{S6V}}(k; y)$ counts the number of arrows of color at least $k$ which exit from the right and top boundaries of $\intint{1,M}\times\intint{1,N}$ from a site which is $(y+1)$\st or later when counted counterclockwise from the bottom right corner.

 Consider a sequence of Bernoulli paths $L^{(k)}:\intint{0,N+M}\to\Z$ such that, (i) $L^{(k)}(0) = \#\sigma^{-1}(\intint{k,N})$ and $L^{(k)}(N+M) = 0$, and (ii) $L^{(k)}(\bm\cdot) - L^{(k+1)}(\bm\cdot)$ is also a Bernoulli path. We will show~that 
\begin{equation}\label{e.matching to prove}
\begin{split}
\MoveEqLeft[8]
\P\left((\tilde h^{\mrm{S6V}}(k; y))_{k\in\intint{1,N}, y\in\intint{0,N+M}} = (L^{(k)}(y))_{k\in\intint{1,N}, y\in\intint{0,N+M}}\right)\\
&= \P\left((L^{\mrm{cHL}, (k)}_1(y))_{k\in\intint{1,N}, y\in\intint{0,N+M}} = (L^{(k)}(y))_{k\in\intint{1,N}, y\in\intint{0,N+M}}\right);
\end{split}
\end{equation}
this implies Proposition~\ref{p.colored line ensembles} by a projection of $y$ onto $\intint{1,N-1}$ and \eqref{e.tilde h first definition}.

To establish \eqref{e.matching to prove}, we first define an exiting boundary condition $(j'_{(-1, y)})_{y\in\intint{1,N+M}}$ in terms of $\smash{(L^{(k)}(y))_{k\in\intint{1,N}, y\in\intint{0,N+M}}}$ as follows: letting $\Delta(L,y) = L(y-1)-L(y) \in \{0,1\}$ for a Bernoulli path $L:\intint{0, N+M}\to \Z$, for $y\in\intint{1,N+M}$,
\begin{align*}
j'_{(-1,y)} := \min\left\{k\in\intint{1,N}: \Delta(L^{(k+1)}, y) = \Delta(L^{(k)}, y)-1\right\},
\end{align*}
where by convention $\Delta(L^{(N+1)}, y) = 0$ for all $y$ and $\min\emptyset = 0$.

Let $(j^{\mrm{cB}}_{v})_{v\in\mc R_2}$ be sampled according to the colored $q$-Boson measure (as in Section~\ref{s.higher spin configuration} and Definition~\ref{d.colored q-Boson}) with boundary condition $\sigma$.
 It follows from Definition~\ref{d.colored line ensemble} of $\bm L^{\mrm{cHL}, (k)}$ that the righthand side of \eqref{e.matching to prove} equals
\begin{align*}
\P\left(j^{\mrm{cB}}_{(-1,y)} = j'_{(-1,y)} \text{ for all } y\in\intint{1,N+M}\right).
\end{align*}
 This in turn equals the ratio of two partition functions, $\mc Z^{\mrm{cB}}_{\sigma}(\bm j')/\mc Z^{\mrm{cB}}_{\sigma}$ (see the left panel of Figure~\ref{f.partition functions}): the numerator is the partition function $\smash{\mc Z^{\mrm{cB}}_{\sigma}(\bm j')}$ of the colored $q$-Boson measure with (incoming) boundary condition $\sigma$ as well as the (outgoing) boundary condition $(j'_{(-1,y)}: y\in\intint{1,N+M})$ at $\{-1\}\times\intint{1,N+M}$, and the denominator is the partition function $\smash{\mc Z^{\mrm{cB}}_{\sigma} = ((1-qz)/(1-z))^{NM}}$ of the whole measure, as derived in Section~\ref{s.partition function expression}.

\begin{figure}[t]
\begin{tikzpicture}[scale=0.6]

  \node[scale=1] at (\Mval-5.4, 0) {$\mc Z^{\mrm{cB}}(\bm j')=$};
  \node[scale=1.1] at (\Mval-3.4,0) {$\displaystyle\sum$};

  \foreach \y in {1,..., \Mval}
    \draw (\Mval-1, -\Mval + \y) -- ++(\thelength-0.5,0);

  \foreach \y in {1,..., \Nval}
   \draw (\Mval-1, \Nval - \y + 1) -- ++(\thelength-0.5,0);

  \foreach \x in {1,..., 3, 4.5, 5.5}
    \draw[thick] (\Mval-1 + \x, -\Mval + 0.5) -- ++ ($(0,\Mval+\Nval)$);

  \node[scale=1.8] at (\Mval-1+3.85, 0.5) {$\cdots$};

  \foreach \x in {1,..., \Mval}
    \draw[->]  (\Mval-1.5, \x-1) -- node[scale=0.7, anchor = east, left=2pt]{$z$} ++(0.35, 0);

  \foreach \x in {1,..., \Nval}
    \draw[->]  (\Mval-1.5, \x-1 -\Nval) -- node[scale=0.7, anchor = east, left=2pt]{$1$} ++(0.35, 0);

  \foreach \x in {1,..., \Nval}
    \node[scale=0.8, anchor=south] at  (\Mval - 0.6, \x-\Nval-1) {$\sigma(\x)$};

  \foreach \x/\thecolor in {1/blue,2/orange, 3/red}
  {
    \draw[\thecolor, very thick] (\Mval - 1,-\Mval + \x) -- ++(3,0);
    \draw[\thecolor, very thick, dashed] (\Mval+2, -\Mval+\x) -- ++(1,0);
  }

  \foreach \y/\thecolor in {3/orange, 5/blue, 6/red}
    \draw[very thick, \thecolor, ->] (8.5, -4+\y) -- ++(0.6,0);
  %
  %
  %

  \draw [decorate,decoration={brace, mirror}, semithick]
  (\Mval-2.2, \Nval) --node [pos=0.5, anchor = east, scale=0.8, left=2pt] {$M$}  ++(0, -\Mval+1) ;

  \draw [decorate,decoration={brace, mirror}, semithick]
  (\Mval-2.2, -1) --node [pos=0.5, anchor = east, scale=0.8, left=2pt] {$N$}  ++(0, -\Nval+1) ;

  \node[scale=1] at (6.8, -0.5) {$\widetilde{\mc R}_2$};

  \foreach \y/\thelabel in {1/1,2/2,6/N+M-1,7/N+M}
    \node[scale=0.75, anchor=west] at (\Mval+\thelength-1.5, -4+\y) {$j'_{\thelabel}$};

  \foreach \y in {-1,0,1}
    \node[scale=1.2, anchor=west] at (\Mval+\thelength-1.5, \y*0.5) {$\bm .$};

\begin{scope}[shift={(17.4,0)}, scale=1.1]

  \node[scale=1] at (\Mval-7.8, 0) {$\mc Z^{\mrm{S6V}}(\tilde{\bm j}, \tilde{\bm b})=$};
  \node[scale=1.1] at (\Mval-5.6,0) {$\displaystyle\sum$};

  \foreach \x in {1,..., \Nval}
    \draw (-1+0.5*\x, 0.5*\x) -- ++(\Mval-1.5, -\Mval+1.5);

  \foreach \y in {1,..., \Mval}
    \draw (-1+0.5*\y, -0.5*\y) -- ++(\Mval - 2, \Mval - 2);

  \foreach \y/\thelabel in {1/1, 4/N}
    \node[anchor=south west, scale=0.8] at (\Mval-3+0.5*\y-0.1, \Mval-2-0.5*\y-0.1) {$\tilde b_\thelabel$};

  \foreach \x/\thelabel in {1/1, 3/M}
    \node[anchor=north west, scale=0.8, right=-0.5pt] at (\Mval-2.5+0.5*\x, -\Mval+1.5+0.5*\x) {$\tilde j_\thelabel$};

  \foreach \y in {-1,0,1}
    \node[scale=1.1, anchor=west] at (\Mval-2.5 + 0.95 + 0.15*\y, -\Mval+1.5 + 0.95 +\y*0.15) {$.$};

  \foreach \y in {-1,0,1}
    \node[scale=1.1, anchor=west] at (\Mval-2 + 0.15 + 0.18*\y, \Mval-3 -0.05 - \y*0.18) {$.$};

  \foreach \i in {1,...,\Mval}
  {
    \foreach \j in {1,...,\Nval}
      \draw[densely dotted] (0.5*\i+0.5*\j-1, -0.5*\i+0.5*\j) + (-40:{0.25)})  arc (-45:45:{0.25});
  }

  \foreach \x in {1,..., \Nval}
    \draw[->]  (-1+0.5*\x - 0.3, 0.5*\x+ 0.3) -- node[scale=0.7, anchor = south east]{$1$} ++(0.2, -0.2);

  \foreach \x in {1,..., \Mval}
    \draw[->]  (-1+0.5*\x - 0.3, -0.5*\x- 0.3) -- node[scale=0.7, anchor = north east]{$z$} ++(0.2, 0.2);

  \foreach \x in {1, 2, 3}
    \node[scale=0.6, anchor=south] at  (-1+0.5*\x+0.475, 0.5*\x-0.3) {$\sigma(\x)$};

  \foreach \x/\thecolor in {1/blue,2/orange, 3/red}
    \draw[\thecolor, very thick] (-1+0.5*\x, 0.5*\x) -- ++(0.5,-0.5);

  \draw[red, very thick] (1.5,0.5) -- ++(0.5,0.5);
  \draw[blue, very thick] (2,0) -- ++(0.5,0.5);
  \draw[orange, very thick] (2.5,-0.5) -- ++(0.5,-0.5);

  \draw [decorate,decoration={brace, mirror}, semithick]
  (-1+0.5 - 0.7, -0.5- 0.7) --node [pos=0.5, anchor = north east, scale=0.8] {$M$}  (-1+0.5*\Mval - 0.7, -0.5*\Mval- 0.7) ;

  \draw [decorate,decoration={brace}, semithick]
  (-1+0.5 -0.7, 0.5 +0.7) --node [pos=0.5, anchor = south east, scale=0.8] {$N$}  (-1+0.5*\Nval - 0.7, 0.5*\Nval + 0.7) ;

\end{scope}

\end{tikzpicture}
\caption{The definitions of the partition functions $\mc Z^{\mrm{cB}}_{\sigma}(\bm j')$ (left) and $\mc Z^{\mrm{S6V}}_{\sigma}(\tilde{\bm b}, \tilde{\bm j})$ (right); here the summation means to add the weights of all configurations which satisfy the depicted boundary conditions defined by $\sigma$ and $\bm j'$ in the left panel and $\sigma$, $\tilde{\bm b}$, and $\tilde{\bm j}$ in the right panel.}\label{f.partition functions}
\end{figure}
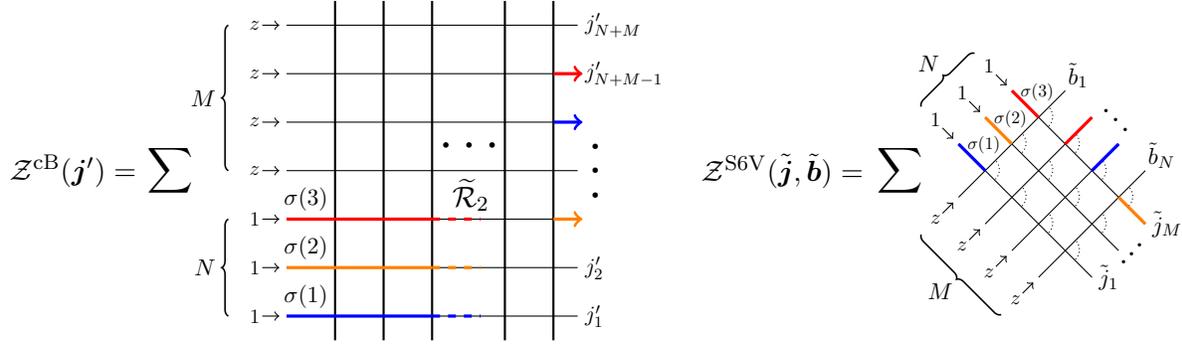

Recall that $(b^{\mrm{S6V}}_{(x,N)},j^\mrm{S6V}_{(M,y)}:x\in\intint{1,M}, y\in\intint{1,N})$ are sampled according to the colored S6V model with boundary condition $\sigma$. Let $\smash{\tilde b_{x} := j'_{(-1,N+M+1-x)}}$ for $x\in\intint{1,M}$ and $\smash{\tilde j_{y} := j'_{(-1,y)}}$ for $y\in\intint{1,N}$. Then the lefthand side of \eqref{e.matching to prove} equals
\begin{align*}
\P\left(b^{\mrm{S6V}}_{(x,N)} = \tilde b_{x}, j^\mrm{S6V}_{(M,y)} = \tilde j_{y} \text{ for all } x\in\intint{1,M}, y\in\intint{1,N}\right),
\end{align*}
which in turn equals the partition function $\mc Z^{\mrm{S6V}}_{\sigma}(\tilde{\bm b},\tilde{\bm j})$ (see the right panel of Figure~\ref{f.partition functions}) for colored S6V with (incoming) boundary condition $\sigma$ and the (outgoing) boundary condition $((\tilde b_{(x,N)}, \tilde j_{(M,y)}) : x\in\intint{1,M}, y\in\intint{1,N})$ (we do not divide by the  normalization constant of the model as a whole since it is $1$).

Thus, to establish \eqref{e.matching to prove}, we must show that 
\begin{equation}\label{e.line ensemble matching relation to show}
\frac{\mc Z^{\mrm{cB}}_{\sigma}(\bm j')}{\mc Z^{\mrm{cB}}_{\sigma}} = \mc Z^{\mrm{S6V}}_{\sigma}(\tilde{\bm b},\tilde{\bm j}).
\end{equation}
We establish this by using the Yang-Baxter equation \eqref{e.yang-baxter} via a similar argument as in the last section.
We start from an augmented model with $\mc R_1$ to the left of $\smash{\widetilde{\mc R}_2:= \Z_{\leq -1}\times\intint{1,N+M}}$, with (incoming) boundary conditions $\sigma$ and (outgoing) boundary conditions $(j'_{(-1,y)})_{y\in\intint{1,N+M}}$; see the left side of Figure~\ref{f.cross almost there}. By reasoning as in Section~\ref{s.partition function expression}, the partition function of this model is $\smash{((1-z)/(1-qz))^{NM}\cdot \mc Z^{\mrm{cB}}_{\sigma}(\bm j') = \mc Z^{\mrm{cB}}_{\sigma}(\bm j')/\mc Z^{\mrm{cB}}_{\sigma}}$.

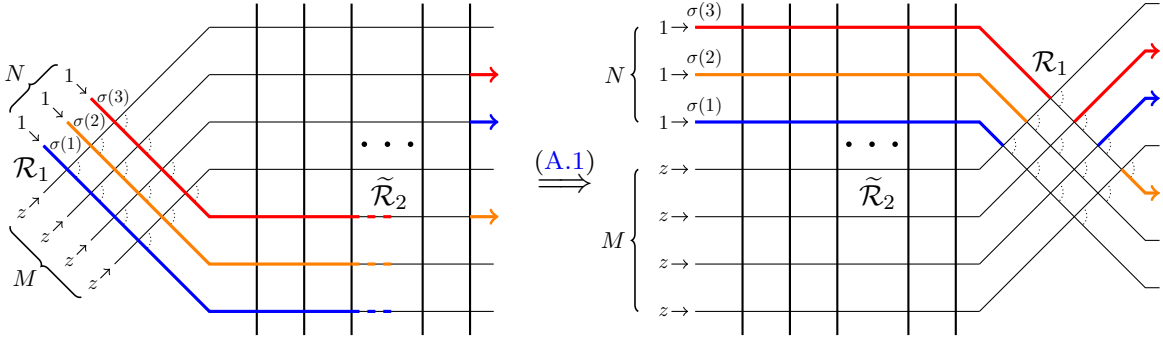
\begin{figure}[h]
\begin{tikzpicture}[scale=0.6]

  \foreach \x in {1,..., \Nval}
    \draw (-1+0.5*\x, 0.5*\x) -- ++(\Mval - 0.5*\x,-\Mval + 0.5*\x);

  \foreach \y in {1,..., \Mval}
    \draw (-1+0.5*\y, -0.5*\y) -- ++(\Mval - 0.5*\y, \Mval - 0.5*\y);

  \foreach \i in {1,...,\Mval}
  {
    \foreach \j in {1,...,\Nval}
      \draw[densely dotted] (0.5*\i+0.5*\j-1, -0.5*\i+0.5*\j) + (-40:{0.25)})  arc (-45:45:{0.25});
  }

  \foreach \y in {1,..., \Mval}
    \draw (\Mval-1, -\Mval + \y) -- ++(\thelength-0.5,0);

  \foreach \y in {1,..., \Nval}
   \draw (\Mval-1, \Nval - \y + 1) -- ++(\thelength-0.5,0);

  \foreach \x in {1,..., 3, 4.5, 5.5}
    \draw[thick] (\Mval-1 + \x, -\Mval + 0.5) -- ++ ($(0,\Mval+\Nval)$);

  \node[scale=1.8] at (\Mval-1+3.85, 0.5) {$\cdots$};

  \foreach \x in {1,..., \Nval}
    \draw[->]  (-1+0.5*\x - 0.3, 0.5*\x+ 0.3) -- node[scale=0.7, anchor = south east]{$1$} ++(0.2, -0.2);

  \foreach \x in {1,..., \Mval}
    \draw[->]  (-1+0.5*\x - 0.3, -0.5*\x- 0.3) -- node[scale=0.7, anchor = north east]{$z$} ++(0.2, 0.2);

  \foreach \x in {1, 2, 3}
    \node[scale=0.6, anchor=south] at  (-1+0.5*\x+0.475, 0.5*\x-0.3) {$\sigma(\x)$};

  \foreach \x/\thecolor in {1/blue,2/orange, 3/red}
  {
    \draw[\thecolor, very thick] (-1+0.5*\x, 0.5*\x) -- ++(\Mval - 0.5*\x,-\Mval + 0.5*\x) -- ++(3,0);
    \draw[\thecolor, very thick, dashed] (\Mval+2, -\Mval+\x) -- ++(1,0);
  }

  \foreach \y/\thecolor in {3/orange, 5/blue, 6/red}
    \draw[very thick, \thecolor, ->] (8.5, -4+\y) -- ++(0.6,0);

  \draw [decorate,decoration={brace, mirror}, semithick]
  (-1+0.5 - 0.7, -0.5- 0.7) --node [pos=0.5, anchor = north east, scale=0.8] {$M$}  (-1+0.5*\Mval - 0.7, -0.5*\Mval- 0.7) ;

  \draw [decorate,decoration={brace}, semithick]
  (-1+0.5 -0.7, 0.5 +0.7) --node [pos=0.5, anchor = south east, scale=0.8] {$N$}  (-1+0.5*\Nval - 0.7, 0.5*\Nval + 0.7) ;

  \node[scale=1] at (-0.75, 0) {$\mc R_1$};
  \node[scale=1] at (6.8, -0.5) {$\widetilde{\mc R}_2$};

\node[scale=1.2] at (10.5,0) {$\stackrel{\eqref{e.yang-baxter}}{\Longrightarrow}$};


\begin{scope}[shift={(10.25,0)}]

  \foreach \x in {1,..., \Nval}
    \draw (\Mval+\thelength-1.5, \x) -- ++(\Mval-0.5, -\Mval +0.5 ) -- ++(0.3,0);

  \foreach \y in {1,..., \Mval}
    \draw (\Mval+\thelength-1.5, -\y+1) -- ++(\Mval-0.5, \Mval - 0.5) -- ++(0.3,0);

  \foreach \i in {1,...,\Mval}
  {
    \foreach \j in {1,...,\Nval}
    \draw[densely dotted] (\Mval+\thelength+0.5*\i+0.5*\j-2, -0.5*\i+0.5*\j+0.5) + (-40:{0.25)})  arc (-45:45:{0.25});
  }

  \foreach \y in {1,..., \Mval}
    \draw (\Mval-1, -\Mval + \y) -- ++(\thelength-0.5,0);

  \foreach \y in {1,..., \Nval}
   \draw (\Mval-1, \Nval - \y + 1) -- ++(\thelength-0.5,0);

  \foreach \x in {2,..., 4, 5.5, 6.5, ...,\thelength}
    \draw[thick] (\Mval-1 + \x-1, -\Mval + 0.5) -- ++ ($(0,\Mval+\Nval)$);


  \node[scale=1.8] at (\Mval-1+3.825, 0.5) {$\cdots$};

  \foreach \x in {1,..., \Nval}
    \draw[->]  (\Mval-1.5, \x) -- node[scale=0.7, anchor = east, left=2pt]{$1$} ++(0.35, 0);

  \foreach \x in {1,..., \Mval}
    \draw[->]  (\Mval-1.5, \x-1 -\Nval) -- node[scale=0.7, anchor = east, left=2pt]{$z$} ++(0.35, 0);

    \foreach \x/\thecolor in {1/blue, 2/orange, 3/red}
    \draw[\thecolor, very thick] (\Mval-1, \x) -- ++(\thelength-0.5,0) -- ++(0.5*\x, -0.5*\x);

  \foreach \x in {1,..., \Nval}
    \node[scale=0.65, anchor=south] at  (\Mval-0.8, \x) {$\sigma(\x)$};

  \draw[very thick, red, ->] (\Mval+\thelength+0.5*\Mval+0.5*1-2, -0.5*\Mval+0.5*1+2.5) -- ++(1.5,1.5) -- ++(0.3,0);
  \draw[very thick, blue, ->] (\Mval+\thelength+0.5*\Mval+0.5*2-2, -0.5*\Mval+2.5) -- ++(1,1) -- ++(0.3,0);
  \draw[very thick, orange, ->] (\Mval+\thelength+0.5*\Mval+0.5*3-2, -0.5*\Mval - 0.5 + 2.5) -- ++(0.5,-0.5) -- ++(0.3,0);

  \draw [decorate,decoration={brace, mirror}, semithick]
  (\Mval-2.2, \Nval) --node [pos=0.5, anchor = east, scale=0.8, left=2pt] {$N$}  ++(0, 1-\Nval) ;

  \draw [decorate,decoration={brace, mirror}, semithick]
  (\Mval-2.2, 0) --node [pos=0.5, anchor = east, scale=0.8, left=2pt] {$M$}  ++(0, -\Mval+1) ;

  \node[scale=1] at (10.5, 2.25) {$\mc R_1$};
  \node[scale=1] at (6.85, -0.5) {$\widetilde{\mc R}_2$};\end{scope}
\end{tikzpicture}
\caption{Left: The augmented model with outgoing boundary conditions. Right: the same when $\mc R_1$ is moved all the way through $\widetilde{\mc R}_2$ using repeated applications of the Yang-Baxter equation \eqref{e.yang-baxter}; again, only a single configuration of arrows till $\mc R_1$ has non-zero weight.}
\label{f.cross almost there}
\end{figure}

Now we move the crosses in $\mc R_1$ to the right of $\widetilde{\mc R}_2$ using \eqref{e.yang-baxter} again and obtain the situation in the right side of Figure~\ref{f.cross almost there}.  Now, as depicted, there is a single configuration of non-zero weight for the arrows in $\widetilde{\mc R}_2$: they go horizontally straight through and are incident on $\mc R_1$ in the same way as the latter having incoming boundary condition $\sigma$, and outgoing boundary condition given by $(\tilde{\bm b},\tilde{\bm j})$. Thus the partition function of the right side of Figure~\ref{f.cross almost there} is simply $\mc Z^{\mrm{S6V}}_\sigma(\tilde b, \tilde j)$, since the weight of the fixed configuration in the $\smash{\widetilde{\mc R}_2}$ region is $1$. Since the partition function is not modified by applying \eqref{e.yang-baxter}, we obtain \eqref{e.line ensemble matching relation to show} and the proof is complete.

\section{Proofs of miscellaneous lemmas}\label{s.random gibbs}

In this section we prove various miscellaneous statements in the main text whose proofs were deferred. Lemmas~\ref{l.asep finite speed of propagation} and \ref{l.s6v finite pseed of propagation} are proved in Section~\ref{s.finite speed of propagation proofs}; Lemma \ref{l.random walk bridge fluctuation} in Section~\ref{s.two point fluc proof}; Lemma \ref{l.quantitative monotonicity for n=1} in Section~\ref{s.quant mono proof}; Lemma \ref{l.Gibbs on random interval} in Section~\ref{s.random gibbs proof}; and Lemma~\ref{l.non-int RW closeness} in Section~\ref{s.rw separation}. 

\subsection{Finite speed of propagation bounds for colored ASEP and colored S6V}\label{s.finite speed of propagation proofs}

\begin{proof}[Proof of Lemma~\ref{l.asep finite speed of propagation}]
The lemma follows from the claim that, with probability at least $1-\exp(-ct)$, no particle in colored ASEP (with packed initial condition) of color greater than $-\floor{2t}$ is at any location to the right of $\floor{4t}$ by time $t$. To see this, we merge all colors greater than or equal to $-\floor{2t}$ to be particles and all particles of color less than $-\floor{2t}$ to be holes, resulting in uncolored ASEP with step initial condition from location $\floor{2t}$, in which setting we must show the probability that the rightmost particle at time $t$ is at location greater than $\floor{4t}$ is at most $\exp(-ct)$. By ignoring left jump attempts, and letting $\smash{(X_i)_{i=1}^{\floor{2t}-1}}$ be a collection of independent and identically distributed exponential rate 1 random variables, this probability is upper bounded by $\smash{\P(\sum_{i=1}^{\floor{2t}-1} X_i < t)}$, which is indeed upper bounded by $\exp(-ct)$ by standard concentration results (e.g., \cite[Theorem 5.1]{janson2018tail}).
\end{proof}

\begin{proof}[Proof of Lemma~\ref{l.s6v finite pseed of propagation}]
By monotonicity of $\hssv(x,0;\bm\cdot, t)$ and since $\hssv(x,0;\bm\cdot, t)\leq \floor{\delta^{-1}t}-x+1$ for all $x$, it is sufficient to prove the case $z=\floor{4t}$. The event in the lemma with $z=\floor{4t}$ contains the event that no arrow of color at least $-\floor{2t}$ exits horizontally from a vertex in $I_t:=\{\floor{\delta^{-1}t}\}\times \intint{-\floor{2t},\floor{\delta^{-1}t} - \floor{4t}}$. To lower bound the probability of this event, we first merge all arrows of color at least $-\floor{2t}$ to color 1, and all other arrows to color $0$, obtaining an uncolored S6V model with step initial condition at $-\floor{2t}$. Then, equivalently, we need to upper bound the probability that the lowest arrow, i.e., the one starting at $(1, -\floor{2t})$, exits horizontally from $I_t$; we call this event $\msf E_t$. It can be described as the bottom-most arrow traveling at least $\floor{\delta^{-1}t}$ horizontally before traveling $\floor{\delta^{-1}t}-\floor{2t}$ vertically. The trajectory of the bottom-most arrow can be broken up into alternating intervals of horizontal travel and intervals of vertical travel. Note that the $i$\th interval of horizontal travel begins with a vertex configuration as in the fourth panel of Figure~\ref{f.R weights} and is followed by a random number $X_i$ of configurations as in the third panel of Figure~\ref{f.R weights}, where $X_i$ is a geometric random variable, with distribution given by $\P(X_i = k) = \delta^k(1-\delta)$ for $k\in\Z_{\geq 0}$; thus the $i$\th interval of horizontal travel has length $1+X_i$. The $X_i$ are independent for different $i$, and there are at most $\floor{\delta^{-1}t}-\floor{2t}$ many intervals of horizontal travel, so $\P(\msf E_t) \leq \P(\sum_{i=1}^{\smash{\floor{\delta^{-1}t} - \floor{2t}}}(1+ X_i) > \floor{\delta^{-1}t}) = \P(\sum_{i=1}^{\smash{\floor{\delta^{-1}t}-\floor{2t}}}X_i > \floor{2t})$. Using that $\smash{\E[\sum_{i=1}^{\floor{\delta^{-1}t}-\floor{2t}}X_i] = t(1-2\delta)/(1-\delta) < 2t}$, this is upper bounded by $\exp(-ct)$ for a $c>0$ independent of $\delta$, again by standard concentration inequalities (e.g., \cite[Theorem 2.1]{janson2018tail}).
\end{proof}

\subsection{Two-point fluctuation bound for Bernoulli random walk bridge}\label{s.two point fluc proof}
\begin{proof}[Proof of Lemma~\ref{l.random walk bridge fluctuation}]
It is enough to prove the two-point bound without the supremum, i.e., to prove that there exist $C$, $c>0$ such that for all $M>0$,
\begin{align*}
\P\left(|B(x) - B(y) - p(x-y)| > M|x-y|^{1/2}\right) \leq C\exp(-cM^2).
\end{align*}
From here the claim is proved by a standard chaining argument, i.e., doing a union bound over dyadically decreasing scales and using the above two point bound at each scale; e.g., we may apply \cite[Lemma~2.3]{dauvergne2023wiener} or imitate the proofs of Proposition~\ref{p.lower tail induction step} or Lemma~\ref{l.no quick drop uniform}.

Assume without loss of generality that $x>y$.
Observe that $\{B(w)-B(w-1)\}_{w=1}^{n}$ is a collection of $pn$ many $-1$s and $(1-p)n$ many 0s in uniformly random order. Clearly, $B(x)-B(y)$ is the number of $-1$s lying in positions $y+1, \ldots, x$. Observe that this has the same distribution as the number of $-1$s in the first $k:=x-y$ positions, which is the number $X$ of $-1$s drawn without replacement in the first $k$ draws from a collection of $n$ numbers of which $K := pn$ are $-1$s. In other words, $X$ is a hypergeometric distribution with parameters $n$, $K$, and $k$. By a theorem of Hoeffding \cite[Theorems 1 and 4]{hoeffding1963probability}, it follows that
\begin{align*}
\P\left(|X - pk| \geq Mk^{1/2}\right) \leq 2\exp(-2M^2),
\end{align*}
which completes the proof.
\end{proof}

\subsection{Quantitative monotonicity for non-intersecting Bernoulli random walk bridge}\label{s.quant mono proof}

\begin{proof}[Proof of Lemma~\ref{l.quantitative monotonicity for n=1}]
We construct a Markov chain whose value at any time $t$ is a pair of Bernoulli paths $(B^{\shortuparrow}_t, B^{\downarrow}_t)$ such that
\begin{align}\label{e.quantitative monotonicity}
B^{\downarrow}_t(x) \leq B^{\shortuparrow}_t(x) \leq B^{\downarrow}_t(x) + M
\end{align}
holds almost surely and the large $t$ marginal distributions of the chain are respectively the marginal distributions of $B^{\shortuparrow}$ and $B^{\downarrow}$. Any subsequential limiting coupling so produced will establish the lemma. We may assume $M\geq 1$ as otherwise there is nothing to prove.

For $*\in\{\shortuparrow,\downarrow\}$, define $B^{*}_0$ to be the minimal Bernoulli path from $(a,x^*)$ to $(b, y^*)$ which satisfies  $g(x) \leq B^*_0(x)\leq f(x)$ for all $x\in\intint{a,b}$ (the assumptions guarantee at least one such path exists). We claim that \eqref{e.quantitative monotonicity} holds at $t=0$. That the first inequality holds is immediate by minimality of the paths.

For the second inequality, suppose to the contrary that there exists $x_0\in\intint{a,b}$ such that $\smash{B^{\shortuparrow}_0(x_0) = B^{\downarrow}_0(x_0) + M+1}$, and we take $x_0$ to be the smallest such point. Let $y_0 > x_0$ be the first such that $\smash{B^{\shortuparrow}_0(y_0) = B^{\downarrow}_0(y_0) + M}$; it exists since the final separation of the paths at $b$ is at most $M$ and the separation changes by an element of $\{-1,0,1\}$ at each step. Now, by minimality of $x_0$ and $y_0$, $B^\shortuparrow$ has a flat-step starting at $x_0$ and a down-step starting at $y_0-1$. Let $\smash{\tilde B^\shortuparrow_0}$ be the path obtained from $B_0$ by flipping the flat-step at $x_0$ to a down-step and flipping the down-step at $y_0-1$ to a flat-step; then $\smash{\tilde B^\shortuparrow_0}$ is lower than $B^{\shortuparrow}_0$ and does not intersect $g$ (as $\smash{\tilde B^\shortuparrow_0(x)\geq B^\downarrow_0(x)\geq g(x)}$ for all $x\in\intint{x_0, y_0}$), contradicting the minimality of $\smash{B^{\downarrow}_0}$. Thus our claim is established, i.e., \eqref{e.quantitative monotonicity} holds at $t=0$.

Now we specify the Markov dynamics. At each time $t\geq 1$, an $X\in\intint{a,b}$ and $\sigma\in\{0, -1\}$ are chosen independently and uniformly at random. For each $*\in\{\shortuparrow,\downarrow\}$, if $B^*_{t-1}(X+1) = B^*_{t-1}(X-1) - 1$ (i.e., there is a down-step followed by a flat-step or vice versa from $X-1$ to $X+1$ in $B^*_{t-1}$), we set $B^*_t(X) = B^*_{t-1}(X-1) + \sigma$ and $B^*_t(x) = B^*_{t-1}(x)$ for all $x\neq X$, as long as the resulting $B^*_t$ satisfies $g \leq B^*_t\leq f$; if not, no change is made.

It is easy to check that \eqref{e.quantitative monotonicity} is satisfied under these dynamics, and that the dynamics are aperiodic and irreducible. It is also straightforward that the marginal distributions of $B^\shortuparrow$ and $B^{\downarrow}$ are stationary for the marginal Markov chains; thus the limiting stationary distribution has the correct marginals. This completes the proof.
\end{proof}

\subsection{Gibbs property on a certain random interval}\label{s.random gibbs proof}

\begin{proof}[Proof of Lemma~\ref{l.Gibbs on random interval}]
For a stochastic process $X:\intint{a,b}\to\Z$ taking values in $\Omega_{\intint{a,b}}$ (the space of Bernoulli paths on $\intint{a,b}$, as defined before Lemma~\ref{l.Gibbs on random interval}),  $t\in\intint{a,b}$, and $p\in(-1,0)$, let $\overline X(x) = X(x)-px$, $a'(X):= \max\{x\in\intint{a,t}:\overline X(x)\geq m\}$, and $b'(X):= \min\{x\in\intint{t,b}: \overline X(x) \geq m\}$.
With the notational convention for event names from Remark~\ref{r.event name convention} in mind, for such a process $X$, define $\msf E(x,y,z_1, z_2, X) := \msf{BdyVal}(x,y,z_1, z_2, X)\cap \msf{IntVal}(x,y, X)$, where
\begin{align*}
\msf{BdyVal}(x,y,z_1, z_2, X) &:= \{X(x) = z_1, X(y)=z_2\}, \quad\text{and}\\
\msf{IntVal}(x,y, X) &:= \{a'(X)=x, b'(X)=y\}.
\end{align*}
Let  $\F_{x,y} := \Fext(\{k\}, \intint{x,y}, \bm L)$, $a':= a'(L_k)$, $b' := b'(L_k)$, and $\msf E(x,y,z_1, z_2) := \msf E(x,y,z_1, z_2, L_k)$ for short. In the case that $a' = b' = t$, there is nothing to prove, so we my assume that this does not hold. By the definition of  $\F$ from Lemma~\ref{l.Gibbs on random interval}, we must show that, for any $x\in\intint{a,t}$, $y\in\intint{t,b}$, $z_1,z_2\in\Z$, and $\msf A\in\F_{x,y}$ (recalling the notation $\E^{z_1,z_2,\intint{x,y}}$ introduced just before Lemma~\ref{l.Gibbs on random interval}),
\begin{equation}\label{e.conditional statement to prove}
\begin{split}
\MoveEqLeft[-0.5]
\E\left[F(L_k|_{\intint{a',b'}}, a', b') \one_{\msf A\cap \msf E(x,y,z_1,z_2)}\right]\\
&= \E\Biggl[\frac{\E^{z_1,z_2,\intint{x,y}}\left[F(B,x,y)W(B, L_{k-1}, L_{k+1})\mid \overline B(s) < m \ \forall s\in\intint{x+1,y-1}\right]}{\E^{z_1,z_2,\intint{x,y}}\left[W(B, L_{k-1}, L_{k+1})\mid \overline B(s) < m \ \forall s\in\intint{x+1,y-1}\right]}\\
&\qquad\qquad \times\one_{\msf A\cap \msf E(x,y,z_1,z_2)}\Biggr],
\end{split}
\end{equation}
where the $W$ factors are evaluated on the interval $\intint{x,y}$.
First, if $z_1-px < m$ or $z_2-py < m$, $\msf E(x,y,z_1,z_2) = \emptyset$  by definition of $a'$ and $b'$, in which case the equality trivially holds. So we assume $z_1-px, z_2-py \geq m$. By conditioning on $\F_{x,y}$, we see that the lefthand side of the previous display equals
\begin{align}\label{e.F conditional expression}
\E\Bigl[\E_{\F_{x,y}}\bigl[F(L_k|_{\intint{x,y}}, x, y) \one_{\msf E(x,y,z_1,z_2)}\bigr]\one_{\msf A}\Bigr].
\end{align}
Adopt the shorthand $\msf{BdyVal}(x,y,z_1, z_2) :=\msf{BdyVal}(x,y,z_1, z_2, L_k)$  and recall that $\msf E(x,y,z_1,z_2) = \msf{BdyVal}(x,y,z_1,z_2)\cap\msf{IntVal}(x,y,z_1,z_2, L_k)$. Applying the Hall-Littlewood Gibbs property shows that the (inner) conditional expectation in \eqref{e.F conditional expression} equals 
\begin{align*}
\frac{\E^{z_1,z_2,\intint{x,y}}\left[F(B, x,y)\one_{\msf{IntVal}(x,y, B)} W(B, L_{k-1}, L_{k+1})\right]}{\E^{z_1,z_2,\intint{x,y}}\left[W(B, L_{k-1}, L_{k+1})\right]}\one_{\msf{BdyVal}(x,y,z_1, z_2)}.
\end{align*}
Note that $\msf{IntVal}(x,y, B) = \{\overline B(s) < m \ \forall s\in\intint{x+1,y-1}\}$ when $\msf{BdyVal}(x,y,z_1, z_2)$ holds, since $z_1-px, z_2-py \geq m$. Thus we can rewrite the previous display as (with the shorthand $\overline B<m$ for $\overline B(s) < m \ \forall s\in\intint{x+1,y-1}$)
\begin{align*}
\MoveEqLeft[0]
\frac{\E^{z_1,z_2,\intint{x,y}}\left[F(B, x,y) W(B, L_{k-1}, L_{k+1}) \mid \overline B < m \right]}{\E^{z_1,z_2,\intint{x,y}}\left[W(B, L_{k-1}, L_{k+1})\right]}\cdot \P^{z_1,z_2,\intint{x,y}}\left(\overline B < m \right)\one_{\msf{BdyVal}(x,y,z_1, z_2)}\\
&= \frac{\E^{z_1,z_2,\intint{x,y}}\left[F(B, x,y) W(B, L_{k-1}, L_{k+1}) \mid \overline B < m \right]}{\E^{z_1,z_2,\intint{x,y}}\left[W(B, L_{k-1}, L_{k+1}) \mid \overline B< m\right]}\\
&\qquad\qquad\times \frac{\E^{z_1,z_2,\intint{x,y}}\left[W(B, L_{k-1}, L_{k+1})\one_{\overline B< m}\right]}{\E^{z_1,z_2,\intint{x,y}}\left[W(B, L_{k-1}, L_{k+1})\right]} \one_{\msf{BdyVal}(x,y,z_1, z_2)}\\
&= \frac{\E^{z_1,z_2,\intint{x,y}}\left[F(B, x,y) W(B, L_{k-1}, L_{k+1}) \mid \overline B < m \right]}{\E^{z_1,z_2,\intint{x,y}}\left[W(B, L_{k-1}, L_{k+1}) \mid \overline B< m\right]}\\
&\qquad\qquad\times \P_{\F_{x,y}}\left(\overline L_k(s) < m\ \forall s\in\intint{x,y}\right)\one_{\msf{BdyVal}(x,y,z_1, z_2)}\\
&= \frac{\E^{z_1,z_2,\intint{x,y}}\left[F(B, x,y) W(B, L_{k-1}, L_{k+1}) \mid \overline B < m \right]}{\E^{z_1,z_2,\intint{x,y}}\left[W(B, L_{k-1}, L_{k+1}) \mid \overline B< m\right]}\\
&\qquad\times \P_{\F_{x,y}}\left(\msf{IntVal}(x,y, L_k)\right)\one_{\msf{BdyVal}(x,y,z_1, z_2)}.
\end{align*}
Since all the factors outside the $\P_{\F_{x,y}}$ factor are $\F_{x,y}$-measurable, and since by definition $\msf E(x,y,z_1,z_2) = \msf{IntVal}(x,y, L_k)\cap\msf{BdyVal}(x,y,z_1, z_2)$, the tower property of conditional expectation yields that the previous display equals
\begin{align*}
\E_{\F_{x,y}}\left[\frac{\E^{z_1,z_2,\intint{x,y}}\left[F(B, x,y) W(B, L_{k-1}, L_{k+1}) \mid \overline B < m \right]}{\E^{z_1,z_2,\intint{x,y}}\left[W(B, L_{k-1}, L_{k+1}) \mid \overline B< m\right]}\one_{\msf E(x,y,z_1, z_2)}\right].
\end{align*}
Substituting this in \eqref{e.F conditional expression}, using the tower property again, and that $\msf A$ is $\F_{x,y}$ measurable, yields \eqref{e.conditional statement to prove}. This completes the proof.
\end{proof}

\subsection{Uniform separation for Brownian and random walk bridges}\label{s.rw separation}

In this section we establish Lemma~\ref{l.non-int RW closeness} which, recall, states that conditional on a Bernoulli random walk bridge staying above a lower boundary curve $f$, it will maintain some on-scale uniformly positive separation from $f$ with high probability. The argument goes by proving a similar statement for Brownian bridges and then using the KMT coupling to obtain it for the Bernoulli case.

The Brownian bridge case is the following.

\begin{lemma}[Brownian bridge uniform separation]\label{l.non-int Br bridge closeness}
Fix $M>0$, $\sigma>0$, $p>0$ and $T\geq 2$. Let $U_1\in[-2T, -\frac{3}{2}T]$ and $U_2\in[\frac{3}{2}T, 2T]$. Let $\Bbr$ be a Brownian bridge from $(U_1N^{2/3}, z_1)$ to $(U_2N^{2/3}, z_2)$ of variance $\sigma^2>0$, for some $z_i$ such that $z_i - pU_iN^{2/3}\in[-MN^{1/3}, MN^{1/3}]$ for each $i\in\{1,2\}$. Also let $f:[U_1N^{2/3},U_2N^{2/3}]\to \R$ be a function with $f(U_iN^{2/3})\leq z_i$. Let $\{\Bbr\geq f\}$ denote $\{\Bbr(x) \geq f(x) \ \forall x\in [U_1,U_2]\}$ and suppose, for some $\rho >0$, that
\begin{align*}
\P\left(\Bbr \geq f\right) \geq \rho.
\end{align*}
Then for any $\varepsilon>0$ there exists $\delta = \delta(\varepsilon, \rho, M, \sigma, T)$ such that
\begin{align*}
\P\left(\inf_{x\in[-TN^{2/3},TN^{2/3}]}\left(\Bbr(x) - f(x)\right) \leq \delta N^{1/3} \midd \Bbr\geq f\right) \leq \varepsilon.
\end{align*}

\end{lemma}

Using the previous lemma and Lemma~\ref{l.KMT}, we give the straightforward proof of Lemma~\ref{l.non-int RW closeness}.

\begin{proof}[Proof of Lemma~\ref{l.non-int RW closeness}]
We will eventually set $\delta<\omega$ and will use this fact in the coming arguments. By Lemma~\ref{l.KMT}, there exists $N_0 = N_0(\varepsilon, \delta, T, M)$ such that there is a coupling of $\Brw$ and $\Bbr$ (of variance lower bounded by a positive constant depending on only $p$) such that, for $N\geq N_0$, with probability at least $1-\varepsilon\rho$, $\sup_{[U_1N^{2/3},U_2N^{2/3}]}|B_k -\Bbr| \leq \delta N^{1/3}$. Thus
\begin{align*}
\MoveEqLeft[2]
\P\left(\inf_{x\in[-TN^{2/3},TN^{2/3}]}\left(\Brw(x) - f(x)\right) \leq \delta N^{1/3} \midd \Brw\geq f\right)\\
&= \frac{\P\left(\inf_{x\in[-TN^{2/3},TN^{2/3}]}\left(\Brw(x) - f(x)\right) \leq \delta N^{1/3}, \Brw\geq f\right)}{\P\left(\Brw\geq f\right)}\\
&\leq \frac{\P\left(\inf_{x\in[-TN^{2/3},TN^{2/3}]}\left(\Bbr(x) - f(x)\right) \leq 2\delta N^{1/3}, \Bbr\geq f-\delta N^{1/3}\right) + \varepsilon\rho}{\P\left(\Bbr\geq f+\delta N^{1/3}\right) - \varepsilon\rho}.
\end{align*}
Since $\delta <\omega$, $\P\left(\Bbr\geq f+\delta N^{1/3}\right) \geq \rho$ by hypothesis. So (assuming $\varepsilon < \frac{1}{2}$ without loss of generality) the previous display is upper bounded by
\begin{align*}
2\rho^{-1}\cdot\P\left(\inf_{x\in[-TN^{2/3},TN^{2/3}]}\left(\Bbr(x) - f(x)\right) \leq 2\delta N^{1/3}, \Bbr\geq f-\delta N^{1/3}\right) + 2\varepsilon.
\end{align*}
By Lemma~\ref{l.non-int Br bridge closeness} we may pick $\delta < \omega$ with $\delta = \delta(\varepsilon, \rho, M, p)$ such that the probability is less than $\rho\varepsilon$, which completes the proof after relabeling $\varepsilon$.
\end{proof}

Next we prove Lemma~\ref{l.non-int Br bridge closeness} on the uniform on-scale separation of a Brownian bridge from its lower boundary condition. The idea is to condition on the entire bridge except for its value at zero; more precisely, one conditions on the shape of the curve so that, given the value of $\Bbr(0)$, the whole process is determined by an affine shift. The conditional distribution of $\Bbr(0)$ is explicit and the proof comes down to a calculation about the probability of this random variable being close to the lower edge of its support.

Similar ideas were introduced in \cite{hammond2016brownian} and have been used in a number of subsequent works, including \cite{calvert2019brownian,corwin2023exceptional,ganguly2023brownian}.

\begin{proof}[Proof of Lemma~\ref{l.non-int Br bridge closeness}]
We assume that $\sigma^2=1$ and $p=0$ without loss of generality. Recall that $B^{\mrm{Br}}$ is a Brownian bridge defined on $[U_1N^{2/3}, U_2N^{2/3}]$. For an interval $[a,b]\subseteq [U_1N^{2/3}, U_2N^{2/3}]$, we adopt the notation  $B^{\mrm{Br}, [a,b]};[a,b]\to\R$ to be $B^{\mrm{Br}}$ affinely shifted to equal zero at $a$ and $b$, i.e., for $x\in[a,b]$,
\begin{align}\label{e.bridge definition}
B^{\mrm{Br}, [a,b]}(x) := B^{\mrm{Br}}(x) - \frac{x-a}{b-a}B^{\mrm{Br}}(b) - \frac{b-x}{b-a}B^{\mrm{Br}}(a).
\end{align}

We let $\tilde \P$ be the probability measure $\P(\bm\cdot\mid \Bbr \geq f)$, so that under $\tilde \P$, $\Bbr$ is distributed as a Brownian bridge conditioned to stay above $f$. We consider the $\sigma$-algebra $\F$ generated by $B^{\mrm{Br},[U_1N^{2/3},0]}(\bm\cdot)$ and $B^{\mrm{Br},[0,U_2N^{2/3}]}(\bm\cdot)$. Then, conditionally on $\tilde \P_\F$, the only remaining randomness is in the value $\Bbr(0)$. The distribution of this random variable is that of a normal random variable of mean $\mu = \frac{1}{2}(z_1+z_2)$ and variance $\widetilde\sigma^2 = |U_1|\cdot U_2 N^{2/3}/(U_2+|U_1|)$ conditioned to stay above an $\F$-measurable random variable $\mrm{Cor}$ (short for corner, a name given to a similar quantity in \cite{hammond2016brownian}) defined by
\begin{align*}
\mrm{Cor} := \inf\left\{w: B^{\mrm{Br},w}(x) \geq f(x) \text{ for all }x\in[U_1N^{2/3},U_2N^{2/3}]\right\},
\end{align*}
where $B^{\mrm{Br},w}$ is the reconstruction of $\Bbr$ from $B^{\mrm{Br},[U_1N^{2/3},0]}$ and $B^{\mrm{Br},[0,U_2N^{2/3}]}$ when $B(0) = w$,~i.e.,
\begin{align*}
B^{\mrm{Br},w}(x) := \begin{cases}
\frac{U_1-xN^{-2/3}}{U_1}w + \frac{xN^{-2/3}}{U_1}z_1 + B^{\mrm{Br}, [U_1N^{2/3},0]}(x) & x\in[U_1N^{2/3},0]\\
\frac{U_2-xN^{-2/3}}{U_2}w + \frac{xN^{-2/3}}{U_2}z_2 + B^{\mrm{Br}, [0,U_2N^{2/3}]}(x) & x\in[0,U_2N^{2/3}].
\end{cases}
\end{align*}
Notice that for any $R$,
\begin{align*}
\tilde\P\left(\mrm{Cor} > RN^{1/3}\right)\leq \tilde\P\left(\Bbr(0) > RN^{1/3}\right)
&= \P\left(\Bbr(0) > RN^{1/3}\midd \Bbr \geq f\right)\\
&\leq \rho^{-1}\P\left(\Bbr(0) > RN^{1/3}\right)\\
&= \rho^{-1}\P\left(\mc N(\mu, \widetilde\sigma^2) > RN^{1/3}\right).
\end{align*}
Since $\mu\geq -MN^{1/3}$ and $\widetilde\sigma^2 \leq TN^{2/3}$,  there is $R = R(\varepsilon, \rho, M, T)$ (and since $T\geq 1$) such that the previous display is at most $\varepsilon$. We fix such a value of $R$ for the rest of the proof.

Now observe that, if $w\geq \mrm{Cor} + 3\delta N^{1/3}$, then
\begin{align*}
\inf_{x\in[-TN^{2/3},TN^{2/3}]} \left(B^{\mrm{Br}, w}(x) - f(x)\right) \geq \delta N^{1/3}.
\end{align*}
This is because the lefthand side is at least $0$ at $w=\mrm{Cor}$ by definition of $\mrm{Cor}$, $B^{\mrm{Br}, w}(x)$ increases with $w$, and  $\frac{U_1-xN^{-2/3}}{U_1}\cdot 3\delta N^{1/3} \geq \delta N^{1/3}$   when $x\in[-TN^{2/3},TN^{2/3}]$ since $U_1\in[-2T,-\frac{3}{2}T]$, and similarly for $U_2$. Thus it follows that (recall $\F$ from right after \eqref{e.bridge definition})
\begin{align*}
\MoveEqLeft[12]
\tilde\P\left(\inf_{x\in[-TN^{2/3},TN^{2/3}]}\left(\Bbr(x) - f(x)\right) \leq \delta N^{1/3}\right)\\
&= \tilde \E\left[\tilde \P_\F\left(\inf_{x\in[-TN^{2/3},TN^{2/3}]}\left(\Bbr(x) - f(x)\right) \leq \delta N^{1/3}\right)\right]\\
&\leq \tilde \E\left[\tilde \P_\F\left(\Bbr(0) \leq \mrm{Cor} + 3\delta N^{1/3}\right)\right].
\end{align*}
Letting $X\sim \mc N(0, 1)$, the previous display equals
\begin{align*}
\MoveEqLeft[0]
\tilde \E\left[\tilde \P_\F\left(\widetilde\sigma X+\mu \leq \mrm{Cor} + 3\delta N^{1/3} \mid \widetilde\sigma X+\mu \geq \mrm{Cor}\right)\right]\\
&\leq \tilde \E\left[\tilde \P_\F\left(X \leq \widetilde\sigma^{-1}(\mrm{Cor} -\mu) + 3\widetilde\sigma^{-1}\delta N^{1/3} \mid X \geq \widetilde\sigma^{-1}(\mrm{Cor} -\mu)\right)\one_{\mrm{Cor} \leq R N^{1/3}}\right] + \varepsilon.
\end{align*}
Now by Lemma~\ref{l.normal conditional prob montonicity}, the conditional probability in the previous display is upper bounded by the same with $\widetilde\sigma^{-1}(\mrm{Cor}-\mu)$ replaced by an upper bound on it. Since $\widetilde\sigma^2 \geq cTN^{2/3}$ and $\mu\geq -MN^{1/3}$, on the event $\mrm{Cor}\leq RN^{1/3}$, it holds that $\widetilde\sigma^{-1}(\mrm{Cor}-\mu) \leq c^{-1/2}T^{-1/2}(R+M)$. By the upper bound on $\widetilde\sigma^{-1}$, we also see that $3\widetilde\sigma^{-1}\delta N^{1/3} \leq 3c^{-1/2}T^{-1/2}\delta$. Thus, the previous display is upper bounded (using the bound on $3\widetilde\sigma^{-1}\delta N^{1/3}$ for the first line and Lemma~\ref{l.normal conditional prob montonicity} for the second) by
\begin{align*}
\MoveEqLeft[2]
\tilde \E\left[\tilde \P_\F\left(X \leq \widetilde\sigma^{-1}(\mrm{Cor} - \mu) + 3c^{-1/2}T^{-1/2}\delta \midd X \geq \widetilde\sigma^{-1}(\mrm{Cor}-\mu)\right)\one_{\mrm{Cor}\leq RN^{1/3}}\right] + \varepsilon\\
&\leq \tilde \E\left[\tilde \P_\F\left(X \leq c^{-1/2}T^{-1/2}(R + M + 3\delta) \midd X \geq c^{-1/2}T^{-1/2}(R + M)\right)\right] + \varepsilon.
\end{align*}
It is now clear using standard estimates on normal probabilities that, for $\delta$ small enough depending on $R$, $M$, and $T$ (therefore on $\varepsilon$, $\rho$,  $M$, and $T$), the previous display is upper bounded by $2\varepsilon$. This completes the proof.
\end{proof}

\section{LPP maximizer location asymptotics}\label{s.G_k asymptotics}

In this section, we give the proof of Proposition~\ref{p.maximizer location}. Before doing so, we 
set up some preliminary lemmas that will be needed. Recall the definition of $Z^N_k$ as the maximizer of the maximum over $z$ of \eqref{e.scaled approximate LPP representation}.

\subsection{Monotonicity properties}
We need a monotonicity property of $Z^N_k$ and of LPP values. We start with the first; its proof is the same as that of \cite[Lemma 3.6]{dauvergne2018directed} and is omitted.

\begin{lemma}\label{l.Z_k monotone}
For fixed $N$, $x$, and $y$, it almost surely holds that $Z^N_k(x,y) \geq Z^N_{k+1}(x,y)$ for all $k\in\intint{1,N^{1/6}-2}$.
\end{lemma}

\begin{lemma}[{\cite[Lemma~3.3]{dauvergne2018directed}}]\label{l.modified LPP monotone}
Let $f=(f_1, \ldots, f_n)$ be a sequence of continuous functions with $f_i:\R\to\R$. Then for any $x\in\R$,
\begin{align*}
z\mapsto f[(z,k)\to (x,1)] + f_k(z)
\end{align*}
is non-increasing.
\end{lemma}

\subsection{Prelimiting quantities}

Define $F^N_k$, $G^N_k$ and $H^N$ by
\begin{align*}
F^N_k(x,z) &:= \sup_{w\leq z}\left(\cL^N_{N^{1/6}}(x; w) + \bm \cL^{(1), N}[(w,N^{1/6}-1)\to (z,k)]\right) - \cL^{(1),N}_{k}(z)\\
G^N_k(z,y) &:= \bm \cL^{(1),N}[(z,k)\to(y,1)] + \cL^{(1),N}_{k}(z) \\
H^N(x,y) &:= \S^N(x;y) = \cL^N_{1}(x;y);
\end{align*}
note that the definition of $G^N_k(z,y)$ requires $z\leq y$. Also define
\begin{equation}\label{e.limiting FGH}
\begin{split}
G_k(z,y) &:= \bm\cP[(z,k)\to(y,1)] + \cP_{k}(z) \\
H(x,y) &:= \S(x;y).
\end{split}
\end{equation}

 Observe that $G^N_k\to G_k$ (by Theorem~\ref{t.line ensemble convergence to parabolic Airy}). The quantities in \eqref{e.limiting FGH} are defined in this way to obtain the monotonicity stated in Lemma~\ref{l.modified LPP monotone}, i.e., $G^N_k(z,y)$ and $G_k(z,y)$ are non-increasing in $z$ by Lemma~\ref{l.modified LPP monotone}, a property not possessed by $\bm\cP[(z,k)\to(y,1)]$ or $\bm \cL^{(1),N}[(z,k)\to(y,1)]$. We make use of a similar observation to note that $G^N_k(z,y)$ is bounded above in $z$:

\begin{lemma}\label{l.G^N_k bounded}
Let $k\in\N$. For any $y$ fixed and all $z\leq y$, it holds that $G^N_k(z,y) \leq \mf o(\sqrt{k})$.
\end{lemma}

\begin{proof}
Consider $G^N_k(z,y) - \cL^{(1),N}_1(y) = \bm \cL^{(1),N}[(z,k)\to(y,1)] + \cL^{(1),N}_{k}(z) - \cL^{(1),N}_1(y)$. By recalling the definition of the LPP value, one can write this expression as a maximum of a sum of differences of the values of $\cL^{(1),N}_j$ at the ``jumps'' of the paths whose weights are being maximized (also called the ``sum of gaps representation'' in \cite{dauvergne2018directed}):
\begin{align*}
G^N_k(z,y) - \cL^{(1),N}_1(y) = \sup_{z\leq t_k <  \ldots  < t_2 < y} \sum_{i=2}^k \left(\cL^{(1),N}_i(t_i) - \cL^{(1),N}_{i-1}(t_i)\right).
\end{align*}
Since $\cL^{(1),N}_i(w) \leq \cL^{(1),N}_{i-1}(w)$ for all $i$ and $w$, the supremum is non-positive. So we see that $G^N_k(z,y) \leq \cL^{(1),N}_1(y)$ for all $z$ simultaneously. Since $\cL^{(1),N}_1(y) = \mf o(\sqrt{k})$ by Theorem~\ref{t.line ensemble convergence to parabolic Airy}, the proof is complete.
\end{proof}

The proof of Proposition~\ref{p.maximizer location} will also need some asymptotic information about $G_k$, which we record in the following lemma. We defer its proof to Section~\ref{s.actual G_k asymptotics} as it requires the introduction of the prelimiting model of Brownian LPP. Recall the meaning of the notation $\mf o$ from \eqref{e.Borel-Cantelli notation}.

\begin{lemma}\label{l.G_k asymptotic}
Fix $z<y$. Then it holds that
\begin{align*}
G_k\left(z\sqrt{k},y\sqrt{k}\right) = -y^2k - \frac{\sqrt{k}}{y-z} + \mf o(\sqrt{k}).
\end{align*}
\end{lemma}

Now we give the proof of Proposition~\ref{p.maximizer location}, which is an adaptation of the argument from \cite[Lemma 7.1]{dauvergne2018directed}.

\begin{proof}[Proof of Proposition~\ref{p.maximizer location}]
Note that $H^N(x,y) \stackrel{d}{=}  \cL^{(1),N}_1(y-x)$ for fixed $x,y$ by \eqref{e.rescaled two parameter stationarity} and $\cL^{(1),N}_1\stackrel{d}{\to} \cP_1$ as $N\to\infty$ by Theorem~\ref{t.line ensemble convergence to parabolic Airy}. We added $\cL^{(1),N}_{k}(z)$ in the definition of $G^N_k$ because we will need $G^N_k$ to converge to $G_k$. We subtracted the same quantity from $F^N_k$ so that the following holds for $x>0$ and $y\geq z$, which follows from \eqref{e.scaled approximate LPP representation}:
\begin{align}\label{e.FGH triangle}
H^N(x,y) &\geq F^N_k(x,z) + G^N_k(z,y) + \mf o(N^{-1/11}).
\end{align}
Now the lefthand side is tight in $N$ and independent of $k$ by the tightness of $\cL^{(1)}_N$ (Theorem~\ref{t.line ensemble convergence to parabolic Airy}). We claim there exists $c>0$ such that, for any $\varepsilon>0$, if $|z+\sqrt{k/2x}| > \varepsilon\sqrt{k}$, then the righthand side is at most $-c\varepsilon^2\sqrt{kx}/2$. By \eqref{e.FGH triangle} and the tightness of $H^N(x,y)$, thus would imply that $Z^N_k$, the maximizer over $z$, must lie within $\varepsilon\sqrt{k}$ of $-\sqrt{k/2x}$, establishing the proposition.

For $z$ and $y$ in a fixed compact subset independent of $N$ (but which may depend on $k$), we know from Theorem~\ref{t.line ensemble convergence to parabolic Airy} that $\smash{G^N_k(z,y) \stackrel{d}{\to} G_k(z,y)}$ as $N\to\infty$. Further, we know from Lemma~\ref{l.G_k asymptotic} and by the convergence of $G^N_k$ to $G_k$ that, for each fixed $z$ and $y$,
\begin{align}\label{e.G_N upper bound}
G^N_k\left(z\sqrt{k},y\sqrt{k}\right) = -y^2k - \frac{\sqrt{k}}{y-z} + \mf o(\sqrt{k}).
\end{align}

We can use this information to obtain an upper bound on $F^N_k(x,z\sqrt{k})$. First, we know from \eqref{e.scaled approximate LPP representation} and the definitions of $F^N_k$, $G^N_k$, and $H^N$ that, for all $x,y,z$,
\begin{align*}
F^N_k(x,z\sqrt{k}) \leq H^N(x, y\sqrt{k}) - G^N_k(z\sqrt{k}, y\sqrt{k}) +  \mf o(N^{-1/11}).
\end{align*}
Define
\begin{equation}\label{e.tilde F definition}
\tilde F^N_k(x,z\sqrt{k}) := \inf_{y} \left(H^N(x, y\sqrt{k}) - G^N_k(z\sqrt{k}, y\sqrt{k})\right),
\end{equation}
so that $z\mapsto \tilde F(x,z\sqrt{k})$ is non-decreasing, since, by Lemma~\ref{l.modified LPP monotone}, $G^N_k(z\sqrt{k}, y\sqrt{k})$ is monotonically non-increasing in $z$. Note also that
\begin{equation}\label{e.tilde F and F relation}
F^N_k(x,z\sqrt{k}) \leq \tilde F^N_k(x,z\sqrt{k}) + \mf o(N^{-1/11})
\end{equation}
for all $x  > 0 $ and $z$. Our next step is to upper bound $\tilde F^N_k(x,z\sqrt{k})$.

Since for fixed $x,y$ we know $H^N(x,y)\stackrel{d}{\to} \cP_1(y-x)$ as $N\to\infty$, $x\mapsto\cP_1(x)+x^2$ is stationary, and exponentially decaying tail bounds for $\cP_1(0)$ (i.e., the GUE Tracy-Widom distribution) are well-known (e.g., \cite[Exercise 3.8.3]{anderson2010introduction}), it follows that $H^N(x, y\sqrt{k}) = -(x-y\sqrt{k})^2 + \mf o(\sqrt{k})$. Combining this with the estimate from \eqref{e.G_N upper bound} of $G^N_k(z\sqrt{k}, y\sqrt{k})$ yields that, for any fixed $y \geq z$,
\begin{align*}
H^N(x,y) + G^N_k(z\sqrt{k}, y\sqrt{k})
&= -(x-y\sqrt{k})^2 - \left(-y^2k - \frac{\sqrt{k}}{y-z}\right) + \mf o(\sqrt{k})\\
&= 2xy\sqrt{k} + \frac{\sqrt{k}}{y-z} + \mf o(\sqrt{k}),
\end{align*}
the second equality by absorbing $-x^2$ into $\mf o(\sqrt{k})$ since $x$ is uniformly bounded independently of $k$. This is optimized at $y = z+\frac{1}{\sqrt{2x}}$. Substituting this into the previous display and using \eqref{e.tilde F definition} yields that, for each fixed $x>0$ and $z$,
%
%
\begin{align}\label{e.F_N upper bound}
\tilde F^N_k(x, z\sqrt{k}) \leq 2\sqrt{2kx} + 2xz\sqrt{k} + \mf o(\sqrt{k}).
\end{align}
With this information, we can show that the maximizer in \eqref{e.scaled approximate LPP representation} equals $-\sqrt{k/2x} + \mf o(\sqrt{k})$. Indeed,
\eqref{e.scaled approximate LPP representation} can be written as
\begin{align*}
H^N(x,y) = \sup_{z\leq y}\left(F^N_k(x,z) + G^N_k(z,y)\right) + \mf o(N^{-1/11}).
\end{align*}
 Now, since $x$ and $y$ are bounded independently of $k$, the lefthand side of the previous display is  $\mf o(\sqrt{k})$ (by Theorem~\ref{t.line ensemble convergence to parabolic Airy}). So at $z=Z^N_k$ the maximizer, we must have
\begin{align*}
F^N_k(x,Z^N_k) + G^N_k(Z^N_k,y) = \mf o(\sqrt{k}).
\end{align*}
We first show that $Z^N_k$ does lie in a compact set. By Lemma~\ref{l.G^N_k bounded}, for any $z\leq y$, $G^N_k(z, y) \leq \mf o(\sqrt{k})$, and from \eqref{e.F_N upper bound}, $\tilde F^N_k(x, -M\sqrt{k}) \leq 2\sqrt{2kx} -2xM\sqrt{k} + \mf o(\sqrt{k})$. Thus we see from \eqref{e.tilde F and F relation} and the monotonicity of $\tilde F^N_k$ that, for all $z\leq -M\sqrt{k}$,
\begin{align*}
F^N_k(x, z) + G^N_k(z, y) \leq \tilde F^N_k(x, -M\sqrt{k}) + \mf o(\sqrt{k}) \leq \left(2\sqrt{2x} - 2xM\right)\sqrt{k} + \mf o(\sqrt{k}).
\end{align*}
So for $M > \sqrt 2 x^{-1/2}$ it follows that $Z^N_k\geq -M\sqrt{k}$ as we know that $H^N(x,y) = \mf o(\sqrt{k})$.

Having reduced to a compact set, we will now show that if $|z+\sqrt{k/2x}| > \varepsilon\sqrt{k}$, then $F^N_k(x,z) + G^N_k(z,y)$ is  smaller than $-c\varepsilon^2\sqrt{k}$. More precisely, we will show that for some absolute constant $c>0$,
\begin{align*}
\sup_{\substack{z: z\leq y\\|z + \sqrt{k/2x}| > \varepsilon\sqrt{k}}} \left(\tilde F^N_k(x,z) + G^N_k(z,y)\right) \leq -c\varepsilon^2 \sqrt{k} + \mf o(\sqrt{k});
\end{align*}
this suffices to establish our claim since, by \eqref{e.tilde F and F relation}, $F^N_k(x,z) \leq \tilde F^N_k(x,z) + \mf o(N^{-1/11})$. Recall from just after \eqref{e.tilde F definition} that $\tilde F^N_k(x, z\sqrt{k})$ is monotonically non-decreasing in $z$.
We first reduce the previous display to a supremum over a finite set using the monotonicity in $z$ of $\tilde F^N_k$ and $G^N_k$. Letting
\begin{align*}
\floor{x}_{N,k} &:= \max\left\{ w< x: w\in \tfrac{1}{4}\sqrt{\frac{k}{x}}\varepsilon^2\Z \right\}\\
\ceil{x}_{N,k} &:= \min\left\{ w> x: w\in \tfrac{1}{4}\sqrt{\frac{k}{x}}\varepsilon^2\Z \right\},
\end{align*}
the monotonicity in $z$ of $\tilde F^N_k$ and $G^N_k$ implies that it is sufficient to show that
\begin{align}\label{e.F G bound sufficient to show}
\sup_{\substack{z: -M\sqrt{k} \leq z\leq y\\|z + \sqrt{k/2x}| > \varepsilon\sqrt{k}}} \tilde F^N_k(x,\ceil{z}_{N,k}) + G^N_k(\floor{z}_{N,k},y) \leq -c\varepsilon^2 \sqrt{k}.
\end{align}
Combining the bounds on $G^N_k$ and $\tilde F^N_k$ from \eqref{e.G_N upper bound} and \eqref{e.F_N upper bound} respectively (taking $y$ in the $G^N_k$ bound to be $yk^{-1/2}$ here, and using the fact that $y$ is bounded independently of $k$), we obtain, for each fixed $x>0$ and $z\leq y$,
\begin{equation}\label{e.F G final inequality}
\begin{split}
    \MoveEqLeft[2]
\tilde F^N_k(x,\ceil{z}_{N,k}) + G^N_k(\floor{z}_{N,k},y)\\
&\leq 2\sqrt{2kx} + 2x(z+\tfrac{1}{4}\varepsilon^2x^{-1/2}\sqrt{k}) +\frac{k}{z-\tfrac{1}{4}\varepsilon^2\sqrt{k/x}} + \mf o(\sqrt{k}).
\end{split}
\end{equation}
If one substitutes $\varepsilon=0$, the righthand side is maximized at $z=-\sqrt{k/2x}$ (under the condition $z\leq y$), at which value it equals $0$. Further, by the concavity of the righthand side as a function in $z$ (and the negativity of the third derivative) and using Taylor's theorem, if $|z+\sqrt{k/2x}|>\varepsilon\sqrt{k}$, the righthand side is at most $-c\varepsilon^2\sqrt{k}$. Since these bounds also apply simultaneously for all the finite number of $z$'s in \eqref{e.F G bound sufficient to show}, we obtain the asymptotic for $Z^N_k$.

The tightness of $\{Z^N_k(x,y)\}_{N=1}^\infty$ for each $k$ follows from the asymptotic and the monotonicity property $y \geq Z^N_1(x,y) \geq Z^N_2(x,y) \geq  \ldots \geq Z^N_{N^{1/6}-1}(x,y)$ from Lemma~\ref{l.Z_k monotone}.
\end{proof}

\begin{remark}
Let remark briefly on the difference of the above argument compared to \cite{dauvergne2018directed}. There, an analog of the $F^N_k$, $G^N_k$ and $H^N_k$ quantities were defined in a model known as Brownian last passage percolation (introduced here in Section~\ref{s.actual G_k asymptotics} ahead). In their context, asymptotics for $F^N_k$ were available (via estimates for a limiting analog $F_k$), and these were used to derive an upper bound on $G^N_k$ to then upper bound the analog of \eqref{e.F G final inequality}. Here, we do not \emph{a priori} have such estimates for $F^N_k$ as we do not know it weakly converges to $F_k$; instead, we know the convergence of $G^N_k$ to its analog in the parabolic Airy line ensemble, where the previously mentioned estimates for $F_k$ apply. These will lead to the estimates on $G_k$ recorded in Lemma~\ref{l.G_k asymptotic}, yielding the estimates for $G^N_k$ that were used. We essentially reversed the argument in \cite{dauvergne2018directed} to obtain estimates for $F^N_k$ that allowed us to then bound \eqref{e.F G final inequality}.
\end{remark}

\begin{remark}

Recall from \eqref{e.rescaled h general definition} and \eqref{e.rescaled two parameter stationarity} that $L_1^{(k), N}$ is equal in law to a translation of $L_1^{(1), N}$, which was important for our arguments in Section~\ref{s.airy sheet}. This statement is false for the lower-indexed curves, namely, $L_j^{(k), N}$ does not have the same law as that of the corresponding translation of $\smash{L_j^{(1), N}}$ for $j > 1$. This is closely related to the fact that, as noted in Remark~\ref{r.no color merging with zero}, we cannot merge colors 1 and 0 in the colored $q$-Boson model. It can also be seen from the fact that, if such an equality in law were true, then it would imply precise asymptotics for $F_k^N$ (since the asymptotics of $\smash{L_j^{(1), N}}$, converging to the $j$\th curve of the Airy line ensemble, are well understood, e.g., \cite[Corollary 5.3]{dauvergne2018basic}), which are actually in disagreement with those shown in \eqref{e.tilde F and F relation} and \eqref{e.F_N upper bound}.
\end{remark}

\subsection{$\bm{G_k}$ asymptotics}\label{s.actual G_k asymptotics}
 Here we give the proof of Lemma~\ref{l.G_k asymptotic}. In order to access a certain variational formula, we prove its analog in the model of Brownian LPP, and use the convergence of the latter to the Airy sheet to obtain the result for $G_k$ itself. We briefly recall the model.

An up-right path from $(x,j)$ to $(y,k)$ for $x<y$ and $j>k$ is a function $\gamma:[x,y]\to\intint{k,j}$ which is non-increasing and c\`adl\`ag (right continuous with left-limits) with $\gamma(x) = j$ and $\gamma(y) = k$. The jump times $(t_1, \ldots, t_{N-1})$ of a path from line $N$ to line $1$ defined by $t_i$ being the time $\gamma$ jumps from line $i+1$ to line $i$, i.e., the time $t$ such that $\gamma(t^-) = i+1$ and $\gamma(t) = i$. A pair of up-right paths $\pi$ and $\pi'$ from $(x, j)$ to $(y,k)$ and $(x',j')$ to $(y', k')$, respectively, are said to be disjoint if either $\pi(t) > \pi' (t)$ for all $t \in (x, y) \cap (x', y')$ (i.e., excluding endpoints), or $\pi (t) < \pi' (t)$ for all $t \in (x, y) \cap (x', y')$.

 Let $\bm B = (B_1, \ldots, B_N)$ be a collection of $N$ independent and identically distributed Brownian motions on $[0,\infty)$. The weight of a path $\gamma$ from $(x,N)$ to $(y,1)$ with jump times $(t_1, \ldots, t_{N-1})$ is given by (with the convention $t_0=y$ and $t_N = x$)
\begin{align*}
\bm B[\gamma] = \sum_{i=1}^{N} \bigl(B_i(t_{i-1}) - B_i(t_i)\bigr).
\end{align*}
Let $\bm W\!\bm B = (W\!B_1, \ldots, W\!B_N)$ be the curves of the \emph{watermelon}, defined by
\begin{align*}
\sum_{i=1}^k(W\!B)_i(t) = \bm B[(0,N)^k\to(t,1)^k]
\end{align*}
for each $k=1, \ldots, N$, where $\bm B[(0,N)^k\to(t,1)^k]$ is the maximum weight over all collections of $k$ disjoint paths $(\gamma_1, \ldots, \gamma_k)$, with $\gamma_i$ from $(0,N-i+1)$ to $(t,k-i+1)$, with the weight of a collection of paths being the sum of the weights (equivalently, the paths all go from $(0,N)$ to $(t,1)$, in which case they will adopt the mentioned values instantly in order to maintain disjointness).

Recall the LPP notation from \eqref{e.LPP definition} (which agrees with the above convention with $k=1$). Let $x_N = 2xN^{2/3}$, and similarly for $y_N$ and $z_N$. Define
\begin{align*}
F^{N,\mrm{Br}}_k(x,z) &= N^{-1/3}\left(\bm W\!\bm B[(x_N,N)\to (N+z_N,k)]  - (W\!B)_k(N+z_N) + 2xN^{2/3}\right)\\
G^{N,\mrm{Br}}_k(z,y)&= N^{-1/3}\left(\bm W\!\bm B[(N+z_N,k)\to(N+y_N,1)] + (W\!B)_k(N+z_N) - 2N-2yN^{2/3}\right)\\
H^{N,\mrm{Br}}(x,y)&= N^{-1/3}\left(\bm W\!\bm B[(x_N, N)\to(N+y_N,1)] - 2N - 2(y-x)N^{2/3}\right).
\end{align*}
It follows from \cite[Proposition 2.2]{dauvergne2018directed} or \cite[Proposition 4.1]{corwin2014brownian} that $G^{N,\mrm{Br}}_k(z,y)\stackrel{d}{\to} G_k(z,y)$ for any fixed $z$ and $y$ (which may depend on $k$).

We will also need a certain stationarity property of $G_k$:

\begin{lemma}\label{l.translation relation for G}
Fix $\varepsilon\in \R$ and $z<y$. Then
\begin{align*}
G_k(z\sqrt{k}, y\sqrt{k}) \stackrel{d}{=} G_k((z+\varepsilon)\sqrt{k}, (y+\varepsilon)\sqrt{k})) + (\varepsilon^2+2\varepsilon y)k.
\end{align*}
\end{lemma}

\begin{proof}
Recall that $G_k(z,y) = \bm\cP[(z,k)\to (y,1)] + \cP_k(z)$ and that $\cP_i(x) = \cA_i(x) -x^2$ for all $i\in\N$ and $x\in\R$, where $\bm\cA$ is the Airy line ensemble and is stationary. Also note that for any up-right path $\gamma$ from $(z,k)$ to $(y,1)$, $\bm\cP[\gamma] = \bm\cA[\gamma]-y^2 + z^2$. Thus it follows that
\begin{align*}
G_k(z,y) &= \bm\cA[(z,k) \to (y,1)] -y^2 + z^2 + \cA_k(z) - z^2\\
&\stackrel{d}{=} \bm\cA[(z+\varepsilon,k) \to (y+\varepsilon,1)] -y^2 + \cA_k(z+\varepsilon) \\
&= \bm\cP[(z+\varepsilon,k) \to (y+\varepsilon,1)] - y^2 + (y+\varepsilon)^2 + \cP_k(z+\varepsilon).
\end{align*}
This equals $G(z+\varepsilon, y+\varepsilon) + \varepsilon^2+2\varepsilon y$. Replacing $z$ by $z\sqrt{k}$, $y$ by $y\sqrt{k}$, and $\varepsilon$ by $\varepsilon\sqrt{k}$ completes the proof.
\end{proof}

\begin{proof}[Proof of Lemma~\ref{l.G_k asymptotic}]
In the proof we will prove the statement with $z_0$ in place of $z$ as we will need to use $z$ as a generic variable. As mentioned, we will prove that
$$G^{N,\mrm{Br}}_k(z_0\sqrt{k}, y\sqrt{k}) = -y^2k -\sqrt{k}/(y-z_0) + \mf o(\sqrt{k});$$
then since $G^{N,\mrm{Br}}_k(z_0\sqrt{k}, y\sqrt{k})\stackrel{d}{\to} G_k(z_0\sqrt{k}, y\sqrt{k})$ as $N\to\infty$, we will be done.

We know from \cite[Proposition~6.1]{dauvergne2018directed} that 
\begin{equation}\label{e.F^n asymptotic}
F^{N,\mrm{Br}}_k(x, z\sqrt{k}) = (2\sqrt{2x} + 2xz)\sqrt{k} + \mf o(k^{1/2})
\end{equation}
for any $0<x<z$, and from \cite[Theorem 1.4, Proposition 4.1]{dauvergne2018directed} that $\smash{H^{N,\mrm{Br}}(x,y) \stackrel{d}{\to} \S(x;y)}  = -(x-y)^2 + \mf o(\sqrt{k})$.
Also note that $z\leq y$ from the definition of $G^{N,\mrm{Br}}_k(z,y)$. By the variational definition of the LPP values, we also know that 
\begin{align}\label{e.FGH variational}
H^{N,\mrm{Br}}(x,y) = \sup_{z\leq y} \left(F^{N,\mrm{Br}}_k(x,z) + G^{N,\mrm{Br}}_k(z,y)\right).
\end{align}
 Let $x_0 = 1/(2(y-z_0)^2)$. The previous display, \eqref{e.F^n asymptotic}, and the continuity of $F^{N,\mrm{Br}}_k$ and $G^{N,\mrm{Br}}_k$ as functions of $z$ imply that there is a random $z^{*,N}_k$ such that
\begin{equation}\label{e.FGH equality at maximizer}
F^{N,\mrm{Br}}_k(x_0, z^{*,N}_k\sqrt{k}) + G^{N,\mrm{Br}}_k(z^{*,N}_k\sqrt{k}, y\sqrt{k}) = H^{N,\mrm{Br}}(x_0, y\sqrt k) = -(x_0-y\sqrt{k})^2 + \mf o(\sqrt{k}).
\end{equation}
We claim that
\begin{equation}\label{e.z^* value claim}
z^{*,N}_k = -\sqrt{\frac{1}{2x_0}} + y + \mf o(1) = z_0 + \mf o(1)
\end{equation}
and that $\{z^{*,N}_k\}_{N=1}^\infty$ is tight. By \eqref{e.FGH variational}, we know
$$G^{N,\mrm{Br}}_k(z\sqrt{k}, y\sqrt{k})\leq H^{N,\mrm{Br}}(x,y\sqrt{k}) - F^{N,\mrm{Br}}_k(x, z\sqrt{k})$$
for all $x, y, z$. Taking $x=1/(2(y-z)^2)$ and the asymptotics $H^{N,\mrm{Br}}(x,y\sqrt{k}) = -(x-y\sqrt{k})^2 + \mf o(\sqrt{k})$, $F^{N,\mrm{Br}}_k(x, z\sqrt{k}) = 2\sqrt{2kx} + 2\sqrt{k}xz + \mf o(\sqrt k)$ yields that
\begin{equation}\label{e.G_k upper bound}
G^{N,\mrm{Br}}_k(z\sqrt{k}, y\sqrt{k}) \leq -y^2k - \frac{\sqrt{k}}{y-z} + \mf o(\sqrt{k});
\end{equation}
by taking $N\to\infty$ and using $G^{N,\mrm{Br}}_k(z\sqrt{k}, y\sqrt{k}) \to G_k(z\sqrt{k}, y\sqrt{k})$ in distribution we also obtain the upper bound half of the overall equality we are aiming to prove. Using this upper bound and the same asymptotics on $H^{N,\mrm{Br}}$ and $F^{N,\mrm{Br}}_k$ in \eqref{e.FGH equality at maximizer} yields that the latter cannot hold if the claimed \eqref{e.z^* value claim} does not hold. As with the tightness of $\smash{Z^N_k}$ in Proposition~\ref{p.maximizer location}, the tightness of $\smash{z^{*,N}_k}$ follows from the just established asymptotic and the monotonicity of the sequence in $k$ from Lemma~\ref{l.Z_k monotone}.

It remains to prove a lower bound on $G_k(z\sqrt{k}, y\sqrt{k})$ matching the upper bound in \eqref{e.G_k upper bound}. With \eqref{e.z^* value claim} we see that
\begin{align*}
\MoveEqLeft[4]
G^{N,\mrm{Br}}_k(z^{*,N}_k\sqrt{k}, y\sqrt{k})\\
&= H^{N,\mrm{Br}}(x_0, y\sqrt{k}) - F^{N,\mrm{Br}}_k(x_0, z^{*,N}_k\sqrt{k})\\
&=-(x_0-y\sqrt{k})^2 - (2\sqrt{2x_0} + 2x_0z^{*,N}_k)\sqrt{k} + \mf o(\sqrt{k})\\
&= -y^2k + 2x_0y\sqrt{k} - (2\sqrt{2x_0} + 2x_0(-(2x_0)^{-1/2} + y))\sqrt{k} + \mf o(\sqrt{k}).
\end{align*}
By simplifying, the previous line equals
\begin{align}\label{e.G_k(z^*) asymptotic}
G^{N,\mrm{Br}}_k(z^{*,N}_k\sqrt{k}, y\sqrt{k}) = -y^2k -\sqrt{2x_0k} + \mf o(\sqrt{k})
&= -y^2k - \frac{\sqrt{k}}{y-z_0} + \mf o(\sqrt{k}).
\end{align}
By using the tightness of $z^{*,N}_k$ and the weak convergence of $G^{N,\mrm{Br}}_k$ to $G_k$, we obtain the existence of a $z^*_k$ satisfying \eqref{e.z^* value claim} and the previous display with $G_k$ and $z^*_k$ in place of their prelimiting analogs.

To go from $G_k(z^*_k\sqrt{k}, y\sqrt{k})$ to $G_k(z_0\sqrt{k}, y\sqrt{k})$, we use that $G_k(\bm\cdot, y)$ is monotonically decreasing by Lemma~\ref{l.modified LPP monotone}, as well as the stationarity property from Lemma~\ref{l.translation relation for G}. In the case that $z^*_k \geq z_0$, the monotonicity of $G_k(\bm\cdot, y\sqrt{k})$ along with \eqref{e.G_k(z^*) asymptotic} immediately yields that
\begin{align}\label{e.G_k lower bound one case}
G_k(z_0\sqrt{k}, y\sqrt{k}) \geq G_k(z^*_k\sqrt k, y\sqrt k) \geq -y^2k - \frac{\sqrt{k}}{y-z_0} + \mf o(\sqrt{k}).
\end{align}
In the case that $z^*_k < z_0$, we know from above that $z^*_k > z_0 - \mf o(1)$, i.e, for any $\varepsilon>0$ it holds that $\sum_k\P\left(z^*_k< z_0 - \varepsilon\right) < \infty$. We fix $\varepsilon > 0$, and apply Lemma~\ref{l.translation relation for G} along with monotonicity of $G_k(\bm\cdot, y\sqrt{k})$ again to obtain that, for any deterministic sequence $M_k$,
\begin{align*}
\MoveEqLeft[6]
 \P\left(G_k(z_0\sqrt{k}, y\sqrt{k}) < M_k - \varepsilon\sqrt{k}\right)\\
 &= \P\left(G_k((z_0-\varepsilon)\sqrt{k}, (y-\varepsilon)\sqrt{k})  < M_k -(\varepsilon^2-2\varepsilon y)k - \varepsilon\sqrt{k}\right)\\
 &\leq \P\left(G_k(z^*\sqrt{k}, (y-\varepsilon)\sqrt{k}) < M_k -(\varepsilon^2-2\varepsilon y)k - \varepsilon\sqrt{k}\right)\\
 &\qquad + \P\left(z^*_k < z_0-\varepsilon\right).
\end{align*}
Now by \eqref{e.G_k(z^*) asymptotic}, when $M_k = -(y-\varepsilon)^2k - \sqrt{k}/(y-\varepsilon-z_0) + (\varepsilon^2-2\varepsilon y)k$, the first term in the previous display is summable, and we already noted that the second is also summable. We can write
\begin{align*}
M_k = -y^2k - \frac{\sqrt{k}}{y-z_0} - \varepsilon \sqrt{k}\cdot\left(\frac{1}{(y-z_0)(y-z_0-\varepsilon)}\right),
\end{align*}
and the parenthetical term is positive and $O(1)$. Thus we have established that $\P(G_k(z_0\sqrt{k}, y\sqrt{k}) < -y^2k - \sqrt{k}/(y-z_0) - \varepsilon\sqrt{k})$ is summable for every $\varepsilon>0$, i.e., $G_k(z_0\sqrt{k}, y\sqrt{k}) \geq -y^2k - \sqrt{k}/(y-z_0) + \mf o(\sqrt{k})$.
This, along with \eqref{e.G_k lower bound one case} and \eqref{e.G_k upper bound}, completes the proof.
\end{proof}

\section{Directed landscape and KPZ fixed point convergence}\label{s.general initial condition}

In this section we prove Theorems~\ref{t.asep general initial condition} (ASEP) and \ref{t.s6v landscape} (S6V). We prove Theorem~\ref{t.asep general initial condition} in Section~\ref{s.asep to DL} and Theorem~\ref{t.s6v landscape} in Section~\ref{s.s6v to dl}. The latter relies on an approximate monotonicity result (Proposition~\ref{p.s6v approximate monotonicity}) which is proved in Section~\ref{s.proof of approximate monotonicity}, and a finite speed of discrepancy bound (Lemma~\ref{l.finite speed of prop for s6v basic coupling}) proved in Section~\ref{s.finite speed of prop s6v basic coupling}.

\begin{remark}\label{r.appendix contributions}
    The key idea underlying the proofs of Theorems~\ref{t.asep general initial condition} (1) and \ref{t.s6v landscape} (1) was communicated to us by Shalin Parekh, and was also used by him in the work \cite{parekh2023convergence} establishing joint convergence of ASEP in the weakly asymmetric regime under multiple initial conditions with the basic coupling to the KPZ equation with multiple initial conditions driven by the same noise. Specifically, it is the observation that a one-sided bound on the prelimiting height functions (which is a result of monotonicity properties of the height function evolution) along with the continuity of the limit (i.e., the directed landscape), allows one to reduce process-level convergence to one-point convergence, which is already known.

    While the necessary height monotonicity properties are well-known for ASEP (Lemma~\ref{l.asep height mono}), they are not available for S6V. To address this, we prove an approximate version of height monotonicity for S6V (Proposition~\ref{p.s6v approximate monotonicity}), which constitutes our main new contribution in this appendix.
\end{remark}

\subsection{ASEP to directed landscape} \label{s.asep to DL}

We start by giving a lemma on height monotonicity for ASEP under the basic coupling. Recall the definition from below \eqref{e.height to config} of the ASEP height function $h^{\mrm{ASEP}}(h_0, 0; \bm\cdot, t)$ when started from an arbitrary height function $h_0$. 

\begin{lemma}[Height monotonicity of ASEP]\label{l.asep height mono}
Let $H\in\Z$ be an integer and $h_0, h_0':\Z\to\Z$ be Bernoulli paths such that $h_0(x) +H \geq h_0'(x)$ for all $x\in\Z$. Then, under the basic coupling, almost surely, $h^{\mrm{ASEP}}(h_0, 0; x,t) +H \geq h^{\mrm{ASEP}}(h_0', 0; x,t)$ for all $x\in\Z$ and $t>0$.
\end{lemma}

\begin{proof}
This follows immediately from the definitions of the basic coupling in Section~\ref{s.asep basic coupling} and $h^{\mrm{ASEP}}(h_0, 0; \bm\cdot, t)$ from below \eqref{e.height to config}.
\end{proof}

We will also need to use the following basic observation, whose proof is trivial and omitted.

\begin{lemma}\label{l.X=Y}
Let $X,Y$ be random variables such that $\E[|X|], \E[|Y|] < \infty$. Suppose $X\geq Y$ almost surely and $\E[X] = \E[Y]$. Then $X=Y$ almost surely.
\end{lemma}

\begin{proof}[Proof of Theorem~\ref{t.asep general initial condition} (1) \& (2)]
Theorem~\ref{t.asep general initial condition} (2) follows immediately from (1), so we only need to establish the latter. 
The convergence of $(\mc L^{\mrm{ASEP},\varepsilon}(\bm\cdot, s; \bm\cdot, t) : (s,t)\in\mc T^2_<)$ to $(\mc L(\bm\cdot, s; \bm\cdot, t) : (s,t)\in\mc T^2_<)$ in the space $\smash{\mc C(\R^2,\R)^{\mc T^2_<}}$ is equivalent to the convergence 
\begin{align}\label{e.finite-dimensional convergence to show}
\left(\mc L^{\mrm{ASEP},\varepsilon}(\bm\cdot, s; \bm\cdot, t): (s,t)\in\mc U^2_<\right) \xrightarrow{\smash{d}} \left(\mc L(\bm\cdot, s; \bm\cdot, t): (s,t)\in\mc U^2_<\right)
\end{align}
for every finite subset $\mc U\subseteq \mc T$. Fix such a $\mc U$ and write $\mc U = \{s_i : i\in\intint{1,n}\}$ with $0 \leq s_1 <  \ldots  < s_n$ for some $n\in\N$.

We first observe that $(\mc L^{\mrm{ASEP},\varepsilon}(\bm\cdot, s; \bm\cdot, t): (s,t)\in\mc U^2_<)$ is a tight sequence in $\varepsilon$. This follows from combining the fact that,  by Definitions~\ref{d.asep sheet} and \ref{d.asep landscape},
\begin{equation}\label{e.asep prelimiting scaling relation}
\mc L^{\mrm{ASEP},\varepsilon}(x, s; y, t) \stackrel{\smash{d}}{=} (t-s)^{1/3}\S^{\mrm{ASEP},(t-s)^{-1}\varepsilon}((t-s)^{-2/3}x; (t-s)^{-2/3}y)
\end{equation}
as processes in $(x,y)$, with Theorem~\ref{t.asep airy sheet} and the relation of $\mc L$ with $\S$ from Definition~\ref{d.directed landscape} (i), namely, for any $t>0$, and as processes in $x$ and $y$,
\begin{align}\label{e.cL scale invariance}
\mc L(x, s; y, t) \stackrel{d}{=} (t-s)^{1/3}\S((t-s)^{-2/3}x, (t-s)^{-2/3}y).
\end{align}

Let $(\tilde {\mc L}(\bm\cdot, s; \bm\cdot, t): (s,t)\in\mc U^2_<)$ be any subsequential limit point of the lefthand side of \eqref{e.finite-dimensional convergence to show}. Then we further see from Theorem~\ref{t.asep airy sheet}, \eqref{e.asep prelimiting scaling relation}, and \eqref{e.cL scale invariance} that, for each $(s,t)\in\mc U^2_{<}$, $\tilde{\mc L}(\bm\cdot, s; \bm\cdot, t) \stackrel{\smash{d}}{=} \mc L(\bm\cdot, s; \bm\cdot, t)$. 
We next observe that $(\tilde{\mc L}(\bm\cdot, s_i; \bm\cdot, s_{i+1}) : i\in\intint{1,n-1}) \stackrel{\smash{d}}{=}  (\mc L(\bm\cdot, s_i; \bm\cdot, s_{i+1}) : i\in\intint{1,n-1})$, since each of the components of the lefthand vector are independent, and the same for the righthand vector.

With these facts in hand, \eqref{e.finite-dimensional convergence to show} is reduced to showing that, for any $r,s,t\in\mc U$ with $r<s<t$ fixed, almost surely, for all $x,y\in\R^2$,
\begin{align}\label{e.variational to show}
\tilde{\mc L}(x,r; y,t) = \sup_{z\in\R}\left(\tilde{\mc L}(x,r; z,s) + \tilde{\mc L}(z,s; y,t)\right);
\end{align}
this suffices since the righthand side is a deterministic function of the independent quantities $\tilde{\mc L}(\bm\cdot, r; \bm\cdot,s)$ and $\tilde{\mc L}(\bm\cdot,s; \bm\cdot,t)$ and the same relation holds with $\mc L$ in place of $\tilde{\mc L}$.
We show \eqref{e.variational to show} by returning to the prelimit. Fix $r<s$ as above. We first claim that, 
almost surely, for all $x,y,z\in\Z$ (recall $\gamma = 1-q$),
\begin{align}
\MoveEqLeft[18]
h^{\mrm{ASEP}}(x, 2\gamma^{-1} r\varepsilon^{-1}; z, 2\gamma^{-1} s\varepsilon^{-1}) + h^{\mrm{ASEP}}(z, 2\gamma^{-1} s\varepsilon^{-1}; y, 2\gamma^{-1} s\varepsilon^{-1})\nonumber\\
&\geq h^{\mrm{ASEP}}(x, 2\gamma^{-1} r\varepsilon^{-1}; y, 2\gamma^{-1} s\varepsilon^{-1});\label{e.common time asep inequality}
\end{align}
note that the start and end time coordinates are the same in the second term on the lefthand side. The previous display holds for $z\geq y$ as both sides are equal at $z=y$, the first term on the lefthand side is non-increasing and decreases by at most 1 in going from $z$ to $z+1$, and the second term on the lefthand side is
$$h^{\mrm{ASEP}}(z, 2\gamma^{-1} s\varepsilon^{-1}; y, 2\gamma^{-1} s\varepsilon^{-1}) = (z-y)\one_{y\leq z}$$
and so increases by $1$ on going from $z$ to $z+1$. Similarly, for $z\leq y$, we have equality in \eqref{e.common time asep inequality} at $z=y$, the first term on the lefthand side weakly increases on going from $z$ to $z-1$, and the second term on the lefthand side is $0$ for all $z\leq y$.

We now view \eqref{e.common time asep inequality} as an ordering of initial conditions: for each $x,z$ fixed, we regard the first term $h^{\mrm{ASEP}}(x, 2\gamma^{-1} r\varepsilon^{-1}; z, 2\gamma^{-1} s\varepsilon^{-1})$ as a (random) constant, and $h^{\mrm{ASEP}}(z, 2\gamma^{-1} s\varepsilon^{-1}; \bm\cdot, 2\gamma^{-1} s\varepsilon^{-1})$ and $h^{\mrm{ASEP}}(x, 2\gamma^{-1} r\varepsilon^{-1}; \bm\cdot, 2\gamma^{-1} s\varepsilon^{-1})$ both as initial conditions. Evolving both for time $2\gamma^{-1} (t-s)\varepsilon^{-1}$ under the basic coupling, Lemma~\ref{l.asep height mono} guarantees that, almost surely, for all $x,y,z\in\Z$,
\begin{align*}
\MoveEqLeft[18]
h^{\mrm{ASEP}}(x, 2\gamma^{-1} r\varepsilon^{-1}; z, 2\gamma^{-1} s\varepsilon^{-1}) + h^{\mrm{ASEP}}(z, 2\gamma^{-1} s\varepsilon^{-1}; y, 2\gamma^{-1} t\varepsilon^{-1})\\
&\geq h^{\mrm{ASEP}}(x, 2\gamma^{-1} r\varepsilon^{-1}; y, 2\gamma^{-1} t\varepsilon^{-1}).
\end{align*}
Centering and rescaling according to Definition~\ref{d.asep landscape}, and taking $\varepsilon\to 0$, the previous display yields that, almost surely, for all $x,y\in\R$,
\begin{align}\label{e.subsequential inequality}
\tilde{\mc L}(x,r; y,t) \geq \sup_{z\in\R}\left(\tilde{\mc L}(x,r; z,s) + \tilde{\mc L}(z,s; y,t)\right).
\end{align}
However, we know that both sides of this inequality are marginally distributed as the corresponding expression with $\mc L$ in place of $\tilde{\mc L}$ and, further, we know that, almost surely, for all $x,y\in\R$,
\begin{align}\label{e.L variational}
\mc L(x,r; y,t) = \sup_{z\in\R}\Bigl(\mc L(x,r; z,s) + \mc L(z,s; y,t)\Bigr).
\end{align}
Thus, we know that the expectations of the lefthand and righthand sides of \eqref{e.subsequential inequality} are equal for every fixed $x,y$. Applying Lemma~\ref{l.X=Y} with \eqref{e.subsequential inequality} yields that \eqref{e.variational to show} holds for all $x,y\in\Q^2$ on a single probability one event; since both sides are continuous in $x$ and $y$, we obtain that \eqref{e.variational to show} holds on a probability one event for all $x,y\in\R^2$. This completes the proof.
\end{proof}

\begin{proof}[Proof of Theorem~\ref{t.asep general initial condition} (3)]
This is exactly the same as \cite[Theorem 2.10]{aggarwal2024kpz} after setting $s_i=0$ for $i\in\intint{1,k}$ there.
\end{proof}

\subsection{S6V to directed landscape}\label{s.s6v to dl}

The basic coupling for S6V does not have the height monotonicity properties it does in ASEP (Lemma~\ref{l.asep height mono}). Nevertheless, an approximate version holds, as recorded next. It will be proved in Section~\ref{s.proof of approximate monotonicity}.

\begin{proposition}[Approximate height monotonicity of S6V]\label{p.s6v approximate monotonicity}
Fix $N\in\N$. Let $H\in\Z$ and let $\smash{h_0^{(1)}, h_0^{(2)}}:\intray{-N-1}\to\Z_{\geq 0}$ be two Bernoulli paths such that $\smash{h_0^{(1)}(x)} +H \geq \smash{h_0^{(2)}(x)}$ for all $x\in\intray{-N-1}$, and, for $i=1,2$, $\smash{h_0^{(i)}}(-N-1)\leq N$ and $\smash{h_0^{(i)}}(x) = 0$ for all large $x$. There exist absolute constants $C, c>0$ such that, for $t\in\N$ and under the basic coupling, the following holds. For any $M\geq (\log N)^2$, with probability at least $1-C\exp(-cM)$, $\hssv(h_0^{(1)}; x,t) +H \geq \hssv(h_0^{(2)}; x,t) - M$ for all $x\in\intray{-N-1}$.
\end{proposition}

    In our argument, we will actually need to know that an approximate form of height monotonicity holds even if we only have the initial condition ordering on a sufficiently large domain and not the entirety of $\llbracket -N-1, \infty \rrparen$. We record this in the next corollary, whose proof will also be given in Section~\ref{s.proof of approximate monotonicity}.

\begin{corollary}\label{c.flexible approximate monotonicity}
    Fix $N\in\N$. Let $H\in\Z$ and let $\smash{h_0^{(1)}, h_0^{(2)}}:\intray{-N-1}\to\Z_{\geq 0}$ be two Bernoulli paths such that $\smash{h_0^{(1)}(x)} +H \geq \smash{h_0^{(2)}(x)}$ for all $x\in\intint{-N-1, N}$, and, for $i=1,2$, $\smash{h_0^{(i)}}(-N-1)\leq N$ and $\smash{h_0^{(i)}}(x) = 0$ for all large $x$. There exist absolute constants $C, c>0$ such that, for $t\in\intint{1,\frac{1}{4}(1-b^{\shortrightarrow})N}$ and under the basic coupling, the following holds. For any $(\log N)^2 \leq M \leq N$, with probability at least $1-C\exp(-cM)$, $\hssv(h_0^{(1)}; x,t) +H \geq \hssv(h_0^{(2)}; x,t) - M$ for all $x\in\intint{-\frac{1}{2}N, \frac{1}{2}N}$.
\end{corollary}

To prove this corollary as well as in the proof of Theorem~\ref{t.s6v landscape}, we will also need a finite speed of discrepancy estimate for the S6V model under the basic coupling. Its straightforward proof will be given in Section~\ref{s.finite speed of prop s6v basic coupling}.

\begin{lemma}[Finite speed of discrepancy]\label{l.finite speed of prop for s6v basic coupling}
For any real numbers $0\leq b^{\shortuparrow}\leq b^{\shortrightarrow}< 1$, there exist constants $c>0$ and $C > 1$ such that the following holds. Let $m\in\N$ be an integer, and for $k\in\intint{1,m}$ and $i\in\{1,2\}$, let $s_k\in \Z_{\geq 1}$, $N_i\in\N$, and $h^{(i), k}_0 : \intray{-N_i} \to \Z$ be Bernoulli paths which equal $0$ for all large $x$. Assume that $N_1 \leq N_2$ and, for $k\in\intint{1,m}$, $h^{(1), k}_0(x) = h^{(2), k}_0(x)$ for all $x\in\intint{-N_1, N_1}$. For $i\in\{1,2\}$ and $k\in\intint{1,m}$, let $(j^{(i), k}_{(x,y)} : (x,y)\in\intray{-N_i}\times\intray{s_k})$ denote samples from stochastic six-vertex models with boundary conditions given by $\smash{h^{(i), k}_0}$ at time $s_k$ (recall from Section~\ref{s.s6v basic coupling}), all coupled via the basic coupling. Then, for any $T\in\N$ such that $T\leq \frac{1}{4}(1-b^{\shortrightarrow})N_1$, with probability at least $1-Cme^{-cN_1}$, we have that $j^{(1), k}_{(x,y)} = j^{(2), k}_{(x,y)}$ for all $k\in\intint{1,m}$ and $(x,y)\in\intint{1,T}\times\intint{-\floor{\frac{1}{2}N_1}, \floor{\frac{1}{2}N_1}}$.
\end{lemma}

With these two statements in hand, we can give the proof of Theorem~\ref{t.s6v landscape} along very similar lines as the proof of Theorem~\ref{t.asep general initial condition} (1).

\begin{proof}[Proof of Theorem~\ref{t.s6v landscape} (1) and (2)]
Item (2) is an immediate consequence of (1), so we only need to prove the latter. We first claim that we may assume that $N \leq \varepsilon^{-3}$, where recall $\varepsilon$ is the parameter in Definition~\ref{d.s6v landscape} of $\mc L^{\mrm{S6V}, \varepsilon}$. More precisely, if $N > \varepsilon^{-3}$, we may couple an ASEP landscape $\tilde{\mc L}^{\mrm{S6V}, \varepsilon}$ with that value of $N$ to a copy ${\mc L}^{\mrm{S6V}, \varepsilon}$ with $N=\varepsilon^{-3}$ such that, with probability at least $1-C\exp(-c\varepsilon^{-3})$, ${\mc L}^{\mrm{S6V}, \varepsilon}(x,s;y,t) = \tilde{\mc L}^{\mrm{S6V}, \varepsilon}(x,s;y,t)$ or all $|x|,|y|, |s|, |t|\leq \varepsilon^{-1/6}$. As in the proof of Theorem~\ref{t.s6v airy sheet} in Section~\ref{s.airy sheet convergence setup}, the existence of this coupling amounts to a finite speed of discrepancy bound; because we need it for the basic coupling and not the colored coupling as there, we use Lemma~\ref{l.finite speed of prop for s6v basic coupling}. We record the parameters we will apply Lemma~\ref{l.finite speed of prop for s6v basic coupling} with: $T = \ceil{\varepsilon^{-7/6}}$, $N_1 = \varepsilon^{-3}$, $N_2$ the value of $N$ associated to $\tilde{\mc L}^{\mrm{S6V},\varepsilon}$, and $m=T\cdot (2N_1+1)$. The height functions we consider are, at each time $s\in \intint{1, T}$, the collection of step initial conditions starting from $x$ for $x\in \intint{-N_1, N_1}$, where the step initial condition from $x$ means, for $i\in\{1,2\}$ and in system $i$, particles at every location in $\intint{x,N_i}$ (note that this indeed yields $m$ many height functions).
Lemma~\ref{l.finite speed of prop for s6v basic coupling} then yields the claimed coupling. 

In the remainder of the proof we assume $N=\varepsilon^{-3}$ by working on the event of the above coupling holding, and we denote by $\mc L^{\mrm{S6V}, \varepsilon}$ the associated S6V landscape. Since the probability of the coupling holding goes to $1$ as $\varepsilon\to 0$, it suffices to show that $(\mc L^{\mrm{S6V},\varepsilon}(\bm\cdot, s; \bm\cdot, t) : (s,t)\in\mc T^2_<) \stackrel{\smash{d}}{\to} (\mc L(\bm\cdot, s; \bm\cdot, t) : (s,t)\in\mc T^2_<)$.

Now, the convergence of $(\mc L^{\mrm{S6V},\varepsilon}(\bm\cdot, s; \bm\cdot, t) : (s,t)\in\mc T^2_<)$ to $(\mc L(\bm\cdot, s; \bm\cdot, t) : (s,t)\in\mc T^2_<)$ in the space $\smash{\mc C(\R^2,\R)^{\mc T^2_<}}$ is equivalent to the convergence 
\begin{align}\label{e.finite-dimensional convergence to show s6v}
\left(\mc L^{\mrm{S6V},\varepsilon}(\bm\cdot, s; \bm\cdot, t): (s,t)\in\mc U^2_<\right) \xrightarrow{\smash{d}} \left(\mc L(\bm\cdot, s; \bm\cdot, t): (s,t)\in\mc U^2_<\right)
\end{align}
for every finite subset $\mc U\subseteq \mc T$. Fix such a $\mc U$ and write $\mc U = \{s_i : i\in\intint{1,n}\}$ with $0 \leq s_1 <  \ldots  < s_n$ for some $n\in\N$.

We observe that $(\mc L^{\mrm{S6V},\varepsilon}(\bm\cdot, s; \bm\cdot, t): (s,t)\in\mc U^2_<)$ is a tight sequence in $\varepsilon$. This follows from combining the fact that,  by Definitions~\ref{d.s6v sheet} and \ref{d.s6v landscape},
\begin{equation}\label{e.s6v prelimiting scaling relation}
\mc L^{\mrm{S6V},\varepsilon}(x, s; y, t) \stackrel{\smash{d}}{=} (t-s)^{1/3}\S^{\mrm{S6V},(t-s)^{-1}\varepsilon}((t-s)^{-2/3}x; (t-s)^{-2/3}y)
\end{equation}
as processes in $(x,y)$, with Theorem~\ref{t.s6v airy sheet} and the relation of $\mc L$ with $\S$ noted in \eqref{e.cL scale invariance}.

Let $(\tilde {\mc L}(\bm\cdot, s; \bm\cdot, t): (s,t)\in\mc U^2_<)$ be any subsequential limit point of the lefthand side of \eqref{e.finite-dimensional convergence to show s6v}. Then we further see from Theorem~\ref{t.s6v airy sheet}, \eqref{e.cL scale invariance}, and \eqref{e.s6v prelimiting scaling relation} that, for each $(s,t)\in\mc U^2_{<}$, $\tilde{\mc L}(\bm\cdot, s; \bm\cdot, t) \stackrel{\smash{d}}{=} \mc L(\bm\cdot, s; \bm\cdot, t)$. 
We next observe that $(\tilde{\mc L}(\bm\cdot, s_i; \bm\cdot, s_{i+1}) : i\in\intint{1,n-1}) \stackrel{\smash{d}}{=}  (\mc L(\bm\cdot, s_i; \bm\cdot, s_{i+1}) : i\in\intint{1,n-1})$, since each of the components of the lefthand vector are independent, and the same for the righthand vector.

With these facts in hand, \eqref{e.finite-dimensional convergence to show s6v} is reduced to showing that, for any $r,s,t\in\mc U$ with $r<s<t$ fixed, almost surely, for all $x,y\in\R^2$,
\begin{align}\label{e.variational to show s6v}
\tilde{\mc L}(x,r; y,t) = \sup_{z\in\R}\left(\tilde{\mc L}(x,r; z,s) + \tilde{\mc L}(z,s; y,t)\right);
\end{align}
this suffices since the righthand side is a deterministic function of the independent quantities $\tilde{\mc L}(\bm\cdot, r; \bm\cdot,s)$ and $\tilde{\mc L}(\bm\cdot,s; \bm\cdot,t)$ and the same relation holds with $\mc L$ in place of $\tilde{\mc L}$.
We show \eqref{e.variational to show s6v} by returning to the prelimit. Fix $r,s,t$ as above. We first claim that almost surely, for all $x,y,z\in\intint{-N-1, N}$,
\begin{equation}\label{e.common time s6v inequality}
\begin{split}
h^{\mrm{S6V}}(x, r\varepsilon^{-1}; y, s\varepsilon^{-1})
&\geq h^{\mrm{S6V}}(x, r\varepsilon^{-1}; z, s\varepsilon^{-1}) + h^{\mrm{S6V}}(z, s\varepsilon^{-1}; y, s\varepsilon^{-1})\\
&\qquad -(N-z)
\end{split}
\end{equation}
%
To establish the claim, we view both sides as functions of $y$. Both sides are equal at $y=z$ as $h^{\mrm{S6V}}(z, s\varepsilon^{-1}; z, s\varepsilon^{-1}) = N-z$ from \eqref{e.S6V step initial condition}. When $y\in \intint{z,N}$ and $y$ increases by 1, the lefthand side decreases by at most 1, while the second term on the righthand side decreases by exactly 1. When $y\in\intint{-N, z}$ and $y$ decreases by 1, the lefthand side remains the same or increases by $1$, while the righthand side remains the same. This establishes the claim.
%
%
 %

We now view \eqref{e.common time s6v inequality} as an ordering of initial conditions: for each $x,z$ fixed, we regard the term $h^{\mrm{S6V}}(x, r\varepsilon^{-1}; z, s\varepsilon^{-1}) - (N-z)$ as a (random) constant, and $h^{\mrm{S6V}}(z, s\varepsilon^{-1}; \bm\cdot, s\varepsilon^{-1})$ and $h^{\mrm{S6V}}(x, r\varepsilon^{-1}; \bm\cdot, s\varepsilon^{-1})$ both as initial conditions. Evolving both for time $(t-s)\varepsilon^{-1}$ under the basic coupling, Corollary~\ref{c.flexible approximate monotonicity} and a union bound over $x, z \in \intint{-N-1,N}$ guarantees that there exist $c,C>0$ such that, with probability at least $1-C\exp(-c(\log N)^2)$, for all $x,z\in\intint{-N-1, N}$, $y\in\intint{-\frac{1}{2}N, \frac{1}{2}N}$, and $t\in\intint{1, \frac{z(1-q)}{4(1-qz)}\varepsilon N}$,
\begin{align*}
h^{\mrm{S6V}}(x, r\varepsilon^{-1}; y, t\varepsilon^{-1}) 
&\geq h^{\mrm{S6V}}(x, r\varepsilon^{-1}; z, s\varepsilon^{-1}) + h^{\mrm{S6V}}(z, s\varepsilon^{-1}; y, t\varepsilon^{-1}) \\
&\qquad -(N-z) - (\log N)^2.
\end{align*}
Next we center and rescale according to Definition~\ref{d.s6v landscape} to obtain $\mc L^{\mrm{S6V}, \varepsilon}$ and note that since $N =\varepsilon^{-3} \geq 2\alpha \varepsilon^{-1}$, the arguments of $\hssv$ still lie within the domain where the previous display holds. Recall also that the rescaling involves both sides being multiplied by $\varepsilon^{1/3}$, and that $N= \varepsilon^{-3}$; thus, the $(\log N)^2$ term vanishes in the $\varepsilon\to 0$ limit. So by taking $\varepsilon\to 0$, the previous display yields that, almost surely, for all $x,y\in\R$,
\begin{align*}%
\tilde{\mc L}(x,r; y,t) \geq \sup_{z\in\R}\left(\tilde{\mc L}(x,r; z,s) + \tilde{\mc L}(z,s; y,t)\right).
\end{align*}
However, we know that both sides of this inequality are marginally distributed as the corresponding expression with $\mc L$ in place of $\tilde{\mc L}$, so we further know from \eqref{e.L variational} that the expectations of the two sides are equal. Now, Lemma~\ref{l.X=Y} yields that \eqref{e.variational to show s6v} holds for all $x,y\in\Q^2$ on a single probability one event; since both sides are continuous in $x$ and $y$, we obtain that \eqref{e.variational to show s6v} holds on a probability one event for all $x,y\in\R^2$. This completes the proof.
\end{proof}

Next we prove Theorem~\ref{t.s6v landscape} (3), which is a quick consequence of \cite[Theorem~2.10]{aggarwal2024kpz} after addressing minor differences between that paper and this in the definitions of the objects.

\begin{proof}[Proof of Theorem~\ref{t.s6v landscape} (3)]
For a given S6V initial condition $h_0$, we write out the relations between (i) its rescaled form, the S6V height function started from the initial condition, and the latter's rescaled form as defined in this paper and (ii) the same as defined in \cite{aggarwal2024kpz}. We will call the objects in this paper with the same notation as in Definitions~\ref{d.s6v general initial condition}, i.e., $\mf h_0^\varepsilon$, $h^{\mrm{S6V}}(h_0; y, t)$, and $\mf h^{\mrm{S6V},\varepsilon}(h_0; y, t)$, respectively, and we will call the analogous objects in \cite{aggarwal2024kpz} by adding tildes, i.e., $\tilde{\mf h}_0^\varepsilon$ (rescaled from $h_0$), $\tilde h^{\mrm{S6V}}(h_0; y, t)$, and $\tilde{\mf h}^{\mrm{S6V},\varepsilon}(h_0; y, t)$, respectively. Then by comparing Definition~\ref{d.s6v general initial condition} and \cite[Definition~2.4]{aggarwal2024kpz} it holds that
\begin{equation}\label{e.relation of two height functions}
\begin{split}
\h_0^{\varepsilon}(y) &= \tilde{\h}_0^{\varepsilon}(y) + \sigma(\alpha)^{-1}\mu'(\alpha)N\varepsilon^{1/3}\\
h^{\mrm{S6V}}(h_0; y, t) &= \tilde{h}^{\mrm{S6V}}(h_0; y, t)\\
\mf h^{\mrm{S6V},\varepsilon}(h_0; y, t) &= \tilde{\mf h}^{\mrm{S6V},\varepsilon}(h_0; y, t)  + \sigma(\alpha)^{-1}\mu'(\alpha)N\varepsilon^{1/3}
\end{split}
\end{equation}
We also recall from \cite{aggarwal2024kpz} that the evolution of $\tilde{h}^{\mrm{S6V}}(h_0; y, t)$ respects vertical shifts by constants, i.e., for any $M\in\Z$,
\begin{align}\label{e.relation of two height functions 2}
\tilde{h}^{\mrm{S6V}}(h_0 + M; y, t) = \tilde{h}^{\mrm{S6V}}(h_0; y, t) + M.
\end{align}

Recall the $k$ initial conditions $h_0^{(1)}, \ldots, h_0^{(k)}$ from the statement of Theorem~\ref{t.s6v landscape} (3). Consider the initial conditions $\tilde h_0^{(i)}(y) := h_0^{(i)}(y) - \floor{\mu'(\alpha)N}$ for $i=1, \ldots, k$ and run the S6V model (with dynamics as defined in \cite{aggarwal2024kpz} as $h^{(i)}_0(x)$ need not be zero for all large $x$ as assumed in this paper) from these initial conditions to obtain the height functions $\tilde h^{\mrm{S6V}}(\tilde h^{(i)}_0; y,t)$ as defined in \cite[eq. (2.2)]{aggarwal2024kpz}. Then \eqref{e.relation of two height functions} and \eqref{e.relation of two height functions 2} yield that (with $\tilde \h^{(i),\varepsilon}_0$ now rescaled from $\tilde h_0^{(i)}$, and $\h^{(i),\varepsilon}_0$ still corresponding to $h^{(i)}_0$)
\begin{align}
\tilde \h^{(i),\varepsilon}_0(y) &= \h_0^{(i),\varepsilon}(y) + O(\varepsilon^{1/3}) \quad \text{and}\nonumber\\
\tilde \h^{\mrm{S6V},\varepsilon}(\tilde h^{(i)}_0; y, t) &= \h^{\mrm{S6V},\varepsilon}(h^{(i)}_0; y, t) + O(\varepsilon^{1/3}),\label{e.relation of two height functions 3}
\end{align}
where the $O(\varepsilon^{1/3})$ term is uniform over $y$ (in fact, bounded in absolute value by $\varepsilon^{1/3}$ as it arises from taking the floor of $\mu'(\alpha)N$).

Now, Theorem~\ref{t.s6v landscape} (3) assumes that $\h^{(i),\varepsilon}_0 \to \h^{(i)}_0$ in distribution in the topology of uniform convergence on compact sets as $\varepsilon\to 0$ for each $i\in\intint{1,k}$, so it follows that $\tilde\h^{(i),\varepsilon}_0 \to \h^{(i)}_0$ as well. We have also assumed that $\inf_{[-C,C]}\smash{\h_0^{(i),\varepsilon}} \geq -C$ and $\h_0^{(i),\varepsilon}(y) \leq C(1+|y|)$ for some $C$ and for all $y$ and $\varepsilon>0$, so it again follows from the previous display that the same holds for $\tilde \h_0^{(i),\varepsilon}$. Then \cite[Theorem 2.10]{aggarwal2024kpz} guarantees that 
\begin{align*}
(\tilde \h^{\mrm{S6V},\varepsilon}(\tilde h^{(i)}_0; \bm\cdot, t))_{i\in\intint{1,k}, t\in\mc T} \stackrel{d}{\to} (\h(\h_0^{(i)}; \bm\cdot, t))_{i\in\intint{1,k}, t\in\mc T}
\end{align*}
in the topology of uniform convergence on compact sets, where the components of the righthand side are coupled using the same directed landscape in \eqref{e.kpz fixed point}. Our proof is then complete by invoking \eqref{e.relation of two height functions 3}.
\end{proof}

\subsection{Approximate monotonicity for S6V}\label{s.proof of approximate monotonicity}
Here we give the proof of Proposition~\ref{p.s6v approximate monotonicity}.

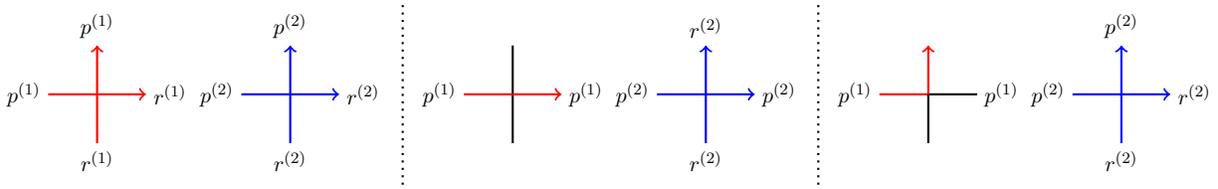
\begin{figure}[h]

\begin{tikzpicture}[scale=0.58]
  \draw[red, thick] (-1,0) -- ++(1,0);
  \draw[red, thick, ->] (0,0) -- ++(1,0);
  \draw[red, thick] (0,-1) -- ++(0,1);
  \draw[red, thick, ->] (0,0) -- ++(0,1);

  \node[anchor=east, scale=0.65] at (-1,0) {$p^{(1)}$};
  \node[anchor=north, scale=0.65] at (0,-1) {$r^{(1)}$};

  \node[anchor=west, scale=0.65] at (1,0) {$r^{(1)}$};
  \node[anchor=south, scale=0.65] at (0,1) {$p^{(1)}$};  

  \begin{scope}[shift={(3.95,0)}]
  \draw[blue, thick] (-1,0) -- ++(1,0);
  \draw[blue, thick, ->] (0,0) -- ++(1,0);
  \draw[blue, thick] (0,-1) -- ++(0,1);
  \draw[blue, thick, ->] (0,0) -- ++(0,1);

  \node[anchor=east, scale=0.65] at (-1,0) {$p^{(2)}$};
  \node[anchor=north, scale=0.65] at (0,-1) {$r^{(2)}$};

  \node[anchor=west, scale=0.65] at (1,0) {$r^{(2)}$};
  \node[anchor=south, scale=0.65] at (0,1) {$p^{(2)}$};  
  \end{scope}

  \draw[dotted, thick] (6.25,-1.85) -- ++(0,3.7);

\begin{scope}[shift={(8.5,0)}]
  \draw[thick] (0,-1) -- ++(0,1);
  \draw[thick] (0,0) -- ++(0,1);
  \draw[red, thick] (-1,0) -- ++(1,0);
  \draw[red, thick, ->] (0,0) -- ++(1,0);

  \node[anchor=east, scale=0.65] at (-1,0) {$p^{(1)}$};

  \node[anchor=west, scale=0.65] at (1,0) {$p^{(1)}$};

  \begin{scope}[shift={(3.95,0)}]
  \draw[blue, thick] (-1,0) -- ++(1,0);
  \draw[blue, thick, ->] (0,0) -- ++(1,0);
  \draw[blue, thick] (0,-1) -- ++(0,1);
  \draw[blue, thick, ->] (0,0) -- ++(0,1);

  \node[anchor=east, scale=0.65] at (-1,0) {$p^{(2)}$};
  \node[anchor=north, scale=0.65] at (0,-1) {$r^{(2)}$};

  \node[anchor=west, scale=0.65] at (1,0) {$p^{(2)}$};
  \node[anchor=south, scale=0.65] at (0,1) {$r^{(2)}$};  
  \end{scope}

  \draw[dotted, thick] (6.25,-1.85) -- ++(0,3.7);
\end{scope}

\begin{scope}[shift={(17,0)}]
  \draw[red, thick] (-1,0) -- ++(1,0);
  \draw[thick] (0,0) -- ++(1,0);
  \draw[thick] (0,-1) -- ++(0,1);
  \draw[red, thick, ->] (0,0) -- ++(0,1);

  \node[anchor=east, scale=0.65] at (-1,0) {$p^{(1)}$};

  \node[anchor=south, scale=0.65] at (0,1) {$p^{(1)}$};

  \begin{scope}[shift={(3.95,0)}]
  \draw[blue, thick] (-1,0) -- ++(1,0);
  \draw[blue, thick, ->] (0,0) -- ++(1,0);
  \draw[blue, thick] (0,-1) -- ++(0,1);
  \draw[blue, thick, ->] (0,0) -- ++(0,1);

  \node[anchor=east, scale=0.65] at (-1,0) {$p^{(2)}$};
  \node[anchor=north, scale=0.65] at (0,-1) {$r^{(2)}$};

  \node[anchor=west, scale=0.65] at (1,0) {$r^{(2)}$};
  \node[anchor=south, scale=0.65] at (0,1) {$p^{(2)}$};  
  \end{scope}
\end{scope}
\end{tikzpicture}
\caption{The left subfigure in each panel, with red arrows, represents a vertex in system 1, and the right subfigure, with blue arrows, represents a vertex in system 2; notice that each system is a uncolored, so here blue and red do not represent colors (in the sense of priority) associated to the arrows as in the main text of the paper. Black edges represent an absence of an arrow. For $i\in\{1,2\}$, in system $i$, the label of the horizontally incoming arrow is $p^{(i)}$, and that of the vertically incoming arrow is $r^{(i)}$. The outgoing labels are assigned as depicted.}\label{f.arrow pairings}

\end{figure}

\begin{proof}[Proof of Proposition~\ref{p.s6v approximate monotonicity}]
We break up the proof into a few parts. For $i\in\{1,2\}$, we will refer to the S6V model with boundary condition $\smash{h^{(i)}_0}$ as ``system $i$'' (though note that the two systems are coupled via the basic coupling and so use the same Bernoulli random variables). Throughout this proof, we assume that $H=0$; otherwise we may place $H$ arrows at $+\infty$ in system 1 at time $0$ if $H>0$, or $-H$ arrows at $+\infty$ in system 2 if $H<0$.

\medskip

\noindent\textbf{Specification of trajectories.} Given a sample from the S6V model started from $h^{(1)}_0$ and $h^{(2)}_0$ coupled via the basic coupling, we define arrow trajectories as follows. We define them in a local way: for every pair of input and output configurations of arrows at a vertex, we define how to pair the inputs with the outputs; this extends to define arrow trajectories globally. We assume that the incoming arrows have labels, and we specify the assignment of the same labels to the outgoing arrows. For $i\in\{1,2\}$, we use the notation $\smash{\ell^{(i), *}_{\mrm{h}}}$ for $*\in\{\mrm{in}, \mrm{out}\}$ for the label of the arrow horizontally incoming and horizontally outgoing, respectively, in the $i$\th system and $\smash{\ell^{(i), *}_{\mrm{v}}}$ for $*\in\{\mrm{in}, \mrm{out}\}$ for the label of the arrow vertically incoming and vertically outgoing, respectively, in the $i$\th system.

 In the case that in each system a single arrow or no arrow is incident at a vertex, the pairing is obvious. If two arrows are incident at a vertex in both, then we assign $\ell^{(i),\mrm{out}}_{\mrm{h}} = \ell^{(i),\mrm{in}}_{\mrm{v}}$ and $\smash{\ell^{(i),\mrm{out}}_{\mrm{v}} = \ell^{(i),\mrm{in}}_{\mrm{h}}}$ for each $i\in\{1,2\}$; see the left panel of Figure~\ref{f.arrow pairings}. If in one system no arrows are incident at a vertex $v$ and in the other two arrows are incident at $v$, we do the pairing in the latter system (system $j$) as in the case where two arrows are present in both systems, i.e., $\ell^{(j),\mrm{out}}_{\mrm{h}} = \ell^{(j),\mrm{in}}_{\mrm{v}}$ and $\smash{\ell^{(j),\mrm{out}}_{\mrm{v}} = \ell^{(j),\mrm{in}}_{\mrm{h}}}$.

 This leaves the case that one arrow is incident  at a vertex $v$ in one system and two arrows are incident at $v$ in the other. We give the pairing in the case that the first system (with one arrow incident) is system 1 and the second system is system 2, and the analogous pairing is prescribed in the other case. If the arrow enters and exits horizontally in system $1$, then we set $\ell^{(i), \mrm{out}}_{\mrm{h}} = \ell^{(i), \mrm{in}}_{\mrm{h}}$ for $i\in\{1,2\}$ (this is the only option in the $i=1$ case) and $\smash{\ell^{(2), \mrm{out}}_{\mrm{v}} = \ell^{(2), \mrm{in}}_{\mrm{v}}}$; see middle panel of Figure~\ref{f.arrow pairings}. If the arrow enters horizontally and exits vertically in system $1$, then we set $\smash{\ell^{(i), \mrm{out}}_{\mrm{v}} = \ell^{(i), \mrm{in}}_{\mrm{h}}}$ for $i\in\{1,2\}$ and $\ell^{(2), \mrm{out}}_{\mrm{h}} = \ell^{(2), \mrm{in}}_{\mrm{h}}$; see right panel of Figure~\ref{f.arrow pairings}. If the arrow in system 1 enters vertically and exits horizontally, we set $\smash{\ell^{(i), \mrm{out}}_{\mrm{h}} = \ell^{(i), \mrm{in}}_{\mrm{v}}}$ for $i\in\{1,2\}$ and $\smash{\ell^{(2), \mrm{out}}_{\mrm{v}} = \ell^{(2), \mrm{in}}_{\mrm{h}}}$, and if it enters vertically and exits vertically, we set $\smash{\ell^{(i), \mrm{out}}_{\mrm{v}} = \ell^{(i), \mrm{in}}_{\mrm{v}}}$ for $i\in\{1,2\}$ and $\smash{\ell^{(2), \mrm{out}}_{\mrm{h}} = \ell^{(2), \mrm{in}}_{\mrm{h}}}$.

 The definition is made such that if the trajectory of an arrow in system 1 coincides at some edge with the trajectory of an arrow in system 2, then the two trajectories agree from then onwards, as can be easily checked; we will refer to this situation as the two arrows ``coupling.''

\medskip

\noindent\textbf{Reduction to claim about arrow position ordering.} %
For $t\in\N$ fixed and $i\in\{1,2\}$, consider the set $\mrm{Uncoupled}_i$ of arrows in system $i$ which have not coupled with any arrow of the other system by time $t$, and let $U_i =  \# \mrm{Uncoupled}_i$. Note that $U_1\geq U_2$. For $s<t$, $i\in\{1,2\}$, and $k\in\intint{1,U_i}$, let us label by $X^{(i)}_s(k)$ the location at time $s$ of the arrow in $\mrm{Uncoupled}_i$ which, among arrows from the same set, is $k$\th from the bottom in the initial configuration $\eta^{(i)}_0$. Note that $X^{(1)}_0(k) \geq X^{(2)}_0(k)$ for all $k\in\intint{1, U_2}$.

We first claim that, if $X^{(2)}_t(k) \leq X^{(1)}_t(k+M)$ for all $k\in\intint{1, \min(U_1 - M, U_2)}$, then $\hssv(h^{(2)}_0; x,t) \leq \hssv(h^{(1)}_0; x,t) + M$ for all $x\in\intray{-N-1}$. Indeed, since arrows coupled by time $t$ are counted in the height functions of both systems,
\begin{align*}
\MoveEqLeft[6]
\hssv(h^{(2)}_0; x,t) - \hssv(h^{(1)}_0; x,t)\\
&= \#\left\{k\in\intint{1,U_2}: X^{(2)}_t(k) > x\right\} - \#\left\{k\in\intint{1,U_1}: X^{(1)}_t(k) > x\right\}\\
&\leq \#\left\{k\in\intint{1,\min(U_1 - M, U_2)} : X^{(1)}_t(k+M) > x\right\}\\
&\qquad - \#\left\{k\in\intint{1,U_1}: X^{(1)}_t(k) > x\right\} + M
\leq M.
\end{align*}
So it suffices to show that there exist $c>0$ such that, for any fixed $k$ and $M$,
\begin{align}\label{e.overtaking to show}
\P\left(X^{(2)}_t(k) > X^{(1)}_t(k+M)\right) \leq N\exp(-cM);
\end{align}
indeed, then for any $M\geq (\log N)^2$ we may do a union bound over $k\in\intint{1, \min(U_1-M, U_2)}$, noting that the hypothesis that $\smash{h^{(i)}_0(-N-1)} \leq N$ for $i\in\{1,2\}$ implies that $U_1, U_2\leq N$ (and using that $N^2\exp(-\frac{1}{2}c(\log N)^2)$ is bounded by an absolute constant for all $N$).

\medskip

\noindent\textbf{Reduction to overtaking event for totally uncoupled arrow.} We introduce some terminology to show \eqref{e.overtaking to show}.  We will say an arrow in system 2 ``overtakes'' an arrow in system 1 at a location $v$ when the following situation occurs:  $v$ lies on the trajectory of the specified arrow in system 2 and
\begin{align}\label{e.overtake definition}
(a_v^{(1)}, i_v^{(1)}; b_v^{(1)}, j_v^{(1)}) = (0, 1; 0, 1)\quad \text{and} \quad(a_v^{(2)}, i_v^{(2)}; b_v^{(2)}, j_v^{(2)}) = (1, 0; 1, 0),
\end{align}
where for $i\in\{1,2\}$, $(a_v^{(i)}, i_v^{(i)}; b_v^{(i)}, j_v^{(i)})$ are the vertically entering, horizontally entering, vertically exiting, and horizontally exiting arrows, respectively, at vertex $v$ in system $i$. Note that this implies that the arrows in both systems are uncoupled up to the point of exiting $v$. 

We claim that, on the event in \eqref{e.overtaking to show}, it must be the case that at at least $M$ vertices, the arrow in system 2 corresponding to $X^{\smash{(2)}}_t(k)$ overtook an arrow in system $1$. Indeed, this follows from the observation that, for any $j$ such that $\smash{X^{(2)}_0(k) \leq X^{(1)}_0(j)}$, if the arrow corresponding to $\smash{X^{\smash{(2)}}_t(k)}$ does not overtake the arrow corresponding to $X_t^{(1)}(j)$ at any vertex, then $X^{(2)}_t(k) \leq X^{(1)}_t(j)$; this is because uncoupled arrows cannot cross except by one overtaking the other as in \eqref{e.overtake definition}.

Note that the terminology of ``overtaking'' does not include cases such as $(a_v^{(1)}, i_v^{(1)}; b_v^{(1)}, j_v^{(1)}) = (0, 1; 0, 1)$ and $\smash{(a_v^{(2)}, i_v^{(2)}; b_v^{(2)}, j_v^{(2)})} = (1, 1; 1, 1)$ (see, e.g., arrows labeled $p^{(1)}$ and $r^{(2)}$ in middle panel of Figure~\ref{f.arrow pairings}) or $(a_v^{(1)}, i_v^{(1)}; b_v^{(1)}, j_v^{(1)}) = (1, 1; 1, 1)$ and $(a_v^{(2)}, i_v^{(2)}; b_v^{(2)}, j_v^{(2)}) = (1, 0; 1, 0)$ which are perhaps also suggested by the terminology; this is due to the fact that such configurations correspond to the crossing of a coupled arrow with an uncoupled one, which is not relevant to our arguments (which only concern the relative orderings of the uncoupled paths in the systems).

\medskip

\noindent\textbf{Reduction to overtaking event for an initially uncoupled arrow.} Fix an $x\in\intray{-N}$ such that $\eta_0^{(2)}(x) = 1$ and $\eta_0^{(1)}(x) = 0$. Consider the arrow in system 2 starting from $x$ at time $0$; in particular, the arrow is initially uncoupled (but may couple later). We will refer to it as ``arrow $x$.'' We will show the following: there exists $c>0$ such that
\begin{align}\label{e.reduced to show}
\MoveEqLeft[26]
 \P\Bigl(\exists v_1, \ldots,  v_M \in \Z_{\geq 1}\times \intray{-N}: \eqref{e.overtake definition} \text{ occurs at each } v_i \text{ for arrow } x\Bigr)\nonumber\\
 &\leq \exp(-cM).
 \end{align}
Taking a union bound over the at most $N$ values of $x$ with the just imposed condition will then yield \eqref{e.overtaking to show}, since the arrow corresponding to $X^{\smash{(2)}}_t(k)$ must be arrow $x$ for some such $x$.

We call a vertex $v$ ``potentially coupling'' for arrow $x$ if $(a_v^{(1)}, i_v^{(1)}) = (0, 1)$ and $(a_v^{(2)}, i_v^{(2)}) = (1, 0)$, i.e., in system 1 an arrow enters $v$ horizontally and no arrow enters vertically, and in system 2 arrow $x$  enters $v$ vertically and no arrow enters horizontally. 
We say that arrow  $x$ in system 2 passes through a potentially coupling vertex $v$ without coupling if $v$ is potentially coupling for arrow $x$ and $(B_v^{\shortuparrow}, B_v^{\shortrightarrow}) \in \{(1,1), (0,0)\}$, i.e, the output arrow configurations at $v$ differ in the two systems: the arrow in system $1$ exits horizontally and arrow $x$ in system $2$ exits vertically, or vice versa. Recall that if arrow $x$ in system 2 couples at a potentially coupling vertex $v$, i.e., the output direction is the same at $v$ in both systems, then for the rest of arrow $x$'s trajectory it agrees with the arrow in system 1 incident at $v$; in particular, it cannot overtake any further arrow in system $1$ (as defined in \eqref{e.overtake definition}).

We see that the event whose probability we are bounding is contained in the event that there exist at least $M$ potentially coupling vertices for arrow $x$ and it does not couple at at least the first $M$ of them; here, ``first'' is under the usual component-wise partial order, i.e., $w<v$ for $v, w\in\Z_{>0}^2$ (with $v=(v_1,v_2)$ and $w=(w_1,w_2)$) if (i) $w_1<v_1$ or (ii) $w_1=v_1$ and $w_2 < v_2$. Observe that the event of a given vertex $v$ being potentially coupling for arrow $x$ is determined by the Bernoulli random variables associated to vertices $w < v$. That is, $\{v \text{ is potentially coupling for arrow } x\} \in \sigma(\{B_w^{\shortrightarrow}, B_w^{\shortuparrow} : w < v\})$. A consequence is that, conditional on $v$ being potentially coupling for arrow $x$, the random variables $(B_{v}^{\shortrightarrow}, B_{v}^{\shortuparrow})$ are distributed as a pair of independent Bernoulli random variables of parameters $b^{\shortrightarrow}$ and $b^{\shortuparrow}$, respectively (as they are independent of the random variables $(B_{w}^{\shortrightarrow}, B_{w}^{\shortuparrow})_{w < v}$).

\medskip

\noindent\textbf{The probability bound.} Let $\tau_1 < \tau_2 < \ldots$ be the (random) coordinates of arrow $x$'s potentially coupling vertices in order of appearance along its trajectory; if there is no $j$\th potentially coupling vertex, we define $\tau_j = (\infty,\infty)$. Define the $\sigma$-algebra $\F_i$ by
\begin{align*}
\F_i := \Bigl\{\msf E\cap \{\tau_i = v\} : v\in\Z_{>0}^2,  \msf E\in\sigma\bigl(\{B_w^{\shortuparrow}, B_w^{\shortrightarrow}: w\leq v\}\bigr)\Bigr\}.
\end{align*}
Let $\msf{NC}_i$ (short for ``no coupling'') be the event that $\tau_i\neq (\infty,\infty)$ and arrow $x$ passes through $\tau_i$ without coupling, and let $\msf{NC}_{\intint{1,j}} = \cap_{i=1}^j \msf{NC}_i$ for each $j$. Note that $\msf{NC}_i, \tau_i \in \F_i$ for each $i$. Next note that 
$$\P(\msf{NC}_1) \leq b^\shortrightarrow b^{\shortuparrow} + (1-b^{\shortrightarrow})(1-b^{\shortuparrow}) =: \delta < 1.$$
We claim that $\P(\msf{NC}_i \mid \msf{NC}_{\intint{1,i-1}}) \leq \delta$ for each $i$. By the tower property of conditional expectations, this is implied by the stronger inequality 
$$\P(\msf{NC}_i \mid \F_{i-1})\one_{\msf{NC}_{\intint{1,i-1}}} \leq \delta \one_{\msf{NC}_{\intint{1,i-1}}}.$$
The latter holds since, conditional on $\F_{i-1}$, $\{\tau_i = v\}$ is determined by the collection of random variables $\{B_w^{\shortrightarrow}, B_w^{\shortuparrow} : \tau_{i-1} \leq w < v\}$; in particular, letting $\msf E_v$ be the event that $(B_v^{\shortrightarrow}, B_v^{\shortuparrow}) = (1,1)$ or $(0,0)$, we see that $\msf E_v$ is conditionally independent of $\{\tau_i = v\}$ given $\F_{i-1}$, so
\begin{align*}
\P\left(\msf{NC}_i \mid \F_{i-1}\right)\one_{\msf{NC}_{\intint{1,i-1}}}
&= \sum_{v > \tau_{i-1}}\P\left(\msf E_v, \tau_i = v \mid \F_{i-1}\right)\one_{\msf{NC}_{\intint{1,i-1}}}\\
&= \sum_{v > \tau_{i-1}}\P\left(\msf E_v\mid \F_{i-1}\right) \cdot \P\left(\tau_i = v\mid \F_{i-1}\right)\one_{\msf{NC}_{\intint{1,i-1}}}\\
&= \delta \sum_{v > \tau_{i-1}} \P\left(\tau_i = v\mid \F_{i-1}\right)\one_{\msf{NC}_{\intint{1,i-1}}} \leq \delta \one_{\msf{NC}_{\intint{1,i-1}}}.
\end{align*}
Thus it holds that
\begin{align*}
\P\left(\msf{NC}_{\intint{1,M}}\right) = \P(\msf{NC}_1)\cdot \prod_{i=2}^M\P\left(\msf{NC}_i\mid \msf{NC}_{\intint{1,i-1}}\right)  \leq \delta^M.
\end{align*}
This establishes \eqref{e.reduced to show} and completes the proof.
\end{proof}

\begin{proof}[Proof of Corollary~\ref{c.flexible approximate monotonicity}]
Consider two S6V models with the following pairs of initial conditions. In the first system, let $h^{(i), 1}_0 = h^{(i)}_0$ for $i\in\{1,2\}$, where $h^{(1)}_0$ and $h^{(2)}_0$ are as in the statement of Proposition~\ref{p.s6v approximate monotonicity}. In the second system, let $h^{(1),2}_0$ and $h^{(2),2}_0$ be defined by setting, for $i\in\{1,2\}$, $h^{(i),2}_0(x) = h^{(i)}_0(x)$ for $x\in\intint{-N, N}$, $h^{(i),2}_0(x) = h^{(i), 2}_0(-N)$ for $x\leq -N$, and $h^{(i),2}_0(x) = h^{(i), 2}_0(N)$ for $x\geq N$. Thus the initial configurations of the second system satisfy
\begin{align*}
h^{(1), 2}_0(x) + H \geq h^{(2), 2}_0(x) \quad\text{ for all } x\in\intray{-N-1}.
\end{align*}
Coupling all the evolutions across both systems via the basic coupling, we see from Proposition~\ref{p.s6v approximate monotonicity} that, with probability at least $1-C\exp(-cM)$,
\begin{align*}
h^{\mrm{S6V}}(h_0^{(1), 2};x,t) + H\geq h^{\mrm{S6V}}(h_0^{(2),2}; x,t) - M
\end{align*}
for all $x\in\intray{-N-1}$. By Lemma~\ref{l.finite speed of prop for s6v basic coupling}, with probability at least $1-C\exp(-cN)$, $h^{\mrm{S6V}}(h_0^{(i),1}; x,t) = h^{\mrm{S6V}}(h_0^{(i),2}; x,t)$ for $i\in\{1,2\}$ and all $(t,x)\in \intint{1,T}\times \intint{-\frac{1}{2}N, \frac{1}{2}N}$ as long as $T \leq \frac{1}{4}(1-b^{\shortrightarrow})N$. Since $C\exp(-cN) < C\exp(-cM)$, this completes the proof.
\end{proof}

\subsection{Finite speed of discrepancy for S6V under basic coupling}\label{s.finite speed of prop s6v basic coupling}

Here we give the proof of Lemma~\ref{l.finite speed of prop for s6v basic coupling}; it bears similarities to that of \cite[Proposition 2.17]{aggarwal2020limit}.

\begin{proof}[Proof of Lemma~\ref{l.finite speed of prop for s6v basic coupling}]
We will prove the statement for the case of $m=1$ and use a union bound to obtain the case of general $m\in\N$. Since we are in the $m=1$ case, we may assume without loss of generality that $s_1 = 1$, and we drop the second superscript in the boundary conditions and write them, for $i\in\{1,2\}$, as $h_0^{(i)}$, and analogously we write $(j^{(i)}_{(x,y)})_{(x,y)\in\intray{-N_i}\times\intray{s}}$, for $i\in\{1,2\}$, for the samples from the S6V models with respective boundary conditions $\smash{h_0^{(i)}}$ (coupled via the basic coupling).
We adopt the convention for arrow trajectories in each system (i.e., the two S6V models with initial conditions $h_0^{(1)}$ and $h_0^{(2)}$) from the beginning of the proof of Proposition~\ref{p.s6v approximate monotonicity} from Section~\ref{s.proof of approximate monotonicity} (recall Figure~\ref{f.arrow pairings}): in particular, all arrows in the two systems which start from the same positions (which recall we refer to as being ``coupled'') evolve first, and then the initially uncoupled arrows evolve subsequently, ``bouncing'' off of other uncoupled arrows and bouncing off or passing through coupled arrows. Recall that arrows which are uncoupled at least until a given time in their trajectories maintain their initial ordering until at least that time. Observe that all arrows which initially start in either system in $\intint{-N_1,N_1}$ are coupled by assumption.

Let $k$ be the largest value in $\intint{-N_2, -N_1-1}$ such that there is an uncoupled arrow present at $(1,k)$ in the configuration $\smash{\eta_{h_0^{(2)}}}$ corresponding to $h_0^{(2)}$ (as in \eqref{e.s6v particle from h}) (thus there is no arrow at $(1,k)$ in $\eta_{h_0^{(1)}}$). Now, we see that the event that $j^{(1)}_{(x,y)} \neq j^{(2)}_{(x,y)}$ for some $(x,y)\in\intint{1,T}\times\intint{-\frac{1}{2}N_1, \frac{1}{2}N_1}$ is contained in the event that the uncoupled arrow starting initially at $(1,k)$ in $\smash{\eta_{h_0^{(2)}}}$ is incident on a vertex in $\intint{1,T}\times\{-\floor{\frac{1}{2}N_1}\}$. 

For any $y,t\in\N$, consider the event $\msf E_{y,t}$ that the uncoupled arrow exits horizontally from $(t,x)$ and then exits horizontally from $(t+1,x+y')$ for some $y'\geq y$, i.e., it travels vertical distance at least $y$ in the time step $t$ to $t+1$. Let $n$ be the number of coupled arrows that exit horizontally between $(t,x+1)$ and $(t,x+y-1)$ (inclusive). Now, if $\msf E_{y,t}$ occurs, it implies that all $n$ of these coupled arrows passed straight through the vertex they horizontally entered on the vertical line $t+1$, and the uncoupled arrow passed vertically through at least $y-n -1$ many vertices which had no coupled arrow incident (it must also have not passed horizontally through $(t+1,x)$, but we will not include this in the probability upper bound). Since all these sub-events are independent, the probability of $\msf E_{y,t}$ occurring is at most $(b^{\shortrightarrow})^{n}(b^{\shortuparrow})^{y-n-1} \leq (b^{\shortrightarrow})^{y-1}$, since $b^{\shortuparrow}\leq b^{\shortrightarrow}$.

Thus the distance the arrow starting at $(1,k)$ travels vertically in the horizontal interval $\intint{1,T}$ is bounded by a sum of $T$ independent random variables $(X_i)_{i\in\intint{1,T}}$, each distributed as $\P\left(X_i \geq r\right) = (b^{\shortrightarrow})^{r-1}$ for $r\in\N$. This sum has expectation $(1-b^{\shortrightarrow})^{-1}T$. So the probability that this sum exceeds $\frac{1}{2}N_1$, under the assumed condition that $T\leq \frac{1}{4}(1-b^{\shortrightarrow})N_1$, is bounded by $\exp(-cN_1)$ by standard concentration inequalities for sums of independent geometric random variables (e.g., \cite[Theorem 2.1]{janson2018tail}), completing the proof.
\end{proof}

\section{Decoupling along characteristics}\label{s.proofs of asymptotic independence}

In this section we give the proofs of Corollaries~\ref{p.asymptotic independence ASEP} and \ref{p.covariance matrix}. The former will be proved in Section~\ref{s.asymptotic independence of kpz fixed point} as a consequence of an analog of it for the limiting objects, KPZ fixed points (Proposition~\ref{p.asymptotic independence of kpz fixed point}), which will also be proven in the same section. Corollary~\ref{p.covariance matrix} will be proven in Section~\ref{s.space-time covariance proofs}, using an asymptotic independence statement for the stationary horizon (Lemma~\ref{l.sh representation}), which in turn will be proven in Section~\ref{s.stationary horizon asymptotic independence}.

\subsection{Asymptotic independence for the KPZ fixed point}\label{s.asymptotic independence of kpz fixed point}
Let $\h^{(i)}(x,t) = \h(\h_0^{(i)}, 0; x,t)$ be the KPZ fixed point started from $\h_0^{(i)}$, coupled via the directed landscape, i.e., for $i\in\{1,2\}$,
\begin{align}\label{e.kpz fixed point variational formula}
\h^{(i)}(x,t) = \sup_{y\in\R}\left(\h^{(i)}_0(y) + \mc L(y,0; x,t)\right).
\end{align}
The following is the analogous statement of Corollary~\ref{p.asymptotic independence ASEP} for the limiting objects, namely, KPZ fixed points coupled via the directed landscape. It will be proven later in this section.

\begin{proposition}\label{p.asymptotic independence of kpz fixed point}
Fix $R>0$ and $T\geq 1$. For any $\slope\in\R$, $\varepsilon>0$, and $i\in\{1,2\}$, suppose $\smash{\h^{(1)}_0, \h^{(2)}_0} :\R\to\R$ are independent (in particular, they may be deterministic), continuous, and satisfy \eqref{e.uniform slope condition} with $\slope$ and $R$ almost surely. There exist constants $C, \delta>0$ depending on $R$ and $T$ such that the following holds. (1) There exist independent processes $\tilde \h^{(1)}, \tilde \h^{(2)} : \R\times[T^{-1}, T]\to\R$ and a coupling with $(\h^{(1)}, \h^{(2)})$ such that, for $i\in\{1,2\}$,
\begin{equation}\label{e.asymptotic independence}
\lim_{|\slope|\to\infty} \P\left(\sup_{s\in[T^{-1},T]}\sup_{|x|\leq \delta|\slope|^{1/2}}\left|\h^{(i)}(x,s) - \tilde \h^{(i)}(x,s)\right| \geq C|\slope|^{-1/12}\log|\slope|\right) = 0.
\end{equation}
(2) There exist independent processes $\tilde \h^{(1)}_{\mrm{char}}, \tilde \h^{(2)}_{\mrm{char}}:\R\times[T^{-1}, T]\to\R$ and a coupling with $(\h^{(1)}, \h^{(2)})$ such that
\begin{align}\label{e.asymptotic independence along characteristic}
\lim_{|\slope|\to\infty}\P\left(\sup_{s\in[T^{-1},T]}\sup_{|x|\leq \delta|\slope|^{1/2} t}\left|\h^{(1)}(x,s) - \tilde \h^{(1)}_{\mrm{char}}(x,s)\right| \geq C|\slope|^{-1/12}\log|\slope|\right) &= 0 \quad\text{and}\\
\lim_{|\slope|\to\infty}\P\left(\sup_{s\in[T^{-1},T]}\sup_{|x+\slope s|\leq \delta\slope^{1/2} t}\left|\h^{(2)}(x,s) - \tilde \h^{(2)}_{\mrm{char}}(x,s)\right| \geq C|\slope|^{-1/12}\log|\slope|\right) &= 0.\nonumber
\end{align}
\end{proposition}

Now we give the proof of Corollary~\ref{p.asymptotic independence ASEP}, which is essentially an immediate consequence of Proposition~\ref{p.asymptotic independence of kpz fixed point}.

\begin{proof}[Proof of Corollary~\ref{p.asymptotic independence ASEP}]
We prove \eqref{e.asymptotic independence asep}; an analogous argument holds for \eqref{e.asymptotic independence along characteristic asep}. First recall from Theorem~\ref{t.asep general initial condition} that, for $i\in\{1,2\}$, $\h^{(i),\varepsilon}$ converges in distribution (in the product topology, across $\mc T$, of uniform convergence on compact sets), as $\varepsilon\to0$, to $\h^{(i)}$ as given by \eqref{e.kpz fixed point variational formula}, coupled across $i\in\{1,2\}$ via the same directed landscape $\mc L$. Note that \eqref{e.asymptotic independence asep} is a statement about the $\varepsilon\to 0$ limit, so we may define any processes $\tilde\h^{(i), \varepsilon}$ and coupling of $(\h^{(i), \varepsilon})_{i\in\{1,2\}}$ with $(\tilde\h^{(i), \varepsilon})_{i\in\{1,2\}}$ for positive $\varepsilon$ as long as the respective limiting processes $(\h^{(i)})_{i\in\{1,2\}}$ and $(\tilde\h^{(i)})_{i\in\{1,2\}}$ are coupled so that \eqref{e.asymptotic independence} is satisfied and $\tilde\h^{(1)}$ and $\tilde\h^{(2)}$ are independent. By Proposition~\ref{p.asymptotic independence of kpz fixed point}, there exist such limiting processes $(\tilde{\h}^{(i)})_{i \in \{1,2\}}$ with such a coupling with $(\h^{(i)})_{i \in \{1,2\}}$.
We define $\tilde \h^{(i),\varepsilon} = \tilde \h^{(i)}$ for all $\varepsilon>0$ and $i\in\{1,2\}$. Since, as elements of $\mc C([-\delta\slope^{1/2}, \delta\slope^{1/2}])^{\mc T}$, $\smash{(\h^{(i),\varepsilon}, \tilde \h^{(i), \varepsilon}) = (\h^{(i),\varepsilon}, \tilde \h^{(i)}) \stackrel{d}{\to} (\h^{(i)}, \tilde \h^{(i)})}$ as $\varepsilon\to0$, and since $\mc T$ is finite, the Portmanteau theorem yields that \eqref{e.asymptotic independence asep} is upper bounded by
\begin{align*}
 \lim_{|\slope|\to\infty}\,\P\left(\max_{s\in\mc T}\sup_{|x|\leq \delta|\slope|^{1/2}} \left|\h^{(i)}(x,s) - \tilde \h^{(i)}(x,s)\right| \geq C|\slope|^{-1/12}\log|\slope|\right),
\end{align*}
which equals zero by Proposition~\ref{p.asymptotic independence of kpz fixed point}.
\end{proof}

Next we turn to Proposition~\ref{p.asymptotic independence of kpz fixed point}. First, let $\mc K(y,s;x,t) : \R^4_{\uparrow}\to\R$ be defined by 
$$\mc K(y,s;x,t) = \mc L(y,s;x,t) + (x-y)^2/(t-s);$$
it satisfies certain stationarity properties \cite[Lemma 10.2]{dauvergne2018directed}, e.g., its one-point distribution at any $(y,s;x,t)$ is given by the GUE Tracy-Widom distribution.

The following localization statement will be used in the proof of Proposition \ref{p.asymptotic independence of kpz fixed point} and its proof in turn will be given after. It restricts the time domain to $[t,t+1]$ for arbitrary $t>0$; as we will see in the proof of  Proposition~\ref{p.asymptotic independence of kpz fixed point}, this will suffice for our purposes by scale invariance properties of the limiting objects.

\begin{lemma}\label{l.localization of maximizer}
Let $\h_0$ satisfy the second inequality in \eqref{e.uniform slope condition} for some $\slope\in\R$ and let $\h(x,t) = \sup_{y\in\R}(\mc L(y,0;x,t) + \h_0(y))$. Let $R$ be as in \eqref{e.uniform slope condition}. The supremum has a minimum maximizer $y^*(x,t)$ and, for each $t>0$ fixed, $x\mapsto y^*(x,t)$ is non-decreasing.
Additionally, for any $K\geq \smash{(1+|\slope|)^{-1/4}}$ and any fixed $x,t$ with $|x|\leq K(1+|\slope|)t$, there exist constants $C,c>0$ depending only on $R$ and $t$ such that the following holds. With probability at least $1-C\exp(-cK^{3/2}(1+|\slope|^{3/2}))$, 
\begin{align}\label{e.h^(2) variational formula}
\sup_{s\in[t,t+1]}|y^*(x,s) - x - \slope s| \leq K(1+|\slope|) t.
\end{align}
\end{lemma}

\begin{proof}[Proof of Proposition~\ref{p.asymptotic independence of kpz fixed point}]
We give the proof of \eqref{e.asymptotic independence} in detail, and at the end we will explain the modifications needed to establish \eqref{e.asymptotic independence along characteristic}.
We assume that $\slope \geq 0$, as the argument for $\slope<0$ is analogous. By scale invariance properties of the directed landscape and KPZ fixed point (see Definition~\ref{d.directed landscape} and \cite[Theorem 4.5]{matetski2016kpz}, respectively), it suffices to prove the statement for the time interval $[t,t+1]$ instead of $[T^{-1},T]$; one then takes $t = (T^{2}-1)^{-1}$ and stretches the time dimension by a factor of $T-T^{-1}$. Finally, as the proposition only concerns the limit $\slope\to\infty$, we will assume $\slope$ is large when required. We set the value $\delta = \frac{1}{16}$, so $|x|\leq \frac{1}{16}\slope^{1/2}t$ in \eqref{e.asymptotic independence} and \eqref{e.asymptotic independence along characteristic}; however we will use $\delta$ in the proof for notational convenience instead of its explicit value.

Recall $\h^{(2)}(x,s) = \sup_{y\in\R}(\h^{(2)}_0(y) + \mc L(y,0;x,s))$.  By the monotonicity in the maximizer from Lemma~\ref{l.localization of maximizer} and since $|x| \leq \delta\slope^{1/2} t$, we know that the minimum maximizer $y^*(x,s)$ of this supremum satisfies $y^*(-\delta\slope^{1/2} t, s) \leq y^*(x,s) \leq y^*(\delta\slope^{1/2} t, s)$ for all $s\in[t,t+1]$. The probability bounds asserted in Lemma~\ref{l.localization of maximizer} can be used (with $K = \frac{1}{16}\frac{|\slope|}{1+|\slope|}$ and $c=c(R,t)$ with $R$ as in \eqref{e.uniform slope condition}) to control the left and righthand sides of these inequalities to yield that, with probability at least $1-C\exp(-c\slope^{3/2})$, for all $|x|\leq \delta\slope^{1/2} t$ and $s\in[t,t+1]$,
\begin{align*}
\h^{(2)}(x,s) = \sup_{y\in[\frac{7}{8}\slope t, \frac{9}{8}\slope t]}\left(\h^{(2)}_0(y) + \mc L\bigl(y,0;x,s\bigr)\right).
\end{align*}
Next, by Definition~\ref{d.directed landscape} of the directed landscape, with probability one, it holds for all $x,y\in\R$ and $s\in[t,t+1]$ that $\mc L(y,0;x,s) = \sup_{z\in\R}(\mc L(y,0;z,s(1-\slope^{-1/2})) + \mc L(z,s(1-\slope^{-1/2}); x,s))$. So we write, on an event with probability at least $1-C\exp(-c\slope^{3/2})$, for all $s\in[t,t+1]$,
\begin{align*}
\h^{(2)}(x,s) = \sup_{y\in[\frac{7}{8}\slope t, \frac{9}{8}\slope t]}\sup_{z\in\R}\left(\h^{(2)}_0(y) + \mc L\bigl(y,0;z,s(1-\slope^{-1/2})\bigr) + \mc L\bigl(z,s(1-\slope^{-1/2}); x,s\bigr)\right).
\end{align*}

We wish to put bounds on the set of $z$ where the inner supremum can be attained.  For each $y$ and $s$, let $z_s^*(y,x)$ be the smallest such maximizer; it exists by the fact that $\lim_{|y|\to\infty} (\mc L(y,0;x,s) + \h_0(y)) = -\infty$ almost surely; to see the latter we use \eqref{e.uniform slope condition} and the fact that $\mc L(y,0;x,s) = -s^{-1}(x-y)^2 + \mc K(y,0;x,s)$ with $\mc K(y,0;x,s)$ being of unit order. 

We recognize $z_s^*(y,x)$ as the location at height $(1-\slope^{-1/2})s$ of a directed landscape geodesic from $(0,y)$ to $(x,s)$. By geodesic ordering \cite[Lemma 2.7]{bates2019hausdorff}, for each fixed $x$ and $s$, $y\mapsto z_s^*(y,x)$ is a non-decreasing function and for each fixed $y$ and $s$, $x\mapsto z_s^*(y,x)$ is also non-decreasing. Thus it is sufficient to lower bound $z_s^*(\frac{7}{8}\slope t, -\delta\slope^{1/2}t)$ and upper bound $z_s^*(\frac{9}{8}\slope t, \delta\slope^{1/2}t)$, uniformly over $s\in[t,t+1]$. By transversal fluctuation estimates on geodesics in the directed landscape which are uniform in the endpoints (see \cite[Lemma 3.11]{ganguly2022fractal}), we know that
\begin{align*}
\MoveEqLeft[12]
\P\left(\sup_{s\in[t,t+1]}\left|z_s^*(\tfrac{7}{8}\slope t, -\delta\slope^{1/2}t)- (1-\slope^{-1/2})(-\delta\slope^{1/2}t) - \slope^{-1/2} \cdot\tfrac{7}{8}\slope t\right| > \tfrac{1}{16}\slope^{1/2}t\right)\\
&\leq C\exp(-c\slope t^2),\\\\
\MoveEqLeft[12]
\P\left(\sup_{s\in[t,t+1]}\left|z_s^*(\tfrac{9}{8}\slope t,\delta\slope^{1/2}t)- (1-\slope^{-1/2})(\delta\slope^{1/2}t) - \slope^{-1/2} \cdot\tfrac{9}{8}\slope t\right| > \tfrac{1}{16}\slope^{1/2}t\right)\\
&\leq C\exp(-c\slope t^2)
\end{align*}
(the righthand side bounds are not optimal).
So we see that, with probability at least $1-C\exp(-c\slope t^2)$, for all $s\in[t,t+1]$, $y\in[\frac{7}{8}\slope t, \frac{9}{8}\slope t]$, and $|x|\leq \delta\slope^{1/2}t$
$$\tfrac{3}{4}\slope^{1/2}t\leq z_s^*(y,x) \leq \tfrac{5}{4}\slope^{1/2}t,$$
since $\delta = \frac{1}{16}$.
In other words, absorbing $t$ into the constant $c$, with probability at least $1-C\exp(-c\slope)$, for all $s\in[t,t+1]$,
\begin{align*}
\MoveEqLeft[0]
\h^{(2)}(x,s)\\
&=  \sup_{y\in[\frac{7}{8}\slope t, \frac{9}{8}\slope t]}\sup_{z\in [\frac{3}{4}\slope^{1/2}t, \frac{5}{4}\slope^{1/2}t]}\left(\h^{(2)}_0(y) + \mc L\bigl(y,0;z,s(1-\slope^{-1/2})\bigr) + \mc L\bigl(z,s(1-\slope^{-1/2}); x,s\bigr)\right).
\end{align*}

Now, applying Lemma~\ref{l.localization of maximizer} to $\h^{(1)}_0$ (recall that it satisfies \eqref{e.uniform slope condition}) with $K = \frac{1}{4}\frac{\slope}{1+\slope} t$ to bound the maximizer of the supremum of $y$ and a similar argument as above to bound the maximizer of the supremum over $z$ yields that, with probability at least $1-C\exp(-c\slope)$, for all $|x|\leq \delta\slope^{1/2}t$,
\begin{align*}
\MoveEqLeft[0]
\h^{(1)}(x,t)\\
&= \sup_{|y| \leq \frac{3}{8}\slope t}\sup_{|z| \leq \frac{1}{2}\slope^{1/2}t}\left(\h^{(1)}_0(y) + \mc L\bigl(y,0;z,s(1-\slope^{-1/2})\bigr) + \mc L\bigl(z,s(1-\slope^{-1/2}); x,s\bigr)\right).
\end{align*}
Now, $\mc L(z, s(1-\slope^{-1/2}); x,s) = -s^{-1}\slope^{1/2}(x-z)^2 + \mc K(z, s(1-\slope^{-1/2}); x,s)$, and it follows from \cite[Proposition 10.6]{dauvergne2018directed} that,  with probability at least $1-C\exp(-c\slope^{1/8})$,
\begin{equation}\label{e.airy sup bound}
\sup_{s\in[t,t+1]}\sup_{|x|\leq \delta\slope^{1/2}t}\sup_{|z|\leq 2\slope^{1/2}t} |\mc K(z, s(1-\slope^{-1/2}); x,s)| \leq C\slope^{-1/12}\log\slope.
\end{equation}
Thus we see that, with probability at least $1-C\exp(-c\slope^{1/8})$, for all $s\in[t,t+1]$ (and with the implicit constant in the $O$ notation below depending only on $t$),
\begin{align*}
\MoveEqLeft[2]
\h^{(1)}(x,s)\\
&= \sup_{|y|\leq \frac{3}{8}\slope t}\sup_{|z| \leq \frac{1}{2}\slope^{1/2}t}\left(\h^{(1)}_0(y) + \mc L\bigl(y,0;z,s(1-\slope^{-1/2})\bigr) - \frac{\slope^{1/2}(x-z)^2}{s}\right)\\
&\qquad + O(\slope^{-1/12}\log\slope).\\\\
\MoveEqLeft[2]
\h^{(2)}(x,s)\\
&=  \sup_{y\in[\frac{7}{8}\slope t, \frac{9}{8}\slope t]}\sup_{z\in [\frac{3}{4}\slope^{1/2}t, \frac{5}{4}\slope^{1/2}t]}\left(\h^{(2)}_0(y) + \mc L\bigl(y,0;z,s(1-\slope^{-1/2})\bigr) - \frac{\slope^{1/2}(x-z)^2}{s}\right)\\
&\qquad + O(\slope^{-1/12}\log\slope).
\end{align*}
Observe that the ranges of $y$ in the above two suprema are separated by a distance of order $\slope$, and those of $z$ are separated by distance of order $\slope^{1/2}$. Now, it follows from \cite[Proposition 2.6]{NTM} (after using scale invariance of $\mc L$) that there exist independent directed landscapes $\tilde{\mc L}_1$ and $\tilde{\mc L}_2$ such that, with probability which converges to $1$ as $\slope\to\infty$,
\begin{equation}\label{e.landscape coupling}
\begin{split}
&\mc L(u) = \tilde{\mc L}_1(u) \text{ for all } u = (y,0;z,s) \text{ with } |y|\leq \tfrac{3}{8}\slope t, |z| \leq \tfrac{1}{2}\slope^{1/2}t, s\in[0,t+1]\quad\text{and}\\
&\mc L(u) = \tilde{\mc L}_2(u) \text{ for all } u = (y,0;z,s) \text{ with } y\in[\tfrac{7}{8}\slope t, \tfrac{9}{8}\slope t], z \in[\tfrac{3}{4}\slope^{1/2}t,\tfrac{5}{4}\slope^{1/2}t], s\in[0,t+1].
\end{split}
\end{equation}
This completes the proof of \eqref{e.asymptotic independence} by taking
\begin{align*}
\tilde \h^{(1)}(x,s) &= \sup_{|y|\leq \frac{3}{8}\slope t}\sup_{|z| \leq \frac{1}{2}\slope^{1/2}t}\left(\h^{(1)}_0(y) + \tilde{\mc L_1}\bigl(y,0;z,s(1-\slope^{-1/2})\bigr) - s^{-1}\slope^{1/2}(x-z)^2\right)
\end{align*}
and
\begin{align*}
\MoveEqLeft[0]
\tilde \h^{(2)}(x,s)\\
&= \!\sup_{y\in [\frac{7}{8}\slope t, \frac{9}{8}\slope t]}\sup_{z \in [\frac{3}{4}\slope^{1/2}t, \frac{5}{4}\slope^{1/2}t]}\left(\h^{(1)}_0(y) + \tilde{\mc L_2}\bigl(y,0;z,s(1-\slope^{-1/2})\bigr) - s^{-1}\slope^{1/2}(x-z)^2\right).
\end{align*}

A very similar argument also yields \eqref{e.asymptotic independence along characteristic} on the asymptotic independence of $\h^{(1)}(0, s)$ and $\h^{(2)}(-\slope s, s)$, uniformly for $s\in[t,t+1]$, in the $\slope\to\infty$ limit, i.e., evaluating each height function at a point which lies on the characteristic emanating from the origin. In more detail, by Lemma~\ref{l.localization of maximizer}, it follows that the maximizer for each height function in the variational formula $\sup_{y}(\smash{\h^{(i)}_0}(y) + \mc L(y,0; x,s))$ is, for some large constant $K$, attained inside $\smash{[-K\delta\slope^{1/2} t, K\delta\slope^{1/2} t]}$ with probability at least $1-C\exp(-c\slope^{3/4})$ for all $s\in[t,t+1]$ and $|x|\leq \delta\slope^{1/2}t$. Then, analogous to what was done in the above proof of \eqref{e.asymptotic independence}, one introduces a further supremum at height $\slope^{-1/2}s$ and shows using geodesic fluctuation estimates \cite[Lemma 3.11]{ganguly2022fractal} that the maximizers of that supremum of the two height functions are separated by at least $c\slope^{1/2}$ for some $c>0$ with probability at least $1-C\exp(-c\slope)$ (in the above proof the height chosen was $(1-\slope^{-1/2})s$ instead as the endpoints were close at height $s$ and not $0$). The contribution to the height functions from the temporal strip $[0,\slope^{-1/2}s]$ is again $O(\slope^{-1/12}\log\slope)$ with probability at least $1-C\exp(-c\slope^{1/8})$ as $\slope\to\infty$ and one can then again invoke the asymptotic independence from \cite[Proposition 2.6]{NTM} to complete the proof.
\end{proof}

\begin{remark}[Upgrading Proposition~\ref{p.asymptotic independence of kpz fixed point} to a quantitative statement]
An inspection of the proof of \cite[Proposition 2.6]{NTM} shows that it can be upgraded to obtain that \eqref{e.landscape coupling} holds with at least probability $1-C\exp(-c\slope^{3/2})$, i.e., a quantitative probability lower bound can be derived instead of just that the limit as $\slope\to\infty$ is $1$. With this estimate, it would follow that the limit \eqref{e.asymptotic independence} in the statement of Proposition~\ref{p.asymptotic independence of kpz fixed point} (and thus also the limit in $\slope$ in \eqref{e.asymptotic independence asep} in Corollary~\ref{p.asymptotic independence ASEP}) can be replaced by a lower bound of $1-C\exp(-c\slope^{1/8})$ (coming from \eqref{e.airy sup bound}, which then becomes the bottleneck in the overall probability estimate). To obtain the upgraded estimate on the probability of \eqref{e.landscape coupling} holding, one needs an upper bound of $C\exp(-c\slope^{3/2})$ on the probability that a geodesic in the prelimiting LPP model of Poissonian LPP has transversal fluctuations at least order $\slope^{1/2}$, the minimum separation between the areas of interest; the latter estimate is recorded in \cite[Theorem 2.6]{hammond2018modulus}.
\end{remark}

\begin{remark}[Weakening the independence assumption]\label{r.weakening independence assumption}
One can also obtain the coupling in Proposition~\ref{p.asymptotic independence of kpz fixed point} (and thus also a form of Corollary~\ref{p.asymptotic independence ASEP}) under the weaker assumption that, with $R>0$ fixed, for every $\slope\in\R$, the initial conditions $
\smash{(\h_0^{(1)}, \h_0^{(2)})}$ satisfying \eqref{e.uniform slope condition} with $\slope$ and $R$ can be coupled with independent initial conditions $(\h_0^{(1), \mrm{ind}},\h_0^{(2), \mrm{ind}})$ such that, for $i\in\{1,2\}$,
\begin{align*}
\lim_{\slope\to\infty}\P\left(\sup_{|x|\leq \slope^2} \left|\h_0^{(i)}(x) - \h_0^{(i), \mrm{ind}}(x)\right| \geq C|\slope|^{-1/12}\log|\slope|\right) = 0;
\end{align*}
such a condition is satisfied, for example, for the stationary horizon (see Lemma~\ref{l.sh representation} ahead).
Indeed, on the intersection of the complement of the event in the previous display and the event that for all $(x,t)\in[-\delta\slope^{1/2},\delta\slope^{1/2}]\times[T^{-1},T]$ the maximizers in $\h(\smash{\h_0^{(i)}}, 0; x,t) = \sup_{y\in\R}(\mc L(y,0;x,t) + \smash{\h^{(i)}_0(y)})$ and the analogous variational formula for $\h(\smash{\h_0^{(i), \mrm{ind}}}, 0; x,t)$ both lie in $[-\slope^2,\slope^2]$ for $i\in\{1,2\}$, it follows that $\h(\smash{\h_0^{(i)}}, 0; x,t)$ and $\h(\smash{\h_0^{(i), \mrm{ind}}}, 0; x,t)$ differ by at most $C|\slope|^{-1/12}\log|\slope|$ for all such $(x,t)$. One can then apply Proposition~\ref{p.asymptotic independence of kpz fixed point} to obtain  a process $\tilde \h$ which is $C|\slope|^{-1/12}\log|\slope|$ close to $\smash{\h(\h_0^{(i), \mrm{ind}}, 0; \bm\cdot,\bm\cdot)}$ and thus also to $\h(\smash{\h_0^{(i)}}, 0; \bm\cdot,\bm\cdot)$ (changing $C$ to $2C$) on an event whose probability tends to 1. The probability of the condition that the maximizers lie in $[-\slope^2,\slope^2]$ (that the above holds on) can be bounded using Lemma~\ref{l.localization of maximizer}.
\end{remark}

\begin{proof}[Proof of Lemma~\ref{l.localization of maximizer}]
In the proof we assume $\slope\geq 0$; an analogous argument applies to the case of $\slope<0$. We also note that it is sufficient to prove the result under the assumption that $K(1+\slope) \geq K_0$ for some $K_0 = K_0(R,t)$ (in addition to the already assumed $K\geq (1+\slope)^{-1/4})$; all lower values of $K$ are handled by adjusting the constant $C$ (i.e., the bound is made to be trivial in this range).

That the maximizer $y^*(x,s)$ exists again follows (as in the existence of $z^*_s(y,x)$ in the proof of Proposition~\ref{p.asymptotic independence of kpz fixed point}) from the fact that $\lim_{|y|\to\infty} (\mc L(y,0;x,s) + \h_0(y)) = -\infty$ almost surely.
\cite[Lemma 3.8]{IGCI} yields that $x\mapsto y^*(x,s)$ is non-decreasing for every fixed $s$.
Next, let 
\begin{align*}
\bar \h_0(y) = \h_0(y) - 2\slope y;
\end{align*}
then \eqref{e.uniform slope condition} implies that $|\bar \h_0(y)| \leq |y| + R\slope$ for all $y\in\R$.  Now, it is easy to calculate that $\sup_{y\in\R}(-(x-y)^2/s + 2\slope y)$ (i.e., ignoring $\bar \h_0$ and $\mc K$ in $\sup_{y\in\R} (\mc L(y,0;x,s) + \h_0(y))$) is attained at $y=\slope s + x$, where the supremum equals $\slope^2 s + 2\slope x$. By \cite[Proposition 10.6]{dauvergne2018directed}, $\sup_{s\in[t,t+1]}\mc K(\slope s + x, 0; x,s)$ is lower bounded by $-K(1+\slope)$ with probability at least $1-C\exp(-cK^{3/2}(1+\slope^{3/2}))$. So we see that, with probability at least $1-C\exp(-cK^{3/2}(1+\slope^{3/2}))$, for all $s\in[t,t+1]$,
\begin{align}\label{e.h^(2) lower bound}
\h(\slope s + x,s) \geq \slope^2 s + 2\slope x  -\slope (t+1) - |x| - R\slope -K(1+\slope).
\end{align}
We recall the estimate 
$$\P\left(\sup_{y\in\R}(\mc K(y,0;x,t) - |y-x|)> \tfrac{1}{2}K(1+\slope)\right)  \leq C\exp(-cK^{3/2}(1+\slope^{3/2}))$$
from \cite[Proposition 9.5]{ganguly2022sharp} and $\P(\sup_{s\in[t,t+1]} |\mc K(y,0;x,s)-\mc K(y,0;x,t)| \geq \frac{1}{2}K(1+\slope)) \leq C\exp(-cK^{3/2}(1+\slope^{3/2}))$ from \cite[Proposition 10.5]{dauvergne2018directed}. Together these imply that 
\begin{equation}\label{e.bound on airy sheet minus linear}
\P\left(\sup_{y\in\R, s\in[t,t+1]}\left(\mc K(y,0;x,s) - |y-x|\right)> K(1+\slope)\right) \leq C\exp\left(-cK^{3/2}(1+\slope^{3/2})\right).
\end{equation}
 Recall also that $\h_0(y) \leq 2\slope y + |y| +R\slope$. Then we see that, with probability at least $1-C\exp(-cK^{3/2}(1+\slope^{3/2}))$, for all $s\in[t,t+1]$,
\begin{align}
\MoveEqLeft[6]
\sup_{y:|y-\slope s - x|\geq K(1+\slope) t}\bigl(\mc L(y,0;x,s) + \h_0(y)\bigr)\nonumber\\
&\leq \sup_{y:|y-\slope s - x|\geq K(1+\slope) t}\left(-\frac{(x-y)^2}{s} + 2\slope y + 2|y|\right)\label{e.L supremum break up}\\
&\qquad\qquad + \sup_{y:|y-\slope s - x|\geq K(1+\slope) t}\Bigl(\mc K(y,0;x,s) - |y| + R\slope\Bigr).\nonumber
\end{align}
We may bound the first supremum as follows. By completing the square and the triangle inequality,
\begin{align*}
\MoveEqLeft[3]
-\frac{(x-y)^2}{s} + 2\slope y + 2|y|\\
&= -\frac{(x+\slope s-y)^2}{s} + 2\slope x + 2|y| + \slope^2 s\\
&\leq -\frac{(x+\slope s-y)^2}{s} + 2\slope x + 2|y-\slope s -x| + 2|x| + 2\slope s + \slope^2 s.
\end{align*}
When $|y-x-\slope s| \geq K(1+\beta)t$ and $s\in[t,t+1]$, it follows that 
\begin{align*}
-\frac{(x+\slope s-y)^2}{s} + 2|y-\slope s -x|
&\leq -\frac{K^2(1+\slope)^2 t^2}{s} + 2K(1+\slope)t\\
&\leq -\frac{K^2(1+\slope)^2 t^2}{t+1} + 2K(1+\slope)t,
\end{align*}
so that, overall (and again using $s\leq t+1$),
\begin{align*}
\MoveEqLeft[4]
\sup_{y:|y-\slope s - x|\geq K(1+\slope) t}\left(-\frac{(x-y)^2}{s} + 2\slope y + 2|y|\right)\\
&\leq \slope^2 s + 2\slope x -\frac{K^2(1+\slope)^2t^2}{t+1} + 2|x| + 2\slope (t+1) + 2tK(1+\slope).
\end{align*}

By \eqref{e.bound on airy sheet minus linear}, the second supremum in \eqref{e.L supremum break up} can be bounded as
\begin{align*}
\MoveEqLeft[7]
\sup_{y:|y-\slope s - x|\geq K(1+\slope) t}\Bigl(\mc K(y,0;x,s) - |y| + R\slope\Bigr)\\
&\leq \sup_{y:|y-\slope s - x|\geq K(1+\slope) t}\Bigl(\mc K(y,0;x,s) - |y-\slope s-x| + |\slope s+x| + R\slope\Bigr)\\
&\leq K(1+\slope) + \slope(t+1) + |x| + R\slope.
\end{align*}
Substituting these bounds on the two supremums into \eqref{e.L supremum break up}, we obtain that 
\begin{align}
\MoveEqLeft[0]
\sup_{y:|y-\slope s - x|\geq K(1+\slope) t}\bigl(\mc L(y,0;x,s) + \h_0(y)\bigr)\nonumber\\
&\leq \slope^2 s + 2\slope x - K^2(1+\slope)^2 \frac{t^2}{t+1} + 3|x| + 2 t (K(1+\slope)+\slope) + K(1+\slope) + (R+2)\slope. \label{e.h^(2) restricted left upper bound}
\end{align}

Thus \eqref{e.h^(2) lower bound} is strictly larger than \eqref{e.h^(2) restricted left upper bound} with probability at least $1-C\exp(-cK^{3/2}(1+\slope^{3/2}))$ as long as
\begin{align}\label{e.required K, slope inequality}
K^2(1+\slope)^2\frac{t^2}{t+1} > 4|x| + 2R\slope + 2(1+t)K(1+\slope) + 3\slope t + 3\slope.
\end{align}
Now, for all $K$, $\slope$ such that $K\geq (1+\slope)^{-1/4}$ and $K(1+\slope) \geq K_0$ for some fixed large constant $K_0 = K_0(R,t)$ it holds that $\slope(3 t + R+3) \leq \frac{1}{2}K^2(1+\slope)^{2}t^2/(t+1)$  (by letting $K_0\geq (4+R)^2((t+1)\cdot\max(t,1)/t^2)^2$ so that the inequality is implied by $\frac{1}{2}(4+R)K^{3/2}(1+\slope)^{3/2}\max(t,1) \geq \slope(3 t + R + 3)$, which is satisfied under the assumptions). Since also $|x|\leq K(1+\slope)t$, it follows that \eqref{e.required K, slope inequality} holds for large enough $K_0(R,t)$. Thus
\eqref{e.h^(2) lower bound} 
is larger than \eqref{e.h^(2) restricted left upper bound}   with probability at least $1-C\exp(-cK^{3/2}(1+\slope^{3/2}))$. This completes the proof.
\end{proof}

\subsection{Proofs of space-time covariance statements}\label{s.space-time covariance proofs}

We begin by proving half of Corollary~\ref{p.covariance matrix}, namely the limits of the on-diagonal terms in \eqref{e.on diagonal covvariance}, which is essentially a consequence of results of~\cite{landon2023tail}.

\begin{proof}[Proof of Corollary~\ref{p.covariance matrix}, on-diagonal limits in \eqref{e.on diagonal covvariance}]
We are studying the limits of $S^{\slope,\varepsilon}_{kk}$. The $k=2$ case was established in \cite[Corollary 2.6]{landon2023tail}.
We next deal with the $k=1$ case. By color merging (Remark~\ref{r.asep color merging}), this is equivalent to the case of uncolored ASEP with stationary initial condition of density $\frac{1}{2}+\frac{1}{2}\slope \varepsilon^{1/3}$, i.e., the particle configuration is independent Bernoulli($\frac{1}{2} + \frac{1}{2}\slope \varepsilon^{1/3}$).
Now consider ASEP run from that stationary initial condition, and introduce a second-class particle at the origin. In terms of colored particles, we have a single particle of color 1 at the origin, and the color 2 particles are placed according to an independent Bernoulli($\frac{1}{2} + \frac{1}{2}\slope \varepsilon^{1/3}$) process everywhere else. Let $X^{\slope, \varepsilon}(t)$ be the position of the color $1$ particle at time $t$ and let $\eta_t$ be the particle configuration of the color 2 particles at time $t$, i.e., $\eta_t(x) = 1$ if $x$ is occupied by a color $2$ particle and $\eta_t(x) =0$ otherwise. Then by \cite[eq. (2.1)]{balazs2009fluctuation}, for all $t>0$ and $x\in\Z$,
\begin{align*}
\P\left(X^{\slope,\varepsilon}(t) = x\right) = 4(1 - \slope^2\varepsilon^{2/3})^{-1}\cdot\Cov\left(\eta_t(x), \eta_0(0)\right) = 4(1-\slope^2\varepsilon^{1/3})S^{\slope, \varepsilon}_{11}(t, x);
\end{align*}
note that this implies $S^{\slope, \varepsilon}_{11}(t, x)\geq 0$. For $g:\Z\to\R$, let $\Delta_i g(i) = g(i+1) - g(i)$. We now perform the integration against a smooth compactly supported test function $f:\R\to\R$. We define $f_\varepsilon:\Z\to\Z$ by $f_\varepsilon(x) = f(\varepsilon^{2/3}x)$ and let $M$ be such that $f(x) = 0$ on $[-M,M]^c$.  It follows from the previous display that, for all $\varepsilon>0$ small enough,
\begin{align*}
\MoveEqLeft[2]
\left|\sum_{i\in\Z}S^{\slope,\varepsilon}_{11}(2\gamma^{-1}\varepsilon^{-1}, \floor{2x\varepsilon^{-2/3}} + i) f_\varepsilon(i)\right|\\
&= 4\left(1-\slope^2\varepsilon^{2/3}\right)^{-1} \left|\sum_{i= -\floor{M\varepsilon^{-2/3}}}^{\floor{M\varepsilon^{-2/3}}}\P\left(X^{\slope,\varepsilon}(2\gamma^{-1}\varepsilon^{-1}) = \floor{2x\varepsilon^{-2/3}} + i\right) f_\varepsilon(i)\right|\\
&\leq 8\cdot \sup_{x\in\R}|f(x)| \sum_{i= -\floor{M\varepsilon^{-2/3}}}^{\floor{M\varepsilon^{-2/3}}} \P\left(X^{\slope,\varepsilon}(2\gamma^{-1}\varepsilon^{-1}) = \floor{2x\varepsilon^{-2/3}}+i\right).
\end{align*}
The sum of probabilities is upper bounded by $\P(X^{\slope,\varepsilon}(2\gamma^{-1}\varepsilon^{-1}) \geq \floor{(2x-M)\varepsilon^{-2/3}})$, which in turn can be upper bounded, by adding and subtracting the location $2\slope\varepsilon^{-2/3}$ on the characteristic and using \cite[Theorem~2.5]{landon2023tail}, by
\begin{align*}
\P\left(|X^{\slope,\varepsilon}(2\gamma^{-1}\varepsilon^{-1}) +2\slope \varepsilon^{-2/3} | \geq \floor{(2x-M)\varepsilon^{-2/3}} +2\slope \varepsilon^{-2/3}\right)
&\leq C\exp(-c\slope^{3}),
\end{align*}
which tends to 0 as $\slope\to\infty$. This completes the proof.
\end{proof}

In the remainder of this section we focus on proving the second half of Corollary~\ref{p.covariance matrix}, namely, \eqref{e.space-time covariance vanishes}.
The following is a consequence of summation by parts and will be useful in its proof. Here $h_0^{(\ell)}$ and $h^{\mrm{ASEP}}(h^{(k)}_0, 0; x,t)$ are the ASEP height functions associated to $\eta^{(\ell)}_0$ and $\eta^{(k)}_t$, respectively, as defined in Section~\ref{s.asep basic coupling}.

\begin{lemma}\label{l.summation by parts}
Let $f:\Z\to\R$ be finitely supported and $x,y\in\Z$. Then for any $t>0$, $M\in\N$, and $\{k,\ell\} = \{1,2\}$,
\begin{align*}
\MoveEqLeft[0]
\sum_{i\in\Z} \Cov\left(\eta^{(k)}_{t}(x+i), \eta^{(\ell)}_0(y)\right) f(i)\\
&= -(6M+1)^{-1}\sum_{|j|\leq 3M}\sum_{i\in\Z} \Cov\left(h^{\mrm{ASEP}}(h_0^{(k)},0; x+i, t), h_0^{(\ell)}(y+j)\right) (\Delta^2 f)(i-j)\\
&-(6M+1)^{-1}\!\sum_{i\in\Z}\! \Cov\!\left(h^{\mrm{ASEP}}\!(h^{(k)}_0, 0; x+i, t), h^{(\ell)}_0(y+3M+1)\!\right)\! \Delta^{\hspace{-0.1em}\mrm{f}}_i(f(i-3M-1))\\
&+ (6M+1)^{-1}\sum_{i\in\Z} \Cov\left(h^{\mrm{ASEP}}(h^{(k)}_0, 0; x+i, t), h^{(\ell)}_0(y-3M)\right) \Delta^{\hspace{-0.1em}\mrm{f}}_i(f(i+3M)).
\end{align*}
where $(\Delta^2 f)(x) = f(x+1)+f(x-1)-2f(x)$ is the discrete Laplacian and $\Delta^{\hspace{-0.1em}\mrm{f}}_i f(i) = f(i+1) - f(i)$ is the discrete forward derivative with respect to $i$. 
\end{lemma}

\begin{proof}
We first note that, by the stationarity of the particle configuration, 
$$\Cov(\eta^{(k)}_{t}(x+i), \eta^{(\ell)}_0(y)) = \Cov(\eta^{(k)}_{t}(x+i+j), \eta^{(\ell)}_0(y+j))$$
for any $x,y, i, j\in\Z$. Thus for any $M\in\N$,
\begin{align*}
\MoveEqLeft[4]
\sum_{i\in\Z} \Cov\left(\eta^{(k)}_{t}(x+i), \eta^{(\ell)}_0(y)\right) f(i)\\
&= (6M+1)^{-1}\sum_{|j|\leq 3M}\sum_{i\in\Z} \Cov\left(\eta^{(k)}_{t}(x+i+j+1), \eta^{(\ell)}_0(y+j+1)\right) f(i).
\end{align*}
%
Now, it holds by definition that $\eta_0^{(\ell)}(z+1) = -\Delta^{\hspace{-0.1em}\mrm{f}}_z h^{(\ell)}_0(z)$ and $\eta_t^{(k)}(z+1) = -\Delta^{\hspace{-0.1em}\mrm{f}}_z h^{\mrm{ASEP}}(h^{(k)}_0, 0; z, t)$.
We now do the change of variable $i\mapsto i-j$ and apply summation by parts on $i$ to obtain that the previous display equals (the boundary terms disappear since $f$ is finitely supported)
\begin{align*}
-(6M+1)^{-1}\!\!\!\sum_{|j|\leq 3M}\!\sum_{i\in\Z} \Cov\!\left(\!h^{\mrm{ASEP}}\!(h^{(k)}_0, 0; x+i, t), \eta^{(\ell)}_0(y+j+1)\right)\! \Delta^{\hspace{-0.1em}\mrm{f}}_i(f(i-j-1)).
\end{align*}
Changing the order of summation (they are both finite sums) and applying summation by parts with respect to $j$ yields
\begin{align*}
\MoveEqLeft[0]
(6M+1)^{-1}\sum_{i\in\Z} \sum_{|j|\leq 3M} \Cov\left(h^{\mrm{ASEP}}(h^{(k)}_0, 0; x+i, t), h^{(\ell)}_0(y+j)\right) \Delta^{\hspace{-0.1em}\mrm{f}}_j\Delta^{\hspace{-0.1em}\mrm{f}}_i(f(i-j))\\
&-(6M+1)^{-1}\!\sum_{i\in\Z}\! \Cov\!\left(h^{\mrm{ASEP}}\!(h^{(k)}_0, 0; x+i, t), h^{(\ell)}_0(y+3M+1)\!\right)\! \Delta^{\hspace{-0.1em}\mrm{f}}_i(f(i-3M-1))\\
&+ (6M+1)^{-1}\sum_{i\in\Z} \Cov\left(h^{\mrm{ASEP}}(h^{(k)}_0, 0; x+i, t), h^{(\ell)}_0(y-3M)\right) \Delta^{\hspace{-0.1em}\mrm{f}}_i(f(i+3M)).
\end{align*}
%
This completes the proof after noting that $\Delta^{\hspace{-0.1em}\mrm{f}}_i\Delta^{\hspace{-0.1em}\mrm{f}}_j(f(i-j)) = -(\Delta^2 f)(i-j)$.
\end{proof}

We will also need to know an asymptotic independence statement for the stationary horizon, the limit of the height functions associated to the stationary colored particle configuration. The following will be proven in Section~\ref{s.stationary horizon asymptotic independence}.

\begin{lemma}\label{l.sh representation}
For every $\slope>0$, there exists a pair of independent standard two-sided Brownian motions $B_1$ and $B_2$, a random process $M_\slope:\R\to\R$, and a coupling of $(\mf h_0^0, \mf h_0^\slope)$ of the stationary horizon with parameter $(0,\slope)$ with $B_1, B_2, M_\slope$ such that the following holds:
(i) $(\h_0^0(x), \h_0^\slope(x)) = (2^{1/2}B_1(x), 2^{1/2}B_2(x) + 2\slope x + M_\slope(x))$ for all $x\in\R$, and (ii) $M_\slope(x)$ satisfies, for some absolute constants $C,c>0$ and all $K, R, \slope>0$,
\begin{align}\label{e.M_beta uniform bound}
\P\left(\sup_{|x|\leq R\slope^2}|M_\slope(x)|> K\slope^{-1}\right) \leq CR\slope^4\exp(-cK),
\end{align}
and, for any fixed $x\in\R$, 
\begin{equation}\label{e.M_beta single point bound}
\P\left(|M_\slope(x)|> K\slope^{-1}\right) \leq C\exp(-cK).
\end{equation}
In particular, by taking $K = \slope^{1/2}$, $\h_0^0(x)$ and $\h_0^\slope(x) - 2\slope x$ are asymptotically independent as $\slope\to\infty$.
\end{lemma}

\begin{corollary}\label{c.second moment of M_beta}
For any $\varepsilon>0$ and $R>0$, there exists $C_1=C_1(\varepsilon,R)$ such that, for any $\slope>0$,
\begin{align*}
\E\left[\sup_{|x|\leq R\slope^2}M_\slope(x)^2\right]^{1/2}   \leq C_1\beta^{-1+\varepsilon}.
\end{align*}   
There also exists an absolute constant $C_2$ such that, for any $x\in\R$, $\E[M_\slope(x)^2]^{1/2}   \leq C_2\beta^{-1}.$
\end{corollary}

\begin{proof}
The second claim follows from \eqref{e.M_beta single point bound} and the identity $\E[X^p] = \int_0^\infty pt^{p-1}\cdot \P(X\geq t)\,\diff t$ for any non-negative random variable $X$, with $p=2$.

For the first claim, we first use that $\E[X^2]^{1/2} \leq \E[X^p]^{1/p}$ for any $p\geq 2$ and non-negative random variable $X$ along with the just mentioned identity to obtain that
\begin{align*}
\slope\E\left[\sup_{|x|\leq R\slope^2}M_\slope(x)^2\right]^{1/2}
&\leq \left(\int_0^{\infty} pt^{p-1} \P\left(\sup_{|x|\leq R\slope^2}|M_\slope(x)| \geq t\slope^{-1}\right)\,\diff t\right)^{1/p}\\
&\leq \left(CR\slope^4\int_0^{\infty} pt^{p-1} \exp(-ct)\,\diff t\right)^{1/p} = C_1\slope^{4/p}
\end{align*}
for some constant $C_1 = C_1(R,p)$, using \eqref{e.M_beta uniform bound}. Taking $p = 4\varepsilon^{-1}$ completes the proof.
\end{proof}

\begin{proof}[Proof of Corollary~\ref{p.covariance matrix}, \eqref{e.space-time covariance vanishes} and off-diagonal terms of \eqref{e.on diagonal covvariance}]
The limits of the off-diagonal terms of \eqref{e.on diagonal covvariance} being zero is implied by \eqref{e.space-time covariance vanishes}, so it is sufficient to prove the latter. We start under the assumed condition that $\{k,\ell\} = \{1,2\}$; we will later treat the two cases $(k,\ell) = (1,2)$ and $(k,\ell) = (2,1)$ separately.
Let $f:\R\to\R$ be a smooth compactly supported function, and let $M>0$ be such that $f$ is supported on $[-M,M]$. Define $f_\varepsilon : \Z\to\Z$ by $f_\varepsilon(x) = f(\varepsilon^{2/3}x)$ for all $\varepsilon>0$; so $f_\varepsilon$ is supported on $\intint{-\ceil{M\varepsilon^{-2/3}}, \ceil{M\varepsilon^{-2/3}}}$.
For shorthand, let $h^{(k)}_{2\gamma^{-1}\varepsilon^{-1}t}(x) = h^{\mrm{ASEP}}(h_0^{(k)},0; x, 2\gamma^{-1}\varepsilon^{-1}t)$. We note that $(h_0^{(1)}, h_0^{(2)})$ are the height functions at time 0 for the multi-color stationary ASEP with parameters $\bm\rho = (\frac{1}{2} \beta \varepsilon^{1/2}, \frac{1}{2})$ as defined in Section~\ref{s.asep stationary horizon}.

By Lemma~\ref{l.summation by parts}, taking $M$ there to be $\floor{M\varepsilon^{-2/3}}$,
\begin{align}
\MoveEqLeft[0]
\sum_{i\in\Z}S^{\slope, \varepsilon}_{kl}(2\gamma^{-1}\varepsilon^{-1}, \floor{2x\varepsilon^{-2/3}}+i)f_\varepsilon(i)\nonumber\\
&=-\frac{\varepsilon^{2/3}}{6\ceil{M\varepsilon^{-2/3}}+1}\nonumber\\
&\ \ \times \sum_{|j|\leq 3\ceil{M\varepsilon^{-2/3}}}\sum_{i\in\Z} \Cov\left(\varepsilon^{1/3}h^{(k)}_{2\gamma^{-1}\varepsilon^{-1}t}(\floor{2x\varepsilon^{-2/3}}+i), \varepsilon^{1/3}h_0^{(\ell)}(\floor{2y\varepsilon^{-2/3}}+j)\right)\nonumber\\[-2pt]
&\hspace{5cm}\times(\varepsilon^{-2/3})^2(\Delta^2 f_\varepsilon)(i-j)\label{e.covariance sum}\\
&\quad-\frac{1}{6\ceil{M\varepsilon^{-2/3}}+1}\nonumber\\
&\quad \times \Biggl(\sum_{i\in\Z} \Cov\left(\varepsilon^{1/3}h^{(k)}_{2\gamma^{-1}\varepsilon^{-1}t}(\floor{2x\varepsilon^{-2/3}}+i), \varepsilon^{1/3}h_0^{(\ell)}(\floor{2y\varepsilon^{-2/3}}+3\floor{M\varepsilon^{-2/3}}+1)\right)\nonumber\\[-2pt]
&\hspace{5cm}\times \varepsilon^{-2/3}\Delta^{\hspace{-0.1em}\mrm{f}}_i (f_\varepsilon(i-3\floor{M\varepsilon^{-2/3}}-1)) \nonumber\\
%
%
&\qquad\quad-\sum_{i\in\Z} \Cov\left(\varepsilon^{1/3}h^{(k)}_{2\gamma^{-1}\varepsilon^{-1}t}(\floor{2x\varepsilon^{-2/3}}+i), \varepsilon^{1/3}h_0^{(\ell)}(\floor{2y\varepsilon^{-2/3}}-3\floor{M\varepsilon^{-2/3}})\right)\nonumber\\[-2pt]
&\hspace{5cm}\times \varepsilon^{-2/3}\Delta^{\hspace{-0.1em}\mrm{f}}_i (f_\varepsilon(i+3\floor{M\varepsilon^{-2/3}})) \Biggr). \nonumber
\end{align}
In our case, since $h_0^{(k)}$ and $h_0^{(\ell)}$ are (marginally) stationary for (uncolored) ASEP, we know by Cauchy-Schwarz along with the tail probability bounds of \cite[Theorem 2.4]{landon2023tail} that, for any deterministic sequences $(i_\varepsilon)_{\varepsilon>0}$ and $(j_\varepsilon)_{\varepsilon>0}$ such that $i_\varepsilon, j_\varepsilon\in[-CM\varepsilon^{-2/3}, CM\varepsilon^{-2/3}]$ for each $\varepsilon>0$ and any fixed~$C$,
\begin{align*}
\MoveEqLeft[10]
\varepsilon^{1/3}\left(h^{(k)}_{2\gamma^{-1}\varepsilon^{-1}t}(\floor{2x\varepsilon^{-2/3}}+i_{\varepsilon}) + \floor{x\varepsilon^{-2/3}} + \tfrac{1}{2}i_\varepsilon - \tfrac{1}{2}t\varepsilon^{-1}\right)\\
&\times \varepsilon^{1/3}\left(h_0^{(\ell)}(\floor{2y\varepsilon^{-2/3}}+j_\varepsilon) + \floor{y\varepsilon^{-2/3}} + \tfrac{1}{2}j_\varepsilon\right)
\end{align*}
is uniformly integrable as a sequence in $\varepsilon$. Thus the convergence in distribution asserted by Theorem~\ref{t.asep general initial condition} and Corollary~\ref{c.SH convergence} implies convergence of the covariance, i.e., we have, letting $z_i = \frac{1}{2}\lim_{\varepsilon\to 0} \varepsilon^{2/3}i_\varepsilon$ and $z_j = \frac{1}{2}\lim_{\varepsilon\to 0} \varepsilon^{2/3}j_\varepsilon$,
\begin{align*}
\MoveEqLeft[14]
\lim_{\varepsilon\to 0}\Cov\left(\varepsilon^{1/3}h^{(k)}_{2\gamma^{-1}\varepsilon^{-1}t}(\floor{2x\varepsilon^{-2/3}}+i_{\varepsilon}), \varepsilon^{1/3}h_0^{(\ell)}(\floor{2y\varepsilon^{-2/3}}+j_\varepsilon)\right)\\
&= \tfrac{1}{4}\cdot\Cov\left(\h(\h_0^{(k)},0; x+z_i, t), \h_0^{(\ell)}(y + z_j)\right),
\end{align*}
where $(\mf h_0^{(1)}, \mf h_0^{(2)})$ are the stationary horizon with parameters $(\slope, 0)$, and the $\smash{\frac{1}{4}}$ comes from extra scaling factors of $\frac{1}{2}$ needed for each of $h^{(k)}_{2\gamma^{-1}\varepsilon^{-1}t}$ and $h_0^{(\ell)}$ to converge to $\h(\h^{(k)}_0, 0;\cdot, t)$ and $\h^{(\ell)}_0$, respectively; see Definition~\ref{d.asep general initial condition}.

Since $(\Delta^2 f_\varepsilon)(i_\varepsilon- j_\varepsilon) = f(\varepsilon^{2/3}(i_\varepsilon-j_\varepsilon) + \varepsilon^{2/3}) + f(\varepsilon^{2/3}(i_\varepsilon-j_\varepsilon)-\varepsilon^{2/3}) - 2f(\varepsilon^{2/3}(i_\varepsilon-j_\varepsilon))$, it is also immediate that $(\varepsilon^{-2/3})^2(\Delta^2 f_\varepsilon)(i_\varepsilon-j_\varepsilon) \to f''(2(z_i-z_j))$ as $\varepsilon\to 0$. Similarly,  $\varepsilon^{-2/3}\Delta^{\hspace{-0.1em}\mrm{f}}_i (f_\varepsilon(i_\varepsilon+3\floor{M\varepsilon^{-2/3}})) \to f'(2z_i+3M)$. It thus follows that the righthand side of \eqref{e.covariance sum} converges, as $\varepsilon\to0$, to
\begin{align*}
\MoveEqLeft[2]
\frac{1}{24M}\int_{\R\times[-3M,3M]} \Cov\left(\h(\h_0^{(k)},0; x+z_i, t), \h_0^{(\ell)}(y + z_j)\right) f''(2(z_i - z_j))\,\diff z_i\,\diff z_j\\
& -\frac{1}{24M}\int_{\R} \Cov\left(\h(\h_0^{(k)},0; x+z_i, t), \h_0^{(\ell)}(y + \tfrac{3}{2}M)\right) f'(2z_i - 3M))\,\diff z_i\\
& +\frac{1}{24M}\int_{\R} \Cov\left(\h(\h_0^{(k)},0; x+z_i, t), \h_0^{(\ell)}(y - \tfrac{3}{2}M)\right) f'(2z_i + 3M))\,\diff z_i.
\end{align*}
Since $f$ is smooth and supported on $[-M,M]$ by assumption, $|f'|$ is bounded and supported on $[-M,M]$ and $|f''(2(z_i-z_j))|$ is bounded and supported on $|z_i-z_j| \leq \frac{1}{2}M$. These facts (along with the fact that $|z_j|$ is bounded) impose bounds on $|z_i|$ of order $M$ in all three integrals. 
So the proof is complete if we show that 
\begin{equation}\label{e.sup cov bound}
\sup_{z_i\in[-4M, 4M]}\sup_{z_j\in[-M,M]}\Cov\left(\h(\h_0^{(k)},0; x+z_i, t), \h_0^{(\ell)}(y + z_j)\right)\to 0
\end{equation}
as $\slope\to\infty$; recall also from our assumptions that $|x|,|y|$ are allowed to be as large as $\slope^{2-\delta}$ for any fixed $\delta>0$. The previous display in turn is a consequence of Lemma~\ref{l.sh representation} and Corollary~\ref{c.second moment of M_beta}, as we explain next.

We first address the case $(k,\ell) = (2,1)$; the case $(k,\ell) = (1,2)$ is analogous and will be handled after.
For notational convenience, let $\h^{(2)}_t(z) = \h(\h^{(2)}_0, 0; z, t)$ for all $z$. Since $(\h_0^{(1)}, \h_0^{(2)})$ is distributed as the stationary horizon with parameter $(\slope, 0)$, by the coupling from Lemma~\ref{l.sh representation}, we can write  $(\h^{(1), \slope}_0(z), \h^{(2)}_0(z)) = (2^{1/2}B_1(z)+2\slope z + M_\slope(z), 2^{1/2}B_2(z))$ for all $z$, where $(B_1$, $B_2$) are two independent standard two-sided Brownian motions and $M_\slope$ is a process on $\R$. Then we can write, for every $\slope >0$,
\begin{align}
\MoveEqLeft[2]
\Cov\left(\h^{(2)}_t(x), \h_0^{(1)}(y+z)\right)\nonumber\\
&= \Cov\left(\sup_{w\in\R}\Bigl(\mc L(w,0; x, t) + B_2(w)\Bigr), B_1(y+z) + 2\slope(y+z) + M_\slope(y+z)\right)\nonumber\\
&= \Cov\left(\sup_{w\in\R}\Bigl(\mc L(w,0; x, t) + B_2(w)\Bigr), M_\slope(y+z)\right),\label{e.first case covariance}
\end{align}
the second equality by the independence of $B_1$, $B_2$, and $\mc L$. We bound this covariance using the Cauchy-Schwarz inequality. We start by bounding the second moment of the first term. Let $\tilde B_2(w) = B_2(w+x) - B_2(x)$. By doing the change of variable $w\mapsto w+x$ in the supremum above and using that, as a process in $w$, $\mc L(w+x,0; x, t) \stackrel{d}{=} \mc L(w,0; 0, t)$ we obtain that
\begin{align*}
\MoveEqLeft[8]
\E\left[\sup_{w\in\R}\Bigl(\mc L(w,0; x, t) + B_2(w)\Bigr)^2\right]^{1/2}\\
&= \E\left[\sup_{w\in\R}\Bigl(\mc L(w+x,0; x, t) + B_2(w+x)\Bigr)^2\right]^{1/2}\\
&= \E\left[\sup_{w\in\R}\Bigl(\mc L(w,0; 0, t) + \tilde B_2(w) + B_2(x)\Bigr)^2\right]^{1/2}\\
&\leq \E\left[\sup_{w\in\R}\Bigl(\mc L(w,0; 0, t) + \tilde B_2(w)\Bigr)^2\right]^{1/2} + \E\left[B_2(x)^2\right]^{1/2}.
\end{align*}

 Now, all moments of $\sup_{w\in\R}(\mc L(w,0; 0, t) + \tilde B_2(w))$ are uniformly bounded in $\slope$ (as follows, for instance, from \cite[Theorem 1.1]{ferrari2021upper} and using that $\tilde B_2$ is a rate two Brownian motion). Since $\E[B_2(x)^2]^{1/2} = |2x|^{1/2}$ and $|x|\leq \slope^{2-\delta}$, we obtain overall that the previous display is bounded by $C\slope^{1-\delta/2}$ for all large $\slope$. Thus since $\E[M_\slope(y+z)^2]^{1/2} \leq C\slope^{-1}$ by Corollary~\ref{c.second moment of M_beta}, returning to \eqref{e.first case covariance} we obtain \eqref{e.sup cov bound} in the $(k,\ell) = (2,1)$ case.

Finally we turn to the case of $(k,\ell) = (1,2)$. For notational convenience we let $\h^{(1)}_t(z) = \h(\h^{(1), \slope}_0, 0; z, t)$ for all $z$. We also have by Lemma~\ref{l.sh representation} that $(\h^{(1), \slope}_0, \h^{(2)}_0)$ can be coupled with $B_1$, $B_2$, $M_\slope$ such that $(\h^{(1), \slope}_0(z), \h^{(2)}_0(z)) = (2^{1/2}B_1(z) + 2\slope z + M_\slope(z), 2^{1/2}B_2(z))$ for all $z$. Then we may write, for all $\slope>0$,
\begin{align*}
\Cov\left(\h^{(1)}_t(x), \h_0^{(2)}(y+z)\right)
&= \Cov\left(\sup_{w\in\R}\Bigl(\mc L(w,0; x, t) + \h_0^{(1)}(w)\Bigr), \h_0^{(2)}(y+z)\right).
\end{align*}
Now, it follows from Lemma~\ref{l.localization of maximizer} (since $\h_0^{(1)}$ satisfies the second inequality in \eqref{e.uniform slope condition}, i.e., the one written there for $\h_0^{(2)}$) that the event $\msf{LocalizedMax}$ that the supremum in the first argument of the covariance term is attained within $w\in[-\slope^2, \slope^2]$ has probability at least $1-C\exp(-c\slope^3)$. Recall $\E[\h_0^{(2)}(y+z)] = 0$ since it is a two-sided Brownian motion. %
So we may write
\begin{align}
\MoveEqLeft[0]
\Cov\left(\sup_{w\in\R}\Bigl(\mc L(w,0; x, t) + \h_0^{(1), \slope}(w)\Bigr), \h^{(2)}_0(y+z)\right)\nonumber\\
&= \E\left[\left(\sup_{w\in\R}\Bigl(\mc L(w,0; x, t) + \h_0^{(1), \slope}(w)\Bigr)\right)\h_0^{(2)}(y+z)\left(\one_{\msf{LocalizedMax}} + \one_{\msf{LocalizedMax}^c}\right)\right]\nonumber\\
&\leq \E\!\left[\Bigl|\sup_{w\in\R}\Bigl(\mc L(w,0; x, t) + \h_0^{(1), \slope}(w)\Bigr)\Bigr|^3\right]^{1/3}\!\!\E\!\left[|\h_0^{(2)}(y+z)|^3\right]^{1/3}\!\P(\msf{LocalizedMax}^c)^{1/3}\nonumber\\
&\qquad + \E\left[\left(\sup_{|w|\leq \slope^2}\Bigl(\mc L(w,0; x, t) + \h_0^{(1), \slope}(w)\Bigr)\right)\h_0^{(2)}(y+z)\one_{\msf{LocalizedMax}}\right],\label{e.covariance breakup}
\end{align}
using the general form of the H\"older inequality. Now, the first expectation in the first term of \eqref{e.covariance breakup} grows at most polynomially in $\slope$: this follows by the triangle inequality since we can write $\sup_{w\in\R}(\mc L(w,0; x, t) + \h_0^{(1), \slope}(w)) \stackrel{\smash{d}}{=} \sup_{w\in\R}(\mc L(w,0; x, t) + 2^{1/2}\tilde B(w) + 2\slope w)$ for $\tilde B$ a standard two-sided Brownian motion; it holds as a process in $w$ that $\mc L(w,0; x, t) + 2\slope w \stackrel{\smash{d}}{=} \mc L(w-\slope t,0; x,t) + 2\slope x + \slope^2t$ \cite[Lemma 10.2]{dauvergne2018directed}; it holds as a process in $w$ that $\tilde B(w) = B(w-\slope t) + (\slope t)^{1/2}X$ for $B$ another standard two-sided Brownian motion and $X\sim\mc N(0,1)$ an independent standard normal random variable; and since we have exponentially decaying tails for $\sup_{ w\in\R}(\mc L(w,0; x, t) + 2^{1/2}B(w))$ \cite{ferrari2021upper}. 
Next, the second expectation in the first term of \eqref{e.covariance breakup} does not depend on $\slope$, and $\P(\msf{LocalizedMax}^c)\to 0$ superexponentially quickly in $\slope$ as $\slope\to\infty$, so the first term in \eqref{e.covariance breakup} vanishes in the limit.  

Using Lemma~\ref{l.sh representation}, we write the second term as
\begin{align}
\MoveEqLeft[1]
\E\left[\left(\sup_{|w| \leq \slope^2}\Bigl(\mc L(w,0; x, t) + B_1(w) + 2\slope w + M_\slope(w)\Bigr)\right)B_2(y+z)\right] \label{e.rewritten covariance}\\
&- \E\left[\left(\sup_{|w|\leq\slope^2}\Bigl(\mc L(w,0; x, t) + B_1(w) + 2\slope w + M_\slope(w)\Bigr)\right)B_2(y+z)\one_{\msf{LocalizedMax}}^c\right].\nonumber
\end{align}
The second term in the previous display tends to $0$ as $\slope\to\infty$ by a similar argument as just above, using the H\"older inequality (as for the first term of \eqref{e.covariance breakup}) and using that the third moment of the first term in the expectation grows at most polynomially while $\P(\msf{LocalizedMax}^c)$ decays at least exponentially, both in $\slope$. Next we note that the quantity inside the first expectation of the previous display is upper and lower bounded, respectively, by
\begin{align*}
&\sup_{|w| \leq \slope^2}\Bigl(\mc L(w,0; x, t) + B_1(w) + 2\slope w\Bigr)B_2(y+z) + \sup_{|w|\leq \slope^2} |M_\slope(w)|\cdot |B_2(y+z)| \quad\text{and}\\
&\sup_{|w| \leq \slope^2}\Bigl(\mc L(w,0; x, t) + B_1(w) + 2\slope w\Bigr)B_2(y+z) - \sup_{|w|\leq \slope^2} |M_\slope(w)|\cdot |B_2(y+z)|.
\end{align*}
By independence of $\mc L$, $B_1$, and $B_2$, the expectation of the first terms in both these quantities is zero. Since $\E[B_2(y+z)^2] = |y+z| \leq C\slope^{2-\delta}$ and $\E[\sup_{|w|\leq \slope^2} M_\slope(w)^2] \leq C\beta^{-2+\delta/2}$ for some $C$ and all $\slope$ by Corollary~\ref{c.second moment of M_beta}, by Cauchy-Schwarz the expectation of the second term in both quantities tends to zero as $\slope\to\infty$. Putting this together yields that \eqref{e.rewritten covariance} tends to zero as $\slope\to\infty$, which implies the same for the lefthand side of \eqref{e.covariance breakup}. This establishes \eqref{e.sup cov bound} in the case of $(k,\ell) = (1,2)$ and hence completes the proof.
\end{proof}

\subsection{Asymptotic independence of the stationary horizon} \label{s.stationary horizon asymptotic independence}

Here we give the proof of Lemma~\ref{l.sh representation}.

\begin{proof}[Proof of Lemma~\ref{l.sh representation}]
It is well-known (see, e.g., \cite[Exercise 2.16]{morters2010brownian} or \cite[(5.13)]{karatzas2014brownian}) that if $B:[0,\infty)\to\R$ is a standard Brownian motion and $\sigma > 0$, then
\begin{align}\label{e.drifted bm supremum}
\P\left(\sup_{x\geq 0}(\sigma B(x) - 2\slope x) > m\right) = \exp\left(-\frac{4m\slope}{\sigma^2}\right).
\end{align}
Let $B_1$ and $B_2$ be independent standard two-sided Brownian motions. By \cite[Definition D.1 and eq. (D.1)]{busani2022stationary}, $\h_0^0(x)$ and $\h_0^\slope(x)$ jointly (as processes) have the same distribution as $2^{1/2}B_1(x)$ and
\begin{align*}
\MoveEqLeft[8]
2^{1/2}B_1(x) + \sup_{-\infty < y\leq x} \left(2^{1/2}B_2(y) + 2\slope y - 2^{1/2}B_1(y)\right)\\
& - \sup_{-\infty < y\leq 0} \left(2^{1/2}B_2(y) + 2\slope y - 2^{1/2}B_1(y)\right).
\end{align*}
We will assume we work on a probability space which supports $B_1, B_2, \h_0^0$, and $\h_0^\slope$ such that the above identity in law holds as an almost sure equality on the space. Then we can rewrite $\h_0^\slope(x)-2\slope x$ as $2^{1/2} B_2(x) + M_\slope(x)$, where
\begin{equation*}%
\begin{split}
\MoveEqLeft[16]
M_\slope(x) := \sup_{-\infty < y\leq x} \left(2^{1/2}(B_2(y)- B_2(x)) + 2\slope (y-x) - 2^{1/2}(B_1(y)-B_1(x))\right)\\
&- \sup_{-\infty < y\leq 0} \left(2^{1/2}B_2(y) + 2\slope y - 2^{1/2}B_1(y)\right).
\end{split}
\end{equation*}
Now, observe that by the Markov property, invariance under time reversal, and scale invariance of $B_1$ and $B_2$, each of the suprema are for fixed $x$ (i.e., marginally) equal in distribution to 
\begin{align}\label{e.supremum marginal distribution}
\sup_{y \geq 0} \left(2 B(y) - 2\slope y\right),
\end{align}
where $B$ is a standard Brownian motion. Thus by \eqref{e.drifted bm supremum} and a union bound we see that  $M_\slope(x)$ satisfies, for an absolute constant $c>0$ and all $K>0$ and $x\in\R$,
\begin{equation}\label{e.M one-point bound}
\P(|M_\slope(x)| > K \slope^{-1})\leq 2\exp(-c K).
\end{equation}

We have to obtain a similar bound for $\sup_{x\in[-R\slope^2, R\slope^2]} |M_\slope(x)|$.  We first claim that, there exist $c>0$ and $C$ such that, for any $x \in \R$,
\begin{align}\label{e.M increment bound}
\P\left(\sup_{z\in[-\slope^{-2},\slope^{-2}]} |M_\slope(x) - M_\slope(x+z)| \geq K\slope^{-1} \right) \leq C\exp(-cK).
\end{align}
By symmetry, it suffices to bound $\sup_{z\in[-\slope^{-2}, \slope^{-2}]}M_\slope(x) - M_\slope(x+z)$. Let $B(w) = 2^{-1/2}(B_2(x-w)-B_2(x)- (B_1(x-w) - B_1(x)))$; note that it is distributed as a standard two-sided Brownian motion. Then, for $z\in[-\slope^{-2},\slope^{-2}]$,
\begin{align*}
\MoveEqLeft[2]
M_\slope(x) - M_\slope(x+z)\\
&= \sup_{0\leq w < \infty} \bigl(2B(w) - 2\slope w\bigr) - \sup_{-z\leq w < \infty} \bigl(2B(w) - 2\slope (w+z) - 2B(-z)\bigr)\\
&\leq \sup_{0\leq w < \infty} \bigl(2B(w) - 2\slope w\bigr) - \sup_{\slope^{-2}\leq w < \infty} \bigl(2B(w) - 2\slope w + 2\slope^{-1}\bigr)\\
&\qquad+ 2\sup_{w\in[-\slope^{-2},\slope^{-2}]} |B(-w)|.
\end{align*}
Now, $\sup_{z\in[-\slope^{-2},\slope^{-2}]} |B(-z)|$ is at most $K\slope^{-1}$ with probability at least $1-\exp(-cK^2)$. 
Next, by \eqref{e.drifted bm supremum}, the first supremum is at most $K\slope^{-1}$ with probability at least $1-\exp(-cK)$. We can write the second supremum as
\begin{align*}
\sup_{\slope^{-2}\leq w < \infty} \Bigl(2\bigl(B(w) - B(\slope^{-2})\bigr) - 2\slope (w-\slope^{-2})\Bigr) + 2B(\slope^{-2});
\end{align*}
by the Markov property, the supremum is equal in distribution to \eqref{e.supremum marginal distribution}. Thus, by \eqref{e.drifted bm supremum}, Gaussian tail bounds, and a union bound, the last display is smaller than $K\slope^{-1}$ with probability at least $1-2\exp(-cK)$. By relabeling $K$, this establishes \eqref{e.M increment bound}, and, on combining with \eqref{e.M one-point bound}, we obtain $\P(\sup_{y\in[-\slope^{-2},\slope^{-2}]}|M_\slope(x+y)| > K\slope^{-1})\leq C\exp(-cK)$ for any $x\in \R$. Now doing a union bound over a selection of order $R\slope^{4}$ many intervals of size $\slope^{-2}$ which cover $[-R\slope^2,R\slope^2]$ yields that, for all $\slope>0$,
\begin{equation*}
\P\left(\sup_{x\in[-R\slope^2,R\slope^2]} |M_\slope(x)| > K\slope^{-1}\right) \leq CR\slope^4\exp(-cK). \qedhere
\end{equation*}
\end{proof}

\printbibliography

\end{document}